\documentclass[10pt]{article}
\usepackage[utf8]{inputenc}

\usepackage[dvipsnames]{xcolor}
\usepackage[toc,page,header]{appendix}
\usepackage{titletoc}
\usepackage{tocvsec2}                   
\usepackage{setspace}

\usepackage{mathrsfs}
\usepackage{authblk}
\newcommand{\Rp}{\mbox{\upshape Re}}
\newcommand{\Ip}{\mbox{\upshape Im}\;}

\usepackage{amsmath, amsthm, amssymb, amsfonts}
\usepackage[makeroom]{cancel}

\DeclareMathOperator{\Ci}{Ci}
\DeclareMathOperator{\Si}{Si}

\definecolor{greenPlots}{RGB}{0,153,76}

\DeclareMathOperator{\err}{err}

\DeclareMathOperator{\corr}{corr}
\DeclareMathAlphabet\euscr{U}{eus}{m}{n}
\allowdisplaybreaks

\usepackage{enumerate}
\usepackage[shortlabels]{enumitem}

\usepackage{thmtools}
\usepackage{thm-restate}

\usepackage{physics}
\usepackage{braket}

\setitemize{itemsep=10pt}

\usepackage{color}

\title{Compressed Super-Resolution I: Maximal Rank Sum-of-Squares}

\author{Augustin Cosse}
\affil{Ecole Normale Sup\'erieure, Ulm, Paris\\
PSL Research University}

\date{December 31, 2019}

\usepackage{graphicx}
\usepackage{color}
\usepackage{MnSymbol}

\definecolor{greenA}{rgb}{0.1333, 0.5451, 0.1333}

\usepackage[colorlinks = true,
            linkcolor = blue,
	citecolor = blue, urlcolor = blue, linktocpage]{hyperref}

\usepackage{tikz}
\usepackage{pgfplots}
\usepackage{pgfplots}
\usepgfplotslibrary{fillbetween}

\usetikzlibrary{patterns}
\usetikzlibrary{matrix}
\usepgfplotslibrary{groupplots}
\pgfplotsset{compat=newest}

\newcommand{\sign}{\text{sign}}

\newcommand{\Id}{\text{Id}}

\newcommand{\diag}{\mathrm{diag}}

\newcommand{\scal}[2]{\left\langle #1,#2 \right\rangle}
\def\bs{\ensuremath\boldsymbol}

\def\bs{\ensuremath\boldsymbol}

\newtheorem{theorem}{Theorem}

\newtheorem{proposition}{Proposition}
\newtheorem{lemma}{Lemma}

\numberwithin{equation}{section}
\usepackage{setspace}

\usepackage{tocloft}

\newcounter{relctr} 
\everydisplay\expandafter{\the\everydisplay\setcounter{relctr}{0}} 

\AtBeginDocument{} 

\begin{document}

\renewcommand{\labelitemi}{$\bullet$}

\maketitle

\begin{abstract}

Let $\mu(t) = \sum_{\tau\in S} \alpha_\tau \delta(t-\tau)$ denote an $|S|$-atomic measure defined on $[0,1]$, satisfying $\min_{\tau\neq \tau'}|\tau - \tau'|\geq |S|\cdot n^{-1}$. Let $\eta(\theta) = \sum_{\tau\in S} a_\tau D_n(\theta - \tau) + b_\tau D'_n(\theta - \tau)$, denote the polynomial obtained from the Dirichlet kernel $D_n(\theta) = \frac{1}{n+1}\sum_{|k|\leq n} e^{2\pi i k \theta}$ and its derivative by solving the system $\left\{\eta(\tau) = 1, \eta'(\tau) = 0,\;  \forall \tau \in S\right\}$. We provide evidence that for sufficiently large $n$, $\Delta\gtrsim |S|^2 n^{-1}$, the non negative polynomial $1 - |\eta(\theta)|^2$ which vanishes at the atoms $\tau \in S$, and is bounded by $1$ everywhere else on the $[0,1]$ interval, can be written as a sum-of-squares with associated Gram matrix of rank $n-|S|$. Unlike previous work, our approach does not rely on the Fej\'er-Riesz Theorem, which prevents developing intuition on the Gram matrix, but requires instead a lower bound on the singular values of a (truncated) large ($O(1e10)$) matrix. Despite the memory requirements which currently prevent dealing with such a matrix efficiently, we show how such lower bounds can be derived through Power iterations and convolutions with special functions for sizes up to $O(1e7)$. We also provide numerical simulations suggesting that the spectrum remains approximately constant with the truncation size as soon as this size is larger than $100$. 
\end{abstract}
\hspace{.85cm}\begin{minipage}{14.8cm}\date{\textbf{Acknowledgement.} This work was funded by the Fondation Sciences Math\'ematiques de Paris, the CNRS and the Air Force Office of Scientific research by means of AFOSR Grant FA9550-18-1-7007. AC is grateful to Gabriel Peyr\'e and Ir\`ene Waldspurger for their help and valuable comments. AC also acknowledges funding from the FNRS and the Francqui Foundation.}\end{minipage}

\tableofcontents
\addtocontents{toc}{\setcounter{tocdepth}{1}}

\section{Introduction}

In this paper, we are interested in the problem of recovering a complex $|S|$-atomic measure $\mu(t) = \sum_{\tau\in S} \alpha_\tau \delta(t-\tau)$ from a low pass (with cutoff frequency $\Omega_c = 2\pi n$) version of its spectrum:
\begin{align}
\hat{\mu}(\omega_\ell) &= \int_{0}^1 e^{-2\pi i \omega_\ell t} \mu(dt) = \sum_{\tau \in S} \alpha_\tau e^{i\omega_\ell \tau}, \quad \omega_\ell \in \left\{-\Omega_c, \Omega_c\right\}\label{measureFourierTransform}
\end{align}

One approach  at solving~\eqref{measureFourierTransform}, consists in searching, among all measures fitting the observations, for the one with the smallest total variation~\cite{fernandez2016super}, leading to the following convex (yet infinite dimensional) program
\begin{align}
\min_{\tilde{\mu}}\;\|\tilde{\mu}\|_{\text{\upshape TV}}\quad \text{subject to}\quad \mathcal{F}_n \tilde{\mu} = y\label{TVminimization}
\end{align}
In~\eqref{TVminimization}, we use $\mathcal{F}_n$ denotes the Fourier transform of $\tilde{\mu}$. The convex problem dual to problem~\eqref{TVminimization} is known to be defined as a maximization over bounded trigonometric polynomials, i.e.
\begin{align}
\max_c\; \text{Re}\langle y, c \rangle\quad \text{subject to}\quad \|\mathcal{F}_n^* c\|_\infty \leq 1\label{dualProblem}
\end{align}
The condition $\|\mathcal{F}_n^* c\|_\infty \leq 1$ which appears in the dual~\eqref{dualProblem}, is equivalent to requiring the polynomial $1 - |\mathcal{F}_n^* c|^2$ to be non negative which, in this simple framework, as is formalized by the Fej\'er-Riesz Theorem (see~\cite{fejer}) is known to be equivalent to requiring this polynomial to be a sum of squares $1 - |\mathcal{F}_n^* c| = |s(e^{2\pi i \theta})|^2$. This result has an interesting consequence on the numerical solvability of problem~\eqref{dualProblem} as any sum-of-squares trigonometric polynomial as an associated semidefinite Gram matrix (and vice versa). In this case, we are thus interested in finding a matrix $X\succeq 0$ such that $\psi(\theta)^* X \psi(\theta) = 1 - |\mathcal{F}_n^* c|^2(\theta)$ where $\psi(\theta)$ denote the canonical vector $\psi(\theta) = [e^{-2\pi i n \theta}, \ldots, e^{2\pi i n \theta}]$. This problem can further read as the recovery of a complex matrix $Q\in \mathbb{C}^{n\times n}$ which together with $c$ satisfies conditions~\eqref{conditionsPSDmeasure} below
\begin{align}
\left[\begin{array}{cc}
Q & c^*\\
c & 1\end{array}\right] \succeq 0, \quad \sum_{i=1}^{n-j} Q_{i,i+j} = \left\{\begin{array}{l}
1,\quad  j=0,\\
0, \quad j=1,2,\ldots, n-1
\end{array}
\right.\label{conditionsPSDmeasure}
\end{align}
Following this equivalence between sum-of-squares trigonometric polynomials and semidefinite Gram matrices, the dual to the convex problem over measure can thus read as 
\begin{align}
\max_c\; \text{Re}\langle y, c \rangle\quad \text{subject to}\quad \eqref{conditionsPSDmeasure}\label{dualProblemSDP}
\end{align}
The Fej\'er-Riesz Theorem in this case therefore provides a straightforward relation between the original problem~\eqref{TVminimization} on measures and the semidefinite program, providing in passing an elegant numerical algorithm to solve this problem which at first might have appeared difficult to implement. Because of the equivalence between non negative trigonometric polynomials and sum-of-squares provided by the Fej\'er-Riesz Theorem, the measure $\mu(t)$ can be recovered through the semidefinite program~\eqref{dualProblemSDP} as soon as one can exhibit a non negative trigonometric polynomial of order $n$ taking the value $\sign(\alpha_\tau)$ at the location $\tau\in S$ of the atoms and bounded by one on the $[0,1]$ interval as shown by Proposition~\ref{prop:uniqueRecovery} below which is proved in~\cite{candes2014towards}. 
\begin{proposition}[\label{prop:uniqueRecovery}unique recovery]
Consider the complex measure $\mu = \sum_{\tau \in S} \alpha_\tau\delta(t-\tau)$ and suppose that we observe the samples $\hat{\mu}(\omega_\ell) = \int_0^1 e^{i\omega_\ell t} \mu(dt)$ at frequencies $|\omega_\ell |\leq \Omega_c$. The semidefinite program~\eqref{dualProblemSDPCompressed} has a unique solution corresponding to the polynomial $a(\theta)$ which takes the value $\sign(\alpha_k)$ on $S$ and is stricly bounded by $1$ on $[0,1]\setminus S$, provided that there exists a (dual) trigonometric polynomial,
\begin{align}
\eta(\theta) = \sum_{k|\omega_k\in \Omega} c_k e^{i \omega_k \theta} 
\end{align}
satisfying the following properties
\begin{enumerate}
\item \label{condition1} $|\eta(\theta)|\leq 1$, for all $\theta \in [0,1]$
\item \label{condition2} $\eta(\tau) = \sign(\alpha_\tau)$, for $\tau \in S$
\item \label{condition3}There exists trigonometric polynomials $q_j(\theta) = \sum_{k} b_{j,k} e^{2\pi ik \theta}$ satisfying 
\begin{align}
1 - |\eta(\theta)|^2 = \sum_{j} |q_j(\theta)|^2
\end{align}
\end{enumerate}
The third condition in Proposition~\ref{prop:uniqueRecovery} being induced through Fej\'er-Riesz, from the first two. 
\end{proposition}

Despite this attractive connection to semidefinite programming, it remains surprising that the recovery of an $S$-atomic measure with $|S|$ possibly much smaller than $n$ requires solving semidefinite programs of size $n^2$.

Compressed sensing was recently introduced as an approach to bridge this gap. In~\cite{tang2013compressed}, the authors manage to reduce the number of frequency samples needed from $\mathcal{O}(n)$ to only $\mathcal{O}(K)$, through various concentration arguments, and provided that the separation distance satisfies $\Delta\geq 1/n$ and that the signs are drawn uniformly, identically and idenpendently from the unit circle. Recovery of the measure however still requires solving $O(n^2)$ SDPs.

The semidefinite conditions~\eqref{conditionsPSDmeasure} are written on a matrix encoding the coefficient of the trigonometric polynomial $1 - |p(e^{2\pi i \theta})|^2$. It therefore seems intuitively right to assume that the semidefinite formulation following following from a random decimation of the set of samples in~\eqref{dualProblem} should be compressible as well (let's say on the order of $\mathcal{O}(|S|)$). Otherwise the effort of compressed sensing to reduced to match the complexity to the actual number of unknown will be vain. This idea was introduced in~\cite{da2018tight} without a complete proof. The idea of this paper is that when problem~\eqref{dualProblem} can be solved on a reduced set, $T\subseteq \Omega$ of frequencies, that is when considering a formulation of the form
\begin{align}
\min_{\tilde{\mu}}\;\|\tilde{\mu}\|_{\text{\upshape TV}}\quad \text{subject to}\quad \mathcal{F}_{T, n} \tilde{\mu} = y\label{TVminimizationDecimated}
\end{align}
where we let $\mathcal{F}_{T,n}$ to denote the partial Fourier transform, $\mathcal{F}_{T,n} \tilde{\mu} = \left\{\int \tilde{\mu} e^{2\pi i \omega t}\; dt,\; \omega\in \Omega\right\}$, it should be possible to solve the problem through a semidefinite program on the same subset of Fourier coefficients, i.e considering the reduced conditions
\begin{align}
\left[\begin{array}{cc}
A & a^*\\
a & 1\end{array}\right] \succeq 0, \quad \sum_{i=1}^{n-j} \left(\mathcal{S}_T^*A\mathcal{S}_{T}\right)_{i,i+j} = \left\{\begin{array}{l}
1,\quad  j=0,\\
0, \quad j=1,2,\ldots, n-1
\end{array}
\right.\label{conditionsPSDmeasureCompressed}
\end{align}
where $A \in \mathbb{C}^{|\Omega|\times |\Omega|}$, $a\in \mathbb{C}^{|\Omega|}$ now only encodes the considered coefficients (or equivalently the retained set of frequencies) and $\mathcal{S}_{T}:\;\left\{-n,n\right\}\mapsto T$ is used to denote the operator returning the subset of frequencies $\Omega$ from the complete discretized spectrum. The reconstruction problem~\eqref{TVminimizationDecimated} defined with optimal sample complexity would then be solved through a semidefinite program of the form
\begin{align}
\max_a\; \text{Re}\; \langle \mathcal{S}_{T} y, a \rangle\quad \text{subject to}\quad \eqref{conditionsPSDmeasureCompressed}\label{dualProblemSDPCompressed}
\end{align}

Extending the recovery guarantees of~\cite{candes2014towards} or~\cite{tang2013compressed} to a formulation such as~\eqref{dualProblemSDPCompressed} is not straightforward, in particular because of the lack of understanding regarding the structure of the sum-of-squares polynomial provided by the Fej\'er-Riesz Theorem. Althgouh guiaranteeing the equivalence between non negative trigonometric polynomials and SOS, Fej\'er-Riesz comes at a price : it is now impossible to manipulate the SOS certificate to make it fit inside extended (e.g. compressed sensing) frameworks. In this paper we provide an alternative construction of the certificate of Proposition~\ref{prop:uniqueRecovery} which does not rely on Fej\'er-Riesz Theorem and provides an explicit expression of the sum of squares decomposition. This result which constitutes the first step towards a better understanding of the computational complexity needed in the recovery of complex measures through compressed semidefinite programs is summarized through Theorem~\ref{MainTheorem} below.

\begin{theorem}\label{MainTheorem}
let $\mu(t) = \sum_{\tau\in S} \alpha_\tau \delta(t-\tau)$ denote a $S$-atomic measure with atoms satisfying $\min_{\tau\neq \tau'}|\tau - \tau'|\geq \Delta |S|$. We define the polynomial $\eta(\theta)$ from the Dirichlet kernel, $D_n(\theta) = $ and its derivative $\eta(\theta) = \sum_{\tau\in S} a_\tau D_n(\theta - \tau) + \sum_{\tau \in S} b_\tau D_n'(t - \tau)$ by requiring $\left\{\eta(\tau) = 1, \eta'(\tau) = 0, \;  \forall \tau \in S \right\}$. That is $\eta(\theta)$ takes the value $1$ and has zero derivative at the positions of the atoms. 

We let $U = [\psi(e^{2\pi i \tau_0}), \ldots, \psi(e^{2\pi i \tau_{|S|}})]$ and use $\mathcal{P}_U^\perp$ to denote the projector onto the orthogonal complement of $U$. We define the operator $\mathcal{A}\; :\;X \mapsto p(e^{2\pi i \theta})$, $\tilde{\mathcal{A}}^*\; :\;p(e^{2\pi i \theta})\mapsto X\in \mathbb{C}^{n\times n}$ as $\mathcal{A}(X) = \mathcal{T}\left\{\mathcal{P}_U^\perp X\mathcal{P}_U^\perp\right\}$ and $\tilde{\mathcal{A}}^* (p)= \mathcal{P}_U^\perp\mathcal{T}\left(Wp\right)\mathcal{P}_U^\perp$ where $W$ is the diagonal matrix $W = \diag(1, 2, \ldots, n+1, \ldots, 2, 1)$.  Let $\mathcal{M} = \Id - \mathcal{A}\tilde{\mathcal{A}}^*$. We let $\mathcal{Q}_K$ to denote the $O(K)$ matrix whose $[\ell_1, \ell_2]$ entry is defined from $\mathcal{M}$ as $[\mathcal{M}(D_n(\theta - \frac{\ell_2}{2n+1}))](\frac{\ell_1}{2n+1})$ (i.e the expression of $\mathcal{M}$ into the Dirichlet basis). We use $Q_K^\infty$ to denote the limit $\lim_{n\rightarrow \infty} Q_K$. Finally we use $P^\infty $ to encode the orthogonal projector onto $\text{span}(v_K, e_0)$ where $e_K$ is the $2K+1$ zero vector with $1$ at position $K+1$ and $v_K$ is defined as 
\begin{align}
v_K^\infty(k) & = \left\{\begin{array}{ll}
\frac{1}{k} \frac{(-1)^k}{\sqrt{\sum_{k\neq 0} \frac{1}{k^2}}} & \text{for $k\neq 0$} \\
 0 & \text{otherwise}
\end{array}\right.
\end{align}
Let $\sigma_1>\sigma_2>\ldots>\sigma_{2K+1}$ denote the singular values of the matrix $I- \mathcal{Q}_K + P^\infty$. As soon as $\sigma_{2K+1}(I- \mathcal{Q}_K + P^\infty)>.5$ for $K\geq O(1e10)$, there exists a sum of squares polynomial $\sum_{j} |s_j(\theta)|^2$ with associated Gram matrix of rank $n-|S|$ satisfying $1 - \eta(\theta)|^2 = \sum_{j} |s_j(\theta)|^2$
\end{theorem}

\subsection{Related work}

Super-resolution has been widely used in a variety of frameworks and has consequently led to multiple interpretations.  Mathematically speaking the term \textit{super-resolution} is usually reserved to denote the \textit{attempt at recovering the object outside the band of the instrument or, in other words, at restoring the object beyond the diffraction limit}~\cite{bertero1996super}. In the one dimensional framework, the Rayleigh distance coincides with the Nyquist criterion which both correspond to a bandwidth $\Omega$ satisfying $\Omega\geq \frac{\pi}{\tilde{\Delta}}$ where $\tilde{\Delta}$ is the grid spacing. When the frequency band is given by $[-\text{Nyquist},\text{Nyquist}]$, i.e. when the cutoff frequency is above the Nyquist frequency and there is no loss of information, as reminded in~\cite{donoho1992superresolution, bertero1996super}, the measure can be recovered by a simple Fourier inversion formula
\begin{align}
\alpha_k = \frac{\Delta}{2\pi} \int_{-\pi/\Delta}^{\pi/\Delta} e^{i \omega k \Delta} \mu(\omega)d\omega
\end{align}
or equivalently
\begin{align}
\alpha_k = \sum_{n=-\infty}^\infty \mu(n \frac{\pi}{\Omega}) \text{sinc}\left(\frac{\Omega}{\pi}(t_k - n \frac{\pi}{\Omega} )\right)
\end{align}
Knowing $f_\Omega$ is thus equivalent to knowing $f$ within the resolution limit $\Delta_x = \frac{\pi}{\Omega_c}$. 
The problem really becomes interesting on a reduced frequency band, $[-\Omega_c, \Omega_c]$. In this framework,  the term \textit{super resolution} thus implies some form of \textit{extrapolation of the spectrum} which is only feasible under appropriate priors on the object. One possible condition is the \textit{analyticity of the spectrum} which is a property of \textit{objects vanishing outside some finite region of space}~\cite{bertero1996super}.

The original mathematical framework introduced to describe super-resolution was focusing on non negative measures as such measures do not require any separation condition. In their original report~\cite{donoho1992maximum}, Donoho, Johnstone, Hoch and Stern consider the general problem of recovering a signal $x$ from measurements of the form $y = Kx +z$ when the system is ill posed. Super-resolution then refers to the particular choice $K = F_m$ consisting in the first $m$ rows of the discrete Fourier transform matrix. With this in mind, Donoho introduces the supremum
\begin{align}
\omega(\Delta; \mu) = \sup\{\|\mu - \tilde{\mu}\|_1\;:\; \|K\mu - K\tilde{\mu}\|_2\leq\Delta,\; \text{and}\; \alpha,\tilde{\alpha}\geq 0\}
\end{align}
In this case, the measure $\mu$ is said to admit super-resolution if $\omega(\Delta; \mu)\rightarrow 0$ as $\Delta\rightarrow 0$. The main result of the paper then shows that under nearly blackness assumption of the object ($\mu$ in this case has fewer than $(m-1)/2$ non zero elements), the Lipschitz constant is finite and recovery is possible. 

Those results are particularly interesting as  general super-resolution problem can be expressed from the relation of two lattices $\{k \Delta\}_{k = -\infty}^{\infty}$ and $\{k \tilde{\Delta}\}_{k = -\infty}^{\infty}$. The first lattice, on which the unknown measure is defined, and the second lattice defined from the cut-off frequency. The connection between the two lattices is then formalized through the \textit{super-resolution factor} (SRF) which was introduced by Donoho in 1991. In fact in~\cite{donoho1992maximum} Donoho defines the Rayleigh distance as $n/m$ where $n$ denotes the resolution of the measure lattice and $m$ is used to encode the cut-off frequency.

The general case of signed measures however requires a minimum separation $\Delta$ between the atoms, and the work of Cand\`es and Fernandez-Granda (see~\cite{candes2014towards} and subsequent papers) provided the first insight on the possible precise balance between such a separation distance and the cut-off frequency. In a series of papers, recovery of the measure is respectively certified for a separation $\Delta>2\lambda_c$ ($\Delta>1.87\lambda_c$ in the real framework)~\cite{candes2014towards}, and $\Delta>1.26\lambda_c$~\cite{fernandez2016super}.

What about necessary conditions then?  How far can we bring the spikes close to each other while maintaining the recovery? A first answer to that question was provided in the original paper of Donoho~\cite{donoho1992superresolution} by means of the \textit{upper} and \textit{lower uniform densities} of Beurling~\cite{beurling1989interpolation}, defined for a discrete set $S$ (i.e the support of the measure),
\begin{align}
u.u.d(S) &= \lim_{r\rightarrow \infty} r^{-1}\sup_t \#(S\cap [t,t+r))\\
l.u.d(S) &= \lim_{r\rightarrow \infty} r^{-1}\inf_t \#(S\cap [t,t+r))
\end{align}
Donoho certifies uniqueness of the measure for $u.u.d.(\text{supp}(\mu))<1$, $\Omega\geq 2\pi$ and non uniqueness for $l.u.d>1$ and $\Omega\leq \pi$. Along that line, section 6.1 in~\cite{fernandez2016super} provides numerical evidence that a separation of at least $\Delta>\lambda_c$ is needed. 

The gap between the sufficient $\Delta\geq 1.87\lambda_c$ and the numerical evidence for the necessary condition $\Delta\geq \lambda_c$ was bridged by Moitra in~\cite{moitra2015super}. The main results of this paper show that (i) as soon as $\Delta>\frac{1}{n-1}$, there exists a polynomial time algorithm which recovers the amplitudes and positions of the atoms and \textit{converges, in the presence of noise, to the true values at a rate inversely polynomial in the magnitude of the noise} and (ii) for cut-off frequency and separation distance satisfying $n<(1-\varepsilon)/\Delta$, there always exist a pair of measures $\mu$ $\tilde{\mu}$ with $|T|$ $\Delta$-separated atoms that satisfy
\begin{align}
\left|\sum_{t\in T} \mu_t e^{2\pi i t\omega} - \sum_{t'\in T} \tilde{\mu}_{t'} e^{2\pi i t'\omega} \right|\leq  \exp(-\varepsilon |T|)
\end{align}
The result of~\cite{moitra2015super} is also particularly interesting in that it provides a bound for stable recovery. Recovery of the measure depends on the conditioning of the Vandermonde matrix and the underlying algorithm can be understood as a greedy search for the atoms   

The question of stable recovery motivates the use of semidefinite programs. However such programs are notoriously painful to use on large dimensional problems because of the polynomial increase in the memory requirements and the complexity of SDP solvers. Moreover despite the clear interest of convex programming in terms of stability, this approach currently requires storing $O(\Omega_c)$ matrices, even when the measure is defined by a very small $|S|\ll \Omega_c$ number of atoms. This mathematical curiosity was studied through the lens of compressed sensing in~\cite{tang2013compressed} first (in this paper compressed sensing is used to reduce the number of Fourier coefficients needed from $O(\Omega_c)$ to $O(|S|)$) and~\cite{da2018tight} then. None of these papers however, were able to explain whether one could use $O(|S|)$ semidefinite programs with theoretical garantees on the recovery of the measure. One of the impediments in the quest for reduced semidefinite programs is the dependence of most results upon the famous Fej\'er-Riesz Theorem which provides a straightforward connection between non negative trigonometric polynomials and sum-of-squares which underly the use of semidefinite programs. Fej\'er-Riesz provide a simple and beautiful connection but does not provide any intuition regarding the structure of the sum-of-squares decomposition whose existence it guarantees. The lack of such information is unfortunate regarding the extension of semidefinite programs and their improvement as an undertanding of the structure of the sum-of-squares decomposition would make it possible to use this sum-of-squares as a basis for the definition of new certificates of optimality.

To bridge this gap and make the first step towards theoretical guarantees for compressed semidefinite programming in super resolution, this paper introduces an alternative proof technique that does not rely on the Fej\'er-Riesz Theorem and construct instead an explicit sum-of-squares decomposition for the optimality certificate. 

\section{Proof of Theorem~\ref{MainTheorem}}

%
%



\subsection{Ansatz}

A natural approach at constructing a sum of squares decomposition for the interpolating polynomial $\eta(\theta)$ that appears in Proposition~\ref{prop:uniqueRecovery} would be to start froma weighted version of the orthogonal projector $\mathcal{P}_U = \frac{1}{n+1}\left(I - U(U^*U)^{-1}U^*\right)$ where the columns of $U \in \mathbb{C}^{n\times |S|}$ are given by the canonical vectors 
$$\psi(\theta) \equiv [e^{-2\pi i n \theta}, \ldots, e^{2\pi i n \theta}]^T$$
at the atoms, i.e. $U = [\psi(\tau_1), \ldots, \psi(\tau_n)]$. The underlying polynomial $q_U^\perp(\theta) = \psi(\theta)^*\frac{1}{n+1}\mathcal{P}_U^\perp \psi(\theta)$ clearly satisfies $q_U^\perp (\tau)=\psi(\tau)^* \frac{1}{n+1}\mathcal{P}_U^\perp \psi(\tau) = u_\tau^* \frac{1}{n+1}(I - U(U^*U)^{-1}U^*)u_\tau = 0$. Both the interpolating polynomial $1- |\eta(\theta)|^2$ and the sum-of-squares polynomial $q_U^\perp(\theta)$ vanish at the atoms. To get exact equality, one approach would be to update the PSD matrix $\frac{1}{n+1}\mathcal{P}_U^\perp$ by adding to this matrix a correction $X_{\corr}$. We however want this correction to be such that it does not affect the nullspace of $\frac{1}{n+1}\mathcal{P}_U^\perp$ for any arbitrarily small change to this nullspace might add a negative eigenvalue to the spectrum of the Gram matrix $\frac{1}{n+1}\mathcal{P}_U^\perp$ and hence break the SoS nature of the underlying trigonometric polynomial. In other words, we want the correction to be a small as possible, in particular to have eigenvalues at most on the order of the eigenvalues of $\mathcal{P}_U^\perp$, but also to share the kernel of $\frac{1}{n+1}\mathcal{P}_U^\perp$ to avoid introducing negative eigenvalues. One approach could then be to look for the correction that has the smallest Frobenius norm. Such an approach would read as 
\begin{align}\label{introducingT}
\begin{split}
\min \quad &\|X_{\corr}\|_F^2\\
s.t. \quad & \mathcal{T}\left\{\mathcal{P}_U^\perp X_{\corr} \mathcal{P}_U^\perp\right\} = (1- |\eta(\theta)|^2) - q_U^\perp(\theta)  
\end{split}
\end{align}
In~\eqref{introducingT}, we introduce the notation $\mathcal{T}$ to denote the operator which maps a given Gram matrix onto its corresponding polynomial, i.e. if $ p = \mathcal{T}(H)$ for a given matrix $H\in \mathbb{C}^{(n+1)^2}$ then $p(\theta) = \sum_{|s|\leq n} e^{2\pi i s \theta }$ where $p_{s} = \sum_{(k-\ell)=s} H_{k, \ell}$. In other words, $\mathcal{T}$ is an operator that sums up the elements along the diagonals of $H$. The solution to~\eqref{introducingT} can be written in closed form by using $\mathcal{A}$ to denote the linear map $\mathcal{A}(X)=  \mathcal{T}\left\{\mathcal{P}_U^\perp X \mathcal{P}_U^\perp\right\}$ and relying on the pseudo-inverse, $X_{\corr} = \mathcal{A}^+\left[(1- |\eta(\theta)|^2) - q_U^\perp(\theta)\right]$, with $\mathcal{A}^+ = \mathcal{A}(\mathcal{A}\mathcal{A})^{-1}$. In what follows, we will use the notation $p_{\err}(\theta)$ to refer to the difference $p_{\err}(\theta) = (1- |\eta(\theta)|^2) - q_U^\perp(\theta)$. The  equality $1-|\eta(\theta)|^2 = \psi^*(\theta)(\frac{1}{n+1}\mathcal{P}_U^\perp + X_{\corr})\psi(\theta)$ is now satisfied by construction and we are left with proving that the resulting Gram matrix $\frac{1}{n+1}\mathcal{P}_U^\perp + X_{\corr}$ remains positive semidefinite. The non zero eigenvalues of $\frac{1}{n+1}\mathcal{P}_U^\perp$ have magnitude $\frac{1}{n+1}$. For this last condition to hold, it therefore suffices to show that that the Frobenius norm of the correction $X_{\corr}$ does not exceed $\frac{1}{n+1}$. We now use $W$ to denote the operator whose action on a polynomial $p = \sum_{|k|\leq n} p_k e^{2\pi i k \theta}$ is defined by $q = W p = \sum_{|k|\leq n} \frac{p_k}{n+1-|k|}e^{2\pi i k \theta}$. That is to say $Wp$ returns a polynomial $p$ whose coefficients are scaled according to the number of elements appearing on the corresponding diagonal in the Gram matrix. In particular, this implies that the operator $\tilde{T}^*$ defined by 
\begin{align}
\tilde{\mathcal{T}}^*(p) &\equiv \mathcal{T}(Wp) = \left(\begin{array}{ccc}
\frac{p_0}{n+1}& \ldots & p_n\\
\vdots & \ddots & \vdots\\
p_{-n} & \ldots & \frac{p_0}{n+1}
\end{array}\right)
\end{align}
satisfies $\mathcal{T}\tilde{\mathcal{T}}^*p = p$. To control the largest eigenvalue of $\mathcal{A}^+p_{\err}$, note that
\begin{align}
\|\mathcal{A}^+(p_{\err}) \|_\infty &\stackrel{(a)}{\leq} \|\mathcal{A}^+(p_{\err})\|_F\label{tmp000-3}\\
 &\stackrel{(b)}{\leq} \| \mathcal{A}^* W^{1/2}(W^{1/2}\mathcal{A}\mathcal{A}^*W^{1/2})^{-1}\|_{\text{op}} \|W^{1/2} p_{\err}\|_2\\
&\stackrel{(c)}{\leq}\lambda^{-1/2}_{\min}(W^{1/2}\mathcal{A}\mathcal{A}^*W^{1/2}) \|W^{1/2}p_{\text{err}}\|_2\label{decompositionWithoutReweighting}\\
&\stackrel{(d)}{\leq} \lambda^{-1/2}_{\min}(\mathcal{A}\mathcal{A}^*W) \|p_{\text{err}}\|_W\label{decompositionWithoutReweighting2}\\
&\stackrel{(e)}{\leq} \lambda^{-1/2}_{\min}(\mathcal{A}\tilde{\mathcal{A}}^*) \|p_{\text{err}}\|_W\label{decompositionWithoutReweighting3}
\end{align}
(d) follows from $W^{1/2}\mathcal{A}\mathcal{A}^*W^{1/2}\sim \mathcal{A}\mathcal{A}^*W$ through the similarity transformation $P = W^{1/2}$. This similarity in particular implies $\lambda_{\min}(W^{1/2}\mathcal{A}\mathcal{A}^*W^{1/2}) = \lambda_{\min}(\mathcal{A}\mathcal{A}^*W)\in \mathbb{R}^+$. In (e), we introduce the norm $\|p\|_W$ defined on the vector of coefficients of $p$ as $\sqrt{\sum_{|k|\geq n} \frac{|p_k|^2}{n+1-|k|}}$. The largest singular value of the correction can thus be bounded as
\begin{align}
\left\|X_{\text{corr}}\right\|_F&\leq  \lambda_{\min}^{-1/2}(\mathcal{A}\tilde{\mathcal{A}}^*)   \|p_{\err}\|_W\label{boundingFrobeniusNormDeter}
\end{align}

Bounding each of the two factors in~\eqref{boundingFrobeniusNormDeter} is precisely the point of the two lemmas below. Lemma~\ref{deterministicBoundSingularValueLemma} derives a lower bound on the spectrum of the non decimated normal operator $\mathcal{A}\tilde{\mathcal{A}}^*$. This lemma is proved in section~\ref{proofSingularValueLemma}

\begin{restatable}{lemma}{deterministicBoundSingularValueLemmaKey}
\label{deterministicBoundSingularValueLemma}
Consider the $K$ atomic measure $\mu(t) = \sum_{\tau\in S} \alpha_\tau \delta(t-\tau)$. The operators $\mathcal{A}$ and $\tilde{\mathcal{A}}$ are defined as $\mathcal{A}(X)= \mathcal{T}\left\{\mathcal{P}_U^\perp X\mathcal{P}_U^\perp\right\}$ and $\tilde{\mathcal{A}}(p) = \mathcal{P}_U^\perp \tilde{\mathcal{T}}p \mathcal{P}_U^\perp$ respectively. We define $\mathcal{Q}_K^\infty$ and $P^\infty$ as in the statement of Theorem~\ref{MainTheorem}. As soon as $\sigma_{2K+1}(I- \mathcal{Q}_K + P^\infty)>.5$, the map $\mathcal{A}\tilde{\mathcal{A}}^*$ obeys 
\begin{align}
\lambda_{\min}(\mathcal{A}\tilde{\mathcal{A}}^*) \geq 0.1 
\end{align}
for $n$ sufficiently large.
\end{restatable}

Given the lower bound on the eigenvalues of the pseudo inverse, we are left with showing that the norm $\|p\|_W$ never exceeds $(n+1)^{-1}$ and that $\eta(\theta)$ satisfies the conditions of Proposition~\ref{prop:uniqueRecovery}. This is the point of lemmas~\ref{interpolationPolynomialTangRecht} and~\ref{boundOnNormPerrWithoutSampling} below. Lemma~\ref{interpolationPolynomialTangRecht} first certifies that the polynomial $1-|\eta(\theta)|^2 = q_U^\perp(\theta) + \psi(\theta)^*X_{\corr}\psi(\theta) = \psi(\theta)^*\left(\frac{1}{n+1}\mathcal{P}_U^\perp+ X_{\corr}\right)\psi(\theta)$ vanishes at the atoms and is bounded by $1$ on $[0,1]$. Lemma~\ref{boundOnNormPerrWithoutSampling} then controls the $W$-norm of the deviation $q_U^\perp(\theta) - (1 - |\eta(\theta)|^2)$. Those two lemmas are respectively proved in sections~\ref{proofDirichletNoInterpolation} and~\ref{sec:boundPerr} below.

\begin{restatable}{lemma}{interpolationPolynomialCarlosKey}
\label{interpolationPolynomialTangRecht}
Consider the polynomial $\eta(\theta)$ defined as
\begin{align}
\eta(\theta) = \sum_{\tau \in S} a_\tau D_{N}(\theta - \tau) + \sum_{\tau\in S} b_\tau D_{N}(\theta - \tau).
\end{align}
whose coefficients $a_\tau, b_\tau$ are solutions of the system
\begin{align}
\left[\begin{array}{cc}
D_{0} & \frac{1}{\sqrt{\left|D_{N}''(0)\right|}}D_{1}\\
-\frac{1}{\sqrt{\left|D_{N}''(0)\right|}} D_{1} & \frac{1}{\left|D_{N}''(0)\right|} D_{2}\\
\end{array}\right] \left[\begin{array}{c}
a\\
b/\sqrt{\left|D_{N}''(0)\right|}
\end{array}\right]  =  \left[\begin{array}{c}
v\\
0\end{array}\right]
\end{align}
For $(D_{0})_{j,k} =  D_{N}(\tau_j - \tau_k)$, $(D_{1})_{j,k} = D'_{N}(\tau_j - \tau_k)$ and $(D_{2})_{j,k} = D''(\tau_j - \tau_k)$. This polynomial exactly takes the value $1$ at each of the points $\theta \in S$ and is stricly smaller than $1$ otherwise. In particular, is satisfies the condition~\ref{condition1} and~\ref{condition2} of lemma~\ref{prop:uniqueRecovery}. 
\end{restatable}

\begin{restatable}{lemma}{boundOnNormPerrWithoutSamplingKey}
\label{boundOnNormPerrWithoutSampling}
Consider the $K$ atomic measure $\mu(t) = \sum_{\tau\in S} \alpha_\tau \delta(t-\tau)$. We define the polynomial $p_{\err}(\theta)$ as the difference $p_{\err}(\theta) = (1 - |\eta(\theta)|^2) - q^\perp_U(\theta)$ where $q_U^\perp(\theta)$ is the polynomial associated to the projector $q_U^\perp(\theta) = \frac{1}{n+1}\psi(\theta)^*\mathcal{P}_U^\perp \psi(\theta) = \frac{1}{n+1}\psi(\theta)^*(I - U(U^*U)^{-1}U^*)\psi(\theta)$, and $\eta(\theta)$ is the interpolation polynomial defined from the Dirichlet kernel and its derivative, $\eta(\theta) = \sum_{\tau \in S} a_\tau D_n(\theta - \tau) + \sum_{\tau \in S}b_\tau D_n(\theta - \tau)$, by requiring $\left\{\eta(\tau) = 1, \eta'(\tau) = 0, \; \forall \tau \in S\right\}$. For a trigonometric polynomial $p(\theta) = \sum_{|s|\leq n} p_k e^{2\pi i k \theta}$ of order $n$, we define the norm $\|p\|_N$ as the weighted $\ell_2$ norm of the coefficients, $\left\|p\right\|_N = \sqrt{\sum_{|k|\leq n} \frac{|p_{k}|^2}{n+1-|k|}}$. The polynomial $p_{\err}$ obeys
\begin{align}
\|p_{\err}\|_W & \leq n^{-1} 
\end{align}
as soon as $\Delta \gtrsim |S|^2 \cdot n^{-1}$ (up to log factors). 
\end{restatable}
\

\section{\label{proofSingularValueLemma}Proof of lemma~\ref{deterministicBoundSingularValueLemma}}

To control the smallest eigenvalue of $\mathcal{A}\tilde{\mathcal{A}}^*$, we will show that the largest eigenvalue of the deviation between this operator and the identity is always stricly smaller than $1$ (in the developments below we choose to show that the largest eigenvalue of this deviation never exceeds $0.9$). This in turn implies that the smallest eigenvalue of $\mathcal{A}\tilde{\mathcal{A}}^*$ is lower bounded by $0.1$. We now use $\mathcal{M}$ to denote the deviation $\mathcal{M} = I - \mathcal{A}\tilde{\mathcal{A}}^*$. As we don't impose any constraints on the order of the polynomial, there is no restriction on the size of the matrice representing the linear map $\mathcal{M}$. We thus need to find a way to control the eigenvalues of this operator despite the lack of a bound on its size. It is in fact possible to show that any eigenpolynomial of $\mathcal{M}$, satisfying $\|p\|_\infty\leq 1$ and which would be associated to an eigenvalue $\lambda > 0.9$ satisfies an upper bound of the form $|p(\theta)|\leq \min(1,\sum_{\tau\in S} \frac{C_1}{1+n|\theta - \tau|})$ for some absolute constant $C_1\in \mathbb{R}^+$. The operator $\mathcal{M}$ has an approximately block structure with each polynomial $p_{\tau}(\theta)$ corresponding to the truncation of $p(\theta)$ around the atom $\tau$, being an approximate eigenpolynomial for $\mathcal{M}_\tau$ the corresponding block in $\mathcal{M}$. Thanks to this block structure, and the upper bound on the eigenpolynomial, one can show that the submatrix obtained by truncating the operator $\mathcal{M}$ around the block corresponding to any particular $\tau$, for a sufficiently large (yet constant) truncation size (1e7 in this case), and projecting this operator onto the Dirichlet basis, must have an eigenvector (in that same basis) with a sufficiently large eigenvalue. Since the resulting matrix is finite dimensional we can compute its extremal singular values (we do this by writing the matrix as a series of convolutions with special functions and using power iterations), show that the conditions derived on the eigenvalues of the the truncated operator cannot be satisfied, hence refuting the original assumption $\lambda>0.9$. These ideas are developed below. Section~\ref{boundAAoneatomic} start by developing the expression of $\mathcal{M}$ in the case of a one atomic measure. This expression is extended to the multi-atomic framework through section~\ref{boundAAMultiatomic}. Section~\ref{sec:truncation} uses the bound on the eigenpolynomial derived in section~\ref{boundAAMultiatomic} to derive the order of the truncation required for the truncated polynomial to be sufficiently close to an eigenvector of the truncated block. Given this order, section~\ref{sec:PowerIterations} finally computes the extremal singular values of the truncated block, expressed in the Dirichlet basis, $Q_K$ through convolutions with special (exponential, sine and cosine integral) functions and power iterations. The uncertainty on the singular value can be controled through an a priori error estimate derived from the residual. The python script used to compute the numerical estimates is available at~\url{http://www.augustincosse.com/research}.

\subsection{\label{boundAAoneatomic}One atomic operator}



We start by developing the operator $\mathcal{M}=\mathop{Id}-\mathcal{A}\tilde{\mathcal{A}}^*$ for a one atomic measure.  For a polynomial $p(e^{2\pi it})$ we let $||p||_\infty=\sup_{t\in\mathbb{R}}|p(e^{it})|$.

%
%
%


%

For $\theta \in [-1/2, 1/2]$, we use $\psi(\theta)$ to denote the canonical vector $\psi(\theta)=\left(\begin{smallmatrix}1\\e^{2\pi i\theta}\\\vdots\\e^{2\pi i n\theta}\end{smallmatrix}\right)$. We further use $u$ to denote the vector $u=\psi(0) = \left(\begin{smallmatrix}1\\\vdots\\1\end{smallmatrix}\right)$. From the definition of $\mathcal{A}\tilde{\mathcal{A}}^*$, one can check that we have  
  \begin{equation}
    \mathcal{M}(p)
    = T\left(\frac{1}{n+1}uu^*\tilde{T}^*(p)+\frac{1}{n+1}\tilde{T}^*(p)uu^*-\frac{1}{(n+1)^2}uu^*\tilde{T}^*(p)uu^*\right),\label{ApplicationOfMpOntoPolynomial}
  \end{equation}
  where $T$ encodes the mapping of a polynomial $p$ onto its corresponding Gram matrix
  \begin{equation*}
    \tilde{T}^*(p)=\begin{pmatrix}
      \frac{p_0}{n+1}&\dots&\frac{p_n}{1}\\
      \vdots&\ddots&\vdots\\
      \frac{p_{-n}}{1}&\dots&\frac{p_0}{n+1}
    \end{pmatrix}.
  \end{equation*}
The polynomial whose coefficients are defined by~\eqref{ApplicationOfMpOntoPolynomial} can thus read as 
  \begin{align}
    \mathcal{M}(p)(e^{2\pi i\theta})
    &= \psi(\theta)^* \left(\frac{1}{n+1}uu^*\tilde{T}^*(p)+\frac{1}{n+1}\tilde{T}^*(p)uu^*-\frac{1}{(n+1)^2}uu^*\tilde{T}^*(p)uu^*\right) \psi(\theta)\\
    &= \overline{D(e^{2\pi i\theta})}u^*\tilde{T}^*(p)\psi(\theta) + D(e^{2\pi i\theta})\psi(\theta)^*\tilde{T}^*(p)u - |D(e^{2\pi i\theta})|^2 p(1),\label{decompositionOfMasPolynomial}
  \end{align}
where $D$ here encodes the normalized Dirichlet kernel, $D(e^{2\pi i\theta})=\frac{1}{n+1}\sum_{k=0}^{n}e^{2\pi i k\theta}$. We start by considering the first term in~\eqref{decompositionOfMasPolynomial}. 
\begin{align}
  u^*\tilde{T}^*(p)v(\theta)\label{eq:bounduTv}
  & = \sum_{k=-n}^n \frac{p_k}{n+1-|k|}\underset{t-s=k}{\sum_{0\leq t,s\leq n}}e^{2\pi it\theta}\\
  & = \sum_{k=-n}^{0} \frac{p_k}{n+1+k}\left(\sum_{t=0}^{n+k}e^{2\pi it\theta}\right)
    - \frac{p_0}{n+1}{\color{black} \left(\frac{1 - e^{i2\pi (n+1)\theta}}{1-e^{i2\pi\theta}}\right)} + \sum_{k=0}^n \frac{p_k}{n+1-k}\left(\sum_{t=k}^ne^{2\pi it\theta}\right)\\
  & = \frac{1}{e^{2\pi i\theta}-1}\sum_{k=-n}^{0} \frac{p_k}{n+1+k}\left(e^{2\pi i(n+1+k)\theta}-1\right)
    - \frac{p_0}{n+1}{\color{black} \left(\frac{1 - e^{i2\pi (n+1)\theta}}{1-e^{i2\pi\theta}}\right)} \label{developmentUTtildepsiTheta1}\\
  & \qquad + \frac{e^{2\pi i(n+1)\theta}}{e^{2\pi i\theta}-1}\sum_{k=0}^n \frac{p_k}{n+1-k}\left(1-e^{-2\pi i(n+1-k)\theta}\right)\label{developmentUTtildepsiTheta2}\\
  & = \frac{1}{e^{2\pi i\theta}-1} \frac{1}{2\pi i} \int_0^\theta p_-(e^{2\pi i t})dt
    - \frac{p_0}{n+1}{\color{black} \left(\frac{1 - e^{i2\pi (n+1)\theta}}{1-e^{i2\pi\theta}}\right)} - \frac{e^{2\pi i(n+1)\theta}}{e^{2\pi i\theta}-1}\frac{1}{2\pi i}
    \int_0^{-\theta} p_+(e^{2\pi it})dt,
\end{align}
where $p_-$ et $p_+$ are defined as 
\begin{align*}
p_-(e^{2\pi i\theta})&=e^{2\pi i(n+1)\theta}\sum_{k=-n}^0 p_k e^{2\pi ik\theta};\\
p_+(e^{2\pi i\theta})&=e^{2\pi i(n+1)\theta}\sum_{k=0}^n p_k e^{-2\pi ik\theta}.
\end{align*}
If we further let $\tilde D$ to denote the Dirichlet kernel (without normalization) ($\tilde D(e^{2\pi i\theta})=\sum_{k=0}^n e^{2\pi ik\theta}$), one can check that for all $t$, we have 
\begin{equation*}
p_+(e^{2\pi it}) = e^{2\pi i(n+1)t}p\star \tilde D(e^{-2\pi it}),
\end{equation*}
from which one can write for all $\theta\in[0;0,5]$,
\begin{align}\label{eq:dev_int_p+}
  \frac{1}{e^{2\pi i\theta}-1} \int_0^{-\theta} p_+(e^{2\pi it})dt
  = \frac{1}{e^{2\pi i\theta}-1} \int_0^1 p(e^{2\pi is})\left(\int_0^{-\theta}e^{2\pi i(n+1)t}\tilde D(e^{-2\pi i(s+t)})dt\right)ds.
\end{align}
Similarly, we let 
\begin{align}\label{eq:dev_int_p-}
  \frac{1}{e^{2\pi i\theta}-1} \int_0^{\theta} p_-(e^{2\pi it})dt
  = \frac{1}{e^{2\pi i\theta}-1} \int_0^1 p(e^{-2\pi is})\left(\int_0^{\theta}e^{2\pi i(n+1)t}\tilde D(e^{-2\pi i(s+t)})dt\right)ds.
\end{align}
Any bound on~\eqref{eq:dev_int_p+} gives a bound on~\eqref{eq:dev_int_p-} and vice versa. We can thus focus on controling~\eqref{eq:dev_int_p+}.

%
%
%
%
%
In the developments below, we will also use the notation 
{\color{black}
\begin{align}
F_\theta(s)&=\int_0^{-\theta}e^{2\pi i(n+1)t}\tilde D(e^{-2\pi i(s+t)})dt.\\
& = -\int_{-\theta}^0e^{2\pi i(n+1)t}\tilde D(e^{-2\pi i(s+t)})dt.\label{integralFthetaDef}
\end{align}
for any $\theta\in[0;0,5]$ et every $s\in[-0,5;0,5]$,
}
{\color{black}

\subsection{\label{boundAAMultiatomic}Multi-atomic measures}

\begin{lemma}\label{lemmaBoundPmultiAtomic}
  Let $\lambda\in[0,9;1[$ and let $p$ denote a polynomial with $||p||_\infty\leq 1$ and
  \begin{align}
    \mathcal{M}(p)&=\lambda p\label{eigenPolyReminder}
  \end{align}
  There exists a universal constant $C_1$, such that, as soon as $\Delta \gtrsim \frac{\log^2(s)}{n}$, any eigenpolynomial $p(e^{\pi i \theta})$ of the operator $\mathcal{M}$ with $\lambda\in \lambda\in[0,9;1[$ obeys 
  \begin{align}
    \forall \theta\in\mathbb{R},\quad
    |p(e^{2\pi i\theta})| &\leq \min\left(1,\sum_{\alpha \in S}\frac{C_1}{1+n |\theta - \alpha|) }\right).\label{boundPolynomnialMultiSpikes}
 \end{align}
\end{lemma}

{\color{black}

We start by showing that the operator $\mathcal{M} = \mathcal{A}\tilde{\mathcal{A}}^* - I$ resulting from a multi-atomic measure is well approximated by the sum of each of the one atomic operators $\mathcal{M}_s$, provided that the separation distance between the atoms is sufficent (i.e proportionnal to $\lambda_c |T|$ where $|T|$ denotes the number of atoms.)
\begin{align}
\mathcal{M}[p](e^{2\pi i \theta}) &= (I - \mathcal{A}\tilde{\mathcal{A}}^*)[p](e^{i\theta})\\
& =  \psi(\theta)^*\tilde{T}^*(p)\psi(\theta)- \psi(\theta)^* \mathcal{P}_U^\perp \tilde{T}(p) \mathcal{P}_U^\perp \psi(\theta)\\
& = -\psi(\theta)^*\mathcal{P}_U\tilde{T}[p]\psi(\theta) - \psi(\theta)^*\tilde{T}[p]\mathcal{P}_U\psi(\theta) + \psi(\theta)^* \mathcal{P}_U\tilde{T}[p] \mathcal{P}_U\psi(\theta)\label{threeContributionsMultiSpikes}\\
\mathcal{M}[p](e^{2\pi i \theta})  &= Q_1(\theta) + Q_2(\theta) + Q_3(\theta)
\end{align}
where we have introduced the three polynomials $Q_1(\theta)  =  -\psi(\theta)^*\mathcal{P}_U \tilde{T}[p]\psi(\theta)$, $Q_2(\theta) = -\psi(\theta)^*\tilde{T}[p]\mathcal{P}_U\psi(\theta)$ and $Q_3(\theta) = \psi(\theta)^* \mathcal{P}_U\tilde{T}[p]\mathcal{P}_U \psi(\theta)$. The Hermitian matrix $\frac{1}{(n+1)}U^*U$ is invertible if $\|I - \frac{1}{2n+1}U^*U\|_\infty<1$, where $\|A\|_\infty = \max_{\|x\|_\infty\leq 1} \|Ax\|_\infty = \max_i \sum_{j}|a_{ij}|$. In this case the Neumann series converges and we can write 
\begin{align}
(U^*U)^{-1}& = \frac{1}{(2n+1)}I + \frac{1}{(2n+1)}\sum_{k=1}^\infty (I- \frac{1}{(2n+1)}U^*U)^k\label{NeumannSeriesMultispikePart}
\end{align}
Assuming that the atoms are separated by a distance of at least $\Delta$, we can write 
\begin{align}
\left\|I - \frac{1}{(2n+1)}U^*U\right\|_\infty &= \sup_{\tau\in  S} \sum_{\tau'\in S} \frac{1}{(2n+1)} |\langle \psi(\tau), \psi(\tau')\rangle|\\
&\leq \sup_{\tau \in S} \sum_{\tau' \in S} \left|\sum_{k=-n/2}^{n/2} \frac{1}{(n+1)} e^{2\pi i k(\tau_s - \tau_{s'})}\right|\\
&=  \sup_{\tau \in S} \sum_{\tau'\in S} \left|\frac{1}{(2n+1)}  \frac{1-e^{2\pi i (\tau - \tau')(2n+1)} }{1-e^{2\pi i (\tau - \tau')}}\right|\\
&\leq \sup_{\tau \in S}\frac{1}{(2n+1)} \sum_{\tau'\in S} \frac{1}{2 |\tau - \tau'|} \leq \frac{2\log(S)}{(2n+1)\Delta}\label{boundDeviationIminusUUstarMultiatomicOp}
\end{align}
The bound~\eqref{boundDeviationIminusUUstarMultiatomicOp} can be made sufficiently small as soon as $\Delta \geq 2\lambda_c\log(S)$. In particular, note that this bound implies $\sum_{k\geq 1} \|I - \frac{1}{2n+1}U^*U\|_\infty^k\leq \frac{2\log(S)}{(2n+1)\Delta} \frac{1}{1 - \frac{2\log(S)}{n\Delta}}$ which for $\Delta$ sufficiently larger than $\lambda_c\log(S)$ gives $\sum_{k\geq 1} \|I - \frac{1}{2n+1}U^*U\|_\infty^k\leq \varepsilon$. We now use $D_{n}(\theta-\theta_\ell)$ to denote the normalized Dirichlet kernel centered at the position $\theta = \theta_\ell$. Substituting~\eqref{NeumannSeriesMultispikePart} in the expression of $\mathcal{P}_U$, then back in~\eqref{threeContributionsMultiSpikes}. we get 
\begin{align}
\left|\mathcal{M}[p](e^{2\pi i \theta})\right| & \leq  \left|\psi(\theta)^*\frac{1}{2n+1}UU^*\tilde{T}[p]\psi(\theta) -\sum_{k=1}^\infty \psi(\theta)^*\frac{1}{2n+1}U(I- \frac{1}{2n+1}U^*U)^kU^*\tilde{T}[p]\psi(\theta)\right|\label{multiSpikeTerm001}\\
& +\left|\psi(\theta)^*\tilde{T}[p] \frac{1}{2n+1}UU^*\psi(\theta) -\sum_{k=1}^\infty \psi(\theta)^*\tilde{T}[p]\frac{1}{2n+1}U(I- \frac{1}{2n+1}U^*U)^kU^*\psi(\theta)\right|\label{multiSpikeTerm002}\\
& +\left| \psi(\theta)^*\frac{1}{2n+1}UU^*\tilde{T}[p]\frac{1}{2n+1}UU^*\psi(\theta)\right|\label{multiSpikeTerm003} \\
&+\left| \psi(\theta)^*\frac{1}{2n+1}UU^*\tilde{T}[p]U\sum_{k=1}^\infty (I - \frac{1}{2n+1}U^*U)^kU^* \psi(\theta)\right|\label{multiSpikeTerm004}\\
& + \left|\psi(\theta)^*\frac{1}{2n+1}U\sum_{k=1}^\infty(I - \frac{1}{2n+1}U^*U)^kU^*\tilde{T}[p]UU^*\psi(\theta)\right|\label{multiSpikeTerm005}\\
& + \left|\psi(\theta)^*\frac{1}{2n+1}U\sum_{k=1}^\infty (I-\frac{1}{2n+1}U^*U)^k U^* \tilde{T}[p] \sum_{k=1}^\infty (I - \frac{1}{2n+1}U^*U)^k \psi(\theta)\right|\label{multiSpikeTerm006}. 
\end{align}
From this, we can control the moduli of each of the terms above as 
\begin{align}
\eqref{multiSpikeTerm001}& \leq \sum_{\tau\in S}\left|D_n(\theta - \tau)\right| \left|u_{\tau}^*\tilde{T}[p]\psi(\theta)\right| +\sum_{k=1}^\infty \sum_{\tau\in S}\sup_{\tau'\in S} \left|D_n(\theta - \tau)\right| \left\|I-\frac{1}{2n+1}U^*U\right\|_\infty^k\left|u_{\tau'}^*\tilde{T}[p]\psi(\theta)\right|\\
&\leq \sum_{\tau\in S}\left|D_n(\theta - \tau)\right| \sup_{\tau'}\left|u_{\tau'}^*\tilde{T}[p]\psi(\theta)\right|(1+\varepsilon)\label{boundOperatorLeftRight1}\\ 
\eqref{multiSpikeTerm002}&\leq  \sum_{\tau\in S}\left|\psi(\theta)^*\tilde{T}[p] u_{\tau}\right|\left|D_{n}(\theta - \tau)\right| + \sum_{k=1}^\infty\sum_{\tau\in S}\sup_{\tau'\in S} \left|D_n(\theta - \tau)\right|\left|\psi(\theta)^*\tilde{T}[p] u_{\tau'}\right|\left\|I-\frac{1}{2n+1}U^*U\right\|^k_\infty\\
&\leq  \sum_{\tau\in S}\sup_{\tau'\in S}\left|\psi(\theta)^*\tilde{T}[p] u_{\tau'}\right|\left|D_{n}(\theta - \tau)\right| (1+\varepsilon).\label{boundOperatorLeftRight2}
\end{align}
as well as 
\begin{align}
\eqref{multiSpikeTerm003}&\leq \sum_{\tau\in S}\sum_{\tau'\in S} \left|D_n(\theta - \tau)u_\tau^*\tilde{T}[p]u_{\tau'}D_n(\theta -\tau')\right|\label{OperatorMAllSpikesTermaTermtmp1}\\
&\leq \sum_{\tau\in S}\sum_{\tau'\in S} \left|D_n(\theta - \tau)\right|\left|D_n(\theta - \tau')\right| \left|u_\tau^* \tilde{T}[p] u_{\tau'}\right|\\
&\leq \sum_{\tau\in S}\left|D_n(\theta - \tau)\right|\left|u_\tau^* \tilde{T}[p] u_{\tau'}\right|\label{boundDoubleTerm1}\\
\eqref{multiSpikeTerm004}&\leq \left|\sum_{\tau\in S}\sum_{\tau'\in S}\sum_{\tau''\in S}D_n(\theta - \tau)u_\tau^*\tilde{T}[p] u_{\tau'}^* \sum_{k=1}^\infty (I - U^*U)^k[\tau', \tau''] D_n(\theta - \tau'')\right|\\
&\leq \sum_{\tau\in S}\sum_{\tau''\in S}\sum_{k=1}^\infty\sup_{\tau'} \left|D_n(\theta - \tau)\right| \left|u_\tau^*\tilde{T}[p]u_{\tau'}\right| \left|D_n(\theta - \tau'')\right|\left\|I- U^*U\right\|^k_\infty. \\
&\leq \sum_{\tau\in S}\sum_{k=1}^\infty\sup_{\tau'} \left|D_n(\theta - \tau)\right| \left|u_\tau^*\tilde{T}[p]u_{\tau'}\right| \left\|I- U^*U\right\|^k_\infty. \label{boundDoubleTerm2}\\
\eqref{multiSpikeTerm005}&\leq \left|\sum_{\tau\in S}\sum_{\tau''\in S}\sum_{\tau'\in S} D_{n}(\theta - \tau) \left(\sum_{k=1}^\infty \left(I - U^*U\right)^k
[\tau, \tau']\right)u_{\tau'}^*\tilde{T}[p] u_{\tau''} D_n(\theta - \tau'')\right|\\
&\leq \sum_{k=1}^\infty\sum_{\tau''\in S}\sum_{\tau\in S} \left|D_n(\theta - \tau)\right| \sup_{\tau'\in S} \left|u_{\tau'}^*\tilde{T}[p]u_{\tau''}\right|\left|D_n(\theta - \tau'')\right| \left\|I - U^*U\right\|_\infty^k\\
&\leq \sum_{k=1}^\infty\sum_{\tau''\in S} \sup_{\tau'\in S} \left|u_{\tau'}^*\tilde{T}[p]u_{\tau''}\right|\left|D_n(\theta - \tau'')\right| \left\|I - U^*U\right\|_\infty^k\label{boundDoubleTerm3}\\
\eqref{multiSpikeTerm006} &\leq \left|\psi(\theta)^*U\sum_{k'=1}^\infty (I-U^*U)^{k'} U^* \tilde{T}[p] U\sum_{k=1}^\infty (I-U^*U)^kU^* \psi(\theta)\right|\\
&\leq\sum_{\tau,\tau'\in S} \sum_{k=1}^\infty\sum_{k'=1}^\infty |D_n(\theta - \tau)||D_n(\theta - \tau')| \left\|I-U^*U\right\|_\infty^k \sup_{\tau, \tau'}|u_\tau^*\tilde{T}[p]u_{\tau'}|\\
&\leq \sum_{\tau, \tau'\in S}\sum_{k=1}^\infty\sum_{k'=1}^\infty |D_n(\theta - \tau)||D_n(\theta - \tau')| \left\|I-U^*U\right\|_\infty^k\left\|I-U^*U\right\|_\infty^{k'} \sup_{\tau, \tau'}|u_\tau^*\tilde{T}[p]u_{\tau'}|\\
&\leq \sum_{\tau\in S} |D_n(\theta - \tau)|\left(\sum_{k\geq 1} \left\|I-U^*U\right\|_\infty^k\right)^2 \sup_{\tau'', \tau'''}|u_{\tau''}^*\tilde{T}[p]u_{\tau'''}|\label{OperatorMAllSpikesTermaTermtmpLast}
\end{align}
In the lines above, we use $\left(I - U^*U\right)^k[\tau, \tau']$ to denote the $(\tau,\tau')$ entry of the $k^{th}$ power of $\left(I - U^*U\right)$.  Adding~\eqref{boundDoubleTerm1},~\eqref{boundDoubleTerm2} and~\eqref{boundDoubleTerm3} and~\eqref{OperatorMAllSpikesTermaTermtmpLast} together, we get 
\begin{align}
\eqref{multiSpikeTerm003} + \eqref{multiSpikeTerm004} + \eqref{multiSpikeTerm005} + \eqref{multiSpikeTerm006} &\leq  \sum_{\tau\in S}\left|D_n(\theta - \tau)\right|\sup_{\tau', \tau''}\left|u_{\tau'}^* \tilde{T}[p] u_{\tau''}\right|(1+3\varepsilon)
\end{align}
Adding~\eqref{boundOperatorLeftRight1} and~\eqref{boundOperatorLeftRight2} gives the final bound on the modulus $|\mathcal{M}[p](e^{2\pi i \theta})|$, 
\begin{align}
|\mathcal{M}[p](e^{2\pi i \theta})|&\leq \sum_{\tau\in S}\left|D_n(\theta - \tau)\right| \sup_{\tau'}\left|u_{\tau'}^*\tilde{T}[p]\psi(\theta)\right|(1+\varepsilon)\\
&+ \sum_{\tau\in S}\sup_{\tau'\in S}\left|\psi(\theta)^*\tilde{T}[p] u_{\tau'}\right|\left|D_{n}(\theta - \tau)\right| (1+\varepsilon)\\
&+ \sum_{\tau\in S}\left|D_n(\theta - \tau)\right|\sup_{\tau', \tau''}\left|u_{\tau'}^* \tilde{T}[p] u_{\tau''}\right|(1+3\varepsilon)\label{totalBoundOnModulusofMpTighterWithEpsilon}
\end{align}

We introduce the notation $K_p(\tau, \theta)$ to denote the term $K_p(\tau, \theta) = \psi(\theta)\tilde{T}[p]\psi(\tau)$. In particular, note that we have 
\begin{align}
K_p(\tau, \theta) &= \sum_{|s|\leq n}\sum_{(k-\ell) = s} \frac{p_{(k-\ell)}}{n+1 - |k-\ell|}e^{2\pi i k \tau}e^{-2\pi i \ell \theta}\\
& = \sum_{|s|\leq n} \frac{p_s}{n+1-|s|} \sum_{k-\ell = s} e^{2\pi i k \tau}e^{-2\pi i \ell \theta}\\
& = \sum_{|s|\leq n} \frac{p_s}{n+1-|s|} \sum_{(\ell-k) = s} e^{2\pi i \ell \tau}e^{-2\pi i k \theta}\\
& = \sum_{|s|\leq n} \frac{p_s}{n+1-|s|} \sum_{(\ell - k) = -s} e^{-2\pi i \ell \tau} e^{2\pi i k \theta}\\
& = \sum_{|s|\leq n} \frac{p_{-s}}{n+1-|s|} \sum_{(\ell-k)= s} e^{-2\pi i \ell \tau}e^{2\pi i k \theta} = K_{\tilde{p}}(\theta, \tau).  
\end{align}
In the last line, we introduced the notation $\tilde{p}(\theta)$ to denote the reflection of $p(\theta)$ around the origin. I.e. $\tilde{p}(\theta) = \sum_{|s|\leq n} p_{-s}e^{2\pi i s \theta} = p(-\theta)$. Developing $K_p(\tau,\theta)$, we get  

\begin{align}
K_p(\tau, \theta) &= \sum_{s = 1}^n \frac{p_s}{n+1-s} \sum_{-\ell = -n/2 + s}^{n/2} e^{2\pi i k \tau} e^{-2\pi i \ell \theta}+ \sum_{s = -n}^{-1} \frac{p_s}{n+1+s} \sum_{k=-n/2}^{n/2 + s} e^{2\pi i k \tau} e^{-2\pi i \ell \theta}   + \frac{p_0}{n+1}\sum_{k=-n/2}^{n/2} e^{2\pi i k (\theta -\tau)}\\
& = \sum_{s=1}^n \frac{p_s}{n+1-s}\sum_{-\ell = -n/2+s}^{n/2} e^{2\pi i (s - (-\ell))\tau} e^{-2\pi i \ell \theta}+ \sum_{s=-n}^{-1} \frac{p_s}{n+1+s} \sum_{k=-n/2}^{n/2+s} e^{2\pi i k \tau} e^{-2\pi i (k-s) \theta}\\
& + \frac{p_0}{n+1}\sum_{k=-n/2}^{n/2} e^{2\pi i k (\theta -\tau)}\\
& = \sum_{s=1}^n \frac{p_s}{n+1-s} \sum_{\ell' = -n/2+s}^{n/2} e^{2\pi i (s-\ell')\tau} e^{2\pi i \ell'\theta} + \sum_{s=-n}^{-1} \frac{p_s}{n+1+s} \sum_{k = -n/2}^{n/2+s} e^{2\pi i k \tau} e^{-2\pi i (k-s)\theta}\\
&+ \frac{p_0}{n+1}\sum_{k=-n/2}^{n/2} e^{2\pi i k (\theta- \tau)}\\
& = \sum_{s=1}^n \frac{p_s}{n+1-s} \sum_{\ell'=s - n/2}^{n/2} e^{2\pi i s\tau} e^{2\pi i \ell' (\theta - \tau)}+ \sum_{s=-n}^{-1} \frac{p_s}{n+1+s}\sum_{k=-n/2}^{n/2+s} e^{-2\pi i k (\theta - \tau)}e^{2\pi i s \theta}\\
&+ \frac{p_0}{n+1}\sum_{k=-n/2}^{n/2} e^{2\pi i k (\theta- \tau)}\\
& = \sum_{s=1}^n \frac{p_s}{n+1-s} e^{-(n/2) 2\pi i (\theta - \tau) } \sum_{\ell' = s}^{n} e^{2\pi i s\tau}e^{2\pi i \ell' (\theta - \tau)}+ \sum_{s=-n}^{-1}\frac{p_s}{n+1+s}\sum_{k= -n/2}^{n/2+s} e^{2\pi i k \tau} e^{-2\pi i (k-s) \theta} \\
& + \frac{p_0}{n+1}\sum_{k=-n/2}^{n/2} e^{2\pi i k (\theta - \tau)}\\
& = \sum_{s=1}^n \frac{p_s}{n+1-s} e^{-(n/2)2\pi i (\theta - \tau)}\sum_{\ell'=s}^{n} e^{2\pi i s \tau} e^{2\pi i \ell' (\theta - \tau)}+ \sum_{s=-n}^{-1} \frac{p_s}{n+1+s} \sum_{k=-n/2}^{n/2+s} e^{-2\pi i (k-s)\tau} e^{2\pi i k \theta}\\
& + \frac{p_0}{n+1}\sum_{k=-n/2}^{n/2} e^{2\pi i k (\theta - \tau)}\\
& = \sum_{s=1}^n \frac{p_s}{n+1-s} e^{2\pi i s \tau} e^{-(n/2)2\pi i (\theta - \tau)} e^{2\pi i s(\theta - \tau)} \left[\frac{1 - e^{2\pi i (n+1-s)(\theta - \tau)}}{1 - e^{2\pi i (\theta - \tau)}}\right]\\
&+ \sum_{s = -n}^{-1} \frac{p_s}{n+1+s} e^{2\pi i s \tau} e^{-(n/2)2\pi i (\theta - \tau)} \left[\frac{1 - e^{2\pi i (n+1+s)(\theta - \tau)}}{1- e^{2\pi i (\theta - \tau)}}\right]\\
&+ \frac{p_0}{n+1}\sum_{k=-n/2}^{n/2}e^{2\pi i k (\theta - \tau)}\\
& = \sum_{s=0}^n \frac{p_s}{n+1-s} e^{2\pi i s\tau} \left(e^{-(n/2)2\pi i (\theta - \tau)}e^{2\pi i (n+1)(\theta - \tau)}\right)e^{-2\pi i (n+1-s)(\theta - \tau)} \left[\frac{1 - e^{2\pi i (n+1-s)(\theta - \tau)}}{1 - e^{2\pi i (\theta - \tau)}}\right]\label{finalTMPboundMultiSpikeOp1}\\
&+ \sum_{s=-n}^{0} \frac{p_s}{n+1+s} e^{2\pi i s\tau} \left(e^{-(n/2)(2\pi i)(\theta - \tau)}\right) \left[\frac{1 - e^{2\pi i (n+1+s)(\theta - \tau)}}{1-e^{2\pi i (\theta - \tau)}}\right]\\
&- \frac{p_0}{n+1}\sum_{k=-n/2}^{n/2} e^{2\pi i k (\theta - \tau)}. \label{finalTMPboundMultiSpikeOp3}
\end{align}
we now introduce the notations $p_{+,\tau}(e^{2\pi i \theta})$ and $p_{-,\tau}(e^{2\pi i \theta})$ to denote the polynomials 
\begin{align}
p_{+,\tau}(e^{2\pi i \theta})& =e^{2\pi i (n+1) \theta} \sum_{s=0}^n p_s e^{2\pi i s \tau} e^{-2\pi i s \theta}\\
p_{-,\tau}(e^{2\pi i \theta})& = e^{2\pi i (n+1) \theta}\sum_{s=-n}^{0} p_se^{2\pi i s\tau}  e^{2\pi i s \theta}
\end{align}
Now substituting those expressions in~\eqref{finalTMPboundMultiSpikeOp1} to~\eqref{finalTMPboundMultiSpikeOp3}, we get
\begin{align}
K_p(\tau, \theta) &= \frac{e^{-(n/2)2\pi i (\theta - \tau)}e^{2\pi i (n+1)(\theta - \tau)}}{1-e^{2\pi i (\theta - \tau)}} \int_{0}^{-(\theta - \tau)} p_{\tau, +}(e^{2\pi i t}) \; dt + \frac{e^{-(n/2)2\pi i (\theta - \tau)}}{e^{2\pi i (\theta - \tau)}-1} \int_{0}^{\theta-\tau} p_{\tau, -}(e^{2\pi i t})\; dt \\
&- \frac{p_0}{n+1}\sum_{k=-n/2}^{n/2} e^{2\pi i (\theta-\tau)}\end{align}
Note that if we again make use of the notation $\tilde{p}(e^{2\pi i \theta}) = p(-e^{2\pi i \theta})$, we can write the polynomials $p_{+,\tau}(e^{2\pi i \theta})$ and $p_{-,\tau}(e^{2\pi i \theta})$ as
\begin{align}
p_{+, \tau}(e^{2\pi i t}) = e^{2\pi i (n+1)t} p(e^{2\pi i (t+\tau)})\star \tilde{D}(e^{-2\pi i t})\\
 p_{-, \tau}(e^{2\pi i t}) = e^{2\pi i (n+1)t} \tilde{p}(e^{2\pi i (t+\tau)})\star \tilde{D}(e^{-2\pi i t})
\end{align}
From this, and using $D_{n}(\theta)$ to denote the centered, normalized Dirichlet kernel $\frac{1}{n+1}\sum_{k=-n/2}^{n/2}e^{2\pi i \theta}$, we get 
\begin{align}
K_p(\tau, \theta) & =  \frac{e^{-(n/2)2\pi i (\theta - \tau)}e^{2\pi i (n+1)(\theta - \tau)}}{1-e^{2\pi i (\theta - \tau)}} \int_{0}^{1} p_{\tau}(e^{2\pi i s}) \left(\int_{0}^{-(\theta - \tau)} e^{2\pi i (n+1) t} \tilde{D}(e^{-2\pi i k (s+t)})\; dt\right)\; ds \\
&+ \frac{e^{-(n/2)2\pi i (\theta - \tau)}}{e^{2\pi i (\theta - \tau)}-1}\int_{0}^1 p_{\tau}(e^{-2\pi i s}) \left(\int_{0}^{(\theta - \tau)}e^{2\pi i (n+1)t} \tilde{D}(e^{-2\pi i (s+t)}) \; dt\right)\; ds \\
&- \frac{p_0}{n+1}\sum_{k=-n/2}^{n/2} e^{2\pi i (\theta-\tau)}
\end{align}
In particular, the modulus of $|K_p(\tau, \theta)|$ can be bounded as 
\begin{align}
\left|K_p(\tau, \theta)\right| &\leq \frac{1}{|1-e^{2\pi i (\theta - \tau)}|}\left|\int_{0}^{1} p_{\tau}(e^{2\pi i s}) \int_{0}^{-(\theta - \tau)} e^{2\pi i (n+1) t} \tilde{D}(e^{-2\pi i k (s+t)})\; dt\right|\label{firstTermDecompositionModulusKp}\\
&+\frac{1}{|1-e^{2\pi i (\theta - \tau)}|}\left| \int_{0}^1 p_{\tau}(e^{-2\pi i s})\int_{0}^{(\theta - \tau)}e^{2\pi i (n+1)t} \tilde{D}(e^{-2\pi i (s+t)}) \; dt\right|\label{secondTermDecompositionModulusKp} \\
& + p_0\min\left(1, \frac{1}{(n+1)(\theta - \tau)}\right)\label{additionalContributionp0Dirichlet}
\end{align}
Note that for a polynomial of the form~\eqref{boundPolynomnialMultiSpikes}, we can always decompose the first two lines in~\eqref{firstTermDecompositionModulusKp} and~\eqref{additionalContributionp0Dirichlet} as 
\begin{align}
&\frac{1}{|1-e^{2\pi i (\theta - \tau)}|}\left|\int_{0}^{1} p_{\tau}(e^{2\pi i s}) \int_{0}^{-(\theta - \tau)} e^{2\pi i (n+1) t} \tilde{D}(e^{-2\pi i k (s+t)})\; dt\right|\label{SplittingAtomsdiscussionStartingPoint}\\
&\leq \sum_{\alpha \in S}\frac{1}{|1-e^{2\pi i (\theta - \tau)}|} \left|\int_{0}^{1} p_{\alpha, \tau}(e^{2\pi i s}) \int_{0}^{-(\theta - \tau)} e^{2\pi i (n+1) t} \tilde{D}(e^{-2\pi i k (s+t)})\; dt\right|
\end{align}
where $p_{\alpha}(e^{2\pi i \theta})$ is a polynomial bounded as 
\begin{align}
\left|p_\alpha(e^{2\pi i \theta})\right|&\leq \min\left(1, \frac{C_1}{1+n|\theta - \alpha|}\right) 
\end{align}
and consequently $p_{\alpha, \tau}(e^{2\pi i \theta}) = \sum_{|k|\leq n} p_{\alpha}[k] e^{2\pi i k \theta}e^{2\pi i k\tau}$ can be bounded as 
 \begin{align}
\left|p_\alpha(e^{2\pi i \theta})\right|&\leq \min\left(1, \frac{C_1}{1+n|\theta - (\alpha - \tau)|}\right) 
\end{align}
For~\eqref{secondTermDecompositionModulusKp}, a similar reasoning gives 
\begin{align}
&\frac{1}{|1-e^{2\pi i (\theta - \tau)}|}\left| \int_{0}^1 p_{\tau}(e^{-2\pi i s})\int_{0}^{(\theta - \tau)}e^{2\pi i (n+1)t} \tilde{D}(e^{-2\pi i (s+t)}) \; dt\right|\\
&\leq \sum_{\alpha \in S}\frac{1}{|1-e^{2\pi i (\theta - \tau)}|} \left|\int_{0}^{1} p_{\alpha, \tau}(e^{-2\pi i s}) \int_{0}^{-(\theta - \tau)} e^{2\pi i (n+1) t} \tilde{D}(e^{-2\pi i k (s+t)})\; dt\right|
\end{align}
Again, each of the polynomials $p_{\alpha}(e^{-2\pi i s})$ are bounded as 
\begin{align}
\left|p_{\alpha, \tau}(e^{2\pi i (-\theta)})\right| & \leq \min\left(1, \frac{C_1}{1+n|(-\theta) - \alpha + \tau|}\right) \\
&\leq \min\left(1, \frac{C_1}{1+n|\theta - (\tau -\alpha)|}\right)\label{SplittingAtomsdiscussionEndPoint}
\end{align}
Without loss of generality, we can let $\tau' = \tau - \alpha$ and derive bounds of the form $f(|\tau'|)$. We then apply the inverse transform to get the bounds of Lemma~\ref{lemma:boundK} below.

We will control the supremum, $\sup_{\theta} \left|K_p(\tau, \theta)\right| $, as a consequence, we can neglect the translation that appears in the integration bounds as well as on the prefactor, i.e.
\begin{align}
\left|K_p(\tau, \theta)\right| &\leq \sup_{\theta\in [0,1]} \frac{1}{|1-e^{2\pi i \theta}|}\left|\int_{0}^{1} p_{\tau}(e^{2\pi i s}) \int_{0}^{-\theta} e^{2\pi i (n+1) t} \tilde{D}(e^{-2\pi i k (s+t)})\; dt\right|\label{firstSupremum}\\
&+\sup_{\theta\in [0,1]}\frac{1}{|1-e^{2\pi i\theta}|}\left| \int_{0}^1 p_{\tau}(e^{-2\pi i s})\int_{0}^{\theta}e^{2\pi i (n+1)t} \tilde{D}(e^{-2\pi i (s+t)}) \; dt\right| \label{secondSupremumInDefintionKp}\\
& + p_0\min\left(1, \frac{1}{(n+1)\theta}\right)\label{thirdTermDefinitionKp}
\end{align}
Furthermore, note that we have 
\begin{align}
&\sup_{\theta\in [0,1]}\frac{1}{|1-e^{2\pi i \theta}|}\left|\int_{0}^{1} p_{\tau}(e^{2\pi i s}) \int_{0}^{-\theta} e^{2\pi i (n+1) t} \tilde{D}(e^{-2\pi i k (s+t)})\; dt\right| \\
&= \sup_{\theta\in [0,1]} \frac{1}{|1-e^{2\pi i -\theta}|}\left|\int_{0}^{1} p_{\tau}(e^{2\pi i s}) \int_{0}^{\theta} e^{2\pi i (n+1) t} \tilde{D}(e^{-2\pi i k (s+t)})\; dt\right|\\
& =  \sup_{\theta\in [0,1]} \frac{1}{|e^{2\pi i \theta}-1|}\left|\int_{0}^{1} p_{\tau}(e^{2\pi i s}) \int_{0}^{\theta} e^{2\pi i (n+1) t} \tilde{D}(e^{-2\pi i k (s+t)})\; dt\right|
\end{align}
Hence we can thus focus our attention on the first term in~\eqref{firstSupremum}. Lemma~\ref{lemma:boundK} which is proved in section~\ref{sectionProoflemma:boundK} controls the quantity 
\begin{align}
I(\theta; \tau) =  \frac{1}{|e^{2\pi i \theta}-1|}\left|\int_{0}^{1} p_{\tau}(e^{2\pi i s}) \int_{0}^{\theta} e^{2\pi i (n+1) t} \tilde{D}(e^{-2\pi i k (s+t)})\; dt\right|
\end{align}
For a polynomial $p_{\tau}(e^{2\pi i s})$ obeying the bound
\begin{align}
\left|p(e^{2\pi i s})\right|\leq \min\left(1, \frac{C_1}{1+n|s-\alpha|}\right)
\end{align}
For an absolute constant $C_1\in \mathbb{R}$.

\begin{restatable}{lemma}{boundK}
\label{lemma:boundK}
For any integrable function $f\; :\; S^1 \rightarrow \mathbb{C}$, if there exists $\alpha \in \mathbb{R}$, $C_1\geq 1$ such that 
\begin{align}
\forall \theta\in \mathbb{R}, \quad f(e^{2\pi i \theta}) \leq \min(1, \frac{C_1}{1+n|\theta - \alpha|_{mod}}),\label{boundOnFLemmaIntegralTotal}
\end{align}
Then for any $\tau,\tau'$, 
\begin{flalign}
|K_f(\tau,\theta)|&\leq M_1 + M_2\log(C_1) && \text{\upshape if $|\tau - \alpha|_{mod} \leq \frac{C_1}{n}$}\label{constantBoundlemmaK2}\\
&\leq \frac{M_1'}{n|\tau - \alpha|} + \frac{M_2'C_1}{n|\tau - \alpha|} \log(\frac{n|\tau - \alpha|}{C_1}) + \frac{M_3'}{n|\tau - \alpha|}+ \frac{M_4' C_1}{n|\tau - \alpha|} \log^2(|\tau - \alpha|)&& \text{\upshape otherwise}
\label{decreasingBoundLemmaK3}. 
\end{flalign}
\end{restatable}
As the bound~\eqref{boundOnFLemmaIntegralTotal} is the only element in~\eqref{decreasingBoundLemmaK3} that depends on $\tau$ and $\alpha$, without loss of generality we will assume $\tau = 0$. 

In particular, note that lemma~\eqref{lemma:boundK} together with the discussion~\eqref{SplittingAtomsdiscussionStartingPoint} to~\eqref{SplittingAtomsdiscussionEndPoint} imply that for any polynomial bounded as in~\eqref{boundPolynomnialMultiSpikes},
\begin{align}
\left|K_p(\tau, \theta)\right|&\leq M_1 + M_2\log(C_1) +2 \sum_{|\alpha - \tau|>\Delta} \frac{M_1'}{n|\tau-\alpha|} + \frac{M_2' C_1}{n|\tau - \alpha|}\log(\frac{n|\tau - \alpha|}{C_1})\\
& + \frac{M_3'}{n|\tau - \alpha|} \log(n|\tau - \alpha|) + \frac{M_4' C_1}{n|\tau - \alpha|}\log^2(|\tau - \alpha|) + \min(1, \frac{1}{|\theta - \tau|}). 
\end{align}

In particular, if we take $\Delta$ sufficiently large, this lemma implies 
\begin{align}
\left|K_p(\tau, \theta)\right|&\leq M_1 + M_2\log(C_1) + 8+ \min\left(1, \frac{1}{n|\theta - \tau|}\right) \leq (M_1+9) + M_2\log(C_1) \leq M_1''' + M_2\log(C_1). 
\end{align}

%
%
%
Substituting this into~\eqref{totalBoundOnModulusofMpTighterWithEpsilon} 
and using $\sum_{k\geq 1} \|I - \frac{1}{2n+1}U^*U\|_\infty^k \leq \frac{1}{1 - \frac{\log|S|}{(2n+1)\Delta}}\leq 2$ as soon as $\Delta \gtrsim \lambda_c \log|S|$, we get 
\begin{align}
\left|M[p](e^{2\pi i \theta})\right|& \leq (3+3\varepsilon)\sum_{\tau\in S} \left|D_n(\theta - \tau)\right| \left(M_1''' + M_2\log(C_1)\right)
%
\end{align}
as soon as $\Delta \geq 2\lambda_c \log|S|$. Now using $|p(e^{2\pi i \theta})| \leq \min(1, \sum_{\alpha\in S}\frac{C_1}{1+n|\theta - \alpha|})$, as well as $\left|D_n(\theta - \alpha) \right|\leq \frac{2}{1+n|\theta - \alpha|}$, we have 
\begin{align}
|\lambda||p(e^{2\pi i \theta})| = \left|M[p](e^{2\pi i \theta})\right|  \leq (3+3\varepsilon)\sum_{\alpha\in S}\frac{1}{1+n|\theta - \alpha|} \left(M_1'''+M_2\log(C_1)+3\right). 
\end{align}
which implies 
\begin{align}
|p(e^{2\pi i \theta})|\leq \frac{(3+3\varepsilon)}{|\lambda|}\sum_{\alpha\in S} \frac{1}{1+n|\theta - \alpha|} \left(M_1''' + M_2\log(C_1)\right). 
\end{align}
 and hence 
\begin{align}
p(e^{2\pi i \theta})\leq \min\left(1, \frac{(3+3\varepsilon)}{|\lambda|} \sum_{\alpha\in S} \frac{M_1'''+M_2\log(C_1)}{1+n|\theta - \alpha|}\right)
\end{align}
By assumption, since $C_1$ is the smallest constant such that $|p(e^{2\pi i \theta})|\leq \min\left(1, \sum_{\alpha\in S}\frac{C_1}{1+n|\theta - \alpha|}\right)$, we must then have 
\begin{align}
C_1 \leq \frac{(3+3\varepsilon)}{|\lambda|} \left(M_1''' + M_2\log(C_1)\right)
\end{align}

\begin{figure}
\centering
\input{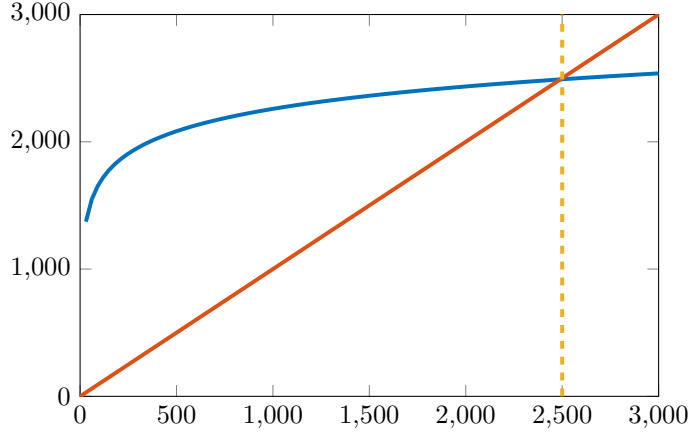}
\caption{$f_1(C_1) = C_1$ and $f_2(C_1) = \frac{(3+3\varepsilon)}{|\lambda|}(M_1''' + M_2\log(C_1))$ for $\varepsilon = .001$}
\end{figure}

Now we let $a_1\equiv 1$, $a_2\equiv -\frac{24M_2}{|\lambda|}$ and $a_3\equiv -\frac{24(M_1+3)}{|\lambda|}$, the solutions of $a_1x + a_2\log(x) + a_3$ can be found through the Lambert W function. Note that we have $r_1\equiv \frac{a_2}{a_1} = \frac{-24M_2}{|\lambda|}$, as well as 
\begin{align}
r_2 &\equiv -\frac{a_3}{a_2} = \left(\frac{(3+3\varepsilon)M_1}{|\lambda|} + \frac{72}{|\lambda|}\right)\frac{|\lambda|}{-(3+3\varepsilon)M_2} = -\left(\frac{M_1}{M_2} + \frac{3}{M_2}\right)
\end{align}
and finally, $r_3\equiv \frac{a_1}{a_2} = \frac{-|\lambda|}{(3+3\varepsilon)M_2}$. From this, we get the solutions
\begin{align}
x_0 &= r_1W_0\left(e^{r2}r_3\right)\\
x_{-1}& = r_1W_{-1}\left(e^{r2}r_3\right).
\end{align}
$W_0$ and $W_{-1}$ are used to respectively denote the $W(x)\geq -1$ and $W(x)\leq -1$ branches of the Lambert function. 

To conclude on the value of $C_1$, we substitute in the expression above the values that are derived for $M_1''''$ and $M_2$ from the proof of lemma~\ref{lemma:boundK} and in particular~\eqref{boundNearKp} (see appendix~\ref{sectionProoflemma:boundK}), $M_1''' = 152$, $M_2 = 76$, as well as our assumed value for $\lambda = 0.9$.  we get 
\begin{align}
r_1 &= \frac{-(3+3\varepsilon)\times 76}{0.9} = -2026\\
r_2& = -\left(\frac{152}{76} + \frac{3}{76}\right) = 2.0395\\
r_3 &= - \frac{-0.9}{(3+3\varepsilon)\times 152} = -2.4671e-04
\end{align}
which gives 
$x_0 = 0.1354$ and $x_{-1} = 2496.7$. In particular, we thus get the condition $C_1 \leq 2500$.

\subsection{\label{sec:truncation}Truncated operator}

\begin{restatable}{lemma}{lemmaTroncatureMultiSpikesKey}
\label{lemma:lemmaTroncatureMultiSpikes}
 Let $\lambda\in[0,9;1[$ and let $p$ to denote a polynomial with $||p||_\infty=1$ and such that
  \begin{equation*}
    \mathcal{M}(p)=\lambda p.
  \end{equation*}
  Let $\theta_0$ to denote the point where $|p(e^{2\pi i\theta_0})|=1$. From the result of lemma~\ref{lemmaBoundPmultiAtomic}, we can always assume that $\theta_0$ belongs to an interval of the form $\left[\tau_s-\frac{C_1}{n};\tau_s+\frac{C_1}{n}\right]$. Without loss of generality we can take $s=0$ and consider a configuration in which $\tau_=0$.

For every integer $K\in\{1,\dots,n\}$, we define $p_K$ and $p_{K,err}$ as follows. We decompose $p$ on the Dirichlet basis (letting $D_0(\theta)= \frac{1}{2n+1}\sum_{s=-n}^n e^{2\pi i s\theta}$) :
  \begin{equation}\label{decompositionpDirichlet}
    \forall \theta\in\mathbb{R},\quad
    p(e^{2\pi i\theta}) = \sum_{k=-n}^n p(e^{2\pi i \frac{k}{2n+1}}) D_0\left(\theta-\frac{k}{2n+1}\right)
  \end{equation}
  and we define for any integer $K\in\{1,\dots,n\}$, the approximations
  \begin{align}
    p_K(e^{2\pi i\theta}) &= \sum_{k=-K}^K p(e^{2\pi i \frac{k}{2n+1}}) D_0\left(\theta-\frac{k}{2n+1}\right);\\
    p_{K,err}(e^{2\pi i\theta}) &= \underset{|k|>K}{\sum_{k=-n}^n} p(e^{2\pi i \frac{k}{2n+1}}) D_0\left(\theta-\frac{k}{2n+1}\right).
  \end{align}
  As soon as $\Delta\gtrsim \log^2(S)\lambda_c$, those polynomials satisfy the following properties (where $c_1,\dots,c_4$ are absolute constants, that do not depend on $n$ nor on $K$). 
  \begin{enumerate}
  \item\label{lemmaTroncatureItem1} $p_K+p_{K,err}=p$.
  \item\label{lemmaTroncatureItem2} For all $\theta\in\left[-\frac{K}{n};\frac{K}{n}\right], |p_{K,err}(e^{2\pi i\theta})|\leq c_1 \frac{\log(1+K)}{1+K}$.
  \item\label{lemmaTroncatureItem3} For all $\theta\in[-0,5;0,5]-\left[-\frac{K}{n};\frac{K}{n}\right]$, $|p_K(e^{2\pi i\theta})| \leq c_1' \frac{\log(1+K)}{1+n|\theta|}$.
  \item\label{lemmaTroncatureItem4} For all $\theta\in[-0,5;0,5]$, $|p_{K,err}(e^{2\pi i\theta})|\leq c_2 \frac{\log(1+K)}{1+K}\min\left(1,\frac{1+K}{1+n |\theta - \tau_0|}\right) + c_2' \min(1, \sum_{\tau\neq \tau_0} \frac{1}{1+n|\theta - \tau|})$.
  \item\label{lemmaTroncatureItem5} For all $\theta\in[-\frac{K}{n};\frac{K}{n}]$, $|\mathcal{M}_1(p_{K,err})|\leq c_3\frac{\log^2(1+K)}{1+K}\frac{1}{1+n|\theta|}$.
  \item\label{lemmaTroncatureItem6} For all $\theta\in[-\frac{K}{n};\frac{K}{n}]$, $|(\mathcal{M}-\mathcal{M}_1)(p)|\leq \frac{c_4}{1+\Delta n}$.
\item\label{lemmaTroncatureItem7} There exists $k\in\{-K,\dots,K\}$ with $|p(e^{2\pi i\frac{k}{2n+1}})| \geq c_5$.
  \end{enumerate}
\end{restatable}

\begin{proof}

We start with item~\ref{lemmaTroncatureItem2}. As before, we write 
\begin{align}
|p_{K, \err}(e^{2\pi i \theta})| &\leq \left|\sum_{k\notin [-K, K]} p(e^{2\pi i k \theta}) D_0(\theta - \frac{k}{2n+1})\right|\\
&\leq \sum_{k\notin [-K, K]} \min(1, \sum_{\tau\in S}\frac{C_1}{1+n|\theta - \tau|})\left|D_0(\theta - \frac{k}{2n+1})\right|\\
&\leq \sum_{\tau\in S}\sum_{k\notin [-K, K]} \min(1, \frac{C_1}{1+n|\theta - \tau|})\left|D_0(\theta - \frac{k}{2n+1})\right|\label{sumWithAllpTau}
\end{align}
We now let $p_{K,\err}^\tau(e^{2\pi i \theta})$ to denote each of the terms in the sum~\eqref{sumWithAllpTau}. I.e.
\begin{align}
p_{K,\err}^{(\tau)}(e^{2\pi i \theta})&\equiv \sum_{k\notin [-K, K]} \min(1, \frac{C_1}{1+n|\theta - \tau|})\left|D_0(\theta - \frac{k}{2n+1})\right|
\end{align}
We start by controling the contributions arising from the atoms $\tau\neq \tau_0$. As we consider truncating the operator within the interval $[\tau_0 - \Delta/2, \tau_0+\Delta/2]$, we can assume $\frac{K}{2n+1} < \Delta/2$. We start by proving item~\ref{lemmaTroncatureItem2}. By symmetry, we can focus on the indices $K/n< k\leq n$. We let $k_\tau$ to denote the nearest integer to $(2n+1)\tau$.  
For any atom located on the right of $K/n < \Delta/2$, and any $\theta<K/(2n+1)$, the contribution to $p(e^{2\pi i \frac{k}{2n+1}})$ can read as 
\begin{align}
\left|p_{K,\err}^{(\tau)}(\theta)\right| &\leq  \sum_{k_\tau - C_1}^{k_\tau + C_1} \frac{1}{1+n(\frac{k}{2n+1} - \theta)} + \sum_{k = k_\tau+C_1}^{n} \frac{C_1}{1+n(\frac{k}{2n+1} - \tau)} \frac{1}{1+n(\frac{k}{2n+1} - \theta)} \\
&+ \sum_{k=K}^{k_\tau - C_1} \frac{C_1}{1+n(\tau - \frac{k}{2n+1})} \frac{1}{1+n(\frac{k}{2n+1} - \theta)}
\end{align}
We label each of those terms as $H_0^{(\tau)}(\theta), H_1^{(\tau)}(\theta)$ and $H_2^{(\tau)}(\theta)$. For $H_0^{(\tau)}(\theta)$, if we assume that $K/n<\Delta/2$, we always have $|\tau - \theta|\geq \Delta/2$. Hence we can write 
\begin{align}
\left|H_0^{(\tau)}(\theta)\right| &\leq \frac{2C_1}{1+n|\tau - \theta - \frac{C_1}{n}|}\leq \frac{8C_1}{4+n\Delta}
\end{align}
The second inequality follows from $\Delta/4\geq \frac{C_1}{n}$ which can always be assumed for $n$ large enough. For $H_1^{(\tau)}$, we write 
\begin{align}
\left|H_1^{(\tau)}(\theta)\right|&\leq \frac{C_1}{1+n|\tau - \theta - \frac{C_1}{n}|}  +\frac{C_1}{n|\tau - \theta|} \int_{k_\tau + C_1}^n \left(\frac{1}{1+n (\frac{k}{2n+1} - \tau)}  - \frac{1}{1+n(\frac{k}{2n+1} - \theta)}\right) \; dk\\
&\leq \frac{C_1}{1+n|\tau - \theta - \frac{C_1}{n}|}  +\frac{C_1}{1+n|\tau - \theta|}\left|\log(\frac{1/n + \frac{n}{2n+1} - \tau}{1/n + C_1/n}) - \log(\frac{1/n + \frac{n}{2n+1} - \theta}{1/n + \tau - \theta + C_1/n})\right|\\
&\leq \frac{C_1}{1+n|\tau - \theta - \frac{C_1}{n}|}  +\frac{C_1}{n|\tau - \theta|} \left[\log(1+ \frac{\tau - \theta}{1/n + \frac{n}{2n+1} - \tau}) + \log(1+ \frac{\tau - \theta}{1/n + C_1/n})\right]
\end{align}
Finally for $H_2^{(\tau)}(\theta)$, a similar reasoning gives 
\begin{align}
\left|H_2^{(\tau)}(\theta) \right|&\leq \sum_{k=K}^{k_\tau - C_1} \frac{C_1}{1+n(\tau - \frac{k}{2n+1})} \frac{1}{1+n(\frac{k}{2n+1} - \theta)} \\
& = C_1 \sum_{k=K}^{k_\tau - C_1} \left(\frac{1}{1+n(\tau - \frac{k}{2n+1})} + \frac{1}{1+n (\frac{k}{2n+1} - \theta)}\right) \frac{1}{2+n|\theta - \tau|}\\
&\leq \frac{1}{1+C_1} \frac{C_1}{n|\tau - \theta| + 2} + \frac{C_1}{n|\tau - \theta|+2} \int_{k_\tau-C_1}^K \frac{1}{1+n(\tau - \frac{k}{2n+1})}\; dk\\
&+ \frac{C_1}{n|\tau - \theta|+2} + \frac{C_1}{n|\tau - \theta|+2}\int_{K}^{k_\tau - C_1} \frac{1}{1+n|\frac{k}{2n+1} - \theta|} \; dk\\
&\lesssim \frac{1}{n|\tau - \theta|+2} + \frac{C_1}{n|\tau - \theta|+2} \frac{2n+1}{n}\log(1+ \frac{\tau - \frac{K}{2n+1} - \frac{C_1}{2n+1}}{\frac{1}{n}+ \frac{C_1}{2n+1}})\\
&+ \frac{C_1}{n|\tau - \theta|+2} + \frac{C_1}{n|\tau - \theta|+2} \frac{2n+1}{n} \log(1+ \frac{\tau - \frac{C_1}{2n+1} - \frac{K}{2n+1}}{\frac{1}{n} + \left|\frac{K}{2n+1} - \theta\right|})\\
&\lesssim \frac{1}{n|\tau - \theta|+2} + \frac{3C_1}{n|\tau - \theta|+2} \left(\log(n|\tau - \theta|) + \log(2)\right)\\
&+ \frac{C_1}{n|\tau - \theta|+2} + \frac{3C_1}{n|\tau - \theta|+2} \left(\log(n|\tau - \theta|)+ \log(2)\right)
\end{align}
Summing over the atoms, we get 
\begin{align}
\sum_{\tau=1}^S \sum_{\ell=0}^2 \left|H_\ell^{(\tau)}(\theta)\right| &\lesssim \frac{8C_1}{1+n\Delta} + \frac{2C_1}{n\Delta}\log |S| +\frac{4C_1}{n\Delta} +  \frac{2C_1}{n\Delta}\log|S| + \sum_{\tau\in S} \frac{3C_1}{n|\tau - \theta|} \left( \log(n|\tau - \theta|)+\log(2)\right) \\
&+ \frac{8C_1}{n\Delta} + \frac{2C_1}{n\Delta}\log|S| + \sum_{\tau\in S} \frac{3C_1}{n|\tau - \theta|} \left[\log(n|\tau  - \theta|) + \log(K) + \log(2)\right]\\
&\stackrel{(a)}{\lesssim }\frac{20C_1}{n\Delta} + \frac{C_1}{n\Delta}\log|S| \left(6+12\log(2)+6\log(K)\right) + \frac{12C_1}{n\Delta}\log|S|\log(n\Delta)\label{boundOnPKerr}
\end{align}
In $(a)$ we use the fact that the function $\log(x)/x$ is decreasing as soon as $x\geq 10$. We can thus always bound $\frac{\log(n|\tau - \theta|)}{n|\tau - \theta|}$ by $\log(\ell \Delta n)/(\ell n\Delta)$. The last bound~\eqref{boundOnPKerr} can be made smaller than $\varepsilon$ as soon as $\Delta \gtrsim \lambda_c \log|S|^2 \varepsilon^{-1}$. In particular, taking $\varepsilon = c_1\frac{1+K}{\log(1+K)}$ for any sufficiently small constant $c_1$ gives the item. Most of the contribution is arising from the atom $\tau_0$ used to truncate the operator. The contribution from this last atom can be bounded as 
\begin{align}
p_{K,\err}^{(\tau_0)}(\theta) & = \sum_{|k|\geq K} \frac{C_1}{1+n\left|\frac{k}{2n+1}\right|} \frac{1}{1+n\left|\theta - \frac{k}{2n+1}\right|}\label{boundpKerrJenaimArreatom0}\\
&= \sum_{k\geq K} \frac{C_1}{n\theta} \left(\frac{1}{1+n\frac{k}{2n+1}} - \frac{1}{1+n\left(\frac{k}{2n+1} - \theta\right)}\right)\\
&+ \sum_{k\leq -K} \frac{C_1}{n\theta} \left(\frac{1}{1+n(-\frac{k}{2n+1})} - \frac{1}{1+n(\theta - \frac{k}{2n+1})}\right)
\end{align}
we focus on the indices $k\geq K$ (the case $k\leq -K$ is identical). For those indices,we can write 
\begin{align}
&\sum_{k\geq K} \frac{C_1}{n\theta} \left(\frac{1}{1+n\frac{k}{2n+1}} - \frac{1}{1+n\left(\frac{k}{2n+1} - \theta\right)}\right) \\
&\leq \frac{C_1}{n|\theta|} \left(\frac{2}{2+K} + \frac{2n+1}{n}\int_{K}^n \frac{n/(2n+1)}{1+n\frac{k}{2n+1}}\; dk - \frac{2n+1}{n} \int_{K}^n \frac{n/(2n+1)}{1+n \left(\frac{k}{2n+1} - \theta\right)}\; dk\right)\\
&\leq \frac{C_1}{n|\theta|}\left(\frac{2}{2+K} 2\log(1+ \frac{\theta}{1/n+1/2-\theta}\vee \frac{-\theta}{1/2+1/n}) + 2\log(1+ \frac{\theta}{1/n + \frac{K}{2n+1} - \theta}\vee \frac{-\theta}{1/n + \frac{K}{2n+1}})\right)\label{lastLinejenaimarre}
\end{align}
To bound the last line, we make the distinction between the case $\theta\leq \frac{K}{2n+1}\frac{1}{2}$ in which we have 
\begin{align}
\eqref{lastLinejenaimarre}&\leq \frac{2C_1}{1+K} + \frac{8C_1}{n} + \frac{4C_1}{K}
\end{align}
and the case $\theta >\frac{K}{2n+1}\frac{1}{2}$, in which we have 
\begin{align}
\eqref{lastLinejenaimarre}&\leq \frac{2C_1}{2+K} + \frac{8C_1}{n} + \frac{4C_1}{K}\log(K)\label{boundJenAImarre22}
\end{align}
The total bound on the modulus $|p_{K,\err}^{(\tau_0)}(\theta)|$ can thus be obtained by multiplying~\eqref{boundJenAImarre22} by two,
\begin{align}
\eqref{boundpKerrJenaimArreatom0}&\leq \frac{4C_1}{2+K} + \frac{16C_1}{n} + \frac{8C_1}{K}\log(K)\lesssim  \frac{4C_1}{2+K} + \frac{8C_1}{K}\log(K)
\end{align}
Finally note that when $|\theta|<1/n$ we simply write 
\begin{align}
\left|p_{K, \err}^{(\tau_0)}(\theta)\right|&\leq \sum_{|k|\geq K} \frac{C_1}{1+n|\frac{k}{2n+1}|} \frac{1}{1+n\left|\frac{k}{2n+1} - \frac{1}{n}\right|}\\
&\leq \sum_{|k|\geq K} C_1\frac{1}{1+n\left|\frac{k}{2n+1} - \frac{1}{n}\right|}\\
&\leq \sum_{|k|\geq K} C_1 \left|\frac{2}{k}\right|\leq \frac{4C_1}{K}.
\end{align}
This concludes the proof for item~\ref{lemmaTroncatureItem2}, and gives the value of the constant $c_1\equiv \frac{4C_1}{\log(K)} + 8C_1$.

For item~\ref{lemmaTroncatureItem2}, note that for $\theta > K/n$, we have 
\begin{align}
\left|p_K(e^{2\pi i \theta})\right|  &\leq \sum_{k=-C_1}^{C_1} \frac{1}{1+n|\theta - \frac{k}{2n+1}|}\label{term1Item3LemmaConstant200001}\\
&+ \sum_{k=C_1}^K \sum_{\tau\in S} \frac{1}{1+n |\tau - \frac{k}{2n+1}|} \frac{1}{1+n \left|\theta - \frac{k}{2n+1}\right|} \label{term2Item3LemmaConstant200002}\\
&+ \sum_{k=-K}^{-C_1}\sum_{\tau\in S} \frac{1}{1 + n|\tau - \frac{k}{2n+1}|} \frac{1}{1+n \left|\theta - \frac{k}{2n+1}\right|}\label{term3Item3LemmaConstant200003}
\end{align}
For the first term, from $\theta>K/n>2C_1/n$, we can write 
\begin{align}
\eqref{term1Item3LemmaConstant200001} \leq \frac{2C_1}{1+n|\theta - \frac{C_1}{2n+1}|}\leq \frac{4C_1}{2+n|\theta|}\label{boundItem31}
\end{align}
For~\eqref{term2Item3LemmaConstant200002} and~\eqref{term3Item3LemmaConstant200003}, we have
\begin{align}
\eqref{term2Item3LemmaConstant200002} & \leq \sum_{\tau\in S\setminus \tau_0} \sum_{k=C_1}^K \frac{C_1}{n|\tau - \theta|} \left(\frac{1}{1+n(\tau - \frac{k}{2n+1})} - \frac{1}{1+n(\theta - \frac{k}{2n+1})}\right)\label{boundPK1a}\\
&+ \sum_{k=C_1}^K \frac{C_1}{n|\theta|} \left(\frac{1}{1+n\frac{k}{2n+1}} + \frac{1}{1+n(\theta - \frac{k}{2n+1})}\right)\label{boundPK1b}
\end{align}
For each of those two terms, one can write 
\begin{align}
\eqref{boundPK1a}&\leq \sum_{k=C_1}^K \frac{C_1}{n|\theta| + 2} \left[\frac{1}{1+\frac{n}{2n+1}C_1} + \int_{C_1}^K \frac{1}{1+n \frac{k}{2n+1}}\; dk\right]\\
&+ \sum_{k=C_1}^K \frac{C_1}{n|\theta|+2} \left[\frac{1}{1+n\left(\theta - \frac{k}{2n+1}\right)} + \int_{C_1}^K \frac{1}{1+n \left(\theta - \frac{k}{2n+1}\right)}\; dk\right]\\
&\leq  \frac{C_1}{n|\theta|+2} \left[\frac{3}{3+C_1} + 3\log(1+ \frac{K-C_1}{C_1})\right]\\
&+ \frac{C_1}{n|\theta|+2} \left[1+ 3\log(1+ \frac{\theta - \frac{C_1}{2n+1}}{1/n + \theta - \frac{K}{2n+1}})\right]\\
&\leq \frac{C_1}{n|\theta|+2} \left[2+3\log(\frac{K}{C_1}) + 3\log(K)\right].\label{finalBoundPositiveKZeroAtom}
\end{align}
as well as 
\begin{align}
\eqref{boundPK1b}&\leq\sum_{\tau>0}\sum_{k=C_1}^K \frac{C_1}{1+n\left(\tau - \frac{k}{2n+1}\right)} \frac{1}{1+n\left(\theta - \frac{k}{2n+1}\right)}\\
&+ \sum_{\tau<0}\sum_{k=C_1}^{K} \frac{C_1}{1+n\left(\frac{k}{2n+1} - \tau\right)} \frac{1}{1+n(\theta - \frac{k}{2n+1})}\\
&\leq \sum_{\tau>0}\sum_{k=C_1}^{K} \frac{C_1}{n|\theta - \tau|} \left(\frac{1}{1+n(\tau - \frac{k}{2n+1})} - \frac{1}{1+n(\theta - \frac{k}{2n+1})}\right)\\
&+ \sum_{\tau<0}\sum_{k=C_1}^{K} \frac{C_1}{n|\theta - \tau|} \left(\frac{1}{1+n(\frac{k}{2n+1} - \tau)} + \frac{1}{1+n(\theta - \frac{k}{2n+1})}\right)\\
&\leq \sum_{\tau<0} \frac{C_1}{n|\theta - \tau|} \left(\frac{1}{1+n(\tau - \frac{K}{2n+1})} + \int_{C_1}^K \frac{1}{1+n(\tau - \frac{k}{2n+1})}\; dk\right)\\
&+ \sum_{\tau>0} \frac{C_1}{n|\theta - \tau|} \left(\frac{1}{1+n(\theta - \frac{K}{2n+1})} + \int_{C_1}^K \frac{1}{1+n(\theta - \frac{k}{2n+1})}\; dk\right)\\
&+\sum_{\tau<0} \frac{C_1}{n|\theta - \tau|} \left(\frac{1}{1+n(\frac{C_1}{2n+1}-\tau)} + \int_{C_1}^K \frac{1}{1+n(\frac{k}{2n+1} - \tau)}\; dk\right)\\
&+\sum_{\tau<0} \frac{C_1}{n|\theta - \tau|} \left(\frac{1}{1+n(\theta - \frac{K}{2n+1})} + \int_{C_1}^K \frac{1}{1+n(\theta - \frac{k}{2n+1})}\; dk\right)\\
&\leq \sum_{\tau>0} \frac{C_1}{n|\theta - \tau|} \left(\frac{2}{n\Delta} + \frac{2n+1}{n} \log(\frac{1/n + \tau - \frac{K}{2n+1}}{1/n+\tau - \frac{C_1}{2n+1}})\right)\\
& + \sum_{\tau>0} \frac{C_1}{n|\theta - \tau|}\left(1 + \frac{2n+1}{n} \log(\frac{1/n + \theta - \frac{K}{2n+1}}{1/n + \theta - \frac{C_1}{2n+1}})\right)\\
&+\sum_{\tau<0}\frac{C_1}{n|\theta - \tau|} \left(\frac{2}{n\Delta} + \frac{2n+1}{n} \log(\frac{1/n + \frac{K}{2n+1} - \tau}{1/n + \frac{C_1}{2n+1}-\tau})\right)\\
&+ \sum_{\tau<0} \frac{C_1}{n|\theta - \tau|} \left(1+ \frac{2n+1}{n} \log(\frac{1/n + \theta - \frac{K}{2n+1}}{1/n + \theta - \frac{C_1}{2n+1}})\right)\\
& \stackrel{(a)}{\leq } \sum_{\tau>0} \frac{2C_1}{n|\theta - \tau|+1}\left(\frac{2}{n\Delta}+1 + 3\log(1+ \frac{K-C_1}{2n+1}\frac{1}{1/n + \tau - \frac{K}{2n+1}}) + 3\log(1+\frac{K-C_1}{2n+1}\frac{1}{1/n + \theta - \frac{K}{2n+1}})\right)\\
&+\sum_{\tau<0}\frac{2C_1}{n|\theta - \tau|+1} \left(\frac{2}{n\Delta} + 1+3\log(1+ \frac{K-C_1}{2n+1} \frac{1}{1/n+\tau - \frac{K}{2n+1}}) + 3\log(1+\frac{K-C_1}{2n+1}\frac{1}{1/n+\theta - \frac{K}{2n+1}})\right)\\
&\leq \sum_{\tau\in S\setminus \tau_0} \frac{2C_1}{n|\theta - \tau|+1} \left(\frac{2}{n\Delta}+1 +6\log 2\right).\label{finalBoundPositiveKNonZeroAtom}
\end{align}
In $(a)$ we use $|\theta - \tau|\geq 1/n$. 
The case $|\theta-\tau|<\frac{1}{n}$ only happens for the atom closest to $\theta$. Moreover since $|\tau - \theta|<\frac{1}{n}$, we necessarily have $\theta>\Delta/2$. In this case we thus write  
\begin{align}
\sum_{k=C_1}^{K}\frac{1}{1+n\left(\tau - \frac{k}{2n+1}\right)}\frac{1}{1+n\left(\theta - \frac{k}{2n+1}\right)}
&\leq \sum_{k=C_1}^{K} \frac{C_1}{1+n(\theta - \frac{1}{n} - \frac{k}{2n+1})} \frac{1}{1+n(\theta - \frac{k}{2n+1})}\\
&\leq \sum_{k=C_1}^K \frac{C_1}{(1+n(\theta - \frac{1}{n} - \frac{k}{2n+1}))^2}\\
&\leq \frac{C_1}{(1+n(\theta - \frac{1}{n} - \frac{K}{2n+1}))^2} + \int_{C_1}^K \frac{C_1}{(1+n(\theta - \frac{1}{n} - \frac{k}{2n+1}))^2}\; dk\\
&\leq  \frac{2}{n\Delta} \frac{2C_1}{n|\theta| +2} + \frac{3C_1}{1+n\left(\theta - \frac{1}{n} - \frac{K}{2n+1}\right)} + \frac{3C_1}{1+n\left(\theta - \frac{1}{n} - \frac{C_1}{n}\right)}\label{lastLineTauEqualTheta1}
\end{align}
To control~\eqref{lastLineTauEqualTheta1}, we use $|\theta - \tau|<1/n$ as well as $|\tau|>\Delta$ and require $\Delta - \frac{2}{n} - \frac{K}{2n+1}<(\Delta+1/n)(0.9)$ which always holds for $\Delta$ and $n$ sufficiently large.
\begin{align}
\sum_{k=C_1}^{K}\frac{1}{1+n\left(\tau - \frac{k}{2n+1}\right)}\frac{1}{1+n\left(\theta - \frac{k}{2n+1}\right)}&\leq \frac{2}{n\Delta} \frac{2C_1}{n|\theta|+1} + \frac{6C_1}{(0.9)(1+n|\theta|)}\\
&\leq \frac{2}{n\Delta} \frac{2C_1}{n|\theta|+1} + \frac{6.7 C_1}{1+n|\theta|}\label{finalBoundThetaCloseToTauPositiveeK}
\end{align}
Grouping~\eqref{finalBoundThetaCloseToTauPositiveeK} as well as~\eqref{finalBoundPositiveKNonZeroAtom} and~\eqref{finalBoundPositiveKZeroAtom}, we get 
\begin{align}
\eqref{term2Item3LemmaConstant200002}& \leq  \frac{C_1}{n|\theta|+1} \left(8.7 + 3\log(\frac{K}{C_1}) + 3\log(K)\right) + \sum_{\tau\neq 0} \frac{2C_1}{n|\theta - \tau|+1} \left(2+6\log(2)\right).\label{finalBoundKpositive}
\end{align}
A similar reasoning holds for~\eqref{term3Item3LemmaConstant200003}. We have 
\begin{align}
\eqref{term3Item3LemmaConstant200003} & \leq \sum_{k=-K}^{-C_1} \frac{C_1}{1+n|\frac{k}{2n+1}|} \frac{1}{1+n\left|\theta - \frac{k}{2n+1}\right|}\label{thirdSumBoundPKThetaLargerThanK1}\\
&+\sum_{k=-K}^{-C_1} \sum_{\tau\in S\setminus \tau_0} \frac{C_1}{1+n|\tau - \frac{k}{2n+1}|} \frac{1}{1+n|\theta - \frac{k}{2n+1}|}\label{thirdSumBoundPKThetaLargerThanK2}
\end{align}
\begin{align}
\eqref{thirdSumBoundPKThetaLargerThanK1}&\leq \sum_{k=-K}^{-C_1} \left(\frac{1}{1-n\frac{k}{2n+1}} - \frac{1}{1+n\left(\theta - \frac{k}{2n+1}\right)}\right) \frac{C_1}{n|\theta|}\\
&\leq \left[\frac{1}{1+\frac{C_1}{2n+1}n} + \int_{C_1}^{K} \frac{1}{1+ n\frac{k}{2n+1}} \; dk\right]\frac{C_1}{n\theta}\\
&+ \left[\frac{1}{1+n(\theta + \frac{C_1}{2n+1})} - \int_{C_1}^{K} \frac{1}{1+ n\frac{k}{2n+1}}\; dk\right]\frac{C_1}{n|\theta|}\\
&\leq \left[\frac{3}{3+C_1} + \frac{2n+1}{n} \log(\frac{1/n + \frac{K}{2n+1}}{1/n + \frac{C_1}{2n+1}})\right] \frac{2C_1}{n|\theta|+1}\\
&+ \left[\frac{1}{1+K} + \frac{2n+1}{n}\log(\frac{\frac{1}{n} + \frac{K}{2n+1}}{\frac{1}{n} + \frac{C_1}{2n+1}})\right] \frac{2C_1}{n|\theta|+1}\\
&\leq \left[\frac{3}{3+C_1} + \frac{1}{1+K} + 6\log(\frac{K}{C_1})\right]\frac{2C_1}{n|\theta|+1}.\label{boundOnSpikeZeroNegativeK}
\end{align}
\begin{align}
\eqref{thirdSumBoundPKThetaLargerThanK2}&\leq \sum_{k=-K}^{-C_1} \sum_{\tau>0} \frac{C_1}{1+n\left(\tau - \frac{k}{2n+1}\right)} \frac{1}{1+n(\theta - \frac{k}{2n+1})}\\
&+ \sum_{\tau<0}\sum_{k=-K}^{-C_1} \frac{C_1}{1+n(\frac{k}{2n+1} - \tau)} \frac{1}{1+n(\theta - \frac{k}{2n+1})}\\
&\leq \sum_{k=-K}^{-C_1}\sum_{\tau>0} \left[\frac{1}{1+n(\tau - \frac{k}{2n+1})} - \frac{1}{1+n(\theta - \frac{k}{2n+1})}\right] \frac{C_1}{n|\theta - \tau|}\label{pKboundMultiAtom1Term3}\\
&+ \sum_{k=-K}^{-C_1}\sum_{\tau<0} \left[\frac{1}{1+n(\frac{k}{2n+1} - \tau)} + \frac{1}{1+n(\theta - \frac{k}{2n+1})}\right]\frac{C_1}{n|\theta - \tau|}.\label{pKboundMultiAtom2Term3}
\end{align}
For each of those terms, we have 
\begin{align}
\eqref{pKboundMultiAtom1Term3}&\leq \sum_{\tau>0} \left[\frac{1}{1+n\left(\tau + \frac{C_1}{2n+1}\right)} + \int_{C_1}^{K} \frac{1}{1+n\left(\tau + \frac{k}{2n+1}\right)}\; dk\right]\\
&+ \sum_{\tau>0} \left[\frac{1}{1+n(\theta + \frac{C_1}{1+n\left(\theta + \frac{C_1}{2n+1}\right)})} + \int_{C_1}^K\frac{1}{1+n(\theta + \frac{k}{2n+1})}\; dk\right]\frac{C_1}{n|\theta - \tau|}\\
&\leq \sum_{\tau>0} \frac{C_1}{n|\theta - \tau|} \left[\frac{1}{1+K} + \frac{2n+1}{n} \log(\frac{1/n + \tau + \frac{C_1}{2n+1}}{1/n + \tau+\frac{K}{2n+1}})\right]\\
&+ \sum_{\tau>0} \frac{C_1}{n|\theta - \tau|} \left[\frac{1}{1+K} + \frac{2n+1}{n}\log(\frac{1/n + \theta + \frac{K}{2n+1}}{1/n + \theta + \frac{C_1}{2n+1}})\right]\\
&\leq \sum_{\tau>0} \frac{C_1}{n|\theta - \tau|} \left[\frac{2}{1+K} + 3\log(1+ \frac{K-C_1}{2n+1}\frac{1}{1/n + \tau + \frac{C_1}{2n+1}}) + 3\log(1+ \frac{K-C_1}{2n+1}\frac{1}{1/n+\theta + \frac{C_1}{n}})\right]\\
&\leq \sum_{\tau>0} \frac{C_1}{n|\theta - \tau|}\left[\frac{2}{2+K}+6\log(2)\right].
\end{align}
Similarly for~\eqref{pKboundMultiAtom2Term3} we have 
\begin{align}
\eqref{pKboundMultiAtom2Term3}&\leq \sum_{\tau<0} \left[\frac{1}{1+n\left(\frac{-K}{2n+1} - \tau\right)} + \int_{-K}^{-C_1} \frac{1}{1+n\left(\frac{k}{2n+1} - \tau\right)}\; dk\right] \frac{C_1}{n|\theta - \tau|}\\
&+ \sum_{\tau<0} \left[\frac{1}{1+n\left(\theta + \frac{C_1}{2n+1}\right)} + \int_{C_1}^K \frac{1}{1+n\left(\theta + \frac{k}{2n+1}\right)}\; dk\right]\frac{C_1}{n|\theta - \tau|}\\
&\leq \sum_{\tau<0} \left[\frac{3}{3+K} + 3\log(\frac{1/n + \frac{-C_1}{2n+1} - \tau}{1/n - \frac{K}{2n+1} - \tau})\right] \frac{C_1}{n|\theta - \tau|}\\
&+\sum_{\tau<0 }  \left[\frac{1}{1+K} + 3\log(\frac{1/n + \theta + \frac{K}{2n+1}}{1/n + \theta + \frac{C_1}{2n+1}})\right] \frac{C_1}{n|\theta - \tau|}\\
&\leq \sum_{\tau<0} \left(\frac{4}{1+K} + 6\log(2)\right)\frac{C_1}{n|\theta - \tau|}
\end{align}
Grouping those two terms, we get 
\begin{align}
\eqref{thirdSumBoundPKThetaLargerThanK2}&\leq \sum_{\tau\neq 0} \frac{C_1}{n|\theta - \tau|} \left(\frac{4}{1+K} + 6\log(2)\right)\label{boundNegativeKAllTau}
\end{align}
Finally when $|\theta - \tau|<\frac{1}{n}$, we proceed as before and write 
\begin{align}
\sum_{k=-K}^{-C_1} \frac{C_1}{1+n|\theta - \frac{1}{n} - \frac{k}{2n+1}|} \frac{1}{1+n|\theta - \frac{k}{2n+1}|}
&\leq \sum_{k=-K}^{-C_1} \frac{1}{(1+n(\theta - \frac{1}{n} - \frac{k}{2n+1}))^2}\\
&\leq \frac{4C_1}{1+n\Delta}\frac{1}{1+n|\theta|} + \int_{-K}^{-C_1} \frac{1}{(1+n(\theta - \frac{1}{n} - \frac{k}{2n+1}))^2}\; dk\\
&\leq \frac{4C_1}{n\Delta} \frac{1}{1+n|\theta|} +3C_1 \left[\frac{1}{1+n(\theta - \frac{1}{n} - \frac{K}{2n+1})} + \frac{1}{(1+n(\theta - \frac{1}{n} - \frac{C_1}{2n+1}))}\right]\\
&\stackrel{(a)}{\leq} \frac{4C_1}{n\Delta} \frac{1}{1+n|\theta|} + \frac{6C_1}{0.9(1+n|\theta|)}\label{specialCaseTauIsThetaNegativeK}
\end{align}
The last line in~\eqref{specialCaseTauIsThetaNegativeK} follows from $|\theta - \tau|<1/n$ as well as $\tau>\Delta$ and the assumption $(\Delta - \frac{2}{n} - \frac{K}{2n+1})\geq (\Delta+1/n)0.9$ which can always be achieved for $\Delta$ and $n$ sufficiently large. Grouping~\eqref{specialCaseTauIsThetaNegativeK} together with~\eqref{boundNegativeKAllTau} and~\eqref{boundOnSpikeZeroNegativeK}, we get 
\begin{align}
\eqref{term3Item3LemmaConstant200003} &\leq \frac{6.7C_1+1}{1+n|\theta|} + \frac{2C_1}{n|\theta|+1} \left(\frac{3}{3+C_1} + \frac{1}{1+K} + 6\log(\frac{K}{C_1})\right) + \sum_{\tau\neq 0} \frac{C_1}{n|\theta - \tau|+1} \left(\frac{8}{1+K} + 12\log(2)\right).\label{finalBoundKnegative}
\end{align}

The final bound on $|p_{K}(\theta)|$ for $\theta>\frac{K}{n}$ follows from combining~\eqref{boundItem31}, ~\eqref{finalBoundKpositive} and~\eqref{finalBoundKnegative}
\begin{align}
\left|p_K(e^{2\pi i \theta})\right|& \leq \frac{4C_1}{1+n|\theta|} +\frac{C_1}{n|\theta|+1}\left(8.7 + 3\log(\frac{K}{C_1}) + 3\log(K)\right) + \sum_{\tau\neq 0} \frac{2C_1}{n|\theta - \tau|+1} \left(2+6\log(2)\right)\\
&+ \frac{C_1}{n|\theta|+1} \left(6.7 + \frac{1}{C_1} + \frac{6}{3+K} + \frac{2}{1+K} + 6\log(K/C_1)\right) + \sum_{\tau\neq 0} \frac{C_1}{n|\theta - \tau|+1}\left(\frac{8}{K+1} + 12\log(2)\right)\\
&\leq \frac{C_1}{1+n|\theta|} \left(21 + 9\log(K/C_1) + 3\log(K)\right) + \sum_{\tau\neq 0} \frac{C_1}{n|\theta|+1}\left(5+24\log(2)\right).
\end{align}
From which we thus recover item~\ref{lemmaTroncatureItem2} with $c_1'$ bounded as 
\begin{align}
c_1'&\equiv \frac{C_1}{\log(K+1)}\left(21 + 9\log(K/C_1) + 3\log(K)\right)
\end{align}

To derive item~\ref{lemmaTroncatureItem4}, note that if we let $\chi_{\mathcal{C}}$ to denote the indicator function of the set $\mathcal{C}$, we have 
\begin{align}
p_{K, \err}(\theta) & = p_{K, \err}(\theta)\chi_{[-K/n, K/n]} + \left(p(\theta)- p_K(\theta)\right) \chi_{[-1/2, -K/n]\cup [K/n, 1/2]}
\end{align}
From this decomposition, we get 
\begin{align}
&|p_{K,\err}(\theta)| \leq \frac{c_1\log(1+K)}{1+K}\chi_{[-K/n,K/n]} + |p(e^{2\pi i \theta})  - p_{K}(e^{2\pi i \theta})|\chi_{[-1/2,-K/n]\cup [K/n, 1/2]}(\theta)\\
&\leq \frac{c_1\log(1+K)}{1+K}\chi_{[-K/n,K/n]}  + \left[ \min(1,\sum_{\tau} \frac{C_1}{1+n|\theta - \tau|})\right]\chi_{[-1/2,-K/n]\cup [K/n, 1/2]}(\theta)\\
& +\left[ c'_1 \frac{\log(1+K)}{1+n|\theta|} + c_2'\sum_{\tau\neq 0} \frac{1}{1+n|\theta - \tau|}\right]\chi_{[-1/2,-K/n]\cup [K/n, 1/2]}(\theta)\\
&\leq \frac{c_1\log(1+K)}{1+K}\min(1, \frac{K+1}{1+n|\theta - \tau_0|}) + \min(1,\sum_{\tau} \frac{C_1}{1+n|\theta - \tau|})\chi_{[-1/2,-K/n]\cup [K/n, 1/2]}(\theta) \\
&+ c'_1 \frac{\log(1+K)}{1+K} \min(1, \frac{1+K}{1+n|\theta - \tau_0|}) +  c_2'\sum_{\tau\neq 0} \frac{1}{1+n|\theta - \tau|}\\
&\leq \frac{(c_1 + c_1')\log(1+K)}{1+K}\min(1, \frac{K+1}{1+n|\theta - \tau_0|}) + \frac{C_1}{1+n|\theta|}\chi_{[-1/2,-K/n]\cup [K/n, 1/2]}(\theta) + \min(1, \sum_{\tau\neq \tau_0} \frac{C_1}{1+n|\theta - \tau|})\\
&+ c_2' \min\left(1, \sum_{\tau\neq 0} \frac{1}{1+n|\theta - \tau|}\right)\\
&\leq \frac{(c_1 + c_1' + C_1/\log(K+1))\log(1+K)}{1+K}\min(1, \frac{K+1}{1+n|\theta - \tau_0|}) \\
&+\min(1, \sum_{\tau\neq \tau_0} \frac{C_1}{1+n|\theta - \tau|}) + c_2' \min(1, \frac{1}{1+|\theta - \tau|})\\
 \label{boundOnpKerr}
\end{align}
Items~\ref{lemmaTroncatureItem5} and~\ref{lemmaTroncatureItem6} both follow from applying the bounds~\eqref{constantBoundlemmaK2} and~\eqref{decreasingBoundLemmaK3} of Lemma~\ref{lemma:boundK} to the bound obtained in item~\ref{lemmaTroncatureItem4} above. In particular, applying the result of this lemma to the bound~\eqref{boundOnpKerr} on $p_{K, err}(\theta)$ gives 
\begin{align}
\left|\mathcal{M}_1[p_{K,err}](\theta)\right| &\leq \frac{(c_1+c_1'+C_1/\log(K+1))(2M_1''' + 2M_2\log(1+K) +1)\log(1+K)}{1+K} + O(\frac{1}{n\Delta})\label{boundM1pKerrConstant}\\
&\lesssim \frac{2C_1(21 + 9\log(K/C_1) + 3\log(K))(M_1''' + M_2\log(1+K))}{1+K}\label{boundM1pKerrConstant}
\end{align}

Finally to get item~\ref{lemmaTroncatureItem7}, we let $\eta = \sup_{k} p(e^{2\pi i \frac{k}{2n+1}})$ from which we have 
\begin{align}
p(e^{2\pi i \frac{k}{2n+1}}) \leq \min\left(\eta, \sum_{\tau\in S} \frac{C_1}{1+n|\frac{k}{2n+1} - \tau|}\right)
\end{align}
Since the polynomial $p(e^{2\pi i \theta})$ is normalized, i.e $\|p\|_\infty = 1$, there exists a $\theta_0 \in [\tau - \frac{C_1}{n}, \tau+\frac{C_1}{n}]$ such that $p(e^{2\pi i \theta_0})=1$. Without loss of generality we can assume that such a $\theta_0$ belongs to the interval $[\tau_0-\frac{C_1}{n}, \tau_0+\frac{C_1}{n}]$, and hence, as we centered the problem with respect to $\tau_0$, in the interval $[-\frac{C_1}{n}, \frac{C_1}{n}]$. Splitting $p(e^{2\pi i \theta})$ into $p_K(e^{2\pi i \theta})$ and $p_{K, \err}(e^{2\i i \theta})$, we have 
\begin{align}
\left|p(e^{2\pi i \theta})\right|&\leq \sum_{k=-K}^K \sum_{\tau\in S} \min(\eta, \frac{C_1}{1+n\left|\frac{k}{2n+1} - \tau\right|}) \frac{1}{1+n|\theta - \frac{k}{2n+1}|}\label{pKerrBadSum11}\\
&+ \sum_{k\notin [-K, K]} \sum_{\tau\in S} \min(1, \frac{C_1}{1+n\left|\frac{k}{2n+1} - \tau\right|}) \frac{1}{1+n\left|\theta - \frac{k}{2n+1}\right|} \label{pKerrInSplittingForEta}
\end{align}
Starting with~\eqref{pKerrInSplittingForEta}, we can write 
\begin{align}
\eqref{pKerrInSplittingForEta}&\leq \sum_{k\notin [-K,K]} \sum_{\tau\neq 0} \min(1, \frac{C_1}{1+n|\frac{k}{2n+1} - \tau|}) \frac{1}{1+n\left|\theta - \frac{k}{2n+1}\right|}\label{pKerrInSplittingForEta111}\\
&+ \sum_{k\notin [-K,K]} \frac{C_1}{1+n\frac{k}{2n+1}} \frac{1}{1+n\left|\theta - \frac{k}{2n+1}\right|}\label{pKerrInSplittingForEta112}
\end{align}
We start by bounding~\eqref{pKerrInSplittingForEta112}. For this term, we first consider the case $\theta>1/n$. In this case we have 
\begin{align}
\eqref{pKerrInSplittingForEta112}&\leq \sum_{k=K}^n \frac{C_1}{1+n\frac{k}{2n+1}} \frac{1}{1+n\left(\frac{k}{2n+1} - \theta\right)}\label{Kpositive1Term3}\\
& +\sum_{k=-n}^{-K} \frac{C_1}{1-n\frac{k}{2n+1}} \frac{1}{1+n\left(\theta - \frac{k}{2n+1}\right)}\label{Knegative1Term3}
\end{align}
For each of those terms, using bounds on the harmonic numbers, one can write 
\begin{align}
\eqref{Knegative1Term3}&\leq \frac{C_1}{n|\theta|} \left(\frac{1}{1+n\frac{K}{2n+1}} + \int_{K}^n \frac{1}{1+n\frac{k}{2n+1}}\; dk - \int_{K}^n \frac{1}{1+n\left(\theta + \frac{k}{2n+1}\right)}\; dk\right)\\
&\leq \frac{C_1}{n|\theta|} \frac{3}{3+K} + \frac{3C_1}{n|\theta|} \left(\log(\frac{1/n + \frac{n}{2n+1}}{1/n + \frac{K}{2n+1}}) - \log(\frac{1/n + \theta + \frac{n}{2n+1}}{1/n + \theta + \frac{K}{2n+1}})\right)\\
&\leq \frac{C_1}{n|\theta|}\frac{3}{3+K} + \frac{3C_1}{n|\theta|} \left(\log(\frac{1/n + \frac{n}{2n+1}+\theta}{1/n + \frac{n}{2n+1}}) + \log(\frac{1/n + \theta + \frac{K}{2n+1}}{1/n + \frac{K}{2n+1}})\right)\\
&\leq \frac{3C_1}{3+K} + \frac{4C_1}{n} + \frac{6C_1}{K}. 
\end{align}
Similarly we have
\begin{align}
\eqref{Kpositive1Term3}&\leq \frac{C_1}{n|\theta|}\frac{1}{1+n\frac{K}{2n+1}} + \frac{C_1}{n\theta}\left[\int_{K}^n \frac{1}{1+n\frac{k}{2n+1}}\; dk -  \int_{K}^{n} \frac{1}{1+n(\frac{k}{2n+1} - \theta)}\; dk\right]\\
&\leq \frac{C_1}{n|\theta|} \frac{3}{3+K} + \frac{C_1}{n\theta}\left[\log(\frac{1/n + \frac{n}{2n+1}}{1/n + \frac{K}{2n+1}}) - \log(\frac{1/n + \frac{n}{2n+1} - \theta}{1/n + \frac{K}{2n+1} - \theta})\right]\\
&\leq \frac{C_1}{n|\theta|}\frac{3}{3+K} + \left[\log(\frac{1/n + \frac{n}{2n+1}}{1/n + \frac{n}{2n+1} - \theta}) + \log(\frac{1/n + \frac{K}{2n+1} - \theta}{1/n + \frac{K}{2n+1}})\right]\frac{C_1}{n\theta}\\
&\leq \frac{C_1}{n|\theta|}\frac{3}{3+K} + \left[\log(1+\frac{\theta}{1/n + \frac{n}{2n+1} - \theta}) + \log(1+ \frac{\theta}{1/n + \frac{K}{2n+1} - \theta})\right]\frac{C_1}{n\theta}\\
&\leq \frac{C_1}{n|\theta|} \frac{3}{3+K} + \frac{C_1}{n} + \frac{3C_1}{K}
\end{align}
Whenever $|\theta|\geq 1/n$. When $|\theta|\leq 1/n$, we write 
\begin{align}
\sum_{k\notin [-K,K]} \frac{C_1}{1+n\frac{k}{2n+1}}\frac{1}{1+n\left(\frac{k}{2n+1} - \frac{1}{n}\right)}&\leq \sum_{k\notin [-K,K]} \frac{C_1}{(1+n(\frac{k}{2n+1} - \frac{1}{n}))^2}\\
&\leq \frac{C_1}{1+n(\frac{K}{2n+1} - \frac{1}{n})} + \frac{2n+1}{n} \left[\frac{C_1}{1+n(\frac{n}{2n+1} - \frac{1}{n})} - \frac{C_1}{1+n(\frac{K}{2n+1} - \frac{1}{n})}\right]
\end{align}
We now bound~\eqref{pKerrInSplittingForEta111}. Without loss of generality, we can focus on the atoms located on the right of the $[-K,K]$ interval. By symmetry, the bound on the sum~\eqref{pKerrInSplittingForEta111} can then be obtained by multiplying this bound by two. Using $k_\tau$ to denote the index associated with the position $\tau$, i.e. $\tau \approx \frac{k_\tau}{2n+1}$, we decompose the sum into the following four contributions
\begin{align}
\eqref{pKerrInSplittingForEta111}&\leq \sum_{\tau>0}\sum_{k=k_\tau - C_1}^{k_\tau+C_1} \frac{1}{1+n(\frac{k}{2n+1} - \theta)}\label{term1B1}\\
&+ \sum_{\tau>0}\sum_{k = k_\tau + C_1}^{n} \frac{C_1}{1+n(\frac{k}{2n+1} - \tau)} \frac{1}{(\frac{k}{2n+1} - \theta)n+1}\label{term2B1}\\
&+ \sum_{\tau>0}\sum_{k=K}^{k_\tau-C_1} \frac{C_1}{1+n\left(\frac{k}{2n+1} - \tau\right)} \frac{1}{1+n\left(\frac{k}{2n+1} - \theta\right)}\label{term3B1}\\
&+ \sum_{\tau>0}\sum_{k<-K} \frac{C_1}{1+n\left(\tau - \frac{k}{2n+1}\right)} \frac{1}{1+n\left(\theta - \frac{k}{2n+1}\right)}\label{term4B1}
\end{align}
For each of those terms, we have 
\begin{align}
\eqref{term1B1}&\leq \sum_{\tau>0} \frac{1}{1+n\left|\tau - \frac{C_1}{n} - \theta\right|} + \int_{k_\tau-C_1}^{k_\tau+C_1} \frac{1}{1+n\left(\frac{k}{2n+1} - \theta\right)}\; dk\\
&\lesssim \sum_{\tau>0} \frac{4}{n(\tau - \theta - \frac{C_1}{n})} + \sum_{\tau>0} 3\log(\frac{1/n + (\tau + \frac{C_1}{n}) - \theta}{1/n + \tau - \frac{C_1}{n} - \theta})\\
&\lesssim \sum_{\tau>0} \frac{4}{n|\tau - \theta - \frac{C_1}{n}|} +3\log(1+ \frac{2C_1}{n} \frac{1}{1/n + \tau - \frac{C_1}{n} - \theta})\\
&\lesssim \frac{4}{n\Delta}\log|S| + 24\frac{C_1}{n\Delta}\log|S|.
\end{align}
\begin{align}
\eqref{term2B1}&\leq \sum_{\tau>0} \sum_{k=k_\tau + C_1}^n \frac{C_1}{n|\theta - \tau|} \left[\frac{1}{1+n\left(\frac{k}{2n+1} - \tau\right)} - \frac{1}{\left(\frac{k}{2n+1} - \theta\right)n+1}\right]\\
&\leq \sum_{\tau>0} \frac{C_1}{n|\theta - \tau|} \left[\frac{1}{1+n\frac{C_1}{n}} + \int_{k_\tau + C_1}^{n} \frac{1}{1+n\left(\frac{k}{2n+1} - \tau\right)}\; dk - \int_{k_\tau +C_1}^{n} \frac{1}{1+n\left(\frac{k}{2n+1} - \theta\right)}\; dk\right]\\
&\lesssim \sum_{\tau>0} \frac{C_1}{n|\theta - \tau|} \left[\frac{1}{1+C_1} + \frac{2n+1}{n}\log(\frac{1/n + 1/2 - \tau}{1/n + \frac{C_1}{n}}) - \frac{2n+1}{n}\log(\frac{1/n + 1/2 - \theta}{1/n + \tau + \frac{C_1}{n} - \theta})\right]\\
&\leq \sum_{\tau>0} \frac{C_1}{n|\theta - \tau|} \left[\frac{1}{1+C_1} + \frac{2n+1}{n} \log(1+ \frac{\tau - \theta}{1/n + 1/2-\tau}) + \frac{2n+1}{n}\log(1+ \frac{\tau - \theta}{1/n + \frac{C_1}{n}})\right]\\
&\lesssim \sum_{\tau>0} \frac{C_1}{n|\theta - \tau|}\frac{1}{1+C_1} + \frac{3C_1}{n|\theta - \tau|} \log(|\tau - \theta|n) + \frac{3C_1}{n|\theta - \tau|} \log(|\tau - \theta|n)\\
&\lesssim \frac{C_1}{n\Delta}\log|S| \frac{1}{1+C_1} + \frac{6C_1}{n\Delta} \left[\log^2|S| + \log(n\Delta)\log|S|\right].
\end{align}
\begin{align}
\eqref{term3B1}&\leq \sum_{\tau>0} \sum_{k=K}^{k_\tau - C_1} \frac{C_1}{n|\theta - \tau|} \left[\frac{1}{1+n(\tau - \frac{k}{2n+1})} + \frac{1}{1+n(\frac{k}{2n+1} - \theta)}\right]\\
&\sum_{\tau>0}  \frac{C_1}{n|\theta - \tau|} \left[\frac{1}{1+C_1} + \int_{k_\tau-C_1}^{K} \frac{1}{1+n\left(\right)} + \int_{K}^{k_\tau - C_1} \frac{1}{1+n\left(\frac{k}{2n+1} - \theta\right)}\; dk\right]\\
&\lesssim \sum_{\tau>0} \frac{C_1}{n|\theta - \tau|} \left[\frac{1}{1+C_1} + \log(\frac{1/n + \frac{C_1}{n}}{1/n+\tau - \frac{K}{n}}) - \log(\frac{1/n +\tau - \frac{C_1}{n} - \theta}{1/n + \frac{K}{n} - \theta})\right]\\
&\lesssim \sum_{\tau>0} \frac{C_1}{n|\theta - \tau|}\left[\frac{1}{1+C_1} + 2\log(n|\theta - \tau|)\right]\\
&\lesssim \frac{2C_1}{n\Delta}\frac{1}{1+C_1}\log|S| + \frac{2C_1}{n\Delta}\left(\log(n\Delta)\log|S| + \log^2|S|\right).
\end{align}
Finally for~\eqref{term4B1}, we write 
\begin{align}
\eqref{term4B1}& = \sum_{\tau >0} \sum_{k=-n}^{-K} \frac{C_1}{1+n\left(\tau - \frac{k}{2n+1}\right)} \frac{1}{1+n\left(\theta - \frac{k}{2n+1}\right)}\\
&\leq \sum_{\tau>0} \sum_{k=-n}^{-K} \frac{C_1}{1+n|\theta - \tau|}\left[\frac{1}{1+n(\theta - \frac{k}{2n+1})} - \frac{1}{1+n(\tau - \frac{k}{2n+1})}\right]\\
&\leq \sum_{\tau>0} \frac{C_1}{1+n|\theta - \tau|} \left[\frac{1}{1+n(\theta + \frac{K}{2n+1} )} + \int_{K}^n \frac{1}{1+n(\theta + \frac{k}{2n+1})}\; dk - \int_{K}^n \frac{1}{1+n(\tau + \frac{k}{2n+1})}\; dk\right]\\
&\lesssim \sum_{\tau>0} \frac{C_1}{1+n\left|\theta - \tau\right|} \left[\frac{1}{1+K} + \log(\frac{1/n + \theta + 1/2}{1/n + \theta + \frac{K}{2n+1}}) - \log(\frac{1/n + \tau + 1/2}{1/n + \tau + \frac{K}{2n+1}})\right]\\
&\lesssim \sum_{\tau>0} \frac{C_1}{1+n|\theta - \tau|} \left[\frac{1}{1+K} + \log(1+ \frac{\tau - \theta}{1/n+\theta + 1/2}) + \log(1+\frac{\tau-\theta}{1/n + \theta + \frac{K}{2n+1}})\right]\\
&\lesssim \frac{C_1}{n\Delta}\frac{1}{1+K}\log|S| + 2\frac{C_1}{n\Delta}\left(\log^2|S| + \log(n\Delta)\log|S|\right).
\end{align}
The sum in~\eqref{pKerrInSplittingForEta111} can thus be made arbitrarily small provided that we take $\Delta$ and $n$ large enough. We now bound~\eqref{pKerrBadSum11}. Again we split this sum into a contribution arising from the atoms $\tau$ that are located outside $[-K,K]$ and the contribution arising from $\tau_0$. By symmetry we can focus on the atoms 
\begin{align}
&\sum_{k=-K}^{K} \sum_{\tau>0} \frac{C_1}{1+n|\tau - \frac{k}{2n+1}|}\frac{1}{1+n|\frac{k}{2n+1} - \theta|}\\
& = \sum_{k=k_\theta}^K \sum_{\tau>0} \frac{C_1}{1+n(\tau - \frac{k}{2n+1})} \frac{1}{1+n\left(\frac{k}{2n+1} - \theta\right)} \\
&+ \sum_{k=-K}^{k_\theta} \sum_{\tau>0} \frac{C_1}{1+n\left(\tau - \frac{k}{2n+1}\right)}\frac{1}{1+n\left(\theta - \frac{k}{2n+1}\right)}\\
&\leq \sum_{k=k_\theta}^{K} \sum_{\tau>0} \frac{C_1}{n|\theta - \tau|}\left(\frac{1}{1+n\left(\tau - \frac{k}{2n+1}\right)} + \frac{1}{1+n\left(\frac{k}{2n+1} - \theta\right)}\right)\\
&+ \sum_{k=-K}^{k_\theta} \sum_{\tau>0} \frac{C_1}{n|\theta - \tau|} \left(\frac{1}{1+n(\tau - \frac{k}{2n+1})} - \frac{1}{1+n\left(\theta - \frac{k}{2n+1}\right)}\right)\\
&\lesssim \sum_{\tau>0} \frac{C_1}{n|\theta - \tau|} \left[1+\frac{1}{1+n\left(\tau - \frac{K}{2n+1}\right)} + \int_{K}^{k_\theta} \frac{1}{1+n\left(\tau - \frac{k}{2n+1}\right)}\; dk + \int_{k_\theta}^K \frac{1}{1+n\left(\frac{k}{2n+1} - \theta\right)}\; dk\right]\\
&+ \sum_{\tau>0}\frac{C_1}{n|\theta - \tau|} \left[\frac{1}{1+n|\tau - \theta|} + \int_{k_\theta}^{-K} \frac{1}{1+n(\tau - \frac{k}{2n+1})}\; dk - \int_{k_\theta}^{-K} \frac{1}{1+n\left(\theta - \frac{k}{2n+1}\right)}\; dk\right]\\
&\lesssim \sum_{\tau>0 } \frac{C_1}{n|\theta - \tau|} \left[\frac{1}{1+n|\theta - \tau|} + \log(\frac{1/n + \tau - \theta}{1/n + \tau - \frac{K}{2n+1}}) + \log(\frac{1/n +\frac{K}{2n+1} - \theta}{1/n})\right]\\
&+\sum_{\tau>0} \frac{C_1}{n|\theta - \tau|} \frac{C_1}{n|\theta - \tau|} \left[\frac{1}{1+n|\theta - \tau|} + \log(\frac{1/n + \tau + \frac{K}{2n+1}}{1/n + \tau - \theta}) -\log(\frac{1/n + \theta + \frac{K}{2n+1}}{1/n})\right]\\
&\lesssim \sum_{\tau>0}\frac{C_1}{n|\theta - \tau|} \left[\frac{1}{1+n|\tau - \theta|} + 2\log(n|\tau - \theta|)\right]\\
&+ \sum_{\tau>0} \frac{C_1}{n|\theta - \tau|} \left[\frac{1}{1+n|\tau - \theta|} +\log(n|\tau - \theta|) + \log(\frac{|\tau - \theta|}{K})\right]\\
&\lesssim 2\frac{C_1}{n\Delta }\log|S| + 2\frac{C_1}{n\Delta}\left( \log(n\Delta)\log|S|+\log^2|S|\right).
\end{align}
We conclude by bounding the contribution arising from $\tau_0$. Note that by assumption we always have $\theta\in [-\frac{C_1-1}{\eta}, \frac{C_1-1}{\eta}]$ as we necessarily have $\theta\in \left[-\frac{C_1-1}{n}, \frac{C_1-1}{n}\right].$. From this, we can decompose the last contribution as  
\begin{align}
\sum_{k=-K}^{K} \min(\eta, \frac{C_1}{1+n|\frac{k}{2n+1}|})\frac{1}{1+n\left|\theta - \frac{k}{2n+1}\right|}&= \sum_{k=-K}^{-\frac{C_1-1}{\eta}} \frac{C_1}{1+n\frac{-k}{2n+1}} \frac{1}{1+n\left(\theta - \frac{k}{2n+1}\right)}\label{term1DecompositionpErr111}\\
&+ \sum_{k = -\frac{C_1-1}{\eta}}^{k_\theta} \eta \frac{1}{1+n(\theta - \frac{k}{2n+1})} + \sum_{k=k_\theta}^{\frac{C_1-1}{\eta}} \eta \frac{1}{1+n(\frac{k}{2n+1} - \theta)}\label{term2and3DecompositionpErr111}\\
&+ \sum_{k=\frac{C_1-1}{\eta}}^{K} \frac{C_1}{1+n\frac{k}{2n+1}} \frac{1}{1+n\left(\frac{k}{2n+1} - \theta\right)}\label{term4DecompositionpErr111}
\end{align}
We start by bounding~\eqref{term1DecompositionpErr111} and~\eqref{term4DecompositionpErr111}. For those two terms, we have 
\begin{align}
\eqref{term1DecompositionpErr111}&\leq\sum_{k=-K}^{-\frac{C_1-1}{\eta}} \frac{C_1}{n|\theta|} \left(\frac{1}{1+n\left(\theta - \frac{k}{2n+1}\right)} - \frac{1}{1+n\left(-\frac{k}{2n+1}\right)} \right)\\
&\leq \frac{C_1}{n|\theta|} \left(\frac{1}{1+n\left(\theta + \frac{C_1-1}{\eta}\frac{1}{2n+1}\right)} + \int_{\frac{C_1-1}{\eta}}^K \frac{1}{1+n(\theta + \frac{k}{2n+1})}\; dk\right) \\
& - \frac{C_1}{n|\theta|} \int_{\frac{C_1-1}{\eta}}^{K} \frac{1}{1+n\frac{k}{2n+1}}\;dk\\
&\leq \frac{C_1}{n|\theta|}\left(\frac{1}{1+n|\theta| + \frac{n}{2n+1}\frac{C_1-1}{\eta}} + \log(\frac{1/n + \theta + \frac{K}{2n+1}}{1/n + \theta +\frac{C_1-1}{\eta}\frac{1}{2n+1}}) - \log(\frac{1/n + \frac{K}{2n+1}}{1/n + \frac{C_1-1}{\eta}\frac{1}{2n+1}})\right)\label{boundFromKtoC1minus1Eta}
\end{align}
To control the last line we make the distinction between the case $\theta<-\frac{C_1-1}{\eta}\frac{1}{2n+1}\frac{1}{2}$ and $\theta>-\frac{C_1-1}{\eta}\frac{1}{2n+1}\frac{1}{2}$. In the former, we write 
\begin{align}
\eqref{boundFromKtoC1minus1Eta}&\leq \frac{C_1}{n|\theta|} \left(\frac{1}{1+n\theta + \frac{n}{2n+1}\frac{C_1-1}{\eta}} + \log(1+\frac{\theta}{1/n + \frac{C_1-1}{\eta}\frac{1}{2n+1}})\vee \log(1+\frac{-\theta}{1/n + \frac{C_1-1}{\eta}\frac{1}{2n+1} + \theta}) \right)\\
&+ \frac{C_1}{n|\theta|}\left( \log(1+\frac{\theta}{1/n + \frac{K}{2n+1}})\vee \log(1+ \frac{-\theta}{1/n + \frac{K}{2n+1} + \theta})\right)\\
&\leq 2\frac{C_1}{C_1-1}\frac{2n+1}{n}\eta + 2\frac{C_1}{C_1-1}\frac{2n+1}{n}\eta + 2\frac{C_1}{K}\frac{2n+1}{n}.
\end{align}
In the latter, we have 
\begin{align}
\eqref{boundFromKtoC1minus1Eta}&\leq 2\frac{C_1}{C_1-1}\eta\frac{2n+1}{n} +2\frac{C_1}{C_1-1} \frac{2n+1}{n}\eta \log(C_1-1) + 2\frac{C_1}{K}\frac{2n+1}{n}\eta.
\end{align}
From those two bounds we thus have 
\begin{align}
\eqref{boundFromKtoC1minus1Eta}&\leq 36\eta.\label{boundnegativeKC1minus1Eta112}
\end{align}

A similar reasoning can be applied to~\eqref{term4DecompositionpErr111}. For this term, we have 
\begin{align}
\sum_{k = \frac{C_1-1}{\eta}}^{K} \frac{C_1}{1+n \frac{k}{2n+1}} \frac{1}{1+n\left(\frac{k}{2n+1} - \theta\right)}
&\leq \sum_{k=\frac{C_1-1}{\eta}}^{K} \frac{C_1}{n\theta} \left(\frac{1}{1+n\frac{k}{2n+1}} - \frac{1}{1+n\left(\frac{k}{2n+1} - \theta\right)}\right)\\
&\leq \frac{C_1}{n\theta} \left(\frac{1}{1+n\frac{C_1-1}{\eta}\frac{1}{2n+1}} + \int_{\frac{C_1-1}{\eta}}^{K} \frac{1}{1+n\frac{k}{2n+1}}\; dk\right)\\
&+ \frac{C_1}{n\theta} \int_{\frac{C_1-1}{\eta}}^{K} \frac{1}{1+n\left(\frac{k}{2n+1} - \theta\right)}\; dk\\
&\leq \frac{C_1}{n\theta} \left(\frac{1}{1+n\frac{C_1-1}{\eta}\frac{1}{2n+1}}+ \log(\frac{1/n + \frac{K}{2n+1}}{1/n + \frac{n}{2n+1}\frac{C_1-1}{\eta}}) - \log(\frac{1/n + \frac{K}{2n+1} - \theta}{1/n + \frac{C_1-1}{\eta}\frac{1}{2n+1} - \theta})\right)\\
&\leq \frac{C_1}{C_1-1}\frac{2n+1}{n}\eta + \frac{C_1}{n\theta} \left(\log(1+\frac{\theta}{1/n + \frac{K}{2n+1} - \theta}\vee \frac{-\theta}{1/n + \frac{K}{2n+1}})\right)\label{boundTMPpKtoC1minus1Eta2222bb}\\
&+ \frac{C_1}{n\theta} \log(1+\frac{\theta}{1/n + \frac{C_1-1}{\eta}\frac{1}{2n+1} - \theta}\vee \frac{-\theta}{1/n + \frac{n}{2n+1}\frac{C_1-1}{\eta}})\label{boundTMPpKtoC1minus1Eta2222}
\end{align}
To control the sum~\eqref{boundTMPpKtoC1minus1Eta2222bb}+\eqref{boundTMPpKtoC1minus1Eta2222}, as in the case $-K\leq k\leq -\frac{C_1-1}{\eta}$ we consider the two frameworks $\theta\geq \frac{C_1-1}{\eta}\frac{1}{2n+1}\frac{1}{2}$ and $\theta<\frac{C_1-1}{\eta}\frac{1}{2n+1}$. In the first case, we write 

\begin{align}
\eqref{boundTMPpKtoC1minus1Eta2222}&\leq \frac{C_1}{C_1-1}\frac{2n+1}{n}\eta + 2\frac{C_1}{K}\frac{2n+1}{n} + 2\frac{C_1}{C_1-1}\eta\frac{2n+1}{n}\log(C_1)
\end{align}

When $\theta<\frac{C_1-1}{\eta}\frac{1}{2n+1}\frac{1}{2}$, we write 
\begin{align}
\eqref{boundTMPpKtoC1minus1Eta2222}&\leq \frac{C_1}{C_1-1}\frac{2n+1}{n}\eta + 2\frac{C_1}{K}\frac{2n+1}{n} + 2\frac{C_1}{C_1-1}\frac{2n+1}{n}\eta. 
\end{align}
For~\eqref{term2and3DecompositionpErr111}, each of the two terms can respectively be bounded as
\begin{align}
\sum_{k = -\frac{C_1-1}{\eta}}^{k_\theta} \eta \frac{1}{1+n(\theta - \frac{k}{2n+1})} &\leq \eta + \int_{k_\theta}^{-\frac{C_1-1}{\eta}} \frac{1}{1+n(\theta - \frac{k}{2n+1})}\; dk\\
&\leq \eta + \frac{2n+1}{n} \log(\frac{1/n +\theta + \frac{C_1-1}{\eta}\frac{1}{2n+1}}{1/n})\\
&\leq \eta + \frac{2n+1}{n} \log(1 + \frac{1}{2}\frac{C_1-1}{\eta} + C_1)
\end{align}
Similarly we have

\begin{align}
\sum_{k = k_\theta}^{\frac{C_1-1}{\eta}} \eta\frac{1}{1+n(\frac{k}{2n+1} - \theta)} &\leq \eta + \int_{k_\theta}^{\frac{C_1-1}{\eta}} \frac{1}{1+n(\frac{k}{2n+1} - \theta)}\; dk\\
&\leq \eta + \frac{2n+1}{n} \log(\frac{1/n + \theta +  \frac{C_1-1}{\eta}\frac{1}{2n+1}}{1/n})\\
&\leq \eta + \frac{2n+1}{n} \log(1 + \frac{1}{2}\frac{C_1-1}{\eta} + C_1)
\end{align}
Combining those bounds with~\eqref{boundnegativeKC1minus1Eta112}, noting that $\theta$ cannot simultaneously be larger than $\frac{C_1-1}{\eta}\frac{1}{2n+1}\frac{1}{2}$ and smaller than $-\frac{C_1-1}{\eta}\frac{1}{2n+1}\frac{1}{2}$ we can thus write 
\begin{align}
\eqref{term1DecompositionpErr111}+\eqref{term2and3DecompositionpErr111}+\eqref{term4DecompositionpErr111}&\leq 42\eta + 4\eta\log(1+\frac{C_1-1}{\eta}\frac{1}{2} + C_1).  
\end{align}

Note that when $|\theta|<1/n$ both in the case $\leq k\leq $ we can write 
\begin{align}
\sum_{k=\frac{C_1-1}{\eta}}^{K} \frac{1}{1+n\left(\frac{k}{2n+1} - \theta\right)}\frac{1}{1+n\frac{k}{2n+1}}&\leq \sum_{k=\frac{C_1-1}{\eta}}^{K} \frac{1}{\left(1+n\left(\frac{k}{2n+1} - \frac{1}{n}\right)\right)^2}\\
&\leq \frac{1}{1+n\left(\frac{C_1-1}{\eta}\frac{1}{2n+1} - \frac{1}{n}\right)} \\
&+ \frac{2n+1}{n}\left(\frac{1}{1+n\left(\frac{C_1-1}{\eta}\frac{1}{2n+1} - \frac{1}{n}\right)} - \frac{1}{1+n\left(\frac{K}{2n+1} - \frac{1}{n}\right)}\right)
\end{align}

Hence the total upper bound on the sum~\eqref{pKerrBadSum11}+\eqref{pKerrInSplittingForEta} can read
\begin{align}
~\eqref{pKerrBadSum11}+\eqref{pKerrInSplittingForEta} &\leq 42\eta + 4\eta\log(1+\frac{C_1-1}{\eta}\frac{1}{2} + C_1).  
\end{align}
From which we get the lower bound on $\eta$
\begin{align}
1\leq 42\eta + 4\eta\log(1+\frac{C_1-1}{\eta}\frac{1}{2} + C_1)
\end{align}
which follows from using $p(e^{2\pi i \theta})=1$. The obtained lower bound is depicted in Fig.~\ref{logEtaFigure}.

\begin{figure}
\centering
%
%
\definecolor{mycolor1}{rgb}{0.00000,0.44700,0.74100}%
\definecolor{mycolor2}{rgb}{0.85000,0.32500,0.09800}%
\definecolor{mycolor3}{rgb}{0.92900,0.69400,0.12500}%
\begin{tikzpicture}

\begin{axis}[%
width=3.028in,
height=2in,
at={(1.011in,0.642in)},
scale only axis,
unbounded coords=jump,
xmin=0,
xmax=0.05,
xlabel style={font=\color{white!15!black}},
xlabel={$\eta$},
ymin=0,
ymax=4.5,
axis background/.style={fill=white},
legend style={legend cell align=left, align=left, draw=white!15!black}
]
\addplot [color=mycolor1, line width=1.5pt]
  table[row sep=crcr]{%
0	nan\\
0.000505050505050505	0.0509550726238766\\
0.00101010101010101	0.0991136272486162\\
0.00151515151515152	0.146219185422375\\
0.00202020202020202	0.192642347200401\\
0.00252525252525253	0.238559119652818\\
0.00303030303030303	0.28407316616786\\
0.00353535353535354	0.329252981982402\\
0.00404040404040404	0.374147247707404\\
0.00454545454545455	0.41879236028665\\
0.00505050505050505	0.463216569633058\\
0.00555555555555556	0.507442442510397\\
0.00606060606060606	0.55148842262111\\
0.00656565656565657	0.595369866939937\\
0.00707070707070707	0.639099761269353\\
0.00757575757575758	0.682689230314506\\
0.00808080808080808	0.726147911146181\\
0.00858585858585859	0.76948423294516\\
0.00909090909090909	0.812705630706294\\
0.0095959595959596	0.855818711313695\\
0.0101010101010101	0.898829384561548\\
0.0106060606060606	0.94174296790951\\
0.0111111111111111	0.984564271242656\\
0.0116161616161616	1.02729766619111\\
0.0121212121212121	1.0699471433731\\
0.0126262626262626	1.11251636008243\\
0.0131313131313131	1.1550086803347\\
0.0136363636363636	1.19742720874435\\
0.0141414141414141	1.23977481937672\\
0.0146464646464646	1.2820541804735\\
0.0151515151515152	1.32426777576419\\
0.0156565656565657	1.36641792293325\\
0.0161616161616162	1.40850678970205\\
0.0166666666666667	1.45053640789893\\
0.0171717171717172	1.49250868582232\\
0.0176767676767677	1.53442541914827\\
0.0181818181818182	1.57628830059055\\
0.0186868686868687	1.6180989284866\\
0.0191919191919192	1.65985881445467\\
0.0196969696969697	1.7015693902445\\
0.0202020202020202	1.74323201388493\\
0.0207070707070707	1.78484797521634\\
0.0212121212121212	1.82641850088317\\
0.0217171717171717	1.86794475885066\\
0.0222222222222222	1.90942786250117\\
0.0227272727272727	1.95086887435805\\
0.0232323232323232	1.99226880947828\\
0.0237373737373737	2.03362863854995\\
0.0242424242424242	2.07494929072611\\
0.0247474747474748	2.11623165622226\\
0.0252525252525253	2.15747658870181\\
0.0257575757575758	2.19868490747057\\
0.0262626262626263	2.23985739949912\\
0.0267676767676768	2.28099482128949\\
0.0272727272727273	2.32209790060099\\
0.0277777777777778	2.36316733804808\\
0.0282828282828283	2.40420380858206\\
0.0287878787878788	2.44520796286685\\
0.0292929292929293	2.48618042855823\\
0.0297979797979798	2.5271218114948\\
0.0303030303030303	2.56803269680817\\
0.0308080808080808	2.60891364995917\\
0.0313131313131313	2.64976521770602\\
0.0318181818181818	2.69058792901008\\
0.0323232323232323	2.73138229588403\\
0.0328282828282828	2.77214881418707\\
0.0333333333333333	2.81288796437114\\
0.0338383838383838	2.85360021218192\\
0.0343434343434343	2.89428600931794\\
0.0348484848484849	2.93494579405095\\
0.0353535353535354	2.97557999181029\\
0.0358585858585859	3.01618901573389\\
0.0363636363636364	3.05677326718814\\
0.0368686868686869	3.097333136259\\
0.0373737373737374	3.13786900221607\\
0.0378787878787879	3.17838123395164\\
0.0383838383838384	3.2188701903963\\
0.0388888888888889	3.25933622091266\\
0.0393939393939394	3.29977966566864\\
0.0398989898989899	3.34020085599156\\
0.0404040404040404	3.38060011470436\\
0.0409090909090909	3.42097775644498\\
0.0414141414141414	3.46133408796992\\
0.0419191919191919	3.50166940844307\\
0.0424242424242424	3.54198400971056\\
0.0429292929292929	3.5822781765625\\
0.0434343434343434	3.6225521869825\\
0.0439393939393939	3.66280631238546\\
0.0444444444444444	3.70304081784447\\
0.044949494949495	3.74325596230747\\
0.0454545454545455	3.78345199880401\\
0.045959595959596	3.823629174643\\
0.0464646464646465	3.86378773160158\\
0.046969696969697	3.90392790610595\\
0.0474747474747475	3.94404992940431\\
0.047979797979798	3.98415402773244\\
0.0484848484848485	4.02424042247239\\
0.048989898989899	4.06430933030445\\
0.0494949494949495	4.10436096335294\\
0.05	4.14439552932598\\
};

\addplot [color=mycolor2, line width=1.5pt]
  table[row sep=crcr]{%
0	1\\
0.000505050505050505	1\\
0.00101010101010101	1\\
0.00151515151515152	1\\
0.00202020202020202	1\\
0.00252525252525253	1\\
0.00303030303030303	1\\
0.00353535353535354	1\\
0.00404040404040404	1\\
0.00454545454545455	1\\
0.00505050505050505	1\\
0.00555555555555556	1\\
0.00606060606060606	1\\
0.00656565656565657	1\\
0.00707070707070707	1\\
0.00757575757575758	1\\
0.00808080808080808	1\\
0.00858585858585859	1\\
0.00909090909090909	1\\
0.0095959595959596	1\\
0.0101010101010101	1\\
0.0106060606060606	1\\
0.0111111111111111	1\\
0.0116161616161616	1\\
0.0121212121212121	1\\
0.0126262626262626	1\\
0.0131313131313131	1\\
0.0136363636363636	1\\
0.0141414141414141	1\\
0.0146464646464646	1\\
0.0151515151515152	1\\
0.0156565656565657	1\\
0.0161616161616162	1\\
0.0166666666666667	1\\
0.0171717171717172	1\\
0.0176767676767677	1\\
0.0181818181818182	1\\
0.0186868686868687	1\\
0.0191919191919192	1\\
0.0196969696969697	1\\
0.0202020202020202	1\\
0.0207070707070707	1\\
0.0212121212121212	1\\
0.0217171717171717	1\\
0.0222222222222222	1\\
0.0227272727272727	1\\
0.0232323232323232	1\\
0.0237373737373737	1\\
0.0242424242424242	1\\
0.0247474747474748	1\\
0.0252525252525253	1\\
0.0257575757575758	1\\
0.0262626262626263	1\\
0.0267676767676768	1\\
0.0272727272727273	1\\
0.0277777777777778	1\\
0.0282828282828283	1\\
0.0287878787878788	1\\
0.0292929292929293	1\\
0.0297979797979798	1\\
0.0303030303030303	1\\
0.0308080808080808	1\\
0.0313131313131313	1\\
0.0318181818181818	1\\
0.0323232323232323	1\\
0.0328282828282828	1\\
0.0333333333333333	1\\
0.0338383838383838	1\\
0.0343434343434343	1\\
0.0348484848484849	1\\
0.0353535353535354	1\\
0.0358585858585859	1\\
0.0363636363636364	1\\
0.0368686868686869	1\\
0.0373737373737374	1\\
0.0378787878787879	1\\
0.0383838383838384	1\\
0.0388888888888889	1\\
0.0393939393939394	1\\
0.0398989898989899	1\\
0.0404040404040404	1\\
0.0409090909090909	1\\
0.0414141414141414	1\\
0.0419191919191919	1\\
0.0424242424242424	1\\
0.0429292929292929	1\\
0.0434343434343434	1\\
0.0439393939393939	1\\
0.0444444444444444	1\\
0.044949494949495	1\\
0.0454545454545455	1\\
0.045959595959596	1\\
0.0464646464646465	1\\
0.046969696969697	1\\
0.0474747474747475	1\\
0.047979797979798	1\\
0.0484848484848485	1\\
0.048989898989899	1\\
0.0494949494949495	1\\
0.05	1\\
};

\addplot [color=mycolor3, dashed, line width=1.5pt]
  table[row sep=crcr]{%
0.0112	0\\
0.0112	0.5\\
0.0112	1\\
0.0112	1.5\\
0.0112	2\\
0.0112	2.5\\
0.0112	3\\
0.0112	3.5\\
0.0112	4\\
0.0112	4.5\\
};

\end{axis}
\end{tikzpicture}%
\caption{\label{logEtaFigure}$f_1(\eta) = 1$ and $f_2(\eta)= 42\eta + 4\eta\log(1+\frac{C_1-1}{2\eta} + C_1)$ for $\eta\in [0, 0.05]$. The lower bound for $\eta$ is achieved at $\eta^* = 0.0112$.}
\end{figure}
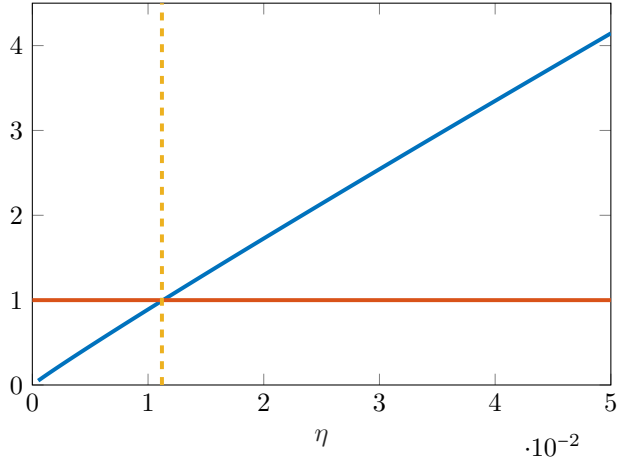

\end{proof}

The next lemma uses the eigenpolynomial $p$ to derive an approximate eigenpolynomial on the restriction $\mathcal{M}_1$ of the whole operator $\mathcal{M}$ to the atom located at $0$ (or equivalently $\tau = \tau_0$).

\begin{restatable}{lemma}{lemmaTroncaturetwoMultiSpikesKey}
\label{lemma:lemmaTroncature2MultiSpikes}
  Let us assume that there exists a real number $\lambda\in[0,9;1[$ and a polynomial $p$ with $||p||_\infty=1$ and such that
  \begin{equation*}
    \mathcal{M}(p)=\lambda p.
  \end{equation*}
  For all $K\in\{1,\dots,n\}$, we define $p_K$ and $p_{K,err}$ as in Lemma~\ref{lemma:lemmaTroncatureMultiSpikes}. 

We use $P_K$ to denote the projector that maps (orthogonally in the sense of the $\ell^2$ inner product) any polynomial on the subspace generated by the $D_0\left(\theta-\frac{k}{2n+1}\right)$ pour $k=-K,\dots,K$.
    
 We then have 
  \begin{align}
    ||P_K\mathcal{M}_1(p_K)-\lambda p_K||_{\ell^2} & \leq \frac{c_5}{\sqrt{n}}\frac{\log^2(1+K)}{1+K}\label{boundTruncationOpMk}
  \end{align}
  as well as 
  \begin{align}
    ||p_K||_{\ell^2} \geq \frac{c_4}{\sqrt{n}}.\label{boundOnTruncatedPK}
  \end{align}

If we use $q_K$ the vector encoding the coefficients of the polynomial $p_K$ in the basis generated by the polynomials $D_0\left(\theta-\frac{k}{2n+1}\right)$ for $k=-K,\dots,K$, and let $\mathcal{Q}_K$ to denote the matrix $P_K\mathcal{M}(p_K)$ in that same basis, the vector $q_K$ as well as the operator $Q_K$ must obey the following properties 
  \begin{enumerate}
  \item\label{item1LemmaMatVecTronques} $||q_K||_2 \geq c_4$.
  \item\label{item2LemmaMatVecTronques}$||\mathcal{Q}_Kq_K-\lambda q_K||_2 \leq c_5 \frac{\log^2(1+K)}{1+K}$.
  \item $q_K(0)=0$. \label{item:orth_qK_1}
  \item $|\scal{q_K}{v_K}|\leq \frac{c_6}{1+K}||v_K||$, où $v_K$ est le vecteur défini par
    \begin{align*}
      \forall k\in\{-K,\dots,K\}-\{0\},&\quad v_K(k)=\frac{(-1)^{k+1}}{\sin\left(\pi\frac{k}{2n+1}\right)},\\
      &\quad v_k(0) = 0.
    \end{align*} \label{item:orth_qK_2}
  \end{enumerate}
\end{restatable}

\begin{proof}

We start by showing~\eqref{boundTruncationOpMk} and~\eqref{boundOnTruncatedPK}. First note that for any polynomial $q(e^{2\pi i \theta}) = \sum_{k=-n}^n q_ke^{2\pi i k \theta}$, we have
\begin{align}
P_K q &= \sum_{k=-K}^{K} q(e^{2\pi i \frac{k}{2n+1}}) D_0(\theta - \frac{k}{2n+1})\\
& = \sum_{k=-K}^{K} q(e^{2\pi i \frac{k}{2n+1}})  \sum_{s=-n}^{n} \frac{1}{2n+1} e^{2\pi i (\theta - \frac{k}{2n+1})s}\\
& = \sum_{s = -n}^{n} e^{2\pi i \theta s} \left\{\sum_{k=-K}^{K} q(e^{2\pi i \frac{k}{2n+1}}) \frac{1}{2n+1} e^{-2\pi i \frac{k}{2n+1}s}\right\}
\end{align}
From this, the $\ell_2$ norm of the polynomial $P_K q$ can read as 
\begin{align}
\|P_K q\|_{\ell_2} &= \sqrt{\sum_{s=-n}^n \left|\left\{\sum_{k=-K}^{K} q(e^{2\pi i \frac{k}{2n+1}}) \frac{1}{2n+1} e^{2\pi i\left(-\frac{k}{2n+1}\right)}\right\}\right|^2}\label{BoundNorml2MPK}
\end{align}
Now note that the norm
\begin{align}
\left|\mathcal{P}_K\mathcal{M}_1p_K - \lambda p_K\right|&\leq \left|\mathcal{P}_K(\mathcal{M}_1 - \mathcal{M})p - \mathcal{P}_K\mathcal{M}_1 p_{K, \err} + \mathcal{P}_K\mathcal{M}p - \lambda p_K\right|\\
&\leq \left|\mathcal{P}_K(\mathcal{M}_1 - \mathcal{M})p\right| + \left|\mathcal{M}_1 p_{K,\err}\right|\\
&\leq O(\frac{\log^2|S|}{n\Delta}) \\
&+ \frac{(c_1+c_1'+C_1/\log(K+1))(2M_1''' +2 M_2\log(1+K)+1)\log(1+K)}{1+K}\frac{1}{1+n|\theta|} \label{boundM1minusLambdaPk}\\
&\leq c_5\frac{\log^2(1+K)}{1+K}\frac{1}{1+n|\theta|}\label{boundM1pKminusLambdapK}.
\end{align}
~\eqref{boundM1minusLambdaPk} follows from~\eqref{boundM1pKerrConstant} and the result of lemma~\ref{lemma:lemmaTroncatureMultiSpikes}. Taking $q(\theta) =( \mathcal{P}_K(\mathcal{M}_1 - \mathcal{M})p)(\theta)$ and substituting~\eqref{boundM1pKminusLambdapK} into~\eqref{BoundNorml2MPK}, we get 
\begin{align}
\left\|\mathcal{P}_K \mathcal{M}_1 p_K - \lambda p _K\right\|_{\ell_2}&\leq \frac{1}{\sqrt{2n+1}}c_5\frac{\log^2(1+K)}{1+K} \log(3+K)  + O(\frac{\log^2|S|}{n\Delta})
\end{align}
item~\ref{item2LemmaMatVecTronques} above follows the same idea. Since we compare the coefficients in the Dirichlet basis, the log disappears. Since $(\mathcal{Q}_Kq_K)_k = \left(\mathcal{P}_K\mathcal{M}_1p_K\right)(\frac{k}{2n+1})$, we can write 
\begin{align}
\left\|\mathcal{Q}_Kq_K - \lambda q_K\right\|_2^2 &\leq \sum_{k=-K}^K \left|(\mathcal{P}_K\mathcal{M}_1p_K)(\frac{k}{2n+1}) - \lambda p_K(\frac{k}{2n+1})\right|^2\\
&\leq \sum_{k=-K}^{K}\left| c_5\frac{\log^2(1+K)}{1+K}\frac{1}{1+n\frac{k}{2n+1}}\right|^2\\
&\leq \sum_{k=-K}^{K}\left| c_5\frac{\log^2(1+K)}{1+K}\frac{2}{2+k}\right|^2\\
&\leq 4c_5^2\left(\frac{\log^4(1+K)}{(1+K)^2}\right)\zeta(2)
\end{align}
for $n$ sufficiently large. From which we have 
\begin{align}
\left\|\mathcal{Q}_Kq_K - \lambda q_K\right\|_2&\leq 2c_5\left(\frac{\log^2(1+K)}{(1+K)}\right)\sqrt{\zeta(2)}\\
&\leq 2\sqrt{\zeta(2)} \frac{(c_1+c_1'+C_1/\log(K+1))(2M_1''' +2 M_2\log(1+K)+1)\log(1+K)}{1+K}
\end{align}

Item~\ref{item1LemmaMatVecTronques} follows from the lower bound on $\eta$ derived in lemma~\ref{lemmaTroncatureItem7} (see item~\ref{lemma:lemmaTroncatureMultiSpikes} in that same lemma or Fig.~\ref{logEtaFigure}). Together those imply $\|q_K\|_2\geq .0112$.

Items~\ref{item:orth_qK_1} and~\ref{item:orth_qK_2} are direct consequences of the defintion of $\mathcal{A}\tilde{\mathcal{A}}^*$. From the definition of the eigenpolynomial $p(e^{2\pi i \theta})$, we have 
\begin{align}
\lambda p(\tau) =
\mathcal{A}\tilde{\mathcal{A}}[p](\tau)  = \psi(\tau)^*\mathcal{P}_U^\perp \tilde{\mathcal{T}}[p]\mathcal{P}_U^\perp\psi(\tau) = 0, \quad \text{for any $\tau\in S$}
\end{align}
as well as 
\begin{align}
\lambda p'(\theta) &= (\psi(\theta)^*)'\mathcal{P}_U^\perp \tilde{T}^*[p] \mathcal{P}_U^\perp \psi(\theta) + \psi(\theta)^*\mathcal{P}_U^\perp \tilde{T}^*[p] \mathcal{P}_U^\perp \psi'(\theta)
\end{align}
Taking this last equation at $\theta = \tau$, we get 
\begin{align}
\lambda p'(\tau) &= \left.(\psi(\theta)^*)'\mathcal{P}_U^\perp \tilde{T}^*[p] \mathcal{P}_U^\perp \psi(\theta) + \psi(\theta)^*\mathcal{P}_U^\perp \tilde{T}^*[p] \mathcal{P}_U^\perp \psi'(\theta)\right|_{\theta = \tau} = 0.
\end{align}

%
%

Taking $\tau=\tau_0 = 0$ implies 
  \begin{align*}
    0 &= \sum_{k=-n}^n p(e^{2\pi i \frac{k}{2n+1}})D_0'\left(-\frac{k}{2n+1}\right)\\
      &= \pi \sum_{k=-n}^n p(e^{2\pi i \frac{k}{2n+1}})\frac{(-1)^{k+1}}{\sin\left(\frac{\pi k}{2n+1}\right)}.
  \end{align*}
  From which we recover the first part of item~\ref{item:orth_qK_2} with $(v_K)_k = \frac{(-1)^{k+1}}{\sin\left(\frac{\pi k}{2n+1}\right)}$. To conclude, we use the result of lemma~\ref{lemma:boundK} together with the lower bound $|\sin(\pi x)|\geq 2|x|$, $x\in[-\frac{1}{2}, \frac{1}{2}]$. We start by bounding the contributions arising from the atoms $\tau>0$, the contribution $\tau<0$ can be bounded following the exact same approach. We consider the following decomposition,

\begin{align}
\sum_{|k|\geq K} \min(1, \sum_{\tau>0} \frac{C_1}{1+n\left|\frac{k}{2n+1} - \tau\right|}) \frac{\pi(2n+1)}{2k}
&\leq \sum_{\tau>0} \sum_{k = k_\tau - C_1-1}^{k_\tau + C_1+1} \frac{\pi(2n+1)}{2k}\label{term100001}\\
&+ \sum_{\tau >0} \sum_{k=K}^{k_\tau - C_1-1} \frac{C_1}{1+n(\tau - \frac{k}{2n+1})} \frac{\pi(2n+1)}{2k}\label{term100002}\\
&+ \sum_{\tau>0} \sum_{k=-n}^{-K} \frac{C_1}{1+n(\tau - \frac{k}{2n+1})} \frac{\pi(2n+1)}{2(-k)}\label{term100003}\\
&+\sum_{\tau>0} \sum_{k = k_\tau+C_1+1}^{n} \frac{C_1}{1+n(\frac{k}{2n+1} - \tau)}\frac{\pi (2n+1)}{2k}\label{term100004}
\end{align}

Using the bounds on the harmonic numbers, we can write the sum as
\begin{align}
&\sum_{|k|\geq K} \min(1, \sum_{\tau>0} \frac{C_1}{1+n\left|\frac{k}{2n+1} - \tau\right|}) \frac{\pi(2n+1)}{2k}\\
&\lesssim \frac{\pi(2n+1)}{2} \frac{4}{(2n+1)\Delta} \log|S| + \frac{\pi}{2}\frac{8(C_1-1)}{(2n+1)\Delta}\\
&+ \sum_{\tau>0} \sum_{k=K}^{k_\tau - C_1-1} \frac{\pi C_1 (2n+1)}{2} \left(\frac{n/(2n+1)}{1+n(\tau - \frac{k}{2n+1})} + \frac{1}{k}\right)\frac{1}{1+n\tau}\\
&+ \sum_{\tau>0} \sum_{k<-K} \frac{\pi C_1(2n+1)}{2} \left(\frac{n/(2n+1)}{1+n(\tau - \frac{k}{2n+1})} - \frac{1}{-k}\right) \frac{1}{1+n\tau}\\
&+ \sum_{\tau>0} \sum_{k=k_\tau+C_1+1}^n \frac{\pi C_1(2n+1)}{2} \left(\frac{n/(2n+1)}{1+n(\frac{k}{2n+1} - \tau)} - \frac{1}{k}\right)\frac{1}{n\tau - 1}\\
&\leq  \frac{\pi(2n+1)}{2} \frac{4}{(2n+1)\Delta}\log|S| + \frac{\pi}{2} \frac{8(C_1-1)}{(2n+1)\Delta}\\
&+\sum_{\tau>0} \frac{\pi C_1(2n+1)}{2} \left(\frac{1}{C_1-1} - \int_{K}^{k_\tau - C_1-1} \frac{-\frac{n}{2n+1}}{1+n\left(\tau - \frac{k}{2n+1}\right)}\; dk\right)\frac{1}{1+n\tau}\\
&+ \sum_{\tau>0}\frac{\pi C_1(2n+1)}{2} \left( \frac{1}{K} + \int_{K}^{k_\tau - C_1-1} \frac{1}{k}\; dk\right)\frac{1}{1+n\tau}\\
&+ \sum_{\tau>0} \frac{\pi C_1(2n+1)}{2} \left(\frac{1}{C_1-1} - \int_{-n}^{-K} \frac{-\frac{n}{2n+1}}{1+n(\tau - \frac{k}{2n+1})}\; dk\right)\frac{1}{1+n\tau}\\
&+ \sum_{\tau>0} \frac{\pi C_1(2n+1)}{2} \left(\frac{1}{K} + \int_{-n}^{-K} \frac{-1}{k}\; dk\right)\frac{1}{1+n\tau}\\
&+\sum_{\tau>0} \frac{\pi C_1 (2n+1)}{2} \left(\frac{1}{C_1-1} + \int_{k_\tau + C_1-1}^n \frac{\frac{n}{2n+1}}{1+n\left(\frac{k}{2n+1} - \tau\right)}\; dk\right)\frac{1}{n\tau-1}\\
&+ \sum_{\tau>0} \frac{\pi C_1(2n+1)}{2} \left(-\int_{k_\tau}^{K} \frac{1}{k}\; dk\right)\frac{1}{n\tau-1}\\
&\leq  \frac{\pi(2n+1)}{2} \frac{4}{(2n+1)\Delta}\log|S| + \frac{\pi}{2} \frac{8(C_1-1)}{(2n+1)\Delta}\\
&+\sum_{\tau>0} \frac{\pi C_1 (2n+1)}{2} \left(\frac{1}{C_1-1} + \frac{1}{K} -\log(\frac{1/n + (C_1-1)/(2n+1)}{1/n + \tau - \frac{K}{2n+1}}) + \log(\frac{\tau - \frac{C_1-1}{2n+1}}{\frac{K}{2n+1}})\right)\frac{1}{1+n\tau}\label{logTermDerivativeDirichlet1}\\
&+ \sum_{\tau>0} \frac{\pi C_1(2n+1)}{2} \left(\frac{1}{C_1-1} + \frac{1}{K} - \log(\frac{1/n + \tau + \frac{K}{2n+1}}{1/n + \tau + 1/2}) + \log(\frac{K}{n})\right)\frac{1}{1+n\tau}\label{logTermDerivativeDirichlet2}\\
&+\sum_{\tau>0} \frac{\pi C_1(2n+1)}{2} \left(\frac{1}{C_1-1} + \log(\frac{1/n + 1/2-\tau}{1/n + \frac{C_1-1}{2n+1}}) - \log(\frac{n}{k_\tau - C_1-1})\right)\frac{1}{n\tau-1}\label{logTermDerivativeDirichlet3}
\end{align}
The logs appearing in the sums~\eqref{logTermDerivativeDirichlet1} to~\eqref{logTermDerivativeDirichlet3} can finally be controled by noting that
\begin{align}
\eqref{logTermDerivativeDirichlet1}&\leq \sum_{\tau>0} \frac{\pi C_1 (2n+1)}{2} \left[\frac{1}{C_1-1} + \frac{1}{K} + \log(1+\frac{(2n+1)\tau - K}{C_1-1}) + \log(\frac{(2n+1)\tau - C_1-1}{K})\right] \frac{1}{1+n\tau}\\
&\lesssim \sum_{\tau>0} (2n+1)\frac{1}{n|\tau|}\log(n|\tau|)\\
\eqref{logTermDerivativeDirichlet2}&\leq \sum_{\tau>0} \frac{\pi(2n+1)C_1}{2} \left(\frac{1}{C_1-1} + \frac{1}{K} + \log(1+ \frac{\tau(2n+1) + 1}{K}) + \log(2)\right)\\
&\lesssim \sum_{\tau>0} (2n+1)\frac{1}{n|\tau|}\log(n|\tau|)
\end{align}
For~\eqref{logTermDerivativeDirichlet3}, first note that $(1/2 -\tau+1/n) (\tau - \frac{C-1-1}{n})>1/n$ hence $\log(\frac{(1/n + 1/2-\tau) (\tau - \frac{C_1-1}{2n+1})}{1/n + \frac{C_1-1}{2n+1}})>0$. In particular, we have $\log(\frac{(1/n + 1/2-\tau) (\tau - \frac{C_1-1}{2n+1})}{1/n + \frac{C_1-1}{2n+1}}) \lesssim \log(n|\tau|)$, from which one can write 
\begin{align}
\eqref{logTermDerivativeDirichlet2}&\leq \sum_{\tau>0} \frac{\pi C_1 (2n+1)}{2} \left[\frac{1}{C_1-1} + \log(n|\tau|)\right] \frac{1}{n\tau - 1}
\end{align}
Combining those bounds, one can thus bound the whole sum~\eqref{term100001} to~\eqref{term100004} as 
\begin{align}
\sum_{|k|\geq K} \min(1, \sum_{\tau>0} \frac{C_1}{1+n\left|\frac{k}{2n+1} - \tau\right|}) \frac{\pi(2n+1)}{2k}
&\lesssim (2n+1) \frac{\log^2|S| + \log|S|\log(n\Delta)}{n\Delta}
\end{align}
The term arising from the atom used to define the truncations, $p_K, \mathcal{M}_K$ can be bounded as
\begin{align}
\sum_{|k|\geq K} \frac{C_1}{1+n \left|\frac{k}{2n+1}\right|} \frac{(2n+1)\pi}{2k} \leq (2n+1)\frac{3C_1 \pi}{K}. 
\end{align}

%
Using the bounds above, we can bound the inner product $|\langle q_K, v_K \rangle|\leq ( \frac{3C_1\pi}{2K} + \varepsilon)(2n+1)$. On the other hand, using upper and lower bounds on the sine, we can write 
\begin{align}
\frac{2n+1}{\pi}\leq \sqrt{\sum_{|k|\leq n}  \left|\frac{1}{\pi\frac{k}{2n+1}}\right|^2} \leq \|v_K\|_{\ell_2} &= \sqrt{\sum_{|k|\leq n} \left|\frac{(-1)^k}{\sin(\pi \frac{k}{2n+1})}\right|^2}\leq \sqrt{\sum_{|k|\leq n}  \left|\frac{1}{2\frac{k}{2n+1}}\right|^2}\leq \sqrt{\zeta(2)} \frac{2n+1}{2}
\end{align}
Combining this with the bound on the inner product $|\langle v_K, q_K\rangle|$ derived above, we get 
\begin{align}
\left|\langle q_K, v_K\rangle \right|&\leq \|v_K\|_{2} \pi \frac{3C_1 \pi}{K}
\end{align}

This concludes the proof of lemma~\ref{lemma:lemmaTroncature2MultiSpikes}.

%
%

\end{proof}

Moreover, if we use $q$ to denote the normalized vector $q_K/\|q_K\|_2$, we have the bound
\begin{align}
\|\mathcal{Q}_Kq - \lambda q\|_2&\leq 2\sqrt{\zeta(2)} \frac{(c_1+c_1'+C_1/\log(K+1))(2M_1''' +2 M_2\log(1+K)+1)\log(1+K)}{1+K}\|q_K\|^{-1}_2\\
&\leq 2(.0112)^{-1}\sqrt{\zeta(2)} \frac{(c_1+c_1'+C_1/\log(K+1))(2M_1''' +2 M_2\log(1+K)+1)\log(1+K)}{1+K}
\end{align}
The second line follows from the estimate $|p(e^{2\pi i \frac{k}{2n+1}})|\geq \eta$ of lemma~\ref{lemma:lemmaTroncatureMultiSpikes}. 
\begin{align}
\|(\mathcal{Q}_K - \Id )q \|_2&\leq 0.1 + 2(.0112)^{-1}\sqrt{\zeta(2)} \frac{(c_1+c_1'+C_1/\log(K+1))(2M_1''' +2 M_2\log(1+K)+1)\log(1+K)}{1+K}
\end{align}
As shown in Fig.~\ref{finalBoundOnK}, taking $K = 1e13$ gives an upper bound on the deviation $\|(\mathcal{Q}_K - \Id)  q_K\|\leq 0.1 + .015$

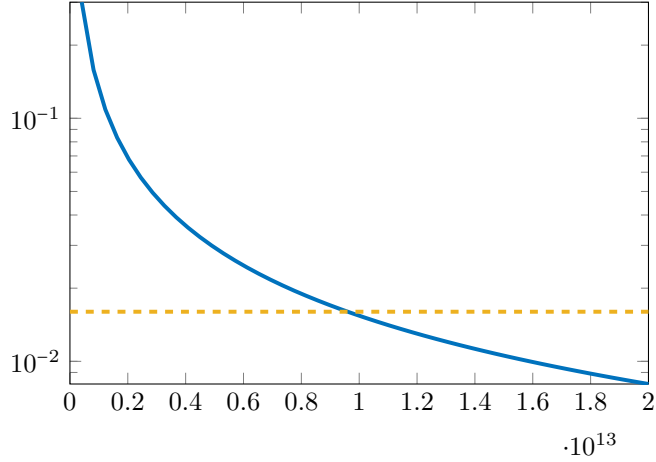
\begin{figure}
\centering
%
%
\definecolor{mycolor1}{rgb}{0.00000,0.44700,0.74100}%
\definecolor{mycolor2}{rgb}{0.85000,0.32500,0.09800}%
\definecolor{mycolor3}{rgb}{0.92900,0.69400,0.12500}%
\begin{tikzpicture}

\begin{axis}[%
width=3.028in,
height=2in,
at={(1.011in,0.642in)},
scale only axis,
unbounded coords=jump,
xmin=0,
xmax=20000000000000,
ymode=log,
ymin=0.00806752452784543,
ymax=0.300463472262949,
yminorticks=true,
axis background/.style={fill=white},
legend style={legend cell align=left, align=left, draw=white!15!black}
]
\addplot [color=mycolor1, line width=1.5pt]
  table[row sep=crcr]{%
1	nan\\
408163265307.102	0.300463472262949\\
816326530613.204	0.158205124585731\\
1224489795919.31	0.108642825467797\\
1632653061225.41	0.0831917005898\\
2040816326531.51	0.0676239259499021\\
2448979591837.61	0.0570874555580331\\
2857142857143.71	0.0494673305888693\\
3265306122449.82	0.0436916414782799\\
3673469387755.92	0.0391580980051531\\
4081632653062.02	0.0355017877713484\\
4489795918368.12	0.0324884993690769\\
4897959183674.22	0.029960897214361\\
5306122448980.33	0.0278093052354992\\
5714285714286.43	0.0259548913203706\\
6122448979592.53	0.024339511988896\\
6530612244898.63	0.0229193273482885\\
6938775510204.73	0.021660646824965\\
7346938775510.84	0.0205371455015342\\
7755102040816.94	0.0195279504338301\\
8163265306123.04	0.0186162950978402\\
8571428571429.14	0.0177885542494961\\
8979591836735.24	0.0170335392266148\\
9387755102041.35	0.0163419751337809\\
9795918367347.45	0.0157061073364924\\
10204081632653.6	0.0151194013844557\\
10612244897959.7	0.014576311438057\\
11020408163265.8	0.014072099599707\\
11428571428571.9	0.0136026935402731\\
11836734693878	0.0131645732618596\\
12244897959184.1	0.012754680261129\\
12653061224490.2	0.0123703440817588\\
13061224489796.3	0.0120092224874485\\
13469387755102.4	0.0116692523931471\\
13877551020408.5	0.0113486093602298\\
14285714285714.6	0.0110456739588473\\
14693877551020.7	0.0107590036746726\\
15102040816326.8	0.0104873093209534\\
15510204081632.9	0.0102294351337505\\
15918367346939	0.0099843418955065\\
16326530612245.1	0.00975109256198718\\
16734693877551.2	0.00952883996923105\\
17142857142857.3	0.00931681627711886\\
17551020408163.4	0.00911432386953722\\
17959183673469.5	0.00892072748160392\\
18367346938775.6	0.00873544736489184\\
18775510204081.7	0.00855795333419716\\
19183673469387.8	0.00838775956580778\\
19591836734693.9	0.0082244200387223\\
20000000000000	0.00806752452784543\\
};

\addplot [color=mycolor3, dashed, line width=1.5pt]
  table[row sep=crcr]{%
1	0.016\\
1052631578948.32	0.016\\
2105263157895.63	0.016\\
3157894736842.95	0.016\\
4210526315790.26	0.016\\
5263157894737.58	0.016\\
6315789473684.89	0.016\\
7368421052632.21	0.016\\
8421052631579.53	0.016\\
9473684210526.84	0.016\\
10526315789474.2	0.016\\
11578947368421.5	0.016\\
12631578947368.8	0.016\\
13684210526316.1	0.016\\
14736842105263.4	0.016\\
15789473684210.7	0.016\\
16842105263158.1	0.016\\
17894736842105.4	0.016\\
18947368421052.7	0.016\\
20000000000000	0.016\\
};

\end{axis}

\end{tikzpicture}%
\caption{\label{finalBoundOnK}$f(K) = 2(.0112)^{-1}\sqrt{\zeta(2)} \frac{(c_1+c_1'+C_1/\log(K+1))(2M_1''' +2 M_2\log(1+K)+1)\log(1+K)}{1+K}$.}
\end{figure}

For that value of $K$, the proof of lemma~\ref{lemma:lemmaTroncature2MultiSpikes} also gives the upper bound $|\langle q, v_K \rangle| = |\langle \frac{q_K}{\|q_K\|}, v_K \rangle| \leq \|q_K\|^{-1}\|v_K\|_2 \frac{3\pi^2C_1}{K}\leq \|v_K\|1e^{-6}$. From this, as soon as we assume an eigenvalue of $\mathcal{A}\tilde{\mathcal{A}}^*$ smaller than $.1$, the following four conditions must necessairily be satisfied

\begin{enumerate}
\item \label{propertyFinalKNonAssympt1} $||q||_2=1$.
\item \label{propertyFinalKNonAssympt2} $||(\mathcal{Q}_K-\mathop{Id})q||_2 \leq 0,1 + \epsilon_1$.
\item \label{propertyFinalKNonAssympt3} $q(0)=0$.
\item \label{propertyFinalKNonAssympt4} $|\scal{q}{v_K}|\leq \varepsilon_2 ||v_K||$.
\end{enumerate}

where $\varepsilon_1 = .016$ and $\varepsilon_2 = 1e^{-6}$. 

%
%

As we have assumed that for all $N$, there always exists a $f_c = n\geq N$ with $\lambda(\mathcal{M})\geq 0.9$
(i.e. by contradiction of $\exists$ $\Omega_c$ or $N$ such that for all $n\geq N$, $\lambda(\mathcal{M})\leq 0.9$), we can always build a subsequence $n_1, n_2, \ldots, n_R$ with corresponding $(Q_K^n, v_K^n)$ such that properties~\ref{propertyFinalKNonAssympt1} to~\ref{propertyFinalKNonAssympt4} above are satisfied for each element in the sequence. In particular, that subsequence necessarily converges to a pair $(Q_K^n, v_K^n)$ satisfying those same properties and since $\lim_{R\rightarrow \infty} (Q_K^{n_R}, v_K^{n_R})$ is well defined, the limit $\lim_{n\rightarrow \infty} (Q_K^n, v_K^n)$ also satisfies properties~\ref{propertyFinalKNonAssympt1} to~\ref{propertyFinalKNonAssympt4} (i.e. The limit of a convergent sequence in a topological space equals the limit of any subsequence of it.). The corresponding vector $q_\infty\in\mathbb{R}^{2K+1}$ must then satisfy


\begin{enumerate}
\item \label{propertyFinalKAssympt1} $||q_\infty||=1$.
\item \label{propertyFinalKAssympt2} $||(\mathcal{Q}_K^\infty-\mathop{Id})q_\infty||_2\leq 0,1 + \epsilon$. 
\item \label{propertyFinalKAssympt3} $q_\infty(0)=0$.
\item \label{propertyFinalKAssympt4} $|\scal{q_\infty}{w_\infty}|\leq \epsilon$, 
\end{enumerate}

Where the limit $w_\infty \equiv \lim_{n\rightarrow \infty} \frac{v_K}{\|v_K\|}$ can be derived by first noting that
\begin{align}
\frac{v_K(k)}{\|v_K\|} = \frac{(-1)^{k+1}}{\sin(\pi \frac{k}{2n+1})} \frac{1}{\sqrt{\sum_{k\neq 0} \frac{1}{\sin^2(\pi \frac{k}{2n+1})}}} = \frac{(-1)^{k+1}}{\pi \frac{k}{2n+1} - \left(\frac{\pi k}{2n+1}\right)^3 + h.o.t} \frac{1}{\sqrt{\sum_{k\neq 0} \frac{1}{\left(\frac{\pi k}{2n+1}\right)^2 + h.o.t}}}
\end{align}
and then taking the limit, which gives
\begin{align}
v_K^\infty(k) & = \left\{\begin{array}{ll}
\frac{1}{k} \frac{(-1)^k}{\sqrt{\sum_{k\neq 0} \frac{1}{k^2}}} & \text{for $k\neq 0$} \\
 0 & \text{otherwise}
\end{array}\right.
\end{align}

From this, if we let $P^\infty\in\mathbb{R}^{\{-K,\dots,K\}}$ to denote the orthogonal projector on $\text{span}(e_0,w_\infty)$ ($e_0$ is the zero vector in which the coordinate corresponding to the zero frequency has been set to $1$). From items~\ref{propertyFinalKAssympt1} to~\ref{propertyFinalKAssympt4} above, we can write $||P^\infty q_\infty||\leq \epsilon$ and the asymptotic vector $q^{\infty}$ as well as the asymptotic operator $\mathcal{Q}_K\infty$ must satisfy
\begin{align*}
  ||(\mathop{Id}-\mathcal{Q}_K^\infty+P^\infty)q_\infty||
  & \leq ||(\mathop{Id}-\mathcal{Q}_K^\infty)q_\infty|| + ||P^\infty q_\infty||\\
  & \leq 0,1 + \epsilon_1 + \varepsilon_2.
\end{align*}
From the definition of the least singular value, $\sigma_{n}(A) = \min_{\|\bs x\|=1}\| A \bs x\|$, implies that the operator $(Id - \mathcal{Q}^\infty_K + P^\infty)$ must have at least one singular value smaller than $0.1 + 2\varepsilon$. Disproving the lower bound on the eigenvalues of $\mathcal{A}\tilde{\mathcal{A}}^*$ (hence concluding the contradiction) can thus be done by proving that the asymptotic matrix $(Id - \mathcal{Q}^\infty_K + P^\infty)$ has its singular values above $0.1+\varepsilon$.

In practice, all the singular values of this truncated matrix are above $0.5$. Verifying this on a matrix of size $K = 1e^{13}$ is challenging. The next two sections focus on reducing the numerical complexity by (1) using the decay in the entries of the asymptotic matrix $\mathcal{Q}^\infty_K$ to further truncate this matrix and (2) write the product $\mathcal{Q}_K^\infty q$ for any vector $q$ as convolutions (which can then be applied through FFTs) with special functions. The first simplification is carried out in section~\ref{entryDecay}. The second is discussed and then used to perform power iterations in section~\ref{sec:PowerIterations}. 

\subsection{\label{asymptoticOperator}Asymptotic operator}

We now derive the expression of the asymptotic operator $\mathcal{Q}_K^\infty$. The matrix 
\begin{align}
Q_K^\infty &= \lim_{n\rightarrow \infty} Q_K^n = \lim_{n\rightarrow \infty} \mathcal{P}_K
\end{align} 
is computed from the one atomic operator by setting $p = D_0(\theta - \frac{\ell_2}{2n+1})$ and taking the value of this operator at $\theta = \frac{\ell_1}{2n+1}$. To compute the limit $\lim_{n\rightarrow }Q_K^n[\ell_1, \ell_2]$ of each $[\ell_1, \ell_2]$ entry, we go back to the expression of the one atomic operator derived in section~\ref{boundAAoneatomic} and write 
\begin{align}
Q_K^\infty[\ell_1,\ell_2] & = \lim_{n\rightarrow \infty} \overline{D}(e^{2\pi i \frac{\ell_1}{2n+1}})u^*\tilde{T}[D_0(s - \frac{\ell_2}{2n+1})] \psi(\frac{\ell_1}{2n+1})\\
&+ \lim_{n\rightarrow \infty} D(e^{2\pi i \frac{\ell_1}{2n+1}}) \psi(\frac{\ell_1}{2n+1}) \tilde{T}^*[D_0(s - \frac{\ell_2}{2n+1})] u\\
& - \lim_{n\rightarrow \infty} |D(e^{2\pi i \frac{\ell_1}{2n+1}})|^2 D_0(-\frac{\ell_2}{2n+1})\\
& = \lim_{n\rightarrow \infty} \overline{D}(e^{2\pi i \frac{\ell_1}{2n+1}}) \lim_{n\rightarrow \infty} u^*\tilde{T}[D_0(s - \frac{\ell_2}{2n+1})] \psi(\frac{\ell_1}{2n+1})\\
&+ \lim_{n\rightarrow \infty} D(e^{2\pi i \frac{\ell_1}{2n+1}}) \lim_{n\rightarrow \infty} \psi(\frac{\ell_1}{2n+1}) \tilde{T}^*[D_0(s-\frac{\ell_2}{2n+1})]u\\
& - \lim_{n\rightarrow \infty} \left|D(e^{2\pi i \frac{\ell_1}{2n+1}})\right|^2 D_0(-\frac{\ell_2}{2n+1}).
\end{align}
We now use $Q^\infty_{K,1}[\ell_1, \ell_2]$, $Q_{K,2}^\infty[\ell_1, \ell_2]$ as well as $Q_{K,3}^\infty[\ell_1, \ell_2]$ to denote each of the terms
\begin{align}
Q_{K,1}^\infty& = \lim_{n\rightarrow \infty} u^*\tilde{T}[D_0(s - \frac{\ell_2}{2n+1})] \psi(\frac{\ell_1}{2n+1})\\
Q_{K,2}^\infty& = \lim_{n\rightarrow \infty} \psi(\frac{\ell_1}{2n+1})\tilde{T}[D_0(s- \frac{\ell_2}{2n+1})]u\\
Q_{K,3}&  =\lim_{n\rightarrow \infty} \left|D(e^{2\pi i \frac{\ell_1}{2n+1}})\right|^2 D_0(-\frac{\ell_2}{2n+1}). 
\end{align}
Recall that $D(e^{2\pi i \theta}) = \frac{\sin((n+1)\pi \theta)}{\sin(\pi \theta)}$ from which we have
\begin{align}
D(e^{2\pi i \frac{\ell_1}{2n+1}})& = \frac{1}{n+1}\frac{\sin((2n+1)\pi \frac{\ell_1}{2n+1})}{\sin(\pi \frac{\ell_1}{2n+1})} = \left\{\begin{array}{ll} 
2\frac{\sin(\pi \ell_1/2)}{\pi \ell_1}& \text{when $\ell_1\neq 0$ }\\
1 & \text{otherwise}. \end{array}\right.\label{boundDirichletOnlyAssymptotic}
\end{align}
Then note that
\begin{align}
\left|D(e^{2\pi i \theta})\right|^2 D_0(\theta - \frac{\ell_2}{2n+1}) & = \frac{4\sin^2(\pi \frac{\ell_1}{2})}{\pi^2 \ell_1^2} \frac{\sin(\pi \ell_2)}{\pi \ell_2}\label{boundDirichletModulusSquared}  
\end{align}
when $\ell_1, \ell_2\neq 0$. When either $\ell_1$ or $\ell_2$ vanishes, we use 
\begin{align}
\left|D(e^{2\pi i \theta})\right|^2 D_0(\theta - \frac{\ell_2}{2n+1}) & = \left\{\begin{array}{ll}
\frac{\sin(\pi \ell_2)}{\pi \ell_2}& \text{when $\ell_1 = 0, \ell_2\neq 0$}\\
\frac{4\sin^2(\pi \ell_1/2)}{\pi^2\ell_1^2} & \text{when $\ell_2 = 0, \ell_1\neq 0$}\\
1 & \text{when $\ell_1 = \ell_2 = 0$}
\end{array}\right.
\end{align}
In the one atomic case, we have $Q_{K,1}^\infty = Q_{K,2}^\infty$. One can thus focus on $Q_{K,1}^\infty$. Recall that for a general polynomial $p$, the term $u^*\tilde{T}[p]\psi(\theta)$ can expand as 
\begin{align}
&\lim_{n\rightarrow \infty}u^* \tilde{T}[D_0(s- \frac{\ell_2}{2n+1})]\psi(\frac{\ell_1}{2n+1})\\
& =\lim_{n\rightarrow \infty} \sum_{s=0}^n \frac{p_s}{n+1-|s|} \sum_{t= -n/2+s}^{n/2} e^{2\pi i t \theta} + \sum_{s= -n}^0 \frac{p_s}{n+1-|s|} \sum_{t = -n/2}^{n/2+s} e^{2\pi i t\theta}\\
& = \lim_{n\rightarrow \infty}\sum_{s=0}^n \frac{p_s}{n+1-|s|} e^{-\pi i n \theta} \left(\sum_{t= s}^{n} e^{2\pi i t\theta}\right) +\sum_{s=-n}^0 \frac{p_s}{n+1-|s|} e^{-\pi i n \theta} \left(\sum_{t=0}^{n+s} e^{2\pi i  t \theta}\right)\\
& = \lim_{n\rightarrow \infty}\sum_{s=0}^n \frac{p_s}{n+1-|s|} e^{-\pi i n \theta} e^{2\pi i s\theta} \left(\sum_{t=0}^{n-s} e^{2\pi i t \theta}\right) + \sum_{s=-n}^0 \frac{p_s}{n+1-|s|} e^{-\pi i n \theta} \left(\sum_{t=0}^{n+s} e^{2\pi i t \theta}\right)\\
& = \lim_{n\rightarrow \infty}\sum_{s=0}^n \frac{p_s}{n+1-|s|} e^{-\pi i n \theta} e^{2\pi i s \theta} \left(\frac{1 - e^{2\pi i (n-s+1)\theta}}{1 - e^{2\pi i \theta}}\right)+ \sum_{s=-n}^0 \frac{p_s}{n+1-|s|} e^{-\pi i n \theta} \left(\frac{1 - e^{2\pi i (n+s+1)\theta}}{1-e^{2\pi i \theta}}\right)\\
& = \lim_{n\rightarrow \infty}\frac{e^{-\pi i n \theta}}{1-e^{2\pi i \theta}} \sum_{s=0}^n \frac{p_s}{n+1-s} e^{2\pi i s \theta} \left(1 - e^{2\pi i (n+1-s)\theta }\right)+ \frac{e^{-\pi i n \theta}}{1 - e^{2\pi i \theta}} \sum_{s=-n}^{0} \frac{p_s}{n+1+s} \left(1 - e^{2\pi i (n+1+s)\theta}\right)
\end{align}
Now letting $\theta = \frac{\ell_1}{2n+1}$ and using $p(\theta) = D_0(\theta - \frac{\ell_2}{2n+1})$, we get 
\begin{align}
\lim_{n\rightarrow \infty}Q_{K,1}^{n}[\ell_1, \ell_2]& =\lim_{n\rightarrow \infty}\mathcal{P}_K\left(u^*\tilde{T}[D_0(\theta - \frac{\ell_2}{2n+1})]\psi(\frac{\ell_1}{2n+1})\right)[\frac{\ell_1}{2n+1}]\\
& =  \lim_{n\rightarrow \infty} \frac{e^{-\pi in \frac{\ell_1}{2n+1}}}{\frac{2\pi i \ell_1}{2n+1}} \sum_{s=0}^{n} \frac{1}{2n+1} \frac{e^{-2\pi i \frac{\ell_2}{2n+1}s}}{n+1-s} e^{2\pi i s \frac{\ell_1}{2n+1}} \left(1 - e^{2\pi i (n-s+1)\frac{\ell_1}{2n+1}}\right)\\
&+ \lim_{n\rightarrow \infty}\frac{e^{-\pi i n \frac{\ell_1}{2n+1}}}{2\pi i \frac{\ell_1}{2n+1}} \sum_{s=-n}^0 \frac{1}{2n+1} \frac{e^{-2\pi i \frac{\ell_2}{2n+1}}}{n+1+s} \left(1 - e^{2\pi i (n+1+s)\frac{\ell_1}{2n+1}}\right)\\
& = \lim_{n\rightarrow \infty}\frac{e^{-\pi i n \frac{\ell_1}{2n+1}}}{2\pi i \ell_1} \sum_{s=0}^{n}  \frac{e^{\frac{n+1}{2n+1} (\ell_1-\ell_2)2\pi i }}{n+1-s} \left(e^{\frac{2\pi i }{2n+1}(\ell_2-\ell_1)(n+1-s)} - e^{\frac{2\pi i (n+1-s)\ell_2}{2n+1}}\right)\\
&+\lim_{n\rightarrow \infty} \frac{e^{-\pi i n \frac{\ell_1}{2n+1}}}{2\pi i \ell_1} \sum_{s=-n}^{0} \frac{e^{-2\pi i \frac{\ell_2}{2n+1}s}}{n+1+s} \left(1 - e^{2\pi i (n+1+s)\frac{\ell_1}{2n+1}}\right)\\
& = \frac{e^{-\pi i n \frac{\ell_1}{2n+1}}}{2\pi i \ell_1} e^{\frac{n+1}{2n+1}(\ell_1-\ell_2)2\pi i } \int_0^1 \frac{1}{s} \left(e^{2\pi i (\ell_2-\ell_1)s} - e^{2\pi i \ell_2s}\right)\; ds\\
&+ \lim_{n\rightarrow \infty}\frac{e^{-\pi i n \frac{\ell_1}{2n+1}}}{2\pi i \ell_1} \sum_{s=-n}^{0} \frac{e^{2\pi i \ell_2 \frac{n+1}{2n+1}}}{n+1+s} \left[e^{-2\pi i \frac{ (n+1+s)}{2n+1}} - e^{2\pi i (n+1+s)\frac{(\ell_1-\ell_2)}{2n+1}}\right]\\
& = \frac{e^{-\pi i n \frac{\ell_1}{2n+1}}}{2\pi i \ell_1} e^{\frac{n+1}{2n+1}(\ell_1-\ell_2)2\pi i } \int_0^1 \frac{1}{s} \left(e^{2\pi i (\ell_2-\ell_1)s} - e^{2\pi i \ell_2s}\right)\; ds\\
&+\frac{e^{-\pi i n \frac{\ell_1}{2n+1}}}{2\pi i \ell_1} e^{2\pi i \ell_2 \frac{(n+1)}{2n+1}} \int_0^1 \frac{1}{s}\left[e^{-2\pi i \ell_2 s} - e^{2\pi i (\ell_1 - \ell_2)s}\right]\; ds\label{lastLineExpressionQ1Kinf}
\end{align}
From~\eqref{lastLineExpressionQ1Kinf}, we thus get 
\begin{align}
Q_{K,1}^\infty[\ell_1, \ell_2]& = \frac{e^{-\pi i n\frac{\ell_1}{2n+1}}}{2\pi i \ell_1} e^{\frac{n+1}{2n+1} (\ell_1-\ell_2)2\pi i }\int_0^1 \frac{1}{s} \left[e^{\pi i (\ell_1-\ell_2)s} - e^{\pi i \ell_2 s}\right]\; ds\\
&+ \frac{e^{-\pi i n \frac{\ell_1}{2n+1}}}{2\pi i \ell_1} e^{2\pi i \frac{\ell_2 (n+1)}{2n+1}} \int_0^1 \frac{1}{s} \left[e^{-\pi i \ell_2 s} - e^{\pi i (\ell_1-\ell_2)s}\right]\; ds
\end{align}

To replace the integral with convolutions, we now use $\Ci(z) = -\int_{z}^\infty \frac{\cos(t)}{t}\; dt$, $\Si(z) = \int_0^z \frac{\sin(t)}{t}\; dt$. Finally, we use $\Gamma[a, z]$ to denote the incomplete Gamma function $\Gamma[a, z] = \int_z^\infty t^{a-1}e^{-t}$. Whenever $a = 0$, this integral reduces to the exponential integral. I.e. $\Gamma[0,z] = \int_z^\infty \frac{e^{-t}}{t}\; dt = E_1(z)$. $\gamma$ here denotes the Euler–Mascheroni constant, $\gamma\approx 0.577$. 

\begin{itemize}
\item When $\ell_1, \ell_2\neq 0$, we have $\lim_{n\rightarrow \infty} D_n(e^{2\pi i \frac{\ell_1}{2n+1}}) = \frac{2}{\pi \ell_1}\sin(\pi \ell_1/2)$ as well as 
\begin{align}
Q_{K,1}^\infty[\ell_1, \ell_2]& = -e^{-\pi i \ell_1/2} \frac{1}{2\pi i \ell_1} e^{(\ell_1-\ell_2)\pi i } \int_{0}^1  \frac{1}{s} \left[e^{\pi i (\ell_2-\ell_1)s} - e^{\pi i \ell_2 s}\right]\; ds\\
& - e^{-\pi i \ell_1/2} \frac{1}{2\pi i \ell_1} e^{\pi i \ell_2} \int_{0}^1 \frac{1}{s} \left[e^{-\pi i \ell_2 s } - e^{\pi i (\ell_1 - \ell_2)s}\right]\; ds\\
& = -e^{-\pi i \ell_1/2} \frac{1}{2\pi i \ell_1} e^{(\ell_1-\ell_2)\pi i } \left[\Ci((\ell_1-\ell_2)\pi) -\Ci(\ell_2\pi ) - \log(\ell_1 - \ell_2) + \log(\ell_2)\right] \label{term1ExpressionQ1Kinf}\\
& -e^{-\pi i \ell_1/2} \frac{1}{2\pi i \ell_1} e^{(\ell_1-\ell_2)\pi i } \left[ - i \Si((\ell_1 - \ell_2)\pi) - i \Si(\ell_2 \pi)\right]\\
&- e^{-\pi i \ell_1/2} \frac{1}{2\pi i \ell_1} e^{\pi i \ell_2} \left[ \Gamma[0, -i(\ell_1 - \ell_2)\pi] - \Gamma[0, i \ell_2 \pi ] + \log(-i(\ell_1 - \ell_2)) - \log(i \ell_2)\right]. \label{term3ExpressionQ1Kinf}
\end{align}

\item When $\ell_1 = 0$, $\ell_2\neq 0$, we have $\lim_{n\rightarrow \infty} D_n(e^{2\pi i \frac{\ell_1}{2n+1}})=1$ as well as
\begin{align}
Q_{K,1}^\infty[\ell_1, \ell_2]= Q_{K,2}^\infty[\ell_1, \ell_2]&  =\frac{1}{2} e^{-\pi i \ell_2} \frac{1}{\pi i \ell_2} \left[e^{\pi i \ell_2} - 1\right] + \frac{1}{2} e^{\pi i \ell_2} \frac{(-1)}{\pi i \ell_2} \left[e^{-\pi i \ell_2} - 1\right]
\end{align}
\item When $\ell_1\neq 0$ and $\ell_2=0$, we have $\lim_{n\rightarrow \infty} D_n(e^{2\pi i \frac{\ell_1}{2n+1}})=\frac{2}{\pi \ell_1}\sin(\pi \ell_1/2)$ as well as 
\begin{align}
Q_{K,1}^\infty[\ell_1, \ell_2]&= - e^{-\pi i \ell_1/2} \frac{1}{2\pi i \ell_1} e^{(\ell_1 - \ell_2)\pi i } \int_{0}^1 \frac{1}{s} \left[e^{\pi i (\ell_2 - \ell_1)s} - e^{\pi i \ell_2 s}\right]\; ds\\
& - \frac{1}{2\pi i \ell_1}e^{-\pi i \ell_1/2} e^{\pi i \ell_2} \int_{0}^1 \frac{1}{s} \left[e^{-\pi i \ell_2 s} - e^{\pi i (\ell_1 - \ell_2)s}\right]\; ds\\
& = - e^{-\pi i \ell_1/2} \frac{1}{2\pi i \ell_1} e^{(\ell_1 - \ell_2)\pi i } \left[-\gamma + \Ci(\ell_1\pi) - \log(\ell_1\pi) - i\Si(\ell_1\pi) \right]\\
&- \frac{1}{2\pi i \ell_1}e^{-\pi i \ell_1/2} e^{\pi i \ell_2} \left[\gamma + \frac{i\pi}{4} - \Ci(-\ell_1\pi) + \frac{\log(-\ell_1)}{2} + \frac{1}{2}\log(i\ell_1) + \log(\pi) -i \Si(\ell_1\pi)\right]
\end{align}
\item Finally, when $\ell_1=  \ell_2 =  0$, we use $\lim_{n\rightarrow \infty} D_n(e^{2\pi i \frac{\ell_1}{2n+1}})=1$, as well as $Q_{K,1}[0,0]^\infty = Q_{K,2}[0,0]^\infty = 1$. 
\end{itemize}

Grouping~\eqref{term1ExpressionQ1Kinf} to~\eqref{term3ExpressionQ1Kinf} together with~\eqref{boundDirichletOnlyAssymptotic} and~\eqref{boundDirichletModulusSquared} gives the result of the lemma
\begin{align}
\left(Q_K^\infty x\right)[\ell_1] &= -2e^{-\pi i \ell_1/2} \frac{1}{2\pi i \ell_1} e^{(\ell_1-\ell_2)\pi i } \left(2\frac{\sin(\pi \ell_1/2)}{\pi \ell_1}\right)\left[\Ci((\ell_1-\ell_2)\pi) -\Ci(\ell_2\pi ) \right]x[\ell_2]\\
& -2e^{-\pi i \ell_1/2} \frac{1}{2\pi i \ell_1} e^{(\ell_1-\ell_2)\pi i }\left(2\frac{\sin(\pi \ell_1/2)}{\pi \ell_1}\right) \left[- \log(\ell_1 - \ell_2) + \log(\ell_2)\right] x[\ell_2]\\
& -2e^{-\pi i \ell_1/2} \frac{1}{2\pi i \ell_1} e^{(\ell_1-\ell_2)\pi i }\left(2\frac{\sin(\pi \ell_1/2)}{\pi \ell_1}\right) \left[ - i \Si((\ell_1 - \ell_2)\pi) - i \Si(\ell_2 \pi)\right]x[\ell_2]\\
&- 2e^{-\pi i \ell_1/2} \frac{1}{2\pi i \ell_1} e^{\pi i \ell_2}\left(2\frac{\sin(\pi \ell_1/2)}{\pi \ell_1}\right) \left[ \Gamma[0, -i(\ell_1 - \ell_2)\pi] - \Gamma[0, i \ell_2 \pi ] \right]x[\ell_2]\\
&- 2e^{-\pi i \ell_1/2} \frac{1}{2\pi i \ell_1} e^{\pi i \ell_2}\left(2\frac{\sin(\pi \ell_1/2)}{\pi \ell_1}\right)\left[ \log(-i(\ell_1 - \ell_2)) - \log(i \ell_2)\right]x[\ell_2]. \\
& - \frac{4\sin^2(\pi \frac{\ell_1}{2})}{\pi^2 \ell_1^2} \frac{\sin(\pi \ell_2)}{\pi \ell_2}x[\ell_2].
\end{align}

\subsection{\label{entryDecay}Entry Decay}

A first reduction in the numerical complexity of $\mathcal{Q}_K^\infty$ can be obtained by noting that the entries of $Q^\infty[\ell_1, \ell_1]$ decay sufficiently fast with respect to the row and column indices.

We start by showing that the contribution of the upper and lower blocks ($|\ell_1|\geq K_1$ for a sufficiently large $K_1$) to the product $\|Q_K^\infty q\|$ can be neglected. To see this, recall that the product $Q q$ reads as 
\begin{align}
\left|Q_{K,1}^{\infty}[\ell_1, \ell_2]q[\ell_2]\right|& \leq \left|\sum_{\ell_2} \frac{1}{\pi^2\ell_1^2} \int_0^1 \frac{1}{s}\left[e^{\pi i (\ell_2-\ell_1)s} - e^{\pi i \ell_2 s}\right]\; ds \; q[\ell_2]\right|\\
&\leq \left|\sum_{\ell_2} \frac{1}{\pi^2\ell_1^2} \int_{0}^{\frac{1}{\ell_1^2\pi^2}} e^{\pi \ell_2 s} \frac{1}{s}\left[e^{\pi i \ell_1 s} - 1\right]\; ds \; q[\ell_2]\right|\label{singularityTermEll11} \\
&+ \left| \sum_{\ell_2} \frac{1}{\pi^2\ell_1^2} \int_{\frac{1}{\ell_1^2 \pi^2}}^{1} e^{\pi \ell_2 s} \frac{1}{s}\left[e^{\pi i \ell_1 s} - 1\right]\; ds\;  q[\ell_2] \right|
\end{align}
Using a Taylor expansion, one can bound the first term above as
\begin{align}
\eqref{singularityTermEll11}& \leq \left|\sum_{\ell_2} \frac{1}{\pi^2\ell_1^2} \int_{0}^{\frac{1}{\ell_1^2\pi^2}} e^{\pi\ell_2 i s} \left[\pi i \ell_1 + (\ell_1\pi i)^2 s + \text{h.o.t}\right]\; ds\;  q[\ell_2]\right|\\
&\lesssim  \sum_{\ell_2} \frac{1}{\pi^2\ell_1^2} \frac{1}{\pi^2\ell_1^2} (\ell_1\pi + 1) q[\ell_2]\\
&\lesssim \frac{1}{\pi^3\ell_1^3} (C_1 + C_1\log(K)) \|q_K\|_2\\
&\lesssim \frac{1}{\pi^3\ell_1^3} 100 (C_1 + C_1\log(K))
\end{align}
Squaring and taking the sum with respect to $\ell_1$ to get the $\ell_2$-norm, we get 
\begin{align}
\sqrt{\sum_{\ell_1\geq K_1} \left|\eqref{singularityTermEll11}\right|^2}&\leq \sqrt{\sum_{\ell_1\geq K_1}\left( \frac{1}{\pi^3\ell_1^3} (C_1+C_1\log(K)) \right)^2}\\
& \leq \sqrt{\frac{\zeta(2)}{\pi^6 K_1^4} 1e4\left(C_1 + C_1\log(K)\right)}\\
&\leq \frac{\sqrt{\zeta(2)}}{\pi^3 K_1^{2}} 100 (C_1 + C_1\log(K))
\end{align}
We multiply this term by $8$ to account for the two integrals appearing in the two terms $Q_{K,1}$ and $Q_{K,2}$, as well as the two blocks $\ell_1\geq K_1$ and $\ell_1<-K_1$, thus getting the upper bound $B_1$
\begin{align}
B_1&\leq \frac{4\sqrt{\zeta(2)}}{\pi^3 K_1^{2}} 100 (C_1 + C_1\log(K))
\end{align}
This bound can be made sufficiently small (i.e smaller than $1e-4$) as soon as $K_1\geq 1e6$. For the second term, we write 
\begin{align}
&\left|\sum_{\ell_2} \frac{1}{\pi^2\ell_1^2} \int_{\frac{1}{\ell_1^2 \pi^2}}^1 \frac{1}{s} \left[e^{\pi i (\ell_2-\ell_1)s} - e^{\pi i \ell_2 s}\right]\; ds\; q[\ell_2]\right|\\
&\leq \left|\sum_{\ell_2} \frac{1}{\pi^2\ell_1^2} \int_{\frac{1}{\pi^2\ell_1^2}}^1 \frac{\cos(\pi (\ell_2-\ell_1)s) - \cos(\pi \ell_2 s)}{s}\; ds\; q[\ell_2]\right|\label{noSingularityTerm11}\\
& +\left| \sum_{\ell_2} \frac{1}{\pi^2\ell_1^2} \int_{\frac{1}{\pi^2\ell_1^2}}^{1} \frac{\sin(\pi(\ell_2-\ell_1)s) - \sin(\pi \ell_1 s)}{s} \; ds\; q[\ell_2]\right|
\end{align}
For the first term (the sine integral follows the same idea), we have 
\begin{align}
\eqref{noSingularityTerm11}&\leq \left| \sum_{\ell_2} \frac{1}{\pi^2\ell_1^2} \left(\int_{\frac{\pi (\ell_2-\ell_1)}{\ell_1^2\pi^2}}^\infty  \frac{\cos(t)}{t}\; dt - \int_{\frac{\pi \ell_2}{\pi^2\ell_1^2}}^\infty  \frac{\cos(t)}{t}\; dt\right)\; q[\ell_2]\right|\label{firstContributionLargeEll1Cos}\\
& + \left|\sum_{\ell_2} \frac{1}{\pi^2\ell_1^2} \left(-\int_{\pi (\ell_2 - \ell_1)}^\infty \frac{\cos(t)}{t}\; dt + \int_{\pi \ell_2}^\infty \frac{\cos(t)}{t}\; dt\right)\; q[\ell_2]\right|
\end{align}
The last two terms can be bounded from the definition of the cosine integral (resp sine integral) by noting that for both the sine and cosine integral, we have (in the case of $\Ci(x)$ it follows directly from integration by part) $\int_x^\infty \cos(t)\; dt\leq \frac{2}{|x|}$ as well as $\left|\Si(x) - \pi/2\right|\leq \frac{1}{|x|}$. 
\begin{align}
\left|\sum_{\ell_2} \frac{1}{\pi^2\ell_1^2} \left(-\int_{\pi (\ell_2 - \ell_1)}^\infty \frac{\cos(t)}{t}\; dt + \int_{\pi \ell_2}^\infty \frac{\cos(t)}{t}\; dt\right)q[\ell_2]\right| &\leq \left|\sum_{\ell_2}\frac{1}{\pi^2\ell_1^2} \left(\frac{1}{(\ell_2-\ell_1)} + \frac{1}{\ell_2}\right) q[\ell_2]\right|\\
&\leq \frac{2}{\pi^3\ell_1^2} \left(C_1 + C_1\log(K)\right) \|q_K\|_2^{-1}
\end{align}
Squaring and summing to get the contribution to the norm, we get
\begin{align}
&\sqrt{\sum_{\ell_1\geq K_1}\left(\sum_{\ell_2} \frac{1}{\pi^2\ell_1^2} \left(-\int_{\pi (\ell_2 - \ell_1)}^\infty \frac{\cos(t)}{t}\; dt + \int_{\pi \ell_2}^\infty \frac{\cos(t)}{t}\; dt\right)q[\ell_2]\right)^2} \\
&\leq \sqrt{\sum_{\ell_1\geq K_1} \frac{4}{\pi^6\ell_1^4} (C_1+C_1\log(K))^2} \|q_K\|_2^{-1}\\
&\leq \frac{1}{\pi^3K_1^2} (C_1+C_1\log(K))100
\end{align}
The total contribution (including the two integrals, the two terms $Q_{K,1}$ and $Q_{K,2}$ and the sine integral) can thus be bounded by $B_2$ with
\begin{align}
B_2&\leq \frac{16}{\pi^3K_1^2} (C_1+C_1\log(K))100
\end{align}
This contribution can be made sufficiently small (smaller than $.01$) as soon as $K_1\geq 1e6$.
For the first contribution in~\eqref{firstContributionLargeEll1Cos}, we consider two cases for each integral. Either $\frac{(\ell_2 - \ell_1)}{\ell_1^2\pi}>1$ in this case the integral is bounded by $1$ (the sine integral is always bounded by $2$) or $\frac{(\ell_2 - \ell_1)}{\ell_1^2\pi}<1$, in which case we use the expansion 
\begin{align}
\int_{x}^\infty \frac{\cos(t)}{t}\; dt = \gamma + \log(x) + \sum_{n=1}^\infty \frac{(-1)^n x^{2n}}{2n(2n)!} \leq \gamma + 2+ \log(x).
\end{align}
A similar expansion holds for the sine integral. Using this expansion, we can thus bound this term as
\begin{align}
\eqref{firstContributionLargeEll1Cos}&\lesssim \sum_{\ell_2} \frac{1}{\pi^2\ell_1^2} \left(2\gamma+ 4\log(\ell_1) + 2\log(\pi) + \log(\ell_2-\ell_1) + \log(\ell_2) \right) \min\left(1, \frac{C_1}{1+n|\frac{k}{2n+1}|}\right)\|q_K\|^{-1}_2\\
&\leq \left(C_1 + 3C_1\log(K)\right) \frac{1}{\pi^2\ell_1^2} \left(6\log(K) + 3.5 \right)100
\end{align}
Squaring and taking the sum, then multiplying the bound by $16$ we get 
\begin{align}
16\sqrt{\sum_{\ell_1\geq K_1} \left(\eqref{firstContributionLargeEll1Cos}\right)^2}&\leq 16\frac{\left(C_1 + 3C_1\log(K)\right) \left(6\log(K) + 3.5 \right)100}{\pi^2 K_1^2}
\end{align}
That contibution can be made smaller than $.01$ as soon as $K_1\geq 1e6$.

We now consider the left and right blocks $\left\{ \ell_2\geq K_1, |\ell_1|\leq K_1\right\}$ and $\left\{\ell_2\leq -K_1, |\ell_1|\leq K_1\right\}$. For these blocks,  we start by treating the singularity using integration by parts (i.e. stationnary phase),
\begin{align}
\frac{2}{\pi\ell_1^2 2}\int_{0}^{\frac{1}{\pi^2\ell_1^2}} \frac{1}{s}e^{\pi i \ell_2 s} \left[e^{\pi i \ell_1 s} - 1\right]\; ds\; q[\ell_2]
&\leq \frac{1}{\pi^2 \ell_1^2} \left|\frac{e^{\pi i \ell_2}}{\ell_2\pi } \left[\pi i \ell_1 + \frac{(\pi i \ell_1)^2 s}{2!} + \ldots\right]\right|_0^{\frac{1}{\pi^2\ell_1^2}}\; q[\ell_2]\\
& + \frac{1}{\pi \ell_1^2} \int_0^{\frac{1}{\pi^2\ell_1^2}} \frac{e^{\pi i \ell_2 s}}{\ell_2 \pi} \left[ \frac{(\pi i \ell_1)^2}{2!} +  \frac{(\pi i \ell_1)^3 2s}{3!} + \ldots \right]\; ds\; q[\ell_2]
\end{align}
Then we use 
\begin{align}
\left|\pi i \ell_1 + \frac{(\pi i \ell_1)^2 s}{2!} + \ldots\right|_{s=\frac{1}{\pi^2\ell_1^2}} &\leq \left|\pi i \ell_1 + \frac{1}{2!} + \sum_{k=3}^\infty \frac{(\pi i \ell_1)^k}{k!} s^{k-1}\right|_{s = \frac{1}{\pi^2\ell_1^2}}\\
&\leq \pi \ell_1 + \frac{1}{2!} + \sum_{k=3}^{\infty} (\pi i \ell_1)^k \left(\frac{1}{(\pi i \ell_1)^2}\right)^{k-1}\\
&\leq \pi \ell_1 + \frac{1}{2!} + \sum_{k=3}^{\infty} \frac{(\pi \ell_1)^k}{(\pi \ell_1)^{2k-2}}\\
&\leq \pi \ell_1 + \frac{1}{2} + \sum_{k\geq 3} \frac{1}{(\pi \ell_1)^{k-2}}\\
&\leq \pi \ell_1 + \frac{1}{2} + \sum_{k\geq 1} \frac{1}{\pi^k}\\
&\leq \pi \ell_1 + \frac{1}{2} + \frac{1}{\pi} \frac{1}{1-\frac{1}{\pi}}\\
&\leq \pi \ell_1 + \frac{1}{2} + \frac{1}{\pi} \frac{\pi}{\pi -1}\\
&\leq \pi\ell_1 + \frac{1}{2} + \frac{1}{\pi -1}
\end{align}
From this, the total contribution for each $\ell_1$ reads as
\begin{align}
\sum_{\ell_2\geq K_1}\frac{1}{\pi^2\ell_1} \frac{2}{\pi \ell_2} \frac{C_1}{\ell_2} \left(\pi + \frac{1}{2\ell_1} + \frac{1}{(\pi -1)\ell_1}\right)\|q_K\|_2^{-1} \leq \frac{2}{\pi^3 \ell_1} \frac{C_1}{K_1}\left(\pi +\frac{1}{2} + \frac{1}{\pi-1}\right)100
\end{align}
Squaring and taking the sum over $\ell_1$ gives the norm
\begin{align}
\sqrt{\sum_{\ell_1\leq K_1} \left(\frac{2}{\pi^3 \ell_1} \frac{C_1}{K_1}\left(\pi +\frac{1}{2} + \frac{1}{\pi-1}\right)100\right)^2}\leq \frac{2}{\pi^3} \frac{100C_1}{K_1}\left(\pi +\frac{1}{2} + \frac{1}{\pi-1}\right)\sqrt{\zeta(2)}
\end{align}
Multiplying by $8$ gives the final bound 
\begin{align}
B_3\leq \frac{16}{\pi^3} \frac{C_1}{K_1}\left(\pi +\frac{1}{2} + \frac{1}{\pi-1}\right)\sqrt{\zeta(2)} \leq \frac{8.5e4}{K_1}
\end{align}
This bound can be made smaller than $0.01$ as soon as $K_1$ is larger than $1e7$. 

We now control the integral on $[\frac{1}{(\pi \ell_1)^2}, 1]$. As for the blocks $\left\{|\ell_1|\geq K_1\right\}$, we write 
\begin{align}
\left|\frac{2}{\pi\ell_1^2 2}\int_{\frac{1}{\pi^2\ell_1^2}}^{1} \frac{1}{s}e^{\pi i \ell_2 s} \left[e^{\pi i \ell_1 s} - 1\right]\; ds\; q[\ell_2]\right|
&\leq \left|\frac{1}{\pi^2 \ell_1^2} \left( \int_{\frac{\pi (\ell_2 - \ell_1)}{\pi^2\ell_1^2}}^{\infty} \frac{\cos(t)}{t}\; dt - \int_{\frac{\pi \ell_2}{\pi^2\ell_1^2}}^{\infty} \frac{\cos(t)}{t}\; dt\right)q[\ell_2]\right|\\
&+ \left| \frac{1}{\pi^2\ell_1^2} \left(\int_{\pi (\ell_2 - \ell_1)}^\infty \frac{\cos(t)}{t}\; dt - \int_{\pi \ell_2}^{\infty} \frac{\cos(t)}{t}\; dt\right)q[\ell_2]\right|\label{boundDirectLargexCosint}\\
&+ \left|\frac{1}{\pi^2 \ell_1^2} \left( \int_{\frac{\pi (\ell_2 - \ell_1)}{\pi^2\ell_1^2}}^{\infty} \frac{\sin(t)}{t}\; dt - \int_{\frac{\pi \ell_2}{\pi^2\ell_1^2}}^{\infty} \frac{\sin(t)}{t}\; dt\right)q[\ell_2]\right|\\
&+ \left|\frac{1}{\pi^2\ell_1^2} \left(\int_{\pi (\ell_2 - \ell_1)}^\infty \frac{\sin(t)}{t}\; dt - \int_{\pi \ell_2}^{\infty} \frac{\sin(t)}{t}\; dt\right)q[\ell_2]\right|.\label{boundDirectLargexSinint}
\end{align}
Any bound on the cosine gives a bound on the sine following the discussion above. As before, ~\eqref{boundDirectLargexCosint} and~\eqref{boundDirectLargexSinint} can be bounded by using  the $1/|x|$ upper bound on the cosine (resp. $\pi/2 + 1/|x|$ bound on the sine)
\begin{align}
\eqref{boundDirectLargexCosint}&\leq \sum_{\ell_2\geq K_1}\frac{1}{\pi^3\ell_1^2}\left( \frac{2C_1}{\ell_2}\frac{1}{\ell_2-\ell_1} + \frac{2C_1}{\ell_2^2}\right)\|q_K\|_2^{-1}\leq  \frac{C_1}{K_1} (1+\log(K))\frac{1}{\pi^3\ell_1^2}100
\end{align}
Squaring, summing over $\ell_1$ and taking the square root to get the norm, then multiplying by $16$ to account for the appearance of the terms in both blocks and both $Q_{K,1}$ and $Q_{K,2}$ we get a total contribution
\begin{align}
B_4 \equiv 16\sqrt{\sum_{\ell_1\leq K_1} \left( \frac{C_1}{K_1} \log(K)\frac{1}{\pi^3\ell_1^2}\right)^2}&\leq \sqrt{\zeta(2)} \frac{C_1}{K_1} \log(K)\frac{100}{\pi^3} \leq \frac{1.35e5}{K_1}
\end{align}
This bound can be made smaller than $.02$ as soon as $K_1\geq 1e7$. 

For the remaining integrals we make the distinction between the case $(\ell_2 - \ell_1)>\pi \ell_1^2$ and $(\ell_2 - \ell_1)<\pi \ell_1^2$. In the first case, note that since $\ell_2\leq K$ we necessarily have $\ell_1\leq \frac{1}{2\pi} + \frac{1}{2\pi}\sqrt{K} \frac{\sqrt{K^{-1} + 4\pi}}{2\pi}\leq \frac{1}{2\pi} + \sqrt{K} 0.56$. Using this bound on the cosine integrals above, and noting again that $\int_{x}^\infty \frac{\cos(t)}{t}\; dt\leq \frac{2}{|x|}$, we can write 
\begin{align}
\sum_{\ell_2\geq K_1} \frac{2}{\pi^2\ell_2} \frac{\pi^2\ell_1^2}{\pi(\ell_2 - \ell_1)}\frac{C_1}{\ell_2} &\leq \sum_{\ell_2\geq K_1} \frac{2}{\pi (\ell_2-\ell_1)} \frac{C_1}{\ell_2}\\
&\stackrel{(a)}{\leq } \sum_{\ell_2\geq K_1} \frac{2}{\pi\ell_2 0.9} \frac{C_1}{\ell_2}
\end{align}
$(a)$ follows from $\ell_1\leq \sqrt{K}$ and $\ell_2\geq K_1$ and assuming the upper bound $K\leq 1e12$ and $K_1\geq 1e7$, we always have $\ell_2-\ell_1\geq 0.9\ell_2$. Squaring this last contribution and summing over $\ell_1$ to get a bound on the norm, we have 
\begin{align}
\sqrt{\sum_{\ell_1\leq K_1} \left(\frac{2}{\pi\ell_2 0.9} \frac{C_1}{\ell_2}\right)^2 }\leq \sqrt{\frac{1}{2\pi} + \sqrt{K} 0.56}\left(\frac{2}{ 0.9\pi} \frac{C_1}{K_1}\right) \leq \frac{2.36e6}{K_1}.
\end{align}
The last bound follows from using $K\leq 1e13$. Multiplying the bound by $32\times \|q_K\|_2^{-1}\leq 3200$, we get 
\begin{align}
\frac{7.54e9}{K_1}
\end{align}
Taking $K_1\geq 7.54e10$ gives an upper bound of $0.1$.

Whenever $\ell_2-\ell_1 \leq \pi\ell_1^2$, since $\ell_2\geq K_1$ we necessarily have $K_1\leq \ell_1^2 + \ell_1$ which implies $|\ell_1|\geq \frac{-1 + \sqrt{1 + 4K_1\pi}}{2\pi}$. Hence we can write 
\begin{align}
 \frac{1}{\pi^2\ell_1^2} C_1\log(K) \left(\gamma + \ln(K)+3\ln(K_1) + \frac{1}{4}\right)
\end{align}
Squaring and summing over $\ell_1$, we get 
\begin{align}
&\sum_{\ell_1\geq \frac{-1 + \sqrt{1 + 4K_1\pi}}{2\pi}} \left(\frac{1}{\pi^2\ell_1^2} C_1\log(K) \left(\gamma + \ln(K)+3\ln(K_1) + \frac{1}{4}\right)\right)^2\\
&\leq \frac{1}{N}\left(\frac{1}{\pi^2} C_1\log(K) \left(\gamma + \ln(K)+3\ln(K_1) + \frac{1}{4}\right)\right)
\end{align}

where $N = \left(\frac{-1+\sqrt{1+4K_1\pi}}{2\pi}\right)^2  = \frac{1+4K_1\pi}{4\pi^2} + \frac{1}{4\pi^2} - \frac{1}{\pi} \sqrt{1+4K_1\pi}\geq \frac{1+4K_1\pi}{8\pi^2} + \frac{1}{4\pi^2}$. Hence we get 
\begin{align}
&\frac{1}{\frac{1+4K_1\pi}{8\pi^2} + \frac{1}{4\pi^2}} \left(\frac{1}{\pi^2} C_1\log(K) \left(\gamma + \log(\pi) +2\log(K_1)+ \frac{1}{4}\right)\right)\\
&\leq \frac{2\pi}{K_1} \frac{C_1\log(K)\left(\gamma + \log(\pi) +2\log(K_1)+ \frac{1}{4}\right)}{\pi^2}\\
&\leq \frac{4.55e5}{K_1}.
\end{align}

The last bound follows from using $K=1e13$. Multiplying the bound by $32\|q_K\|_2^{-1}$, we get a total bound of 
\begin{align}
B_6&\leq \frac{1.46e9}{K_1}.
\end{align}
Taking $K_1\geq 1.46e10$ hence gives a bound of $0.1$ on that contribution as well. From the discussion above, the numerical complexity can thus be reduced from $O(1e13)$ to $O(1e10)$. 

\subsection{\label{sec:PowerIterations}Power iterations}

To conclude, we provide an illustration of how Power iterations can be applied to derive a lower bound on the smallest sinngular value of the matrix $(I - Q_K^\infty + P^\infty)$ for $K$ up to $1e7$ (which is what current memory limitations enable to achieve). We apply such iterations for $K=1e7$, first to compute the largest eigenvalue of the hermitian operator $(I - Q_K^\infty + P^\infty)^*(I - Q_K^\infty + P^\infty)$, then to compute the largest eigenvalue of the shifted operator $\lambda_{\max}I - (I - Q_K^\infty + P^\infty)*(I - Q_K^\infty + P^\infty)$ (hence getting an estimate for the smallest eigenvalue of $(I - Q_K^\infty + P^\infty)*(I - Q_K^\infty + P^\infty)$). In order to control the accuracy of the estimate returned by the Power Method, we compute an a posteriori error estimate from the residual $\eta \equiv (I - Q_K^\infty + P^\infty)\bs x^{t} - \mu^{(t)}\bs x^{(t)}$ for iterates $(\bs x^{(t)}, \mu^{(t)})$ returned by the Power iterations. This residual gives a bound on the distance of the iterate returned by the Power method and the losest eigenvalue from the spectrum of $(I - Q_K^\infty + P^\infty)$. Proposition~\ref{ThFranklin} below (Theorem 5 in~\cite{isaacson2012analysis})

\begin{proposition}\label{ThFranklin}
Let $A$ be a Hermitian matrix of order $n$ and have eigenvalues $\left\{\lambda_i\right\}$. For an approximated eigenpair $\left\{\bs x, \lambda\right\}$, we define the residual $\bs \eta$ as $\bs \eta\equiv A\bs x - \lambda \bs x.$ Then 
\begin{align}
\min_{i} |\lambda - \lambda_i| \leq \frac{\|\bs \eta\|_2}{\|\bs x\|_2} \label{posterioriErrorBound}
\end{align}
\end{proposition}
In particular, this theorem implies that that if we iterate until $\|A\bs x_{(k)} - \lambda^t\bs x^{(k)}\|_2\leq \varepsilon \|\bs x^{(k)}\|_2$, then we have $|\lambda_1|(1-\varepsilon) \leq |\lambda|\leq |\lambda_1|(1+\varepsilon)$.

\begin{figure}
%
%
\definecolor{mycolor1}{rgb}{0.00000,0.44700,0.74100}%
\definecolor{mycolor2}{rgb}{0.85000,0.32500,0.09800}%
\begin{tikzpicture}
\hspace{1cm}
\begin{axis}[%
width=2.016in,
height=2.016in,
at={(1.5in,2.011in)},
scale only axis,
point meta min=9.86545967944245e-21,
point meta max=1,
axis on top,
xmin=0.5,
xmax=41.5,
y dir=reverse,
ymin=0.5,
ymax=41.5,
axis background/.style={fill=white},
title={Modulus of $Q_K^\infty$},
legend style={legend cell align=left, align=left, draw=white!15!black},
colormap={mymap}{[1pt] rgb(0pt)=(0.2422,0.1504,0.6603); rgb(1pt)=(0.25039,0.164995,0.707614); rgb(2pt)=(0.257771,0.181781,0.751138); rgb(3pt)=(0.264729,0.197757,0.795214); rgb(4pt)=(0.270648,0.214676,0.836371); rgb(5pt)=(0.275114,0.234238,0.870986); rgb(6pt)=(0.2783,0.255871,0.899071); rgb(7pt)=(0.280333,0.278233,0.9221); rgb(8pt)=(0.281338,0.300595,0.941376); rgb(9pt)=(0.281014,0.322757,0.957886); rgb(10pt)=(0.279467,0.344671,0.971676); rgb(11pt)=(0.275971,0.366681,0.982905); rgb(12pt)=(0.269914,0.3892,0.9906); rgb(13pt)=(0.260243,0.412329,0.995157); rgb(14pt)=(0.244033,0.435833,0.998833); rgb(15pt)=(0.220643,0.460257,0.997286); rgb(16pt)=(0.196333,0.484719,0.989152); rgb(17pt)=(0.183405,0.507371,0.979795); rgb(18pt)=(0.178643,0.528857,0.968157); rgb(19pt)=(0.176438,0.549905,0.952019); rgb(20pt)=(0.168743,0.570262,0.935871); rgb(21pt)=(0.154,0.5902,0.9218); rgb(22pt)=(0.146029,0.609119,0.907857); rgb(23pt)=(0.138024,0.627629,0.89729); rgb(24pt)=(0.124814,0.645929,0.888343); rgb(25pt)=(0.111252,0.6635,0.876314); rgb(26pt)=(0.0952095,0.679829,0.859781); rgb(27pt)=(0.0688714,0.694771,0.839357); rgb(28pt)=(0.0296667,0.708167,0.816333); rgb(29pt)=(0.00357143,0.720267,0.7917); rgb(30pt)=(0.00665714,0.731214,0.766014); rgb(31pt)=(0.0433286,0.741095,0.73941); rgb(32pt)=(0.0963952,0.75,0.712038); rgb(33pt)=(0.140771,0.7584,0.684157); rgb(34pt)=(0.1717,0.766962,0.655443); rgb(35pt)=(0.193767,0.775767,0.6251); rgb(36pt)=(0.216086,0.7843,0.5923); rgb(37pt)=(0.246957,0.791795,0.556743); rgb(38pt)=(0.290614,0.79729,0.518829); rgb(39pt)=(0.340643,0.8008,0.478857); rgb(40pt)=(0.3909,0.802871,0.435448); rgb(41pt)=(0.445629,0.802419,0.390919); rgb(42pt)=(0.5044,0.7993,0.348); rgb(43pt)=(0.561562,0.794233,0.304481); rgb(44pt)=(0.617395,0.787619,0.261238); rgb(45pt)=(0.671986,0.779271,0.2227); rgb(46pt)=(0.7242,0.769843,0.191029); rgb(47pt)=(0.773833,0.759805,0.16461); rgb(48pt)=(0.820314,0.749814,0.153529); rgb(49pt)=(0.863433,0.7406,0.159633); rgb(50pt)=(0.903543,0.733029,0.177414); rgb(51pt)=(0.939257,0.728786,0.209957); rgb(52pt)=(0.972757,0.729771,0.239443); rgb(53pt)=(0.995648,0.743371,0.237148); rgb(54pt)=(0.996986,0.765857,0.219943); rgb(55pt)=(0.995205,0.789252,0.202762); rgb(56pt)=(0.9892,0.813567,0.188533); rgb(57pt)=(0.978629,0.838629,0.176557); rgb(58pt)=(0.967648,0.8639,0.16429); rgb(59pt)=(0.96101,0.889019,0.153676); rgb(60pt)=(0.959671,0.913457,0.142257); rgb(61pt)=(0.962795,0.937338,0.12651); rgb(62pt)=(0.969114,0.960629,0.106362); rgb(63pt)=(0.9769,0.9839,0.0805)},
colorbar
]
\addplot [forget plot] graphics [xmin=0.5, xmax=41.5, ymin=0.5, ymax=41.5] {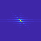};
\end{axis}

\begin{axis}[%
width=1.603in,
height=2.454in,
at={(5in,1.742in)},
scale only axis,
point meta min=0,
point meta max=1,
xmin=1,
xmax=41,
ymin=0.58139991904902,
ymax=1.36511725949153,
axis background/.style={fill=white},
title={Spectrum of $\text{\upshape Id} - Q_K^\infty + P^\infty$},
legend style={legend cell align=left, align=left, draw=white!15!black}
]
\addplot [color=mycolor1, draw=none, mark=asterisk, mark options={solid, mycolor1}, mark size=3pt, line width=.7pt]
  table[row sep=crcr]{%
1	1.36511725949153\\
2	1.14552857075169\\
3	1.02191995928725\\
4	1.0153253478922\\
5	1.00630196182876\\
6	1.00515335673159\\
7	1.00265692079567\\
8	1.00233832975967\\
9	1.00129611928752\\
10	1.00119434605329\\
11	1.00067097058316\\
12	1.00063239651377\\
13	1.00034303451186\\
14	1.00032826306925\\
15	1.00015943485809\\
16	1.00015375293409\\
17	1.00005693760009\\
18	1.00005523003641\\
19	1.00000841567084\\
20	1.00000817175586\\
21	1\\
22	0.996984046278199\\
23	0.996971614128155\\
24	0.995605479245064\\
25	0.995589525504715\\
26	0.993711995009559\\
27	0.993698203515983\\
28	0.990925709121733\\
29	0.990912576011023\\
30	0.986657659223269\\
31	0.986517890701875\\
32	0.97977251939067\\
33	0.979128980549852\\
34	0.967869381438925\\
35	0.965518309218185\\
36	0.943940068736234\\
37	0.937328437844638\\
38	0.905402080944446\\
39	0.885240110562967\\
40	0.65703159328013\\
41	0.58139991904902\\
};
\addlegendentry{Singular values}

\addplot [color=mycolor2, draw=none, mark=asterisk, mark options={solid, mycolor2}, mark size=3pt, line width=.7pt]
  table[row sep=crcr]{%
1	0.666633330796082\\
2	0.833238891790422\\
3	0.899748711503905\\
4	0.933240136045353\\
5	0.951729559418231\\
6	0.964608852696851\\
7	0.971401207692496\\
8	1.00605149707723\\
9	0.978813678094083\\
10	0.981510624157415\\
11	0.986471049758124\\
12	0.987651385713735\\
13	0.990962395504706\\
14	0.991538892955537\\
15	0.998866221630176\\
16	0.993767586933855\\
17	0.99407529885827\\
18	0.995807522992904\\
19	0.99563111964584\\
20	0.9970753796812\\
21	0.996968322987335\\
22	1\\
23	1\\
24	0.999999999999999\\
25	1\\
26	0.999999999999999\\
27	1\\
28	1\\
29	0.999999999999999\\
30	0.999999999999999\\
31	1\\
32	1\\
33	1\\
34	1\\
35	1\\
36	1\\
37	1\\
38	1\\
39	1\\
40	0.999999999999999\\
41	1\\
};
\addlegendentry{Eigenvalues}

\end{axis}

\end{tikzpicture}%
\caption{\label{DistributionSingularValues}Truncation of $Q^\infty$ for $K = 40$ and corresponding spectrum of $Id - Q_K^\infty + P^\infty$. The minimum singular value of $Id - Q_K^\infty + P^\infty$ takes the value $0.6754$.}
\end{figure}

\begin{figure}
\centering
\input{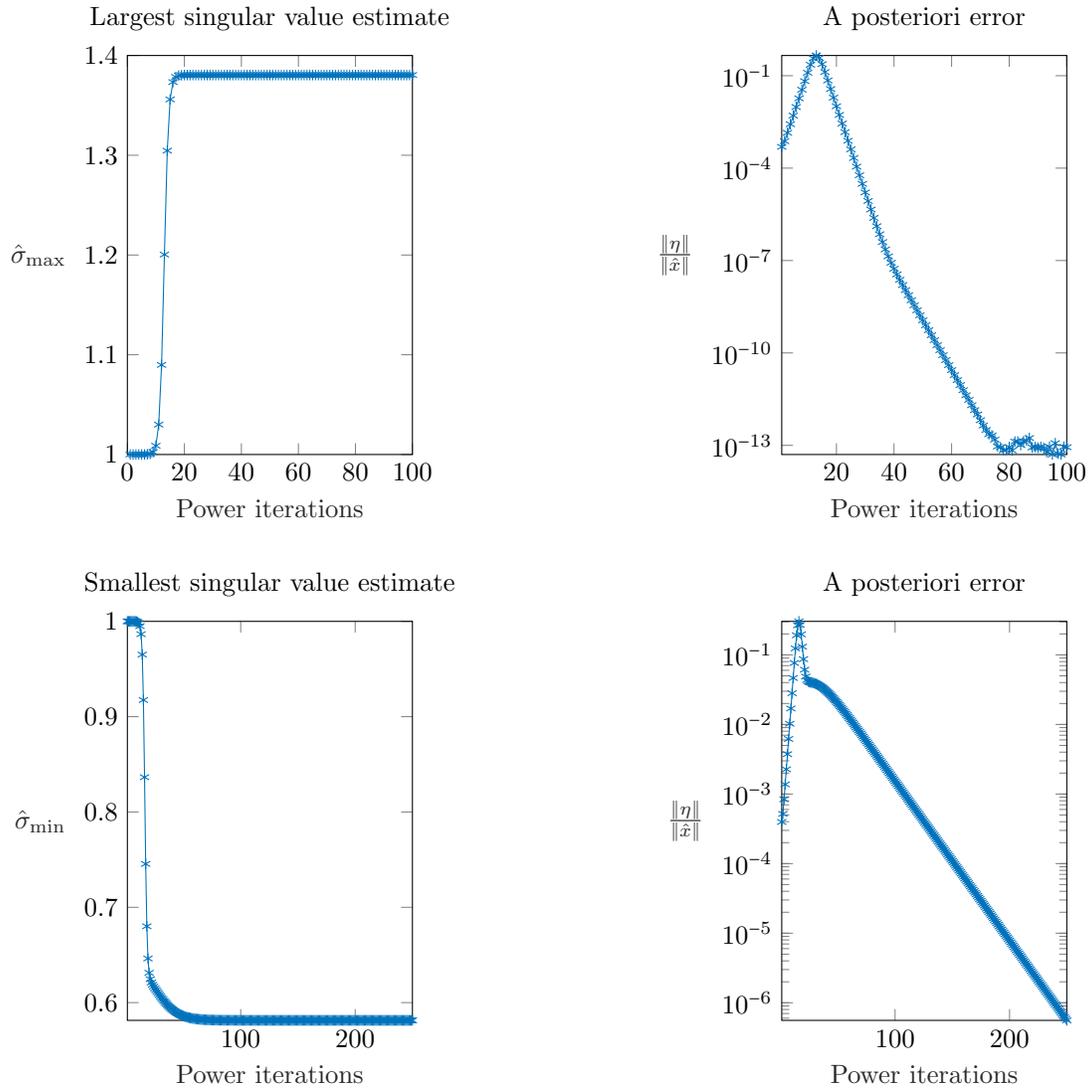}
\caption{Estimates on the largest (Top) and smallest (Bottom) singular value and corresponding a posteriori error bounds~\eqref{posterioriErrorBound} (we use $\bs x$ and $\mu$ to denote the iterates of the power method) obtained through the Power Iterations on the matrices $M = (I-Q_K^\infty + P^\infty)^*(I-Q_K^\infty + P^\infty)$ and $\lambda_{\max}I - (I-Q_K^\infty + P^\infty)^*(I-Q_K^\infty + P^\infty)$. The code that was used to generate the figures is available at \url{http://www.augustincosse.com/research}}
\end{figure}

\section{\label{proofDirichletNoInterpolation}Proof of lemma~\ref{interpolationPolynomialTangRecht}}

The proof is a transposition of the proof of Proposition 2.1  in~\cite{candes2014towards} to the Dirichlet kernel. We therefore save ourselves some of the details. Proving the conditions on $\eta(\theta)$ can be done following the approach in this paper and is summarized by lemma~\ref{lemmaUndecimatedPolynonialSatisfiesConditionsDualCert} below. Consider the interpolating polynomial $\eta(\theta)$ defined for the $S$-atomic measure $\mu(t) = \sum_{\tau \in S} \alpha_\tau \delta(t-\tau)$, as
\begin{align}
\eta(\theta) = \sum_{\tau \in S} a_\tau D_N(\theta - \tau) + \sum_{\tau\in S} b_\tau D'_N(\theta - \tau)\label{simpleInterpolatingPolynomialEtaDirichlet}
\end{align}
where $D_N(\theta) = \frac{1}{2n+1} \sum_{k=-n}^n e^{2\pi i k \theta}$ denotes the normalized Dirichlet kernel and the coefficients $a, b$ satisfy the linear system
\begin{align}
\left[\begin{array}{cc}
D_{0} & \frac{1}{\sqrt{\left|D_{N}''(0)\right|}}D_{1}\\
-\frac{1}{\sqrt{\left|D_{N}''(0)\right|}} D_{1} & -\frac{1}{\left|D_{N}''(0)\right|} D_{2}\\
\end{array}\right] \left[\begin{array}{c}
a\\
b\sqrt{\left|D_{N}''(0)\right|}
\end{array}\right]  =  \left[\begin{array}{c}
v\\
0\end{array}\right]\label{linearSystemInteroplationNonDecimated}
\end{align} 
with $(D_{0})_{j,k} =  D_{N}(\tau_j - \tau_k)$, $(D_{1})_{j,k} = D'_{N}(\tau_j - \tau_k)$ and $(D_{2})_{j,k} = D_N''(\tau_j - \tau_k)$.

\begin{lemma}\label{lemmaUndecimatedPolynonialSatisfiesConditionsDualCert}
Consider the $S$-atomic measure $\mu$ with support $S\subset \mathbb{R}$, i.e. $\mu(t) = \sum_{\tau\in S}\delta(t-\tau)$.  Let $\Delta$ denote the minimum separation distance between the atoms $\Delta  = \min_{\tau\neq \tau'}|\tau - \tau'|$ where $|\tau - \tau'|$ denotes the modulo $1$ distance.  As soon as
$\Delta\gtrsim \lambda_c \log(S)$, the system~\eqref{linearSystemInteroplationNonDecimated} is invertible, the coefficients $a$ and $b$ are thus well defined and the polynomial $\eta(\theta)$ defined in~\eqref{simpleInterpolatingPolynomialEtaDirichlet} satisfies conditions~\ref{condition1} and~\ref{condition2} in proposition~\ref{prop:uniqueRecovery}. 
\end{lemma}
  
\begin{proof}
First note that we have $D_N(0) = 1$, $D_N'(0) = 0$ as well as $D_N''(0)  = \frac{-4\pi^2n(n+1)}{3}$. Moreover, the first 3 derivatives of $D_N(\theta)$ can expand as 
\begin{align}
D_N(\theta) &= \frac{1}{2n+1} \frac{\sin((2n+1)\pi\theta)}{\sin(\pi \theta)}\label{boundDNderivativeOrder0}\\
D_N'(\theta)& = \frac{1}{2n+1} \left(\frac{(2n+1)\pi \cos((2n+1)\pi \theta)\sin(\pi \theta) - \sin((2n+1)\pi \theta) \pi \cos(\pi \theta)}{\sin^2(\pi \theta)}\right)\\
& = \frac{\pi}{\sin(\pi \theta)} \left(\cos((2n+1)\pi \theta) - \frac{\sin((2n+1)\pi \theta)}{(2n+1)} \frac{\cos(\pi \theta)}{\sin(\pi \theta)}\right)\label{boundDNderivativeOrder1}\\
D_N''(\theta)& = \pi \left(\frac{\pi\sin(\pi \theta(2n+1))}{(2n+1)} - \pi \sin(\pi \theta (2n+1)) (2n+1)\right.\\
&\left. - \frac{\pi\cos(\pi \theta)\cos(\pi \theta (2n+1))}{\sin(\pi \theta)}  + \frac{\pi \cos^2(\pi \theta)\sin(\pi \theta(2n+1))}{\sin^2(\pi \theta)(2n+1)}\right) \frac{1}{\sin(\pi \theta)} \\
& - \frac{\pi^2 \cos(\pi \theta)}{\sin^2(\pi \theta)} \left\{\cos(\pi \theta(2n+1)) - \frac{\cos(\pi \theta) \sin(\pi \theta(2n+1))}{\sin(\pi \theta) (2n+1)}\right\}\\
& = \frac{\pi^2 (2n+1)}{\sin(\pi \theta)} \left\{\frac{\sin(\pi \theta)}{(2n+1)^2} - \sin(\pi \theta(2n+1)) - \frac{\cos(\pi \theta)\cos(\pi \theta (2n+1))}{\sin(\pi \theta) (2n+1)}\right.\\ 
&\left.+ \frac{\cos^2(\pi \theta) \sin(\pi \theta(2n+1))}{\sin^2(\pi \theta) (2n+1)^2} -\frac{\cos(\pi \theta) \cos(\pi \theta(2n+1))}{\sin(\pi \theta) (2n+1)} + \frac{\cos(\pi \theta) \sin(\pi \theta (2n+1) )}{ \sin^2(\pi \theta) (2n+1)^2}\right\}\label{boundDNderivativeOrder2}\\
D_N'''(\theta) & = \frac{3\pi^3\cos(\pi \theta (2n+1))}{\sin(\pi \theta)} - \frac{\pi^3 \cos(\pi \theta (2n+1))(2n+1)^2}{\sin(\pi \theta)} + \frac{6\pi^3 \cos^2(\pi \theta)\cos(\pi \theta(2n+1))}{\sin^3(\pi \theta)} \\
&- \frac{6\pi^3 \cos^3(\pi \theta)\sin(\pi \theta (2n+1))}{\sin^4(\pi \theta)(2n+1)} + \frac{3\pi^3 \cos(\pi \theta) \sin(\pi \theta(2n+1))(2n+1)}{\sin^2(\pi \theta)}\\
& - \frac{5\pi^3\cos(\pi \theta) \sin(\pi \theta(2n+1))}{\sin^2(\pi \theta)(2n+1)}
\label{boundDNderivativeOrder3}
\end{align}
From those expressions, using $\sin(\pi \theta ) \geq 2 \left|\theta\right|_{\text{mod}\; 1}$, we get the upper bounds
\begin{align}
D_N'(\tau - \tau') &\leq \frac{\pi}{|\tau-\tau'|_{\text{mod}\;1}} \left(1+ \frac{1}{(2n+1)|\tau - \tau'|_{\text{mod}\; 1}}\right)\label{boundDirichletOrder1simplified}\\
D_N''(\tau - \tau') & \leq \frac{\pi^2(2n+1)}{|\tau - \tau'|_{\text{mod}\; 1}} \left(\frac{1}{(2n+1)^2} + 1  + \frac{2}{|\tau - \tau'|_{\text{mod}\; 1}(2n+1)} \right|
\label{boundDirichletOrder2simplified}\\
D_N'''(\tau - \tau') &\leq 24\pi^3\left\{ \frac{(2n+1)^2}{\sin(\pi (\tau - \tau'))} \vee \frac{1}{\sin^3(\pi (\tau - \tau'))}  \right\}\label{boundDNderivativeOrder3afterCleaning}
\end{align}
Also note that using~\eqref{boundDNderivativeOrder0} to~\eqref{boundDNderivativeOrder2}, and introducing the quantity $a(\theta)\equiv \frac{2}{\pi \left(1 - \frac{\pi^2 |\theta|^2}{6}\right)}$, we can derive bounds similar to Lemma 2.6 in~\cite{candes2014towards}. In this case, we simply write 

\begin{align}
D_N(\theta) \leq \left\{\begin{array}{ll}
\frac{1}{2(2f_c+1)|\theta|} & \lambda_c \leq |\theta|\leq \frac{\sqrt{2}}{\pi}\\
\frac{a(\theta)}{2(2f_c+1)|\theta|}& \frac{\sqrt{2}}{\pi}\leq |\theta|\leq \frac{1}{2}.  
\end{array}\right.\label{upperboundDwitha}
\end{align}
\begin{align}
D'_N(\theta) \leq \left\{\begin{array}{ll}
\frac{\pi}{|\theta|} & \lambda_c \leq |\theta|\leq \frac{\sqrt{2}}{\pi}\\
\frac{\pi a(\theta)}{|\theta|}& \frac{\sqrt{2}}{\pi}\leq |\theta|\leq \frac{1}{2}.  
\end{array}\right.\label{upperboundDprimewitha}
\end{align}
as well as 
\begin{align}
D''_N(\theta) \leq \left\{\begin{array}{ll}
\frac{3\pi^2(2f_c+1)}{|\theta|}& \lambda_c \leq |\theta|\leq \frac{\sqrt{2}}{\pi}\\
\frac{3\pi^2 (2f_c+1)a(\theta)}{|\theta|}& \frac{\sqrt{2}}{\pi}\leq |\theta|\leq \frac{1}{2}.  
\end{array}\right.\label{upperboundDprimeprimewitha}
\end{align}
If we use $\Delta  \leq \min_{\tau \neq \tau'}|\tau - \tau'|_{\text{mod}\; 1}$, we also have 
\begin{align}
D_N(\tau - \tau') &\leq \frac{1}{4\Delta f_c }\\
\frac{D_N'(\tau - \tau')}{\sqrt{|D_N''(0)|}} & \leq  \frac{2\pi f_c}{\Delta f_c} \frac{\sqrt{3}}{2\pi f_c}\\
\frac{D_N''(\tau - \tau')}{|D_N''(0)|} &\leq \frac{12\pi^2 f_c^2}{\Delta f_c} \frac{3}{4\pi^2 f_c^2}
\end{align}
And hence 
\begin{align}
\left\|D_0 - I\right\|_\infty &\leq \frac{\log(S)}{4 f_c \Delta}\label{boundNormD0}\\
\left\|\frac{D_1}{\sqrt{|D_N''(0)|}}\right\|_\infty &\leq \frac{\sqrt{3}\log(S)}{\Delta f_c}\\
\left\|\frac{D_2}{|D_N''(0)|} - I\right\|_\infty &\leq \frac{9\log(S)}{4 \Delta f_c} \label{boundNormD2}
\end{align}
From this, using Gershgorin (see~\cite{tang2013compressed} for details), we can write 
\begin{align}
\left\|I-D\right\| &\leq \max\left\{\|I- D_0\|_\infty + \left\|\frac{1}{\sqrt{|D_N''(0)|}} D_1\right\|_\infty, \left\|\frac{1}{\sqrt{|D_N''(0)|}}D_1\right\|_\infty + \left\|I - \left( - \frac{1}{|D_N''(0)|} D_2\right)\right\|_\infty\right\}\\
&\leq \max \left\{\frac{\log(S)}{4f_c\Delta} + \frac{\sqrt{3}\log(S)}{\Delta f_c},  \frac{\sqrt{3}\log(S)}{\Delta f_c} + \frac{9\log(S)}{4\Delta f_c}\right\}\label{operatorNormD}
\end{align}
In particular, this last line can be made sufficiently smaller than $1$ as soon as $\Delta \geq \lambda_c \log(S)$. The matrix $D$ is invertible and we have $\|D\|_{\text{op}} \leq \|I-D\| + \|I\|\leq 1+ \varepsilon$. The system~\eqref{linearSystemInteroplationNonDecimated} is thus invertible as soon as $\Delta$ is sufficiently larger than $\lambda_c \log(S)$ and the solution $(a,b)$ is well defined. Using the Schur complement, we can express this solution as
\begin{align}
\left[\begin{array}{c}
a\\
b
\end{array}\right] = \left[\begin{array}{c}
I\\
-\sqrt{|D_N''(0)|}D_2^{-1}D_1\end{array}\right]\left(D_0 - D_1D_2^{-1}D_1\right)^{-1} v \label{expressionAandB001}
\end{align}

where $v$ is the vector defined as $v = [\sign(\alpha_1), \ldots, \sign(\alpha_{|S|})]$. Using the bounds~\eqref{boundNormD0} to~\eqref{boundNormD2}, and following~\cite{candes2014towards}, we can write  
\begin{align}
\left\|a\right\|_\infty &\leq \left\|(D_0 - D_1D_2^{-1}D_1)^{-1}\right\|_\infty\|v\|_\infty\label{boundAlpha}\\
\left\|\sqrt{|D_N''(0)|}\; b\right\|_\infty &\leq \sqrt{|D_N''(0)|} \|D_2^{-1}D_1\|_\infty\left\|(D_0 - D_1D_2^{-1}D_1)^{-1}\right\|_\infty\|v\|_\infty\label{boundBeta}
\end{align}
Note that (see~\cite{candes2014towards} for details)
\begin{align}
\left\|\left(\frac{D_2}{D_N''(0)}\right)^{-1}\right\|_\infty \leq \frac{1}{1 - \left\|I- \frac{D_2}{D_N''(0)}\right\|_\infty} \leq \frac{1}{1- \frac{9\log(S)}{4\Delta f_c}} \equiv h 
\end{align}
which can be made sufficiently close to one by taking $\Delta$ sufficiently larger than $\lambda_c\log(S)$. we also have 
\begin{align}
\left\|I- (D_0 - D_1D_2^{-1}D_1)\right\|_\infty &\leq \|I- D_0\|_\infty + \left\|D_1\right\|_\infty^2 \|D_2^{-1}\|_\infty \\
&\leq \frac{\log(S)}{4\Delta f_c} + \left(\sqrt{3} \frac{\log(S)}{\Delta f_c}\right)^2 \frac{1}{1- \frac{9\log(S)}{4\Delta f_c}} 
\end{align}
In particular, using a Neumann series, we can write 
\begin{align}
\left(D_0 - D_1D_2^{-1}D_1\right)^{-1} = &=\sum_{k=0}^\infty \left(I - D_1D_2^{-1}D_1\right)^k = I + \sum_{k=1}^\infty \left(I - D_1D_2^{-1}D_1\right)^k\label{NeumannWithoutNorm}
\end{align}
which implies a natural bound on the norm of this inverse,
\begin{align}
\|(D_0 - D_1D_2^{-1}D_1)^{-1}\|_\infty &\leq 1+ \sum_{k=1}^\infty \left\|I - D_1D_2^{-1}D_1\right\|_\infty^k \leq 1+ \frac{h}{1-h}. 
\end{align}
Note that from~\eqref{NeumannWithoutNorm} and~\eqref{expressionAandB001}, we also have 
\begin{align}
|\Ip(\alpha_\tau)| \leq \frac{h}{1-h}, \quad \Rp(\alpha_\tau) \geq 1 - \frac{h}{1-h}. 
\end{align}
Again, this last bound can be made sufficiently small by taking $\Delta$ larger than $\lambda_c\log(S)$. Substituting those two bounds in~\eqref{boundAlpha} and~\eqref{boundBeta}, we get 
\begin{align}
\left\|a\right\|_\infty &\leq 1 + \left(\frac{\log(S)}{4\Delta f_c} + \left(\sqrt{3} \frac{\log(S)}{\Delta f_c}\right)^2 \frac{1}{1- \frac{9\log(S)}{4\Delta f_c}} \right)\\
&\leq 1 + \varepsilon_1\label{boundA}\\
\left\|b\right\|_\infty &\leq \frac{1}{\sqrt{|D_N''(0)|}}\left(\frac{\log(S)}{4\Delta f_c} + \left(\sqrt{3} \frac{\log(S)}{\Delta f_c}\right)^2 \frac{1}{1- \frac{9\log(S)}{4\Delta f_c}} \right)\frac{1}{1- \frac{9\log(S)}{4\Delta f_c}}  \left(\frac{\sqrt{3}\log(S)}{\Delta f_c}\right)^2\\
&\leq \frac{\varepsilon_2 \sqrt{3}}{2\pi\sqrt{f_c(f_c+1)}}\leq \frac{\varepsilon_2' }{f_c}\label{boundB}
\end{align}

We now show how to control the modulus $|\eta(\theta)|$. We will use the following bounds which are transposition of the bounds from lemma 2.3 in~\cite{candes2014towards} to the Dirichlet kernel. As in this last paper, the bounds follow from series expansion of the kernel and its derivatives around the origin.
\begin{align}
D_n(\theta) &\geq 1 - \frac{4\pi^2}{3} n(n+1)\frac{\theta^2}{2}\label{boundTaylorDn}\\
D_n'(\theta)& \leq \frac{4\pi^2}{3} n(n+1)\theta\label{boundTaylorDn1}\\
D_n''(\theta)&\leq -\frac{4\pi^2 n (n+1)}{3} + \frac{16n(n+1)(3n^2 + 3n-1)}{15}\frac{\theta^2}{2}\\
&\leq -\frac{4\pi^2 n (n+1)}{3}  + \frac{8n(n+1)(3n^2 +3n-1) \theta^2}{15}\label{boundTaylorDn2}\\
\left|D_n''(\theta)\right|&\leq \frac{4\pi^2 n (n+1)}{3}\label{boundTaylorDn2mod}\\
\left|D_n'''(\theta)\right|&\leq \frac{16 \pi^3 n(n+1)(3n^2 +3n-1) \theta}{15}\label{boundTaylorDn3}
\end{align}


In~\eqref{boundTaylorDn3} we use the Faulhaber's formula (see~\cite{faulhaberacademia,bernoulli1713ars})
\begin{align}
\sum_{k=1}^n k^m = \frac{1}{m+1} \sum_{k=0}^m (-1)^k {m+1\choose k} B_k n^{m+1-k}\label{Faulhabers}
\end{align}
where the $B_k$ are the Bernoulli numbers of the first kind ($B_1 = -\frac{1}{2}$). In particular, we have 
\begin{align}
\sum_{k=1}^n k^4 = \frac{n(n+1)(2n+1)(3n^2 + 3n-1)}{30}
\end{align}
By definition of $a$ and $b$ we always $\eta(\tau)\leq 1$ as well as $\eta'(\tau) = 0$ for all $\tau\in S$. To show that $|\eta(\theta)|<1$ for all $\theta \notin S$, we start by showing that $\eta''(\theta)<0$ in the interval $(\tau, \tau \pm \lambda_c/5]$, that is to say for $|\theta - \tau| \in (0,\lambda_c/5)$. Since we have $\eta'(\tau) = 0$, to show $|\eta(\theta)| < 1$ on $(\tau, \tau+\lambda_c/5)$, it suffices to show that $\eta''(\theta)$ is negative on that interval. The second derivative expands as 
\begin{align}
\frac{d^2|\eta|(\theta)}{d\theta^2} = \frac{|\eta'|^2 + \eta_R\eta_R'' + \eta_I\eta_I''}{|\eta|} - \frac{(\eta_R\eta_R' + \eta_I\eta_I')^2}{|\eta|^3}
\end{align}
We thus need to show that $\eta''(\theta)$ is negative on $(\tau, \tau + \lambda_c/5]$. Without loss of generatlity, we only focus on $(\tau_0, \tau_0+\tau + \lambda_c/5)$. The result remains true for any $\tau\in S$. Following~\cite{candes2014towards}, and labelling as $\tau_0, a_0, b_0$ the nearest atom to $\theta$ and its corresponding coefficients in $\eta(\theta)$, we have 
\begin{align}
\eta_R(\theta) \geq \Rp(a_0) D_N(\theta - \tau_0) - \|a\|_\infty \left|\sum_{\tau\neq \tau_0} D_N(\theta - \tau)\right| - \|b\|_\infty |D_N'(\theta)| - \|b\|_\infty \left|\sum_{\tau\neq \tau_0} D_N(\theta - \tau)\right|
\end{align}
Now using~\eqref{boundTaylorDn} as well as~\eqref{boundTaylorDn1} and the fact that $\theta \leq \frac{\Delta}{2}$, and noting that $a(\theta)\leq 1.1$, we get 
\begin{align}
|\eta_R(\theta)|\geq (1+\varepsilon_1) \left(1 - \frac{4\pi^2}{3} f_c(f_c+1)\frac{\theta^2}{2}\right) -2 (1+\varepsilon_1) \left(\frac{\log(S)}{\Delta f_c} + \frac{4}{\Delta f_c}\right) - 2\frac{\varepsilon_2}{f_c} \left(\frac{2\pi\log(S)}{\Delta}\right) - \varepsilon_2 \frac{8\pi^2}{3}. \geq 0.71\label{boundOnEtaR}
\end{align}
as soon as $f_c\geq 100$ and for $\theta\leq \frac{f_c}{5}$. For $\eta_I(\theta)$, we write 
\begin{align}
\left|\eta_I(\theta)\right| &\leq \left|\sum_{\tau\in S} \Ip(a_\tau) D_N(\theta - \tau) + \sum_{\tau\in S}\Ip(b_\tau)D_N'(\theta - \tau)\right|\\
&\leq \left|\Ip(a_{\tau_0})\right| |D_N(\theta - \tau_0)| + \|a\|_\infty \sum_{\tau\neq \tau_0}\left|D_N(\theta - \tau)\right| + \|b\|_\infty \left|D_N'(\theta - \tau_0)\right| + \|b\|_\infty \left|\sum_{\tau \neq \tau_0} D_N'(\theta - \tau)\right| \label{ImaginaryEtaAfterSubstitutingBoundDDprimeBeforeSubstitution}\\
&\leq \frac{h}{1-h} + (1+\varepsilon_1) \left(\frac{2\log(S)}{\Delta f_c} + \frac{2}{\Delta f_c}\right) +\frac{\varepsilon_2'}{f_c}\left(\frac{4\pi^2}{3} n(n+1)\right)\frac{1}{2f_c} + \frac{\varepsilon_2'}{f_c} \left(2.2\frac{\log(S)}{\Delta} + \frac{2}{\Delta}\right) \label{ImaginaryEtaAfterSubstitutingBoundDDprime}
\end{align}
\eqref{ImaginaryEtaAfterSubstitutingBoundDDprime} follows from substituting~\eqref{boundTaylorDn} to~\eqref{boundTaylorDn1} as well as~\eqref{upperboundDwitha} to~\eqref{upperboundDprimewitha} inside~\eqref{ImaginaryEtaAfterSubstitutingBoundDDprimeBeforeSubstitution}. This last line can be made sufficiently small provided that $\Delta$ is taken sufficiently larger than $\lambda_c\log(S)$. Similar results hold for $\eta_R''(\theta)$ and $\eta_I''(\theta)$ for which we have 
\begin{align}
\eta_R''(\theta) =  \sum_{\tau\in S} \Rp(a_\tau) D_N''(\theta - \tau) + \sum_{\tau \in S} \Rp(b_\tau) D_N'''(\theta - \tau)
\end{align}
From this we have 
\begin{align}
\eta_R''(\theta) &= \sum_{\tau \in S} \Rp (\alpha_\tau) D_N''(\theta - \tau) + \sum_{\tau \in S} \Rp(\beta)D_N'(\theta - \tau)\\
&\leq \Rp(\alpha_0) D_N''(\theta - \tau_0) + \sum_{\tau - \tau_0} \|a\|_\infty \left|D_N''(\theta - \tau)\right| + \|b\|_\infty \left|D_N'''(\theta - \tau_0)\right| + \sum_{\tau \neq \tau_0} \left|D_N'''(\theta - \tau)\right|\label{expressionetaprimeprime}
\end{align}
Using $\Rp(\alpha_0) \leq 1 - \frac{h}{1-h}$ as well as 
\begin{align}
D_N''(\theta - \tau_0)& \leq -\frac{4\pi^2n(n+1)}{3}+ \frac{8n(n+1)(3n^2 + 3n-1)\theta^2}{15}\\
&\leq \frac{-4\pi^2 n (n+1)}{3} + \frac{8n(n+1)}{15} \frac{3n(n+1)}{25}, \quad \text{when $\theta \leq \frac{\lambda_c}{5}$}\\
&\leq \left(-\frac{4\pi^2}{3} + \left(\frac{8}{15} \frac{3}{25} (1.01)\right)\right)n(n+1), \quad \text{as soon as $n\geq 100$.}
\end{align}
Using~\eqref{boundDirichletOrder2simplified} and~\eqref{boundAlpha}, we can thus control the first term in~\eqref{expressionetaprimeprime} as
\begin{align}
&\Rp(\alpha_0) D_N''(\theta - \tau_0) + \sum_{\tau - \tau_0} \|a\|_\infty \left|D_N''(\theta - \tau)\right|\\
&\leq \left(1 - \frac{h}{1-h}\right) n(n+1) \left(-\frac{4\pi^2}{3} + 0.0646\right) + \sum_{\tau\neq \tau_0} \|a\|_\infty \left|D_N''(\theta - \tau)\right|\\
&\leq -13 n(n+1) + \left(1+\varepsilon_1\right) \frac{4\pi^2(2n+1)n}{n\Delta}\left(2 + 2\log(S)\right)\label{boundDerivativeInterpolatingPolyOrder2}
\end{align}
which can be made less than $-13f_c(f_c+1)$ as soon as $\Delta \geq 12\pi^2 (1+\varepsilon_1)(2+2\log(S))\lambda_c$

For the third order derivatives, recall that from~\eqref{boundDNderivativeOrder3afterCleaning}, we have 
\begin{align}
D_N'''(\theta - \tau)&\leq \frac{\pi^3 f_c^3}{8f_c^3|\theta|^3_{\text{mod}\, 1}} \left(16 + 9\pi^2 + 9\pi\right), \quad \theta \leq \lambda_c.
\end{align}
This last line can be made less than $\varepsilon_3f_c^2$ as soon as $\Delta \geq 2\lambda_c \log(S)(16 + 9\pi^2+9\pi) $. Together with $\|b\|_\infty \leq \varepsilon_2/f_c$ and $|D_N'''(\theta - \tau_0)|\leq \frac{16 n (n+1)}{15} \frac{3n(n+1)}{25n}$, we can thus write 
\begin{align}
\left|\sum_{\tau} b_\tau D_N'''(\theta - \tau)\right|&\leq \|b\|_\infty D_N'''(\theta - \tau_0) + \sum_{\tau\neq \tau_0}\|b\|_\infty \left|D_N'''(\theta - \tau)\right|\\
& \leq \varepsilon_2' \frac{\pi^3f_c^2}{8}\left(\frac{8}{\Delta^3f_c^3} + \frac{2\log(S)}{f_c^3\Delta^3}\right) (16 + 9\pi^2+9\pi)  + \varepsilon_2' \frac{16}{15} \frac{3.03}{25} f_c(f_c+1)\label{boundDerivativeInterpolatingPolyOrder3}
\end{align}
The last line holds as soon as $f_c\geq 100$. Finally, grouping~\eqref{boundDerivativeInterpolatingPolyOrder2} and~\eqref{boundDerivativeInterpolatingPolyOrder3}, we get 
\begin{align}
\eta_R''(\theta) \leq -13 f_c(f_c+1) + \varepsilon_2'' f_c(f_c+1) < 0\label{boundEtaRpprime}
\end{align}
A similar reasoning holds for the derivative of the imaginary part, for which we have 
\begin{align}
\left|\eta_I''(\theta)\right| &\leq \Ip(a_0) \left|D_N''(\theta - \tau_0)\right| + \|a\|_\infty \sum_{\tau \neq \tau_0} \left|D_N''(\theta - \tau_0)\right|\\
& + \Ip(b_0) \left|D_N''(\theta - \tau_0)\right| + \left\|b\right\|_\infty\sum_{\tau \neq \tau_0} \left|D_N'''(\theta  -\tau)\right|\\
&\leq \frac{h}{1-h} \frac{4\pi^2n(n+1)}{3} + \frac{\varepsilon_2'}{f_c} \frac{16n(n+1)(3n^2 + 3n - 1)}{15}\frac{1}{5f_c}\\
&+ \left(1+ \frac{h}{1+h}\right) \frac{4\pi^2 (2n+1)}{\Delta f_c}f_c \left(2+ 2\log(S)\right) + \frac{\varepsilon_2'}{f_c} \frac{\pi^3 f_c^3}{8f_c^3} \left(2+ 2\log(S)\right) (16+9\pi^2 + 9\pi).\label{boundEtaIpprime}
\end{align}
provided that $\theta \leq f_c/5$. Finally for $|\eta'(\theta)|$, we write
\begin{align}
\left|\eta'(\theta)\right| &\leq \left|\sum_{\tau \in S} a_\tau D_N'(\theta - \tau) + \sum_{\tau\in S} b_\tau D_N''(\theta - \tau)\right|\\
&\leq \|a\|_\infty \left|D_N'(\theta - \tau)\right| + \sum_{\tau \neq \tau_0} \|a\|_\infty \left|D_N'(\theta - \tau)\right|\\
&+ \|b\|_\infty \left|D_N'(\theta - \tau_0)\right| + \sum_{\tau \neq \tau_0} \|b\|_\infty \left|D_N''(\theta - \tau)\right|\\
&\stackrel{(a)}{\leq } (1+\varepsilon_1)\frac{4\pi^2 n (n+1)\theta}{3} + (1+\varepsilon_1)\frac{2\pi f_c}{f_c} \left(2+2\log(S)\right)\label{tmpStepBoundModulusEtaprime1}\\
&+ \frac{\varepsilon_2'}{f_c} \frac{4\pi^2n(n+1)}{3} + \frac{4\varepsilon_2'}{f_c} \frac{\pi^2 (2n+1)f_c}{f_c\Delta} \left(2+ 2\log(S)\right)\label{tmpStepBoundModulusEtaprime2}\\
&\stackrel{(b)}{\leq } (1+\varepsilon_1) \left( \frac{8\pi^2 }{15} + \frac{1}{6\pi}\right) f_c + \varepsilon_2' \left(1+ \frac{8\pi^2}{3}\right)f_c\label{boundModulusEtaPrime}
\end{align}
(a) follows from~\eqref{boundDirichletOrder1simplified} and~\eqref{boundDirichletOrder2simplified} ~\eqref{boundTaylorDn1} and~\eqref{boundTaylorDn2mod} as well as~\eqref{boundA} and~\eqref{boundB}.  (b) follows from $\theta \leq \lambda_c/5$ as well as $n\geq 1$ and $\Delta \geq 12\lambda_c (2+2\log(S)) \pi^2$. 

To conclude, we must show that $|\eta(\theta)|<1$ on every $(\tau_\ell+\lambda_c/5, \tau_\ell + \frac{\tau_{\ell+1} - \tau_\ell}{2}]$ as well as every $\theta \in [\tau_\ell + \frac{\tau_{\ell-1} - \tau_\ell}{2}, \tau_\ell - \lambda_c/5]$. To bound the modulus, we follow~\cite{candes2014towards} and use a series expansion around the origin for $D_N(\theta - \tau_0)$. Relying once again on Faulhaber's formula~\eqref{Faulhabers}, we write 
\begin{align}
D_N(\theta) & \leq 1 - \frac{8\pi^2 n (n+1) \theta^2}{12} + \frac{2}{2n+1}  \frac{16 \pi^4 n(n+1)(2n+1)(3n^2 + 3n - 1)}{30}\frac{\theta^4}{24}
\end{align}
Recall that we have 
\begin{align}
|\eta(\theta)| &= \left|\sum_{\tau\in S} a_\tau D_N(\theta - \tau) + \sum_{\tau\in S} b_\tau D_N(\theta - \tau)\right|\\
&\leq \left|a_0 D_N(\theta - \tau_0) + b_0 D_N'(\theta - \tau_0) + \sum_{\tau\neq \tau_0} a_\tau D_N(\theta - \tau) + \sum_{\tau\neq \tau_0} D_N'(\theta - \tau)\right|\label{boundEtaFinalremainingThetaLargerThanlambdaCon5}
\end{align}
where $a = (a_\tau)_{\tau\in S}$ and $b = (b_\tau)_{\tau\in S}$ are solutions to~\eqref{linearSystemInteroplationNonDecimated}. Using the bound on $\|a\|_\infty$ derived above, we define $H_1(\theta)$ as the bound 
\begin{align}
H_1(\theta) \equiv \|a\|_\infty \left[1 - \frac{4\pi^2 n (n+1)\theta^2}{6} + \frac{4}{3} \theta^4 \pi^4 \frac{n(n+1)(3n^2 + 3n - 1)}{30} \right] + \|b\|_\infty \frac{4\pi^2}{3} n (n+1)
\end{align}
Note that we have 
\begin{align}
H_1(\lambda_c/5) &\leq \left(1+ \frac{h}{1+h}\right)\left[1 - \frac{4\pi^2 n (n+1)}{25 n^2 3} + \frac{4}{3} \left(\frac{\pi}{5}\right)^4 \frac{n(n+1)(3n^2 + 3n-1)}{30}\right] + \frac{\varepsilon_2'}{n}\frac{4\pi^2}{3}n(n+1)\\
& = \left(1+ \frac{h}{1+h}\right) \left[1 - \frac{4\pi^2}{75} \left(1+ \frac{1}{n}\right) + \frac{4}{3} \left(\frac{\pi}{5}\right)^4 \frac{3}{30} + \frac{4}{3} \left(\frac{\pi}{5}\right)^4 \left(\frac{6}{n} + \frac{2}{n^2} - \frac{1}{n^3}\right)\right] + \frac{\varepsilon_2'}{n}\frac{4\pi^2}{3}n(n+1)\frac{1}{5n}\label{tmpBoundH1}
\end{align}
in particular, depending on the bound we fix on $\varepsilon_2'$ we can make~\eqref{tmpBoundH1} sufficiently close to $0.5518$ (corresponding to $\varepsilon_2' = 0$). I.e. 
\begin{align}
H_1(\lambda_c/5) &\leq 0.5784, \quad \text{as soon as $\varepsilon_2'\leq 0.01$}\label{boundH1a}\\
H_1(\lambda_c/5) &\leq 0.8442, \quad \text{as soon as $\varepsilon_2'\leq 0.1$}\label{boundH1b}\\
\end{align}
The derivative of $H_1(\theta)$ can expand similarly as 
\begin{align}
H_1'(\theta) &= - \|a\|_\infty \frac{4 \pi^2}{3} \theta n(n+1) \left[1 - \frac{4\theta^2 \pi^2 (3n^2 + 3n - 1)}{30}\right] + \frac{\varepsilon_2'}{n}\frac{4\pi^2}{3} (n+1)n \\
&\leq -\|a\|_\infty \frac{4}{3} \pi^2 \theta n(n+1) \left[0.9 - \frac{4\theta}{}\right]
\end{align}
In the last line, for simplicity, we assume $\varepsilon_2' \leq \|a\|_\infty 0.1 n\theta$ for all $\theta \geq \lambda_c/5$ which means $\varepsilon_2' \leq 0.02$. From this, the derivative remains negative for all $\theta = \frac{\zeta}{n}$ satisfying 
\begin{align}
\left(0.9 - \frac{4\zeta^2}{n^2}\pi^2 \frac{(3n^2 + 3n - 1)}{30}\right)>0
\end{align}
which implies $\zeta\leq \sqrt{\frac{0.9}{1.01}4\pi^2}$ and hence $\theta \leq 0.4751\lambda_c$.  From this we therefore know that $H_1(\theta)$ remains less than~\eqref{boundH1a} or~\eqref{boundH1b} depending on the choice of $\varepsilon_2'$ and remains smaller than this upper bound on the whole interval $[\lambda_c/5, 0.4751 \lambda_c]$. Using the bounds derived above, we also have 
\begin{align}
\left|\sum_{\tau \neq \tau_0} D_N(\theta - \tau) + \sum_{\tau\neq \tau_0} b_\tau D_N'(\theta - \tau)\right| &\leq (1+\varepsilon_1) \frac{1}{2\Delta f_c} (2 + 2\log(S)) + \varepsilon_2' \frac{\pi}{\Delta f_c} \left(2+2\log(S)\right).
\end{align}
As we saw above, this last quantity can be made sufficiently small by taking $\Delta$ sufficiently larger than $\lambda_c(2+2\log(S))(1+\varepsilon_1 + \varepsilon_2')$. Combining this with the bounds~\eqref{boundH1a} or~\eqref{boundH1b} on $H_1$, and substituting those expressions in~\eqref{boundEtaFinalremainingThetaLargerThanlambdaCon5}, we get 
\begin{align}
|\eta(\theta)| < \eqref{boundH1a} + O(\frac{\log(S)}{\Delta f_c}) < 1, \quad \text{for all $\theta \in [\lambda_c/5, 0.4751\lambda_c]$}
\end{align}
as soon as $\Delta$ is sufficiently larger than $\lambda_c \log(S)$. On the remaining $\tau_0 + [0.4751\lambda_c, \Delta/2]$, we simply use 
\begin{align}
|\eta(\theta)|& \leq \|a\|_\infty \sum_{\tau \in S} \left|D_N(\theta - \tau)\right| + \sum_{\tau\in S} \|b\|_\infty \left|D'_N(\theta - \tau)\right|\label{boundModulusEtaOnRemainingIntervalAlmostFinished}\\
&\leq  \frac{(1+\varepsilon_1)}{4f_c|\tau - \tau_0|} + \frac{(1+\varepsilon_1)}{2\Delta f_c} (2+2\log(S))+ \frac{\varepsilon_2'}{f_c}\frac{\pi }{|\tau-\tau_0|} + \frac{\varepsilon_2' \pi}{f_c \Delta}\left(2+2\log(S)\right)\\
&\leq \frac{(1+\varepsilon_1)}{(4)(0.4)}  + \frac{(1+\varepsilon_1)}{2\Delta f_c}(2+2\log(S)) + \frac{\varepsilon_2' \pi}{0.4} + \frac{\varepsilon_2' \pi}{f_c\Delta} (2+2\log(S))\label{boundEtaRemainingInterval}
\end{align}
In~\eqref{boundEtaRemainingInterval} we use $|\tau - \tau_0|\geq 0.4\lambda_c$. This last bound remains true on $[\Delta/2, \tau_0 + \frac{\tau_1 - \tau_0}{2}]$ as well as $[\tau_0 - \Delta/2]$. In particular, provided that $\varepsilon_2' \leq \frac{0.4}{10\pi}$ and that $\varepsilon_1<0.1$, ~\eqref{boundEtaRemainingInterval} can be made smaller than $(1.1)( 0.6250) + 0.1 < 0.7875 $. The remaining two terms in~\eqref{boundModulusEtaOnRemainingIntervalAlmostFinished} can be made arbitrarily small by taking the minimum separatino distance sufficiently larger than $(2+2\log(S)) \lambda_c$. This concludes the proof of lemma~\ref{lemmaUndecimatedPolynonialSatisfiesConditionsDualCert}.


\end{proof}

\section{\label{sec:boundPerr}Proof of lemma~\ref{boundOnNormPerrWithoutSampling}}

\boundOnNormPerrWithoutSamplingKey*

Recall that we have $(U^*U)^{-1} = \frac{1}{n+1}I + \sum_{k=1}^\infty \frac{1}{n+1}( I - \frac{1}{n+1}U^*U)^k$ as $U^*U$ is invertible. Using this expression, we get 
\begin{align}
\frac{1}{n+1}\mathcal{P}_U = \frac{1}{n+1}U(U^*U)^{-1}U^*&= \frac{1}{(n+1)^2}\sum_{k=0}^\infty U(I - \frac{1}{n+1}U^*U)^kU^*
\end{align}
To control the norm $\|\cdot \|_N$, note that we can write
\begin{align}
&\frac{1}{n+1}\left\|\mathcal{T}\left(\frac{1}{n+1}U\left(I - \frac{1}{n+1}U^*U\right)^kU^*\right)\right\|_N \\
&\leq \frac{1}{n+1}\|\mathcal{T}\|_{\|\cdot \|_F\rightarrow N} \left\|\sum_{k=1}^\infty\frac{1}{n+1}U\left(I - \frac{1}{n+1}U^*U\right)^kU^*\right\|_F, \quad \forall k\geq 1\label{perrStep1}
\end{align}
We further have $\|\mathcal{T}\|_{\|\cdot \|\rightarrow N} = \sup_{\|T\|_F\leq 1}\left\|\mathcal{T}(T)\right\|_N\leq 1$. 
%
From~\eqref{perrStep1}. The sub-multiplicativity of the Frobenius norm does the rest.
\begin{align}
\left\|\frac{1}{n+1} \sum_{k\geq 1} U\left(I - \frac{1}{n+1}U^*U\right)^kU^*\right\|_F  &\leq  
\sum_{k\geq 1} \frac{1}{n+1}  \|U\|^2_F \left\|I - \frac{1}{n+1}U^*U\right\|_F^{k} \leq |S|\sum_{k \geq 1} \left\|I - \frac{1}{n+1}U^*U\right\|_F^{k}\label{deviationToIdentity}
\end{align}
Using $\|I - \frac{1}{n+1}U^*U\|_F \leq \sqrt{|S|}\|I - \frac{1}{n+1}U^*U\|_\infty \leq \frac{\sqrt{|S|}\log|S|}{\Delta n}$, we can write 
\begin{align}
\frac{1}{n+1}\left\|\mathcal{T}\left(\frac{1}{n+1}U\left(I - \frac{1}{n+1}U^*U\right)^kU^*\right)\right\|_N  \lesssim \frac{1}{n} \left( \frac{|S|^{3/2}\log|S|}{\Delta n} \right)
\end{align}
as soon as $\Delta \gtrsim |S|^2 n^{-1} $ up to log factors. We now control the norm $N(|\eta(\theta)|^2 - n^{-2}\psi(\theta)^*UU^*\psi(\theta))$. First notice that we have 
\begin{align}
\eta\eta^* = \frac{1}{n^2} \left[U, V\right] \left[\begin{array}{cc}
\alpha\alpha^* & \alpha \beta^*\\
\beta\alpha^* & \beta\beta^*
\end{array}\right] \left[U, V\right]
\end{align}
where we let $V$ to denote the matrix whose columns are given by the derivative of the Dirichlet kernel at the atoms $\tau\in S$. We start by bounding the difference $|\eta(\theta)|^2 - \psi(\theta)^*U\alpha\alpha^*U^*\psi(\theta)$. For this first term, we have 
\begin{align}
\left\|\mathcal{T}\left\{\eta\eta^*\right\} - \mathcal{T}\left\{U\alpha\alpha^*U^*\right\}\right\|^2_N &= \sum_{|s|\leq n}\frac{1}{n+1-|s|}\frac{1}{n^4} \left|\sum_{k+\ell=s} \sum_{\tau\in S}\sum_{\tau'\in S} \left(\delta_{\tau, \tau'} - \alpha_\tau\alpha_{\tau'}\right)\left(\sum_{k = -n/2+s}^{n/2}e^{2\pi i k (\tau - \tau')}\right) e^{-2\pi i s \tau}\right|^2\\
& \stackrel{(a)}{\leq } \frac{1}{n^4\Delta^2} \sum_{|s|\leq n} \frac{1}{n+1-|s|} \left|\sum_{\tau, \tau'\in S}|\delta_{\tau, \tau'} - \alpha_\tau \alpha_{\tau'}|\right|^2\label{normNbound1}
\end{align}
In (a) we use $\left|\sum_{|k|\leq n}\right|\leq \Delta^{-1}$. To control~\eqref{normNbound1}, we use~\eqref{NeumannWithoutNorm} together with~\eqref{expressionAandB001} which gives $\alpha = (M_0 - M_1M_2^{-1}M_1)^{-1}\sigma $, $\sigma \equiv \sign(\mu) = (\sign(\alpha_\tau))_{\tau\in S}$, and hence
\begin{align}
\alpha\alpha^* = \sigma\sigma^*+ \sum_{k\geq 1}M^k\sigma\sigma^* + \sum_{k\geq 1}\sigma\sigma^*M^k + \sum_{k',k\geq 1} M^k\sigma\sigma^*M^{k'}
\end{align}
where we use $M\equiv I-\left(M_0 - M_1M_2^{-1}M_1\right)$. Noting that $\sum_{\tau, \tau'\in S} |\delta_{\tau, \tau'} - \alpha_\tau\alpha_{\tau'}|\leq |S|\|\alpha\alpha^*\|_\infty$, and using 
\begin{align}
\|\alpha\alpha^*\|_\infty &\leq  \left\|\sigma\sigma^*\right\|_\infty + \sum_{k\geq 1} \left\|M\right\|^k_\infty \left\|\sigma\sigma^*\right\|_\infty + \sum_{k\geq 1} \left\|\sigma\sigma^*\right\|_\infty \left\|M\right\|_\infty^k + \sum_{k',k\geq 1} \left\| M\right\|_\infty^k \left\|\sigma\sigma^*\right\|_\infty \left\|M\right\|_\infty^{k'}\\
&\leq |S| + |S| \left(2\left\|M\right\|_\infty^k  + \left(\left\|M\right\|_\infty^k\right)^2 \right)\\
&  \stackrel{(a)}{\leq }S + S \left( \frac{\log |S|}{4\Delta n} +  6 \frac{\log|S|}{(\Delta n)^2} + \left( \frac{\log |S|}{4\Delta n} +  6\frac{\log|S|}{(\Delta n)^2}\right)^2\right)\\
&\lesssim |S|
\end{align}
(a) follows from~\eqref{NeumannWithoutNorm} and $\Delta \gtrsim n \log{|S|}$. Substituting this back in~\eqref{normNbound1} gives 
\begin{align}
\left\|\mathcal{T}\left\{\eta\eta^*\right\} - \mathcal{T}\left\{U\alpha\alpha^*U^*\right\}\right\|^2_N&\leq \frac{1}{n^2}\frac{\log n}{(n\Delta)^2} |S|^4
\end{align}
we thus get the desired result as soon as $\Delta\geq |S|^2n^{-1}\log(n)$. For the remaining three terms $\left\|\mathcal{T}\left\{n^{-2} U\alpha \beta^* V^*\right\}\right\|_N^2$, $\left\|n^{-2} V\beta\alpha^* U^*\right\|_N^2$ and $\left\|\mathcal{T}\left\{n^{-2} V\beta\beta^* V^*\right\}\right\|_N^2$, we simply use $\left\|\mathcal{T}|\right\|^2_{F\rightarrow N}\leq 1$ as well as the bounds~\eqref{boundA} and~\eqref{boundB} and $\|\alpha \|_\infty$ and $\left\|\beta\right\|_\infty$. For both the first and second term, this gives 
\begin{align}
\left\|\mathcal{T}\left\{n^{-2} U\alpha \beta^* V^*\right\}\right\|_N^2&\leq \left\|U\right\|_F^2 \left\|V\right\|_F^2 \left\|a\right\|_\infty^2\|b\|_\infty^2\stackrel{(a)}{\leq } \frac{1}{n^4} \left(|S|^2n^4\right) |S|^2 \frac{(\varepsilon_2')}{n^2}\stackrel{(b)}{\lesssim } n^{-2}\\
\end{align}
as soon as $\Delta \geq |S|^2\log^2|S|$. (a) above follows from $\|V\|_F^2 \lesssim  |S|\sum_{|k|\leq n} k^2 = \frac{n(n+1)(2n+1)}{6}\lesssim n^3$. $(b)$ follows from~\eqref{boundB} and noting that $\varepsilon_2' = O(\log|S|(\Delta n)^{-1})$. For $\left\|\mathcal{T}\left\{V\beta\beta^* V^*\right\}\right\|_N^2$, the same reasoning yields 
\begin{align}
\left\|n^{-2}\mathcal{T}\left\{Vbb^* V^*\right\}\right\|_N^2&\leq n^{-4}\|V\|_F^4 \left\|b\right\|_\infty^4 |S|^2 \lesssim n^{-2}
\end{align}
as soon as $\Delta \geq n |S|^2\log|S|$. This concludes the proof of lemma~\ref{boundOnNormPerrWithoutSampling}.

\appendix

\section{\label{integrationLemmas}Integration auxilliary lemmas}


\startcontents[short]



\printcontents[short]{}{2}{\addtocontents{ptc}{\setcounter{tocdepth}{2}}}

\subsection{\label{sectionProoflemma:boundK}Proof of lemma~\ref{lemma:boundK}}

Before proving the lemma, we recall the statement below. 

\boundK*

We will start by proving the bound~\eqref{constantBoundlemmaK2}. We then show how to derive~\eqref{decreasingBoundLemmaK3}. For the first bound, we don't need to take into account the relative position of the maximum of the polynomial. For~\eqref{decreasingBoundLemmaK3}, for a polynomial $p(e^{2\pi i \theta})$ upper bounded as in~\eqref{boundOnFLemmaIntegralTotal}, we have 
\begin{align}
&\sup_{\theta\in [0,1]} \frac{1}{|1-e^{2\pi i \theta}|}\left|\int_{0}^{1} p_{\tau}(e^{2\pi i s}) \int_{0}^{-\theta} e^{2\pi i (n+1) t} \tilde{D}(e^{-2\pi i k (s+t)})\; dt\right|\\
&\leq \sup_{\theta\in [0,1]} \frac{1}{|1-e^{2\pi i \theta}|}\int_{0}^{1} \left|p_{\tau}(e^{2\pi i s}) \right| \left|\int_{0}^{-\theta} e^{2\pi i (n+1) t} \tilde{D}(e^{-2\pi i k (s+t)})\; dt\right|\\
&\leq \sup_{\theta\in [0,1]} \frac{1}{|1-e^{2\pi i \theta}|}\int_{0}^{1} \min\left(1, \frac{C_1}{1+n|\theta - (\alpha - \tau)|_{\text{mod}}}\right)  \left|\int_{0}^{-\theta} e^{2\pi i (n+1) t} \tilde{D}(e^{-2\pi i k (s+t)})\; dt\right|\; ds\label{TotalOuterIntegral}
\end{align}

We now let $F(s; \theta)$  to denote the modulus of the inner integral, i.e. 
\begin{align}
F(s; \theta)& \equiv  \left|\int_{0}^{-\theta} e^{2\pi i (n+1) t} \tilde{D}(e^{-2\pi i k (s+t)})\; dt\right|
\end{align}

To derive the bounds~\eqref{constantBoundlemmaK2} and~\eqref{decreasingBoundLemmaK3}, we will rely on the following lemma which controls the quantity $F(s; \theta)$ on the domain $(s, \theta)$ defined below. 

To state this lemma, we split the domain as in Fig.~\ref{domainDecomposition1} below and introduce the subdomains $D_1^+, D_2^+, D_3^+$ and $D_0^+$  as
\begin{align}
D_0^+ &= \left\{(s,\theta)\in [0,\frac{1}{2}]\times [0,\frac{1}{2}] \;|\;|s|,|\theta|>\frac{1}{2n+3},\; s<\theta-\frac{1}{2n+3}\right\}\label{domainD0plus}\\
D_1^+ &= \left\{(s,\theta)\in [0,\frac{1}{2}]\times [0,\frac{1}{2}] ,\;|\;|s|\leq \frac{1}{2n+3},\;|\theta|\leq \frac{1}{2n+3}\right\}\\
&\cup \left\{(s,\theta),\;|\; |s|\leq \frac{1}{2n+3}, \frac{1}{2n+3}\leq|\theta|\leq s+\frac{1}{2n+3}\right\}\\
D_2^+& =\left\{(s,\theta)\in [0,\frac{1}{2}]\times [0,\frac{1}{2}] ,\; |\;|s|,|\theta|>\frac{1}{2n+3},\;|s-\theta|<\frac{1}{2n+3}\right\}\\
D_3^+& =\left\{(s,\theta)\in [0,\frac{1}{2}]\times [0,\frac{1}{2}] ,\; |\;|s|<\frac{1}{2n+3},\;|\theta|,|s-\theta|>\frac{1}{2n+3}\right\}\\
D_4^+& = \left\{(s,\theta)\in [0,\frac{1}{2}]\times [0,\frac{1}{2}] \;|\;|s|,|\theta|>\frac{1}{2n+3},\; |s-\theta|>\frac{1}{2n+3},\; -s<-\theta\right\}\label{domainD4plus}
\end{align}

as well as 
\begin{align}
D_0^{-} &= \left\{(s,\theta) \in [-\frac{1}{2},0]\times [0,\frac{1}{2}]\;|\;|s|>\frac{1}{2n+3},\;s-\theta\geq -1/2\right\}\label{definitionD0minus}\\
D_1^{-} &= \left\{(s,\theta)\in [-\frac{1}{2},0]\times [0,\frac{1}{2}]\;|\;|s|\leq \frac{1}{2n+3},\;s-\theta\geq -1/2, |\theta|>\frac{1}{2n+3}\right\}\\
D_2^{-} &= \left\{(s,\theta)\in [-\frac{1}{2},0]\times [0,\frac{1}{2}]\;|\;|s|, |\theta|\leq \frac{1}{2n+3}\right\}\\
D_3^{-} &= \left\{(s,\theta)\in [-\frac{1}{2},0]\times [0,\frac{1}{2}],\; -1+\frac{1}{2n+3}\leq s-\theta \leq -1/2, |s|\geq \frac{1}{2n+3}\right\}\label{definitionD3minus}\\
D_4^{-} &= \left\{(s,\theta)\in [-\frac{1}{2},0]\times [0,\frac{1}{2}],\; s-\theta \leq -1/2, |s|\leq \frac{1}{2n+3}\right\}\\
D_5^{-} &= \left\{(s,\theta)\in [-\frac{1}{2},0]\times [0,\frac{1}{2}],\; -1\leq (s-\theta)\leq -1+\frac{1}{2n+3}\right\}\label{definitionD5minus}
\end{align}

\begin{figure}\centering
\definecolor{forestGreen}{RGB}{34,139,34} 
\definecolor{mycolor2}{rgb}{0.85000,0.32500,0.09800}%
\definecolor{myOrange}{RGB}{255,165,0}
\definecolor{myYellow}{RGB}{255,255,0}
\definecolor{myViolet}{RGB}{199,21,133}

{\color{black}
\begin{tikzpicture}
\draw[->,ultra thick] (1,0)--(8,0) node[right]{$s$};
\draw[->,ultra thick] (4.5,-1)--(4.5,4) node[above]{$\theta$};
\draw[ultra thick] (2.5,0) rectangle (6.5,2);
\draw (4.5,0)--(6.5,2) ;
\draw (4.5,0.2)--(6.3,2) ;
\draw (4.7,0.2)--(4.7,2) ;

\draw[fill = blue, fill opacity=0.4, color = blue]  (4.5,2) -- (4.7,2) -- (4.7,0.2) -- (4.5,0.2) -- cycle;
\draw[fill = myOrange, fill opacity=0.4, color = myOrange]  (4.5,0.2) -- (6.3,2) -- (6.5,2) -- (4.7,0.2) -- cycle;
\draw[fill = myOrange, fill opacity=0.4, color = myOrange]  (4.5,0) -- (6.5,2) -- (6.5,1.8) -- (4.7,0) -- cycle;

\draw[fill = red, fill opacity=0.4, color = red]  (4.7,0.4) -- (4.7,2) -- (6.3,2) -- cycle;

\draw[fill = forestGreen, fill opacity=0.4, color = forestGreen]  (4.7,0) -- (6.5,0) -- (6.5,1.8) -- cycle;

\draw[fill=MidnightBlue,color = MidnightBlue]  (4.5,0.2) -- (4.7,0.2) -- (4.7,0) --(4.5,0) -- cycle;

\node at (2.5,-0.3) {$-0.5$};
\node at (6.5,-0.3) {$0.5$};
\node at (4.2,-0.3) {$0$};

\draw[color =MidnightBlue ]  (4.7,0.1) -- (5.5,-0.5);
\node at (5.7,-0.7) {\textcolor{MidnightBlue}{$D_1^+$}};

\node at (3.5,0.9) {$D^-$};

\node at (4.1,2.3) {$1/2$};

\node at (4.9,2.3) {\textcolor{blue}{$D_3^+$}};
\node at (5.6,2.3) {\textcolor{red}{$D_0^+$}};
\node at (6.5,2.3) {\textcolor{myOrange}{$D_2^+$}};

\node at (5.9,0.5) {\textcolor{forestGreen}{$D_4^+$}};


\end{tikzpicture}}
\hspace{1cm}
{\color{black}
\begin{tikzpicture}
\draw[->,ultra thick] (1,0)--(8,0) node[right]{$s$};
\draw[->,ultra thick] (5.5,-1)--(5.5,4) node[above]{$\theta$};
\draw (2.5,0) rectangle (5.5,3);
\draw (2.5,0)--(5.5,3) ;

\draw[ultra thick, fill = forestGreen, fill opacity=0.4, color = forestGreen]  (2.5,0) -- (2.5,3) -- (5.5,3) -- cycle;
\draw[fill=MidnightBlue]  (2.5,2.7) -- (2.5,3) -- (2.8,3) -- cycle;

\draw[fill=red]  (5.2,3) -- (5.5,3) -- (5.2,2.7) -- cycle;

\draw[fill = myOrange, fill opacity=0.4, color = myOrange, line width = 1.5pt]  (5.5,0.3) -- (5.2,0.3)  -- (5.2,0) -- (5.5,0)  -- cycle;

\node at (2.5,-0.3) {$-0.5$};

\node at (4.2,-0.3) {$0$};

\draw[fill=red, fill opacity=0.4, color = red] (2.5,0) -- (5.2,0) -- (5.2,2.7) -- cycle;
\node at (4.3,1) {\textcolor{red}{$D_0^-$}};

\draw[fill=blue, fill opacity=0.4, color = blue]  (5.2,3) -- (5.5,3) -- (5.5,0.3) -- (5.2,0.3) -- cycle;

\draw[color = forestGreen, line width =1pt]  (3.8,1.8) -- (3.2,3.3);
\node at (3.2,3.5) {\textcolor{forestGreen}{$D_3^-$}};

\node at (2.4,3.5) {\textcolor{MidnightBlue}{$D_5^-$}};
\node at (5.1,3.5) {\textcolor{myViolet}{$D_4^-$}};

\node at (6,1.7) {\textcolor{blue}{$D_1^-$}};

\node at (6,0.4) {\textcolor{myOrange}{$D_2^-$}};

\end{tikzpicture}}
\caption{\label{domainDecomposition1}$(s,\theta)$-Domain decomposition for the integral~\eqref{integralFthetaDef}. We consider $5$+$6$ subdomains depending on whether $s$ is non negative (subdomains $D_i^+$) or negative (subdomains $D_i^-$). The subdomains $D_2^+$, $D_0^+$, $D_3^+$ and $D_1^+$ correspond to the configurations shown in Fig.~\ref{frameworkDirichlet}. The expression of the line separating $D_3^-$ from $D_0^-$ is given by $\theta>\frac{1}{2} + s$ and the one of the line separating $D_4^+$ from $D_0^+$ is given by $\theta = s$.}
\end{figure}
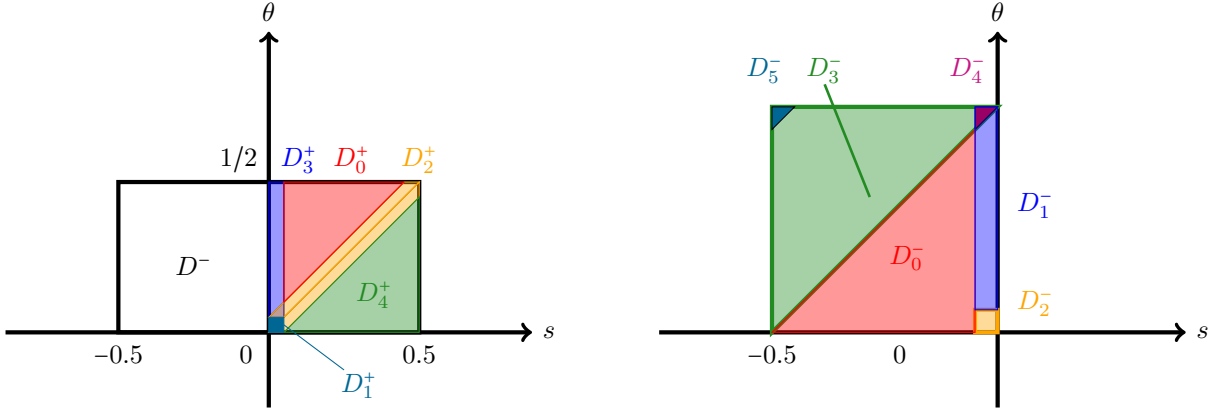

\begin{lemma}\label{boundInteriorIntegralFSthetaReal}
We define the constant $c$ as $c\equiv \frac{1}{4(2n+3)}(1/2 + 1/\pi)$. Consider the function $F(s; \theta)$ defined as 
\begin{align}
F(s; \theta)&=\int_0^{-\theta}e^{2\pi i(n+1)t}\tilde D(e^{-2\pi i(s+t)})dt.
\end{align}
and let $F_{R,D_\lambda}^{\sigma}$, $\sigma\in \left\{+,-\right\}$ denote the restriction of the real part of $F(s; \theta)$ to the subdomain $D_\lambda^\sigma$. Those restrictions obey the following upper bounds
\begin{align}
F_{R,D_0}^-&\leq  c \left(\frac{1}{-s} + \frac{1}{\theta-s}\right) + \frac{\theta}{2},&& \text{when $\theta \geq \frac{1}{2n+3}$}\\
&\leq \frac{3}{2} \frac{\theta}{\theta-s} + \frac{\theta}{8(2n+3)s(s-\theta)} + \frac{\theta}{2}, && \text{when $\theta\leq \frac{1}{2n+3}$}\\
F_{R,D_1}^-&\leq \frac{\pi}{4} + c \left(\frac{1}{-s + \frac{1}{2n+3}} + \frac{1}{\theta-s}\right) + \frac{\theta}{2}\\
F_{R,D_2}^-&\leq \frac{(2n+3)\pi \theta}{4} + \frac{\theta}{2}\\
F_{R,D_3}^-&\leq c \left(2+ \frac{1}{1+s-\theta}\right) + c\left(\frac{1}{-s} + 2\right) + \frac{\theta}{2}\\
F_{R,D_4}^-&\leq c\left(2+ \frac{1}{1+s-\theta}\right) + \frac{\pi}{4} + c \left(\frac{1}{\frac{1}{2n+3}-s} + 2\right)\\
F_{R,D_5}^-&\leq \frac{\pi}{4} + c\left(2+ \frac{1}{1+s-\theta + \frac{1}{2n+3}}\right) + c\left(2+ \frac{1}{-s}\right) + \frac{\theta}{2}\\
F_{R, D_0}^+& \leq \frac{\pi}{2} + c \left(2(2n+3) + \frac{1}{s} + \frac{1}{\theta-s}\right) + \frac{\theta}{2}\\
F_{R, D_1}^+& \leq \frac{\pi (2n+3)\theta}{4} + \frac{\theta}{2}\\
F_{R, D_2}^+& \leq \frac{\pi}{2} + c\left( \frac{1}{s} + \frac{1}{|s-\theta + \frac{2}{2n+3}|}\right) + \frac{\theta}{2}\\
F_{R, D_3}^+&\leq  \frac{(2n+3)\pi}{4} \left(s + \frac{1}{2n+3}\right) + c\left((2n+3) + \frac{1}{|s-\theta|}\right) + \frac{\theta}{2}\\
F_{R, D_4}^+&\leq \frac{3}{2} \frac{\theta}{\theta-s} + \frac{\theta}{8(2n+3)s(s-\theta)} + \frac{\theta}{2},&&\text{when $\theta \leq \frac{1}{2n+3}$}\\
&\leq c\left(\frac{1}{s} + \frac{1}{s-\theta}\right) + \frac{\theta}{2},&& \text{when $\theta \geq \frac{1}{2n+3}$}\\
\end{align}
\end{lemma} 
An equivalent series of bounds hold for the imaginary part. 
\begin{lemma}\label{boundInteriorIntegralFSthetaImaginary}
We define the constant $c$ as $c\equiv \frac{1}{4(2n+3)}(1/2 + 1/\pi)$. Consider the function $F(s; \theta)$ defined as 
\begin{align}
F(s; \theta)&=\int_0^{-\theta}e^{2\pi i(n+1)t}\tilde D(e^{-2\pi i(s+t)})dt.
\end{align}
and let $F_{I,D_\lambda}^{\sigma}$, $\sigma\in \left\{+,-\right\}$ denote the restriction of the imaginary part of $F(s; \theta)$ to the subdomain $D_\lambda^\sigma$. Those restrictions obey the following upper bounds
\begin{align}
F_{I,D_0}^-&\leq \frac{3}{2}\frac{\theta}{s-\theta} + \frac{\theta}{8(2n+3) s(s-\theta)} + \frac{1}{4} \log(\frac{-s}{\theta-s}) && \text{when $\theta \leq \frac{1}{2n+3}$} \\
&\leq c\left(\frac{1}{|s|} + \frac{1}{|s-\theta|}\right) + \frac{1}{4} \log(\frac{-s}{\theta-s})&& \text{when $\theta \geq \frac{1}{2n+3}$}\\
F_{I,D_1}^-&\leq c\left(\frac{1}{\theta-s} + \frac{1}{1/n - s}\right) + \frac{\pi(n+1)}{4n} + \log(\frac{\theta-s}{1/n - s}) \\
F_{I,D_2}^-&\leq \frac{\pi(n+1)\theta}{4}\\
F_{I,D_3}^-&\leq \frac{1}{4} \left(\log(\frac{1}{-s}) + \log(\frac{1}{1+s-\theta})\right) + 2\log(2) + c\left(4+ \frac{1}{-s} + \frac{1}{1+s-\theta}\right) \\
F_{I,D_4}^-&\leq \frac{(n+1)\pi}{4(n)} + c \left(\frac{1}{1+s-\theta} + \frac{1}{-s+1/n}+  4\right) \left|\log(\frac{1/2}{1+s-\theta}) + \log(\frac{1/2}{-s+1/n})\right|\\
F_{I,D_5}^-&\leq \frac{\pi}{2} + \frac{1}{4} \log(\frac{1/2}{-s} + \frac{1/2}{1+s-\theta+1/n}) + c\left(2+ \frac{1}{1+s-\theta + 1/n} + \frac{1}{-s}\right) \\
F_{I, D_0}^+& \leq c\left(\frac{1}{s} + \frac{1}{\theta-s}\right) + \log(\frac{s}{\theta-s}\vee \frac{\theta-s}{s})\\
F_{I, D_1}^+& \leq \frac{(n+1)\pi \theta}{2}\\
F_{I, D_2}^+&\leq  \frac{(n+1)\pi}{2n} + c\left(\frac{1}{|s|} + \frac{1}{|s-\theta + 2/n|}\right) + \frac{1}{4} \left|\log(\frac{s}{s-\theta + \frac{2}{n}})\right| \\
F_{I, D_3}^+&\leq c\left(\frac{1}{\theta-s} + \frac{1}{2/n - s}\right) + \frac{1}{4} \log(\frac{\theta-s}{2/n - s}) + \frac{\pi(n+1)}{n}\\
F_{I, D_4}^+&\leq \frac{4\pi^2\theta}{8(s-\theta)}  + \frac{\theta}{8(2n+3) s (s-\theta)} + \frac{1}{4} \log(\frac{s}{s-\theta}), && \text{when $\theta \leq \frac{1}{2n+3}$}\\
&\leq c\left(\frac{1}{s} + \frac{1}{s-\theta}\right) + \frac{1}{4} \log(\frac{s}{s-\theta}), && \text{when $\theta \geq \frac{1}{2n+3}$}
\end{align}
\end{lemma}

\subsubsection{Small $|\tau - \alpha|$ bound}

Let $F(s; \theta)$ be defined as $F(s; \theta) \equiv \int_{0}^{-(\theta - \theta_\ell)}  e^{2\pi i (n+1) t} \tilde{D}(e^{-2\pi i (s+t)})\; dt$. We now use $\Phi_\lambda^\sigma(\theta)$, to denote the associated integral
\begin{align}
\Phi_\lambda^\sigma \equiv \frac{1}{|1-e^{2\pi i \theta}|}\int_{D_\lambda^\sigma(\theta)} p(e^{2\pi i s}) \; F(s; \theta)\;  ds
\end{align}
for any subdomain $D_{\lambda}^{\sigma}$ where $\sigma = \left\{+,-\right\}$ following from the decomposition~\eqref{domainD0plus} to~\eqref{domainD4plus} as well as~\eqref{definitionD0minus} to~\eqref{definitionD5minus} illustrated in Fig.~\ref{domainDecomposition1}.

\begin{lemma}\label{lemmaTroncD3minusCons}
$ \left| \Phi_3^-\right| \leq 8.4 +14\log(C_1).$
\end{lemma}

\begin{lemma}\label{lemmaTroncD0minusCons}
$\left| \Phi_0^-\right|\leq 18 + 8\log(C_1).$
\end{lemma}

\begin{lemma}\label{lemmaTroncD0plusCons}
$ \left|\Phi_0^+\right|\leq 9.5 + 5\log(C_1).$
\end{lemma}

\begin{lemma}\label{lemmaTroncD4plusCons}
$ \left| \Phi_4^+\right| \leq 17 + 11\log(C_1).$
\end{lemma}

\begin{lemma}\label{lemmaTroncSmallDomainsNegativeCons}
$\sum_{\lambda\in\left\{1,2,4,5\right\}} \left| \Phi_\lambda^+\right| \leq 4.8.$
\end{lemma}

\begin{lemma}\label{lemmaTroncSmallDomainsPositiveCons}
$\sum_{\lambda=1}^3 \left| \Phi_\lambda^+\right| \leq 13.$
\end{lemma}

The total bound on the supremum~\eqref{firstSupremum} when $|\tau- \alpha|<1/n$ is thus given by
\begin{align}
\sup_{\theta\in [0,1]}\frac{1}{|1-e^{2\pi i \theta}|}\left|\int_{0}^{1} p_{\tau}(e^{2\pi i s}) \int_{0}^{-\theta} e^{2\pi i (n+1) t} \tilde{D}(e^{-2\pi i k (s+t)})\; dt\right|  &\leq \sum_{\lambda = 0}^5 \left|\Phi_\lambda^-\right| + \sum_{\lambda = 0}^4 \left|\Phi_\lambda^+\right| 
\end{align}
Multiplying this bound by $2$ and adding $1$ to get the total bound on the sum $~\eqref{firstSupremum}+\eqref{secondSupremumInDefintionKp}+\eqref{thirdTermDefinitionKp}$, gives the first part of lemma~\ref{lemma:boundK}
\begin{align}
\sup_{\theta \in [0,1]} |K_f(\tau, \theta)| \leq 2\left(71+38\log(C_1)\right)+1 \leq 143 + 76\log(C_1) \label{boundNearKp}
\end{align}

\subsubsection{Large $|\tau - \alpha|$ bound}

The bound for large values of $|\tau - \alpha|$ is controled through the following six lemmas which are respectively proved in sections~\ref{sectionProofLemmaD3minusLarge}, ~\ref{proofD0minusLarge}, ~\ref{sectionproofD0plusLarge}, ~\ref{sectionProofD4plusLarge}, ~\ref{proofSmallerDomainsLarge} as well as~\ref{proofD2plusLarge}

\begin{lemma}\label{lemmaTroncD3minus}
On $D_3^-$, we have 
\begin{align}
 \Phi(|\tau|)\leq 44\frac{C_1}{n|\tau|}  + 21 \frac{C_1}{n|\tau|} \log(\frac{n|\tau|}{C_1}) + 10\frac{C_1}{n|\tau|} \log(n|\tau|) + 8 \frac{C_1}{n|\tau|} \log^2(n|\tau|). \label{boundD3minusLargeTauMinusAlphaLemmaFinal}
\end{align}
\end{lemma}

\begin{lemma}\label{lemmaTroncD0minus}
On $D_0^-$, we have 
\begin{align}
\Phi(|\tau|)\leq  \frac{19C_1}{n|\tau|} + 14\frac{C_1}{n|\tau|}\log(n|\tau|).\label{boundD0minusLargeTauMinusAlphaLemmaFinal}
\end{align}
\end{lemma}

\begin{lemma}\label{lemmaTroncD0plus}
On $D_0^+$, we have 
\begin{align}
\Phi(|\tau|)\leq 25 \frac{C_1}{n|\tau|} + 27\frac{C_1}{n|\tau|}\log(n|\tau|) + \frac{C_1}{n|\tau|}\log^2(n|\tau|)\label{boundD0plusLargeTauMinusAlphaLemmaFinal}
\end{align}
\end{lemma}

\begin{lemma}\label{lemmaTroncD4plus}
On $D_4^+$, we have 
\begin{align}
\Phi(|\tau|) \leq 38 \frac{C_1}{n|\tau|} + 31\frac{C_1}{n|\tau|}\log(n|\tau|) +  6\frac{C_1}{n|\tau|}\log^2(n|\tau|)\label{boundD4plusLargeTauMinusAlphaLemmaFinal}
\end{align}
\end{lemma}

\begin{lemma}\label{lemmaTroncD2plus}
On $D_2^+$, we have 
\begin{align}
\Phi(|\tau|)\leq 32\frac{C_1}{n|\tau|} + 12\log(n|\tau|)\frac{C_1}{n|\tau|}. 
\label{boundD2plusLargeTauMinusAlphaLemmaFinal}
\end{align}
\end{lemma}

\begin{lemma}\label{lemmaTroncSmallerDomains}
Let $D^-\equiv D_5^-\cup D_4^- \cup D_1^- \cup D_2^-$ and $D^+ \equiv D_3^+\cup D_1^+$, then we can write
\begin{align}
\Phi(|\tau|) \leq 37\frac{C_1}{n|\tau|} + \frac{11}{n|\tau|} + \frac{1}{n|\tau|}\log(n|\tau|). \label{boundDminusAndPlusSmallDomainsplusLargeTauMinusAlphaLemmaFinal}
\end{align}
\end{lemma}

Proceeding as in the case $|\tau - \alpha|>\Delta$, we derive the total bound on $|K_p(\tau, \theta)|$ by multiplying the sum $\eqref{boundD3minusLargeTauMinusAlphaLemmaFinal} + \eqref{boundD0minusLargeTauMinusAlphaLemmaFinal} + \eqref{boundD0plusLargeTauMinusAlphaLemmaFinal} + \eqref{boundD4plusLargeTauMinusAlphaLemmaFinal} + \eqref{boundD2plusLargeTauMinusAlphaLemmaFinal} + \eqref{boundDminusAndPlusSmallDomainsplusLargeTauMinusAlphaLemmaFinal}$ by $2$ and adding $1$ to account for~\eqref{thirdTermDefinitionKp}.  First combining the results of lemmas~\eqref{lemmaTroncD3minus} to~\ref{lemmaTroncD2plus} we get
\begin{align}
 &\sup_{\theta\in [0,1]} \frac{1}{|1-e^{2\pi i \theta}|}\left|\int_{0}^{1} p_{\tau}(e^{2\pi i s}) \int_{0}^{-\theta} e^{2\pi i (n+1) t} \tilde{D}(e^{-2\pi i k (s+t)})\; dt\right|\\
&\leq \frac{195C_1+11}{n|\tau|} +\frac{21C_1}{n|\tau|}\log(\frac{n|\tau|}{C_1}) +  \frac{94}{n|\tau|} + \frac{25C_1}{n|\tau|} \log^2(n|\tau|).
\end{align}
As explained in section~\ref{boundAAMultiatomic}, we can then multiply this contribution by two, replace the modulus $|\tau|$ by $|\tau - \alpha|$, and add the contribution from~\eqref{additionalContributionp0Dirichlet}, to get the total bound on $|K_p(\tau, \theta)|$ for a polynomial bounded as in~\eqref{boundPolynomnialMultiSpikes}
\begin{align}
\left|K_p(\tau, \theta)\right| &\leq  2\sum_{\alpha\in S}\frac{390C_1 + 22}{n|\tau-\alpha|} + \frac{42C_1}{n|\tau-\alpha|} \log(\frac{n|\tau-\alpha|}{C_1}) + \frac{188}{n|\tau-\alpha|} + \frac{50 C_1}{n|\tau-\alpha|} \log^2(n|\tau-\alpha|) + \min(1, \frac{1}{n|\theta - \tau|}). \label{boundFarKp}
\end{align}
Combining this bound with~\eqref{boundNearKp}, for any $\tau$ and any polynomial $p$ satisfying
\begin{align}
p(e^{2\pi i \theta})&\leq \min\left(1, \sum_{\alpha\in S} \frac{C_1}{1+|\theta - \alpha|}\right),
\end{align}
\begin{align}
\left|K_p(\tau, \theta)\right|  &\leq 2\sum_{|\alpha - \tau|>\Delta}\frac{390C_1 + 22}{n|\tau-\alpha|} + \frac{42C_1}{n|\tau-\alpha|} \log(\frac{n|\tau-\alpha|}{C_1}) + \frac{188}{n|\tau-\alpha|} + \frac{50 C_1}{n|\tau-\alpha|} \log^2(n|\tau-\alpha|)\\
& + \min(1, \frac{1}{n|\theta - \tau|})+ \left(143 + 76\log(C_1)\right).
\end{align}

This concludes on the second part of lemma~\ref{lemma:boundK}.

\subsection{Proof of lemma~\ref{boundInteriorIntegralFSthetaReal}}
We will again use the following bounds on $\sin(\pi(t+s))$.
\begin{align}
2|s+t|\leq |\sin(\pi(s+t))|\leq \pi|s+t|\quad\mbox{whenever $-1/2\leq (s+t)\leq 1/2$}\label{boundSineSand}
\end{align}

We first write 
\begin{align}
|F_{\theta}(s)| &= \left|\int_{0}^{-\theta} e^{2\pi i (n+1)t} \sum_{k=0}^{n} e^{-i2\pi (s+t)k}dt\right|\\
& = \left|\int_{0}^{-\theta} e^{2\pi i (n+1)t} \frac{1 - e^{-i2\pi (n+1)(s+t)}}{1 -e^{-i2\pi (s+t)} }dt\right|\\
& = \left|\int_{0}^{-\theta} e^{2\pi i (n+1)t} \frac{e^{-i\pi(n+1)(s+t)}}{e^{-i\pi(s+t)}} \frac{e^{i\pi (n+1)(s+t)} - e^{-i\pi (n+1)(s+t)}}{e^{i\pi (s+t)} -e^{-i\pi (s+t)} }dt \right| \\
&= \left|\int_{0}^{-\theta} e^{2\pi i (n+1)t} e^{-i\pi n(s+t)} \frac{e^{i\pi (n+1)(s+t)} - e^{-i\pi (n+1)(s+t)}}{e^{i\pi (s+t)} -e^{-i\pi (s+t)} }dt \right| \\
&= \left|\int_{0}^{-\theta} e^{2\pi i (n+1)(t+s)} e^{-i\pi n(s+ t)} e^{-2\pi i (n+1)s} \frac{e^{i\pi (n+1)(s+t)} - e^{-i\pi (n+1)(s+t)}}{e^{i\pi (s+t)} -e^{-i\pi (s+t)} }dt \right| \\
& = \left|\int_{0}^{-\theta} e^{\pi i (n+2)(t+s)} e^{-\pi i (2n+2)s} \frac{e^{i\pi (n+1)(s+t)} - e^{-i\pi (n+1)(s+t)}}{e^{i\pi (s+t)} -e^{-i\pi (s+t)} }dt \right| \\
& = \left|\int_{0}^{-\theta} e^{\pi i (n+2)(t+s)} \frac{e^{i\pi (n+1)(s+t)} - e^{-i\pi (n+1)(s+t)}}{e^{i\pi (s+t)} -e^{-i\pi (s+t)} }dt \right| \label{moduleIntegralTotal}\\
& = \left|\int_{-\theta}^{0} H(t; s,\theta)\; dt\right|  
\end{align}
Where we define $H(t; s,\theta)$ as
\begin{align}
H(t; s,\theta)\equiv \frac{e^{i\pi (n+1)(s+t)} - e^{-i\pi (n+1)(s+t)}}{e^{i\pi (s+t)} -e^{-i\pi (s+t)} }e^{\pi i (n+2)(t+s)} 
\end{align}

Let us use $F_R$ and $F_I$ to denote the real and imaginary part of $F_\theta(s)$. We will successively bound those two parts from which we have 
\begin{align}
|F_{\theta}(s)| \leq \left|\text{Re}\int_{0}^{-\theta} H(t; s,\theta)\right| +  \left|\text{Im}\int_{0}^{-\theta} H(t; s,\theta)\right| \label{integralRealImaginaryDecomposition}
\end{align}
For this real part, note that we have 
\begin{align}
F_R = &\Rp \int_{0}^{-\theta} e^{\pi i (n+2)(t+s)} \frac{e^{i\pi (n+1)(s+t)} - e^{-i\pi (n+1)(s+t)}}{e^{i\pi (s+t)} -e^{-i\pi (s+t)} }dt\\
& = \int_{0}^{-\theta} \frac{\cos(\pi (n+2)(t+s))\sin(\pi (n+1)(s+t))}{\sin(\pi(s+t))}\;dt\\
& = \int_{0}^{-\theta} \frac{\sin((2n+3)\pi (t+s)) - \sin(\pi (t+s))}{2\sin(\pi(s+t))}\; dt\\
|F_R|& \leq \left|\int_{0}^{-\theta} \frac{\sin((2n+3)\pi (t+s)) }{2\sin(\pi(s+t))}\;dt\right| + \frac{|\theta|}{2}\label{bound1}
\end{align}

Recall that we use $F^+_{R,D_0},F^+_{R,D_1},F^+_{R,D_2},F^+_{R,D_3}$ and $F^+_{R,D_4}$ to denote the corresponding contributions to the real part of the integral~\eqref{integralFthetaDef}. 

We start with $D_0^+$. On this first subdomain, we have $s\leq \theta -\frac{1}{2n+3}$ so that the Dirichlet peak is located inside the interval $[-\theta,0]$. We are in the second configuration of Fig.~\ref{frameworkDirichlet}. We therefore need to treat separately the small interval around the peak.
\begin{align}
F_{R,D_0}^+ &\leq \left|\int_{0}^{-\theta} \frac{\sin((2n+3)\pi (t+s)) }{2\sin(\pi(s+t))}\;dt\right| \\
&\leq  \left|\int_{-\theta}^{-s-1/(2n+3)} \frac{\sin((2n+3)\pi (t+s))}{2\sin(\pi (t+s))}\; dt\right|+ \left|\int_{-s-1/(2n+3)}^{-s+1/(2n+3)} \frac{\sin((2n+3)\pi (s+t))}{2\sin(\pi (s+t))}\; dt\right| \\
&+ \left|\int_{-s+1/(2n+3)}^{0} \frac{\sin((2n+3)\pi (s+t))}{2\sin(\pi (s+t))}\; dt\right| + \frac{|\theta|}{2}\\
&\leq \left|\int_{-\theta}^{-s-1/(2n+3)} \frac{\sin((2n+3)\pi (t+s))}{2\sin(\pi (t+s))}\; dt\right|+ \left|\int_{-s-1/(2n+3)}^{-s+1/(2n+3)} \frac{(2n+3)\pi |s+t|}{4|s+t|}\; dt\right| \\
&+ \left|\int_{-s+1/(2n+3)}^{0} \frac{\sin((2n+3)\pi (s+t))}{2\sin(\pi (s+t))}\; dt\right| + \frac{|\theta|}{2}\\
&\leq \left|\int_{-\theta}^{-s-1/(2n+3)} \frac{\sin((2n+3)\pi (t+s))}{2\sin(\pi (t+s))}\; dt\right|\\
&+ \frac{\pi}{2} + \left|\int_{-s+1/(2n+3)}^{0} \frac{\sin((2n+3)\pi (s+t))}{2\sin(\pi (s+t))}\; dt\right|+\frac{|\theta|}{2}\label{integralTmpD0}
\end{align}
For the $[-\theta,-s-\frac{1}{2n+3}]$, and $[-s+\frac{1}{2n+3}, 0]$ intervals, one can integrate by part
\begin{align}
&\left|\int_{-\theta}^{-s-\frac{1}{2n+3}} \frac{\sin(2n+3)\pi(s+t)}{2\sin(\pi (s+t))}\; dt \right|\\
&\leq \left|\left|\frac{\cos((2n+3)\pi(t+s))}{2(2n+3)\pi \sin(\pi (s+t))}\right|_{-\theta}^{-s-\frac{1}{2n+3}}\right. \left.+ \left|\int_{-\theta}^{-s-\frac{1}{2n+3}} \frac{\cos((2n+3)\pi (t+s))\cos(\pi(s+t))}{2(2n+3) \sin^2(\pi (t+s))} \right|\right|\\
& \leq \left|- \frac{\cos((2n+3)\pi (s-\theta))}{2(2n+3)\pi \sin(\pi (s-\theta))} + \frac{\cos(\pi)}{2(2n+3)\pi \sin(-\pi /(2n+3))}\right|\\
& + \left|\int_{-\theta}^{-s-\frac{1}{2n+3}} \frac{\cos((2n+3)\pi (t+s))\cos(\pi(s+t))}{2(2n+3) \sin^2(\pi (t+s))} \right|\label{continuousExplanation}
\end{align}
From~\eqref{continuousExplanation} we can bound the modulus as 
\begin{align}
&\left|\int_{-\theta}^{-s-\frac{1}{2n+3}} \frac{\sin(2n+3)\pi(s+t)}{2\sin(\pi (s+t))}\; dt \right|\\
&\leq \left|- \frac{\cos((2n+3)\pi (s-\theta))}{2(2n+3)\pi \sin(\pi (s-\theta))} +\frac{1}{4\pi}\right|  + \left|\int_{-\theta}^{-s-\frac{1}{2n+3}} \frac{|\cos((2n+3)\pi (s+t))\cos(\pi (s+t))|}{2(2n+3)|\sin^2(\pi(s+t))|}\; dt\right|\\
& \leq  \left|- \frac{\cos((2n+3)\pi (s-\theta))}{2(2n+3)\pi \sin(\pi (s-\theta))} +\frac{1}{4\pi }\right|+ \left|\int_{-\theta}^{-s-\frac{1}{2n+3}} \frac{1}{8 (s+t)^2 (2n+3)}\; dt\right|\label{tmpD0LeftInterval}\\
& \leq  \left|- \frac{\cos((2n+3)\pi (s-\theta))}{2(2n+3)\pi \sin(\pi (s-\theta))} +\frac{1}{4\pi}\right| + \frac{1}{8}+ \frac{1}{8(\theta-s)(2n+3)}\\
& \leq  \frac{1}{4(2n+3)\pi |\theta-s|} +\frac{1}{4\pi}+\frac{1}{8} + \frac{1}{8(\theta-s)(2n+3)}\\
&\leq \frac{1}{4(2n+3)}\left((2n+3) +\frac{1}{(\theta-s)}\right)\left(\frac{1}{2}+ \frac{1}{\pi}\right)\\
& = c \left((2n+3)+ \frac{1}{\theta-s}\right) = \frac{c'}{2n+3}\left((2n+3)+ \frac{1}{\theta-s}\right)\label{bound1D0}
\end{align}
The third line~\eqref{tmpD0LeftInterval} follows from the lower bound on the sine as well as the fact that $|t+s|\leq 1/2$ on $D_0^+$. In the lines above we also defined $c$ as 
\begin{align}
c &\equiv  \frac{1}{(2n+3)}c' = \frac{1}{4(2n+3)}\left(\frac{1}{\pi}+ \frac{1}{2}\right)\label{definitionc}\\
c' & \equiv \frac{1}{4}\left(\frac{1}{2}+ \frac{1}{\pi}\right) \label{definitioncprime}.
\end{align}
A similar reasoning holds on the interval $[-s+\frac{1}{2n+3}, 0]$ and gives 
\begin{align}
\left|\int_{-s+\frac{1}{2n+3}}^0 \frac{\sin((2n+3)\pi(s+t))}{2\sin(\pi(s+t))} \; dt\right| & = \left|\frac{-\cos(\pi)}{(2n+3)2\pi \sin(\pi/(2n+3))} + \frac{\cos((2n+3)\pi s)}{(2n+3)2\pi \sin(\pi s)}\right|\\
&+ \left|\int_{-s+\frac{1}{2n+3}}^0 \frac{\cos((2n+3)\pi (s+t))\cos(\pi (s+t))}{2\sin^2(\pi (s+t))(2n+3)}\; dt\right|\\
& \leq \frac{1}{4\pi} + \frac{1}{(2n+3)4\pi s} + \frac{1}{8s(2n+3)} + \frac{1}{8}\\
&\leq \frac{1}{4(2n+3)}\left(\frac{1}{2} + \frac{1}{\pi}\right)\left((2n+3)+ \frac{1}{s}\right)\label{bound2D0}
\end{align}
Substituting~\eqref{bound1D0},~\eqref{bound2D0} into~\eqref{integralTmpD0}, the integral on $D_0^+$ is thus bounded as 
\begin{align}
F_{R, D_0}^+ &\leq \frac{\pi}{2} + \frac{1}{4(2n+3)}\left((2n+3) +\frac{1}{(\theta-s)}\right)\left(\frac{1}{2}+ \frac{1}{\pi}\right) +\frac{1}{4(2n+3)}\left(\frac{1}{2} + \frac{1}{\pi}\right)\left((2n+3)+ \frac{1}{s}\right) + \frac{\theta}{2}\\
&\leq \frac{\pi}{2} + \frac{1}{4(2n+3)}\left(\frac{1}{2}+\frac{1}{\pi}\right)\left(2(2n+3)+ \frac{1}{s} + \frac{1}{\theta-s}\right) + \frac{\theta}{2}
\end{align}
On $D_1^+$, the integral reduces to integrating the Dirichlet peak on an interval of length $\mathcal{O}(|\theta|)$. We are in the last configuration in Fig.~\ref{frameworkDirichlet}.
\begin{align}
F_{R,D_1}^+ \leq \left|\int_{-\theta}^0 \frac{\sin((2n+3)\pi (s+t))}{2\sin(\pi(s+t))}\right| \leq \frac{\pi(2n+3)\theta}{4} + \frac{\theta}{2}
\end{align}
On $D_2^+$, we need to remove the small part of the $[-\theta,0]$ interval that contains the Dirichlet peak. We are in the first configuration of Fig.~\ref{frameworkDirichlet}. This gives 
\begin{align}
F_{R,D_2}^+ &\leq \left|\int_{-\theta}^{-\theta+\frac{2}{2n+3}} \frac{\sin((2n+3)\pi(s+t))}{2\sin(\pi(s+t))}\; dt\right|\\
& + \left|\int_{-\theta+\frac{2}{2n+3}}^0 \frac{\sin((2n+3)\pi (s+t))}{2\sin(\pi (s+t))}\; dt\right| + \frac{\theta}{2}\\
& \leq \left|\frac{(2n+3)\pi}{4}\right|_{-\theta}^{-\theta+\frac{2}{2n+3}} + \left|\frac{\cos((2n+3)\pi (s+t))}{(2n+3)\pi 2\sin(\pi (s+t))}\right|_{-\theta+\frac{2}{2n+3}}^0\\
& + \left|\int_{-\theta+\frac{2}{2n+3}}^0 \frac{\cos((2n+3)\pi (s+t))\cos(\pi (s+t))}{(2n+3)2\sin^2(\pi (s+t))}\; dt\right| + \frac{\theta}{2}\\
&\leq \frac{\pi}{2} + \left(\frac{1}{s} + \frac{1}{\left|s-\theta + \frac{2}{2n+3}\right|}\right)\frac{1}{(2n+3)4}\left(\frac{1}{2}+\frac{1}{\pi}\right) + \frac{\theta}{2}
\end{align}
On $D_3^+$, we are in third framework from Fig.~\ref{frameworkDirichlet}. We use a similar reasoning and treat separately the subinterval $[-s-\frac{1}{2n+3}, 0]$ on which the peak is located, thus getting 
\begin{align}
F_{R,D_3}^+ = & \left|\int_{-\theta}^{-s-\frac{1}{2n+3}} \frac{\sin((2n+3)\pi (s+t))}{2\sin(\pi (s+t))}\; dt + \int_{-s-\frac{1}{2n+3}}^0 \frac{\sin((2n+3)\pi (t+s))}{2\sin(\pi(s+t))}\; dt\right| + \frac{\theta}{2}\\
&\leq \frac{(2n+3)\pi}{4}\left(s+\frac{1}{2n+3}\right) + \left|\frac{\cos((2n+3)\pi (s+t))}{(2n+3)2\pi \sin(\pi (s+t)) }\right|_{-\theta}^{-s-\frac{1}{2n+3}}\\
&+ \left|\int_{-\theta}^{-s-\frac{1}{2n+3}} \frac{\cos((2n+3)\pi (s+t))\cos(\pi (s+t))}{(2n+3)2\sin^2(\pi (s+t))}\; dt \right| + \frac{\theta}{2}\\
& \leq \frac{(2n+3)\pi}{4}\left(s+\frac{1}{2n+3}\right) + \left(\frac{1}{2} + \frac{1}{\pi}\right) \frac{1}{4(2n+3)}\left((2n+3) + \frac{1}{|s-\theta|}\right) + \frac{\theta}{2}
\end{align}
This bound is well defined provided that $|s-\theta|>1/(2n+3)$ which always holds on $D_3^+$. Finally, on $D_4^+$, the peak is located outside the interval and we can simply integrate on the whole $[-\theta,0]$ interval. The bound however requires a little more calculations in order to ensure that it can compensate for the prefactor in~\eqref{eq:dev_int_p+}. I.e on the other subdomains, the bound does not have to exhibit an explicit dependence in $\theta$ as the domains themselves have a vanishing measure when $\theta\rightarrow 0$. $D_4^+$ always has, however, a non zero area, including when $\theta\rightarrow 0$ and we must thus derive an explicit bound in $\theta$ to avoid a blow-up when multiplying by the prefactor in~\eqref{eq:dev_int_p+}.
\begin{align}
&F_{R,D_4}^+ = \left|\int_{-\theta}^0 \frac{\sin((2n+3)\pi (t+s))}{2\sin(\pi (t+s))}\; dt\right|  + \frac{\theta}{2} \\
&\leq \left|\left|\frac{\cos((2n+3)\pi (t+s))}{2(2n+3)\pi\sin(\pi (t+s))}\right|_{-\theta}^0\right| +\left| \int_{-\theta}^0 \frac{\cos((2n+3)\pi (t+s))\cos(\pi(s+t))}{2(2n+3)\sin^2(\pi (t+s))}\; dt\right|+ \frac{\theta}{2}\label{D4term2a}\\
\end{align}

\begin{align}
&F_{R,D_4}^+ = \left|\int_{-\theta}^0 \frac{\sin((2n+3)\pi (t+s))}{2\sin(\pi (t+s))}\; dt\right| + \frac{\theta}{2}  \\
&\leq \left| \left|\frac{\cos((2n+3)\pi (t+s))}{2(2n+3)\pi\sin(\pi (t+s))}\right|_{-\theta}^0\right|+\frac{\theta}{2}
+\left| \int_{-\theta}^0 \frac{\cos((2n+3)\pi (t+s))\cos(\pi(s+t))}{2(2n+3)\sin^2(\pi (t+s))}\; dt\right|\\
& \leq\left| \frac{\cos((2n+3)\pi s) \sin(\pi (s-\theta))  - \sin(\pi s)\cos((2n+3)\pi (s-\theta))}{2(2n+3)\sin(\pi (s-\theta))\sin(\pi s)}\right| + \frac{\theta}{2}\label{D4term1}\\
& +\left| \int_{-\theta}^0 \frac{\cos((2n+3)\pi (t+s))\cos(\pi(s+t))}{2(2n+3)\sin^2(\pi (t+s))}\; dt\right|\label{D4term2}
\end{align}
Developing the numerator in~\eqref{D4term1}, we get 
\begin{align}
&\cos((2n+3)\pi s)\sin(\pi (s-\theta)) - \sin(\pi s)\cos((2n+3)\pi (s-\theta))\label{step1DevelopmentSmalltheta} \\
=&  \cos((2n+3)\pi s )\left[\sin(\pi s)\cos(\pi \theta) - \sin(\pi \theta)\cos(\pi s)\right]\\
& - \sin(\pi s)\left[\cos((2n+3)\pi s)\cos(\pi (2n+3)\theta) + \sin((2n+3)\pi s)\sin((2n+3)\pi \theta)\right]\\
=&  -\cos((2n+3)\pi s ) \sin(\pi \theta)\cos(\pi s)\\
& - \sin(\pi s)\sin((2n+3)\pi s)\sin((2n+3)\pi \theta)\\
&-\sin(\pi s)\cos((2n+3)\pi s) \left[\cos((2n+2)\pi \theta)\cos(\pi \theta) - \sin((2n+2)\pi \theta)\sin(\pi \theta)\right]\\
& +\cos((2n+3)\pi s)\sin(\pi s)\cos(\pi \theta) \\
& =  -\cos((2n+3)\pi s ) \sin(\pi \theta)\cos(\pi s)\label{finalTrigonometricCooking1}\\
& - \sin(\pi s)\sin((2n+3)\pi s)\sin((2n+3)\pi \theta)\\
& + \sin(\pi s)\cos((2n+3)\pi s)\sin((2n+2)\pi \theta)\sin(\pi \theta)\\
& +\cos((2n+3)\pi s) \sin(\pi s)\cos(\pi \theta)(1 - \cos((2n+2)\pi \theta))\label{finalTrigonometricCookingend}
\end{align}
Now taking the modulus and using $|\sin(x)|\leq |x|$ as well as $1-\cos(x) = 2\sin^2(x/2)$, and $\cos((2n+3)\pi s) \leq (2n+3)\pi |s|$ from $1/(2n+3)\leq s\leq 1/2$, the numerator in~\eqref{D4term1} can be bounded as 
\begin{align}
&\left|\cos((2n+3)\pi s)\sin(\pi (s-\theta)) - \sin(\pi s)\cos((2n+3)\pi (s-\theta))\right| \\
& \leq (2n+3)\pi^2 s \theta  + (2n+3)\pi^2 s \theta  + \pi^2 s \theta + \pi^2 s\theta (n+1)\\
&\leq \pi^2 s \theta (6n+9)\label{stepEndDevelopmentSmalltheta} 
\end{align}
Substituting this bound back into~\eqref{D4term1}, noting that both $|s|$ as well as $|\theta-s|$ are less than $1/2$ we get 
\begin{align}
\left|\left|\frac{\cos((2n+3)\pi (t+s))}{2(2n+3)\pi\sin(\pi (t+s))}\right|_{-\theta}^0 \right|\leq \frac{\pi^2 s \theta (6n+9)}{2(2n+3)\pi^2 s (s-\theta)} \leq \frac{3}{2}{\color{green}\pi/2}\frac{\theta}{s-\theta}\label{refinedBoundSmallthetaFirstTerm}
\end{align}
For the second term in~\eqref{D4term2a}, we can more simply write 
\begin{align}
\left|\int_{-\theta}^0 \frac{\cos((2n+3)\pi (t+s))\cos(\pi(s+t))}{2(2n+3)\sin^2(\pi (t+s))}\; dt\right|&\leq \left|\int_{-\theta}^0 \frac{1}{8(2n+3)(t+s)^2}\; dt\right|\\
&\leq\left|\frac{1}{8(2n+3)(s)} - \frac{1}{8(2n+3)(s-\theta)}\right|\\
&\leq \frac{\theta}{8(2n+3)s|s-\theta|}
\end{align}

%
%
%

The total bound on $D_4^+$ is thus given by 


\begin{align}
F_{R,D_4}^+ \leq \frac{3}{2}\frac{\theta}{\theta-s} + \frac{\theta}{8(2n+3)s|s-\theta|} + \frac{\theta}{2} \label{D4term2}
\end{align}

For large values of $\theta$, we will use the bound

\begin{align}
&F_{R,D_4}^+ = \left|\int_{-\theta}^0 \frac{\sin((2n+3)\pi (t+s))}{2\sin(\pi (t+s))}\; dt\right| + \frac{\theta}{2}  \\
&\leq \left|\left|\frac{\cos((2n+3)\pi (t+s))}{2(2n+3)\pi\sin(\pi (t+s))}\right|_{-\theta}^0\right| +\left| \int_{-\theta}^0 \frac{\cos((2n+3)\pi (t+s))\cos(\pi(s+t))}{2(2n+3)\sin^2(\pi (t+s))}\; dt\right|+ \frac{\theta}{2}\\
&\leq \frac{\theta}{2} + \left(\frac{1}{2} + \frac{1}{\pi}\right)\frac{1}{4(2n+3)} \left(\frac{1}{s} + \frac{1}{s-\theta}\right)\label{largeThetaBoundD4+}
\end{align}



\begin{figure}\centering
\hspace{0.1cm}\input{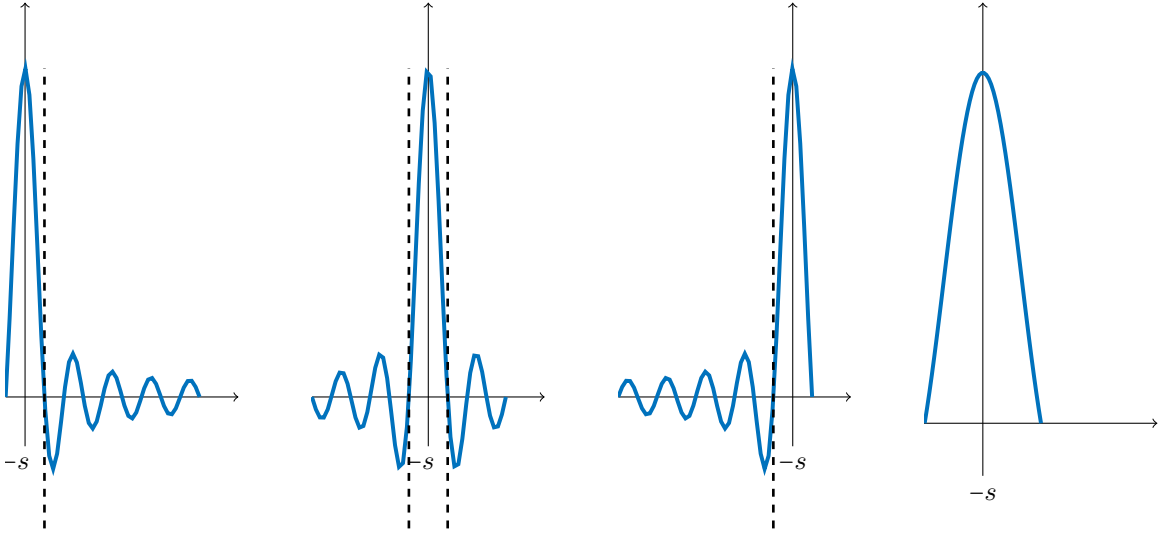}
\caption{\label{frameworkDirichlet}The four frameworks that we consider to control the integral~\eqref{moduleIntegralTotal}. From Left to Right, $|-\theta+s| < 1/(2n+3)$ with $|\theta|\geq 1/(2n+3)$, $-\theta + 1/(2n+3)\leq -s\leq -1/(2n+3)$, $-1/(2n+3)\leq -s$ and $\theta \leq 1/(2n+3)$. Those four frameworks respectively correspond to the subdomains $D_2^+$, $D_0^+$, $D_3^+$ and $D_1^+$ in Fig.~\ref{domainDecomposition1}}\label{totalBoundconfigurationTwo}
\end{figure}



In $D_0^-$, the peak is located sufficiently far from the interval $[-\theta,0]$ so that we can simply integrate over this whole interval. To get a bound proportional to $\theta$, we use the symmetry of the kernel, together with a reasoning similar to the one we used for $D_4^+$, 
\begin{align}
F_{R,D_0}^- &= \left| \int_{-\theta}^0 \frac{\sin((2n+3)\pi (s+t))}{2\sin(\pi (s+t))}\; dt\right| + \frac{\theta}{2}\\
&\leq \left|\left|\frac{\cos((2n+3)\pi (s+t))}{2(2n+3)\pi \sin(\pi (s+t))}\right|_{-\theta}^0 - \int_{-\theta}^0 \frac{\cos((2n+3)\pi (s+t))\cos(\pi (s+t))}{2(2n+3)\sin^2(\pi (s+t))} \; dt\right| + \frac{\theta}{2}\\
&\leq \left|\left|\frac{1}{2\pi (2n+3)2|s+t|}\right|_{-\theta}^0 + \left|\frac{1}{8(2n+3)(s+t)}\right|_{-\theta}^0\right| + \frac{\theta}{2}\\
&\leq \frac{1}{4(2n+3)} \left(\frac{1}{\pi}+\frac{1}{2}\right) \left[\frac{1}{|s|} + \frac{1}{|s-\theta|}\right] + \frac{\theta}{2}\\
&\leq \frac{1}{4(2n+3)} \left(\frac{1}{\pi} + \frac{1}{2}\right)\left[\frac{1}{-s}+\frac{1}{\theta-s}\right]+ \frac{\theta}{2}\label{boundFD0minusReal}
\end{align}
Again when $\theta$ is smaller than $\frac{1}{2n+3}$ we will need a tighted bound. In this case, we therefore apply the same reasoning as in~\eqref{step1DevelopmentSmalltheta} to~\eqref{stepEndDevelopmentSmalltheta} which gives  

\begin{align}
F_{R,D_0}^- &\leq \frac{3}{2} \frac{\theta}{\theta-s} + \left(\frac{1}{8(2n+3)s} - \frac{1}{8(2n+3)(s-\theta)}\right) + \frac{\theta}{2}\\
& = \frac{3}{2} \frac{\theta}{\theta-s}  + \frac{\theta}{8(2n+3)s(s-\theta)} + \frac{\theta}{2}\label{boundNegativeDomainSmallTheta}
\end{align}

On $D_1^-$, we remove the sub-interval located near $0$, which we integrate separately 
\begin{align}
F_{R,D_1}^- &\leq \left|\int_{-\theta}^{-\frac{1}{2n+3}} \frac{\sin((2n+3)\pi (s+t))}{2\sin(\pi (s+t))}\; dt\right| + \left|\int_{-\frac{1}{2n+3}}^0 \frac{\sin((2n+3)\pi (s+t))}{2\sin(\pi (s+t))}\; dt\right| + \frac{\theta}{2}\\
&\leq \frac{\pi}{4} + \left|\frac{\cos((2n+3)\pi (s+t))}{2\pi (2n+3)\sin(\pi (s+t))}\right|_{-\theta}^{-\frac{1}{2n+3}}  + \int_{-\theta}^{-\frac{1}{2n+3}} \frac{\cos((2n+3)\pi (s+t))\cos(\pi (s+t))}{(2n+3)\sin^2(\pi (s+t))}\; dt + \frac{\theta}{2}\\
&\leq \frac{\pi}{4} + \frac{1}{4(2n+3)}\left(\frac{1}{|s-\frac{1}{2n+3}|} + \frac{1}{|\theta-s|}\right)\left(\frac{1}{\pi} + \frac{1}{2}\right) + \frac{\theta}{2}\\
&\leq \frac{\pi}{4} + \frac{1}{4(2n+3)}\left(\frac{1}{-s+\frac{1}{2n+3}} + \frac{1}{(\theta-s)}\right)\left(\frac{1}{\pi} + \frac{1}{2}\right) + \frac{\theta}{2}
\end{align}
We use a specific bound on $D_2^-$ to be able to compensate for the denominator in~\eqref{eq:dev_int_p+},
\begin{align}
F_{R,D_2}^-  = \left|\int_{-\theta}^0 \frac{\sin((2n+3)\pi(s+t))}{2\sin(\pi(s+t))}\; dt\right| + \frac{\theta}{2}\leq \frac{(2n+3)\pi \theta}{4} + \frac{\theta}{2}
\end{align}
On $D_3^-$, we split the integral between a first subdomain on which $t+s <-1/2$ and a second subdomain on which $t+s>-1/2$. On the first subdomain, in order to apply the bounds~\eqref{boundSineSand}, we must first replace the sines using $|\sin(x)| = |\sin(\pi +x)|$, to make sure that the angle lies in the range $[-\frac{\pi}{2}, \frac{\pi}{2}]$ and then apply the bound. We write 
\begin{align}
F_{R,D_3}^- & = \left|\int_{-\theta}^{-\frac{1}{2} - s} \frac{\sin((2n+3)\pi(s+t))}{2\sin(\pi (s+t))}\; dt\right| + \left| \int_{-\frac{1}{2}-s}^0 \frac{\sin((2n+3)\pi(s+t))}{2\sin(\pi (s+t))}\; dt\right| + \frac{\theta}{2}\\
&\leq  \left|\frac{\cos((2n+3)\pi(s+t))}{2\pi (2n+3)\sin(\pi (s+t))}\right|_{-\theta}^{-1/2-s} + \left|\int_{-\theta}^{-1/2-s} \frac{\cos((2n+3)\pi (s+t))\cos(\pi(s+t))}{(2n+3)2\sin^2(\pi (s+t))}\; dt\right|\label{partBeforeMinus1/2} \\
&+ \left|\frac{\cos((2n+3)\pi(s+t))}{2\pi (2n+3)\sin(\pi (s+t))}\right|_{-1/2-s}^{0} + \left|\int_{-1/2-s}^{0} \frac{\cos((2n+3)\pi (s+t))\cos(\pi (s+t))}{(2n+3)2\sin^2(\pi (s+t))}\; dt\right|+\frac{\theta}{2}\label{partAfterMinus1/2}
\end{align}
For~\eqref{partAfterMinus1/2}, we thus use $|\sin(\pi t)|\geq 2t$, for~\eqref{partBeforeMinus1/2} we use $|\sin(\pi t)| = |\sin(\pi + \pi t)|\geq 2|1+(t+s)|$. This gives 
\begin{align}
F_{R,D_3}^- & \leq \left|\frac{1}{\pi 4(2n+3)\pi |1+(s+t)|}\right|_{-\theta}^{-\frac{1}{2}-s} + \left|\frac{1}{(2n+3)8|1+(s+t)|}\right|_{-\theta}^{-\frac{1}{2}-s}\\
& +\left|\frac{1}{4\pi (2n+3)|s+t|}\right|_{-\frac{1}{2}-s}^0 + \left|\frac{1}{8(2n+3)|s+t|}\right|_{-\frac{1}{2}-s}^0 + \frac{\theta}{2} \\
&\leq \frac{1}{4(2n+3)}\left(\frac{1}{\pi}+\frac{1}{2}\right)\left(2 + \frac{1}{1+s-\theta}\right)\label{FD3minusa}\\
&+ \frac{1}{4(2n+3)} \left(\frac{1}{\pi}+\frac{1}{2}\right) \left(\frac{1}{-s} + 2\right) + \frac{\theta}{2}\label{FD3minusb}
\end{align}

Note that on $D_3^-$, we don't need to seprately deal with the "small $\theta$" case as the bound on the integral are given by $-1/2$ and $-1/2+\theta$. Hence, even a constant integrand would give $\int_{D_3^-}1\; ds \leq O(\theta)$.

Finally we need to treat the remaining two triangles $D_{4}^-$ and $D_{5}^-$. As we work on the torus, $D_5^-$ is equivalent to having the peak right on the left of $-\theta$ (i.e. $-\theta+s\approx 1$) and we must therefore treat separately the integral on the sub-interval $[-\theta, -\theta+\frac{1}{2n+3}]$ and the contribution on the interval $[-\theta+\frac{1}{2n+3}, 0]$. 

On $D_4^-$, the reasoning is the same as for $D_{1}^-$, except we need to split the integral between a first contribution on which $s+t<-1/2$ and a second contribution on which $s+t>-1/2$. For those two subdomains, we thus write 

\begin{align}
F_{R,D_5}^-  &= \left|\int_{-\theta}^{0} \frac{\sin((2n+3)\pi (s+t))}{2\sin(\pi (s+t))}\; dt\right| + \frac{\theta}{2} \\
&\leq \left|\int_{-\theta}^{-\theta+\frac{1}{2n+3}} \frac{\sin((2n+3)\pi (t+s))}{2\sin(\pi (s+t))}\; dt\right| + \left|\int_{-\theta+\frac{1}{2n+3}}^{-\frac{1}{2}- s} \frac{\sin((2n+3)\pi (t+s))}{2\sin(\pi (t+s))}\; dt\right| \\
& + \left|\int_{-\frac{1}{2}-s}^0 \frac{\sin((2n+3)\pi (t+s))}{2\sin(\pi (t+s))}\; dt\right|+ \frac{\theta}{2}\\
&\leq \frac{\pi}{4} + \left|\frac{\cos((2n+3)\pi (t+s))}{2\pi (2n+3)\sin(t+s)}\; dt\right|_{-\theta+\frac{1}{2n+3}}^{-1/2-s} +\left| \int_{-\theta+\frac{1}{2n+3}}^{-1/2-s} \frac{\cos((2n+3)\pi (t+s))\cos(\pi (s+t))}{2(2n+3)\sin^2(s+t)}\; dt\right|\label{boundPeakD5}\\
& + \left|\frac{\cos((2n+3)\pi (t+s))}{2(2n+3)\pi \sin(\pi (t+s))}\right|_{-1/2-s}^0 + \left|\int_{-1/2-s}^0 \frac{\cos((2n+3)\pi(t+s))\cos(\pi (t+s))}{2(2n+3)\sin^2(\pi(t+s))}\; dt\right|+\frac{\theta}{2}\\
&\leq \frac{\pi}{4} + \frac{1}{4(2n+3)} \left(\frac{1}{\pi} + \frac{1}{2}\right)\left(2 + \frac{1}{1+s-\theta+\frac{1}{2n+3}}\right)+ \frac{1}{4(2n+3)} \left(\frac{1}{\pi}+ \frac{1}{2}\right)\left(2+ \frac{1}{-s}\right) + \frac{\theta}{2}\label{boundRealPartF5minus}
\end{align}
In~\eqref{boundPeakD5} we use $|\sin((2n+3)\pi (t+s))|\leq (2n+3)\pi (1+t+s)$. 

On $D_4^-$,  the idea is similar and we split the integral as
\begin{align}
F_{R,D_4}^- &= \left|\int_{-\theta}^0 \frac{\sin((2n+3)\pi (s+t))}{2\sin(\pi (s+t))}\; dt\right| + \frac{\theta}{2}\\
&\leq \left|\int_{-\theta}^{-s-\frac{1}{2}} \frac{\sin((2n+3)\pi (s+t))}{2\sin(\pi (s+t))}\right|_{-\theta}^{-s-\frac{1}{2}} + \left|\int_{-s-\frac{1}{2}}^{-\frac{1}{2n+3}} \frac{\sin((2n+3)\pi (s+t))}{2\sin(\pi (s+t))}\; dt\right|\\
& + \left|\int_{-\frac{1}{2n+3}}^{0} \frac{\sin((2n+3)\pi (s+t))}{2\sin(\pi(s+t))}\; dt\right|+\frac{\theta}{2}\\
&\leq \left|\left| -\frac{\cos((2n+3)\pi (s+t))}{(2n+3)\pi 2\sin(\pi (s+t))}\right|_{-\theta}^{-s-\frac{1}{2}} - \frac{1}{2}\int_{-\theta}^{-s-\frac{1}{2}} \frac{\cos((2n+3)\pi (s+t))\cos(\pi (s+t))}{(2n+3)\sin^2(\pi (s+t))}\right|+\frac{\theta}{2}\\
& +\left|\left|-\frac{\cos((2n+3)\pi (s+t))}{(2n+3)2\pi \sin(\pi (s+t))}\right|_{-s-\frac{1}{2}}^{-\frac{1}{2n+3}} -\int_{-s-\frac{1}{2}}^{-\frac{1}{2n+3}} \frac{\cos((2n+3)\pi (s+t))\cos(\pi (s+t))}{(2n+3)2\sin^2(\pi (s+t))}\; dt\right| \\
& + \frac{\pi}{4}+ \frac{\theta}{2}\\
&\leq \frac{1}{4(2n+3)}\left(\frac{1}{\pi} + \frac{1}{2}\right)\left(2 + \frac{1}{1+(s-\theta)}\right) + \frac{\pi}{4} + \frac{1}{4(2n+3)} \left(\frac{1}{\pi} + \frac{1}{2}\right)\left(\frac{1}{\frac{1}{2n+3}-s} + 2\right)+\frac{\theta}{2}\label{boundD4minusImag}
%
%
%
\end{align}

\subsection{Proof of lemma~\ref{boundInteriorIntegralFSthetaImaginary}}


We now bound the imaginary part~\eqref{moduleIntegralTotal} (i.e the second term in~\eqref{integralRealImaginaryDecomposition}). This imaginary part reads as 
\begin{align}
\text{Im} F_{s}(\theta) = \int_{-\theta}^0 \frac{\sin((n+2)(t+s))\sin((n+1)\pi (s+t))}{\sin(\pi (s+t))}\; dt
\end{align}
and the integrand which is shown in Fig.~\ref{integrandImagPart1} exhibits odd symmetry. We can thus always remove the interval around $-s$ where the denominator vanishes except when the zero of the denominator is located on the border of the $[-\theta,0]$ interval. 

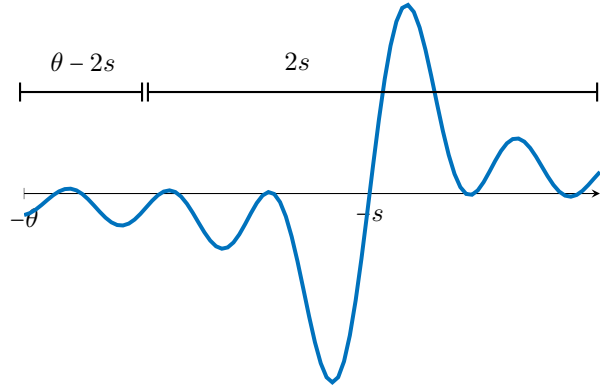
\begin{figure}\centering
%
%
\definecolor{mycolor1}{rgb}{0.00000,0.44700,0.74100}%
\begin{tikzpicture}

\begin{axis}[%
axis lines=left, xtick=\empty, ytick=\empty,
axis x line = middle,
axis y line = none,
width=6.028in,
height=4.754in,
at={(1.011in,0.642in)},
scale only axis,
xmin=-0.5,
xmax=0,
ymin=-10,
ymax=10,
axis background/.style={fill=white},
scale = .5,
xtick={-0.5,-0.2},
xticklabels={$-\theta$,$-s$}]
\addplot [color=mycolor1, forget plot, line width=1.5pt]
  table[row sep=crcr]{%
-0.5	-0.951056516295152\\
-0.494949494949495	-0.862124131472331\\
-0.48989898989899	-0.706303325079882\\
-0.484848484848485	-0.504239936309846\\
-0.47979797979798	-0.283215466070806\\
-0.474747474747475	-0.0737491585197326\\
-0.46969696969697	0.0943046026673774\\
-0.464646464646465	0.195672427771345\\
-0.45959595959596	0.21308684987934\\
-0.454545454545455	0.139765701859675\\
-0.44949494949495	-0.0193989824181942\\
-0.444444444444444	-0.248131471225891\\
-0.439393939393939	-0.520547198075322\\
-0.434343434343434	-0.804227518579083\\
-0.429292929292929	-1.06423876048722\\
-0.424242424242424	-1.26760913035262\\
-0.419191919191919	-1.38770401005518\\
-0.414141414141414	-1.40793567670496\\
-0.409090909090909	-1.32430996391131\\
-0.404040404040404	-1.14644353402997\\
-0.398989898989899	-0.896866732532385\\
-0.393939393939394	-0.60863732162867\\
-0.388888888888889	-0.321504446441852\\
-0.383838383838384	-0.0770535161591636\\
-0.378787878787879	0.0865931182485866\\
-0.373737373737374	0.139867046517389\\
-0.368686868686869	0.0659839952957157\\
-0.363636363636364	-0.136257847021278\\
-0.358585858585859	-0.451528092651716\\
-0.353535353535354	-0.849068242886721\\
-0.348484848484849	-1.28581016221702\\
-0.343434343434343	-1.71122015457523\\
-0.338383838383838	-2.07335581928893\\
-0.333333333333333	-2.32545197313814\\
-0.328282828282828	-2.43225651541669\\
-0.323232323232323	-2.37533120160388\\
-0.318181818181818	-2.15661969128787\\
-0.313131313131313	-1.79975883645151\\
-0.308080808080808	-1.34885149571217\\
-0.303030303030303	-0.864704134497587\\
-0.297979797979798	-0.418828289678295\\
-0.292929292929293	-0.0857776538038256\\
-0.287878787878788	0.0653901904154333\\
-0.282828282828283	-0.020396770224185\\
-0.277777777777778	-0.376743436590212\\
-0.272727272727273	-1.01020085298427\\
-0.267676767676768	-1.89728676530458\\
-0.262626262626263	-2.98456622423803\\
-0.257575757575758	-4.19192665324895\\
-0.252525252525252	-5.41883894000637\\
-0.247474747474747	-6.55306621295504\\
-0.242424242424242	-7.48100183652719\\
-0.237373737373737	-8.09861974716759\\
-0.232323232323232	-8.32192698492381\\
-0.227272727272727	-8.0958323102165\\
-0.222222222222222	-7.40048551443066\\
-0.217171717171717	-6.25438595667471\\
-0.212121212121212	-4.71388100692077\\
-0.207070707070707	-2.86904162917515\\
-0.202020202020202	-0.836274314232606\\
-0.196969696969697	1.25163390214571\\
-0.191919191919192	3.25707388596119\\
-0.186868686868687	5.04964579319545\\
-0.181818181818182	6.51748353562192\\
-0.176767676767677	7.57676620109598\\
-0.171717171717172	8.17849913529064\\
-0.166666666666667	8.31193086803655\\
-0.161616161616162	8.00431956672304\\
-0.156565656565657	7.31713681598484\\
-0.151515151515151	6.33915720771173\\
-0.146464646464646	5.17719119869634\\
-0.141414141414141	3.94544363718286\\
-0.136363636363636	2.75459841920456\\
-0.131313131313131	1.70172994007524\\
-0.126262626262626	0.862026346614488\\
-0.121212121212121	0.283092647653609\\
-0.116161616161616	-0.0176911100991568\\
-0.111111111111111	-0.0526169116614754\\
-0.106060606060606	0.139897699357913\\
-0.101010101010101	0.501339273678664\\
-0.095959595959596	0.960856383283907\\
-0.0909090909090909	1.44383598665082\\
-0.0858585858585859	1.88033759122603\\
-0.0808080808080808	2.21248049819495\\
-0.0757575757575757	2.40003273029356\\
-0.0707070707070707	2.42368096658663\\
-0.0656565656565656	2.28574080572518\\
-0.0606060606060606	2.00836303357722\\
-0.0555555555555556	1.62957042336981\\
-0.0505050505050505	1.19769001761075\\
-0.0454545454545455	0.7649032863824\\
-0.0404040404040404	0.380705513481915\\
-0.0353535353535354	0.0860413011138937\\
-0.0303030303030303	-0.0912288433571098\\
-0.0252525252525252	-0.13905940406763\\
-0.0202020202020202	-0.0619082484354972\\
-0.0151515151515151	0.120606661355202\\
-0.0101010101010101	0.376831197484868\\
-0.00505050505050503	0.667526820411641\\
0	0.951056516295153\\
};
\end{axis}
\draw [black, thick, |-|] (4.2, 6) -- (10.2, 6);
\draw [black, thick, |-|] (2.5, 6) -- (4.15, 6);
\node at (3.35, 6.4) {$\theta-2s$};
\node at (6.2, 6.4) {$2s$};
\end{tikzpicture}%
\caption{\label{integrandImagPart1}Imaginary part of the integral~\eqref{integralFthetaDef}. The interval $[-\theta,0]$ is divided into the sub-interval $[-2s,0]$ on which the function is even, and on which the integral thus vanishes, and the remaining interval $[-\theta,-2s]$ which is bounded below using finite order expansions. }
\end{figure}

We keep the same domain decomposition as for the real part (see Fig.~\ref{domainDecomposition1}), up to to ligth changes (we might for example replace the boundaries in $1/(2n+3)$ with boundaries at $1/n$ as the period of the numerator is not $1/(2n+3)$ anymore). We start with $D_1^+$. Just as before, on that subdomain, we integrate by using upper and lower bounds on the sines at the numerator and denominator 
\begin{align}
|F_{D_1}^+|\leq \left|\int_{-\theta}^0 \frac{\sin((n+1)\pi (s+t))\sin((n+2)\pi (s+t))}{\sin(\pi (s+t))}\; dt\right| \leq \frac{(n+1)\pi}{2}\theta \label{boundImaginaryFD1plus}
\end{align}
On $D_3^+$ and $D_2^+$, we remove the subinterval located near the zero of the denominator and then integrate over the remaining interval. We use a reasoning similar to the one we used for the real part and decompose the integral as 
\begin{align}
&\int_{a}^b \frac{\sin((n+1)\pi (s+t))\sin(\pi (n+2)(s+t))}{\sin(\pi (s+t))}\; dt \\
&= \int_{a}^b -\frac{\cos((2n+3)\pi(t+s))}{2\sin(\pi (s+t))}\; dt + \int_{a}^b \frac{\cos(\pi (s+t))}{2\sin(\pi (s+t))}\; dt\\
& \leq  \left|\frac{\sin((2n+3)\pi (s+t))}{4\pi (2n+3)\sin(\pi (s+t))}\right|_{a}^b + \int_{a}^b \frac{\sin((2n+3)\pi (s+t))\cos(\pi (s+t))}{2(2n+3)\sin^2(\pi (s+t))}\; dt \\
&+ \int_{a}^b \frac{\cos(\pi (s+t))}{2\sin(\pi (s+t))}\; dt\\
& \leq \left(\frac{\sin((2n+3)\pi (s+b))}{2\pi (2n+3)\sin(\pi (s+b))} - \frac{\sin((2n+3)\pi (s+a))}{4\pi (2n+3)\sin(\pi (s+a))}\right) + \int_{a}^b \frac{\sin((2n+3)\pi (s+t))\cos(\pi (s+t))}{2(2n+3)\sin^2(\pi (s+t))}\; dt \\
&+ \int_{a}^b \frac{\cos(\pi (s+t))}{2\sin(\pi (s+t))}\; dt\\
&\leq \frac{1}{4\pi(2n+3)|s+b| } + \frac{1}{4\pi(2n+3) |s+a|} + \left|\int_{a}^b \frac{1}{8(2n+3)(s+t)^2}\; dt\right|\\
&+ \int_{a}^b \frac{\cos(\pi (s+t))}{2\sin(\pi (s+t))}\; dt\\
&\leq \frac{1}{4(2n+3)}\left(\frac{1}{\pi}+\frac{1}{2}\right)\left(\frac{1}{|s+b|}+ \frac{1}{|s+a|}\right)+ \int_{a}^b \frac{\cos(\pi (s+t))}{2\sin(\pi (s+t))}\; dt\\
\end{align}
In this case we thus have an additional term given by the integral of the cotangent so we add a bound of the form
\begin{align}
\left|\int_{a}^b \frac{\cos(\pi (s+t))}{2\sin(\pi (s+t))}\; dt\right| \leq \int_{a}^b \left|\frac{1}{4|s+t|}\right|\; dt\\
\end{align}
From this, we can write 
\begin{align}
&\left|F_{D_2}^+\right|\leq \left|\int_{-\theta}^{-\theta+\frac{2}{n}} \frac{(n+1)\pi}{2}\; dt\right| + \int_{-\theta+\frac{2}{n}}^0 \frac{\sin((n+1)\pi (s+t))\sin((n+2)\pi (s+t))}{\sin(\pi (s+t))}\; dt\\
&\leq \frac{(n+1)\pi}{2n} + \frac{1}{4(2n+3)} \left(\frac{1}{\pi}+\frac{1}{2}\right)\left(\frac{1}{|s|} + \frac{1}{|s-\theta+ \frac{2}{n}|} \right) + \frac{1}{4}\left|-\log((s-\theta)  + \frac{2}{n})+\log(s)\right|\label{boundImaginaryFD2plus}
\end{align}

When on $D_3^+$, we remove the subinterval $[-\frac{2}{n},0]$ and split the integral as 

\begin{align}
\left|F_{D_3}^+\right| &\leq \left|\int_{-\theta}^{0} \frac{\sin((n+1)\pi (s+t))\sin((n+2)\pi (s+t))}{\sin(\pi (s+t))}\; dt \right|\\
&\leq \left|\int_{-\theta}^{-\frac{2}{n}} \frac{\sin((n+1)\pi (s+t))\sin((n+2)\pi (s+t))}{\sin(\pi (s+t))}\; dt\right| + \left|\int_{-\frac{2}{n}}^{0} \frac{(n+1)\pi |s+t|}{|2 (s+t)|}\; dt \right|\\
&\leq \frac{1}{4(2n+3)}\left(\frac{1}{\pi}+\frac{1}{2}\right) \left(\frac{1}{(\theta-s)} + \frac{1}{\frac{2}{n}-s}\right) + \frac{1}{4} \left|\log(\theta-s) - \log(\frac{2}{n}-s)\right|\\
&\leq \frac{1}{4(2n+3)}\left(\frac{1}{\pi}+\frac{1}{2}\right) \left(\frac{1}{\theta-s} + \frac{1}{\frac{2}{n}-s}\right) + \frac{1}{4} \log(\frac{\theta-s}{\frac{2}{n}-s}) + \frac{\pi (n+1)}{n}\label{boundFD3plusImag}
\end{align}
%

We then bound the integral on $D_0^+$.  Using the odd symmetry of the integrand, we only need to integrate over $[-\theta, -2s]$ whenever $-s\in [-\theta/2,0]$ and over $[\theta-2s,0]$ whenever $-s\in [-\theta,-\theta/2]$. The two frameworks are equivalent by symmetry of the integrand. In this first case we write 
\begin{align}
&\left|\int_{-\theta}^0 \frac{\sin(\pi (n+2)(t+s))\sin((n+1)\pi (t+s))}{2\sin(\pi (s+t))}\; dt\right|\\
&\leq \left|\int_{-\theta}^{-2s} \frac{\cos(\pi (2n+3)(t+s))}{2\sin(\pi (t+s))}\; dt - \int_{-\theta}^{-2s} \frac{\cos(\pi (t+s))}{2\sin(\pi (t+s))}\; dt\right|\\
&\leq \left|\left| \frac{\sin(\pi (2n+3)(t+s))}{2\pi (2n+3)\sin(\pi (t+s))}\right|_{-\theta}^{-2s} - \int_{-\theta}^{-2s} \frac{\sin(\pi (2n+3)(t+s))\cos(\pi (s+t))}{2(2n+3)\sin^2(s+t)}\; dt\right|\\
& + \left|-\int_{-\theta}^{-2s} \frac{\cos(\pi (t+s))}{2\sin(\pi(s+t))}\; dt\right|\\
& \leq \frac{1}{(2n+3)4}\left(\frac{1}{\pi} + \frac{1}{2}\right) \left(\frac{1}{|s|} + \frac{1}{|\theta-s|}\right)+ \left|\int_{-\theta}^{-2s} \frac{1}{4|s+t|}\right|\\
&\leq \frac{1}{(2n+3)4}\left(\frac{1}{\pi} + \frac{1}{2}\right) \left(\frac{1}{|s|} + \frac{1}{|\theta-s|}\right) +\frac{1}{4}\left| \log(\frac{s}{\theta-s})\right|\\
&\leq \frac{1}{(2n+3)4}\left(\frac{1}{\pi} + \frac{1}{2}\right) \left(\frac{1}{|s|} + \frac{1}{|\theta-s|}\right) + \frac{1}{4} \log(\frac{\theta - s}{s})
\end{align}

The last line follows from $\theta-s>s$. When $-s\in [-\theta,-\theta/2]$, we get 

\begin{align}
&\left|\int_{-\theta}^0 \frac{\sin((n+2)\pi (s+t))\sin((n+1)\pi (s+t))}{\sin(\pi (s+t))}\; dt\right|\\
&\leq \left|\int_{\theta-2s}^0 \frac{\cos((2n+3)\pi (s+t)) - \cos(\pi (s+t))}{2\sin(\pi (s+t))}\; dt\right|\\
& \left|\int_{\theta-2s} ^0\frac{1}{4|s+t|}\; dt\right| + \left|\frac{\sin((2n+3)\pi (t+s))}{2\pi (2n+3)\sin(\pi (t+s))}\right|_{\theta-2s}^0\\
& +\left| \int_{\theta-2s}^{0} \frac{\sin((n+1)\pi (s+t))\sin((n+1)\pi (s+t))}{2\sin^2(\pi (s+t))(2n+3)}\; dt\right|\\
&\frac{1}{4}\log(\frac{s}{\theta-s}) + \frac{1}{4(2n+3)}\left(\frac{1}{\pi} + \frac{1}{2}\right)\left(\frac{1}{s} + \frac{1}{\theta-s}\right)
\end{align}
The last line follows from $\theta-s<s$.  The general bound on $D_0^+$ thus reads as 

\begin{align}
F_{D_0^+}\leq \frac{1}{4(2n+3)}\left(\frac{1}{\pi} + \frac{1}{2}\right)\left(\frac{1}{s} + \frac{1}{\theta-s}\right) + \log(\frac{s}{\theta-s}\vee \frac{\theta-s}{s})\label{FD0Imaginary}
\end{align}

On $D_4^+$, provided, as $|\theta-s|>1/n$, we integrate directly. As for the real part, we will also need a tighted bound for the small $\theta$ to compensate for the prefactor $|1-e^{2\pi i \theta}|$ appearing in front of~\eqref{eq:dev_int_p+}. 

%
%
%
%
%
%
%
%
%
%
%

\begin{align}
F_{D_4}^+ &= \left|\int_{-\theta}^0 \frac{\sin((n+2)\pi (s+t))\sin(\pi (n+1)(s+t))}{2\sin(\pi(s+t))}\; dt\right|\\
&\leq \left|\int_{-\theta}^0 \frac{\cos((2n+3)\pi(s+t))}{2\sin(\pi (s+t))}\; dt\right| + \left|\frac{1}{2}\int_{-\theta}^0 \cot(\pi (s+t))\right|\\
&\leq \left|\frac{\sin((2n+3)\pi (s+t))}{2\pi(2n+3)\sin(\pi (s+t))}\right|_{-\theta}^0 + \left|\int_{-\theta}^0 \frac{\sin((2n+3)\pi (s+t))\cos(\pi (s+t))}{2\sin^2(\pi (s+t))}\; dt\right|+\frac{1}{4}\left|\log(\frac{s}{s-\theta}) \right|\\
&\leq \left|\frac{\sin((2n+3)\pi (s+t))}{2\pi(2n+3)\sin(\pi (s+t))}\right|_{-\theta}^0 + \frac{1}{8(2n+3)}\left|\frac{1}{s} - \frac{1}{s-\theta}\right| +\frac{1}{4}\left|\log(\frac{s}{s-\theta}) \right| \label{temporaryBoundD4plus1}
\end{align}

As before, for small $\theta$, we will also need the bound to be decreasing with $\theta$ in order to compensate for the prefactor which is $\mathcal{O}(\theta^{-1})$. We thus apply a reasoning similar to the one used for the real part

\begin{align}
&\frac{\sin((2n+3)\pi s)}{2(2n+3)\sin(\pi s)} - \frac{\sin((2n+3)\pi (s-\theta))}{2(2n+3)\sin(\pi(s-\theta))}\\
& = \frac{\sin((2n+3)\pi s)\sin(\pi (\theta-s)) - \sin((2n+3)\pi (s-\theta))\sin(\pi s)}{2(2n+3)\sin(\pi s)\sin(\pi (s-\theta))}\label{temporaryBoundD4plus2}
\end{align}
In particular, developing the numerator as for the real part, we get 
%
%
%

\begin{align}
\left|\frac{\sin((2n+3)\pi (s+t))}{2(2n+3)\sin(\pi (s+t))}\right|_{-\theta}^0 \leq \frac{4 \pi\theta}{8(s-\theta)}
\end{align}
%
%
%
%
%
%

Substituting this bound in~\eqref{temporaryBoundD4plus1}, we finally get 
\begin{align}
F_{D_4}^+ &\leq \frac{4 \pi^2\theta}{8(s-\theta)} + \frac{1}{8(2n+3)}\left|\frac{1}{s} - \frac{1}{s-\theta}\right| +\frac{1}{4}\left|\log(\frac{s}{s-\theta}) \right|\\
&\leq \frac{4 \pi^2\theta}{8(2n+3)s(s-\theta)} + \frac{1}{8}\frac{\theta}{s-\theta} +\frac{1}{4}\left|\log(\frac{s}{s-\theta}) \right|\\
&\leq \frac{4\pi\theta}{8(s-\theta)}+\frac{\theta}{8(2n+3)s(s-\theta)} + \frac{1}{4}\log(\frac{s}{s-\theta})\label{boundD4plusImag}
\end{align}

%
%

%
%

As for the real case, we will also need a specific bound for large $\theta$. In the large $\theta$ regime ($\theta\geq \frac{1}{2n+3}$),  as before, we turn to the simpler bound 
\begin{align}
F_{D_4}^+ \leq \frac{1}{4(2n+3)}\left(\frac{1}{\pi}+\frac{1}{2}\right)\left(\frac{1}{s} +\frac{1}{s-\theta}\right) + \frac{1}{4}\log(\frac{s}{s-\theta})\label{largeThetaImaginaryD4+}
\end{align}

We now bound the imaginary part on the negative domain. Again, we keep the decomposition of Fig.~\ref{domainDecomposition1}. we start with $D_0^-$.  On this subdomain, we obtain the bound almost directly, integrating by parts and noting that the bound $|\sin(\pi (s+t))|\leq |s+t|$ remains valid as $-1/2<s+t$ in this subdomain.

\begin{align}
F_{D_0}^-& = \left|\int_{-\theta}^0 \frac{\sin((n+1)\pi (s+t))\sin((n+2)\pi (s+t))}{\sin(\pi(s+t))}\; dt\right|\\
&\leq \left|\int_{-\theta}^0 \frac{\cos((2n+3)\pi (s+t))}{2\sin(\pi (s+t))}\; dt \right| + \frac{1}{2}\left|\int_{-\theta}^0 \cot(\pi (s+t))\; dt\right| \\
&\leq \frac{1}{4}\left|\log(\frac{-s}{\theta-s})\right| + \frac{1}{4(2n+3)}\left(\frac{1}{|s|} + \frac{1}{|s-\theta|}\right)\left(\frac{1}{\pi} + \frac{1}{2}\right)
\end{align}

The first term follows from $s-\theta < 0$ and $s<0$.

{\color{black} When dealing with small values of $\theta$, we use the same reasoning as in~\eqref{step1DevelopmentSmalltheta} to~\eqref{refinedBoundSmallthetaFirstTerm} and introduce a refined bound. 
%


\begin{align}
F_{D_0}^-& = \left|\int_{-\theta}^0 \frac{\sin((n+1)\pi (s+t))\sin((n+2)\pi (s+t))}{\sin(\pi(s+t))}\; dt\right|+ \frac{1}{2}\left|\int_{-\theta}^0 \cot(\pi (s+t))\; dt\right|\\
&\leq \left| \left|\frac{\sin((2n+3)\pi (s+t))}{(2n+3)\pi2\sin(\pi s)}\right|_{-\theta}^0 - \int_{-\theta}^0 \frac{\sin((2n+3)\pi (s+t))\cos(\pi (s+t))}{2(2n+3)\sin^2(\pi (s+t))}\; dt\right|\\
& + \frac{1}{2}\left|\int_{-\theta}^0 \cot(\pi (s+t))\; dt\right|\\
&\leq \left| \frac{\sin((2n+3)\pi s)}{(2n+3)\pi 2 \sin(\pi s)} - \frac{\sin((2n+3)\pi (s-\theta))}{(2n+3)\pi 2\sin(\pi s )}\right| + \left|\frac{1}{8(2n+3)(s)} - \frac{1}{8(2n+3)(s-\theta)}\right|\\
& + \frac{1}{2}\left|\int_{-\theta}^0 \cot(\pi (s+t))\; dt\right|\\
&\leq \frac{3 \theta}{2(s-\theta)} + \frac{\theta}{8(2n+3)s(s-\theta)} + \frac{1}{4}\left|\log(\frac{-s}{\theta-s})\right|\label{trueBoundD0minusImag}
\end{align}

}

On $D_2^-$, the interval $[-\theta,0]$ is of size $\mathcal{O}(1/n)$ and the denominator vanishes near the interval, we replace the integrand by using appropriate upper and lower bounds on the numerator and denominator. A direct comparison with previous results gives 
\begin{align}
F_{D_2}^- &=\left| \int_{-\theta}^0 \frac{\sin((n+1)\pi (s+t))\sin((n+2)\pi (s+t))}{2\sin(\pi (s+t))}\; dt\right|\leq \frac{\pi (n+1)\theta}{4}
\end{align}

On $D_1^-$ we treat separately the subintervals $[-\theta,-1/n]$ and $[-1/n,0]$ (which is near the vanishing denominator). Following the same approach as for the real part, we get 
\begin{align}
F_{D_1}^- &= \left|\int_{-\theta}^{-1/n} \frac{\sin((n+1)\pi (s+t)) \sin((n+2)\pi (s+t))}{2\sin(\pi (s+t))}\; dt\right| + \left|\int_{-1/n}^0 \frac{\sin((n+1)\pi (s+t))\sin((n+2)\pi (s+t))}{2\sin(\pi (s+t))}\right|\\
&\leq \left|\int_{-\theta}^{-1/n} \frac{\sin((n+1)\pi (s+t)) \sin((n+2)\pi (s+t))}{2\sin(\pi (s+t))}\; dt\right|  + \frac{\pi (n+1)}{4}\frac{1}{n}\\
&\leq \left|\int_{-\theta}^{-1/n} \frac{\cos((2n+3)\pi (s+t))}{\sin(\pi (s+t))}\; dt\right| +\left| \int_{-\theta}^{-1/n} \cot(\pi (s+t))\; dt\right| + \frac{\pi (n+1)}{4}\frac{1}{n}\\
&\leq \left|\log(\frac{|s-\frac{1}{n}|}{|s-\theta|})\right| + \frac{1}{(2n+3)4}\left(\frac{1}{\pi} + \frac{1}{2}\right)\left(\frac{1}{|s - 1/n|} + \frac{1}{|s-\theta|}\right) + \frac{\pi (n+1)}{4}\frac{1}{n}\\
&\leq \log(\frac{\theta-s}{\frac{1}{n}-s}) + \frac{1}{4(2n+3)}\left(\frac{1}{\pi} + \frac{1}{2}\right)\left(\frac{1}{\theta-s} + \frac{1}{1/n-s}\right) + \frac{\pi (n+1)}{4n}\label{boundFD1minusImaginary}
\end{align}
Noting that on that subdomain, we always have $-\theta<-1/n$ as well as $s<0$.

Finally we give a separate treatment to each of the subdomains $D_3^-$, $D_5^-$ and $D_4^-$ for which we must adapt the lower bounds on the denominator of the integrand to account for the case $t+s<-\frac{1}{2}$. Starting with $D_3^-$, we split the integral between $-\theta\leq t\leq -s-\frac{1}{2}$ and $-s-\frac{1}{2}\leq t \leq 0$, and use appropriate lower bounds

\begin{align}
F_{D_3}^-  &= \left|\int_{-\theta}^0 \frac{\sin((n+1)\pi (t+s))\sin(\pi (n+2)(t+s))}{2\sin(\pi (t+s))}\; dt\right|\\
& \leq \left|\int_{-\theta}^{-s-\frac{1}{2}} \frac{\sin((n+1)\pi (t+s))\sin((n+2)\pi (t+s))}{2\sin(\pi (s+t))}\; dt\right.\\
& + \left.\int_{-s-\frac{1}{2}}^0 \frac{\sin((n+1)\pi (s+t))\sin(n+2)\pi (s+t)}{2\sin(\pi (s+t))}\; dt\right|\\
&\leq \left|\int_{-\theta}^{-1/2-s}\frac{\cos(2n+3)\pi(s+t)}{2\sin(\pi (s+t))}\; dt - \frac{1}{2}\int_{-\theta}^{-1/2-s} \cot(\pi (s+t))\; dt\right.\\
& + \left.\int_{-s-\frac{1}{2}}^0 \frac{\cos((2n+3)\pi (s+t))}{2\sin(\pi (s+t))}\; dt -\frac{1}{2} \int_{-s-\frac{1}{2}}^0 \cot(\pi (s+t))\; dt\right|\\
&\leq \left|\int_{-\theta}^{-1/2-s} \frac{\cos((2n+3)\pi (s+t))}{2\sin(\pi (s+t))}\; dt + \int_{-s-1/2}^{0} \frac{\cos((2n+3)\pi (s+t))}{\sin(\pi (s+t))}\; dt\right|\\
& + \frac{1}{4}\left|\int_{-1/2-s}^{0} \frac{1}{|s+t|}\; dt +\int_{-\theta}^{-1/2-s} \frac{1}{|1+s+t|}\; dt \right|\\
&\leq \left|\int_{-\theta}^{-1/2-s} \frac{\cos((2n+3)\pi (s+t))}{2\sin(\pi (s+t))}\; dt + \int_{-s-1/2}^{0} \frac{\cos((2n+3)\pi (s+t))}{\sin(\pi (s+t))}\; dt\right|\\
& + \frac{1}{4}\left|\int_{-1/2-s}^{0} \frac{1}{-s-t}\; dt +\int_{-\theta}^{-1/2-s} \frac{1}{(1+s+t)}\; dt \right|\\
&\leq \left|\int_{-\theta}^{-1/2-s} \frac{\cos((2n+3)\pi (s+t))}{2\sin(\pi (s+t))}\; dt + \int_{-s-1/2}^{0} \frac{\cos((2n+3)\pi (s+t))}{\sin(\pi (s+t))}\; dt\right|\\
& + \frac{1}{4}\left|\left| -\log(-s-t)\right|_{-1/2-s}^{0} +\left| \log(1+s+t)\right|_{-\theta}^{-1/2-s} \; dt \right|\\
&\leq \left|\int_{-\theta}^{-1/2-s} \frac{\cos((2n+3)\pi (s+t))}{2\sin(\pi (s+t))}\; dt + \int_{-s-1/2}^{0} \frac{\cos((2n+3)\pi (s+t))}{\sin(\pi (s+t))}\; dt\right|\\
& + \frac{1}{4}\left|-\log(-s) + \log(1/2)  + \log(1/2) - \log(1+s-\theta) \right|\\
&\leq \left|\int_{-\theta}^{-1/2-s} \frac{\cos((2n+3)\pi (s+t))}{2\sin(\pi (s+t))}\; dt + \int_{-s-1/2}^{0} \frac{\cos((2n+3)\pi (s+t))}{\sin(\pi (s+t))}\; dt\right|\\
 &+\frac{1}{4}\left(\log(2) - \log(1+s-\theta)\right) + \frac{1}{4} \left(-\log(-s) +\log(2)\right)\label{tightBoundDecreaseLemmaTronc}
\end{align}

%

The final bound on $F_{D_3}^-$ is thus given by
\begin{align}
\label{trueBoundD3minusImag}
\begin{split}
F_{D_3}^- &\leq \frac{1}{4}|\log(-s) + \frac{1}{4}\log(1+s-\theta)| + |2\log(2)| +   \frac{1}{4(2n+3)}\left(\frac{1}{\pi} + \frac{1}{2}\right)\left(2+ \frac{1}{-s}\right) \\
&+ \frac{1}{4(2n+3)}\left(\frac{1}{\pi} + \frac{1}{2}\right)\left(\frac{1}{1+s-\theta} + 2\right) 
\end{split}
\end{align}

On $D_5^-$, we use a similar decomposition, noting that, as we are working on the torus, we now have a subinterval for which $\pi (s+t) \approx \pi$. We thus treat independently a small interval of length $\mathcal{O}(1/n)$ located near $-\theta$.

\begin{align}
\left|F_{D_5}^-\right| &\leq \left|\int_{-\theta}^{-\theta+\frac{1}{n}} \frac{\sin((n+1)\pi (s+t))\sin(\pi (n+2)(s+t))}{2\sin(\pi (s+t))} \; dt\right|\\
& + \left|\int_{-\theta+\frac{1}{n}}^{0} \frac{\sin((n+1)\pi (s+t))\sin(\pi (n+2)(s+t))}{2\sin(\pi (s+t))}\right|\\
&\leq \left|\int_{-\theta+\frac{1}{n}}^{0} \frac{1}{2}\frac{\cos((2n+3)\pi (s+t))}{\sin(\pi (s+t))}\; dt\right| + \left|\int_{-\theta+\frac{1}{n}}^{0} \frac{\cos(\pi (s+t))}{2\sin(\pi (s+t))}\; dt\right|\\
& + \left|\frac{\pi (n+1) \left|s+t\right|}{4 n\left|s+t\right|}\right|
\end{align}
We then need to split the remaining integral between the interval on which $s+t<-\frac{1}{2}$ and the remaining interval on which $s+t>-\frac{1}{2}$.
\begin{align}
\left|F_{D_5}^-\right| &\leq \left|\int_{-\theta+\frac{1}{n}}^{-\frac{1}{2}-s} \frac{1}{2} \frac{\cos((2n+3)\pi (s+t))}{\sin(\pi (1+s+t))}\; dt\right| + \left|\int_{-\theta+\frac{1}{n}}^{-s-\frac{1}{2}} \frac{\cos(\pi (1+s+t))}{2\sin(\pi (1+s+t))}\; dt\right|\label{FD5-debut}\\
& + \left|\int_{-s-\frac{1}{2}}^0 \frac{1}{2}\frac{\cos((2n+3)\pi (s+t))}{\sin(\pi (s+t))}\right| + \left|\int_{-s-\frac{1}{2}}^0 \frac{1}{2}\frac{\cos(\pi (s+t))}{\sin(\pi (s+t))}\; dt\right|\\
&+ \frac{(n+1)\pi}{4}\\
&\leq 2\frac{\pi }{4} +  \frac{1}{4}\left|\log(\frac{-s}{1/2})\right| + \frac{1}{4}\left|\log(\frac{1/2}{1+s-\theta+\frac{1}{n}})\right|\\
&+\frac{1}{4(2n+3)} \left(\frac{1}{\pi}+\frac{1}{2}\right)\left(\frac{1}{1/2} + \frac{1}{1+s-\theta+\frac{1}{n}}\right)\\
& + \frac{1}{4(2n+3)} \left(\frac{1}{\pi}+\frac{1}{2}\right)\left(\frac{1}{1/2} + \frac{1}{-s}\right)\label{FD5-fin}
\end{align}

For $D_4^-$, we use a similar decomposition except that we now need to remove the small subinterval located around $0$ where the denominator is nearly vanishing.

\begin{align}
\left|F_{D_4^-} \right| &\leq \left|\int_{-\theta}^{-s-\frac{1}{2}} \frac{\sin((n+1)\pi (1+s+t))\sin(\pi (n+2)\pi (1+s+t))}{2\sin(\pi (1+s+t))}\; dt \right| \\
&+ \left|\int_{-s-\frac{1}{2}}^{-\frac{1}{n}} \frac{\sin((n+1)\pi (s+t))\sin((n+2)\pi(s+t))}{2\sin(\pi (s+t))}\; dt\right|\\
& + \left|\int_{-\frac{1}{n}}^{0} \frac{\sin((n+1)\pi(s+t))\sin((n+2)\pi(s+t))}{2\sin(\pi(s+t))}\; dt\right|\\
& \leq \frac{1}{n}\frac{(n+1)\pi}{4}\\
&+ \left|\int_{-s-\frac{1}{2}}^{-\frac{1}{n}} \frac{\cos((2n+3)\pi (s+t))}{2\sin(\pi (s+t))}\; dt\right| + \left|\int_{-s-\frac{1}{2}}^{-\frac{1}{n}} \frac{\cos(\pi (s+t))}{\sin(\pi (s+t))}\; dt \right|\\
& +  \left|\int_{-\theta}^{-s-\frac{1}{2}} \frac{\cos((2n+3)\pi (s+t))}{2\sin(\pi (1+s+t))}\; dt\right| + \left|\int_{-\theta}^{-s-\frac{1}{2}} \frac{\cos(\pi (s+t))}{\sin(\pi (1+s+t))}\; dt \right|\\
&\leq \frac{1}{n}\frac{(n+1)\pi}{4}\\ 
& + \frac{1}{4(2n+3)}\left(\frac{1}{2}+\frac{1}{\pi}\right) \left(\frac{1}{1+s-\theta} + \frac{1}{1/2}\right) + \left|\log(\frac{1/2}{1+s-\theta})\right|\\
&+ \frac{1}{4(2n+3)}\left(\frac{1}{2}+\frac{1}{\pi}\right) \left(\frac{1}{-s+\frac{1}{n}} + \frac{1}{1/2}\right) + \left|\log(\frac{1/2}{-s+\frac{1}{n}})\right|
\end{align}

\subsection{Proof of lemma~\ref{lemmaTroncD3minusCons} ($D_3^-$, small $|\tau - \alpha|$)}

For $D_3^-$, recall that the real part can be controled as  
\begin{align}
F_{R, D_3^-}(s) &\leq c\left(4 + \frac{1}{-s} + \frac{1}{1+s-\theta}\right) + \frac{\theta}{2}\label{rappelSecondContribution}
\end{align}
For the real part, we simply use the fact that the integration interval has a length equal to $\theta$.  
\begin{align}
\frac{1}{|1-e^{2\pi i \theta}|}\int_{-1/2}^{-1/2+\theta} |p_\ell(s) | \cdot F_{D_3^-}(s)\; ds & \leq \frac{1}{\pi \theta} \sup_{s \in [-1/2, -1/2+\theta]}\left\{|p_\ell (s) | \cdot F_{D_3^-}(s)\right\}\\
&\leq \frac{1}{\pi} \left( (4c + \frac{\theta}{2}) + \frac{2c}{(1/2 - \theta) \vee \frac{1}{2n+3}} \right)\\
&\leq \frac{1}{\pi}\left((4c+\frac{\theta}{2})+ 6c'\right)\\
&\leq 0.5224\label{boundRealPartD3minusSmallTauMinusAlpha}
\end{align}
For the imaginary part, recall that we have 
\begin{align}
F_{I,D_3^-} \leq F_{R,D_3^-} + \frac{1}{4} \log(\frac{-1}{s}) + \frac{1}{4} \log(\frac{1}{1+s-\theta}) +2\log(2)\label{reminderBoundImaginaryFD3minus}
\end{align}
We denote the total integral arising from the imaginary part as $Z$ and consider the following decomposition,
\begin{align}
Z & =\frac{1}{4|e^{2\pi i \theta}-1|}\int_{-\frac{1}{2}+\theta - 2\frac{C_1-1}{n}}^{-\frac{1}{2}+\theta } -\frac{1}{s} \; ds + \frac{1}{4|e^{2\pi i \theta}-1|}\int_{-\frac{1}{2}}^{-\frac{1}{2}+ 2\frac{C_1-1}{n}} \frac{1}{1+s-\theta}\; ds\label{CorrectionFD3minusConstantContribution}\\
&+\frac{1}{4|e^{2\pi i \theta}-1|} \int_{-\frac{1}{2}+\theta}^{-\frac{1}{2}} - \frac{1}{s} \frac{C_1}{1+n(-\frac{1}{2} + \theta + \frac{C_1-1}{n}-s)}\; ds\label{CorrectionFD3minusThirdTerm}\\
&+ \frac{1}{4|e^{2\pi i \theta}-1|}\int_{-\frac{1}{2}}^{-\frac{1}{2}+\theta} \frac{1}{1+s-\theta} \frac{C_1}{1+n(s- (-\frac{1}{2} - \frac{C_1-1}{n}))}\; ds\label{CorrectionFD3minusFourthTermCentral}\\
&+ \sup_{0\leq \tau\leq \theta}\frac{1}{4|e^{2\pi i \theta}-1|} \int_{-\frac{1}{2}}^{-\frac{1}{2}+\tau} \frac{1}{1+s-\theta} \frac{C_1}{1+n(-\frac{1}{2} + \tau + \frac{C_1-1}{n} - s)}\; ds\label{CorrectionFD3minusFourthTermCentral2}\\
&+ \sup_{0\leq \tau \leq \theta}\frac{1}{4|e^{2\pi i \theta}-1|} \int_{-\frac{1}{2}+\theta - \tau}^{-\frac{1}{2}+\theta} \frac{1}{-s} \frac{C_1}{1+n(s-(-\frac{1}{2}+\theta -\tau - \frac{C_1-1}{n}))}+ \frac{2\log(2)}{\pi}.\label{CorrectionFD3minusFourthTermCentral3}
\end{align}

We now control each of the terms above. First note that we can safely assume  $\theta \geq \frac{C_1-1}{n}$, otherwise, we simply compute the correction as 
\begin{align}
\frac{1}{\pi \theta} \int_{-\frac{1}{2}}^{-\frac{1}{2}+\theta} \log(-s) + \log(1+s-\theta) + 2\log(2)\; ds &= \frac{1}{\pi \theta} \left\{2\log(1+ \frac{\theta}{1/2-\theta}) +2\log(2)\theta\right\}\label{ImagboundCorrectionFD3minusStart0001}\\
&\leq \frac{2}{\pi} \frac{1}{1/2-\theta} + \frac{2\log(2)}{\pi}\\
&\leq \frac{8}{\pi}+\frac{2\log(2)}{\pi}
\end{align}

For the first integral in~\eqref{CorrectionFD3minusConstantContribution}, we can write 
\begin{align}
\frac{1}{4|e^{2\pi i \theta} - 1|}\int_{-\frac{1}{2}+ \theta - 2\frac{C_1-1}{n}}^{-\frac{1}{2}+\theta}  -\frac{1}{s} \; ds &\leq \frac{1}{4\pi \theta} \int_{-\frac{1}{2}+\theta-2\frac{C_1-1}{n}}^{-\frac{1}{2}+\theta} - \frac{1}{s}\; ds\\
&\leq \frac{1}{4\pi \theta} \log(1+2\frac{C_1-1}{n}\frac{1}{1/2-\theta}) \\
&\leq \frac{1}{4}\left(\frac{8}{\pi} + \frac{4}{\pi}\log(C_1)\right) \label{additionalExplanationNeeded000100}
\end{align}
In~\eqref{additionalExplanationNeeded000100}, we make the distinction between $\theta \leq 1/4$ and $\theta \geq 1/4$. In the first case, the integral can be upper bounded by $8/\pi$. In the second case, using the definition of $D_3^-$ and hence $1/2 - \theta \geq \frac{1}{2n+3}$ the integral is upper bounded by $(4/\pi) \log(1+ 6C_1 - 5)$.

For the second term in~\eqref{CorrectionFD3minusConstantContribution}, we have 
\begin{align}
\frac{1}{4|e^{2\pi i \theta}-1|}\int_{-\frac{1}{2}}^{-\frac{1}{2} + 2\frac{C_1-1}{n}} \frac{1}{1+s-\theta}\; ds &\leq \frac{1}{4\pi \theta} \log(1+ 2\frac{C_1-1}{n}\frac{1}{1/2-2\frac{C_1-1}{n}})\leq \frac{8}{4\pi}
\end{align}
The last line holds as soon as $n\geq 8(C_1-1)$ and follows from $\log(1+x)\leq x$ for any $x\geq 0$. 

For the third integral~\eqref{CorrectionFD3minusThirdTerm} we have 
\begin{align}
&\frac{1}{4|e^{2\pi i \theta}-1|}\int_{-\frac{1}{2}}^{-\frac{1}{2}+\theta} \frac{1}{-s} \frac{C_1}{1+n(-\frac{1}{2} + \theta + \frac{C_1-1}{n}-s)}\; ds\\
 & \leq \frac{1}{4\pi \theta} \int_{-\frac{1}{2}}^{-\frac{1}{2} + \theta}\left(-\frac{1}{s} + \frac{-n}{1+n(-\frac{1}{2} + \theta + \frac{C_1-1}{n} -s)}\right)\frac{1}{-\frac{1}{n} + \frac{1}{2} - \theta - \frac{C_1-1}{n}}\\
&\leq \frac{1}{4\pi \theta} \left[\log(\frac{1/2}{1/2-\theta}) + \log(\frac{C_1/n}{C_1/n + \theta})\right] \frac{1}{\frac{1}{2} - \theta - \frac{C_1}{n}}\frac{C_1}{n}\label{expressionWhichRepeats1}
\end{align}
To control the last line, we consider two cases. 
\begin{itemize}
\item Either $|1/2 - \theta - \frac{C_1}{n}|\geq 1/4$. This necessarily implies either $1/2 - \theta \geq 1/4+ \frac{C_1}{n}$ which gives
\begin{align}
&\frac{1}{4|e^{2\pi i \theta}-1|} \int_{-1/2}^{-1/2+\theta} \frac{1}{-s} \frac{C_1}{1+n(-1/2 + \theta + \frac{C_1-1}{n} - s)}\; ds\label{startBounexpressionWhichRepeats1}\\
&\leq \frac{1}{4\pi \theta} \left[\log(\frac{1/2}{1/2-\theta}) + \log(\frac{C_1/n}{C_1/n + \theta})\right] \frac{1}{1/2 - \theta - \frac{C_1}{n}} \frac{C_1}{n}\\
&\leq \frac{2}{\pi \theta} \left[\log(1+ \frac{\theta}{1/2-\theta}) + \log(1+ \frac{\theta}{C_1/n})\right] \frac{C_1}{n}\\
&\leq \frac{2}{\pi}\label{tmpBoundD3minusImagCascadeOfBounds}
\end{align}
\item When $|\frac{1}{2} - \theta - \frac{C_1}{n}| \leq 1/4$, we write
\begin{align}
&\frac{1}{4|e^{2\pi i \theta}-1|} \int_{-1/2}^{-1/2+\theta} \frac{1}{-s} \frac{C_1}{1+n(-1/2 + \theta + \frac{C_1-1}{n} - s)}\; ds\\
& \leq \frac{1}{4\pi \theta} \left|\log(\frac{1/2}{1/2-\theta}) + \log(\frac{C_1/n}{C_1/n + \theta})\right|\frac{1}{1/2 - \theta - \frac{C_1}{n}} \frac{C_1}{n}\\
&\leq \frac{1}{4\pi \theta} \left|\log(1+ \frac{1/2 - C_1/n - \theta}{C_1/n + \theta}\vee \frac{C_1/n-1/2 + \theta}{1/2-\theta})\right| \frac{C_1}{n} \frac{1}{|1/2 - \theta - C_1/n|}\label{lastLineFramework3D3minusImag00}\\
&+ \frac{1}{4\pi \theta}\left|\log(1+ \frac{1/2 - \theta - \frac{C_1}{n}}{C_1/n}\vee \frac{C_1/n - 1/2 + \theta}{1/2 - \theta}) \right| \frac{1}{|1/2 - \theta - \frac{C_1}{n}|} \frac{C_1}{n}\label{lastLineFramework3D3minusImag01}
\end{align}
To bound the sum~\eqref{lastLineFramework3D3minusImag00} + \eqref{lastLineFramework3D3minusImag01}, we consider the following two frameworks:
\begin{itemize}
\item If $(1/2- \theta)\geq \frac{C_1}{n}$, then we have 
\begin{align}
~\eqref{lastLineFramework3D3minusImag00} + \eqref{lastLineFramework3D3minusImag01} &\leq \frac{1}{4\pi \theta} \left|\log(1+ \frac{1/2 - C_1/n - \theta}{C_1/n + \theta}) + \log(1+ \frac{1/2 - \theta - \frac{C_1}{n}}{C_1/n})\right| \frac{1}{|1/2 - \theta - \frac{C_1}{n}|} \frac{C_1}{n}\\
&\leq \frac{2}{\pi}
\end{align}
\item Or $(1/2-\theta) \leq C_1/n$. In this case, 
\begin{enumerate}[a)]
\item If $(1/2 - \theta)\leq \frac{C_1}{n}\frac{1}{2}$, we control $|1/2 - \theta - \frac{C_1}{n}|$ as $|1/2 - \theta - \frac{C_1}{n}|\geq \frac{C_1}{2n}$, and upper bound the sum $~\eqref{lastLineFramework3D3minusImag00} + \eqref{lastLineFramework3D3minusImag01} $ as  
\begin{align}
~\eqref{lastLineFramework3D3minusImag00} + \eqref{lastLineFramework3D3minusImag01}  \leq \frac{4}{4\pi \theta} \left|\log(1+ 3C_1)\right| \leq \frac{8}{\pi}\left[ \log(4)+ \log(C_1)\right]\label{lastLineFramework3D3minusImagA}
\end{align}
\item If $(1/2 - \theta) \geq \frac{C_1}{2n}$, we use $\log(1+x) \leq x$ to cancel the prefactor $|1/2 - \theta - \frac{C_1}{n}|$ and write 
\begin{align}
~\eqref{lastLineFramework3D3minusImag00} + \eqref{lastLineFramework3D3minusImag01}  \leq \frac{8}{4\pi} \left|\frac{4}{C_1/n}\right| \frac{C_1}{n} \leq \frac{32}{4\pi}\label{lastLineFramework3D3minusImagB}
\end{align}
\end{enumerate}
Combining~\eqref{lastLineFramework3D3minusImagA}, ~\eqref{lastLineFramework3D3minusImag00} and~\eqref{lastLineFramework3D3minusImag01}, we get 
\begin{align}
~\eqref{lastLineFramework3D3minusImag00} + \eqref{lastLineFramework3D3minusImag01} \leq \left\{\eqref{lastLineFramework3D3minusImagA} \vee \eqref{lastLineFramework3D3minusImagB}\right\}  \leq \frac{8}{\pi} \left(\log(4)+ \log(C_1)\right).
\end{align}
Combining this with~\eqref{tmpBoundD3minusImagCascadeOfBounds}, we get 
\begin{align}
\eqref{CorrectionFD3minusThirdTerm} = \frac{1}{4|e^{2\pi i \theta}-1|} \int_{-1/2}^{-1/2+\theta} \frac{1}{-s} \frac{C_1}{1+n(-1/2 + \theta + \frac{C_1-1}{n} - s)}\; ds&\leq \frac{8}{\pi} \left(\log(4)+ \log(C_1)\right).\label{finishBounexpressionWhichRepeats1}
\end{align}
\end{itemize}

\end{itemize}

For~\eqref{CorrectionFD3minusFourthTermCentral}, we make the distinction between the cases $\theta \geq \frac{1}{4}$ and $\theta\leq \frac{1}{4}$. In the latter case, we bound the correction as 
\begin{align}
\eqref{CorrectionFD3minusFourthTermCentral} \leq \frac{1}{\pi \theta}\int_{-\frac{1}{2}}^{-\frac{1}{2}+\theta} \frac{1}{1+s-\theta} \; ds \leq \frac{4}{\pi}
\end{align}
When $\theta\geq\frac{1}{4}$, we have 
\begin{align}
\eqref{CorrectionFD3minusFourthTermCentral} &\leq \frac{1}{4\pi \theta} \int_{-\frac{1}{2}}^{-\frac{1}{2} + \theta} \left(-\frac{1}{1+s-\theta} + \frac{n}{1+n(s-(-\frac{1}{2} - \frac{C_1-1}{n}))}\right) \frac{C_1}{n} \frac{1}{(1/2-\theta) - \frac{C_1}{n}}\; ds\\
&\leq \frac{1}{4\pi \theta} \left|-\log(\frac{1/2}{1/2 - \theta}) + \log(\frac{\theta + \frac{C_1}{n}}{C_1/n})\right| \frac{C_1}{n} \frac{1}{1/2 - \theta - \frac{C_1}{n}}
\end{align}
This last expression is the same as~\eqref{expressionWhichRepeats1} and we can thus reuse~\eqref{startBounexpressionWhichRepeats1} to~\eqref{finishBounexpressionWhichRepeats1} which gives 
\begin{align}
\eqref{CorrectionFD3minusFourthTermCentral} &\leq \frac{8}{\pi} \left(\log(4)+ \log(C_1)\right)
\end{align}

We now bound the two suprema. For~\eqref{CorrectionFD3minusFourthTermCentral2}, we have 
\begin{align}
\eqref{CorrectionFD3minusFourthTermCentral2}\leq \frac{1}{\pi \theta}\sup_\tau \int_{-\frac{1}{2}}^{-\frac{1}{2}+\tau} \frac{C_1}{n}\left(-\frac{1}{1+s-\theta} + \frac{n}{1+n(s - (-\frac{1}{2} - \frac{C_1-1}{n}))}\right)\frac{1}{(\frac{1}{2} - \theta) - \frac{C_1}{n}}
\end{align}
To bound the last line, we again consider two cases, depending on whether $\frac{C_1}{n}\geq \frac{1}{2}-\theta$ or not. Following this distinction, we can write 
\begin{align}
\eqref{CorrectionFD3minusFourthTermCentral2}&\leq \frac{1}{\pi \theta} \left(\log(\frac{\tau + \frac{C_1}{n}}{\frac{1}{2} + \tau - \theta}) + \log(\frac{C_1/n}{1/2-\theta})\right)\frac{1}{(1/2-\theta) - \frac{C_1}{n}}\frac{C_1}{n},\quad \text{ $C_1/n \geq 1/2-\theta$}\label{tmp0005001}\\
\eqref{CorrectionFD3minusFourthTermCentral2}&\leq \frac{1}{\pi \theta} \left(\log(\frac{1/2 +\tau - \theta}{\tau+ C_1/n}) + \log(\frac{1/2 - \theta}{C_1/n})\right),\quad \text{ $C_1/n \leq 1/2-\theta$}\label{tmp0005002}
\end{align}

For~\eqref{tmp0005001}, we make one more distinction. Either we have $(1/2-\theta)\leq \frac{C_1}{2n}$, in this case, we write 
\begin{align}
\eqref{tmp0005001}&\leq \left(\log(1+ \frac{C_1}{n}(2n+3)) + \log(1+ \frac{C_1}{n}(2n+3))\right) \frac{8}{\pi}\\
&\leq 2\left(\log(4) + \log(C_1)\right)\frac{8}{\pi}
\end{align}
Or we have $C_1/n\leq (1/2-\theta)$. In this case we write 
\begin{align}
\eqref{tmp0005001}\leq\frac{1}{\pi \theta} \left(\frac{2n}{C_1} + \frac{2n}{C_1}\right) \frac{C_1}{n}\leq \frac{16}{\pi}
\end{align}

Finally, when $C_1/n\leq 1/2-\theta$, ~\eqref{tmp0005002} can be bounded as 
\begin{align}
\eqref{tmp0005002}&\leq \frac{1}{\pi \theta} \left(\log(\frac{1/2+\tau - \theta}{\tau + \frac{C_1}{n}}) + \log(\frac{1/2-\tau}{\frac{C_1}{n}})\right) \frac{1}{(1/2-\theta) - \frac{C_1}{n}} \frac{C_1}{n}\\
&\leq \frac{1}{\pi \theta}\left(\log(1+ \frac{1/2 - \theta - \frac{C_1}{n}}{\tau+\frac{C_1}{n}}) + \log(1+ \frac{1/2-\tau - \frac{C_1}{n}}{\frac{C_1}{n}})\right)\frac{1}{(1/2-\theta) - \frac{C_1}{n}}\frac{C_1}{n}\\
&\leq \frac{8}{\pi}
\end{align}

For the second supremum, we set $\theta\geq \frac{1}{4}$ and simply use 

\begin{align}
\eqref{CorrectionFD3minusFourthTermCentral3}&\leq \sup_{\tau} \frac{1}{\pi \theta} \left[-\log(\frac{1/2-\theta}{1/2-\theta + \tau}) + \log(\frac{1/n + \tau + \frac{C_1-1}{n}}{C_1/n})\right]\frac{C_1}{n} \frac{1}{\frac{C_1}{n}+\tau + \frac{1}{2} - \theta}\\
&\leq \frac{1}{\pi \theta}\left|\log(\frac{1/2 - \theta + \tau}{C_1/n}) + \log(\frac{\tau + C_1/n}{1/2-\theta})\right|\frac{C_1}{n} \frac{1}{\frac{C_1}{n}+\tau + \frac{1}{2} - \theta}\\
&\leq \left[\left(\frac{1/2-\theta + \tau}{C_1/n} \vee \log(3C_1)\right)+ \left(\log(3C_1)\vee \frac{1/2-\theta}{C_1/n}\right)\right] \\
&\leq \frac{8}{\pi}\log(3C_1)
\end{align}

When $\theta\leq \frac{1}{4}$, we use 
\begin{align}
\eqref{CorrectionFD3minusFourthTermCentral3} \leq\frac{1}{\pi \theta} \int_{-\frac{1}{2}}^{-\frac{1}{2}+\theta} \frac{1}{-s}\; ds \leq \frac{1}{\pi \theta} \log(1+ \frac{\theta}{\frac{1}{2} - \theta}) \leq \frac{4}{\pi}.\label{ImagboundCorrectionFD3minusLast0001}
\end{align}

Combining the bounds derived above, we can write 
\begin{align}
Z&\leq  \frac{1}{4} \left(\frac{8}{\pi}  + \frac{4}{\pi}\log(C_1)\right) + \frac{8}{4\pi} + \frac{16}{\pi}\left(\log(4) + \log(C_1)\right)\\
& + 2\left(\log(4)+ \log(C_1)\right) \frac{8}{\pi} + \frac{8}{\pi} \left(\log(3) + \log(C_1)\right) + \frac{2\log(2)}{\pi}\\
&\leq \frac{4}{\pi} \left(\frac{2\log(2)}{4} + 1+ \frac{8}{\pi}\log(4) + \frac{2}{\pi}\log(3)\right) + \frac{41}{\pi}\log(C_1)\\
&\leq 7.15 +14\log(C_1).
\end{align}

Combining this with twice the upper bound~\eqref{boundRealPartD3minusSmallTauMinusAlpha} on the real part gives the result of lemma~\ref{lemmaTroncD3minusCons}.

\subsection{Proof of lemma~\ref{lemmaTroncD0minusCons} ($D_0^-$, small $|\tau - \alpha|$)}

Recall that when $\theta \geq \frac{1}{2n+3}$, the real part of the restriction of $F(s, \theta)$ (as illustrated in Fig.~\ref{domainDecomposition1}) to $D_0^-$ is bounded as 
\begin{align}
F_{R,D_0}^- &\leq \frac{1}{4(2n+3)} \left(\frac{1}{\pi} + \frac{1}{2}\right)\left[\frac{1}{-s}+\frac{1}{\theta-s}\right]+ \frac{\theta}{2} \leq \frac{1}{2(2n+3)} \left(\frac{1}{\pi} + \frac{1}{2}\right)\frac{1}{-s} + \frac{\theta}{2} 
\end{align}
The contribution of the $D_0^-$ to the integral~\eqref{eq:dev_int_p+} then reads as   
\begin{align}
Z& = \frac{1}{|e^{2\pi i \theta}-1|}\int_{-\frac{1}{2n+3}-2\frac{C_1-1}{n}}^{-\frac{1}{2n+3}} F_{R,D_0}^-(s)\; ds\label{partialZ0minusRealLargeThetatmp0001a}\\
& + \frac{1}{|e^{2\pi i \theta}-1|}\int_{-\frac{1}{2n+3}}^{-\frac{1}{2}+\theta} F_{R,D_0^-}(s)\frac{C_1}{1+n(-\frac{1}{2n+3}+\frac{C_1-1}{n} -s) + k'\Delta n}\; ds\label{partialZ0minusRealLargeThetatmp0002a}\\
&+ \frac{1}{|e^{2\pi i \theta}-1|} \sup_{-\frac{1}{2}+\theta\leq \tau\leq -\frac{1}{2n+3}} \int_{\tau}^{-\frac{1}{2n+3}} F_{R,D_0}^-(s) \frac{C_1}{1+n(s - \tau - \frac{C_1-1}{n}) + k'n\Delta}\; ds\label{partialZ0minusRealLargeThetatmp0003a}
\end{align}

We label each of the three integrals that appear in~\eqref{partialZ0minusRealLargeThetatmp0001a}, \eqref{partialZ0minusRealLargeThetatmp0002a} and \eqref{partialZ0minusRealLargeThetatmp0003a} as $Z_C$, $Z_R$ and $Z_L$ respectively. One can write  
\begin{align}
Z_C&\leq \frac{1}{\pi \theta} c \left|\log(\frac{\frac{1}{2n+3}}{\frac{1}{2n+3} + 2\frac{C_1-1}{n}})\right| + \frac{1}{2\pi}\frac{2(C_1-1)}{n} \leq  \frac{c'}{\pi}\left(\log(6)+\log(C_1)\right) + \frac{1}{2\pi}\label{partialZ0minusAD0mainD0minusRealLargeTheta}
\end{align}

\begin{align}
Z_R&\leq \frac{1}{|e^{2\pi i \theta}-1|}  \int_{-\frac{1}{2n+3}}^{-\frac{1}{2} + \theta} \left[\frac{c'}{2n+3}\frac{2}{-s} + \frac{\theta}{2}\right] \frac{C_1}{1+n\left(\frac{C_1-1}{n}-\frac{1}{2n+3} - s\right)}\; ds + \frac{C_1}{n}\frac{1}{2\pi}\\
&\leq \frac{c'}{\pi} \frac{C_1}{n}  \int_{-\frac{1}{2n+3}}^{-\frac{1}{2}+\theta}\left(\frac{1}{-s} - \frac{1}{\frac{1}{n}+\frac{C_1-1}{n}-\frac{1}{2n+3} -s + k'\Delta}\right)\frac{1}{\frac{C_1}{n} - \frac{1}{2n+3}}\; ds \\
&+ \int_{-\frac{1}{2n+3}}^{-\frac{1}{2}+\theta}\frac{C_1}{n}\frac{1}{2\pi} \left|\log(\frac{1}{n}+ \frac{C_1-1}{n}-\frac{1}{2n+3} -s) \right|_{-\frac{1}{2n+3}}^{-\frac{1}{2}+\theta}\\
&\leq \frac{c'}{\pi} \left(-\log(\frac{1/2-\theta}{\frac{1}{2n+3}}) + \log(\frac{\frac{C_1}{n} - \frac{1}{2n+3} + \frac{1}{2} - \theta}{\frac{C_1}{n} + k'\Delta})\right)\frac{1}{\frac{C_1}{n} - \frac{1}{2n+3} + k'\Delta}\\
& +  \frac{C_1}{n}\frac{1}{2\pi} \log(\frac{\frac{C_1}{n} - \frac{1}{2n+3} + \frac{1}{2} - \theta + k'\Delta}{\frac{C_1}{n} })\\
&\leq  \frac{c'}{\pi} \left(\log(3C_1 ) + \log(1+ \frac{C_1/n  - \frac{1}{2n+3}}{1/2-\theta})\right)  \frac{C_1}{n}+ \frac{C_1}{n} \frac{1}{2\pi} \sum_{k'=0}^S \frac{1/2-\theta}{\frac{C_1}{n} +k'\Delta}\\
&\leq \frac{c'}{\pi} \left(2(\log(3) + \log(C_1))\right)\frac{C_1}{C_1-1} \\
&+ \frac{c'}{\pi} \left(2\log(3) + \log(2) +2\log(C_1)\right)\frac{C_1}{n\Delta}\\
&+ \frac{C_1}{n} \frac{1}{4\pi} + \frac{C_1}{n\Delta} \frac{1}{2\pi}\\
&\leq \frac{c'}{\pi}2(\log(3)+\log(C_1)) + \frac{2c'}{\pi} (\log(3) + \log(2) + \log(C_1))+ \frac{3}{4\pi}\label{partialZ0minusBD0mainD0minusRealLargeTheta}
\end{align}
The last line holds as soon as $\Delta \geq \lambda_c C_1^2$. Finally if we let $\euscr{T}\equiv [\underline{\tau}, \overline{\tau}]$ with 
\begin{align}
\underline{\tau}& \equiv -\frac{1}{2n+3} - \frac{\Delta}{2} - \frac{C_1-1}{n}\\
\overline{\tau}&\equiv -\frac{1}{2n+3} -\frac{C_1-1}{n}
\end{align}
we can write the last correction as
\begin{align}
Z_L&\leq \sup_{\tau \in \euscr{T}} \frac{2c'}{\pi} \int_{\tau + \frac{C_1-1}{n}}^{-\frac{1}{2n+3}} \left(\frac{1}{-s} + \frac{1}{\frac{1}{n}+ (s-\tau) }\right) \frac{C_1}{n} \frac{1}{\frac{C_1}{n}-\tau }\; ds\\
& + \frac{1}{2\pi} \sup_{\tau \in \euscr{T}} \int_{\tau + \frac{C_1-1}{n}}^{-\frac{1}{2n+3}} \frac{C_1}{n} \frac{1}{\frac{1}{n}+ (s-\tau )}\; ds\\
& \leq\sup_{\tau \in \euscr{T}}   \frac{2c'}{\pi} \left( - \log(\frac{1/(2n+3)}{-\tau - \frac{C_1-1}{n}}) + \log(\frac{1/n - \frac{1}{2n+3} - \tau }{\frac{C_1}{n} })\right)\frac{C_1}{n}\frac{1}{\frac{C_1}{n}-\tau }\\
&+ \frac{1}{2\pi} \sup_{\tau \in \overline{\tau}}   \frac{C_1}{n} \log(\frac{1/n - \frac{1}{2n+3} - \tau }{C_1/n })\\
&\leq \frac{2c'}{\pi} \sup_{\tau \in \euscr{T}} \left[\log(\frac{(1/n - \frac{1}{2n+3} - \tau)(2n+3)}{C_1}) + \log(C_1) + \log(\frac{-\tau - \frac{C_1-1}{n}}{C_1/n})\right]\frac{C_1}{n} \frac{1}{C_1/n - \tau}\\
&+  \frac{2c'}{\pi} \frac{C_1}{n}\log(1+ \frac{n\Delta}{2C_1})\\
&\leq \frac{6c'}{\pi}+ \frac{c'}{\pi}\log(C_1) + \frac{c'}{\pi}\left( \log(3) + \log(C_1)\right). \\
&+ \frac{2c'}{\pi}\frac{C_1}{n\Delta}\left(\log(2)+ \log(\frac{n\Delta}{C_1})\right)\\
&\leq \frac{c'}{\pi} \left(8 + 2\log(2)+ \log(3)\right) + \frac{2c'}{\pi}\log(C_1). 
 \label{partialZ0minusCD0mainD0minusRealLargeTheta}
\end{align}
The last bound holds as soon as $\Delta \geq \lambda_c C_1$. 
Combining~\eqref{partialZ0minusAD0mainD0minusRealLargeTheta}, ~\eqref{partialZ0minusBD0mainD0minusRealLargeTheta} and~\eqref{partialZ0minusCD0mainD0minusRealLargeTheta}, we get 
\begin{align}
Z &\leq \frac{6c'}{\pi}\log(C_1) + \frac{2}{\pi} + \frac{c'}{\pi} \left(6\log(3) + 5\log(2) + 8\right)\\
&\stackrel{\eqref{definitioncprime}}{\leq}  0.4\log(C_1) + 2\label{lastBoundPartialZ0CentralDomainD0minus}
\end{align}

}

The bound for $\theta\leq \frac{1}{2n+3}$ on the real part of $F(s; \theta)$, on $D_0^-$, is given by 
\begin{align}
F_{R,D_0^-} = \frac{3}{2}\frac{\theta}{\theta-s} + \frac{\theta}{8(2n+3)s(s-\theta)} + \frac{\theta}{2} \leq \theta \left( \frac{3}{2} + \frac{1}{8}\right)\frac{1}{\theta-s} + \frac{\theta}{2}
\end{align}
This bound is decreasing away from $s = 0$, and the worst case configuration thus corresponds to concentrating the atom near $s = -\frac{1}{2n+3}$. 

As in the $\theta\geq \frac{1}{2n+3}$ regime, we control the integral through the following decomposition, whose three terms are respectively labeled $Z_C$, $Z_R$ and $Z_L$,
\begin{align}
Z& = \frac{1}{|e^{2\pi i \theta}-1|}\int_{-\frac{1}{2n+3}-2\frac{C_1-1}{n}}^{-\frac{1}{2n+3}} F_{R,D_0}^-(s)\; ds\label{partialZ0minusRealLargeThetatmp0001}\\
& + \frac{1}{|e^{2\pi i \theta}-1|}\int_{-\frac{1}{2n+3}}^{-\frac{1}{2}+\theta} F_{R,D_0^-}(s)\frac{C_1}{1+n(-\frac{1}{2n+3}+\frac{C_1-1}{n} -s) + k'\Delta n}\; ds\label{partialZ0minusRealLargeThetatmp0002}\\
&+ \frac{1}{|e^{2\pi i \theta}-1|} \sup_{-\frac{1}{2}+\theta\leq \tau\leq -\frac{1}{2n+3}} \int_{\tau}^{-\frac{1}{2n+3}} F_{R,D_0}^-(s) \frac{C_1}{1+n(s - \tau - \frac{C_1-1}{n}) + k'n\Delta}\; ds\label{partialZ0minusRealLargeThetatmp0003}
\end{align}

\begin{align}
Z_C&\leq  \frac{1}{|e^{2\pi i \theta}-1|} \left[\frac{3}{2}+ \frac{1}{8}\right] \theta \int_{-\frac{1}{2n+3} - 2\frac{C_1-1}{n}}^{-\frac{1}{2n+3}} \frac{1}{s-\theta}\;ds + \frac{\theta}{|e^{2\pi i \theta}-1|}\frac{1}{2} \int_{-\frac{1}{2n+3} - 2\frac{C_1-1}{n}}^{-\frac{1}{2n+3}} 1\; ds\\
&\leq \frac{1}{\pi}\left[\frac{3}{2}+ \frac{1}{8}\right] \left\{\log(\frac{\theta + \frac{1}{2n+3}}{\theta + \frac{1}{2n+3}+2\frac{C_1-1}{n}}) + \frac{1}{2\pi} 2\frac{C_1-1}{n}\right\}\\
&\leq \frac{1}{\pi}\left(\frac{3}{2}+ \frac{1}{8}\right)\log(6C_1) + \frac{1}{2\pi}\label{domainD0minusRealTermaCorrectionCentral}
\end{align}
For the second correction, we get 
\begin{align}
Z_R &= \frac{1}{|e^{2\pi i \theta}-1|} \int_{-\frac{1}{2}+\theta}^{-\frac{1}{2n+3}} \left(\frac{3}{2} + \frac{1}{8(2n+3)}\sup_{\sigma\in [-\frac{1}{2}+\theta, -\frac{1}{2n+3}]} \frac{1}{\sigma}\right) \theta \int_{-\frac{1}{2}+\theta}^{-\frac{1}{2n+3}} \left(\frac{1}{s-\theta} + \frac{\theta}{2}\right) \frac{C_1}{1+n(-\frac{1}{2n+3} + \frac{C_1-1}{n} - s)}\; ds\\
&\leq \frac{1}{\pi} \left(\frac{3}{2} + \frac{1}{8}\right) \int_{-\frac{1}{2}+\theta}^{-\frac{1}{2n+3}} \frac{1}{s-\theta} \frac{C_1}{1+n(-\frac{1}{2n+3} + \frac{C_1-1}{n} - s)} \; ds\\
&+ \frac{1}{2\pi} \frac{C_1}{n} \left|\log(\frac{C_1/n}{C_1/n - \frac{1}{2n+3} + 1/2 - \theta})\right|\\
&\leq \frac{1}{\pi} \left(\frac{3}{2}+ \frac{1}{8}\right) \int_{-\frac{1}{2}+\theta}^{-\frac{1}{2n+3}} \left(\frac{1}{\theta-s} - \frac{n}{1+n(-\frac{1}{2n+3} + \frac{C_1-1}{n} - s)}\right)\frac{C_1}{1-n\theta + n(\frac{C_1-1}{n} - \frac{1}{2n+3})}\; ds\\
&+ \frac{1}{2\pi} \frac{C_1}{n}\log(2)\\
&\leq \frac{1}{\pi}\left(\frac{3}{2}+ \frac{1}{8}\right) \left|-\log(\frac{\theta + \frac{1}{2n+3}}{1/2})+ \log(\frac{C_1/n}{C_1/n - \frac{1}{2n+3} + \frac{1}{2}-\theta})\right| \frac{C_1}{n}\frac{1}{|\frac{C_1}{n} - \theta - \frac{1}{2n+3}|}\\
&\leq \frac{\log(2)}{4}\label{repere1}\\
&+ \frac{1}{\pi}\left(\frac{3}{2}+ \frac{1}{8}\right) \left\{\log(1+ \frac{\frac{1}{2n+3}-\frac{C_1}{n} + \theta}{1/2-\theta - \frac{1}{2n+3} + \frac{C_1}{n}})\vee \log(1+ \frac{C_1/n - \frac{1}{2n+3} - \theta}{1/2})\right\} \frac{C_1}{n}\frac{1}{|\frac{C_1}{n} - \theta - \frac{1}{2n+3}|}\label{lineEnnuyante1FD0minusSMallTheta001}\\
&+ \frac{1}{\pi}\left(\frac{3}{2}+ \frac{1}{8}\right)\left\{\log(1+ \frac{C_1/n - \theta - \frac{1}{2n+3}}{\theta + \frac{1}{2n+3}}) \vee \log(1+ \frac{\theta + \frac{1}{2n+3} - \frac{C_1}{n}}{C_1/n})\right\} \frac{1}{\left|\theta + \frac{1}{2n+3} - \frac{C_1}{n}\right|} \frac{C_1}{n}\label{lineEnnuyante1FD0minusSMallTheta002}
\end{align}
We now bound~\eqref{lineEnnuyante1FD0minusSMallTheta001} and~\eqref{lineEnnuyante1FD0minusSMallTheta002} separately. For the first line, we have 
\begin{align}
\eqref{lineEnnuyante1FD0minusSMallTheta001} \leq \frac{C_1}{n} \left\{\frac{n}{C_1-1} + 2\right\}\leq 4
\end{align}
For the second line, to bound the first log, we consider the following distinction: either $C_1/n \geq 2\left( \theta + \frac{1}{2n+3}\right)$, or $\frac{C_1}{n}\leq 2(\theta + \frac{1}{2n+3})$. In the former case, we write,
\begin{align}
\log(1+ \frac{\frac{C_1}{n} - \theta - \frac{1}{2n+3}}{\theta + \frac{1}{2n+3}}) \frac{C_1}{n}\frac{1}{|\theta + \frac{1}{2n+3} - \frac{C_1}{n}|} \leq \frac{C_1}{n}\log(1+3C_1) \frac{2}{C_1/n}\leq 2\log(4C_1)
\end{align}
Whenever $\frac{C_1}{n}\leq 2(\theta + \frac{1}{2n+3})$, we use
\begin{align}
&\log(1+ \frac{C_1/n - \theta - \frac{1}{2n+3}}{\theta + \frac{1}{2n+3}}) \frac{1}{|\theta + \frac{1}{2n+3} - \frac{C_1}{n}|}\\
&\leq \frac{C_1}{n} \frac{C_1/n - \theta - \frac{1}{2n+3}}{\theta + \frac{1}{2n+3}} \frac{1}{\theta + \frac{1}{2n+3} - \frac{C_1}{n}}\\
&\leq \frac{C_1}{n} \frac{2n}{C_1}\leq 2
\end{align}
Substituting this into~\eqref{repere1} to~\eqref{lineEnnuyante1FD0minusSMallTheta002}, we get 
\begin{align}
Z_R\leq  \frac{\log(2)}{4} + \frac{1}{\pi}\left(\frac{3}{2}+ \frac{1}{8}\right) \left(4 + 4\log(2) +2+ 2\log(C_1)\right)\label{domainD0minusRealTermbCorrectionCentral}
\end{align}

Finally for the supremum, we write 
\begin{align}
Z_L& = \frac{1}{|e^{2\pi i \theta}-1|} \sup_{0\leq \tau \leq \frac{1}{2}-\theta}  \left(\frac{3}{2} + \frac{1}{8(2n+3)} \sup_{\sigma\in [-\frac{1}{2n+3}-\tau, -\frac{1}{2n+3}]} \frac{1}{-\sigma}\right)\int_{-\frac{1}{2n+3}-\tau}^{-\frac{1}{2n+3}} \frac{1}{s-\theta} \frac{C_1}{1+n(s-(-\frac{1}{2n+3} - \tau - \frac{C_1-1}{n}))} \; ds\\
&\leq \frac{1}{\pi}\left(\frac{3}{2}+ \frac{1}{8}\right) \sup_{\tau \in [0,\frac{1}{2}-\theta]} \int_{-\frac{1}{2n+3}-\tau}^{-\frac{1}{2n+3}}  \left(\frac{1}{\theta-s} + \frac{n}{1+n(s-(-\frac{1}{2n+3} - \tau - \frac{C_1-1}{n}))}\right) \frac{C_1}{1+n\theta + (\frac{1}{2n+3} + \tau + \frac{C_1-1}{n})n}\; ds\\
&+ \frac{1}{2\pi} \sup_{0\leq \tau \leq \frac{1}{2}-\theta} \frac{C_1}{n} \log(\frac{\tau+ \frac{C_1}{n}}{C_1/n})\\
&\leq \frac{1}{\pi}\left(\frac{3}{2}+ \frac{1}{8}\right) \sup_{0\leq \tau \leq \frac{1}{2}-\theta} \left|\log(\frac{\frac{1}{2n+3}+\tau+\theta}{\theta + \frac{1}{2n+3}}) + \log(\frac{C_1/n + \tau}{C_1/n})\right| \frac{C_1/n}{\theta + \frac{1}{2n+3} + \tau + \frac{C_1}{n}}+ \frac{1}{4\pi}\\
&\leq \frac{1}{\pi}\left(\frac{3}{2}+ \frac{1}{8}\right) \sup_{\tau} \log(\frac{\theta+ \frac{1}{2n+3}+\tau}{C_1/n} \frac{C_1/n}{\theta+\frac{1}{2n+3}})\\
& +\frac{1}{\pi}\left(\frac{3}{2}+ \frac{1}{8}\right) + \frac{1}{4\pi}\\
&\leq \frac{1}{\pi}\left(\frac{3}{2}+ \frac{1}{8}\right)\left\{\log(\frac{\theta+\frac{1}{2n+3} + \tau}{C_1/n}\vee \frac{C_1/n}{\theta+\frac{1}{2n+3} + \tau}) + \log(\frac{C_1/n}{\theta+\frac{1}{2n+3}}\vee \frac{\theta+\frac{1}{2n+3}}{C_1/n}) \right\} \frac{C_1/n}{\theta + \frac{1}{2n+3} + \tau + \frac{C_1}{n}}\\
& + \frac{1}{\pi}\left(\frac{3}{2}+ \frac{1}{8}\right) + \frac{1}{4\pi}\\
&\leq  2\log(3C_1) + 2+ \frac{1}{\pi}\left(\frac{3}{2}+ \frac{1}{8}\right) + \frac{1}{4\pi}\label{domainD0minusRealTermcCorrectionCentral}
\end{align}

Grouping~\eqref{domainD0minusRealTermcCorrectionCentral},~\eqref{domainD0minusRealTermbCorrectionCentral} and~\eqref{domainD0minusRealTermaCorrectionCentral}, we get the correction
\begin{align}
Z&\leq  8.6 + 3.6\log(C_1)\label{boundSmallThetaD0minusReal}
\end{align}
Taking the maximum of~\eqref{boundSmallThetaD0minusReal} and~\eqref{lastBoundPartialZ0CentralDomainD0minus}, we can bound the real part on both domains as $Z\leq  8.6 + 3.6\log(C_1)$. To conclude on $D_0^-$, we control the imaginary part. The additional contribution appearing in the imaginary part is the same when $\theta \geq \frac{1}{2n+3}$ and when $\theta \leq \frac{1}{2n+3}$. We can once again decompose the contribbution into the following three terms
\begin{align}
Z &= \frac{1}{|e^{2\pi i \theta}-1|} \int_{-\frac{1}{2n+3} - 2\frac{C_1-1}{n}}^{-\frac{1}{2n+3}} \log(1+\frac{\theta}{-s}) \; ds \\
&+ \frac{1}{|e^{2\pi i \theta}-1|} \int_{-\frac{1}{2}+\theta}^{-\frac{1}{2n+3}} \log(1+\frac{\theta}{-s})\frac{C_1}{1+n(-\frac{1}{2n+3} + \frac{C_1-1}{n}-s)}\; ds\\
&+ \frac{1}{|e^{2\pi i \theta}-1|}\sup_{-\frac{1}{2}+\theta\leq \tau \leq -\frac{1}{2n+3}} \int_{\tau}^{-\frac{1}{2n+3}} \log(1+ \frac{\theta}{-s}) \frac{C_1}{1+(s - (\tau - \frac{C_1-1}{n}))}\; ds
\end{align}
We use $Z_C$, $Z_R$ and $Z_L$ to denote each of those terms. Each of those terms can be controled as follows 
\begin{align}
Z_C = \frac{1}{|e^{2\pi i \theta}-1|}\int_{-\frac{1}{2n+3} - 2\frac{C_1-1}{n}}^{-\frac{1}{2n+3}} \log(1+ \frac{\theta}{-s})\; ds &\leq \frac{1}{4\pi} \log(1+2\frac{C_1-1}{n}\frac{1}{\frac{1}{2n+3}})\\
&\leq \frac{1}{4\pi}\left(\log(6) + \log(C_1)\right)\label{boundK1FD0minusImag}
\end{align}
\begin{align}
Z_R& = \frac{1}{|e^{2\pi i \theta}-1|\int_{-\frac{1}{2}+\theta}}^{-\frac{1}{2n+3}} \log(1+\frac{\theta}{-s}) \frac{C_1}{1+n(-\frac{1}{2n+3} + \frac{C_1-1}{n}-s)}\; ds\\
&\leq  \frac{1}{4\pi} \int_{-\frac{1}{2}+\theta}^{-\frac{1}{2n+3}} \left\{\frac{1}{-s} - \frac{n}{1+n(-\frac{1}{2n+3} + \frac{C_1-1}{n}-s)}\right\} \frac{1}{\frac{C_1-2}{n} - \frac{1}{2n+3}}\\
&\leq  \frac{C_1/n}{\frac{C_1-3}{n}} \frac{1}{4\pi} \left\{-\log(\frac{\frac{1}{2n+3}}{1/2-\theta}) + \log(\frac{C_1}{1+n(-\frac{1}{2n+3} + \frac{1}{2}-\theta + \frac{C_1-1}{n})})\right\}\\
&\leq \frac{1}{4\pi}\frac{C_1}{C_1-3} \left\{\log(3C_1) = \log(1+ \frac{C_1/n}{1/2-\theta})\right\}\\
&\leq \frac{1}{4\pi}\frac{C_1}{C_1-3}\left\{\log(3C_1) + \log(3C_1+1)\right\}\\
&\leq \frac{1}{\pi} \left(\log(4)+ \log(C_1)\right) \label{boundK2FD0minusImag}
\end{align}

The last line holds as soon as $C_1\geq 4$. Finally, for the supremum, we write 
\begin{align}
Z_L& = \frac{1}{|e^{2\pi i \theta}-1|} \sup_{-\frac{1}{2}+\theta \leq \tau \leq -\frac{1}{2n+3}} \int_{\tau}^{-\frac{1}{2n+3}} \frac{1}{4}\log(1+ \frac{\theta}{-s}) \frac{C_1}{1+n(s - (\tau - \frac{C_1-1}{n}))}\; ds\\
&\leq \frac{1}{4\pi} \sup_{-\frac{1}{2}+ \theta\leq \tau \leq -\frac{1}{2n+3}} \int_{\tau}^{-\frac{1}{2n+3}} \left(\frac{1}{-s} + \frac{n}{1+n(s - (\tau - \frac{C_1-1}{n}))}\right) \frac{1}{1+n(-\tau + \frac{C_1-1}{n})}\; ds\\
&\leq \frac{1}{4\pi} \sup_{-\frac{1}{2}+\theta \leq \tau \leq -\frac{1}{2n+3}} \left\{-\log(\frac{\frac{1}{2n+3}}{-\tau}) + \log(\frac{1/n - \frac{1}{2n+3} - \tau +\frac{C_1-1}{n}}{C_1/n})\right\} \frac{C_1}{1+n(-\tau + \frac{C_1-1}{n})}\\
&\leq \frac{1}{4\pi} \sup_{-\frac{1}{2}+ \theta\leq \tau \leq -\frac{1}{2n+3}}\left\{\log(\frac{C_1}{n(-\tau)}) \vee \log(\frac{-\tau}{\frac{C_1}{n}})\right\}\frac{C_1}{n(-\tau) + C_1}\\
&+ \frac{1}{4\pi} \sup_{-\frac{1}{2}+\theta\leq \tau \leq -\frac{1}{2n+3}} \log(\frac{1/n-\frac{1}{2n+3} - \tau + \frac{C_1-1}{n}}{\frac{1}{2n+3}})\\
&\leq \frac{1}{4\pi}\sup_{-\frac{1}{2}+\theta\leq \tau \leq -\frac{1}{2n+3}} \left\{\log(3C_1) +1\right\} + \frac{1}{4\pi} \sup_{-\frac{1}{2}+\theta\leq \tau \leq -\frac{1}{2n+3}} \log(\frac{C_1-1}{n}(2n+3)) \frac{C_1}{n(-\tau)+C_1}\\
&\leq \frac{1}{4\pi} \left\{\log(3C_1) + 1\right\} + \frac{1}{4\pi} \log(3C_1)\\
&\leq \frac{1}{2\pi} \left\{\log(3)+\log(C_1)\right\} + \frac{1}{4\pi}\label{boundK3FD0minusImag}
\end{align}
The total bound on the imaginary part is then given by
\begin{align}
Z&\leq \eqref{boundK1FD0minusImag} + \eqref{boundK2FD0minusImag} +   \eqref{boundK3FD0minusImag} \leq 0.8 + 0.7\log(C_1)\label{boundImaginaryPartD0minusSmallTauMinusTheta}
\end{align}

Now grouping this bound with twice the bound on the real part, we get the contribution for $|\tau-\alpha|<1/n$ for the subdomain $D_0^-$, 
\begin{align}
Z&\leq 2\eqref{boundSmallThetaD0minusReal} + \eqref{boundImaginaryPartD0minusSmallTauMinusTheta}\leq 18 + 8\log(C_1).
\end{align}

\subsection{Proof of lemma~\ref{lemmaTroncD0plusCons} ($D_0^+$, small $|\tau - \alpha|$)}

We now deal with $F_{D_0}^+$. Recall that we have 
\begin{align}
F_{D_0}^+ \leq \left(\frac{\pi}{2}+ \frac{1}{4(2n+3)}\left(\frac{1}{2}+ \frac{1}{\pi}\right) \left(2(2n+3) + \frac{1}{s} + \frac{1}{\theta-s}\right)\right)+ \frac{\theta}{2} \; ds
\end{align}
as well as 
\begin{align}
 F_{I,D_0}^+& \leq \frac{1}{4(2n+3)} \left(\frac{1}{\pi}+ \frac{1}{2}\right) \left(\frac{1}{s}+ \frac{1}{\theta-s}\right) + \log(\frac{s}{\theta-s} \vee \frac{\theta-s}{s})\\
&\leq F_{R,D_0}^+ + \log(\frac{s}{\theta-s} \vee \frac{\theta-s}{s})\label{boundFromRealToImagFD0plus}
\end{align}
The domain of integration in this case has length $\theta - \frac{1}{2n+3}$. For the real part, we thus simply take the supermum of $p(e^{2\pi i \theta})$ over this interval,
\begin{align}
&\frac{1}{|1-e^{2\pi i \theta}|}\int_{s\in D_0^+} p(e^{2\pi i \theta}) F_{D_0^+}(s\;\theta) \; ds \\
&\leq \sup_{s \in [\frac{1}{2n+3}, \theta - \frac{1}{2n+3}]} \frac{1}{\pi \theta}\left[ \left(\frac{\pi}{2}+ \frac{1}{4(2n+3)}\left(\frac{1}{2}+ \frac{1}{\pi}\right) \left(2(2n+3) + \frac{1}{s} + \frac{1}{\theta-s}\right)\right)+ \frac{\theta}{2} \right] (\theta - \frac{1}{2n+3})\\
&\leq \frac{1}{2} + \frac{1}{\pi}\left(\frac{3}{2} + \frac{1}{\pi}\right)\label{boundRealPartD0plusLemmaTroncNear}
\end{align}

For the imaginary part, we have 
\begin{align}
\frac{1}{|e^{2\pi i \theta} - 1|}\int_{s\in D_0^+} p(e^{2\pi i \theta}) \log(\frac{s}{\theta - s}\vee \frac{\theta-s}{s}) \; ds & = \frac{1}{|e^{2\pi i \theta}-1|} \int_{\frac{1}{2n+3}}^{\theta/2} p(e^{2\pi i \theta}) \log(\frac{\theta-s}{s}) \; ds\\
& + \frac{1}{|e^{2\pi i \theta}-1|}\int_{\theta/2}^{\theta - \frac{1}{2n+3}} p(e^{2\pi i \theta}) \log(\frac{s}{\theta-s})\; ds
\end{align}
We then use $Z$ and $W$ to denote each of those terms. For both terms, we consider the following decomposition,
\begin{align}
Z&\leq  \frac{1}{|e^{2\pi i \theta} - 1|} \int_{\frac{1}{2n+3}}^{\frac{1}{2n+3}+2\frac{C_1-1}{n}} \log(\frac{\theta-s}{s})\; ds\\
&+ \frac{1}{|e^{2\pi i \theta}-1|} \int_{\frac{1}{2n+3}}^{\theta/2} \log(\frac{\theta-s}{s})\frac{C_1}{1+n(s- (\frac{1}{2n+3} - \frac{C_1-1}{n}))}\; ds\\
&+ \frac{1}{|e^{2\pi i \theta} - 1|} \sup_{0\leq \tau \leq \theta/2} \int_{\frac{1}{2n+3}}^{\frac{1}{2n+3} + \tau} \log(\frac{\theta-s}{s}) \frac{C_1}{1+n\left(\frac{1}{2n+3} + \tau + \frac{C_1-1}{n} -s\right)} \; ds
\end{align}
as well as 
\begin{align}
W& = \frac{1}{|e^{2\pi i \theta}-1|} \int_{\theta - \frac{1}{2n+3} - 2\frac{C_1-1}{n}}^{\theta - \frac{1}{2n+3}} \log(\frac{s}{\theta-s})\; ds \\
&+ \frac{1}{|e^{2\pi i \theta} - 1|} \int_{\theta/2}^{\theta - \frac{1}{2n+3}} \log(\frac{s}{\theta-s}) \frac{C_1}{1+n(\theta-\frac{1}{2n+3}-s)}\; ds\\
&+ \frac{1}{|e^{2\pi i \theta}-1|} \sup_{0\leq \tau \leq\theta/2} \int_{\theta - \frac{1}{2n+3} - \tau}^{\theta - \frac{1}{2n+3}} \log(\frac{s}{\theta-s}) \frac{C_1}{1+n(s- \left(\theta - \frac{1}{2n+3} - \tau - \frac{C_1-1}{n}\right))}\; ds
\end{align}
Starting with $Z$, using $\log(\frac{\theta-s}{s})\leq \log(\theta/s)\leq \theta/s$, we have 
\begin{align}
Z&\leq \frac{1}{|e^{2\pi i \theta}-1|} \int_{\frac{1}{2n+3}}^{\frac{1}{2n+3} + 2\frac{C_1-1}{n}} \log(\frac{\theta-s}{s})\; ds \\
&+ \frac{1}{|e^{2\pi i \theta}-1|}\int_{\frac{1}{2n+3}}^{\theta/2} \log(\frac{\theta-s}{s}) \frac{C_1}{1+n(s- (\frac{1}{2n+3} - \frac{C_1-1}{n}))} \; ds\\
&+ \frac{1}{|e^{2\pi i \theta}-1|} \sup_{0\leq \tau \leq \theta/2} \int_{\frac{1}{2n+3}}^{\frac{1}{2n+3} + \tau} \log(\frac{\theta - s}{s}) \frac{C_1}{1+n(\frac{1}{2n+3} + \tau + \frac{C_1-1}{n}-s)}\; ds\\
& \leq \frac{1}{\pi} \log(1+ 2\frac{C_1-1}{n}(2n+3))+ \frac{1}{\pi} \int_{\frac{1}{2n+3}}^{\theta/2} \left(\frac{1}{s} + \frac{-1}{1/n + s - \frac{1}{2n+3} + \frac{C_1-1}{n}}\right) \frac{C_1/n}{\frac{C_1-1}{n}} \; ds\\
&+ \frac{1}{\pi} \sup_{0\leq \tau \leq \theta/2} \int_{\frac{1}{2n+3}}^{\frac{1}{2n+3}+\tau} \left(\frac{1}{s} + \frac{1}{\frac{C_1}{n} + \frac{1}{2n+3} + \tau - s}\right) \frac{C_1/n}{\frac{C_1}{n} + \frac{1}{2n+3} + \tau}\; ds\\
&\leq \frac{1}{\pi} \left(\log(6) + \log(C_1)\right)+ \frac{1}{\pi} \left[\log(\frac{\theta/2}{\frac{1}{2n+3}}) - \log(\frac{\theta/2 - \frac{1}{2n+3} + \frac{C_1}{n}}{C_1/n})\right] \frac{C_1}{C_1-1}\\
&+ \frac{1}{\pi} \sup_{0\leq \tau \leq \frac{\theta}{2}} \left|\log(\frac{\frac{1}{2n+3} + \tau}{\frac{1}{2n+3}}) - \log(\frac{C_1/n}{C_1/n + \tau})\right| \frac{C_1/n}{C_1/n + \frac{1}{2n+3} + \tau}\\
&\leq \frac{1}{\pi} \left(\log(6)+ \log(C_1)\right)+ \frac{1}{\pi} \left|\log(1+ \frac{C_1/n}{\theta/2}) + \log(\frac{C_1}{n}(2n+3))\right| \frac{C_1}{C_1-1}\\
&+\frac{1}{\pi}\sup_{0\leq \tau\leq \theta/2} \left|\log(3n(\frac{1}{2n+3} + \tau)) + \log(\frac{C_1/n + \tau}{C_1/n})\right| \frac{C_1/n}{\frac{C_1}{n} + \frac{1}{2n+3} + \tau}\\
&\leq \frac{1}{\pi} \left(\log(6) + \log(C_1)\right) + \frac{4}{\pi}\log(4C_1)  \\
&+ \frac{1}{\pi}\sup_{0\leq \tau \leq \theta/2} \left|\log(\frac{n}{C_1} \left(\frac{1}{2n+3}+\tau\right)) + \log(3)+ \log(C_1) + \log(\frac{C_1/n+\tau}{C_1/n})\right| \frac{C_1/n}{C_1/n + \frac{1}{2n+3} + \tau}\\
&\leq \frac{1}{\pi} \left(\log(6)+ \log(C_1)\right)+ \frac{2}{\pi} \left(2\log(4)+ 2\log(C_1)\right) + \frac{1}{\pi} \left(2+ \log(3)+ \log(C_1)\right). \label{tmpExplanationD0plusImag001}
\end{align}
The bound~\eqref{tmpExplanationD0plusImag001} holds as soon as $C_1\geq 2$ which implies $C_1/(C_1-1)\leq 2$ and $\Delta \geq 4(C_1-1)\lambda_c$. Together those last two lines can thus be made less than 
\begin{align}
Z&\leq  \frac{1}{\pi} \left(\log(6)+ 4\log(4)+ 2+ \log(3)\right) + \frac{6}{\pi}\log(C_1)\\
& = 3.5 + 2\log(C_1). \label{boundZD0plusSmallLemmaTronc}
\end{align}
For $W$, we get 
\begin{align}
W & = \frac{1}{|e^{2\pi i \theta}-1|} \int_{\theta - \frac{1}{2n+3} - 2\frac{C_1-1}{n}}^{\theta - \frac{1}{2n+3}} \log(\frac{s}{\theta-s})\; ds\\
&+ \frac{1}{|e^{2\pi i \theta}-1|} \int_{\theta/2}^{\theta - \frac{1}{2n+3}} \log(\frac{s}{\theta-s}) \frac{C_1}{1+n(\theta - \frac{1}{2n+3} - s)}\; ds\\
&+ \frac{1}{|e^{2\pi i \theta}-1|} \sup_{0\leq \tau \leq \frac{\theta}{2}} \int_{\theta - \frac{1}{2n+3} - \tau}^{\theta - \frac{1}{2n+3}} \log(\frac{s}{\theta-s}) \frac{C_1}{1+n(s - (\theta - \frac{1}{2n+3} - \tau - \frac{C_1-1}{n}))} \; ds\\
&\leq \frac{1}{\pi} \log(\frac{2\frac{C_1-1}{n} + \frac{1}{2n+3}}{\frac{1}{2n+3}})+ \frac{1}{\pi} \int_{\theta/2}^{\theta - \frac{1}{2n+3}} \left(\frac{1}{\theta-s} + \frac{-1}{\frac{1}{n} + \left(\theta - \frac{1}{2n+3} - s + \frac{C_1-1}{n}\right)}\right) \frac{C_1}{C_1-1}\; ds\\
&+ \frac{1}{\pi} \sup_{0 \leq \tau \leq \theta/2} \int_{\theta - \frac{1}{2n+3} - \tau}^{\theta - \frac{1}{2n+3}} \left(\frac{1}{\theta-s} + \frac{1}{\frac{1}{n} + \left(s - (\theta - \frac{1}{2n+3} - \tau - \frac{C_1-1}{n})\right)}\right) \frac{C_1/n}{\frac{C_1}{n} - \frac{1}{2n+3} + \tau}\; ds\\
&\leq \frac{1}{\pi} \left(\log(6)+ \log(C_1)\right)+ \frac{1}{\pi} \left|-\log(\frac{\frac{1}{2n+3}}{\theta/2}) + \log(\frac{C_1/n}{C_1/n + \frac{\theta}{2} - \frac{1}{2n+3}})\right| \frac{C_1}{C_1-1}\\
&+ \frac{1}{\pi} \sup_{0\leq \tau \leq \theta/2} \left|\log(\frac{\frac{1}{2n+3}}{\frac{1}{2n+3}+\tau}) + \log(\frac{\tau + \frac{C_1}{n}}{C_1/n})\right| \frac{C_1/n}{\frac{C_1-1}{n}+\tau}\\
&\leq \frac{1}{\pi} \left(\log(6)+ \log(C_1)\right)+ \frac{1}{\pi} \left(\log(3C_1) + \log(1+ \frac{C_1}{n} \frac{2}{\theta})\right)\\
&+ \frac{1}{\pi} \sup_{0\leq \tau \leq \theta/2} \left(\log(3C_1) + \log(1+ \frac{C_1/n}{\tau + \frac{1}{2n+3}})\right) \frac{C_1/n}{\frac{C_1-1}{n} + \tau}\\
&\stackrel{(a)}{\leq } \frac{1}{\pi} \left(\log(6)+ \log(C_1)\right)+ \frac{2}{\pi} \left(\log(3)+ 2\log(C_1) + \log(7)\right)\\
&+ \frac{2}{\pi} \left(\log(3)+2 \log(C_1) + \log(4)\right)
\end{align}
$(a)$ holds as soon as $C_1\geq 2$. From those lines, we have 
\begin{align}
W& \leq \frac{1}{\pi} \left(\log(6)+ 4\log(3)+ 2\log(7) + \log(4)\right) + \frac{9}{\pi}\log(C_1)\\
& \leq 3.7 + 3\log(C_1).\label{boundWD0plusSmallLemmaTronc}
\end{align}
Combining~\eqref{boundWD0plusSmallLemmaTronc}, ~\eqref{boundZD0plusSmallLemmaTronc} as well as~\eqref{boundRealPartD0plusLemmaTroncNear} gives the result of the lemma. 

\subsection{Proof of lemma~\ref{lemmaTroncD4plusCons}, $D_4^+$, small $|\tau - \alpha|$}

We start by controling the real part (for both the $\theta\geq \frac{1}{2n+3}$ as for the $\theta \leq \frac{1}{2n+3}$ regimes). We then show how to control the imaginary part. 

On $D_4^+$, in the large $\theta$ regime, we have the bound
\begin{align}
F_{R, D_4}^+&\leq \frac{\theta}{2} + \left(\frac{1}{2} + \frac{1}{\pi}\right)\frac{1}{4(2n+3)} \left(\frac{1}{s} + \frac{1}{s-\theta}\right)\label{reminderBoundFD4plusReal}
\end{align}

The ``small $\theta$" bound on $D_4^+$ is given by 
\begin{align}
F_{R,D_4}^+ &\leq \frac{3}{2}\frac{\theta}{|\theta-s|} + \frac{\theta}{8(2n+3)s|s-\theta|} + \frac{\theta}{2} 
\end{align}

In the ``large $\theta$" regime, we control the integral through the upper bound $Z$ which we split into the following three contributions 
\begin{align}
Z & = Z_C + Z_R + Z_L\label{totalCorrectionDomainD4plusRealLargeThetaCentral}\\
& = \int_{\theta + \frac{1}{2n+3}}^{\theta + \frac{1}{2n+3} + 2\frac{C_1-1}{n}} 1\cdot \left(\frac{\theta}{2} +c\left(\frac{1}{s} + \frac{1}{s-\theta}\right)\right)\; ds\label{part1CorrectionCentralD4minus}\\
&+ \int_{\theta+\frac{1}{2n+3}}^{1/2} \frac{C_1}{1+n(s - (\theta + \frac{1}{2n+3} - \frac{C_1-1}{n}))} \left\{\frac{\theta}{2} + c\left(\frac{1}{s} + \frac{1}{s-\theta}\right)\right\}\; ds\label{part2CorrectionCentralD4minus}\\
&+ \sup_{0\leq \tau \leq \frac{1}{2}} \frac{1}{\pi \theta} \int_{\theta + \frac{1}{2n+3}}^{\theta + \frac{1}{2n+3} + \tau +} \left\{\frac{\theta}{2} + c \left(\frac{1}{s} + \frac{1}{s-\theta}\right)\right\} \frac{C_1}{1+n(\theta + \frac{1}{2n+3} + \tau +\frac{C_1-1}{n} -s )}\; ds\label{part3CorrectionCentralD4minus}
\end{align}
We now control each of those terms. Starting with $Z_C$, we can write
\begin{align}
Z_C& = \frac{1}{\pi \theta} \int_{\theta + \frac{1}{2n+3}}^{\theta + \frac{1}{2n+3} + 2\frac{C_1-1}{n}} \frac{\theta}{2} + c\left(\frac{1}{s} + \frac{1}{s-\theta}\right) \; ds\\
&\leq\frac{1}{2\pi} \frac{2(C_1-1)}{n} + \frac{c'}{\pi}\left[\log(1+ 2\frac{C_1-1}{n}\frac{1}{\theta + \frac{1}{2n+3}})+ \log(1+ 2\frac{C_1-1}{n} (2n+3))\right]\\
&\leq \frac{1}{2\pi} \frac{2(C_1-1)}{n} + \frac{c'}{\pi}2\log(1+6(C_1-1))\label{partialZ4plusaDomainD4plusRealLargeTheta}
\end{align}
For $Z_R$, we have 
\begin{align}
Z_R& = \frac{1}{\pi \theta} \int_{\theta+\frac{1}{2n+3}}^{1/2} c\left(\frac{1}{s}+ \frac{1}{s-\theta}\right) \frac{C_1}{1+n(s - (\theta + \frac{1}{2n+3} - \frac{C_1-1}{n}) ) }\; ds\\
& + \frac{1}{\pi \theta}\int_{\theta+\frac{1}{2n+3}}^{1/2} \frac{\theta}{2} \frac{C_1}{1+n(s - (\theta+\frac{1}{2n+3} - \frac{C_1-1}{n}) )}\; ds\\
& = \frac{1}{\pi \theta}\int_{\theta+\frac{1}{2n+3}}^{1/2} \frac{c}{s} \frac{C_1}{1+n(s - (\theta + \frac{1}{2n+3} - \frac{C_1-1}{n} ))}\; ds\\
& + \frac{1}{\pi \theta} \int_{\theta+\frac{1}{2n+3}}^{1/2}\frac{c}{s-\theta} \frac{C_1}{1+n(s - (\theta + \frac{1}{2n+3} - \frac{C_1-1}{n}))}\\
&+ \frac{1}{\pi \theta}\int_{\theta+\frac{1}{2n+3} }^{1/2} \frac{\theta}{2} \frac{C_1}{1+n(s - (\theta+\frac{1}{2n+3} - \frac{C_1-1}{n}))}\\
& =H_{1} +  H_2 +  H_0.
\end{align}

First note that for $H_1$, we have
\begin{align} 
H_1& \leq \frac{c'}{\pi} \int_{\theta + \frac{1}{2n+3}}^{1/2} \left(\frac{1}{s} + \frac{-n}{1+n(s- (\theta + \frac{1}{2n+3} - \frac{C_1-1}{n}))}\right) \frac{C_1}{|1 - n(\theta + \frac{1}{2n+3} - \frac{C_1-1}{n})|}\; ds\\
&\leq \frac{c'}{\pi} \left|\log(\frac{1/2}{\theta + \frac{1}{2n+3}}) - \log(\frac{1/2 - \theta - \frac{1}{2n+3} + \frac{C_1}{n}}{C_1/n})\right| \frac{C_1/n}{|\frac{C_1}{n} - \theta - \frac{1}{2n+3}|}
\label{tmplemmaTroncD4plus0001}
\end{align}
To bound~\eqref{tmplemmaTroncD4plus0001}, we make the distinction betwee $\theta < \frac{C_1-1}{2n}$ and $\theta \geq \frac{C_1-1}{2n}$. In the former case, we have  
\begin{align}
H_1& \leq  \frac{c'}{\pi} \left|\log(\frac{1/2}{\theta + \frac{1}{2n+3}}) - \log(\frac{1/2 - \theta - \frac{1}{2n+3} + \frac{C_1}{n}}{C_1/n})\right| \frac{2C_1}{n}\frac{n}{C_1-1}\\
&\stackrel{(a)}{\leq} \frac{4c'}{\pi} \log(\frac{C_1/n}{2(\frac{1}{2n+3}\wedge \frac{C_1-1}{n})\frac{1}{4}}) \\
&\stackrel{(b)}{\leq} \frac{4c'}{\pi} \left(\log(6)+ \log(C_1)\right).   \label{boundH1aD4plusLemmaTroncLargetheta}
\end{align}
In $(b)$, we use $C_1\geq 2$. In $(a)$, we use 
\begin{align}
\inf_{0\leq \theta \leq 1/2} \left(\theta + \frac{1}{2n+3}\right) \left(\frac{1}{2} - \theta - \frac{1}{2n+3} + \frac{C_1}{n}\right) &\geq \frac{1}{2n+3}\left(\frac{1}{2} - \frac{1}{2n+3}+ \frac{C_1-1}{n}\right) \wedge \left(\frac{1}{2}+ \frac{1}{2n+3}\right) \frac{C_1-1}{n}\\
&\geq \left(\frac{1}{2n+3} \wedge \frac{C_1-1}{n}\right)\frac{1}{4}. 
\end{align}
as well as 
\begin{align}
\sup_{0\leq \theta \leq \frac{C_1-1}{2n}} \left(\theta + \frac{1}{2n+3}\right)\left(\frac{1}{2}  - \theta - \frac{1}{2n+3} + \frac{C_1}{n}\right) \leq \frac{C_1}{2n}
\end{align}
When $\theta \geq \frac{C_1-1}{2n}$, we write 
\begin{align}
H_1 &\stackrel{(a)}{\leq} \frac{c}{\pi \theta} \left|\log(\frac{1/2}{\theta + \frac{1}{2n+3}}) - \log(\frac{1/2 - \theta - \frac{1}{2n+3} + \frac{C_1}{n}}{C_1/n})\right| \frac{C_1/n}{|\frac{C_1}{n} - \theta - \frac{1}{2n+3}|}\\
&\stackrel{(b)}{\leq} \frac{2 nc}{\pi (C_1-1)} \left(\log(1+ \frac{\theta + \frac{1}{2n+3} - \frac{C_1}{n}}{1/2 - \theta - \frac{1}{2n+3} + \frac{C_1}{n}} \vee \frac{-\theta - \frac{1}{2n+3} + \frac{C_1}{n}}{1/2})\right) \frac{C_1/n}{C_1/n - \theta - \frac{1}{2n+3}}\\
&+ \frac{2nc}{\pi (C_1-1)} \left(\log(1+ \frac{\theta + \frac{1}{2n+3} - \frac{C_1}{n}}{C_1/n}\vee \frac{C_1/n - \theta - \frac{1}{2n+3}}{\theta + \frac{1}{2n+3}})\right) \frac{C_1/n}{C_1/n - \theta - \frac{1}{2n+3}}\\
&\stackrel{(c)}{\leq} \frac{c}{\pi} \frac{2n}{(C_1-1)} \frac{C_1}{n} \left(\frac{1}{1/2 - \theta - \frac{1}{2n+3} + \frac{C_1}{n}}\vee 2\right) + \frac{2 nc}{\pi} \frac{C_1}{C_1-1}\left(\frac{n}{C_1} \vee \frac{1}{\theta + \frac{1}{2n+3}}\right) \\
&\stackrel{(d)}{\leq} \frac{4c'}{\pi} \left(\frac{1}{C_1-1}\vee \frac{2}{2n+3}\right) + \frac{4c'}{\pi} \left(\frac{1}{C_1}\vee 1\right)\\
&\stackrel{(e)}{\leq} \frac{8c'}{\pi} \label{boundH1bD4plusLemmaTroncLargetheta}
%
\end{align}
$(e)$ holds as soon as $C_1\geq 2$. 
A similar reasoning applies to $H_2$. In this case, we have 
\begin{align}
H_2 &\leq \frac{c'}{\pi}  \int_{\theta + \frac{1}{2n+3}}^{1/2} \left(\frac{1}{s-\theta} + \frac{-n}{1+n(s - (\theta + \frac{1}{2n+3} - \frac{C_1-1}{n}) )}\right)\frac{C_1}{|1-n(\frac{1}{2n+3} - \frac{C_1-1}{n}) |}\\
&\leq \frac{c'}{\pi} \left|\log(\frac{1/2 - \theta}{\frac{1}{2n+3}}) - \log(\frac{1/2 - \theta + \frac{C_1}{n} - \frac{1}{2n+3}}{C_1/n })\right|\frac{C_1}{|1-n(\frac{1}{2n+3} - \frac{C_1-1}{n}) |}\\
&\leq \frac{c'}{\pi}\left|\log(\frac{C_1/n}{1/(2n+3)}) + \log(\frac{1/2 - \theta + \frac{C_1}{n} - \frac{1}{2n+3}}{1/2 - \theta})\right| \frac{C_1/n}{\frac{C_1-1}{n}}\\
&\leq \frac{c'}{\pi} \frac{C_1}{C_1-1}\left( \log(3) + \log(C_1)\right) + \frac{c'}{\pi} \frac{C_1}{C_1-1} \log(1 + \frac{C_1}{n}\frac{1}{1/2 - \theta})\\
&\leq \frac{4c'}{\pi} \left(\log(4) + \log(C_1)\right) \label{boundH2aD4plusLemmaTroncLargetheta}
\end{align}
The last line holds as soon as $C_1\geq 2$ and $\frac{C_1}{n\Delta}\leq 1$. 

For $H_0$, we write 
\begin{align}
H_0 & = \frac{1}{\pi \theta}\int_{\theta + \frac{1}{2n+3}}^{1/2} \frac{\theta}{2} \frac{C_1}{1+n(s - (\theta+\frac{1}{2n+3} - \frac{C_1-1}{n}))}\; ds\\
& = \frac{C_1}{n}\frac{1}{2\pi} \log(\frac{1/n + \frac{1}{2} - (\theta + \frac{1}{2n+3} - \frac{C_1-1}{n} )}{C_1/n })\\
&\leq \frac{1}{2\pi}  \frac{C_1}{n} \log(1+ \frac{1/2 - \theta - \frac{1}{2n+3}}{\frac{C_1}{n} })\\
&\leq \frac{C_1}{n}\frac{n}{C_1}(1/2-\theta)\label{boundH0D4plusLemmaTroncLargetheta}
\end{align}
Combining the results above, the integral~\eqref{part2CorrectionCentralD4minus} can be controled as 
\begin{align}
Z_R &\leq \left(\eqref{boundH1aD4plusLemmaTroncLargetheta}\vee \eqref{boundH1bD4plusLemmaTroncLargetheta}\right) + \eqref{boundH2aD4plusLemmaTroncLargetheta} + \eqref{boundH0D4plusLemmaTroncLargetheta} \\
& = \frac{1}{2} + \frac{4c'}{\pi} \log(4)+ \frac{8c'}{\pi} + \frac{8c'}{\pi}\log(C_1)\\
&\leq 1.4 + 0.6\log(C_1). \label{partialZ4plusaDomainD4plusRealLargeThetaZR}
\end{align}

We now control the last integral~\eqref{part3CorrectionCentralD4minus}, we consider the decomposition
\begin{align}
\sup_{0\leq \tau \leq 1/2 - \theta} \frac{1}{\pi \theta} \int_{\theta + \frac{1}{2n+3}}^{\theta + \frac{1}{2n+3} + \tau} \left\{\frac{\theta}{2} + c \left(\frac{1}{s} + \frac{1}{s-\theta}\right)\frac{C_1}{1+n(\theta + \frac{1}{2n+3} + \tau -s )}\right\}\; ds\leq R_0 + R_1 + R_2 
\end{align}
Where we respectively define $R_0$, $R_1$ and $R_2$ as
\begin{align}
R_0 &\equiv \sup_{0\leq \tau \leq 1/2 - \theta}\frac{1}{\pi \theta} \int_{\theta + \frac{1}{2n+3}}^{\theta + \frac{1}{2n+3} + \tau}\frac{\theta}{2} \frac{C_1}{1+n\left(\theta + \frac{1}{2n+3}+\tau + \frac{C_1-1}{n} - s \right)}\; ds\\
R_1 &\equiv \sup_{0\leq \tau \leq 1/2 - \theta}\frac{1}{\pi \theta} \int_{\theta + \frac{1}{2n+3}}^{\theta + \frac{1}{2n+3} + \tau}\frac{c}{s} \frac{C_1}{1+n\left(\theta + \frac{1}{2n+3}+\tau + \frac{C_1-1}{n} - s \right)}\; ds\\
R_2 &\equiv \sup_{0\leq \tau \leq 1/2 - \theta}\frac{1}{\pi \theta}\int_{\theta + \frac{1}{2n+3}}^{\theta + \frac{1}{2n+3} + \tau} \frac{c}{s-\theta} \frac{C_1}{1+n\left(\theta + \frac{1}{2n+3}+\tau + \frac{C_1-1}{n} - s \right)}\; ds
\end{align}
We then bound each term as follows. Starting with $\partial R_0$, we have 
\begin{align}
R_0 &\leq \sup_{0\leq \tau \leq 1/2 - \theta}\frac{1}{\pi \theta} \int_{\theta + \frac{1}{2n+3}}^{\theta + \frac{1}{2n+3} + \tau} \frac{\theta}{2} \frac{C_1}{1+n\left(\theta + \frac{1}{2n+3}+\tau + \frac{C_1-1}{n} - s \right)}\; ds\\
&\leq \sup_{0\leq \tau \leq 1/2 - \theta} \frac{1}{\pi \theta}\frac{\theta}{2} \frac{C_1}{n}\log(\frac{C_1/n }{C_1/n + \tau })\\
&\leq \frac{1}{\pi \theta}\sup_{0\leq \tau \leq 1/2 - \theta} \frac{\theta}{2} \frac{C_1}{n} \log(1+ \frac{\tau}{C_1/n})\\
&\leq \frac{1}{4\pi}\label{partialR0aD4plusRealLargeTheta}
\end{align}

For $R_1$, we have
\begin{align}
R_1& \leq \frac{c'}{\pi}\sup_{0\leq \tau \leq 1/2 - \theta}\int_{\theta+\frac{1}{2n+3}}^{\theta+\frac{1}{2n+3}+\tau} \frac{1}{s}\frac{C_1}{1+n(\theta+\frac{1}{2n+3}+\tau + \frac{C_1-1}{n}-s )}\\
& = \frac{c'}{\pi}\sup_{0\leq \tau \leq 1/2 - \theta} \int_{\theta+\frac{1}{2n+3}}^{\theta+\frac{1}{2n+3}+\tau} \left\{\frac{1}{s}+ \frac{1}{1/n + \left(\theta + \frac{1}{2n+3}+\tau + \frac{C_1-1}{n}-s \right)}\right\} \frac{C_1/n}{(\theta + \frac{1}{2n+3} + \tau + \frac{C_1}{n} )}\\
& = \frac{c'}{\pi}\sup_{0\leq \tau \leq 1/2 - \theta} \left\{\log(\frac{\theta + \frac{1}{2n+3}+\tau}{\theta + \frac{1}{2n+3}}) - \log(\frac{C_1/n }{ \tau + \frac{C_1}{n}})\right\} \frac{C_1/n}{\theta + \frac{1}{2n+3} + \tau + \frac{C_1}{n}}
\end{align}

We then consider two cases:
\begin{itemize}
\item Either $\theta\geq \frac{C_1}{n}$. In this case, we use 
\begin{align}
R_1&\leq \frac{c'}{\pi}\sup_{0\leq \tau \leq 1/2 - \theta} \left\{\log(\frac{\theta + \frac{1}{2n+3}+\tau}{\theta + \frac{1}{2n+3}}) - \log(\frac{C_1/n}{\frac{1}{n}+ \tau + \frac{C_1-1}{n}})\right\}\frac{C_1/n}{\theta + \frac{1}{2n+3} + \tau + \frac{C_1}{n}}\\
&\leq \frac{c'}{\pi}\frac{C_1}{n}\left\{\log(\frac{\theta + \frac{1}{2n+3} + \tau}{C_1/n}) + \log(\frac{\tau + \frac{C_1}{n}}{\theta + \frac{1}{2n+3}})\right\}\frac{1}{\theta + \frac{1}{2n+3}+\tau + \frac{C_1}{n}}\\
&\leq \frac{2c'}{\pi}\label{partialR1aD4plusRealLargeTheta}
\end{align}
\item When $\theta<\frac{C_1}{n}$, we can write 
\begin{align}
R_1& \leq \sup_{0\leq \tau \leq 1/2 - \theta} \frac{c'}{\pi}\left\{\log(\frac{\theta + \frac{1}{2n+3}+\tau}{\theta + \frac{1}{2n+3}}) - \log(\frac{C_1/n }{\frac{1}{n}+ \tau + \frac{C_1-1}{n} })\right\}\frac{C_1/n}{\theta + \frac{1}{2n+3}+ \tau + \frac{C_1}{n}}\\
&\leq  \sup_{0\leq \tau \leq 1/2 - \theta} \frac{c'}{\pi}\log(\frac{\theta + \tau + \frac{1}{2n+3}}{C_1/n} \vee \frac{C_1/n}{\theta + \tau + \frac{1}{2n+3}}) \frac{C_1/n}{\theta + \frac{1}{2n+3} + \tau + \frac{C_1}{n}}\\
&+  \sup_{0\leq \tau \leq 1/2 - \theta}\frac{c'}{\pi }\log(\frac{\tau + \frac{C_1}{n}}{\theta + \frac{1}{2n+3}}\vee \frac{\theta + \frac{1}{2n+3}}{C_1/n}) \frac{C_1/n}{\theta + \frac{1}{2n+3} + \tau + \frac{C_1}{n}}\\
&\leq \frac{c'}{\pi} \left(1\vee \log(3) + \log(C_1)\right) + \frac{c'}{\pi} \left(1 \vee 1+ \log(3)+ \log(C_1)\right)\\
&\leq \frac{2c'}{\pi}(1+\log(3)+\log(C_1)).\label{partialR1bD4plusRealLargeTheta}
\end{align}
\end{itemize}
The bound on $R_2$ follows from 
\begin{align}
R_2&\leq \sup_{0\leq \tau \leq 1/2 - \theta}  \frac{1}{\pi \theta}\int_{\theta + \frac{1}{2n+3}}^{\theta + \frac{1}{2n+3}+\tau} \left(\frac{1}{2}+\frac{1}{\pi}\right)\frac{1}{4(2n+3)}\frac{1}{s-\theta} \frac{C_1}{1+n(\theta + \frac{1}{2n+3} + \tau + \frac{C_1-1}{n}-s )}\; ds\\
&\leq \sup_{0\leq \tau \leq 1/2 - \theta} \frac{c'}{\pi} \int_{\theta+\frac{1}{2n+3}}^{\theta+\frac{1}{2n+3}+\tau} \left\{\frac{1}{s-\theta} + \frac{1}{\theta+\frac{1}{2n+3}+ \tau + \frac{C_1}{n}-s }\right\}\frac{C_1}{n}\frac{1}{\frac{1}{2n+3}+\tau + \frac{C_1}{n}}\\
&\leq \frac{c'}{\pi} \frac{C_1}{n}\left\{\log(\frac{\frac{1}{2n+3}+\tau}{\frac{1}{2n+3}}) + \log(\frac{\tau+\frac{C_1}{n} }{C_1/n })\right\}\frac{1}{\tau+\frac{1}{2n+3}+\frac{C_1}{n}}\\
&\leq \frac{c'}{\pi} \log( \frac{\frac{1}{2n+3}+\tau}{C_1/n}\vee \frac{C_1/n}{\frac{1}{2n+3} + \tau})\frac{C_1/n}{\tau + \frac{1}{2n+3}+C_1/n}\\
&+ \frac{c'}{\pi} \log(\frac{\tau + C_1/n}{\frac{1}{2n+3}}\vee \frac{\frac{1}{2n+3}}{\tau +C_1/n})\frac{C_1/n}{\tau + \frac{1}{2n+3} + C_1/n}\\
&\leq \frac{c'}{\pi} \left(1\vee \log(3)+\log(C_1)\right) + \frac{c'}{\pi} \left(1+\log(3)+\log(C_1)\vee 1\right)\\
& \leq \frac{2c'}{\pi} \left(1+ \log(3) + \log(C_1)\right)\label{boundpartialR2D4plus RealLargeTheta}
\end{align}

Combining~\eqref{partialR0aD4plusRealLargeTheta}, ~\eqref{partialR1aD4plusRealLargeTheta}~\eqref{partialR1bD4plusRealLargeTheta} as well as~\eqref{boundpartialR2D4plus RealLargeTheta}, we get 
\begin{align}
Z_L &\leq  R_0 + R_1  +R_2 \leq \frac{1}{4\pi} + \frac{4c'}{\pi} \left(\log(3)+1\right) + \frac{4c'}{\pi} \log(C_1)\label{boundpartialZ4cDomainD4plusRealLargeTheta}
\end{align}
Grouping~\eqref{boundpartialZ4cDomainD4plusRealLargeTheta},~\eqref{partialZ4plusaDomainD4plusRealLargeThetaZR} and~\eqref{partialZ4plusaDomainD4plusRealLargeTheta}, we have 
\begin{align}
Z&\leq 2.5+ \log(C_1). \label{boudTotalLargeThetaD4plusLemmaTroncReal}
\end{align}

Both the ``small $\theta$" bound and the imaginary part are functions of the ratio $\frac{\theta}{s-\theta}$. We thus focus on controling the integral $\varphi(\theta)$ defined as
\begin{align}
\varphi(\theta) = \frac{1}{|1 - e^{2\pi i \theta}|} \int_{\theta + \frac{1}{2n+3}}^{1/2} \frac{\theta p(e^{2\pi i s})}{s - \theta}\; ds 
\end{align}
For this integral, proceceding as before, we get 
\begin{align}
|\varphi(\theta)| & \leq \frac{1}{\pi \theta} \int_{\theta + \frac{1}{2n+3}}^{\theta + \frac{1}{2n+3} +2\frac{C_1-1}{n}} \frac{1}{s-\theta}\; ds\\
&+ \frac{1}{\pi} \int_{\theta + \frac{1}{2n+3}}^{1/2} \frac{1}{s-\theta} \frac{C_1}{1+n (s - (\theta + \frac{1}{2n+3} - \frac{C_1-1}{n}))} \; ds\\
&+ \frac{1}{\pi} \sup_{0\leq \tau \leq \frac{1}{2}-\theta} \int_{\theta + \frac{1}{2n+3}}^{\theta + \frac{1}{2n+3} + \tau} \frac{1}{s-\theta} \frac{C_1}{1+n(\frac{1}{2n+3} + \theta + \tau + \frac{C_1-1}{n} - s)}\; ds\\
&\leq \frac{1}{\pi} \log(\frac{\frac{1}{2n+3} + 2\frac{C_1-1}{n}}{\frac{1}{2n+3}}) \\
&+ \frac{1}{\pi} \left|\log(s-\theta) - \log(1/n + s- (\theta + \frac{1}{2n+3} - \frac{C_1-1}{n}))\right|_{\theta + \frac{1}{2n+3}}^{1/2} \frac{C_1/n}{C_1/n - \frac{1}{2n+3}}\\
&+ \frac{1}{\pi} \sup_{0\leq \tau \leq 1/2-\theta} \left|\log(s-\theta) - \log(\frac{1}{2n+3} + \theta + \tau + \frac{C_1}{n} - s)\right|_{\theta + \frac{1}{2n+3}}^{\theta + \frac{1}{2n+3}+\tau} \frac{C_1/n}{\frac{1}{2n+3} + \tau +\frac{C_1}{n}}\\
&\leq \frac{1}{\pi} \left(\log(6)+ \log(C_1)\right)+ \frac{1}{\pi} \left|\log(\frac{1/2 - \theta}{\frac{1}{2n+3}}) - \log(\frac{1/2 - \theta - \frac{1}{2n+3} + \frac{C_1}{n}}{C_1/n})\right| \frac{C_1}{C_1-1}\\
&+ \frac{1}{\pi} \sup_{0\leq \tau \leq 1/2 - \theta} \left|\log(\frac{\tau + \frac{1}{2n+3}}{\frac{1}{2n+3}}) - \log(\frac{C_1/n}{\tau + C_1/n})\right| \frac{C_1/n}{\frac{1}{2n+3} + \tau + \frac{C_1}{n}}\\
&\leq \frac{1}{\pi} \left(\log(6)+ \log(C_1)\right)+ \frac{1}{\pi} \left(\log(3C_1) + \log(4C_1)\right) \frac{C_1}{C_1-1}\\
&+ \frac{1}{\pi} \sup_{0\leq \tau \leq 1/2 - \theta} \left|\log(\frac{\tau + C_1/n}{C_1/n}) + \log(3)+ \log(C_1) + \log(\frac{\tau + \frac{1}{2n+3}}{C_1/n})\right| \frac{C_1/n}{\frac{1}{2n+3} + \tau + C-1/n}\\
&\leq \frac{1}{\pi} \left(\log(6)+ 2\log(3)+ 2\log(4) + 2 + \log(3)\right) + \frac{6}{\pi}\log(C_1) \\
&\leq 3.2 + 2\log(C_1). \label{phiThetaTMPD4plusLemmatronc}
\end{align}
Now using the bound on $\varphi(\theta)$ as well as the ``small $\theta$" bound and the bound on the imaginary part for the integral $ \int_{0}^{-(\theta - \theta_\ell)}  e^{2\pi i (n+1) t} \tilde{D}(e^{-2\pi i (s+t)})\; dt$ on $D_4^+$,
\begin{align}
F_{R,D_4}^+ &\leq \frac{3}{2}\frac{\theta}{|\theta-s|} + \frac{\theta}{8(2n+3)s|s-\theta|} + \frac{\theta}{2}, \quad \text{when $\theta \leq \frac{1}{2n+3}$}, 
\end{align}
\begin{align}
F_{I,D_4}^+ \leq F_{R, D_4}^+ + \frac{1}{4}\log(\frac{s}{s-\theta})\label{tmp00047}
\end{align}
using $s\geq \frac{1}{2n+3}$ on $D_4^+$ as well as $\frac{1}{|e^{2\pi i \theta}-1|}\int_{\theta + \frac{1}{2n+3}}^{1/2} \frac{\theta p(e^{2\pi i s})}{2}\; ds \leq \frac{1}{4\pi}$, we get 
\begin{align}
 \int_{\theta+ \frac{1}{2n+3}}^{1/2} \frac{p(e^{2\pi i s})}{|1-e^{2\pi i \theta}|}\; F_{R,D_4}^+(s, \theta)\; ds & \leq \left(\frac{3}{2}+ \frac{1}{8}\right) \varphi(\theta) + \frac{1}{4\pi }\\
&\leq \left(\frac{3}{2}+ \frac{1}{8}\right) \eqref{phiThetaTMPD4plusLemmatronc}+ \frac{1}{4\pi }\\
& \leq 5.3 + 3.5 \log(C_1), \quad \forall \theta \leq \frac{1}{2n+3}\label{boudTotalLargeThetaD4plusLemmaTroncRealSmallTheta}
\end{align} 
A corresponding result holds for the imaginary part. Using the bound on $\varphi(\theta)$, one can write 
\begin{align}
\int_{\theta + \frac{1}{2n+3}}^{1/2} \frac{p(e^{2\pi i s})}{|1-e^{2\pi i \theta}|}\; F_{I,D_4}^+(s)\; ds \leq \int_{\theta + \frac{1}{2n+3}}^{1/2} \frac{p(e^{2\pi i s})}{|1-e^{2\pi i \theta}|}\; F_{R,D_4}^+(s)\; ds  + \frac{1}{4} \varphi(\theta)\label{boudTotalD4plusImag}
\end{align}
Combining~\eqref{boudTotalLargeThetaD4plusLemmaTroncRealSmallTheta} and~\eqref{boudTotalLargeThetaD4plusLemmaTroncReal}, we get 
\begin{align}
\int_{\theta+ \frac{1}{2n+3}}^{1/2} \frac{p(e^{2\pi i s})}{|1-e^{2\pi i \theta}|}\; F_{R,D_4}^+(s; \theta)\; ds \leq 8+ 5\log(C_1)\label{totalBoundRealD4plusLemmaTronclargeDeviation}
\end{align}
Now adding~\eqref{boudTotalD4plusImag}, we have 
\begin{align}
\int_{\theta+ \frac{1}{2n+3}}^{1/2} \frac{p(e^{2\pi i s})}{|1-e^{2\pi i \theta}|}\; F(s; \theta)\; ds& \leq 2\left(8+ 5\log(C_1)\right) + \frac{1}{4}\left(3.2 + 2\log(C_1)\right)\\
&\leq 17 + 11\log(C_1). 
\end{align}

\subsection{Proof of lemma~\ref{lemmaTroncSmallDomainsNegativeCons} ($D_5^-, D_4^-, D_1^-$ and $D_{2}^-$, small $|\tau-\alpha|$)}

Before proceeding with those lemmas, we recall the bounds on $F_{D_i}$ for both real and imaginary parts that were derived in section~\eqref{boundAAoneatomic}

\begin{align}
F_{R,D_1^-}& = \frac{\pi}{4} + \frac{1}{4(2n+3)}\left(\frac{1}{-s+\frac{1}{2n+3}} + \frac{1}{\theta-s}\right)\left(\frac{1}{\pi}+ \frac{1}{2}\right) + \frac{\theta}{2}\\
F_{R,D_2^-}& = \frac{(2n+3)\pi \theta}{4} + \frac{\theta}{2}\\
F_{R,D_4^-} & = \frac{1}{4(2n+3)}\left(\frac{1}{\pi}+\frac{1}{2}\right) \left(2+ \frac{1}{1+(s-\theta)}\right) + \frac{\pi}{4}+ \frac{1}{4(2n+3)}\left(\frac{1}{\pi}+ \frac{1}{2}\right)\left(\frac{1}{\frac{1}{2n+3}-s}+2\right) + \frac{\theta}{2}\\
F_{R,D_5^-}& = \frac{\pi}{4} + \frac{1}{4(2n+3)}\left(\frac{1}{\pi}+ \frac{1}{2}\right) \left(2+ \frac{1}{1+s-\theta + \frac{1}{2n+3}}\right) + \frac{1}{4(2n+3)}\left(\frac{1}{\pi} + \frac{1}{2}\right)\left(2+\frac{1}{-s}\right)  + \frac{\theta}{2}
\end{align}

for the real part and 

\begin{align}
F_{I,D_1^-}(s)& = \log(\frac{\theta - s}{\frac{1}{n}-s}) + \frac{1}{4(2n+3)} \left(\frac{1}{\pi}+ \frac{1}{2}\right)\left(\frac{1}{\theta-s} + \frac{1}{1/n -s}\right) + \frac{\pi(n+1)}{4n}\\
F_{I,D_2^-}(s) &= \frac{\pi (n+1)\theta}{4}\\
F_{I,D_4^-}(s)& = \frac{1}{n} \frac{(n+1)\pi}{4} + \frac{1}{4(2n+3)}\left(\frac{1}{2}+ \frac{1}{\pi}\right) \left(\frac{1}{1+s-\theta} +2\right) \\
&+ \left|\log(\frac{1/2}{1+s-\theta})\right| + \frac{1}{4(2n+3)}\left(\frac{1}{2}+ \frac{1}{\pi}\right) \left(\frac{1}{-s+\frac{1}{n}} + 2\right) + \left|\log(\frac{1/2}{-s+\frac{1}{n}})\right|\\
F_{I,D_5^-}(s)& = \frac{\pi}{2} + \frac{1}{4} \left|\log(\frac{-s}{1/2})\right| + \frac{1}{4}\left|\log(\frac{1/2}{1+s-\theta+\frac{1}{n}})\right|+ \frac{1}{4(2n+3)}\left(\frac{1}{\pi}+ \frac{1}{2}\right)\left(2+ \frac{1}{1+s-\theta + \frac{1}{n}}\right)\\
&+ \frac{1}{4(2n+3)} \left(\frac{1}{\pi}+ \frac{1}{2}\right)\left(2+ \frac{1}{-s}\right) 
\end{align}
for the imaginary part. Using those bounds, we first control the integral of the real part in section~\ref{realPartStripesMultiMinus}. The integral of the imaginary part is then addressed in section~\eqref{imaginaryPartSmallStripesFDminus}.  

\subsubsection{\label{realPartStripesMultiMinus}Real part}

%

We now bound the contributions from $D_1^-$. for this subdomain, we have 
\begin{align}
\int_{D_1^-} F_{R,D_1^-}\; ds &\leq \frac{1}{\pi \theta} \int_{-\frac{1}{2n+3}}^0\frac{\pi}{4} + \frac{1}{4(2n+3)} \left(\frac{1}{-s+\frac{1}{2n+3}} + \frac{1}{\theta-s}\right) \left(\frac{1}{\pi}+ \frac{1}{2}\right) + \frac{\theta}{2}\; ds\\
&\leq \frac{1}{4} + \frac{1}{4\pi \theta (2n+3)} \left[\log(2)+ \log(1+ \frac{1}{(2n+3)\theta})\right] \left(\frac{1}{\pi} + \frac{1}{2}\right) + \frac{1}{2n+3}\frac{1}{2\pi }\\
&\leq \frac{1}{4} + \frac{\log(2)}{2\pi}\left(\frac{1}{\pi}+ \frac{1}{2}\right) +\frac{1}{2n+3}\frac{1}{2\pi}\\
&\leq 0.4\label{boundFD1minusMS}
\end{align}

Again we use the fact that $\theta \geq \frac{1}{2n+3}$. For $D_2^-$, we have 

\begin{align}
\frac{1}{|1-e^{2\pi i \theta}|}\int_{D_2^-} \frac{(2n+3)\pi \theta}{4} + \frac{\theta}{2} \; ds\leq \frac{1}{4} + \frac{1}{2\pi} \frac{1}{2n+3} \leq 0.3\label{boundFD2minusMS}
\end{align}

Similarly for $D_{4}^-$ and $D_5^-$, we have 
\begin{align}
\frac{1}{|1-e^{2\pi i \theta}|}\int_{D_4^-} F_{D_4}^-\; ds &\leq \frac{1}{\pi \theta} \frac{1}{4(2n+3)}\left(\frac{1}{\pi} + \frac{1}{2}\right) \left(4+ \frac{\pi}{4} + \frac{\theta}{2}\right) \frac{1}{2n+3}\\
&+ \frac{1}{\pi \theta} \frac{1}{4(2n+3)} \left(\frac{1}{\pi} + \frac{1}{2}\right) \left\{\log(1/2) - \log(1-\frac{1}{2n+3} - \theta)\right\}\\
&+ \frac{1}{\pi \theta} \left(\frac{1}{\pi}+ \frac{1}{2}\right) \frac{1}{4(2n+3)} \left|\log(1/2-\theta + \frac{1}{2n+3}) - \log(\frac{2}{2n+3})\right|\\
&\leq \frac{1}{\pi(1/2 - \frac{1}{2n+3})} \frac{1}{4(2n+3)} \left(\frac{1}{\pi}+ \frac{1}{2}\right) \left(4+ \frac{\pi}{4} + \frac{\theta}{2}\right) \frac{1}{2n+3}\\
&+ \frac{1}{\pi (1/2- \frac{1}{2n+3})}\frac{1}{4(2n+3)} \left(\frac{1}{\pi} + \frac{1}{2}\right) \log(1+ \frac{1}{1/2 - \frac{1}{2n+3}})\\
& + \frac{1}{\pi (1/2 - \frac{1}{2n+3})} \left(\frac{1}{\pi} + \frac{1}{2}\right) \frac{1}{4(2n+3)} \log(2)\\
&\leq \frac{1}{\pi} \frac{14}{5} \frac{1}{28} \left(\frac{1}{\pi} + \frac{1}{2}\right) \log(4)\\
&+ \frac{1}{\pi} \frac{14}{5}\frac{1}{28} \left(\frac{1}{\pi}+ \frac{1}{2}\right) \log(2)\\
&+ \frac{1}{\pi}\frac{14}{5} \frac{1}{28} \left(\frac{1}{\pi}+ \frac{1}{2}\right) \log(2)\\
&\leq \frac{1}{\pi} \frac{14}{5} \frac{1}{28} \left(\frac{1}{\pi} + \frac{1}{2}\right) \left(\frac{9}{2}+ \frac{\pi}{4}\right) \frac{1}{7} + 3\log(2) \frac{1}{\pi} \frac{14}{5} \frac{1}{28} \left(\frac{1}{\pi}+ \frac{1}{2}\right)\\
&\leq 0.1\label{boundFD4minusMS}
\end{align}

\begin{align}
\frac{1}{|1-e^{2\pi i \theta}|}\int_{D_5^-} F_{D_5}^- \; ds & \leq  \frac{1}{\pi \theta}\int_{-\frac{1}{2}}^{-1+\theta + \frac{1}{2n+3}} \frac{\pi}{4} + \frac{1}{4(2n+3)}\left(\frac{1}{\pi} + \frac{1}{2}\right) 4 + \frac{\theta}{2} \; ds\\
& + \frac{1}{\pi \theta} \frac{1}{4(2n+3)}\left(\frac{1}{\pi}+ \frac{1}{2}\right)\left|\log(1+s-\theta + \frac{1}{2n+3})\right|_{-\frac{1}{2}}^{-1+\theta + \frac{1}{2n+3}}\\
&+ \left|\log(-s) \frac{1}{4(2n+3)} \left(\frac{1}{\pi} + \frac{1}{2}\right)\right|_{-\frac{1}{2}}^{-1+\theta + \frac{1}{2n+3}}\\
&\leq \frac{1}{\pi (1/2-\frac{1}{2n+3})} \left[\frac{\pi}{4} + \frac{1}{4(2n+3)}\left(\frac{1}{\pi} + \frac{1}{2}\right)4 + \frac{\theta}{2}\right] \left(-\frac{1}{2} + \theta + \frac{1}{2n+3}\right)\\
&+ \frac{1}{\pi (1/2 - \frac{1}{2n+3})} \frac{1}{4(2n+3)} \left(\frac{1}{\pi} + \frac{1}{2}\right) \left[\log(\frac{2}{2n+3}) - \log(\frac{1}{2} - \theta + \frac{1}{2n+3})\right]\\
&+ \frac{1}{\pi (1/2 - \frac{1}{2n+3})} \frac{1}{4(2n+3)} \left(\frac{1}{\pi}+ \frac{1}{2}\right) \log(\frac{1-\theta - \frac{1}{2n+3}}{1/2})\\
&\leq \frac{3}{\pi} \left[\frac{\pi}{4} + \frac{1}{7} \left(\frac{1}{\pi} + \frac{1}{2}\right) + \frac{\theta}{2}\right]\frac{1}{7}\\
&+ \frac{3}{\pi} \frac{1}{28} \left(\frac{1}{\pi}+ \frac{1}{2}\right) \log(2) + \frac{3}{\pi} \frac{1}{28}\left(\frac{1}{\pi}+ \frac{1}{2}\right) \log(1+ \frac{1}{1/2 - \frac{1}{2n+3}})\\
& \leq \frac{3}{\pi} \left[\frac{\pi}{4}+ \frac{1}{7}\left(\frac{1}{\pi}+ \frac{1}{2}\right) + \frac{1}{4}\right]\frac{1}{7}
+ \frac{3}{\pi} \left[\frac{1}{28}\left(\frac{1}{\pi} + \frac{1}{2}\right)3\log(2)\right]\\
&\leq 0.22\label{boundFD5minusMS}
\end{align}
In the lines above, we assumed $n\geq 2$. Combining~\eqref{boundFD1minusMS},~\eqref{boundFD2minusMS},~\eqref{boundFD4minusMS} and~\eqref{boundFD5minusMS}, we get 
\begin{align}
&\sum_{\lambda\in\left\{1,2,4,5\right\}}\left|\int_{s\in D_\lambda^+} \frac{p(e^{2\pi i s})}{|1 - e^{2\pi i\theta}|}  \left|\Rp \left(\int_{0}^{-\theta}e^{2\pi i (n+1)t} \tilde{D}(e^{-2\pi i (s+t)})\; dt\right)\right|\; ds\right| \leq 1.1\label{imaginaryDiverseNegativeLemmaTroncNearNearRealPart}
\end{align}

\subsubsection{\label{imaginaryPartSmallStripesFDminus}Imaginary part}

%

We now derive the bounds on the imaginary parts. Using the expressions that were just recalled above, we can write 
\begin{align}
\frac{1}{\pi \theta}\int_{-\frac{1}{2n+3}}^0 F_{I,D_2^-}(s)\; ds \leq \frac{(n+1)}{(2n+3)}\frac{1}{4}\leq \frac{1}{2} \label{boundFD2MSImaginary}
\end{align}

\begin{align}
\frac{1}{\pi \theta}\int_{D_{1}^-} F_{I,D_1^-}(s)\; ds &= \frac{1}{\pi \theta} \left|\frac{1}{4(2n+3)}\left(\frac{1}{\pi} + \frac{1}{2}\right) \log(\theta-s) + \log(\frac{1}{n}-s)\right|_{-\frac{1}{2n+3}}^0 \\
&+ \frac{\pi (n+1)}{4n }\frac{1}{\pi \theta} \frac{1}{2n+3} \\
&+ \frac{1}{\pi \theta} \int_{-\frac{1}{2n+3}}^0 \log(\frac{\theta-s}{\frac{1}{n}-s})\; ds\\
&\leq \frac{1}{\pi \theta} \frac{1}{4(2n+3)} \left(\frac{1}{\pi}+\frac{1}{2}\right) \log(1+ \frac{1}{2n+3}\frac{1}{\theta}) + \log(1+ \frac{1}{2n+3}n) \\
&+ \frac{1}{\pi \theta} \frac{\pi}{2} \frac{1}{2n+3}\\
&+\frac{1}{\pi \theta} \left|\log(1+ \frac{1/n-\theta}{\theta-s}) \vee \log(1+ \frac{\theta-\frac{1}{n}}{\frac{1}{n} -s})\right|_{-\frac{1}{2n+3}}^0\\
&\leq \frac{1}{\pi \theta} \frac{1}{4(2n+3)} \left(\frac{1}{\pi}+ \frac{1}{2}\right) 2\log(2) + \frac{1}{2}\\
&+ \frac{1}{\pi} \frac{2n+3}{n} \log(1+ \frac{1}{2n+3}\frac{1}{\theta}) + \frac{1}{\pi} \log(1+ \frac{1}{2n+3} n)\\
&\leq \frac{1}{2\pi} \left(\frac{1}{\pi}+ \frac{1}{2}\right) \log(2) + \frac{1}{2} + \frac{3}{\pi} \log(2) + \frac{1}{\pi} \log(2) \\
&\leq \frac{1}{2\pi} \left(\frac{1}{\pi} + \frac{1}{2}\right) \log(2) + \frac{1}{2} + \frac{4}{\pi}\log(2)\\
&\leq 1.5\label{boundFD1minusImaginary}
\end{align}

For $D_{4}^-$, we have 
\begin{align}
\frac{1}{\pi \theta} \int_{-\frac{1}{2n+3}}^{-\frac{1}{2}+\theta} F_{D_4}^-(s)\; ds &\leq \frac{1}{\pi \theta} \int_{-\frac{1}{2n+3}}^{-\frac{1}{2}+\theta} \frac{1}{n} \frac{(n+1)\pi}{n} + 4\left(\frac{1}{4(2n+3)}\right)\left(\frac{1}{2} + \frac{1}{\pi}\right) \; ds\\
&+ \frac{1}{\pi \theta} \left|\log(1+s-\theta) - \log(-s+\frac{1}{n})\right|_{-\frac{1}{2n+3}}^{-\frac{1}{2} + \theta} \frac{1}{4(2n+3)} \left(\frac{1}{\pi}+ \frac{1}{2}\right)\\
&+ \frac{1}{2}\log(1+s-\theta) + \frac{1}{2}\log(-s+\frac{1}{n})\\
&\leq \frac{1}{\pi(\frac{1}{2} - \frac{1}{2n+3})} \left[\frac{2\pi}{n} + \frac{1}{7}\left(\frac{1}{2}+ \frac{1}{\pi}\right)\right] \frac{1}{2n+3}\\
&+ \frac{1}{\pi(1/2- \frac{1}{2n+3})} \frac{1}{4(2n+3)}\left(\frac{1}{\pi}+ \frac{1}{2}\right) \left[\log(\frac{1/2}{1-\frac{1}{2n+3} - \theta}) - \log(\frac{\frac{1}{2}-\theta+\frac{1}{n}}{\frac{1}{2n+3} + \frac{1}{n}})\right]\\
&+ \frac{1}{2} \frac{1}{\pi(1/2 - \frac{1}{2n+3})} \left[\log(\frac{1/2}{1- \frac{1}{2n+3} - \theta}) - \log(\frac{1/2-\theta +\frac{1}{n}}{\frac{1}{2n+3} + \frac{1}{n}})\right]\\
&+ \frac{1}{2} \frac{1}{\pi(\frac{1}{2}-\frac{1}{2n+3})}\left[\log(\frac{1/2}{1-\frac{1}{2n+3} - \theta}) + \log(\frac{1/2-\theta+\frac{1}{n}}{\frac{1}{2n+3}+ \frac{1}{n}})\right]\\
&\leq \frac{1}{\pi} 3 \left[\pi + \frac{1}{7}\left(\frac{1}{2}+ \frac{1}{\pi}\right)\right] \frac{1}{7}\\
&+ \frac{3}{\pi} \frac{1}{28} \left(\frac{1}{\pi}+ \frac{1}{2}\right) \left[\log(1+ \frac{3}{2n+3}) + \log(1+ \frac{n}{2n+3})\right]\\
&+ \frac{3}{\pi}\frac{1}{2} \left[\log(1+ \frac{3}{2n+3}) + \log(1+ \frac{n}{2n+3})\right]\\
&\leq 1.3 \label{boundFD4minusImaginary}
\end{align}

Finally for $F_{D_5}^-$, we have 

\begin{align}
\frac{1}{|1-e^{2\pi i \theta}|}\int_{D_5^-} F_{D_5^-}(s)\; ds & \leq \frac{1}{\pi \theta}\int_{-\frac{1}{2}}^{-1+\theta + \frac{1}{2n+3}} \frac{\pi}{2} + \frac{1}{4} \frac{1}{2}\frac{1}{-s} + \frac{1}{4}\frac{1}{2}\frac{1}{1+s-\theta + \frac{1}{n}}\; ds\\
&+\frac{1}{\pi \theta} \int_{-\frac{1}{2}}^{-1+\theta +\frac{1}{2n+3}} \frac{1}{4(2n+3)} \left(\frac{1}{\pi}+ \frac{1}{2}\right) \left(2+ \frac{1}{1+s-\theta + \frac{1}{n}}\right)\; ds\\
&  + \frac{1}{\pi \theta}\int_{-\frac{1}{2}}^{-1+\theta + \frac{1}{2n+3}} \frac{1}{4(2n+3)} \left(\frac{1}{\pi}+ \frac{1}{2}\right) \left(2+ \frac{1}{-s}\right)\; ds\\
&\leq \frac{1}{\pi \theta} \int_{-\frac{1}{2}}^{-1+\theta + \frac{1}{2n+3}} \frac{\pi}{2} + \frac{1}{4(2n+3)}\left(\frac{1}{\pi}+ \frac{1}{2}\right)4\; ds\\
&+ \frac{1}{\pi \theta}\left|\frac{1}{8}\log(-s) + \frac{1}{8}\log(1+s-\theta + \frac{1}{n})\right|_{-\frac{1}{2}}^{-1+\theta + \frac{1}{2n+3}}\\
&+ \frac{1}{\pi \theta} \left| \log(-s) + \log(1+s-\theta + \frac{1}{n})\right|_{-\frac{1}{2}}^{-1+\theta + \frac{1}{2n+3}} \frac{1}{4(2n+3)}\left(\frac{1}{\pi}+ \frac{1}{2}\right)\\
&\leq \left(\frac{\pi}{2}+ \frac{1}{4(2n+3)} \left(\frac{1}{\pi}+ \frac{1}{2}\right)4\right) \frac{1}{2n+3} \frac{1}{\pi (1/2 - \frac{1}{2n+3})}\\
&+ \frac{1}{8}\left[\log(\frac{1-\theta - \frac{1}{2n+3}}{1/2}) + \log(\frac{\frac{1}{n} + \frac{1}{2n+3}}{\frac{1}{2} - \theta + \frac{1}{n}})\right]\frac{1}{\pi \theta}\\
&+ \frac{1}{\pi \theta} \frac{1}{4(2n+3)}\left(\frac{1}{\pi}+ \frac{1}{2}\right) \left|\log(\frac{1-\theta - \frac{1}{2n+3}}{\frac{1}{2}}) + \log(\frac{\frac{1}{n}+\frac{1}{2n+3}}{\frac{1}{2} - \theta + \frac{1}{n}})\right|\\
&\leq \left[\frac{1}{4} + \frac{1}{28}\left(\frac{1}{\pi}+ \frac{1}{2}\right)\right]\log(2)\frac{14}{5\pi}\\
&+ \frac{14}{35\pi} \left(\frac{\pi}{2} + \frac{1}{7}\left(\frac{1}{\pi}+ \frac{1}{2}\right)\right)\\
&\leq 0.4\label{boundFD5minusImaginary}
\end{align}


Combining~\eqref{boundFD2MSImaginary}~\eqref{boundFD1minusImaginary}~\eqref{boundFD4minusImaginary},and~\eqref{boundFD5minusImaginary} gives 
\begin{align}
\sum_{\lambda\in\left\{1,2,4,5\right\}}\left|\int_{s\in D_\lambda^+} \frac{p(e^{2\pi i s})}{|1 - e^{2\pi i\theta}|}  \left|\Ip \left(\int_{0}^{-\theta}e^{2\pi i (n+1)t} \tilde{D}(e^{-2\pi i (s+t)})\; dt\right)\right|\; ds\right| \leq 3.7\label{imaginaryDiverseNegativeLemmaTroncNearNear}
\end{align}
Combining~\eqref{imaginaryDiverseNegativeLemmaTroncNearNear} and~\eqref{imaginaryDiverseNegativeLemmaTroncNearNearRealPart} gives the result of the lemma. 


\subsection{Proof of lemma~\ref{lemmaTroncSmallDomainsPositiveCons} ($D_1^+$, $D_2^+$, $D_3^+$, small $|\tau-\alpha|$)}

On each of the stripes $D_1^+, D_2^+$ and $D_3^+$, one can simply integrate the bounds $F_{D_1}^+$, $F_{D_2}^+$ and $F_{D_3}^+$, assuming a constant (one) bound on the polynomial. Before deriving the bounds on the integral~\eqref{eq:dev_int_p+} (or~\eqref{eq:dev_int_p-}) for the subdomains, we recall the expressions of each of the $F_{D_i}^+$, $i\in \left\{1,2,3\right\}$

\begin{align}
F_{R, D_1}^+&\leq \frac{\pi (2n+3)\theta}{4} + \frac{\theta}{2}\\
F_{R, D_2}^+ &\leq \frac{\pi}{2} + \left(\frac{1}{s} + \frac{1}{|s-\theta + \frac{2}{2n+3}|}\right) \left(\frac{1}{2n+3} \right)\frac{1}{4}\left(\frac{1}{2} + \frac{1}{\pi}\right) + \frac{\theta}{2}\\
F_{R, D_3}^+ &\leq \frac{(2n+3)\pi}{4} \left(s + \frac{1}{2n+3}\right) + \left(\frac{1}{2}+ \frac{1}{\pi}\right)\frac{1}{4(2n+3)} (2n+3)
\end{align}

as well as 
\begin{align}
F_{I,D_1^+} &\leq \frac{(n+1)\pi \theta}{2}\label{recallFD1plusImag}\\
F_{I,D_2^+}& \leq \left(\frac{n+1}{2n}\right)\pi + \frac{1}{4(2n+3)} \left(\frac{1}{\pi}+ \frac{1}{2}\right) \left(\frac{1}{-s} + \frac{1}{s-\theta + \frac{2}{n}}\right) + \frac{1}{4}(-\log(s-\theta + \frac{2}{n}) +\frac{1}{4}\log(s)\label{recallFD2plusImag}\\
F_{I,D_3^+} &\leq \frac{1}{4(2n+3)} \left(\frac{1}{\pi}+ \frac{1}{2}\right) \left(\frac{1}{\theta-s} + \frac{1}{\frac{2}{n}-s}\right) + \frac{1}{4} \log(\frac{\theta-s}{\frac{2}{n}-s}) + \frac{\pi(n+1)}{n}\label{recallFD3plusImag}
\end{align}

We start by controlling the real part. Control on the imaginary part is then derived in section~\ref{imaginaryPartFDplusSmallStripesMS}.

\subsubsection{\label{realPartSmallStripesFDplusMS}Real part}

Starting with $F_{D_2}^+$, we have (note that here we always assume $\theta\geq \frac{2}{2n+3}$ otherwise we are first integrating over $D_3^+$ and then over $[\frac{1}{2n+3}, \theta+\frac{1}{2n+3}]$) (see the domain decomposition in Fig.~\ref{domainDecomposition1}. We treat this case below). 
\begin{align}
\frac{1}{|e^{2\pi i \theta}-1|} \left|\int_{D_{2}^+} F_{R,D_2}^+(s)\; ds \right|& \leq \frac{1}{\pi \theta} \int_{\theta - \frac{1}{2n+3}}^{\theta + \frac{1}{2n+3}} \frac{\pi}{2} + \left(\frac{1}{s} + \frac{1}{|s - \theta + \frac{2}{2n+3}|}\right) \frac{1}{2n+3}\frac{1}{4} \left(\frac{1}{2}+ \frac{1}{\pi}\right) + \frac{\theta}{2}\; ds \\
&\leq \pi\frac{2}{2n+3} \frac{1}{\pi \theta} + \frac{1}{\pi \theta} \log(1+ \frac{2}{2n+3} \frac{1}{\theta - \frac{1}{2n+3}})\frac{c'}{2n+3} \\
&+ \frac{1}{\pi \theta} 2\log(3/2) \frac{c'}{2n+3} + \frac{1}{\pi \theta} \frac{2}{2n+3}\frac{\theta}{2} 
\end{align}

In the last line we use the fact that on $D_2^+$, we always have $\theta\geq \frac{1}{2n+3}$ as well as $c = \frac{1}{2n+3}\left(\frac{1}{2}+ \frac{1}{\pi}\right)\frac{1}{4} \frac{1}{2n+3}$. Together this gives 
\begin{align}
\frac{1}{|e^{2\pi i \theta}-1|}\left|\int_{D_{2}^+} F_{R,D_2}^+ (s)\; ds\right| \leq 2+ \frac{\log(3)}{\pi}c' + \frac{2}{\pi}\log(3/2) + \frac{2\pi}{2n+3}\leq 3.6 \label{boundRealFD2plusMS}
\end{align}

For $D_3^+$, we have 
\begin{align}
\frac{1}{|e^{2\pi i \theta}-1|}\left| \int_{D_3^+} F_{R, D_3}^+ (s) \; ds\right|& \leq \frac{1}{\pi \theta} \int_{0}^{\frac{1}{2n+3}} \frac{(2n+3)\pi}{4} \left(s + \frac{1}{2n+3}\right) + \left(\frac{1}{2}+ \frac{1}{\pi}\right)\frac{1}{4}\; ds\\
&\leq \frac{1}{\pi} \left(\frac{1}{2}+ \frac{1}{\pi}\right) \frac{1}{4} + \frac{1}{8} + \frac{1}{4}\\
&\leq 0.5\label{boundRealFD3plusMS}
\end{align}

For $F_{D_1}^+$, we get 
\begin{align}
\frac{1}{|1-e^{2\pi i \theta}|}\left|\int_{D_1^+} F_{R,D_1}^+(s) \; ds\right| &\leq \int_{0}^{\frac{1}{2n+3}} \frac{2n+3}{4} + \frac{1}{2\pi}\; ds\leq \frac{1}{4} + \frac{1}{2\pi} \frac{1}{2n+3}\leq 0.3\label{boundRealFD1plusMS}
\end{align}

Now when integrating for $\theta \in [\frac{1}{2n+3}, \frac{2}{2n+3}]$, we keep the bound on $D_3^+$ and the integral on $D_2^+$ simply turns into 
\begin{align}
\frac{1}{|e^{2\pi i \theta} - 1|}\left|\int_{D_2^+} F_{R, D_2^+}\; ds\right| &\leq \frac{1}{\pi \theta} \int_{\frac{1}{2n+3}}^{\theta + \frac{1}{2n+3}} \frac{\pi}{2} + \left(\frac{1}{s}+ \frac{1}{|s-\theta + \frac{2}{2n+3}|}\right)\frac{1}{4}\frac{1}{2n+3} \left(\frac{1}{2}+ \frac{1}{\pi}\right) + \frac{\theta}{2}\; ds\\
&\leq \frac{1}{\pi \theta} \left\{\frac{\pi}{2}\theta + \left[\log(\frac{\theta+\frac{1}{2n+3}}{\frac{1}{2n+3}})+2 \log(\frac{3}{2n+3}\frac{1}{\frac{3}{2n+3} - \theta})\right] \frac{c'}{2n+3} + \frac{\theta^2}{2}\frac{1}{\pi \theta} \right\}\\
&\leq \frac{1}{2} + \left\{\log(2)+ \log(3/2)\right\} 2\frac{c'}{2n+3} +\frac{1}{2\pi}\\
&\leq 0.8 \label{boundFD2plusMSmodified}
\end{align}

Combining~\eqref{boundRealFD1plusMS},~\eqref{boundRealFD2plusMS},~\eqref{boundRealFD3plusMS} and~\eqref{boundFD2plusMSmodified} and using $c' = \frac{1}{4}\left(\frac{1}{2}+ \frac{1}{\pi}\right)$, we get the bound on the real part,
\begin{align}
&\sum_{\lambda=1}^3\left|\int_{s\in D_\lambda^+} \frac{p(e^{2\pi i s})}{|1 - e^{2\pi i\theta}|}  \left|\Rp \left(\int_{0}^{-\theta}e^{2\pi i (n+1)t} \tilde{D}(e^{-2\pi i (s+t)})\; dt\right)\right|\; ds\right| \leq 5.2
 \label{boundSmallStripesplusRealP}
%
\end{align}

\subsubsection{\label{imaginaryPartFDplusSmallStripesMS}Imaginary part}

We follow the same approach as for the real part. Using~\eqref{recallFD1plusImag} to~\eqref{recallFD3plusImag}, we get 
For $F_{D_1^+}$, we have 
\begin{align}
\frac{1}{|e^{2\pi i \theta}-1|}\int_{D_1^+} F_{D_1^+,I}(s)\; ds \leq \frac{1}{\pi \theta} \int_{0}^{\frac{1}{2n+3}} \frac{(n+1)\pi}{2}\theta \; ds \leq \frac{(n+1)\pi}{2} \frac{1}{2n+3}\leq \frac{\pi}{4} \label{boundImaginaryFD1plusMS}
\end{align}
For $D_3^+$, we similarly write 
\begin{align}
\frac{1}{|1-e^{2\pi i \theta}|}\int_{D_3^+} F_{D_3,I}\; ds &\leq \frac{1}{|1-e^{2\pi i \theta}|} \int_{0}^{\frac{1}{2n+3}}\frac{1}{4(2n+3)} \left(\frac{1}{\pi}+ \frac{1}{2}\right) \left(\frac{1}{\theta-s} + \frac{1}{2/n-s}\right) \\
&+ \frac{1}{|1-e^{2\pi i \theta}|} \int_{0}^{\frac{1}{2n+3}}\frac{1}{4} \log(\frac{\theta-s}{\frac{2}{n}-s}) + \frac{\pi(n+1)}{n}\; ds\label{tmp00061}
\end{align}
First note that we have 
\begin{align}
\log\left(\frac{\theta-s}{\frac{2}{n}-s}\right) \leq \left\{\log(1+ \frac{\theta-\frac{2}{n}}{\frac{2}{n}-s})\vee \log(1+ \frac{\frac{2}{n} - \theta}{\theta-s})\right\}\\
\end{align}
In particular, we can then write 
\begin{align}
\frac{1}{|e^{2\pi i \theta}-1|} \frac{1}{4} \int_{0}^{\frac{1}{2n+3}} \left|\log(\frac{\theta-s}{\frac{2}{n}-s})\right| &\leq \frac{1}{|e^{2\pi i \theta} - 1|} \frac{1}{4} \int_{0}^{\frac{1}{2n+3}} \log(1+ \frac{\theta - \frac{2}{n}}{\frac{2}{n} - s})\vee \log(1+ \frac{\frac{2}{n} - \theta}{\theta-s}) \; ds\\
&\leq \frac{1}{4\pi \theta} |\theta - \frac{2}{n}| \left\|\log(\frac{\frac{2}{n} - \frac{1}{2n+3}}{\frac{2}{n}})\vee \log(\frac{\theta - \frac{1}{2n+3}}{\theta})\right|\\
&\leq \frac{1}{4\pi}\log(2)
\end{align}
In the last line, we assume $\theta\geq \frac{2}{2n+3}$ (we treat the case $\theta \in [\frac{1}{2n+3}, \frac{2}{2n+3}]$ later)

Substituting this in~\eqref{tmp00061}, we get 
\begin{align}
\frac{1}{|1-e^{2\pi i \theta}|}\int_{D_3^+} F_{D^+_3,I}(s)\; ds &\leq \frac{1}{\pi \theta}\frac{1}{4(2n+3)}\left(\frac{1}{\pi}+ \frac{1}{2}\right) \left[\log(1+ \frac{1}{2n+3}\frac{1}{\theta - \frac{1}{2n+3}}) + \log(1+ \frac{1}{2n+3} \frac{1}{\frac{2}{n} - \frac{1}{2n+3}})\right]\\
&+ \frac{1}{4\pi} \log(2)+ \frac{\pi (n+1)}{n} \frac{1}{\pi \theta} \frac{1}{2n+3}\\
&\leq \frac{2n+3}{\pi} \frac{1}{4(2n+3)} \left(\frac{1}{\pi}+ \frac{1}{2}\right) 2\log(2) + \frac{\log(2)}{4\pi} +2\\
&\leq \left(\frac{1}{\pi}+ \frac{1}{2}\right) \log(2) \frac{1}{2\pi} + \frac{\log(2)}{4\pi}+2 \\
& \leq 2.2\label{boundFD3plusMS}
\end{align}
in the case $\theta\geq \frac{2}{2n+3}$. When $\theta \in [\frac{1}{2n+3}, \frac{2}{2n+3}]$, we integrate the bound $F_{D_3^+,I}$ on $s \in [0,\theta - \frac{1}{2n+3}]$ as shown in Fig.~\ref{domainDecomposition1},

\begin{align}
&\frac{1}{|1-e^{2\pi i \theta}|} \int_{0}^{\theta - \frac{1}{2n+3}} \frac{1}{4(2n+3)}\left(\frac{1}{\pi} + \frac{1}{2}\right) \left(\frac{1}{\theta-s} + \frac{1}{\frac{2}{n} - s}\right) \; ds\\
&+ \frac{1}{|e^{2\pi i \theta}-1|} \int_{0}^{\theta - \frac{1}{2n+3}} \frac{1}{4} \log(\frac{\theta-s}{\frac{2}{n} - s}) + \frac{\pi (n+1)}{n} \; ds\\
&\leq \frac{1}{\pi \theta} \frac{1}{4(2n+3)} \left(\frac{1}{\pi} + \frac{1}{2}\right) \left\{\log(\frac{\frac{1}{2n+3}}{\theta}) + \log(\frac{\frac{2}{n} - \theta + \frac{1}{2n+3}}{\frac{2}{n}})\right\}\\
&+ \frac{1}{\pi \theta} \frac{1}{4} \int_{0}^{\theta - \frac{1}{2n+3}} \log(\frac{\frac{2}{n}-s}{\theta-s}) \; ds + \frac{1}{\pi \theta} 2\pi \left(\theta - \frac{1}{2n+3}\right)\\
&\leq \frac{1}{4\pi}\left(\frac{1}{\pi}+ \frac{1}{2}\right) \left\{\log(2)+ \log(1+ \frac{\theta - \frac{1}{2n+3}}{\frac{2}{n} - \frac{3}{2n+3}})\right\}\\
&+ \frac{1}{\pi \theta} \frac{1}{4} \int_{0}^{\theta - \frac{1}{2n+3}} \log(1+ \frac{-\theta + \frac{2}{n}}{\theta-s})\; ds\\
&+2\\
&\leq \frac{1}{4\pi} \left(\frac{1}{\pi} + \frac{1}{2}\right) \left\{\log(2)+ \log(1+ \frac{\frac{1}{2n+3}}{\frac{1}{2n}})\right\}\\
&+ \frac{1}{4\pi \theta} \left\{\left|-\theta + \frac{2}{n}\right| \log(\frac{\frac{1}{2n+3}}{\theta})\right\} +2\\
&\leq \frac{1}{4\pi} \left(\frac{1}{\pi} + \frac{1}{2}\right) 2\log(2) + \frac{1}{4\pi} \log(2) + 2\\
&\leq 2.2 \label{boundFD3plusMSbis}
\end{align}

on the triangle $s \in [\theta - \frac{1}{2n+3}, \frac{1}{2n+3}]$, we use a bound similar to $D_1^+$. The argument of the sine at the denominator always obeys $(s+t)\in [s, s-\theta]$ and we can thus write 

\begin{align}
F_{D_3^+\cap D_2^+,I}\leq \left|\int_{-\theta}^{0}  \frac{\sin((n+1)\pi(s+t))\sin((n+2)\pi(s+t))}{\sin(\pi(s+t))}\; dt\right| \leq \frac{(n+1)\pi \theta}{2}
\end{align}
which gives 
\begin{align}
\frac{1}{|e^{2\pi i \theta} - 1|}\int_{\theta - \frac{1}{2n+3}}^{\frac{1}{2n+3}} F_{D_3^+\cap D_2^+, I} \; ds&\leq \frac{1}{\pi \theta} \left|\frac{1}{2n+3} - \theta\right| \frac{(n+1)\pi \theta}{2}\\
&\leq \left(\frac{1}{2n+3} - \theta\right)\frac{n+1}{2}\leq \frac{1}{2} \label{intersectionD2D3Imag}\\
\end{align}
In the last line we used $\theta \in [\frac{1}{2n+3}, \frac{2}{2n+3}]$. 

Finally for $F_{I,D_2}^+$, whenever $\theta \geq \frac{2}{2n+3}$, we have 

\begin{align}
\frac{1}{|1-e^{2\pi i \theta}|}\int_{D_{2}^+} F_{I,D_2^+}(s)\; ds&\leq \frac{1}{\pi \theta} \int_{\theta - \frac{1}{2n+3}}^{\theta+\frac{1}{2n+3}} \frac{n+1}{2n}\pi + \frac{1}{4(2n+3)} \left(\frac{1}{\pi}+ \frac{1}{2}\right) \left(\frac{1}{s} + \frac{1}{s-\theta + \frac{2}{n}}\right)\; ds\\
& + \frac{1}{\pi \theta} \int_{\theta-\frac{1}{2n+3}}^{\theta + \frac{1}{2n+3}}\frac{1}{4}\left(-\log((s-\theta) + \frac{2}{n}) + \log(s)\right)\; ds\\
&\leq \frac{1}{\pi \theta}\left(\frac{n+1}{2n}\right)\pi \frac{2}{2n+3} \\
&+ \frac{1}{4(2n+3)}\left(\frac{1}{\pi}+ \frac{1}{2}\right) \left|\log(1+ \frac{2}{2n+3}\frac{1}{\theta - \frac{1}{2n+3}}) + \log(1+ \frac{2}{2n+3} \frac{1}{\frac{2}{n} - \frac{1}{2n+3}})\right|\\
&+ \frac{1}{4} \int_{\theta - \frac{1}{2n+3}}^{\theta + \frac{1}{2n+3}} \log(1+ \frac{\theta - \frac{2}{n}}{s-\theta + \frac{2}{n}})\\
&\leq 1 + \frac{1}{8}\left(\frac{1}{\pi}+ \frac{1}{2}\right) \left[\log(3)+ \log(2)\right]+ \frac{1}{4}\frac{1}{\pi \theta} |\theta - \frac{2}{n}| \log(\frac{\frac{1}{2n+3} + \frac{2}{n}}{\frac{2}{n} - \frac{1}{2n+3}})\\
&\leq 1+ \frac{1}{8}\left(\frac{1}{\pi}+ \frac{1}{2}\right) \left(\log(3)+ \log(2)\right)\\
&+ \frac{1}{4\pi} \log(1+ \frac{2}{2n+3} \frac{1}{\frac{2}{n} - \frac{1}{2n+3}})\\
&\leq 1+ \frac{1}{8}\left(\frac{1}{\pi}+ \frac{1}{2}\right) \left[\log(3)+ \log(2)\right] + \frac{1}{4\pi} \log(2)\\
&\leq  1.3\label{boundImaginaryFD2plusMSa}
\end{align}

When $\theta \leq \frac{2}{2n+3}$, as we have acccounted for the intersection $D_3^+\cap D_2^+$ in~\eqref{intersectionD2D3Imag}, we can just integrate on $[\frac{1}{2n+3}, \theta + \frac{1}{2n+3}]$. 

\begin{align}
\frac{1}{|1-e^{2\pi i \theta}|}\int_{D_{2}^+} F_{I,D_2^+}\; ds&\leq \frac{1}{\pi \theta} \int_{\frac{1}{2n+3}}^{\theta+\frac{1}{2n+3}} \frac{n+1}{2n}\pi + \frac{1}{4(2n+3)} \left(\frac{1}{\pi}+ \frac{1}{2}\right) \left(\frac{1}{s} + \frac{1}{s-\theta + \frac{2}{n}}\right)\; ds\\
& + \frac{1}{\pi \theta} \int_{\frac{1}{2n+3}}^{\theta + \frac{1}{2n+3}}\frac{1}{4}\left(-\log((s-\theta) + \frac{2}{n}) + \log(s)\right)\; ds\\
&\leq \frac{1}{\pi \theta}\left(\frac{n+1}{2n}\right)\pi \theta\\
&+ \frac{1}{4(2n+3)}\left(\frac{1}{\pi}+ \frac{1}{2}\right) \left|\log(1+(2n+3)\theta) + \log(1+ \frac{2}{2n+3} \frac{1}{\frac{2}{n} + \frac{1}{2n+3} - \theta})\right|\\
&+ \frac{1}{4}\frac{1}{\pi \theta} \int_{\frac{1}{2n+3}}^{\theta + \frac{1}{2n+3}} \log( 1+ \frac{\frac{2}{n}-\theta}{s})\\
&\leq 1 + \frac{1}{8}\left(\frac{1}{\pi}+ \frac{1}{2}\right) \left[\log(3)+ \log(2)\right]\\
&+ \frac{1}{4}\frac{1}{\pi \theta} \theta \log(6)\\
&\leq 1+ \frac{1}{8}\left(\frac{1}{\pi}+ \frac{1}{2}\right) \left(\log(3)+ \log(2)\right)\\
&+ \frac{1}{4\pi} \log(6)\\
&\leq 1+ \frac{1}{8}\left(\frac{1}{\pi}+ \frac{1}{2}\right) \left[\log(3)+ \log(2)\right] + \frac{1}{4\pi} \log(6)\\
&\leq 1.4\label{boundImaginaryFD2plusMSb}
\end{align}

Combining~\eqref{boundImaginaryFD1plusMS}, ~\eqref{intersectionD2D3Imag}, ~\eqref{boundFD3plusMS} and~\eqref{boundFD3plusMSbis} as well as~\eqref{boundImaginaryFD2plusMSa} and~\eqref{boundImaginaryFD2plusMSb}, we get 
\begin{align}
&\sum_{\lambda=1}^3\left|\int_{s\in D_\lambda^+} \frac{p(e^{2\pi i s})}{|1 - e^{2\pi i\theta}|}  \left|\Ip \left(\int_{0}^{-\theta}e^{2\pi i (n+1)t} \tilde{D}(e^{-2\pi i (s+t)})\; dt\right)\right|\; ds\right| \leq 7.7\label{}
\end{align}

Combining this last bound with~\eqref{boundSmallStripesplusRealP} gives the result of lemma~\ref{lemmaTroncSmallDomainsPositiveCons}.


\subsection{\label{sectionProofLemmaD3minusLarge}Proof of lemma~\ref{lemmaTroncD3minus} ($D_3^-$ large $|\tau - \alpha|$)}

We start with the real part. Recall that we have 
\begin{align}
F_{R,D_3^-} &\leq \frac{1}{4(2n+3)} \left(\frac{1}{\pi} + \frac{1}{\pi}\right) \left(4 + \frac{1}{1+s-\theta} + \frac{1}{-s}\right) + \frac{\theta}{2}\\
&\leq \left(4c + \frac{\theta}{2}\right) + \frac{c}{1+s-\theta}  + \frac{c}{-s} 
\end{align}

We first consider the case $\tau \geq 0$ (that is the unique atom is on the right of $D_3^-$). In this case we have 
\begin{align}
&\frac{1}{\pi \theta} \int_{-1/2}^{-1/2+\theta} \left(4c + \frac{\theta}{2}\right) \frac{C_1}{1+n(\tau-s)}\; ds + \frac{1}{\pi \theta} \int_{-1/2}^{-1/2+\theta} \left(\frac{c}{1+s-\theta} + \frac{c}{-s}\right)\frac{C_1}{1+ n (\tau-s)}\; ds\\
&\leq \frac{1}{\pi \theta} \left(4c + \frac{\theta}{2}\right) \frac{C_1}{n} \log(\frac{1/n + 1/2 -\theta + \tau}{1/n + 1/2 + \tau}) + \frac{1}{\pi \theta} \left|\log(1+s-\theta)  - \log(1/n + \tau -s)\right|_{-1/2}^{-1/2+\theta} \frac{1}{1+\tau - \theta+1/n} \frac{C_1}{n}\\
&+ \frac{1}{\pi \theta} \left|\log(-s)  - \log(1/n + \tau-s)\right|_{-1/2}^{-1/2+\theta} \frac{1}{\tau + 1/n}\frac{C_1}{n}\\
&\leq \frac{1}{\pi \theta} \left(4c + \frac{\theta}{2}\right) \frac{C_1}{n} \log(1+ \frac{\theta}{1/n + 1/2-\theta + \tau})\\
&+ \frac{c}{\pi \theta}\left|\log(\frac{1/2-\theta}{1/2})  - \log(\frac{1/n + \tau + 1/2 - \theta}{1/n + 1/2 + \tau})\right| \frac{C_1}{n} \frac{1}{\tau +1/n}\\
&+ \frac{c}{\pi \theta} \left|\log(\frac{1/2}{1/2-\theta}) - \log(\frac{1/n + \tau + 1/2 - \theta}{1/n + 1/2 + \tau})\right| \frac{1}{1+\tau - \theta + 1/n} \frac{C_1}{n}\\
&\stackrel{(a)}{\leq }\frac{1}{\pi } \left(4c + \frac{\theta}{2}\right) \frac{C_1}{n\tau} + \left(\frac{4c}{\pi \theta (1/2-\theta)}\right) \frac{C_1}{n\tau} + \frac{c}{\pi \theta \tau} \frac{C_1}{n} \frac{4}{(1/2-\theta)}\\
&\stackrel{(b)}{\leq } \frac{1}{\pi} \left(4c + \frac{1}{4}\right) \frac{C_1}{n\tau} + \left(\frac{16c'}{\pi} \frac{C_1}{n\tau}\right) + \frac{16c'}{\pi} \frac{C_1}{n\tau}\label{tmp01003}
\end{align}
In $(a)$ we use $\tau \leq 1/2$ as well as $1/n\leq 1$, $(b)$ follows from $(1/2 - \theta)\theta \geq \frac{1}{2n+3}(1/4)$ when $ \frac{1}{2n+3}\leq \theta \leq 1/2-\frac{1}{2n+3}$ and $n\geq 1$

Note that if $\tau>0$, we necessarily have $\tau> \frac{C_1-1}{n}$ because of the separation distance $\Delta$ being larger than $C_1/n$ and we therefore do not need to consider any correction accounting for intervals on which the bound on the eigenpolynonial would achieve the value $1$. We consider a maximum located at $\tau<-1/2$ which corresponds to taking into account the right tail of a bound on the eigenpolynonial with maximum located at $\tau>0$. 
\begin{align}
&\frac{1}{\pi \theta} \int_{-1/2}^{-1/2+\theta} \left(4c+ \frac{\theta}{2}\right) \frac{C_1}{1+n(s-\tau)}\; ds + \frac{1}{\pi \theta} \int_{-1/2}^{-1/2+\theta} \left(\frac{c}{1+s-\theta} + \frac{c}{-s}\right) \frac{C_1}{1+n(s-\tau)}\; ds\\
&\leq \frac{1}{\pi \theta} \left(4c + \frac{\theta}{2}\right) \frac{C_1}{n} \log(\frac{1/n - 1/2 + \theta - \tau}{1/n - 1/2 - \tau})\label{term1D3minusLemmaTroncStep1}\\
&+ \frac{c}{\pi \theta} \left|-\log(1+s-\theta) + \log(1/n + s-\tau)\right|_{-1/2}^{-1/2+\theta} \frac{C_1}{n} \frac{1}{1-\theta + \tau - \frac{1}{n}}\label{term2D3minusLemmaTroncStep1}\\
&+ \frac{c}{\pi \theta} \left|-\log(-s) + \log(1/n + s-\tau)\right|_{-1/2}^{-1/2+\theta} \frac{C_1}{n} \frac{1}{\tau - \frac{1}{n}}\label{term3D3minusLemmaTroncStep1}
\end{align}

Each of the terms above can be bounded as follows. Together with the bounds on $\tau<0$ we also extend those bounds from $|\tau|$ to $|1-\tau|$
\begin{align}
\eqref{term1D3minusLemmaTroncStep1}\leq \frac{1}{\pi \theta}\left(4c+ \frac{\theta}{2}\right) \frac{C_1}{n}\log(\frac{1/n - 1/2 + \theta - \tau}{1/n - 1/2 - \tau}) &\leq \frac{1}{\pi \theta} \left(4c + \frac{\theta}{2}\right)\frac{C_1}{n} \log(1+ \frac{\theta}{1/n - 1/2 - \tau})\\
&\stackrel{(a)}{\leq} \frac{1}{\pi} \left(4c' + 1/2\right) \frac{C_1}{n} \log(2n\tau)
\end{align}
$(a)$ follows from $|\tau|\geq 1/2\geq \theta$. To turn this bound into a bound on any tail arising from an eigenpolynomial with a peak centered on $\tau>1/4$, we replace $\tau$ with $1-\tau'$ and derive a bound in $\tau'$ (distance to the origin) 
\begin{align}
\eqref{term1D3minusLemmaTroncStep1} (1-\tau') &\leq \frac{1}{\pi} \left(4c' + 1/2\right)\frac{C_1}{n}  \log(2n(1-\tau))\\
&\leq \frac{1}{\pi} \left(4c' + 1/2\right) \frac{C_1}{n\tau} \log(2n(1-\tau')|\tau'|/\tau')\\
&\leq \frac{1}{\pi} \left(4c' + 1/2\right) \frac{C_1}{n\tau} \log(2n\tau') + \log(\frac{1-\tau'}{\tau'})\frac{C_1}{n}\frac{1}{\pi} \left(4c'  + 1/2\right)\\
&\stackrel{(a)}{\leq} \frac{1}{\pi} \left(4c' + 1/2\right) \frac{C_1}{n\tau'} \log(2n\tau') + \frac{C_1}{n\tau'} \frac{1}{\pi} \left(4c' + 1/2\right)\label{tmp01002}
\end{align}
$(a)$ follows from $1-\tau'\geq \tau'$.

For~\eqref{term2D3minusLemmaTroncStep1} and~\eqref{term3D3minusLemmaTroncStep1}, we have 
\begin{align}
\eqref{term2D3minusLemmaTroncStep1}+\eqref{term3D3minusLemmaTroncStep1}&\leq \frac{c}{\pi \theta} \left|-\log(\frac{1/2}{1/2 - \theta}) + \log(\frac{1/n - \frac{1}{2} + \theta - \tau}{1/n - 1/2 - \tau})\right|\frac{C_1}{n} \frac{1}{1-\theta + \tau - \frac{1}{n}}\\
&+ \frac{c}{\pi \theta} \left|\log(\frac{1/2 - \theta}{1/2}) + \log(\frac{1/n - 1/2 + \theta - \tau}{1/n - 1/2 - \tau})\right|\frac{C_1}{n} \frac{1}{\tau - 1/n}\\
&\leq \frac{c}{\pi \theta} \left|\log(1+ \frac{-1 + 1/n+\theta -\tau}{1/2}\vee \frac{1 - 1/n - \theta + \tau}{1/n - 1/2 + \theta - \tau})\right| \frac{C_1}{n} \frac{1}{1-\theta+\tau - \frac{1}{n}}\\
&+ \frac{c}{\pi \theta} \left|\log(1+ \frac{-1 + 1/n+\theta -\tau}{1/2-\theta}\vee \frac{1 - 1/n - \theta + \tau}{1/n - 1/2- \tau})\right| \frac{C_1}{n} \frac{1}{1-\theta+\tau - \frac{1}{n}}\\
&+ \frac{c}{\pi \theta} \frac{C_1}{n} \frac{2}{\tau} \left|\log(1+ \frac{\theta}{1/2- \theta}) + \log(1+ \frac{\theta}{1/n - 1/2 - \tau})\right|\\
&\leq \frac{C_1}{n} \left\{\left(\frac{2c'}{\pi} + \frac{c'}{\pi |\tau|}\right) +  \left( \frac{4c'}{\pi} + \frac{c'}{\pi |\tau|}\right)\right\}\\
&+ \frac{2}{\tau} \frac{C_1}{n} \frac{c}{\pi (1/2-\theta)} + \frac{2}{\tau} \frac{C_1}{n}\frac{c}{\pi} \log(n|\tau|)\\
&\leq \frac{C_1}{n\tau} \frac{c'}{\pi}12
\end{align}
Now to turn this bound into a bound on eigenpolynomials having their maximum centered after $1/4$, we write 
\begin{align}
(\eqref{term2D3minusLemmaTroncStep1}+\eqref{term3D3minusLemmaTroncStep1})(1-\tau') \leq \frac{C_1}{n(1-\tau')} \frac{c'}{\pi}12 \leq \frac{C_1}{n|\tau'|_{\mod 1}} \frac{12c'}{\pi}.\label{tmp01001}
\end{align}

The total bound for this configuration is thus given by 

\begin{align}
&\eqref{tmp01001} + \eqref{tmp01002}  + \eqref{tmp01003}\\
&\leq \frac{C_1}{n|\tau|} \left(\frac{4c}{\pi} + \frac{1}{4} + \frac{32c'}{\pi} + \log(2) + \frac{1}{4}(4c' + 1/2) + \frac{12c'}{\pi} \right) +\frac{C_1}{n|\tau|}\log(n|\tau|) \left(4c'+1/2\right)\frac{1}{\pi}\\
&\leq 4.2 \frac{C_1}{n|\tau|} + 0.5 \frac{C_1}{n|\tau|}\log(n|\tau|)\label{D3minusConfig1}
\end{align}

When $\tau + \frac{C_1-1}{n}\geq -1/2$, we control the integral as 
\begin{align}
&\frac{1}{\pi \theta} \int_{\tau + \frac{C_1-1}{n}}^{-1/2+\theta} (4c+ \frac{\theta}{2}) \frac{C_1}{1+n (s-\tau)}\; ds + \frac{1}{\pi \theta}\int_{\tau + \frac{C_1-1}{n}}^{-1/2+\theta} \left(\frac{c}{1+s-\theta} + \frac{c}{-s}\right) \frac{C_1}{1+n(s-\tau)}\; ds\label{tmpTauAbovelemmaTronc001}\\
&+ \frac{1}{\pi \theta} \int_{-1/2}^{\tau + \frac{C_1-1}{n}} \left(4c + \frac{\theta}{2}\right) + \left(\frac{c}{1+s-\theta} + \frac{c}{-s}\right)\; ds \label{tmpTauAbovelemmaTronc002}
\end{align}
We start by controling~\eqref{tmpTauAbovelemmaTronc001}. 
\begin{align}
\eqref{tmpTauAbovelemmaTronc001}&\leq \frac{1}{\pi \theta} \left(4c + \frac{\theta}{2}\right) \frac{C_1}{n} \log(\frac{1/n - 1/2 + \theta - \tau}{C_1/n})\\
&+ \frac{c}{\pi \theta} \left|-\log(1+s-\theta) + \log(1/n +s-\tau)\right|_{\tau + \frac{C_1-1}{n}}^{-1/2+\theta} \frac{1}{1-\theta + \tau - 1/n}\\
&+ \frac{c}{\pi \theta}\left|-\log(-s) + \log(1/n + s-\tau)\right|_{\tau + \frac{C_1-1}{n}}^{-1/2+\theta}\frac{1}{\tau - 1/n}\\
&\stackrel{(a)}{\leq }  \frac{1}{\pi}\left(4c' + 1/2\right)\frac{C_1}{n} \log(2\frac{|\tau|}{C_1})\label{A704}\\
&+ \frac{c}{\pi \theta} \left|-\log(\frac{1/2}{1+\tau + \frac{C_1-1}{n} - \theta}) + \log(\frac{1/n  - 1/2 + \theta - \tau}{C_1/n})\right| \frac{1}{1 - \theta + \tau - 1/n} \frac{C_1}{n}\label{term1D3minusaddOn1lemmaTronc}\\
&+ \frac{c}{\pi \theta} \left|-\log(\frac{1/2 - \theta}{-\tau - \frac{C_1-1}{n}}) + \log(\frac{1/n - 1/2 + \theta - \tau}{C_1/n})\right| \frac{2}{\tau} \frac{C_1}{n}\label{term2D3minusaddOn1lemmaTronc}
\end{align}
In $(a)$, we use $\theta \leq |\tau|$ which follows from $\tau \geq 1/2$. To control~\eqref{term2D3minusaddOn1lemmaTronc}, we simply write 
\begin{align}
\eqref{term2D3minusaddOn1lemmaTronc}  \stackrel{a}{\leq } \frac{c}{\pi \theta} \left(\log(\frac{1/2}{1/2 - \theta}) + \log(2\frac{|\tau|n}{C_1})\right)\frac{2}{\tau}\frac{C_1}{n}\label{boundA706}
\end{align}
In the line above, $(a)$ follows from $-1/2+\theta < 0$ and $|\tau|>2/n$. 

To control~\eqref{term1D3minusaddOn1lemmaTronc}, we make the distinction between the case $\theta \leq (1/2 - \frac{C_1}{n})\frac{1}{2}$ and $\theta \geq \frac{1}{2}(1/2 - \frac{C_1}{n})$. In the latter case, we cancel the prefactor $(\tau  - \theta -1/n +1)^{-1}$ by using $\log(1+x)\leq x$
\begin{align}
\eqref{term1D3minusaddOn1lemmaTronc} &\leq \frac{c}{\pi \theta} \left(\log(1 + \frac{-\tau + \theta + 1/n -1}{1+\tau + \frac{C_1-1}{n} - \theta}\vee \frac{\tau - \theta - 1/n +1}{C_1/n})\right) \frac{1}{1 - \theta + \tau - 1/n} \frac{C_1}{n} \\
&+ \frac{c}{\pi \theta}\left(\log(1+ \frac{1 - 1/n - \theta + \tau}{1/n - 1/2 + \theta - \tau}\vee \frac{-1+1/n + \theta - \tau}{1/2})\right) \frac{1}{1 - \theta + \tau - 1/n} \frac{C_1}{n}\\
&\stackrel{(a)}{\leq } \frac{c}{\pi \theta (1/2-\theta)}\frac{C_1}{n} + \frac{c}{\pi (1/2)(1/2 - \frac{C_1}{n}) } + \frac{n}{C_1} \frac{c}{\pi (1/2)(1/2 - \frac{C_1}{n})} \frac{C_1}{n} + 2\frac{c}{\pi \theta}\frac{C_1}{n}\\
&\stackrel{(b)}{\leq }  \frac{4c'}{\pi}\frac{C_1}{n} + \frac{16c}{\pi} \frac{C_1}{n}+ 2\frac{c'}{\pi}\frac{C_1}{n}\label{A705b}
\end{align}
In $(a)$, we use $\tau+ \frac{C_1-1}{n}\geq -1/2$. In $(b)$, we use $\inf_{\frac{1}{2n+3}\leq \theta \leq 1/2-\frac{1}{2n+3}} \theta (1/2-\theta) \geq (1/2 - 1/(2n+3))\frac{1}{2n+3} \geq \frac{1}{4}(2n+3)^{-1}$ as well as $n\geq 4C_1$. 

When $\theta \leq (1/2 - C_1/n)(1/2)$, we use 
\begin{align}
\frac{c}{\pi \theta}\frac{1}{1- \theta + \tau -1/n} \leq \frac{c'}{\pi} \frac{1}{1/2 - \theta - \frac{C_1}{n}} \leq \frac{c'}{\pi} \frac{2}{1/2 - \frac{C_1}{n}} \leq \frac{4c'}{\pi}, \quad \text{as soon as $n\geq 4C_1$. }
\end{align}
Then we have 
\begin{align}
\eqref{term1D3minusaddOn1lemmaTronc} \leq  \frac{8c'}{\pi} \log(\frac{2n|\tau|}{C_1})  \frac{C_1}{n}\label{A705a}
\end{align}

To control~\eqref{tmpTauAbovelemmaTronc002}, we write 
\begin{align}
\eqref{tmpTauAbovelemmaTronc002}&\leq \frac{1}{\pi \theta} \left(4c + \frac{\theta}{2}\right) \frac{C_1-1}{n} + \frac{c}{\pi \theta} \log(1+ \frac{C_1-1}{n}\frac{1}{1/2-\theta})+ \frac{c}{\pi \theta} \log(1+ \frac{C_1-1}{n} \frac{1}{1/2 - \frac{C_1-1}{n}})\\
&\leq \left(\frac{4c'}{\pi} + \frac{1}{2\pi}\right) \frac{C_1-1}{n} +\frac{c'}{\pi} \frac{C_1-1}{n} \label{boundA700}
\end{align}
The last line holds as soon as $n\geq 4(C_1-1)$. 

Grouping the bounds above, we get 
\begin{align}
&\eqref{boundA700} + \left(\eqref{A705a}\vee\eqref{A705b} \right) + \eqref{boundA706} + \eqref{A704} \\
&\leq \left(5\frac{c'}{\pi} + \frac{1}{2\pi}\right) \frac{C_1}{n|\tau|} + \left(\frac{6c'}{\pi} + \frac{16c}{\pi}\right)\frac{C_1}{n|\tau|}+ \left(\frac{8c'}{\pi} + \frac{c'}{\pi} + \frac{4c' + 1/2}{\pi}\right)\frac{C_1}{n|\tau|}\log(2\frac{|\tau|}{C_1})\\
&\leq 1.1 \left(\frac{C_1}{n|\tau|} + \frac{C_1}{n|\tau|}\log(\frac{2n|\tau|}{C_1})\right).\label{bound1TowardsMax}
\end{align}

We now control the integral on $D_3^-$ when the eigenpolynomial has a peak located on that precise subdomain. In this case, we split the integral as 
\begin{align}
&\frac{1}{\pi \theta} \int_{-1/2}^{\tau - \frac{C_1-1}{n}} \left((4c+ \frac{\theta}{2}) + (\frac{c}{1+s-\theta} + \frac{c}{-s})\right)\frac{C_1}{1+n(\tau-s)}\; ds\label{term1LemmaTroncCentralSpikeD3minus}\\
&+ \frac{1}{\pi \theta} \int_{\tau - \frac{C_1-1}{n}}^{\tau + \frac{C_1-1}{n}} \left((4c+ \frac{\theta}{2}) + \left(\frac{c}{1+s-\theta} + \frac{c}{-s}\right)\right) \; ds\label{term2LemmaTroncCentralSpikeD3minus}\\
&+ \frac{1}{\pi \theta} \int_{\tau + \frac{C_1-1}{n}}^{-1/2+\theta} \left((4c+ \frac{\theta}{2}) + \left(\frac{c}{1+s-\theta} + \frac{c}{-s}\right)\right)\frac{C_1}{1+n(s-\tau)}\; ds\label{term3LemmaTroncCentralSpikeD3minus}
\end{align}

we successively bound each of the three terms above. Starting with~\eqref{term1LemmaTroncCentralSpikeD3minus}, when $\theta \leq 1/4$ we have $1+ \tau + 1/n - \theta \geq 1/4$ and we can write 
\begin{align}
\eqref{term1LemmaTroncCentralSpikeD3minus}&\stackrel{a}{\leq}\frac{C_1}{n} \log(\frac{C_1/n}{1/n + \tau +1/2})\frac{1}{\pi \theta} \left(4c+ \frac{\theta}{2}\right)\\
&+ \frac{c}{\pi \theta}\left|\log(1+s-\theta)  - \log(1/n + \tau -s)\right|_{-1/2}^{\tau - \frac{C_1-1}{n}} \frac{1}{1+\tau+1/n - \theta} \frac{C_1}{n}\label{tmpBound0008}\\
&+ \frac{c}{\pi \theta} \left|\log(-s) - \log(1/n + \tau - s)\right|_{-1/2}^{\tau - \frac{C_1-1}{n}} \frac{1}{-\tau - 1/n}\\
&\stackrel{b}{\leq} \frac{C_1}{n} \log(\frac{n}{C_1}) \frac{1}{\pi} (4c' + 1/2)\\
&+ \frac{c}{\pi \theta} \left[\log(\frac{1+\tau - \theta - \frac{C_1-1}{n}}{1/2 - \theta}) - \log(\frac{C_1/n}{1/n + \tau +1/2})\right] \frac{1}{1+\tau + 1/n - \theta} \frac{C_1}{n}\\
&+ \frac{c}{\pi \theta} \left|\log(\frac{-\tau + \frac{C_1-1}{n}}{1/2})  - \log(\frac{C_1/n}{1/n + \tau + 1/2})\right| \frac{1}{|\tau - 1/n|} \frac{C_1}{n}\\
&\stackrel{c}{\leq} \left(\frac{C_1}{n|\tau|} + \frac{C_1}{n|\tau|} \log(\frac{n|\tau|}{C_1})\right)\frac{1}{\pi} \left(4c' + 1/2\right)\\
&+ \frac{C_1}{n} \frac{4c}{\pi \theta (1/2-\theta)} + \frac{4c}{\pi \theta } \log(\frac{n}{C_1}) \frac{C_1}{n}\\
&+ \frac{2c'}{\pi} \frac{C_1}{n|\tau|} \left\{\log(1+ \frac{-\tau - \frac{1}{n}}{1/n + \tau + 1/2}) + \log(1+ \frac{-\tau - 1/n}{C_1/n})\right\}\\
&\stackrel{d}{\leq} \left(\frac{C_1}{n|\tau|} + \frac{C_1}{n|\tau|} \log(\frac{n|\tau|}{C_1})\right) \frac{1}{\pi} \left(4c' + 1/2\right)\label{A733a}\\
&+ \frac{C_1}{n}  \frac{16c'}{\pi} +  \left(\frac{4c'}{\pi}\log(\frac{n|\tau|}{C_1}) \frac{C_1}{n|\tau|} + \frac{4c'}{\pi} \frac{C_1}{n|\tau|}\right)\label{tmpBound0007}\\
&+ \frac{C_1}{n|\tau|} \frac{2c'}{\pi} \left(\log(3n|\tau|) + \log(\frac{3n|\tau|}{C_1})\right).\label{A733}
\end{align}
$(b)$ follows from $\tau - \frac{C_1-1}{n}\geq -1/2$ and $\tau - \frac{C_1-1}{n}\leq -1/2+\theta$. $d$ follows from $\inf_{\theta\in [\frac{1}{2n+3}, \frac{1}{2}-\frac{1}{2n+3}]}\frac{1}{\theta(1/2 - \theta)} \geq 4(2n+3)$ as soon as $n\geq 1$.

When $\theta \geq 1/4$, we use the log to cancel the prefator in~\eqref{tmpBound0008} and replace~\eqref{tmpBound0007} by the following bound  
\begin{align}
&\frac{c}{\pi \theta} \left[\log(\frac{1+\tau - \theta - \frac{C_1-1}{n}}{1/2 - \theta}) - \log(\frac{C_1/n}{1/n + \tau +1/2})\right] \frac{1}{1+\tau + 1/n - \theta} \frac{C_1}{n}\\
&\leq \frac{4c}{\pi} \left(3C_1+1\right)\\
&\leq \frac{6C_1 c'}{\pi n} + \frac{4c'}{\pi (2n+3)}.\label{A736}
\end{align}

For the last term, we get 
\begin{align}
\eqref{term3LemmaTroncCentralSpikeD3minus}&\leq \frac{1}{\pi \theta}  \left(4c + \frac{\theta}{2}\right) \frac{C_1}{n} \log(\frac{1/n - 1/2 + \theta - \tau}{C_1/n})\\
&+ \frac{c}{\pi \theta} \frac{C_1}{n} \left|-\log(1+s-\theta) + \log(1/n + s-\tau)\right|_{\tau + \frac{C_1-1}{n}}^{-1/2 + \theta} \frac{1}{\tau - \frac{1}{n} - \theta + 1}\\
&+ \frac{c}{\pi \theta} \frac{C_1}{n} \left|-\log(-s) + \log(1/n + s-\tau)\right|_{\tau + \frac{C_1-1}{n}}^{-1/2+\theta} \frac{1}{|-\tau + 1/n|}\\
&\stackrel{a}{\leq } \frac{1}{\pi} \left(4c' + 1/2\right) \frac{C_1}{n} \log(\frac{n}{C_1})\\
&+ \frac{c}{\pi \theta} \frac{C_1}{n} \left|\left(-\log(\frac{1/2}{1+\tau+\frac{C_1-1}{n} - \theta}) + \log(\frac{1/n - 1/2+\theta - \tau}{C_1/n})\right)\right| \frac{1}{\tau - \frac{1}{n} - \theta +1}\\
&+ \frac{c}{\pi \theta} \frac{C_1}{n} \left|-\log(\frac{1/2 - \theta}{-\tau - \frac{C_1-1}{n}}) + \log(\frac{1/n - 1/2 + \theta - \tau}{C_1/n})\right| \frac{1}{\tau - 1/n}\\
&\stackrel{b}{\leq } \frac{1}{\pi} \left(4c' + 1/2\right) \left(\frac{C_1}{n|\tau|} + \frac{C_1}{n|\tau|} \log(\frac{n|\tau|}{C_1})\right)\label{743a}\\
&+ \frac{c}{\pi \theta} \frac{C_1}{n} \left|-\log(\frac{1/2}{1+\tau + \frac{C_1-1}{n} - \theta}) + \log(\frac{1/n - 1/2 + \theta - \tau}{C_1/n})\right| \frac{1}{\tau - 1/n - \theta + 1}\label{tmpLemmaTroncD3minus0020}\\
&+ \frac{c}{\pi \theta} \frac{C_1}{n} \left(\log(2|\tau| n) + \log(\frac{2n|\tau|}{C_1})\right) \frac{1}{|\tau|}\label{743c}
\end{align}

$(b)$ follows from $\tau <0$. To bound~\eqref{tmpLemmaTroncD3minus0020}, we distinguish between $\theta \geq 1/4$ and $\theta\leq 1/4$. In the first case, we have 
\begin{align}
\eqref{tmpLemmaTroncD3minus0020}& \leq \frac{4c}{\pi} \frac{C_1}{n} \left(\log(1+ \frac{1/n - 1 + \theta - \tau}{1/2}\vee \frac{1/2 - 1/n - \theta + \tau}{1/n - 1/2 + \theta - \tau})\right) \frac{1}{|\tau - \frac{1}{n} - \theta +1|}\\
&+ \frac{4c}{\pi} \frac{C_1}{n} \left(\log(1+ \frac{1+\tau - 1/n - \theta}{C_1/n}\vee \frac{-1 + \theta - \tau + 1/n}{1+\tau + \frac{C_1-1}{n} - \theta})\right) \frac{1}{|\tau - \frac{1}{n} - \theta +1|}\\
&\stackrel{a}{\leq } \frac{4c}{\pi}\frac{C_1}{n} \left(2+ \frac{n}{C_1}\right) + \frac{4c}{\pi} \frac{C_1}{n} \left(\frac{n}{C_1} + \frac{1}{1/2-\theta}\right)\\
&\leq \frac{4c'}{\pi (2n+3)}3 + \frac{4c'}{\pi (2n+3)} + \frac{4c'}{\pi}\frac{C_1}{n}\label{743bFirst}
\end{align}
In the inequalities above, $(a)$ relies on $-1/2\leq \tau + \frac{C_1-1}{n} \leq -1/2+\theta$. When $\theta \leq 1/4$, we have 
\begin{align}
\eqref{tmpLemmaTroncD3minus0020}& \leq \frac{4c'}{\pi} \frac{C_1}{n} \left[\log(\frac{1}{2}\frac{1}{1/2 - \theta + \frac{C_1-1}{n}}) + \log(\frac{2n|\tau|}{C_1})\right]\frac{1}{1/4-1/n}\\
&\leq \frac{4c'}{\pi}\frac{C_1}{n} \left[\log(2) + \log(\frac{2n|\tau|}{C_1})\right]8\label{743bSecond}
\end{align}
The last line holds as soon as $n\geq 8$.

Finally the central integral~\eqref{term2LemmaTroncCentralSpikeD3minus} can be bounded as 
\begin{align}
\frac{1}{\pi \theta} \int_{\tau - \frac{C_1-1}{n}\vee -1/2}^{\tau + \frac{C_1-1}{n}\wedge -1/2+\theta} \left((4c+ \frac{\theta}{2}) + \frac{c}{1+s-\theta} + \frac{c}{-s}\right)\; ds&\leq \left(\frac{4c'}{\pi} + \frac{1}{2\pi}\right) \frac{2(C_1-1)}{n} + \frac{2c}{\pi \theta} \log(1+ 2\frac{C_1-1}{n}\frac{1}{1/2-\theta})\\
&\leq \left(\frac{4c'}{\pi} + \frac{1}{2\pi}\right) \frac{2C_1}{n} + \frac{8c'}{\pi} \frac{C_1}{n}\label{753}
\end{align}

Combining the bounds above, we get 
\begin{align}
&\left(\eqref{A733a} + \eqref{tmpBound0007}+\eqref{A733} \vee \eqref{A736}\right) + \left( \eqref{743a}+\eqref{743c} + \eqref{743bFirst}\vee \eqref{743bSecond} \right)  + \eqref{753}\\
&\leq \left(\frac{66c'}{\pi} \log(2) + \frac{2}{\pi} + \frac{44c'}{\pi} + \frac{4c'}{\pi}\log(3)\right)\frac{C_1}{n|\tau|} + \frac{C_1}{n|\tau|} \log(\frac{n|\tau|}{C_1}) \left(\frac{42c'}{\pi} + \frac{1}{\pi}\right) + \frac{8c'}{\pi} \frac{C_1}{n|\tau|} \log(n|\tau|)\\
&\leq 7\frac{C_1}{n|\tau|} + 3.1\frac{C_1}{n|\tau|}\log(\frac{n|\tau|}{C_1}) + \frac{C_1}{n|\tau|}\log(n|\tau|). \label{bound2towardsMax}
\end{align}

Taking the maximum of~\eqref{bound1TowardsMax} and~\eqref{bound2towardsMax}, when $\tau< -1/2+\theta$, we thus get the bound
\begin{align}
\eqref{bound1TowardsMax}\vee \eqref{bound2towardsMax}  =  \eqref{bound2towardsMax}
\end{align}

To this bound we add the symmetric contribution, 
\begin{align}
\eqref{bound2towardsMax} + \eqref{tmp01003}(1-\tau) &\leq 7\frac{C_1}{n|\tau|} + 3.1\frac{C_1}{n|\tau|}\log(\frac{n|\tau|}{C_1}) + \frac{C_1}{n|\tau|}\log(n|\tau|) + 2.3\frac{C_1}{n|\tau|}\\
&\leq 10\frac{C_1}{n|\tau|} + 3.1\frac{C_1}{n|\tau|}\log(\frac{n|\tau|}{C_1}) + \frac{C_1}{n|\tau|}\log(n|\tau|) \label{D3minusConfig2}
\end{align}
since $1-|\tau| \geq |\tau|$ and $\log(x)/x$ is a decreasing function for $x\geq 10$.

Before controling the imaginary part, we consider the case $-1/2+\theta<\tau+\frac{C_1-1}{n}<0$. 
\begin{align}
&\frac{1}{\pi \theta} \int_{-1/2}^{-1/2+\theta} \left(\left(4c+ \frac{\theta}{2}\right) + \frac{c}{1+s-\theta}+ \frac{c}{-s}\right) \frac{C_1}{1+n(\tau-s)}\; ds\\
&\stackrel{a}{\leq } \left(4c' + 1/2\right) \frac{1}{\pi} \frac{C_1}{n} \log(\frac{1/n + \tau + 1/2 - \theta}{1/n + \tau + 1/2})\\
&+ \frac{c}{\pi \theta} \left|\log(1+s-\theta)  - \log(1/n + \tau-s)\right|_{-1/2}^{-1/2+\theta} \frac{C_1}{n} \frac{1}{1+\tau - \theta+1/n}\\
&+ \frac{c}{\pi \theta} \left|\log(-s) - \log(1/n + \tau-s)\right|_{-1/2}^{-1/2+\theta} \frac{C_1}{n}\frac{1}{-\tau - 1/n}\\
&\stackrel{b}{\leq } \left(4c' + \frac{1}{2}\right) \frac{1}{\pi} \frac{C_1}{n} \log(n)\\
&+ \frac{c}{\pi \theta} \left|\log(\frac{1/2}{1/2-\theta}) + \log(\frac{1/n + \tau + 1/2 - \theta}{1/n + \tau + 1/2})\right| \frac{2C_1}{n}\\
&+ \frac{c}{\pi \theta} \left|\log(\frac{1/2}{1/2-\theta}) + \log(\frac{1/n + \tau + \frac{1}{2} - \theta}{1/n + \tau + 1/2})\right| \frac{C_1}{n}\frac{1}{|\tau|}\\
&\stackrel{c}{\leq } (4c' + 1/2) \frac{1}{\pi}\frac{C_1}{n}\log(n)\\
&+ \frac{4c'}{\pi} \log(n) \frac{C_1}{n}\\
&+ \frac{c}{\pi \theta} \frac{C_1}{n|\tau|} \left|\log(1+ \frac{-1/n - \tau}{1/n + \tau + 1/2 - \theta}) + \log(1 + \frac{-\tau - 1/n}{1/n + \tau + 1/2})\right|\\
&\leq \left(4c'+1/2\right)\frac{1}{\pi} \frac{C_1}{n} \log(n)\\
&+ \frac{4c'}{\pi} \log(n) \frac{C_1}{n} \\
&+ \frac{c'}{\pi} \frac{2C_1}{n|\tau|} \log(3n|\tau|).
\end{align}
In $(b)$, we use $\tau \geq -1/2 + \theta$.

When $\tau - \frac{C_1-1}{n} \leq -1/2 + \theta$, the integral turns into
\begin{align}
&\frac{1}{\pi \theta} \int_{-1/2}^{\tau - \frac{C_1-1}{n}} \left\{(4c+ \frac{\theta}{2}) + \frac{c}{1+s-\theta} + \frac{c}{-s}\right\} \frac{C_1}{1+n(\tau-s)}\; ds\label{term1tmpD3minus}\\
&+ \frac{1}{\pi \theta} \int_{\tau - \frac{C_1-1}{n}}^{-1/2+\theta}\left\{(4c+ \frac{\theta}{2}) + \frac{c}{1+s-\theta} + \frac{c}{-s}\right\}\; ds\label{term1tmpD3minusb}
\end{align}
We control each of those terms below. 
\begin{align}
\eqref{term1tmpD3minus}&\stackrel{(a)}{\leq } \left(4c' + \frac{1}{2}\right) \frac{1}{\pi} \frac{C_1}{n}\log(\frac{C_1/n}{1/n + \tau + 1/2})\\
&+ \frac{1}{\pi \theta} \frac{C_1}{n} \left|\log(\frac{1+\tau - \frac{C_1-1}{n} - \theta}{1/2 - \theta})  - \log(\frac{C_1/n}{1/n + \tau +1/2})\right| \frac{1}{1+1/n + \tau - \theta}\\
&+ \frac{c}{\pi \theta} \frac{C_1}{n} \left|\log(\frac{-\tau + \frac{C_1-1}{n}}{1/2}) - \log(\frac{C_1/n}{1/n + \tau + 1/2})\right| \frac{1}{-\tau  - 1/n}\\
&\stackrel{(b)}{\leq } \left(4c' + 1/2\right) \frac{1}{\pi} \frac{C_1}{n} \log(\frac{n}{C_1})\\
&+ \frac{c'}{\pi} 2\frac{C_1}{n} \left(\log(n)+ \log(\frac{n}{C_1})\right) \\
&+ \frac{c'}{\pi}\frac{C_1}{n} \left(\log(1+ \frac{-\tau - 1/n}{1/2+ \tau + 1/n}) + \log(1+ \frac{-\tau - 1/n}{C_1/n})\right)\\
&\stackrel{(c)}{\leq } (4c' + 1/2)\frac{1}{\pi} \frac{C_1}{n} \log(\frac{n}{C_1})\\
&+ \frac{c'}{\pi} \frac{4C_1}{n} \left(\log(n)+ \log(\frac{n}{C_1})\right)\\
&+ \frac{c'}{\pi} \frac{C_1}{n}\left( \log(3n|\tau|) + \log(\frac{3n|\tau|}{C_1})\right)\label{firstTermbetween1/2and0} 
\end{align}
\begin{align}
\eqref{term1tmpD3minusb}&\leq \left(4c' + 1/2\right)\frac{1}{\pi} \frac{C_1-1}{n} + \frac{2c}{\pi \theta} \log(1+ \frac{C_1-1}{n} \frac{1}{1/2-\theta})\\
&\leq (4c' + 1/2) \frac{1}{\pi} \frac{C_1-1}{n} + \frac{8c'}{\pi}\frac{C_1-1}{n}. \label{secondTermbetween1/2and0} 
\end{align}

Grouping~\eqref{firstTermbetween1/2and0} and~\eqref{secondTermbetween1/2and0} and adding at $1-\tau$ as in the case $\tau>0$ gives 

\begin{align}
&\eqref{firstTermbetween1/2and0} + \eqref{secondTermbetween1/2and0} +\eqref{tmp01001} +  \eqref{tmp01002}\\
& \leq  \frac{C_1}{n|\tau'|_{\mod 1}} \frac{12c'}{\pi} + \frac{1}{\pi} \left(4c' + 1/2\right) \frac{C_1}{n\tau'} \log(2n\tau') + \frac{C_1}{n\tau'} \frac{1}{\pi} \left(4c' + 1/2\right)\\
&+ \frac{C_1}{n|\tau|}\log(\frac{n|\tau|}{C_1}) \left((4c' +1/2)(5\pi) + c'\pi\right) + \frac{C_1}{n|\tau|}\log(n|\tau|)\frac{5c'}{\pi} + \frac{C_1}{n|\tau|} \left(\frac{c'}{\pi}2\log(3) + (12c' + 1/2)/pi\right)\\
&\leq 2.2 \frac{C_1}{n|\tau|}\log(\frac{n|\tau}{C_1}) + \frac{C_1}{n|\tau|}\log(n|\tau|)0.8 + \frac{C_1}{n|\tau|} 2.6.\label{D3minusConfig3}
\end{align}

Taking the maximum over the configurations gives 
\begin{align}
\eqref{D3minusConfig1}\vee\eqref{D3minusConfig2} \vee \eqref{D3minusConfig3} \leq 10\frac{C_1}{n|\tau|} + 3.1\frac{C_1}{n|\tau|}\log(\frac{n|\tau|}{C_1}) + \frac{C_1}{n|\tau|}\log(n|\tau|)\label{boundD3minusRealPartAllFW}
\end{align}

We now deal with the imaginary part of $D_3^-$.  Recall that we have 
\begin{align}
F_{D_3^-,I}\leq F_{D_3^-,R} + 2\log(2) + \frac{1}{4} \log(\frac{1}{-s}) + \frac{1}{4} \log(\frac{1}{1+s-\theta})
\end{align}
We will also use the tighter bound 
\begin{align}
F_{D_3^-,I}\leq F_{D_3^-,R} + \frac{1}{4} \log(\frac{1/2}{-s}) + \frac{1}{4} \log(\frac{1/2}{1+s-\theta})
\end{align}

We start by assuming $\tau>0$ , we get 
\begin{align}
&\frac{1}{\pi \theta} \int_{-1/2}^{-1/2+\theta} \left(2\log(2) + \frac{1}{4}\log(\frac{1}{-s}) + \frac{1}{4}\log(\frac{1}{1+s-\theta})\right) \frac{C_1}{1+n(\tau-s)}\; ds\label{boundImaginaryPartD3minusaa}\\
&\leq \frac{C_1}{n} \log(\frac{1/n + \tau + 1/2 - \theta}{1/n + \tau + 1/2}) \frac{1}{\pi \theta}\\
&+ \frac{C_1}{n} \frac{1}{4\pi \theta} \left|\log(-s)  - \log(1/n + \tau - s)\right|_{-1/2}^{-1/2+\theta} \frac{1}{|-\tau - 1/n|}\\
&+ \frac{C_1}{n} \frac{1}{4\pi \theta} \left|\log(1+s-\theta) - \log(1/n + \tau-s)\right|_{-1/2}^{-1/2+\theta} \frac{1}{1-\theta + \tau + 1/n}\\
&\leq \frac{C_1}{n|\tau|\pi} + \frac{C_1}{n}\frac{1}{4\pi \theta} \left(\log(\frac{1/2-\theta}{1/2})  - \log(\frac{1/n + \tau + 1/2 - \theta}{1/n + \tau + 1/2})\right) \frac{1}{|\tau|}\\
&+ \frac{C_1}{n} \frac{1}{4\pi \theta} \left(\log(\frac{1/2}{1/2-\theta}) - \log(\frac{1/n + \tau + 1/2 - \theta}{1/n + \tau + 1/2})\right) \frac{1}{1-\theta +\tau + 1/n}
\end{align}
To control the lines above, we consider two cases. If $\theta \leq 1/4$ we have 
\begin{align}
\eqref{boundImaginaryPartD3minusaa}&\leq \frac{C_1}{n \pi |\tau|} + \frac{C_1}{n}\frac{4}{\pi} \left(\frac{1}{|\tau|} + \frac{1}{1-\theta + |\tau|}\right)  \label{imagArgmax1F1}.
\end{align}
When $\theta \geq 1/4$, we write 
\begin{align}
\eqref{boundImaginaryPartD3minusaa} &\leq \frac{C_1}{n|\tau|\pi} + \frac{C_1}{n\pi} \left(\log(1+ \frac{\tau + 1/n}{1/2-\theta}) + \log(1+ \frac{\tau + 1/n}{1/2})\right) \frac{1}{|\tau|}\\
& + \frac{C_1}{n}\frac{1}{\pi} 2 \left(\log(\frac{1/n+\tau + 1/2}{1/2-\theta}) + \log(\frac{1/2}{1/n + \tau + 1/2 - \theta})\right) \\
&\leq \frac{C_1}{n}\frac{1}{\pi |\tau|} + \frac{C_1}{n\pi |\tau|} \left(\log(3n|\tau|) + \log(2)\right)\\
&+ \frac{C_1}{n\pi}2\left(\log(3n|\tau|) + \log(\frac{1}{|\tau|}) + \log(\frac{n}{C_1}|\tau|) + \log(|\tau|^{-1})\right)\label{imagArgmax2F1}.
\end{align}

The resulting bound is thus given by 
\begin{align}
\eqref{imagArgmax1F1} \vee \eqref{imagArgmax2F1}&\leq \frac{11C_1}{n\pi |\tau|} + \frac{C_1}{n|\tau|}\log(n|\tau|) \frac{3}{\pi} + \frac{2}{\pi}\frac{C_1}{n|\tau|}\log(\frac{n|\tau|}{C_1}). \label{tauPositive1/2}
\end{align}

For $\tau<0$, we go back to the bound~\eqref{tightBoundDecreaseLemmaTronc}. From this earlier result, we can derive the tighter bound
\begin{align}
F_{D_3^-, I} \leq F_{D_3^-, R}+ \frac{1}{4} \left(\log(1/2) - \log(-s)\right) + \frac{1}{4} \left(\log(1/2) - \log(1+s-\theta)\right)
\end{align}

When $\tau<0$, for the first term, we write 
\begin{align}
\frac{1}{\pi \theta} \int_{-1/2}^{-1/2 + \theta} \log(\frac{1/2}{-s}) \frac{C_1}{1+n |\tau-s|}\; ds &\leq \frac{1}{\pi \theta} \log(\frac{1/2}{1/2-\theta}) \frac{C_1}{1+n|\tau-s|}\; ds\\
&\leq \frac{1}{\pi \theta} \log(1+ \frac{\theta}{1/2-\theta}) \int_{-1/2}^{-1/2+\theta} \frac{C_1}{1+n(\tau - s)}\; ds\\
&\leq \frac{1}{\pi \theta} \log(1+ \frac{\theta}{1/2-\theta}) \frac{C_1}{n} \log(\frac{1/n + \tau + 1/2 - \theta}{1/n + \tau + \frac{1}{2}})\label{star1}
\end{align}
We make the distinction between $\theta\geq 1/4$ and $\theta < 1/4$
\begin{itemize}
\item When $\theta \geq 1/4$, we write 
\begin{align}
\eqref{star1}&\leq \frac{4}{\pi} \log(n\theta) \frac{C_1}{n} \log(1+ \frac{\theta}{1/n + \tau +1/2 - \theta})\\
&\leq \frac{4}{\pi}\log(n\theta) \frac{C_1}{n} \log(1+ \frac{\theta}{1/n + \tau + 1/2 - \theta})\\
&\leq \frac{4}{\pi} \log(n\theta) \frac{C_1}{n} \log(1+ \frac{\theta}{1/n + \tau + 1/4}).
\end{align}
Now we make the distinction 
\begin{itemize}
\item $|\tau|\geq 1/8$. In this case, we have 
\begin{align}
\eqref{star1}&\leq \frac{4}{\pi}\log(n) \frac{C_1}{n} \log(n\theta)\\
&\leq \frac{4}{\pi} \left(\log(n|\tau|) + \log(\frac{1}{|\tau|}) \frac{C_1}{n} \left(\log(n|\tau|) + \log(\frac{1}{|\tau|})\right)\right)\\
&\leq \frac{4}{\pi} \log^2(n|\tau|) \frac{C_1}{n}+ 2\frac{C_1}{n|\tau|} \log(n|\tau|) + \frac{C_1}{n|\tau|}
\end{align}
\item When $|\tau|\leq \frac{1}{8}$, we have 
\begin{align}
\eqref{star1}&\leq \frac{4}{\pi} \log(n\theta) \frac{C_1}{n} \log(1+ 8 \theta)\\
&\leq \frac{4}{\pi} \log(n) \frac{C_1}{n}\log(5)\\ 
\end{align}
\end{itemize}
\item Finally when $\theta \leq 1/4$
\begin{align}
\eqref{star1}&\leq \frac{4}{\pi} \frac{C_1}{n} \log(1+ \frac{\theta}{1/n + \tau + 1/2})\\
&\leq \frac{4}{\pi} \frac{C_1}{n} \log(n)\\
&\leq \frac{4}{\pi} \frac{C_1}{n|\tau|} + \frac{4}{\pi} \frac{C_1}{n}\log(n|\tau|) 
\end{align}
\end{itemize}

For the second term $\log(\frac{1/2}{1+s-\theta})$, note that we have 
\begin{align}
\frac{1}{\pi \theta} \int_{-1/2}^{-1/2 + \theta} \log(\frac{1/2}{1+s-\theta}) \frac{C_1}{1+n(\tau-s)}\; ds
&\leq \frac{1}{\pi \theta} \log(\frac{1/2}{1/2-\theta}) \frac{C_1}{n}\log( \frac{1/n + \tau + 1/2 - \theta}{1/n + \tau + 1/2}) 
\end{align}
on which we can apply the exact same reasoning. Hence we get 

\begin{align}
&\frac{1}{\pi \theta} \int_{-1/2}^{-1/2 + \theta}\left( \log(\frac{1/2}{-s}) + \log(\frac{1/2}{1+s-\theta})\right) \frac{C_1}{1+n |\tau-s|}\; ds\\
 &\leq 2\left(\frac{C_1}{n|\tau|}\frac{4\log(5)}{\pi} +2\frac{C_1}{n|\tau|}\log(n|\tau|) + \frac{4}{\pi}\frac{C_1}{n|\tau|}\log^2(n|\tau|)\right)\label{tauNegiativeButLargerThanMinus1/2}
\end{align}

We now treat the case $\tau < -1/2$, which is needed to account for the effect of the modulo $1$ part of the two cases treated above. In this case, we have 
\begin{align}
&\frac{1}{\pi \theta} \int_{-1/2}^{-1/2 + \theta} \frac{1}{4} \left(\log(\frac{1/2}{1+s-\theta}) + \log(\frac{1/2}{-s})\right) \frac{C_1}{1+n(s-\tau)}\; ds\\
&\leq \frac{1}{2}\frac{1}{\pi \theta} \log(\frac{1/2}{1/2-\theta})  \frac{C_1}{n} \log(\frac{1/n - \frac{1}{2} + \theta - \tau}{1/n - 1/2 - \tau})\label{D3minusTauNegativeLemmaTronc}
\end{align}
To control~\eqref{D3minusTauNegativeLemmaTronc}, we consider two cases, 
\begin{itemize}
\item $\theta \geq 1/4$
In this case we keep the $\log$ inside the integral and write 
\begin{align}
\eqref{D3minusTauNegativeLemmaTronc}&\leq \frac{4}{\pi} \log(n) \frac{C_1}{n} \frac{1}{2} \log(1+ \frac{\theta}{1/n - 1/2 - \tau})\\
&\leq \frac{2}{\pi}\frac{C_1}{n|\tau|}\left(2\log(2)\log(n|\tau|) + \log^2(2) +\log^2(n|\tau|)\right)
\end{align}
\item $\theta \leq 1/4$. In this case, we use the log to cancel the prefactor. we consider two cases
\begin{itemize}
\item if $|\tau|_{\mod 1} \geq  \frac{1}{4}$, we write 
\begin{align}
\eqref{D3minusTauNegativeLemmaTronc} &\leq \frac{4}{\pi} \log(n) \frac{C_1}{n} \frac{1}{2} \log(n\theta)\\
&\leq \frac{4}{\pi} \left(\log(n|\tau|) + \frac{1}{|\tau|}\right) \frac{C_1}{n} \frac{1}{2} \left(\log(n|\tau|) + \frac{1}{|\tau|}\right) \\
&\leq \left(\frac{C_1}{n|\tau|^2} + \frac{2C_1}{n|\tau|}\log(n|\tau|) + \frac{C_1}{n} \log^2(n|\tau|)\right)\frac{2}{\pi}\\
&\leq \frac{C_1}{n}\frac{4}{|\tau|} + \frac{2C_1}{n|\tau|} \log(n|\tau|) + \frac{C_1}{n|\tau|} \log^2(n|\tau|)\frac{2}{\pi}
\end{align}
\item When $|\tau|_{\mod 1}\leq \frac{1}{4}$ we have $-1/2-\tau \geq 1/4$ (recall that $\tau$is located on the left of $-1/2$) hence
\end{itemize}
\begin{align}
\eqref{D3minusTauNegativeLemmaTronc} & \leq \frac{4}{\pi} \log(n) \frac{C_1}{n} \frac{1}{2} \log(1+4\theta)\\
&\leq \left(\frac{4}{\pi} \log(n|\tau|)\frac{C_1}{n|\tau|} + \frac{C_1}{n|\tau|} \frac{4}{\pi}\right)\frac{1}{2}\log(3) 
\end{align}
\end{itemize}
Together those bounds give, for $\tau<-1/2$
\begin{align}
&\frac{1}{\pi \theta} \int_{-1/2}^{-1/2 + \theta}\left( \log(\frac{1/2}{-s}) + \log(\frac{1/2}{1+s-\theta})\right) \frac{C_1}{1+n |\tau-s|}\; ds\\
&\leq 4\frac{C_1}{n|\tau|} + \frac{2}{\pi} \frac{C_1}{n|\tau|}\log^2(n|\tau|) + 2\frac{C_1}{n|\tau|}\log(n|\tau|)\label{boundSymmetricForModuloTauLess1/2}
\end{align}

replacing $\tau$ by $1-|\tau'|$, with $|\tau'|\leq 1-|\tau'|$ and adding the resulting bound to~\eqref{tauNegiativeButLargerThanMinus1/2} and~\eqref{tauPositive1/2} respectively gives 

\begin{align}
\eqref{tauNegiativeButLargerThanMinus1/2} + \eqref{boundSymmetricForModuloTauLess1/2}(1-|\tau|)&\leq 9\frac{C_1}{n|\tau|} + 6\frac{C_1}{n|\tau|} \log(n|\tau|) +4\frac{C_1}{n|\tau|}\log^2(n|\tau|)  \label{imaginaryBoundD3minusFW1}\\
\eqref{tauPositive1/2} + \eqref{boundSymmetricForModuloTauLess1/2}(1-|\tau|)&\leq  8\frac{C_1}{n|\tau|} + 3\frac{C_1}{n|\tau|} \log(n|\tau|) + 0.7\frac{C_1}{n|\tau|}\log^2(n|\tau|) + 0.7 \frac{C_1}{n|\tau|} \log(\frac{n|\tau|}{C_1})\label{imaginaryBoundD3minusFW2}
\end{align}


When $\left[\tau-\frac{C_1}{n}, \tau + \frac{C_1}{n}\right]\cap  [-1/2, -1/2+\theta]\neq \varnothing$, we have the decomposition
\begin{align}
&\frac{1}{\pi \theta} \int_{-1/2}^{\tau - \frac{C_1-1}{n}\vee -1/2} \left(\frac{1}{4}\log(\frac{1/2}{-s}) + \frac{1}{4} \log(\frac{1/2}{1+s-\theta})\right)\frac{C_1}{1+s-\theta}\; ds\label{tmpLemmaTroncleftD3minus}\\
&+ \frac{1}{\pi \theta} \int_{\tau - \frac{C_1-1}{n}}^{\tau + \frac{C_1-1}{n}} \left\{\frac{1}{4}\log(\frac{1/2}{-s}) + \frac{1}{4}\log(\frac{1/2}{1+s-\theta})\right\} \; ds\label{tmpLemmaTroncCentralD3minus}\\
&+ \frac{1}{\pi \theta} \int_{\tau + \frac{C_1-1}{n}\wedge -1/2+\theta}^{-1/2 + \theta} \left\{\frac{1}{4}\log(\frac{1/2}{-s}) + \frac{1}{4} \log(\frac{1/2}{1+s-\theta})\right\} \frac{C_1}{1+n(s-\tau)}\; ds\label{tmpLemmaTroncrightD3minus}
\end{align}
For~\eqref{tmpLemmaTroncCentralD3minus}, noting that we always integrate on a length $2C_1$ interval, we can simply write
\begin{align}
\eqref{tmpLemmaTroncCentralD3minus}&\leq 2\frac{C_1-1}{n} \frac{1}{\pi \theta} \log(\frac{1/2}{1/2-\theta}) \leq \frac{C_1}{n|\tau|} \left(1+ \log(n|\tau|)\right).\label{841tmpD3imag}
\end{align}
For the first term in~\eqref{tmpLemmaTroncleftD3minus} and~\eqref{tmpLemmaTroncrightD3minus}, we handle the cases $\theta \geq 1/4$ and $\theta<1/4$ separately.
\begin{itemize}
\item When $\theta \geq 1/4$, we use $\log(\frac{1/2}{-s}) \leq \frac{1/2}{-s}$, from which we have  
\begin{align}
\frac{1}{\pi \theta}\int_{-1/2}^{\tau - \frac{C_1-1}{n}\vee -1/2} \frac{1}{4}\log(\frac{1/2}{-s}) \frac{C_1}{1+n(\tau-s)}\;ds &\leq \frac{4}{\pi}\frac{C_1}{n} \left| \log(-s) - \log(1/n + \tau-s)\right|_{-1/2}^{\tau - \frac{C_1-1}{n}} \frac{1}{-\tau - \frac{1}{n}}\\
&\leq \frac{4}{\pi}\frac{C_1}{n|\tau|} \left|\log(\frac{-\tau + \frac{C_1-1}{n}}{1/2}) - \log(\frac{C_1/n}{1/n + \tau + 1/2})\right|\\
&\leq \frac{4}{\pi} \frac{C_1}{n|\tau|} \left|\log(1+ \frac{-\tau - 1/n}{1/n + \tau + 1/2}) + \log(\frac{n (-\tau + \frac{C_1-1}{n})}{C_1})\right|\\
&\leq \frac{4}{\pi}\frac{C_1}{n|\tau|} \left(\log(3n|\tau|) + \log(\frac{2n|\tau|}{C_1})\right)\label{845}
\end{align}
as well as
\begin{align}
&\frac{1}{\pi \theta} \int_{\tau + \frac{C_1-1}{n}}^{-1/2+\theta} \frac{1}{4} \log(\frac{1/2}{-s}) \frac{C_1}{1+n(s-\tau)}\; ds \leq \frac{4}{\pi}\frac{1}{8}\frac{C_1}{n} \left|-\log(-s) + \log(1/n + s-\tau)\right|_{\tau + \frac{C_1-1}{n}}^{-1/2+\theta} \\
&\leq \frac{4}{\pi} \frac{C_1}{n} \frac{1}{|\tau|} \left|-\log(\frac{1/2-\theta}{-\tau - \frac{C_1-1}{n}}) + \log(\frac{1/n - 1/2 + \theta - \tau}{C_1/n})\right|\\
& \stackrel{(a)}{\leq } \frac{4}{\pi} \frac{C_1}{n|\tau|}\left(\log(\frac{2|\tau|n}{C}) + \log(2|\tau|n)\right)\label{848}
\end{align}
in $(a)$ above we use $-1/2+\theta<0$. 
\item When $\theta \leq 1/4$, we cancel the prefator with the log, 
\begin{align}
&\frac{1}{\pi \theta} \int_{-1/2}^{\tau - \frac{C_1-1}{n}} \log(\frac{1/2}{-s})\frac{C_1}{1+n(\tau-s)}\; ds \leq \frac{C_1}{n}\frac{4}{\pi} \log(\frac{C_1/n}{1/n + \tau + 1/2}) \leq \frac{4}{\pi}\log(\frac{n}{C_1}) \frac{C_1}{n} \leq \frac{4}{\pi}\left(\frac{C_1}{n|\tau|} + \frac{C_1}{n|\tau|} \log(\frac{n|\tau|}{C_1})\right)\label{849}\\
& \frac{1}{\pi \theta} \int_{\tau + \frac{C_1-1}{n}}^{-1/2+\theta} \log(\frac{1/2}{-s}) \frac{C_1}{1+n(s-\tau)}\; ds \leq \frac{4}{\pi}  \frac{C_1}{n}\log(\frac{1/n - \tau - 1/2+\theta}{C_1/n}) \leq \frac{4}{\pi} \frac{C_1}{n} \log(\frac{n2|\tau|}{C_1}) \leq \frac{4}{\pi} \frac{C_1}{n|\tau|} \log(\frac{2n|\tau|}{C_1}).\label{850}
\end{align}
\end{itemize}

For the remaining term, the first integral reads as
\begin{align}
\frac{1}{\pi \theta} \int_{-1/2}^{\tau - \frac{C_1-1}{n}} \frac{1}{4}\log(\frac{1/2}{1+s-\theta})\frac{C_1}{1+n(s-\tau)}\; ds
\end{align}
We first consider the case $\theta\geq 1/4$. If $\tau<-1/4$, we have $|1/2 + \tau|\leq |\tau|$ and we can write 
\begin{align}
&\frac{1}{\pi \theta} \int_{-1/2}^{\tau - \frac{C_1-1}{n}} \frac{1}{4}\log(\frac{1/2}{1+s-\theta})\frac{C_1}{1+n(s-\tau)}\; ds\\
&\leq \frac{1}{\pi \theta} \log(\frac{1/2}{1/2-\theta}) \log(\frac{C_1/n}{1/n + \tau + 1/2})\frac{C_1}{n}\\
&\leq \frac{1}{\pi \theta} \log(\frac{1/2}{1/2-\theta}) \log(\frac{|\tau|n}{C_1}) \frac{C_1}{n}\\
&\leq \frac{4}{\pi}\left( \log(n|\tau|) + \frac{1}{|\tau|}\right) \log(n|\tau|/C_1)\frac{C_1}{n}.\label{tmp855}
\end{align}
If $\tau > -1/4$, we have  
\begin{align}
&\frac{1}{\pi \theta} \int_{-1/2}^{-3/8} \log(\frac{1/2}{1+s-\theta})\frac{C_1}{1+n(\tau-s)}\; ds \\
& + \frac{1}{\pi \theta} \int_{-3/8}^{\tau - \frac{C_1-1}{n}} \log(\frac{1/2}{1+s-\theta})\frac{C_1}{1+n(\tau-s)}\; ds \\
&\leq \frac{1}{\pi \theta} \frac{C_1}{n} \log(\frac{1/2}{1/2-\theta}) \log(\frac{1/n + \tau +3/8}{1/n + \tau + 1/2})\\
&+  \frac{1}{\pi \theta} \frac{C_1}{n} \log(\frac{1/2}{5/8-\theta}) \log(\frac{C_1/n}{1/n + \tau + 3/8})\\
&\stackrel{(a)}{\leq } \frac{4}{\pi} \left\{\frac{C_1}{n} \log(4) \log(\frac{n}{C_1}) + \frac{C_1}{n}\log(n) \log(4)\right\}\label{tmp860}
\end{align}

$(a)$ follows from $\theta \leq 1/2$ as well as $\tau \geq -1/4$

For the second term, when $\theta\geq 1/4$, we write 
\begin{align}
&\frac{1}{\pi \theta}\int_{\tau + \frac{C_1-1}{n}}^{-1/2+\theta} \log(\frac{1/2}{1+s-\theta}) \frac{C_1}{1+n(s-\tau)}\; ds\\
& \leq \frac{4}{\pi} \log(\frac{n}{2}) \log(\frac{n|\tau|}{C_1})\frac{C_1}{n}\label{tmp861}
\end{align}

When $\theta \leq 1/4$, note that we have 
\begin{align}
&\frac{1}{\pi \theta} \log(\frac{1/2}{1/2-\theta}) \int_{-1/2}^{\tau -\frac{C_1-1}{n} } \frac{C_1}{1+n(\tau-s)}\; ds\\
&\leq \frac{4}{\pi} \frac{C_1}{n}\log(\frac{C_1/n}{1/n + \tau +1/2}) \leq \log(\frac{n}{C_1})\frac{C_1}{n}\frac{4}{\pi}\label{tmp862}
\end{align}
as well as
\begin{align}
&\frac{1}{\pi \theta} \log(\frac{1/2}{1/2-\theta}) \int_{\tau + \frac{C_1-1}{n}}^{-1/2+\theta} \frac{C_1}{1+n(\tau-s)}\; ds\\
&\leq \frac{4}{\pi} \frac{C_1}{n}\log(\frac{1/n + \tau + 1/2-\theta}{C_1/n}) \leq \log(\frac{n}{C_1})\frac{C_1}{n}\frac{4}{\pi}\label{tmp863}
\end{align}

Combining~\eqref{849} and~\eqref{850} with~\eqref{845} and~\eqref{848} as well as~\eqref{tmp855},~\eqref{tmp860},~\eqref{tmp861}, ~\eqref{tmp862} and~\eqref{tmp863}, as well as~\eqref{841tmpD3imag}, we get 
\begin{align}
\eqref{tmpLemmaTroncleftD3minus}+\eqref{tmpLemmaTroncCentralD3minus} + \eqref{tmpLemmaTroncrightD3minus} &\leq \frac{C_1}{n|\tau|} \left(\log(n|\tau|) + \log(\frac{n|\tau|}{C_1})\right)\frac{8}{\pi} + (\frac{4}{\pi} +( \log(2)+\log(3))\frac{8}{\pi})\frac{C_1}{n|\tau|} \\
&+ \frac{C_1}{n|\tau|} (1+ \log(n|\tau|))\\
&+ \log(n|\tau|) \frac{C_1}{n|\tau|} \log(\frac{n|\tau|}{C_1}) \frac{8}{\pi} + \log(\frac{n|\tau|}{C_1}) \frac{C_1}{n|\tau|} \frac{8\log(4)}{\pi} \\
&+ \frac{C_1}{n|\tau|} \frac{8+4\log(4)}{\pi}\\
&\leq 12\frac{C_1}{n|\tau|}  + 4\frac{C_1}{n|\tau|} \log(n|\tau|) + 7\frac{C_1}{n|\tau|}\log(\frac{n|\tau|}{C_1}) + 4\frac{C_1}{n|\tau|}\log(n|\tau|)\log(\frac{n|\tau|}{C_1}). \label{imaginaryPartFW3}
\end{align}

We can then mulitply the resulting bound by $2$ to account for the effect of the modulo $1$. 

Taking the maximum over the three frameworks gives the final bound on the imaginary part

\begin{align}
\eqref{imaginaryPartFW3}\vee \eqref{imaginaryBoundD3minusFW1} \vee \eqref{imaginaryBoundD3minusFW2} \leq 24\frac{C_1}{n|\tau|}  + 8\frac{C_1}{n|\tau|} \log(n|\tau|) + 14\frac{C_1}{n|\tau|}\log(\frac{n|\tau|}{C_1}) + 8\frac{C_1}{n|\tau|}\log^2(n|\tau|)
\end{align}

Adding this bound to twice the bound~\eqref{boundD3minusRealPartAllFW} on the real part gives the result of the lemma. 

\begin{align}
&20\frac{C_1}{n|\tau|} + 7\frac{C_1}{n|\tau|}\log(\frac{n|\tau|}{C_1}) +2 \frac{C_1}{n|\tau|}\log(n|\tau|)\\
&+24\frac{C_1}{n|\tau|}  + 8\frac{C_1}{n|\tau|} \log(n|\tau|) + 14\frac{C_1}{n|\tau|}\log(\frac{n|\tau|}{C_1}) + 8\frac{C_1}{n|\tau|}\log^2(n|\tau|)\\
&\leq 44\frac{C_1}{n|\tau|}  + 21 \frac{C_1}{n|\tau|} \log(\frac{n|\tau|}{C_1}) + 10\frac{C_1}{n|\tau|} \log(n|\tau|) + 8 \frac{C_1}{n|\tau|} \log^2(n|\tau|). 
\end{align}

\subsection{\label{proofD0minusLarge}Proof of lemma~\ref{lemmaTroncD0minus} ($D_0^-$ large $|\tau - \alpha|$)}

We start with $D_0^-$. Recall that we have 
\begin{align}
F_{R,D_0^-}& \leq c\left(\frac{1}{-s} + \frac{1}{\theta-s}\right) + \frac{\theta}{2}\\
&\leq \frac{2c}{-s} + \frac{\theta}{2}
\end{align}
as well as 
\begin{align}
F_{I,D_0^-} &\leq F_{R,D_0^-} + \frac{1}{4} \log(\frac{-s}{\theta-s})
\end{align}

We consider two distinct cases: Either $[\tau - \frac{C_1}{n}, \tau + \frac{C_1}{n}] \cap [-1/2+\theta, -\frac{1}{2n+3}]\neq \emptyset$ or $\tau$ does not belong to that interval. In the latter case, when $\tau<-1/2+\theta$, we have

\begin{align}
&\frac{1}{\pi \theta} \int_{-1/2+ \theta}^{-\frac{1}{2n+3}} \frac{1}{4(2n+3)} \left(\frac{1}{\pi}+ \frac{1}{2}\right) \frac{2}{-s} \frac{C_1}{1+n(s-\tau)}\; ds + \frac{1}{\pi \theta} \int_{-1/2 + \theta}^{-\frac{1}{2n+3}} \frac{\theta}{2} \frac{C_1}{1+n(s-\tau)}\; ds\\
&\leq \frac{1}{\pi |\theta|}\frac{1}{4(2n+3)} \left(\frac{1}{\pi}+ \frac{1}{2}\right)\int_{-1/2+\theta}^{-\frac{1}{2n+3}} \left(\frac{1}{-s} + \frac{1}{1/n + (s-\tau)}\right) \frac{C_1}{n}\frac{1}{-\tau + \frac{1}{n}}\\
&+ \frac{1}{\pi |\theta|} \frac{\theta}{2} \frac{C_1}{n} \log(\frac{1/n - \tau - \frac{1}{2n+3}}{1/n - \tau - 1/2+\theta})\\
&\leq \frac{1}{\pi |\theta|} \frac{1}{4(2n+3)} \left(\frac{1}{\pi} + \frac{1}{2}\right) \left[-\log(\frac{\frac{1}{2n+3}}{1/2-\theta}) + \log(\frac{1/n - \tau - \frac{1}{2n+3}}{1/n - \tau -1/2+\theta})\right]\frac{C_1}{n}\frac{1}{|-\tau + \frac{1}{n}|}\\
&+ \frac{1}{\pi |\theta|}\frac{\theta}{2} \frac{C_1}{n}  \log(\frac{1/n - \tau - \frac{1}{2n+3}}{1/n - \tau -1/2+\theta})\\
&\leq \frac{1}{\pi|\theta|}c \frac{2C_1}{n|\tau|} \left(2\log(3n|\tau|)\right)\\
&+\frac{1}{2\pi} \frac{C_1}{n}\log(3n|\tau|)\\
&\leq \frac{c'}{\pi}\frac{4C_1}{n|\tau|} \log(3n|\tau|) + \frac{1}{2\pi} \frac{C_1}{n}\log(3n|\tau|).\label{domainD0minusTaulessthen1.2minustheta}
\end{align}

In the equations above, we used $1/2-\theta \leq |\tau|$. When $\tau>0$, a similar reasoning gives 

\begin{align}
&\frac{1}{\pi \theta } \int_{-1/2+\theta}^{-\frac{1}{2n+3}} \frac{1}{4(2n+3)} \left(\frac{1}{\pi} + \frac{1}{2}\right) \frac{2}{-s} \frac{C_1}{1+n(\tau-s)}\; ds + \frac{1}{\pi |\theta|} \int_{-1/2+\theta}^{-\frac{1}{2n+3}} \frac{\theta}{2} \frac{C_1}{1+n(\tau-s)}\; ds\\
&\leq \frac{1}{\pi |\theta|} \frac{1}{4(2n+3)} \left(\frac{1}{\pi} + \frac{1}{2}\right) \int \left(-\frac{1}{-s} + \frac{1}{1/n +\tau-s}\right) \frac{1}{|\tau + \frac{1}{n}|}\frac{2C_1}{n}\; ds\\
&+ \frac{1}{\pi \theta} \frac{\theta}{2} \frac{C_1}{n} \log(\frac{1/n + \tau + \frac{1}{2n+3}}{1/n + \tau + 1/2-\theta})\\
&\leq \frac{1}{\pi |\theta|}\frac{1}{4(2n+3)}\left(\frac{1}{\pi} + \frac{1}{2}\right) \left[\log(\frac{1/2-\theta}{\frac{1}{2n+3}}) + \log(\frac{1/n + \tau + \frac{1}{2n+3}}{1/n + \tau + 1/2 - \theta})\right] \frac{C_1/n}{|\tau + \frac{1}{n}|}\\
&+ \frac{1}{\pi |\theta|} \frac{\theta}{2} \frac{C_1}{n} \log(\frac{1/n + \tau + \frac{1}{2n+3}}{1/n + \tau + 1/2 - \theta})\\
&\leq \frac{1}{\pi|\theta|}c \frac{2C_1}{n|\tau|} \left(2\log(3n|\tau|)\right)\\
&+\frac{1}{2\pi} \frac{C_1}{n}\log(3n|\tau|)\\
&\leq \frac{c'}{\pi}\frac{4C_1}{n|\tau|} \log(3n|\tau|) + \frac{1}{2\pi} \frac{C_1}{n}\log(3n|\tau|).\label{DomainD0minusTauPositive}
\end{align}
Note that as soon as $\tau \geq \Delta$, we have $\tau - \frac{1}{n}\geq \tau/2$. 

When $\tau \in  [-1/2+\theta, -\frac{1}{2n+3}]$, we write 

\begin{align}
&\frac{1}{\pi \theta} \int_{-1/2+\theta}^{\tau - \frac{C_1-1}{n}} \frac{1}{4(2n+3)} \left(\frac{1}{\pi} + \frac{1}{2}\right) \frac{2}{-s} \frac{C_1}{1+n(\tau - s)}\; ds\\
& + \frac{1}{\pi \theta} \int_{-1/2+\theta}^{\tau - \frac{C_1-1}{n}} \frac{C_1}{1+n(\tau-s)}\frac{\theta}{2} \; ds\\
&+ \frac{1}{\pi \theta} \int_{\tau + \frac{C_1-1}{n}}^{-\frac{1}{2n+3}} \frac{1}{4(2n+3)} \left(\frac{1}{\pi} + \frac{1}{2}\right) \frac{2}{-s} \frac{C_1}{1+n(s-\tau)} \; ds\\
&+ \frac{1}{\pi \theta} \int_{\tau + \frac{C_1-1}{n}}^{-\frac{1}{2n+3}} \frac{\theta}{2} \frac{C_1}{1+n(s-\tau)}\; ds\\
&\leq \frac{1}{\pi \theta} \frac{1}{4(2n+3)}\left(\frac{1}{\pi}+ \frac{1}{2}\right) \frac{2C_1}{n}\left|\log(\frac{-\tau + \frac{C_1-1}{n}}{1/2-\theta}) - \log(\frac{C_1/n}{1/n+\tau+1/2-\theta})\right| \frac{2C_1}{n}\frac{1}{|-\tau - 1/n|}\\
&+ \frac{1}{2\pi} \frac{C_1}{n} \log(\frac{C_1/n}{1/n+1/2-\theta - \tau})\\
&+ \frac{1}{\pi \theta} \frac{1}{4(2n+3)}\left(\frac{1}{\pi} + \frac{1}{2}\right)\frac{2C_1}{n} \left|-\log(\frac{\frac{1}{2n+3}}{-\tau - \frac{C_1-1}{n}}) + \log(\frac{1/n - \frac{1}{2n+3} - \tau}{C_1/n})\right| \frac{1}{|1/n - \tau|}\\
&+ \frac{1}{2\pi} \frac{C_1}{n} \log(\frac{1/n - \frac{1}{2n+3} - \tau}{C_1/n})\\
&\leq \frac{c}{\pi \theta} \frac{2C_1}{n} \left|\log(1+ \frac{-\tau - \frac{1}{n}}{1/n + \tau + 1/2 - \theta}) + \log(1+ \frac{-\tau - \frac{1}{n}}{C_1/n})\right|\frac{1}{|1/n - \tau|}\\
&+ \frac{1}{2\pi} \frac{C_1}{n}\log(\frac{1}{2}\frac{n}{C_1})\\
&+\frac{c}{\pi \theta} \frac{2C_1}{n}\left[\log((-\tau - \frac{C_1-1}{n})(2n+3)) + \log((-\tau - \frac{1}{2n+3}+\frac{1}{n})\frac{n}{C_1})\right]\frac{1}{|1/n - \tau|}\\
&+ \frac{1}{2\pi}\frac{C_1}{n}\log((1/n - \frac{1}{2n+3} - \tau)\frac{n}{C_1})\\
&\leq \frac{c}{\pi \theta} \frac{4C_1}{n} \left( \log(4\frac{|\tau|n}{C_1})\right)\frac{1}{|\tau|}\label{note1LemmeTronc}\\
&+ \frac{1}{2\pi}\left(\frac{C_1}{n|\tau|} \log(\frac{n|\tau|}{C_1}) + \frac{C_1}{n|\tau|}\right)\label{note2LemmeTronc}\\
&+ \frac{c}{\pi \theta} \frac{4C_1}{n|\tau|} \left(\log(3n|\tau|) + \log(\frac{2|\tau|n}{C_1})\right)\label{note3LemmeTronc}\\
&+ \frac{1}{2\pi} \frac{C_1}{n} \log(2\frac{|\tau|n}{C_1})\\
&\leq \frac{C_1}{n|\tau|} \log(\frac{4|\tau|n}{C_1}) \left(\frac{8c'}{\pi} + \frac{1}{\pi} \right) +\frac{1}{2\pi} \frac{C_1}{n|\tau|} + \frac{4c'}{\pi} \frac{C_1}{n|\tau|}\log(3n|\tau|)\label{LeftRightTauInPartRealD0minus}
\end{align}

In~\eqref{note1LemmeTronc}, we use $\tau\geq -1/2+\theta+\frac{C_1-1}{n}$ (otherwise the integral vanishes) as well as $|\tau|\geq 2/n$. In~\eqref{note2LemmeTronc}, we use 
\begin{align}
\frac{C_1}{n}\log(\frac{n}{C_1}\frac{|\tau|}{|\tau|})  &= \frac{C_1}{n}\log(\frac{n|\tau|}{C_1})+ \frac{C_1}{n}\log(\frac{1}{|\tau|})\\
&\leq \frac{C_1}{n|\tau|} \log(\frac{n|\tau|}{C_1}) + \frac{C_1}{n|\tau|}
\end{align}
In~\eqref{note3LemmeTronc}, we use $\tau \geq \frac{C_1-1}{n}$ and hence $1/(|\tau - 1/n|)\leq 2/|\tau|$. 

Finally we need to add the central contribution $[\tau - \frac{C_1-1}{n}, \tau+\frac{C_1-1}{n}]$. For this central contribution, we get 
\begin{align}
\int_{\tau - \frac{C_1-1}{n}}^{\tau + \frac{C_1-1}{n}} \frac{1}{2\pi} + \frac{c}{\pi \theta} \frac{2}{-s}\; ds
&\leq \frac{1}{\pi}\frac{C_1-1}{n} + \frac{c}{\pi \theta} \log(1+ 2\frac{C_1-1}{n}\frac{1}{(-\tau - \frac{C_1-1}{n})})\\
&\leq \frac{1}{\pi}\frac{C_1-1}{n|\tau|} + \frac{c'}{\pi} \frac{C_1-1}{n}\frac{4}{|\tau|}. \label{centralTauInPartRealD0minus}
\end{align}
The last line holds as soon as $\tau\geq 2\frac{C_1-1}{n}$

To get the final bound on each framework, we use 

\begin{align}
2\left(\eqref{centralTauInPartRealD0minus} + \eqref{LeftRightTauInPartRealD0minus}\right)\leq 5\frac{C_1}{n|\tau|} + \frac{C_1}{n|\tau|}\log(n|\tau|) + 2\frac{C_1}{n|\tau|}\log(\frac{n|\tau|}{C_1}) , \quad \tau+\frac{C_1-1}{n}\in[-1/2+\theta,-\frac{1}{2n+3}]\label{boundD0minusLargethetastrong}
\end{align}
as well as
\begin{align}
\eqref{domainD0minusTaulessthen1.2minustheta}(1-|\tau|) + \eqref{DomainD0minusTauPositive} \leq \eqref{domainD0minusTaulessthen1.2minustheta} + \eqref{DomainD0minusTauPositive} \\
\eqref{domainD0minusTaulessthen1.2minustheta} + \eqref{DomainD0minusTauPositive}(1-|\tau|) \leq \eqref{domainD0minusTaulessthen1.2minustheta} + \eqref{DomainD0minusTauPositive}
\end{align}
For 
\begin{align}
\eqref{domainD0minusTaulessthen1.2minustheta} + \eqref{DomainD0minusTauPositive} \leq \frac{C_1}{n|\tau|} + 0.6\frac{C_1}{n|\tau|} \log(n|\tau|)\label{boundD0minusLargethetaweak}
\end{align}
The bound is then obtained by taking the maximum of~\eqref{boundD0minusLargethetastrong} and~\eqref{boundD0minusLargethetaweak} which in this case is given by~\eqref{boundD0minusLargethetastrong}.

When $\theta \leq \frac{1}{2n+3}$, we use the ``small $\theta$" bound
\begin{align}
F_{R,D_0}^- \leq \frac{3}{2} \frac{\theta}{\theta-s} + \frac{\theta}{8(2n+3)s(s-\theta)} + \frac{\theta}{2}.\label{reminderD0minusSmallLemmaTroncLargeTauminusAlpha}
\end{align}
As before we consider three frameworks following from $[\tau - \frac{C_1}{n}, \tau+\frac{C_1}{n}]\cap [-1/2+ \theta, \frac{1}{2n+3}]$ empty or not.
\begin{itemize}
\item When $\tau+\frac{C_1}{n} < -1/2+\theta$, we have 
\begin{align}
&\frac{1}{\pi \theta} \int_{-1/2+\theta}^{-\frac{1}{2n+3}} \left(\frac{3}{2} + \frac{1}{8}\right) \frac{\theta}{\theta-s} \frac{C_1}{1+n(s-\tau)}\; ds + \frac{1}{\pi \theta} \int_{-1/2+ \theta}^{-\frac{1}{2n+3}} \frac{\theta}{2} \frac{C_1}{1+n(s-\tau)}\; ds\\
&\leq \frac{1}{\pi}\left(\frac{3}{2} + \frac{1}{8}\right) \frac{C_1}{n} \left|-\log(\theta-s) + \log(1/n + s-\tau)\right|_{-1/2 + \theta}^{-\frac{1}{2n+3}} \frac{1}{1/n + \theta - \tau} + \frac{1}{2\pi} \frac{C_1}{n} \log(\frac{1/n-\tau - \frac{1}{2n+3}}{1/n - \frac{1}{2} + \theta - \tau})\\
&\leq \frac{1}{\pi}\left(\frac{3}{2} + \frac{1}{8}\right)\frac{C_1}{n} \left(\log(\frac{\theta + \frac{1}{2n+3}}{1/2}) + \log(\frac{1/n - \frac{1}{2n+3} - \tau}{1/n - 1/2 + \theta - \tau})\right) \frac{1}{1/n + \theta - \tau} + \frac{1}{2\pi}\frac{C_1}{n} \log(n|\tau|)\\
&\leq \frac{1}{\pi} \left(\frac{3}{2} + \frac{1}{8}\right) \frac{C_1}{n} \left(2 + \frac{\log(n|\tau|)}{1/n + \theta - \tau}\right) + \frac{1}{2\pi} \frac{C_1}{n} \log(n|\tau|)\\
&\leq \frac{1}{\pi}\left(\frac{3}{2} + \frac{1}{8}\right) \frac{C_1}{n|\tau|}  + \frac{2}{\pi} \left(\frac{3}{2} + \frac{1}{8}\right) \frac{C_1}{n|\tau|} + \frac{1}{2\pi}\frac{C_1}{n|\tau|}\log(n|\tau|). \label{tauNegativeSmallD0minusReal}
\end{align}
When $\tau + \frac{C_1-1}{n}>-1/2 + \theta$, we have 
\begin{align}
&\frac{1}{\pi \theta} \int_{-1/2 + \theta}^{-1/2 + \theta+ \frac{C_1-1}{n}} \left(\frac{3}{2} + \frac{1}{8}\right)\frac{\theta}{\theta - s}\; ds + \frac{1}{\pi \theta} \int_{\tau + \frac{C_1-1}{n}}^{-\frac{1}{2n+3}} \left(\frac{3}{2} + \frac{1}{8}\right) \frac{\theta}{\theta-s} \frac{C_1}{1+n(s-\tau)}\; ds\\
&+ \frac{1}{2\pi } \int_{-1/2+\theta}^{-1/2+\theta + \frac{C_1-1}{n}}\; ds + \frac{1}{2\pi} \int_{-1/2+\theta}^{-1/2 + \theta + \frac{C_1-1}{n}} \frac{C_1}{1+n(s-\tau)}\; ds\\
&\stackrel{(a)}{\leq } \frac{1}{\pi} \left(\frac{3}{2} + \frac{1}{8}\right) \log(\frac{1/2 - \frac{C_1-1}{n}}{1/2}) + \frac{1}{\pi} \left(\frac{3}{2} + \frac{1}{8}\right) \left|-\log(\theta-s) + \log(\frac{1}{n}+s-\tau)\right|_{\tau + \frac{C_1-1}{n}}^{-\frac{1}{2n+3}}\frac{C_1/n}{\theta - \tau + 1/n}\\
&+ \frac{1}{2\pi}\frac{C_1-1}{n} + \frac{1}{2\pi} \frac{C_1}{n} \log(\frac{ - 1/2 + \theta + \frac{C_1}{n}}{1/n - 1/2 + \theta - \tau})\\
&\stackrel{(b)}{\leq }\frac{1}{\pi}\left(\frac{3}{2} + \frac{1}{8}\right) 2\frac{C_1-1}{n} + \frac{1}{\pi} \left(\frac{3}{2}+ \frac{1}{8}\right) \left(-\log(\frac{\theta + \frac{1}{2n+3}}{\theta - \tau - \frac{C_1-1}{n}}) + \log(\frac{1/n - \frac{1}{2n+3} - \tau}{1/n + \frac{C_1-1}{n} - \tau})\right) \frac{C_1/n}{\theta - \tau + 1/n}\\
&+ \frac{1}{2\pi} \frac{C_1-1}{n} + \frac{1}{2\pi} \frac{C_1}{n} \log(n|\tau|)\\
&\stackrel{(c)}{\leq } \frac{1}{\pi}\left(\frac{3}{2} + \frac{1}{8}\right) 2\frac{C_1-1}{n}  + \frac{C_1}{n|\tau|} \frac{1}{\pi} \left(\frac{3}{2}+ \frac{1}{8}\right)\log(1+ \frac{-\tau - \frac{C_1-1}{n} - \frac{1}{2n+3}}{\theta + \frac{1}{2n+3}}) \\
&+ \frac{C_1}{n|\tau|}\frac{1}{\pi} \left(\frac{3}{2}+ \frac{1}{8}\right) \log(1+ \frac{C_1}{n|\tau|}) + \frac{1}{2\pi} \frac{C_1-1}{n|\tau|} + \frac{1}{2\pi} \frac{C_1}{n|\tau|} \log(n|\tau|)\\
&\stackrel{(d)}{\leq }\frac{1}{\pi}\left(\frac{3}{2} + \frac{1}{8}\right) 2\frac{C_1-1}{n}  +  \frac{C_1}{n|\tau|}\frac{1}{\pi} \left(\frac{3}{2}+ \frac{1}{8}\right) \log(1+3n|\tau|)\\
& + \frac{C_1}{n|\tau|}\frac{1}{\pi} \left(\frac{3}{2}+ \frac{1}{8}\right)\log(2) + \frac{1}{2\pi} \frac{C_1}{n|\tau|} + \frac{1}{2\pi}\frac{C_1}{n|\tau|} \log(n|\tau|)\\
&\stackrel{(e)}{\leq }\frac{1}{\pi}\left(\frac{3}{2} + \frac{1}{8}\right) 2\frac{C_1-1}{n} +  \frac{C_1}{n|\tau|} \frac{1}{\pi} \left(\frac{3}{2}+ \frac{1}{8}\right)\log(4n|\tau|) + \frac{C_1}{n|\tau|}\frac{1}{\pi} \left(\frac{3}{2}+ \frac{1}{8}\right) \log(2)\\
&+ \frac{1}{2\pi} \frac{C_1}{n|\tau|} + \frac{1}{2\pi}\frac{C_1}{n|\tau|} \log(n|\tau|)
\end{align}
In the sequence of inequalities above, $(b)$ follows from $\tau + \frac{C_1-1}{n}\geq -1/2+\theta$ which implies $1/2-\theta\geq -\tau - \frac{C_1-1}{n}\geq \frac{C_1-1}{n}$ as soon as $|\tau|\geq 2\frac{C_1-1}{n}$. Both $(c)$ and $(d)$ follow from assuming $|\tau|\geq \frac{2C_1}{n}$. 
\item For $\tau>0$, the case $\tau> \frac{C_1-1}{n}$ as we have $\Delta>C_1/n$. We have 
\begin{align}
&\frac{1}{\pi \theta} \int_{-1/2+ \theta}^{-\frac{1}{2n+3}} \left(\frac{3}{2}+ \frac{1}{8}\right) \frac{\theta}{\theta-s} \frac{C_1}{1+n(\tau-s)}\; ds + \frac{1}{\pi \theta} \int_{-1/2+\theta}^{-\frac{1}{2n+3}} \frac{\theta}{2} \frac{C_1}{1+n(\tau-s)}\; ds\\
&\stackrel{(a)}{\leq } \frac{1}{\pi} \left(\frac{3}{2}+ \frac{1}{8}\right)\left|-\log(\theta-s) + \log(\frac{1}{n} + \tau-s)\right|_{-1/2+\theta}^{-\frac{1}{2n+3}} \frac{1}{|\tau-\theta + \frac{1}{n}|}\\
&+ \frac{1}{2\pi} \frac{C_1}{n} \log(\frac{1/n +\tau + \frac{1}{2n+3}}{1/n + \tau + 1/2-\theta})\\
&\stackrel{(b)}{\leq }\frac{1}{\pi} \left(\frac{3}{2} + \frac{1}{8}\right) \left(-\log(\frac{\theta + \frac{1}{2n+3}}{1/2}) + \log(\frac{1/n + \tau + \frac{1}{2n+3}}{1/n + \tau + 1/2-\theta})\right) \frac{C_1/n}{|\tau - \theta +1/n|}\\
&+ \frac{1}{2\pi} \frac{C_1}{n} \log(2|\tau|n)\\
&\stackrel{(c)}{\leq } \frac{1}{\pi} \left(\frac{3}{2} + \frac{1}{8}\right) \left(\log(1+ \frac{1/n + \tau - \theta}{1/2}\vee \frac{-1/n - \tau + \theta}{1/n + \tau +1/2 - \theta})\right) \frac{C_1/n}{|\tau - \theta + 1/n|}\\
&+ \frac{1}{\pi}\left(\frac{3}{2} + \frac{1}{8}\right) \log(1+ \frac{\theta - \tau + 1/n}{1/n + \tau + \frac{1}{2n+3}}\vee \frac{\tau - \theta + 1/n}{\theta + \frac{1}{2n+3}}) \frac{C_1/n}{|\tau - \theta + 1/n|}\\
&+ \frac{1}{2\pi} \frac{C_1}{n} \log(2|\tau|n)\\
&\stackrel{(d)}{\leq } \frac{1}{\pi} \left(\frac{3}{2} + \frac{1}{8}\right) \left(\log(2)+ \log(n|\tau|)\right)\frac{C_1}{n|\tau|} + \frac{1}{\pi}\left(\frac{3}{2} + \frac{1}{8}\right) \left(\frac{C_1}{n|\tau|}\vee \log(2n|\tau|) \frac{2C_1}{n|\tau|}\right) \\
&+ \frac{1}{2\pi} \frac{C_1}{n} \log(2n|\tau|)\label{tauPostiveSmallD0minusReal}
\end{align}
In $(d)$, we use $\tau - \theta + 1/n \geq \theta + \frac{1}{2n+3}$ $\Rightarrow$ $\tau/2 \geq \theta + (\frac{1}{2n+3} - \frac{1}{n})(1/2)$ and hence $\frac{1}{(\tau - \theta + 1/n)} \leq \frac{2}{|\tau|}$. 
\item Finally, when $[\tau - \frac{C_1}{n}, \tau+\frac{C_1}{n}]\cap [-1/2+\theta, -\frac{1}{2n+3}]\neq \emptyset$, we write 
\begin{align}
&\frac{1}{\pi \theta} \int_{-1/2 + \theta}^{\tau - \frac{C_1-1}{n}} \left(\frac{3}{2} + \frac{1}{8}\right) \left(\frac{\theta}{\theta-s} + \frac{\theta}{2}\right) \frac{C_1}{1+n(\tau-s)}\; ds + \frac{1}{\pi \theta} \int_{\tau - \frac{C_1-1}{n}}^{\tau + \frac{C_1-1}{n}} \left(\frac{3}{2} + \frac{1}{8}\right) \frac{\theta}{\theta-s} + \frac{\theta}{2}\; ds\\
&+ \frac{1}{\pi \theta} \int_{\tau + \frac{C_1-1}{n}}^{-\frac{1}{2n+3}} \left(\frac{3}{2} + \frac{1}{8}\right) \frac{\theta}{\theta-s} \frac{C_1}{1+n(s-\tau)} + \frac{\theta}{2} \frac{C_1}{1+n(s-\tau)}\; ds\\
&\stackrel{(a)}{\leq } \frac{1}{\pi}\left(\frac{3}{2} + \frac{1}{8}\right) \left|\log(\theta - s) - \log(1/n + \tau-s)\right|_{-1/2+\theta}^{\tau - \frac{C_1-1}{n}}\frac{C_1/n}{(-1/n - \tau + \theta)}\\
&+ \frac{1}{2\pi} \frac{C_1}{n} \log(\frac{C_1/n}{1/n + \tau + 1/2-\theta})+ \frac{1}{\pi} \left(\frac{3}{2} + \frac{1}{8}\right) \log(\frac{-\tau - \frac{C_1-1}{n} + \theta}{-\tau + \frac{C_1-1}{n} + \theta}) + \frac{1}{\pi} \frac{C_1-1}{n}\\
&+ \frac{1}{\pi} \left(\frac{3}{2} + \frac{1}{8}\right) \left|-\log(\theta-s) + \log(1/n + s-\tau)\right|_{\tau + \frac{C_1-1}{n}}^{-\frac{1}{2n+3}}\frac{C_1/n}{|-\tau + 1/n + \theta|}\\
&+ \frac{1}{2\pi} \frac{C_1}{n} \log(\frac{1/n - \tau - \frac{1}{2n+3}}{C_1/n})\\
&\stackrel{(b)}{\leq } \frac{1}{\pi} \left(\frac{3}{2} + \frac{1}{8}\right) \left(\log(\frac{\theta - \tau + \frac{C_1-1}{n}}{1/2})  - \log(\frac{C_1/n}{1/n + \tau + 1/2 - \theta})\right) \frac{C_1/n}{\theta - \tau - \frac{1}{n}}\\
&+ \frac{1}{2\pi} \frac{C_1}{n} \log(C_1)+ \frac{1}{\pi} \left(\frac{3}{2} + \frac{1}{8}\right) \log(1+ 2\frac{C_1-1}{n} \frac{1}{-\tau - \frac{C_1-1}{n} + \theta}) + \frac{1}{\pi} 2\frac{C_1-1}{n}\\
&+ \frac{1}{\pi} \left(\frac{3}{2} + \frac{1}{8}\right) \left(-\log(\frac{\theta + \frac{1}{2n+3}}{\theta - \tau - \frac{C_1-1}{n}}) + \log(\frac{1/n - \frac{1}{2n+3} - \tau}{C_1/n})\right)\frac{C_1/n}{-\tau + \theta + 1/n}\\
&+ \frac{1}{2\pi} \frac{C_1}{n} \log(\frac{1/n - \tau - \frac{1}{2n+3}}{C_1/n})\\
&\stackrel{(c)}{\leq }\frac{1}{\pi} \left(\frac{3}{2} + \frac{1}{8}\right) \left(\log(1+ \frac{\theta - \tau - 1/n}{C_1/n}\vee \frac{-\theta + \tau + 1/n}{\theta - \tau + \frac{C_1-1}{n}})\right) \frac{C_1/n}{\theta - \tau - 1/n}\\
&+ \frac{1}{\pi}\left(\frac{3}{2} + \frac{1}{8}\right) \left(\log(1+ \frac{1/n + \tau - \theta}{1/2}\vee \frac{-1/n - \tau + \theta}{1/n + \tau + 1/2-\theta})\right) \frac{C_1/n}{|\theta - \tau - 1/n|}\\
&+ \frac{1}{2\pi} \frac{C_1}{n|\tau|} \log(C_1)+ \frac{1}{\pi} \left(\frac{3}{2} + \frac{1}{8}\right) 4\frac{C_1-1}{n|\tau|} + \frac{1}{\pi}2\frac{C_1-1}{n|\tau|}\\
&+ \frac{1}{\pi} \left(\frac{3}{2}+ \frac{1}{8}\right)\left(\log(1+ \frac{-\tau - \frac{C_1-1}{n} - \frac{1}{2n+3}}{\theta + \frac{1}{2n+3}}) + \log(\frac{2|\tau|n}{C_1})\right)\frac{C_1}{n|\tau|}+ \frac{1}{2\pi} \frac{C_1}{n} \log(\frac{2|\tau|n}{C_1})\\
&\stackrel{(d)}{\leq } \frac{1}{\pi} \left(\frac{3}{2}+ \frac{1}{8}\right) \left(\log(2)+ \log(\frac{n|\tau|}{C_1}) + \log(2)\right) \frac{2C_1}{n|\tau|}\label{tmpD0minusSmallTheta0001}\\
&+ \frac{1}{\pi}\left(\frac{3}{2}+ \frac{1}{8}\right) \frac{C_1}{n|\tau|} + \frac{1}{\pi}\left(\frac{3}{2}+ \frac{1}{8}\right) \left(2\frac{C_1}{n} + \frac{C_1}{n|\tau|}\right)\label{tmpD0minusSmallTheta0002}\\
& + \frac{1}{2\pi} \frac{C_1}{n|\tau|} \log(C_1) + \frac{1}{\pi}\left(\frac{3}{2} + \frac{1}{8}\right) \frac{4(C_1-1)}{n|\tau|} + \frac{1}{\pi} 2\frac{C_1-1}{n}\label{tmpD0minusSmallTheta0003}\\
&+ \frac{1}{\pi}\left(\frac{3}{2} + \frac{1}{8}\right) \left[\log(4n|\tau|) + \log(\frac{2n|\tau|}{C_1})\right] \frac{C_1}{n|\tau|}+ \frac{1}{2\pi}\frac{C_1}{n} \log(\frac{2|\tau|n}{C_1}). \label{tmpD0minusSmallTheta0004}
\end{align}
Both $(c)$ and $(d)$ rely on $|\tau|\geq 2\frac{C_1}{n}$. 
\end{itemize}

The bound on each framework then follows as in the case $\theta\geq \frac{1}{2n+3}$ by computing the sums,  
\begin{align}
&2(\eqref{tmpD0minusSmallTheta0001} + \eqref{tmpD0minusSmallTheta0002}+\eqref{tmpD0minusSmallTheta0003} + \eqref{tmpD0minusSmallTheta0004})\\
&\leq 2\left(5.3 + 0.2\log(C_1)\right)\frac{C_1}{n|\tau|} + 2\log(\frac{n|\tau|}{C_1})\frac{C_1}{n|\tau|} + 2\log(n|\tau|)\frac{C_1}{n|\tau|}. \label{strongerBoundD0minusSmallTheta}
\end{align}
as well as
\begin{align}
\eqref{tauNegativeSmallD0minusReal} + \eqref{tauPostiveSmallD0minusReal} &\leq 2\frac{C_1}{n|\tau|} + 1.2 \frac{C_1}{n|\tau|} \log(n|\tau|) + 1.6\frac{C_1}{n|\tau|} + 0.2 \frac{C_1}{n|\tau|}\log(n|\tau|)\\
&\leq 4\frac{C_1}{n|\tau|} + 1.5\frac{C_1}{n|\tau|}\log(n|\tau|). 
\end{align}

Taking the maximum (~\eqref{strongerBoundD0minusSmallTheta} in this case) of those two bounds gives a bound on the real part when $\theta<\frac{1}{2n+3}$. The total bound on the real part can thus be written as 
\begin{align}
\eqref{strongerBoundD0minusSmallTheta} \vee \eqref{boundD0minusLargethetastrong}\leq 2\log(n|\tau|)\frac{C_1}{n|\tau|} + 2\log(\frac{n|\tau|}{C_1})\frac{C_1}{n|\tau|} + 12\frac{C_1}{n|\tau|}. \label{totalRealD0minus}
\end{align}

To conclude, we control the imaginary part. Recall that we have 
\begin{align}
\log(\frac{\theta-s}{-s}) \leq \log(1+ \frac{\theta}{-s}) \leq \frac{\theta}{-s}
\end{align}
When $\tau>0$, we have 
\begin{align}
\frac{1}{\pi \theta} \int_{-1/2 + \theta}^{-\frac{1}{2n+3}} \frac{\theta}{-s} \frac{C_1}{1+n(\tau-s)}\; ds &\leq \frac{1}{\pi} \frac{C_1}{n} \left|\log(-s)  - \log(1/n + \tau - s)\right|_{-1/2 + \theta}^{-\frac{1}{2n+3}} \frac{1}{|\tau+1/n|}\\
&\leq \frac{1}{\pi}\frac{C_1}{n} \left|\log(\frac{\frac{1}{2n+3}}{1/2-\theta}) - \log(\frac{1/n + \tau + \frac{1}{2n+3}}{1/n + \tau + 1/2 - \theta})\right| \frac{1}{|\tau + 1/n|}\\
&\leq \frac{1}{\pi} \frac{C_1}{n} \left|\log(1+ \frac{1/n + \tau}{\frac{1}{2n+3}})+ \log(1+ \frac{1/n+\tau}{1/2-\theta})\right| \frac{1}{|\tau|}\\
&\leq \frac{1}{\pi} \frac{C_1}{n|\tau|} 2\log(4n|\tau|)\label{tauPositiveImag}
\end{align}
When $\tau<0$, we have 
\begin{align}
\frac{1}{\pi \theta} \int_{-1/2 + \theta}^{-\frac{1}{2n+3}} \frac{\theta}{-s} \frac{C_1}{1+n(s-\tau)}\; ds &\stackrel{(a)}{\leq } \frac{1}{\pi} \left(-\log(\frac{\frac{1}{2n+3}}{1/2-\theta}) + \log(\frac{1/n - \tau - \frac{1}{2n+3}}{1/n - 1/2 + \theta - \tau})\right) \frac{1}{|\tau + 1/n|} \frac{C_1}{n}\\
&\stackrel{(b)}{\leq } \frac{4}{\pi} \log(3n|\tau|) \frac{C_1}{n|\tau|}\label{tauNegativeImag}
\end{align}
In $(b)$ we have used $|\tau|\geq 1/2-\theta$. When $\tau + \frac{C_1-1}{n}\geq -1/2+\theta$, we have 
\begin{align}
&\frac{1}{\pi \theta} \int_{\tau + \frac{C_1-1}{n}}^{-\frac{1}{2n+3}} \frac{\theta}{-s} \frac{C_1}{1+n(s-\tau)}\; ds + \frac{1}{\pi \theta} \int_{-1/2+\theta}^{\tau + \frac{C_1-1}{n}} \frac{\theta}{s} \; ds\\
& \frac{1}{\pi} \left(-\log(\frac{\frac{1}{2n+3}}{\tau + \frac{C_1-1}{n}}) + \log(\frac{1/n - \tau - \frac{1}{2n+3}}{C_1/n})\right)\frac{C_1}{n}\frac{1}{|-\tau + 1/n|}\\
& + \frac{1}{\pi} \log(1+ \frac{C_1-1}{n}\frac{1}{\tau + \frac{C_1-1}{n}}) \\
&\leq \frac{4}{\pi} \left(\log(3n|\tau|) \frac{C_1}{n|\tau|} + \frac{C_1-1}{n|\tau|}\right)
\end{align}

Finally when $-1/2 + \theta \leq \tau \leq -\frac{1}{2n+3}$, we write 
\begin{align}
&\frac{1}{\pi \theta} \int_{-1/2+\theta}^{\tau - \frac{C_1-1}{n}} \frac{\theta}{-s} \frac{C_1}{1+n(\tau-s)}\; ds\\
&+ \frac{1}{\pi \theta} \int_{\tau - \frac{C_1-1}{n}}^{\tau + \frac{C_1-1}{n}} \frac{\theta}{-s}\; ds\\
&+ \frac{1}{\pi \theta} \int_{\tau + \frac{C_1-1}{n}}^{-\frac{1}{2n+3}} \frac{\theta}{-s} \frac{C_1}{1+n(s-\tau)}\; ds\\
&\leq \frac{1}{\pi} \left|\log(-s) - \log(1/n + \tau - s)\right|_{-1/2+\theta}^{\tau - \frac{C_1-1}{n}} \frac{C_1/n}{\tau + 1/n}\\
&+ \frac{1}{\pi} \left|-\log(-s) + \log(1/n + s-\tau)\right|_{\tau + \frac{C_1-1}{n}}^{-\frac{1}{2n+3}} \frac{C_1/n}{|\tau - 1/n|}\\
&+ \frac{1}{\pi} \frac{4}{|\tau|} \frac{C_1-1}{n}\\
&\leq \frac{1}{\pi} \left|\log(\frac{-\tau + \frac{C_1-1}{n}}{1/2 - \theta}) - \log(\frac{C_1/n}{1/n + \tau + 1/2 - \theta})\right| \frac{C_1/n}{|\tau + 1/n|}\\
&+ \frac{1}{\pi} \left(-\log(\frac{\frac{1}{2n+3}}{-\tau - \frac{C_1-1}{n}}) + \log(\frac{C_1/n}{1/n - \frac{1}{2n+3} - \tau})\right) \frac{2C_1}{n|\tau|}\\
&+ \frac{4}{\pi |\tau|} \frac{C_1-1}{n}\\
&\leq \frac{4}{\pi} \left(\log(3n|\tau|) + \log(2\frac{|\tau|n}{C_1})\right) \frac{C_1}{n|\tau|} + \frac{1}{\pi}\frac{4C_1}{n|\tau|}\label{boundCentralD0minusImag}
\end{align}

As for the real part, the final bound on the imaginary part is then derived as 
\begin{align}
(\eqref{tauPositiveImag} + \eqref{tauNegativeImag} )\vee 2\eqref{boundCentralD0minusImag} \leq 7\frac{C_1}{n|\tau|} + 6\frac{C_1}{n|\tau|} \log(n|\tau|)\label{boundD0minusImag}
\end{align}

summing twice the bound~\eqref{totalRealD0minus} to~\eqref{boundD0minusImag} gives the reresult of the lemma. 

\begin{align}
&2\left(2\log(n|\tau|)\frac{C_1}{n|\tau|} + 2\log(\frac{n|\tau|}{C_1})\frac{C_1}{n|\tau|} + 12\frac{C_1}{n|\tau|}\right)\\
&+7\frac{C_1}{n|\tau|} + \frac{16}{\pi}\frac{C_1}{n|\tau|} \log(n|\tau|)\\
&\leq \frac{19C_1}{n|\tau|} + 14\frac{C_1}{n|\tau|}\log(n|\tau|).
\end{align}

\subsection{\label{sectionproofD0plusLarge}Proof of lemma~\ref{lemmaTroncD0plus} ($D_0^+$ large $|\tau-\alpha|$)}

We now treat $D_0^+$. On this domain we have 
\begin{align}
F_{R,D_0}^+&\leq \frac{\pi}{2} + \frac{1}{4(2n+3)}\left(\frac{1}{2}+ \frac{1}{\pi}\right) \left(2(2n+3) + \frac{1}{s} + \frac{1}{\theta-s}\right) + \frac{\theta}{2}\\
F_{I,D_0}^+&\leq \frac{1}{4(2n+3)} \left(\frac{1}{\pi} + \frac{1}{2}\right) \left(\frac{1}{s} + \frac{1}{\theta-s}\right) + \log(\frac{s}{\theta-s}\vee \frac{\theta-s}{s})
\end{align}
The total bound is thus given by 
\begin{align}
\left( \frac{\pi}{2} + 2c(2n+3) + \frac{\theta}{2}\right) + 2 c \left(\frac{1}{s} + \frac{1}{\theta-s}\right) + \log(\frac{s}{\theta-s}\vee \frac{\theta-s}{s})
\end{align}

So that on $D_0^+$, we have the following bound on the interior integral 

We introduce the constant $d$ defined as 
\begin{align}
d\equiv \left(\frac{\pi}{2} + \frac{\theta}{2} + \frac{1}{2}\left(\frac{1}{2} + \frac{1}{\pi}\right)\right) 
\end{align}
We start with the constant term. Depending on whether $\tau < \frac{1}{2n+3}$, $\tau > \theta -\frac{1}{2n+3}$, we get 
\begin{itemize}
\item $\tau < \frac{1}{2n+3}$. In this case, as $|\tau|>\frac{C_1}{n}$, we must have $\tau<0$. hence 
\begin{align}
&\frac{1}{\pi \theta} \int_{\frac{1}{2n+3}}^{\theta - \frac{1}{2n+3}} \left(\frac{\pi}{2} + \frac{\theta}{2} + \frac{1}{2}\left(\frac{1}{2}+ \frac{1}{\pi}\right)\right) \frac{C_1}{1+n(s-\tau)}\; ds\\
&\leq \frac{1}{\pi \theta} \left(\frac{\pi}{2} + \frac{\theta}{2} + \frac{1}{2}\left(\frac{1}{2}+ \frac{1}{\pi}\right)\right) \log(\frac{1/n + \theta - \frac{1}{2n+3} - \tau}{\frac{1}{n}+ \frac{1}{2n+3} - \tau})\\
&\leq \frac{1}{\pi}\left(\frac{\pi}{2} + \frac{\theta}{2} + \frac{1}{2}\left(\frac{1}{2}+ \frac{1}{\pi}\right)\right)\log(1+ \frac{\theta - \frac{2}{2n+3}}{1/n+\frac{1}{2n+3} - \tau})\\
&\leq \frac{1}{\pi}\left(\frac{\pi}{2} + \frac{\theta}{2} + \frac{1}{2}\left(\frac{1}{2}+ \frac{1}{\pi}\right)\right)\frac{C_1}{n|\tau|} \label{constantNegativeTau}
\end{align}

\item $\tau > \theta - \frac{1}{2n+3}$. In this case, we write 
\begin{align}
&\frac{1}{\pi \theta} \int_{\frac{1}{2n+3}}^{\theta - \frac{1}{2n+3} \wedge \tau - \frac{C_1-1}{n}} \left(\frac{\pi}{2} + \frac{\theta}{2} + \frac{1}{2}\left(\frac{1}{2} + \frac{1}{\pi}\right)\right)\frac{C-1}{1+n(\tau -s)}\; ds\\
&\leq \frac{1}{\pi \theta} \left(\frac{\pi}{2} + \frac{\theta}{2} + \frac{1}{2}\left(\frac{1}{2} + \frac{1}{\pi}\right)\right) \frac{C_1}{n}\log(1+ \frac{\theta - \frac{1}{2n+3}}{\frac{1}{n}+\tau - \theta + \frac{1}{2n+3}})\label{tmpDomainD0plus1tronc}
\end{align}
To control the last line, we consider two cases. Either $\theta -\frac{1}{2n+3}<\tau <2\theta$, in this case, we have
\begin{align}
\eqref{tmpDomainD0plus1tronc}\leq \frac{2}{\pi |\tau|}\left(\frac{\pi}{2} + \frac{\theta}{2} + \frac{1}{2}\left(\frac{1}{2} + \frac{1}{\pi}\right)\right)\frac{C_1}{n}\log(1+ \frac{\tau n}{C_1}) \leq \frac{2}{\pi |\tau|}\frac{C_1}{n}\left(\frac{\pi}{2} + \frac{\theta}{2} + \frac{1}{2}\left(\frac{1}{2} + \frac{1}{\pi}\right)\right) \left\{\log(2)\vee \log(\frac{n|\tau|}{C_1})\right\}
\end{align}
Or $\tau>2\theta$. In this case, we write 
\begin{align}
\eqref{tmpDomainD0plus1tronc}\leq \frac{1}{\pi \theta} \frac{C_1}{n}\left(\frac{\pi}{2} + \frac{\theta}{2} + \frac{1}{2}\left(\frac{1}{2} + \frac{1}{\pi}\right)\right) \log(1+ \frac{\theta - \frac{1}{2n+3}}{\tau/2}) \leq \frac{1}{\pi } \frac{C_1}{n}\left(\frac{\pi}{2} + \frac{\theta}{2} + \frac{1}{2}\left(\frac{1}{2} + \frac{1}{\pi}\right)\right)\frac{2}{|\tau|}. 
\end{align}

\item Finally, when $\frac{1}{2n+3}<\tau\pm \frac{C_1}{n}< \theta - \frac{1}{2n+3}$, we consider three contributions
\begin{align}
&\frac{1}{\pi \theta} \int_{\frac{1}{2n+3}}^{\tau - \frac{C_1-1}{n}} \left(\frac{\pi}{2} + \frac{\theta}{2} + \frac{1}{2}\left(\frac{1}{2} + \frac{1}{\pi}\right)\right) \frac{C_1}{1+n|\tau-s|}\; ds\label{firstTermD0plustronc}\\
&+ \frac{1}{\pi \theta} \int_{\tau - \frac{C_1-1}{n}\vee \frac{1}{2n+3}}^{\tau + \frac{C_1-1}{n}\wedge \theta - \frac{1}{2n+3}} \left(\frac{\pi}{2} + \frac{\theta}{2} + \frac{1}{2}\left(\frac{1}{2} + \frac{1}{\pi}\right)\right)\; ds\label{secondTermD0plustronc}\\
&+ \frac{1}{\pi \theta} \int_{\tau+ \frac{C_1-1}{n}}^{\theta - \frac{1}{2n+3}} \left(\frac{\pi}{2} + \frac{\theta}{2} + \frac{1}{2}\left(\frac{1}{2} + \frac{1}{\pi}\right)\right) \frac{C_1}{1+n|s-\tau|}\; ds\label{thirdTermD0plustronc}
\end{align}
We will treat each of the terms above separately. For the first one, we have 
\begin{align}
\eqref{firstTermD0plustronc} \leq \frac{C_1}{n} \left(\frac{\pi}{2} + \frac{\theta}{2} + \frac{1}{2}\left(\frac{1}{2} + \frac{1}{\pi}\right)\right) \log(\frac{C_1/n}{1/n + \tau - \frac{1}{2n+3}}) \frac{1}{\pi \theta} \leq \frac{C_1}{n}\frac{2}{\pi|\tau|} \left(\frac{\pi}{2} + \frac{\theta}{2} + \frac{1}{2}\left(\frac{1}{2} + \frac{1}{\pi}\right)\right)  \log(\frac{n|\tau|}{C_1})\label{secondTermD0plusConstant1}
\end{align}
In the last line, we use $|\tau| \geq \frac{2C_1}{n}$ as well as $\theta>\tau  - \frac{C_1}{n}>|\tau|/2$ . For the second term, we have
\begin{align}
\eqref{secondTermD0plustronc}&\leq \frac{2(C_1-1)}{n}\frac{1}{\pi \theta} \left(\frac{\pi}{2} + \frac{\theta}{2} + \frac{1}{2}\left(\frac{1}{2} + \frac{1}{\pi}\right)\right) \leq \frac{2(C_1-1)}{n}\frac{1}{\pi \tau} \left(\frac{\pi}{2} + \frac{\theta}{2} + \frac{1}{2}\left(\frac{1}{2} + \frac{1}{\pi}\right)\right)\label{secondTermD0plusConstant2}
\end{align}

And for the last term~\eqref{thirdTermD0plustronc}, we get
\begin{align}
\eqref{thirdTermD0plustronc}&\leq \frac{1}{\pi \theta}\left(\theta - \frac{1}{2n+3} - \left(\tau+ \frac{C_1-1}{n}\right)\right) d \frac{C_1}{n}\log(\frac{1/n + \theta - \frac{1}{2n+3} - \tau}{C_1/n})\\
&\leq \frac{1}{\pi} d \frac{C_1}{n}\log(\frac{n}{2C_1})\\
&\leq \frac{d}{\pi} \left(\frac{C_1}{n|\tau|}\log(\frac{n|\tau|}{C_1}) + \frac{C_1}{n|\tau|}\right)\label{intermediatePart3D0plustronc2}
\end{align}
In~\eqref{intermediatePart3D0plustronc2}, we use 
\begin{align}
\frac{C_1}{n}\log(\frac{n}{C_1}) \leq \frac{C_1}{n}\log(\frac{n|\tau|}{C_1|\tau|}) \leq \frac{C_1}{n|\tau|} + \frac{C_1}{n|\tau|} \log(\frac{n|\tau|}{C_1}).
\end{align}
which follows from $|\tau|\leq 1$. 

Combining~\eqref{intermediatePart3D0plustronc2}, ~\eqref{secondTermD0plusConstant1} and~\eqref{secondTermD0plusConstant2} gives 
\begin{align}
\eqref{intermediatePart3D0plustronc2}+\eqref{secondTermD0plusConstant1}+\eqref{secondTermD0plusConstant2}\leq \frac{3d}{\pi} \frac{C_1}{n|\tau|}\log(\frac{n|\tau|}{C_1}) + \frac{3d}{\pi}\frac{C_1}{n|\tau|}\label{boundConstantD0plusCentral}
\end{align}

To get the final contribution from the constant $d$ term accounting for the modulo $1$, we multiply this bound by $2$ and combine~\eqref{constantNegativeTau} and~\eqref{tmpDomainD0plus1tronc}, then take the maximum over the two frameworks,

\begin{align}
2\eqref{boundConstantD0plusCentral}\vee (\eqref{constantNegativeTau} + \eqref{tmpDomainD0plus1tronc})&\leq \frac{3d}{\pi }\frac{C_1}{n|\tau|} + \frac{2d}{\pi}\frac{C_1}{n|\tau|}\log(\frac{n|\tau|}{C_1}) \vee 2\left\{\frac{3d}{\pi} \frac{C_1}{n|\tau|}\log(\frac{n|\tau|}{C_1}) + \frac{3d}{\pi}\frac{C_1}{n|\tau|}\right\}\\
&\leq  \frac{6d}{\pi} \frac{C_1}{n|\tau|}\log(\frac{n|\tau|}{C_1}) + \frac{6d}{\pi}\frac{C_1}{n|\tau|}\\
&\leq 4.5 \frac{C_1}{n|\tau|}\log(\frac{n|\tau|}{C_1}) + 4.5\frac{C_1}{n|\tau|}\label{boundDomainD0plusFFF1}
\end{align}

We now derive the bound on the remaining $1/s$ and $1/(\theta-s)$ terms. When $\tau < \frac{1}{2n+3}$, we respectively write 
\begin{align}
\frac{c}{\pi \theta} \int_{\frac{1}{2n+3}}^{\theta - \frac{1}{2n+3}} \frac{1}{s} \frac{C_1}{1+n|s-\tau|} \; ds &\leq \frac{c}{\pi \theta}  \int_{\frac{1}{2n+3}}^{\theta - \frac{1}{2n+3}} \left(-\frac{1}{s} + \frac{1}{1/n + (s-\tau)}\right) \frac{C_1/n}{-\tau + \frac{1}{n}}\\
&\leq \frac{c}{\pi \theta} \left|-\log(\frac{\theta - \frac{1}{2n+3}}{\frac{1}{2n+3}}) + \log(\frac{1/n + \theta - \frac{1}{2n+3} - \tau}{1/n + \frac{1}{2n+3} - \tau})\right| \frac{C_1/n}{-\tau + \frac{1}{n}}\\
&\leq \frac{c}{\pi \theta} \left|\log(1+ \frac{\tau - 1/n}{1/n + \theta - \tau - \frac{1}{2n+3}}) + \log(1+ \frac{\tau - \frac{1}{n}}{1/n + \frac{1}{2n+3} - \tau})\right|\frac{C_1}{n}\frac{2}{|\tau|}
\end{align}
as well as 
\begin{align}
\frac{c}{\pi \theta} \int_{\frac{1}{2n+3}}^{\theta - \frac{1}{2n+3}} \frac{1}{\theta-s} \frac{C_1}{1+n(s-\tau)}\; ds&\leq 
\frac{c}{\pi \theta}\int_{\frac{1}{2n+3}}^{\theta - \frac{1}{2n+3}} \left(\frac{1}{\theta-s} + \frac{1}{1/n + (s-\tau)}\right) \frac{C_1/n}{\theta - \tau +1/n}\; ds\\
&\leq \frac{c}{\pi \theta} \left|\log(\frac{\frac{1}{2n+3}}{\theta - \frac{1}{2n+3}}) + \log(\frac{1/n + \theta - \frac{1}{2n+3} - \tau}{1/n + \frac{1}{2n+3} - \tau})\right|\frac{C_1/n}{|\theta - \tau + 1/n|}
\end{align}
both terms can thus be bounded as
\begin{align}
\frac{c}{\pi \theta} \int_{\frac{1}{2n+3}}^{\theta - \frac{1}{2n+3}} \left(\frac{1}{s} + \frac{1}{\theta-s}\right) \frac{C_1}{1+n|s-\tau|} \; ds \leq \frac{4c}{\pi \theta} \log(n|\tau|) \frac{2C_1}{n|\tau|}. \label{boundTauNegativeD0plus}
\end{align}

When $\tau > \theta - \frac{1}{2n+3}$, a similar reasoning gives 
\begin{align}
\frac{c}{\pi \theta} \int_{\frac{1}{2n+3}}^{\theta - \frac{1}{2n+3}} \frac{1}{s} \frac{C_1}{1+n(\tau-s)}\; ds& \leq \frac{c}{\pi \theta} \int_{\frac{1}{2n+3}}^{\theta - \frac{1}{2n+3}} \left(\frac{1}{s} + \frac{1}{1/n + (\tau -s)}\right) \frac{C_1/n}{\tau + 1/n}\\
&\leq \frac{c}{\pi \theta} \left|\log(\frac{\theta - \frac{1}{2n+3}}{\frac{1}{2n+3}}) - \log(\frac{1/n - \theta + \frac{1}{2n+3}+\tau}{1/n + \tau - \frac{1}{2n+3}})\right| \frac{C_1/n}{\tau + 1/n}\\
&\leq \frac{c}{\pi \theta}\frac{C_1}{n|\tau|} \left(\log(3n|\tau|) + \log(2|\tau|n)\right)\\
&\leq \frac{2c}{\pi \theta}\frac{C_1}{n|\tau|} \log(3n|\tau|)\label{boundTauLargerThanThetaDomainD0plus1onS}
\end{align}
Whenever $\tau - \frac{C_1-1}{n} \geq \theta - \frac{1}{2n+3}$, we have 
\begin{align}
\frac{c}{\pi \theta} \int_{\frac{1}{2n+3}}^{\tau - \frac{C_1-1}{n}} \left(\frac{1}{s} + \frac{1}{1/n + (\tau-s)}\right)\frac{C_1/n}{\tau + 1/n}\; ds&\leq \frac{c}{\pi \theta} \left|\log(\frac{\tau - \frac{C_1-1}{n}}{\frac{1}{2n+3}}) + \log(\frac{C_1/n}{1/n + \tau - \frac{1}{2n+3}})\right|\\
&\leq \frac{c}{\pi \theta} \frac{C_1}{n|\tau|} \left(\log(3n|\tau|) + \log(\frac{2|\tau|n}{C_1})\right)
\end{align}
Similarly, the $\frac{1}{\theta-s}$ term gives
\begin{align}
&\frac{c}{\pi \theta} \int_{\frac{1}{2n+3}}^{\theta - \frac{1}{2n+3}} \frac{1}{\theta-s} \frac{C_1}{1+n(\tau-s)}\; ds\\
&\leq \frac{c}{\pi \theta} \frac{C_1}{n}\int_{\frac{1}{2n+3}}^{\theta - \frac{1}{2n+3}} \left(-\frac{1}{\theta-s} + \frac{1}{1/n+\tau -s}\right)\frac{1}{\tau - \theta + 1/n}\\
&\leq \frac{C_1}{n}\frac{c}{\pi \theta} \left|-\log(\frac{1}{2n+3}\frac{1}{\theta - \frac{1}{2n+3}}) + \log(\frac{1/n + \tau - \theta + \frac{1}{2n+3}}{1/n + \tau - \frac{1}{2n+3}})\right|\frac{1}{\tau - \theta + 1/n}\label{tmpDomainD0pluslemmatronc}
\end{align}
To bound the last line, we make the distinction between the case $\tau>2\theta$ and $\theta<\tau< 2\theta$. In the former case, we simply write 
\begin{align}
\eqref{tmpDomainD0pluslemmatronc}\leq \frac{c}{\pi \theta} \frac{C_1}{n} \frac{2}{|\tau|}2\log(3n|\tau|)
\end{align}
In the latter, we have 
\begin{align}
\eqref{tmpDomainD0pluslemmatronc}&\leq \frac{c}{\pi \theta} \frac{1}{\tau - \theta + 1/n} \left( \log(1+ \frac{\tau - \theta + 1/n}{\frac{1}{2n+3}}) + \log(1+ \frac{\tau - \theta + 1/n}{\theta - \frac{1}{2n+3}})\right)\frac{C_1}{n}\\
&\leq \frac{4}{\pi|\tau|}c'\frac{C_1}{n} 
\end{align}
The last line follows from $\theta \geq \tau/2$ and $\theta>\frac{2}{2n+3}$. 
Finally when $\tau - \frac{C_1-1}{n}\leq \theta - \frac{1}{2n+3}$, just as for the $1/s$ term, we write 
\begin{align}
\frac{c}{\pi \theta} \int_{\frac{1}{2n+3}}^{\tau - \frac{C_1-1}{n}}\frac{1}{\theta-s} \frac{C_1}{1+n(\tau - s)}\; ds&\leq \frac{c}{\pi \theta} \left|-\log(\frac{\theta - \tau + \frac{C_1-1}{n}}{\theta - \frac{1}{2n+3}}) + \log(\frac{C_1/n}{1/n+\tau - \frac{1}{2n+3}})\right| \frac{1}{\tau - \theta + 1/n} \frac{C_1}{n}\label{tmpDomainD0pluslemmatronc2}
\end{align}
Again, we make the distinction between the case $\theta<\tau<2\theta$ for which we have 
\begin{align}
\eqref{tmpDomainD0pluslemmatronc2}&\leq \frac{c}{\pi \theta} \frac{C_1}{n} \left(\log(1+ \frac{\tau - \theta + 1/n}{\theta - \tau + \frac{C_1-1}{n}}) + \log(1+ \frac{\tau - \theta + 1/n}{\theta - \frac{1}{2n+3}})\right)\frac{1}{\tau - \theta + 1/n} \\
&\leq \frac{c'}{\pi}\frac{2}{\tau} \frac{C_1}{n}.
\end{align}

In any of the two frameworks $\tau < 0 $ and $\tau-\frac{C_1}{n}>\theta-\frac{1}{2n+3}$, the total bound can thus be obtained as before, by summing~\eqref{tmpDomainD0pluslemmatronc} and~\eqref{boundTauNegativeD0plus}. 
\begin{align}
(\eqref{tmpDomainD0pluslemmatronc}+\eqref{boundTauLargerThanThetaDomainD0plus1onS}) +  2\eqref{boundTauNegativeD0plus} &\leq \left(\frac{2c'}{\pi} \frac{C_1}{n|\tau|}\log(3n|\tau|) + \frac{4c'}{\pi} \frac{C_1}{n|\tau|} (1+\log(3n|\tau|))\right) \\
&+ 2\left(\frac{4c'}{\pi} \log(n|\tau|) \frac{2C_1}{n|\tau|}\right)\\
&\leq \frac{C_1}{n|\tau|}\log(n|\tau|) + \frac{C_1}{n|\tau|}. \label{TotalBoundFramewordOutsideDomainDomainD0plusReal}
\end{align}
We now treat the case $\tau \pm \frac{C_1}{n} \in [\frac{1}{2n+3}, \theta - \frac{1}{2n+3}]$. Starting with the $1/s$ term, we decompose the integral into the following three terms
\begin{align}
&\frac{c}{\pi \theta} \int_{\tau + \frac{C_1-1}{n}}^{\theta - \frac{1}{2n+3}} \frac{1}{s} \frac{C_1}{1+n(s-\tau)}\; ds\label{term1oneOversDomaineD0plusRealdecomp}\\ 
&+ \frac{c}{\pi \theta} \int_{\tau - \frac{C_1-1}{n}\vee \frac{1}{2n+3}}^{\tau + \frac{C_1-1}{n}\wedge \theta - \frac{1}{2n+3}} \frac{1}{s} \; ds\label{centralIntegraltronclemma1D0plus}\\
&+ \frac{c}{\pi \theta} \int_{\frac{1}{2n+3}}^{\tau - \frac{C_1-1}{n}} \frac{1}{s} \frac{C_1}{1+n(\tau-s)}\; ds\label{term3oneOversDomaineD0plusRealdecomp}
\end{align}
For the first and last integrals, we respectively have 
\begin{align}
&\frac{c}{\pi \theta} \left|-\log(s) + \log(1/n +s-\tau)\right|_{\tau + \frac{C_1-1}{n}}^{\theta - \frac{1}{2n+3}} \frac{C_1}{n}\frac{1}{\tau - \frac{1}{n}}\\
&\leq \frac{c}{\pi \theta}\frac{2C_1}{n|\tau|} \left|-\log(\frac{\theta - \frac{1}{2n+3}}{\tau + \frac{C_1-1}{n}}) + \log(\frac{1/n + \theta - \frac{1}{2n+3} - \tau}{C_1/n})\right|\\
&\leq \frac{c}{\pi \theta} \frac{2C_1}{n|\tau|}\left|\log(1+ \frac{\tau - 1/n}{1/n+\theta -\frac{1}{2n+3} - \tau}) + \log(1+\frac{\tau - 1/n}{C_1/n})\right|\\
&\leq \frac{c}{\pi \theta} \frac{2C_1}{n|\tau|} \left|2\log(2)\vee 2\log(|\tau|n)\right|\label{domainD0plusFirstIntegralOneOverS1}
\end{align}
as well as
\begin{align}
&\frac{c}{\pi \theta} \left|\log(s) - \log(1/n + \tau-s)\right|_{\frac{1}{2n+3}}^{\tau - \frac{C_1-1}{n}}\frac{C_1}{n}\frac{1}{1/n +\tau}\\
&\leq \frac{c}{\pi \theta} \frac{C_1}{n|\tau|} \left|\log(\frac{\tau - \frac{C_1-1}{n}}{\frac{1}{2n+3}}) - \log(\frac{C_1}{1/n+\tau - \frac{1}{2n+3}})\right|\\
&\leq \frac{c}{\pi \theta} \frac{C_1}{n|\tau|} \left(\log(3n|\tau|) + \log(\frac{n|\tau|}{C_1})\right)
\label{domainD0plusSecondIntegralOneOverS2}
\end{align}
For the central (remaining) integral, we consider four possible frameworks. Either $\tau + \frac{C_1-1}{n}\leq \theta - \frac{1}{2n+3}$, in this case we have 
\begin{align}
\eqref{centralIntegraltronclemma1D0plus}&\leq \frac{c}{\pi \theta} \log(1+ 2\frac{C_1-1}{n}\frac{1}{\tau - \frac{C_1-1}{n}})\\
&\leq \frac{c}{\pi \theta} 2\frac{C_1-1}{n}\frac{2}{|\tau|}
\end{align}
Or $\tau+ \frac{C_1-1}{n}\geq \theta - \frac{1}{2n+3}$. In this case, we write
\begin{align}
\eqref{centralIntegraltronclemma1D0plus}&\leq \frac{c}{\pi \theta} \log(1+ \left(\tau + \frac{C_1-1}{n} - \frac{1}{2n+3}\right)\frac{1}{1/(2n+3)})\\
&\stackrel{(a)}{\leq} \frac{c}{\pi \theta} 2\frac{C_1-1}{n}(2n+3)\\
 &\stackrel{(b)}{\leq} \frac{c'}{\pi \tau} 2\frac{C_1-1}{n}
\end{align}
In $(a)$, we use $\tau \leq \frac{1}{2n+3} + \frac{C_1-1}{n}$. In $(b)$, we use $\theta \geq \tau + \frac{C_1-1}{n} + \frac{1}{2n+3}$. 

Then we consider the case where $\tau + \frac{C_1-1}{n}\geq \theta - \frac{1}{2n+3}$, if $\theta - \frac{C_1-1}{n}\geq \frac{1}{2n+3}$, we get 
\begin{align}
\eqref{centralIntegraltronclemma1D0plus}&\leq \frac{c}{\pi \theta}\log(1+ \left(\theta - \frac{1}{2n+3} - \tau + \frac{C_1-1}{n}\right)\frac{1}{\tau - \frac{C_1-1}{n}})\\
&\leq \frac{c}{\pi} \frac{2}{\tau} \leq \frac{2c'}{\pi(2n+3)|\tau|}\label{boundCentralntegralDomainD0plusa}
\end{align}
else, when $\tau - \frac{C_1-1}{n}\leq \frac{1}{2n+3}$, we get
\begin{align}
\eqref{centralIntegraltronclemma1D0plus}& =  \frac{c}{\pi \theta} \int_{1/(2n+3)}^{\theta - \frac{1}{2n+3}} \frac{1}{s}\; ds\\
&\leq \frac{c}{\pi \theta} \log(1+ \frac{\theta - \frac{2}{2n+3}}{\frac{1}{2n+3}})\\
&\leq \frac{c}{\pi \theta} (\theta -\frac{2}{2n+3}) \frac{1}{\tau - \frac{C_1-1}{n}}\\
&\leq \frac{c'}{\pi}\frac{2}{|\tau|(2n+3)}\label{boundCentralntegralDomainD0plusb}
\end{align}
The last lines follow from $\tau - \frac{C_1-1}{n} \leq \frac{1}{2n+3}$. 

Grouping~\eqref{domainD0plusFirstIntegralOneOverS1},~\eqref{domainD0plusSecondIntegralOneOverS2} as well as~\eqref{boundCentralntegralDomainD0plusa} and~\eqref{boundCentralntegralDomainD0plusb}, we get

\begin{align}
\eqref{term1oneOversDomaineD0plusRealdecomp}+\eqref{centralIntegraltronclemma1D0plus} + \eqref{term3oneOversDomaineD0plusRealdecomp}& \leq \frac{4c'}{\pi|\tau|}\frac{C_1}{n} + \frac{2c'}{\pi}\frac{C_1}{n|\tau|}\left(2\log(2)+ 2\log(n|\tau|)\right) \\
&+ \frac{c'}{\pi}\frac{C_1}{n|\tau|}\left(\log(3) + \log(n|\tau|) + \log(\frac{n|\tau|}{C_1})\right)\\
&\leq 0.6 \frac{C_1}{n|\tau|} + 0.4 \frac{C_1}{n|\tau|}\log(n|\tau|). \label{boundCentralOneSDomainD0plusFF}
\end{align}

For the second term, a similar derivation holds 
\begin{align}
\frac{c}{\pi \theta} \int_{\frac{1}{2n+3}}^{\theta - \frac{1}{2n+3}} \frac{1}{\theta-s} \frac{C_1}{1+n(\tau-s)}\; ds & = \frac{c}{\pi \theta} \int_{\frac{1}{2n+3}}^{\tau - \frac{C_1-1}{n}} \frac{1}{\theta-s} \frac{C_1}{1+n(\tau-s)}\; ds\label{D0plusTerm1lemmatronc1}\\
&+ \frac{c}{\pi \theta} \int_{\tau - \frac{C_1-1}{n}}^{\tau + \frac{C_1-1}{n}} \frac{1}{\theta-s}\; ds\label{D0plusTerm1lemmatronc2}\\
&+ \frac{c}{\pi \theta} \int_{\tau + \frac{C_1-1}{n}}^{\theta - \frac{1}{2n+3}} \frac{1}{\theta-s} \frac{C_1}{1+n(s-\tau)}\; ds\label{D0plusTerm1lemmatronc3}
\end{align}
For the first and last integrals, we can write 
\begin{align}
\eqref{D0plusTerm1lemmatronc1}&\leq \frac{c}{\pi \theta} \left|-\log(\theta-s) + \log(1/n + \tau-s)\right|_{\frac{1}{2n+3}}^{\tau - \frac{C_1-1}{n}}\frac{1}{\tau - \theta + \frac{1}{n}}\\
&\leq \frac{c}{\pi \theta} \left|-\log(\frac{\theta - \tau + \frac{C_1-1}{n}}{\theta - \frac{1}{2n+3}}) + \log(\frac{C_1/n}{1/n+\tau - \frac{1}{2n+3}})\right| \frac{1}{\tau - \theta + \frac{1}{n}}\\
&\leq \frac{c}{\pi \theta} \left(\log(1+ \frac{\theta - \tau - \frac{1}{n}}{\tau + \frac{1}{n} - \frac{1}{2n+3}}) + \log(1+ \frac{\theta - \tau - \frac{1}{n}}{C_1/n})\right) \frac{1}{\tau - \theta + 1/n}\\
&\leq \frac{c}{\pi \theta} \frac{C_1}{|\tau|n} + \frac{c}{\pi \theta} \frac{C_1}{n}\frac{n}{C_1}\\
&\leq \frac{c}{\theta \pi} \frac{C_1}{|\tau|n} + \frac{2c'}{|\tau|\pi (2n+3)}.\label{boundOneOverThetaMinusSDOmainD0plusIntegral1}
\end{align}
 
For the last integral, we get 
\begin{align}
\eqref{D0plusTerm1lemmatronc3}&\leq \frac{c}{\pi \theta} \int_{\tau + \frac{C_1-1}{n}}^{\theta - \frac{1}{2n+3}} \frac{1}{\theta -s} \frac{C_1}{1+n(\tau-s)}\; ds\\
&\leq \frac{c}{\pi \theta} \left|-\log(\theta-s) + \log(\frac{1}{n}+s-\theta)\right|_{\tau + \frac{C_1-1}{n}}^{\theta - \frac{1}{2n+3}} \frac{C_1}{n}\frac{1}{|\theta - \tau + 1/n|}\\
&\leq \frac{c}{\pi \theta} \left(\log(\frac{\theta - \tau - \frac{C_1-1}{n}}{\frac{1}{2n+3}}) + \log(\frac{\theta - \tau + \frac{1}{n}-\frac{1}{2n+3}}{C_1/n})\right) \frac{C_1}{n}\frac{1}{|\theta - \tau + 1/n|}\\
&\leq \frac{c}{\pi \theta} (2n+3) \frac{C_1}{n} + \frac{c}{\pi \theta}\\
&\leq \frac{c'}{\pi |\tau|} \frac{C_1}{n} + \frac{c'}{\pi |\tau| (2n+3)}. \label{boundOneOverThetaMinusSDOmainD0plusIntegral3}
\end{align}
In the last line, we use $\theta \geq \frac{1}{2n+3}+ \tau + \frac{C_1-1}{n}\geq \tau$. Finally to control~\eqref{D0plusTerm1lemmatronc2}, we consider four frameworks. Either $\tau + \frac{C_1-1}{n} \leq \theta - \frac{1}{2n+3}$, in this case we write 
\begin{align}
\frac{c}{\pi \theta} \int_{\tau - \frac{C_1-1}{n}}^{\tau +\frac{C_1-1}{n}} \frac{1}{\theta-s}\; ds &\leq \frac{c}{\pi \theta} \log(1+ 2\frac{C_1-1}{n} \frac{1}{\theta - \tau + \frac{C_1-1}{n}})\\
&\leq \frac{c}{\pi \theta}\\
&\leq \frac{c'}{\pi (2n+3)|\tau|}\label{boundOneOverThetaMinusSDOmainD0plusIntegral2a}
\end{align}
In the last line, we used the assumption $\tau +\frac{C_1-1}{n} \leq \theta - \frac{1}{2n+3}$. Alternatively, when $\tau - \frac{C_1-1}{n}\leq \frac{1}{2n+3}$, we have
\begin{align}
\frac{c}{\pi \theta} \int_{\frac{1}{2n+3}}^{\tau+\frac{C_1-1}{n}} \frac{1}{\theta-s} \; ds &\leq \log(\frac{\theta - \tau  - \frac{C_1-1}{n}}{\theta - \frac{1}{2n+3}})\\
&\leq \frac{c}{\pi |\tau + \frac{C_1-1}{n}|}\label{additionalNotelemmatroncD0plusCentral1} \\
&\leq \frac{c'}{\pi (2n+3)|\tau|}\label{boundOneOverThetaMinusSDOmainD0plusIntegral2b}
\end{align}
In~\eqref{additionalNotelemmatroncD0plusCentral1}, we use the assumption $\tau+\frac{C_1-1}{n}\leq \theta - \frac{1}{2n+3}$.

Combining~\eqref{boundOneOverThetaMinusSDOmainD0plusIntegral1}~\eqref{boundOneOverThetaMinusSDOmainD0plusIntegral3} as well as~\eqref{boundOneOverThetaMinusSDOmainD0plusIntegral2a} and~\eqref{boundOneOverThetaMinusSDOmainD0plusIntegral2b}, we have 
\begin{align}
\eqref{D0plusTerm1lemmatronc1}+ \eqref{D0plusTerm1lemmatronc3} + \eqref{D0plusTerm1lemmatronc2} \leq 2(C_1+1) \frac{c'}{\pi n|\tau|}. \label{boundCentralOneOverThetaMinusSDomainD0plus}
\end{align}
Adding this result to~\eqref{boundCentralOneSDomainD0plusFF} gives the final bound for the case $\tau\pm \frac{C_1}{n} \in [\frac{1}{2n+3}, \theta - \frac{1}{2n+3}]$,
\begin{align}
\eqref{boundCentralOneOverThetaMinusSDomainD0plus} + \eqref{boundCentralOneSDomainD0plusFF}\leq 2 \frac{C_1}{n|\tau|} + 0.4\frac{C_1}{n|\tau|}\log(n|\tau|).\label{TotalBoundCentralFrameworkDomainD0plusReal111}
\end{align}
as soon as we can assume $C_1\geq1$.

When $\tau+ \frac{C_1-1}{n}\geq \theta - \frac{1}{2n+3}$, we again have two options. Either $\tau-\frac{C_1-1}{n}\geq \frac{1}{2n+3}$
\begin{align}
\frac{c}{\pi \theta} \int_{\tau - \frac{C_1-1}{n}}^{\theta - \frac{1}{2n+3}} \frac{1}{\theta-s}\; ds &\leq \frac{c}{\pi \theta} \log(\frac{\frac{1}{2n+3}}{\theta -\tau + \frac{C_1-1}{n}})\\
&\leq \frac{c}{\pi \theta} \log(\frac{\theta - \tau + \frac{C_1-1}{n}}{\frac{1}{2n+3}})\\
&\stackrel{(a)}{\leq }   \frac{c}{\pi \theta} \log(3n \frac{2C_1}{n})\\
&\stackrel{(b)}{\leq } \frac{c'}{\pi (2n+3)} \frac{2}{|\tau|}\log(6C_1)
\end{align}
In $(a)$, we use $\theta \leq \tau + \frac{C_1-1}{n}+\frac{1}{2n+3}$. In $(b)$, we use $\theta - \frac{1}{2n+3} \geq \tau - \frac{C_1-1}{n}\geq \tau/2$.

Or $\tau - \frac{C_1-1}{n}\leq \frac{1}{2n+3}$. In this last case, we have $\tau+\frac{C_1-1}{n}\geq \theta - \frac{1}{2n+3}$ and $\tau - \frac{C_1-1}{n} \leq \frac{1}{2n+3}$. We can write 
\begin{align}
\frac{c}{\pi \theta} \int_{\frac{1}{2n+3}}^{\theta - \frac{1}{2n+3}} \frac{1}{\theta-s} &\leq \frac{c}{\pi \theta} \log(\frac{\theta - \frac{1}{2n+3}}{\frac{1}{2n+3}})\stackrel{(a)}{\leq }   \frac{c'}{\pi(2n+3)} \frac{2}{|\tau|}
\end{align}
In $(a)$, we used $\tau - \frac{C_1-1}{n} \leq \frac{1}{2n+3}$ which implies $\tau/2 \leq \frac{1}{2n+3}$.

\end{itemize}

The total bound arising from the sum $1/s + 1/(\theta - s)$ follows as before by multiplying~\eqref{TotalBoundCentralFrameworkDomainD0plusReal111} by 2 and taking the maximum between the resulting upper bound and~\eqref{TotalBoundFramewordOutsideDomainDomainD0plusReal}, i.e.,
\begin{align}
2\eqref{TotalBoundCentralFrameworkDomainD0plusReal111}\vee \eqref{TotalBoundFramewordOutsideDomainDomainD0plusReal} \leq \frac{C_1}{n|\tau|}\log(n|\tau|) + 2\frac{C_1}{n|\tau|}\label{boundDomainD0plusFFF2}
\end{align}

We conclude with the last term $\log(\frac{s}{\theta-s} \vee \frac{\theta-s}{s})$ (arising from the imaginary part). For this last term, we thus have the bound 
\begin{align}
\log(\frac{s}{\theta-s} \vee \frac{\theta-s}{s}) \leq \frac{s}{\theta-s} \leq \frac{\theta}{\theta-s}, \quad \text{when $\theta/2 \leq s \leq \theta$}\\
\log(\frac{s}{\theta-s} \vee \frac{\theta-s}{s}) \leq \frac{\theta}{s}, \quad \text{when $s\leq \theta/2 $}
\end{align}

When $\tau < 0$, we have 
\begin{align}
&\frac{1}{\pi} \frac{C_1}{n} \int_{\frac{1}{2n+3}}^{\theta/2}  \frac{1}{s} \frac{1}{1+n (s-\tau)}\; ds + \frac{C_1}{n} \frac{1}{\pi} \int_{\theta/2}^{\theta - \frac{1}{2n+3}} \frac{1}{\theta-s} \frac{C_1}{1+n(s-\tau)}\; ds\\
&\leq \frac{C_1}{n\pi} \left|-\log(s) + \log(\frac{1}{n} + s-\tau)\right|_{\frac{1}{2n+3}}^{\theta/2} \frac{1}{|\tau - \frac{1}{n}|}\label{IntegralDomainD0plusImagFramework1}\\
&+ \frac{C_1}{n\pi} \left|-\log(\theta-s) + \log(\frac{1}{n} + s-\tau)\right|_{\theta/2}^{\theta - \frac{1}{2n+3}} \frac{1}{|\tau - \frac{1}{n} - \theta|}\\
&\leq \frac{C_1}{n\pi} \left(-\log(\frac{\theta/2}{\frac{1}{2n+3}}) + \log(\frac{1/n + \theta/2 - \tau}{1/n + \frac{1}{2n+3} - \tau})\right) \frac{1}{|\tau - \frac{1}{n}|}\\
&+ \frac{C_1}{n\pi} \left(- \log(\frac{\theta/2}{\frac{1}{2n+3}}) + \log(\frac{1/n + \theta/2 - \frac{1}{2n+3} - \tau}{1/n + \frac{1}{2n+3} - \tau})\right) \frac{1}{|\tau - \frac{1}{n}|} \\
&+ \frac{C_1}{n\pi} \left(-\log(\frac{\frac{1}{2n+3}}{\theta/2}) + \log(\frac{1/n + \theta - \frac{1}{2n+3} - \tau}{1/n + \theta/2 - \tau})\right) \frac{1}{|\tau - \frac{1}{n} - \theta|}\\
&\leq \frac{C_1}{n\pi} \frac{2}{|\tau|} \left(\log(1+ \frac{-\tau + 1/n}{\theta/2}) + \log(1+ \frac{-\tau + 1/n}{\frac{1}{2n+3}})\right) \frac{1}{|\tau + 1/n|}\\
&+ \frac{C_1}{n\pi} \left(\log(\frac{3n\theta}{2}) + \frac{1/n + \theta/2 - \tau}{1/n + \theta - \frac{1}{2n+3} - \tau}\right) \frac{1}{|\theta - \tau + 1/n|}\\
&\leq \frac{C_1}{n\pi} \frac{4}{|\tau|} \log(7n|\tau|) + \frac{C_1}{n\pi} \left(\left(\log(\frac{3\theta}{2|\tau|}) + \log(n|\tau|)\right) \frac{1}{|\theta| + |\tau|} + \frac{3}{|\tau|}\right)\\
&\leq \frac{C_1}{n\pi} \frac{4}{|\tau|} \left(\log(7) + \log(n|\tau|)\right) + \frac{3C_1}{n \pi |\tau|} + \frac{3C_1}{2n\pi|\tau|} + \frac{C_1}{n\pi |\tau|} \log(n|\tau|)\label{boundDomainD0plusImagFramework1}
\end{align}

When $\tau>\theta - \frac{1}{2n+3}$, we have 
\begin{align}
&\frac{1}{\pi \theta} \frac{C_1}{n} \int_{\frac{1}{2n+3}}^{\theta/2} \log(\frac{\theta-s}{s}) \frac{C_1}{1+n(\tau-s)}\; ds + \frac{1}{\pi \theta} \frac{C_1}{n} \int_{\theta/2}^{\theta - \frac{1}{2n+3}} \log(\frac{s}{\theta-s}) \frac{C_1}{1+n(\tau-s)}\; ds\label{IntegralDomainD0plusImagFramework2}\\
&\leq \frac{1}{\pi} \frac{C_1}{n} \left|\log(s) + \log(\frac{1}{n}+ \tau-s)\right|_{\frac{1}{2n+3}}^{\theta/2} \frac{1}{|\tau + 1/n|}\\
&+ \frac{1}{\pi} \frac{C_1}{n} \left|\log(\theta-s) - \log(\frac{1}{n} + \tau-s)\right|_{\theta/2}^{\theta - \frac{1}{2n+3}} \frac{1}{|-\tau - \frac{1}{n} + \theta|}\\
&\leq \frac{1}{\pi} \frac{C_1}{n} \left|\log(\frac{\theta/2}{\frac{1}{2n+3}}) + \log(\frac{1/n + \tau - \theta/2}{1/n + \tau - \frac{1}{2n+3}})\right| \frac{1}{|\tau + 1/n|}\\
&+ \frac{1}{\pi}\frac{C_1}{n} \left|\log(\frac{\frac{1}{2n+3}}{\theta/2})  - \log(\frac{1/n + \tau - \theta + \frac{1}{2n+3}}{1/n + \tau - \theta/2})\right|\frac{1}{|\tau + 1/n - \theta|}
\end{align}
When $\theta \leq |\tau + 1/n|\frac{1}{2}$, we can control the last two lines above as 
\begin{align}
&\frac{1}{\pi \theta} \frac{C_1}{n} \int_{\frac{1}{2n+3}}^{\theta/2} \log(\frac{\theta-s}{s}) \frac{C_1}{1+n(\tau-s)}\; ds + \frac{1}{\pi \theta} \frac{C_1}{n} \int_{\theta/2}^{\theta - \frac{1}{2n+3}} \log(\frac{s}{\theta-s}) \frac{C_1}{1+n(\tau-s)}\; ds\\
&\leq \frac{1}{\pi} \frac{C_1}{n} \left|\log(\frac{3n|\tau|}{2}) + \log(n|\tau|)\right| \frac{1}{|\tau|}\\
&+ \frac{1}{\pi}\frac{C_1}{n} \left|\log(\frac{3n|\tau|}{2}) + \log(\frac{3n|\tau|}{2})\right| \frac{2}{|\tau|}\label{boundDomainD0plusImagFramework2a}
\end{align}
When $\theta \geq |\tau + 1/n|/2$, we have 
\begin{align}
&\frac{1}{\pi \theta} \frac{C_1}{n} \int_{\frac{1}{2n+3}}^{\theta/2} \log(\frac{\theta-s}{s}) \frac{C_1}{1+n(\tau-s)}\; ds + \frac{1}{\pi \theta} \frac{C_1}{n} \int_{\theta/2}^{\theta - \frac{1}{2n+3}} \log(\frac{s}{\theta-s}) \frac{C_1}{1+n(\tau-s)}\; ds\\
&\leq \frac{1}{\pi} \frac{C_1}{n} \left|\log(\frac{3n|\tau|}{2})+ \log(n|\tau|)\right| \frac{1}{|\tau|}\\
&+ \frac{1}{\pi}\frac{C_1}{n\theta} \left(\log(\frac{3n\theta}{2}) \log(\frac{1/n + \tau - \theta + \frac{1}{2n+3}}{1/n + \tau - \theta/2})\right)\\
&\leq \frac{1}{\pi}\frac{C_1}{n} \left|\log(\frac{3n|\tau|}{2}) + \log(n|\tau|)\right| \frac{1}{|\tau|}\\
&+ \frac{1}{|\tau|} \frac{C_1}{n\theta} \left|\log(\frac{3\theta}{|\tau|}) + \log(n|\tau|)\right|\log(n|\tau|)\\
&\leq \frac{1}{\pi}\frac{C_1}{n} \left(\log(\frac{3n|\tau|}{2}) + \log(n|\tau|)\right) \frac{\log(n|\tau|)}{|\tau|}\\
&+ \frac{1}{\pi}\frac{3C_1}{n|\tau|} + \frac{1}{\pi} \frac{C_1}{n} \frac{4}{|\tau|} \log(n|\tau|) \label{boundDomainD0plusImagFramework2b}
\end{align}

Combining~\eqref{boundDomainD0plusImagFramework1}, ~\eqref{boundDomainD0plusImagFramework2a} and~\eqref{boundDomainD0plusImagFramework2b}, we get 
\begin{align}
\eqref{IntegralDomainD0plusImagFramework1} + \eqref{IntegralDomainD0plusImagFramework2}& \leq \frac{C_1}{n|\tau|} \left[\frac{2\log(3/2)}{\pi} + \frac{3}{\pi} \right] + \frac{C_1}{n|\tau|} \log(n|\tau|) \left[\frac{2}{\pi} + \frac{4}{\pi}+ \frac{1}{\pi}\log(3/2) + \frac{4}{\pi}\right]\\
&+ \frac{C_1}{n|\tau|} \log^2(n|\tau|) \frac{2}{\pi}\\
&\leq 1.3 \frac{C_1}{n|\tau|} + 3.4 \frac{C_1}{n|\tau|} \log(n|\tau|)+ \frac{C_1}{n|\tau|}\log^2(n|\tau|) 0.7. \label{FInalBoundImagDomainD0plusTaunotInDOmainImag}
\end{align}
Whenever $\tau \pm \frac{C_1}{n}\notin [\frac{1}{2n+3}, \theta - \frac{1}{2n+3}]$. 

%

When $\tau$ overlaps with either subdomain, we proceed as follows. We first treat the case $\tau\pm \frac{C_1}{n} \cap [\frac{1}{2n+3}, \theta/2]$. We then treat the case $\tau\pm \frac{C_1}{n} \cap [\theta/2, \theta - \frac{1}{2n+3}]$

\begin{align}
&\frac{1}{\pi \theta} \int_{\frac{1}{2n+3}}^{\tau - \frac{C_1-1}{n}} \log(\frac{\theta-s}{s}) \frac{C_1}{1+n(\tau-s)}\; ds \\
&+ \frac{1}{\pi \theta}  \log(1+ \frac{2(C_1-1)}{n} \frac{1}{\tau - \frac{C_1-1}{n}})\label{explanationAdditionalLemmaTronc0050}\\
&+ \frac{1}{\pi \theta} \int_{\tau + \frac{C_1-1}{n}}^{\theta/2} \log(\frac{\theta-s}{s}) \frac{C_1}{1+n(s-\tau)}\;ds\\
&\leq \frac{1}{\pi} \left|\log(s)+ \log(\frac{1}{n}+\tau-s)\right|_{\frac{1}{2n+3}}^{\tau - \frac{C_1-1}{n}} \frac{1}{\tau + 1/n}\\
&+ \frac{2C_1}{n\pi} \log(6C_1)\\
&+ \frac{1}{\pi} \left|-\log(s) + \log(1/n + s - \tau)\right|_{\tau + \frac{C_1-1}{n}}^{\theta/2}\\
&\leq \frac{1}{\pi} \left(\log(\frac{\tau - \frac{C_1-1}{n}}{\frac{1}{2n+3}}) + \log(\frac{C_1/n}{1/n + \tau - \frac{1}{2n+3}})\right) \frac{1}{|\tau + 1/n|}\\
&+ \frac{2C_1}{n \pi} \log(6C_1)\\
&+ \frac{1}{\pi} \left(-\log( \frac{\theta/2}{\tau + \frac{C_1-1}{n}}) + \log(\frac{1/n + \theta/2 - \tau}{C_1/n})\right)\frac{1}{|\tau - 1/n|}\\
&\leq \frac{C_1}{n\pi}\left(\log(3n|\tau|) + \log(\frac{2n|\tau|}{C_1})\right) \frac{1}{|\tau|}\\
&+ \frac{2C_1}{n\pi |\tau|} \log(6C_1)\\
&+ \frac{1}{\pi|\tau - 1/n|} \left(\log(1+ \frac{\tau - 1/n}{C_1/n}) + \log(1+ \frac{\tau - 1/n}{1/n + \theta/2 - \tau})\right)\frac{C_1}{n}\\
&\leq \frac{C_1}{n}\frac{1}{\pi} \left(\log(3n|\tau|) + \log(\frac{2n|\tau|}{C_1})\right)\frac{1}{|\tau|}\\
&+ \frac{2C_1}{n\pi |\tau|} \log(6C_1)\\
&+ \frac{2}{\pi |\tau|} \left(\log(2)\vee \log(2\frac{\tau n}{C_1}) +\log(2)\vee\log(\frac{n\tau}{C_1}) \right)\frac{C_1}{n}
\end{align}
In~\eqref{explanationAdditionalLemmaTronc0050}, we use $\tau \geq 2\frac{C_1-1}{n}$ which implies $\tau - \frac{C_1-1}{n}\geq \frac{1}{2n+3}$. 

Similarly when $\tau \in [\theta/2, \theta-\frac{1}{2n+3}]$, the integral expands as 
\begin{align}
&\frac{1}{\pi \theta} \int_{\theta/2}^{\tau - \frac{C_1-1}{n}} \log(\frac{s}{\theta-s})\frac{C_1}{1+n(\tau-s)}\; ds\label{integral1BoundImagFramework2Ccentral1}\\
& + \frac{1}{\pi \theta} \int_{\tau - \frac{C_1-1}{n}\vee \theta/2}^{\tau + \frac{C_1-1}{n}\wedge \theta - \frac{1}{2n+3}} \log(\frac{s}{\theta-s})\; ds\label{integral1BoundImagFramework2Ccentral2}\\
&+ \frac{1}{\pi \theta} \int_{\tau + \frac{C_1-1}{n}}^{\theta - \frac{1}{2n+3}} \log(\frac{s}{\theta-s}) \frac{C_1}{1+n(s-\tau)}\; ds\label{integral1BoundImagFramework2Ccentral3}
\end{align}

The first integral above is bounded differently depending on $\theta \geq \tau/2$ or $\theta < \tau/2$. In the second case, we write 
\begin{align}
\frac{1}{\pi \theta} \int_{\theta/2}^{\tau - \frac{C_1-1}{n}} \log(\frac{s}{\theta-s}) \frac{C_1}{1+n|\tau-s|}\; ds &\leq \frac{1}{\pi} \frac{C_1}{n} \left[-\log(\frac{\theta - \tau + \frac{C_1-1}{n}}{\theta/2}) + \log(\frac{C_1/n}{1/n + \tau - \theta/2})\right] \frac{1}{|a/n + \tau - \theta|} \\
&\leq \frac{C_1}{n} \frac{2}{|\tau|}  \left|\log(n|\tau|)\vee \log(2) + \log(\frac{2n|\tau|}{C_1})\right|
\end{align}
When $\theta \geq \tau/2$, we write
\begin{align}
\frac{1}{\pi \theta} \int_{\theta/2}^{\tau - \frac{C_1-1}{n}} \log(\frac{s}{\theta-s}) \frac{C_1}{1+n|\tau-s|}\; ds &\leq \frac{1}{\pi}\frac{2}{|\tau|} \log(3n|\tau|) \log(\frac{C_1/n}{1/n + \tau - \theta/2}) \\
&\leq \frac{1}{\pi} \frac{2}{|\tau|} \frac{C_1}{n} \log(3n|\tau|) \log(\frac{n|\tau|}{C_1}).
\end{align}
The second integral is bounded as
\begin{align}
\frac{1}{\pi \theta} \int_{\tau - \frac{C_1-1}{n}\vee \theta/2}^{\tau + \frac{C_1-1}{n}\wedge \theta - \frac{1}{2n+3}} \log(\frac{s}{\theta-s})\; ds &\leq \frac{1}{\pi} \log(1+ 2\frac{C_1-1}{n} \frac{1}{\tau - \frac{C_1-1}{n}})\\
&\leq \frac{1}{\pi} \frac{C_1}{n} \frac{1}{|\tau|}
\end{align}
Finally for the third integral, we write
\begin{align}
\frac{1}{\pi \theta} \int_{\tau + \frac{C_1-1}{n}}^{\theta - \frac{1}{2n+3}} \log(\frac{s}{\theta-s}) \frac{C_1}{1+n(s-\tau)}\; ds&\leq \frac{1}{\pi}\frac{C_1}{n} \left|-\log(\frac{\frac{1}{2n+3}}{\theta - \tau - \frac{C_1-1}{n}}) + \log(\frac{1/n + \theta - \frac{1}{2n+3} - \tau}{C_1/n})\right| \frac{1}{|1/n+\tau + \theta|}\\
&\leq \frac{1}{\pi} \frac{C_1}{n|\tau|} \left(\log(3n\tau) + \log(\frac{2n |\tau|}{C_1})\right) .
\end{align}
The sum $\eqref{integral1BoundImagFramework2Ccentral1}+\eqref{integral1BoundImagFramework2Ccentral2} +\eqref{integral1BoundImagFramework2Ccentral3}$ is thus bounded as
\begin{align}
\eqref{integral1BoundImagFramework2Ccentral1}+\eqref{integral1BoundImagFramework2Ccentral2} +\eqref{integral1BoundImagFramework2Ccentral3}&\leq \frac{C_1}{n|\tau|} \left(2\log(2) + \frac{1}{\pi}(\log(3) + \log(2))+\frac{1}{\pi}\right)\\
& + \frac{C_1}{n|\tau|} \log(n|\tau|) \left(4 + \frac{2}{\pi}\log(3)\right) \\
&+ \frac{C_1}{n|\tau|} \log^2(n|\tau|) \frac{2}{\pi}\\
&\leq  2.3 \frac{C_1}{n|\tau|} + 5\frac{C_1}{n|\tau|} \log(n|\tau|) + \frac{C_1}{n|\tau|}\log^2(n|\tau|)\label{FinalBoundImagCentralDomainD0plusImag}
\end{align}
Taking the maximum between $2\eqref{FinalBoundImagCentralDomainD0plusImag}$ and~\eqref{FInalBoundImagDomainD0plusTaunotInDOmainImag} gives the bound on the imaginary part 
\begin{align}
2\eqref{FinalBoundImagCentralDomainD0plusImag} \vee \eqref{FInalBoundImagDomainD0plusTaunotInDOmainImag} \leq  5\frac{C_1}{n|\tau|} + 10 \frac{C_1}{n|\tau|} \log(n|\tau|) + \frac{C_1}{n|\tau|} \log^2(n|\tau|). \label{boundDomainD0plusFFF3}
\end{align}

The final bound on the $D_0^+$ domain follows from the sum $\eqref{boundDomainD0plusFFF1}+\eqref{boundDomainD0plusFFF2}+\eqref{boundDomainD0plusFFF3}$, 
\begin{align}
2\left(\eqref{boundDomainD0plusFFF1}+\eqref{boundDomainD0plusFFF2}\right)+\eqref{boundDomainD0plusFFF3} \leq 25 \frac{C_1}{n|\tau|} + 27\frac{C_1}{n|\tau|}\log(n|\tau|) + \frac{C_1}{n|\tau|}\log^2(n|\tau|)
\end{align}
which concludes the proof of the lemma.

\subsection{\label{sectionProofD4plusLarge}Proof of lemma~\ref{lemmaTroncD4plus} ($D_4^+$ large $|\tau - \alpha|$)}

We now bound the integral on $D_4^+$. We first consider the large $\theta$ regime () and start with the framework $\tau + \frac{C_1}{n} \leq \theta + \frac{1}{2n+3}$ (inlcuding $\tau<0$). 

Recall that we have 
\begin{align}
F_{R,D_4^+} \leq \frac{\theta}{2} + c\left(\frac{1}{s} + \frac{1}{s-\theta}\right) \leq \frac{\theta}{2} + \frac{2c}{s-\theta}\\
F_{I,D_4^+} \leq \frac{\theta}{2} + c\left(\frac{1}{s} + \frac{1}{s-\theta}\right)+ \frac{1}{4}\log(\frac{s}{s-\theta}) \leq \frac{\theta}{2} + \frac{2c}{s-\theta}+ \frac{1}{4}\log(\frac{s}{s-\theta}) \\
\end{align}

We treat the real part first. 


For $\theta - \frac{1}{2n+3}>\tau>0$, we have 
\begin{align}
&\frac{1}{\pi \theta} \int_{\theta + \frac{1}{2n+3}}^{1/2} \left(\frac{\theta}{2} + 2c\frac{1}{s-\theta}\right) \frac{C_1}{1+n(s-\tau)}\; ds \\
&\leq \frac{1}{\pi \theta}  \frac{\theta}{2} \frac{C_1}{n}\log(\frac{1/n + 1/2 - \tau}{1/n + \theta + \frac{1}{2n+3} - \tau}) + \frac{2c}{\pi \theta}\frac{C_1}{n} \int_{\theta + \frac{1}{2n+3}}^{1/2} \left(\frac{1}{s-\theta} \frac{1}{1/n + (s-\tau)}\right)\; ds\\
&\leq \frac{1}{2\pi} \frac{C_1}{n} \log(\frac{n}{C_1}) +\frac{2c}{\pi \theta} \left|-\log(\frac{1/2 - \theta}{\frac{1}{2n+3}})+\log(\frac{1/n + 1/2 - \tau}{1/n + \theta + \frac{1}{2n+3} - \tau})\right|\frac{C_1}{n} \frac{1}{|\tau - \theta - \frac{1}{n}|}\\
&\leq \frac{1}{2\pi} \frac{C_1}{n} \log(\frac{n|\tau|}{C_1|\tau|}) + \frac{2c}{\pi \theta} \left|\log(1+ \frac{\theta - \tau+1/n}{1/2 - \theta}) + \log(1+ \frac{1/n + \theta - \tau}{\frac{1}{2n+3}})\right|\frac{1}{|\tau - \theta - \frac{1}{n}|}\\
&\leq \frac{1}{2\pi} \frac{C_1}{n} \left\{\log(\frac{n|\tau|}{C_1}) + \log(\frac{1}{|\tau|})\right\} + \frac{2c}{\pi \theta (1/2-\theta)} \frac{C_1}{n} + \frac{2c}{\pi \theta} (2n+3)\frac{C_1}{n}\\
&\stackrel{(a)}{\leq } \frac{1}{2\pi} \frac{C_1}{n|\tau|} \log(\frac{n|\tau|}{C_1}) + \frac{C_1}{n|\tau|} \frac{1}{2\pi} +\frac{2c'}{\pi (1/2-\frac{1}{2n+3})} \frac{C_1}{n|\tau|} + \frac{2c'}{\pi |\tau|} \frac{C_1}{n}\label{DomainD4plusBoundFrameworkTauPositiveOutRealP}
\end{align}
In $(a)$, we use $\theta \geq \tau + \frac{C_1-1}{n} - \frac{1}{2n+3}$. 

When $\tau<0$, we consider two cases

\begin{itemize}
\item When $\theta \geq \tau + 1/n$, we write 
\begin{align}
&\frac{1}{\pi \theta}\int_{\theta + \frac{1}{2n+3}}^{1/2} \left(\frac{\theta}{2} + 2c\frac{1}{s-\theta}\right) \frac{C_1}{1+n(s-\tau)}\; ds\\
 &\leq \frac{1}{2\pi} \frac{C_1}{n|\tau|} + \frac{2c}{\pi \theta} \left|\log(1+ \frac{\theta - \tau + 1/n}{1/2-\theta}) + \log(1+ \frac{1/n + \theta - \tau}{\frac{1}{2n+3}})\right| \frac{C_1}{n|\tau + 1/n - \theta|}\\
&\leq \frac{1}{2\pi} \frac{C_1}{n|\tau|} + (\frac{8c'}{\pi} + \frac{4c'}{\pi |\tau + 1/n|} ) \frac{C_1}{n}\label{DomainD4plusRealLargeTauFullyNegativeA}
\end{align}
\item When $\theta \leq (\tau + 1/n)/2$, we write 
\begin{align}
&\frac{1}{\pi \theta}\int_{\theta + \frac{1}{2n+3}}^{1/2} \left(\frac{\theta}{2} + 2c\frac{1}{s-\theta}\right) \frac{C_1}{1+n(s-\tau)}\; ds\\
&\leq \frac{1}{2\pi}\frac{C_1}{n} \log(\frac{1/n + 1/2 - \tau}{1/n + \theta + \frac{1}{2n+3} - \tau}) \\
&+\frac{1}{2\pi} \frac{C_1}{n|\tau|} + \frac{c'}{\pi} \left\{-\log(\frac{1/2 - \theta}{\frac{1}{2n+3}}) + \log(\frac{1/n + 1/2 - \tau}{1/n + \theta + \frac{1}{2n+3} - \tau})\right\} \frac{C_1}{n}\frac{2}{|1/n + \tau|}\\
&\leq \frac{1}{2\pi} \frac{C_1}{n|\tau|} + \frac{c'}{\pi} \left\{\log(1+ \frac{1/n + \theta - \tau}{1/2 - \theta}\vee \frac{-1/n + \tau - \theta}{1/n + 1/2 - \tau})\right\}\frac{2C_1}{n|\tau|}\\
&+ \frac{c'}{\pi} \left\{\log(1+ \frac{1/n + \theta - \tau}{\frac{1}{2n+3}}\vee \frac{-1/n - \theta + \tau}{1/n + \theta + \frac{1}{2n+3} - \tau})\right\}\frac{2C_1}{n|\tau|}\\
&\leq \frac{1}{2\pi} \frac{C_1}{n|\tau|} + \frac{8c'}{\pi} \frac{C_1}{n|\tau|} \left\{\log(6) + \log(n|\tau|)\right\}\label{DomainD4plusRealLargeTauFullyNegativeB}
\end{align}
\end{itemize}

To be complete, we consider the situation corresponding to an eigenpolynomial located on the right of $s=1/2$ which is needed in order to account for the reflection of the eigenpolynomial. In this case, we have 
\begin{align}
\frac{1}{\pi \theta} \int_{\theta + \frac{1}{2n+3}}^{1/2} \left(\frac{\theta}{2} + 2c\frac{1}{s-\theta}\right) \frac{C_1}{1+n(\tau'-s)}\; ds &\leq \frac{1}{2\pi} \frac{C_1}{n}\log(\frac{1/n + \tau' - \theta - \frac{1}{2n+3}}{1/n + \tau' - 1/2}) + \frac{1}{\pi \theta} \int_{\theta + \frac{1}{2n+3}}^{1/2}2c\frac{1}{s-\theta} \frac{C_1}{1+n(\tau'-s)}\; ds\\
&\leq \frac{1}{2\pi} \frac{C_1}{n}\log(2|\tau'|n)+ \frac{1}{\pi \theta} \int_{\theta + \frac{1}{2n+3}}^{1/2}2c\frac{1}{s-\theta} \frac{C_1}{1+n(\tau'-s)}\; ds\label{boundFirstTermtauRightD4plusLemmaTronc}
\end{align}

For the second term, note that we have 

\begin{align}
\frac{1}{\pi \theta} \int_{\theta + \frac{1}{2n+3}}^{1/2}2c\frac{1}{s-\theta} \frac{C_1}{1+n(\tau'-s)}\; ds \leq \frac{C_1}{n} \frac{2c}{\pi \theta} \frac{1}{|\tau'-\theta + \frac{1}{n}|} \left|\log(\frac{1/2-\theta}{\frac{1}{2n+3}}) - \log(\frac{1/n + \tau' - 1/2}{1/n + \tau - \theta - \frac{1}{2n+3}})\right|
\end{align}

To control the second term, we distinguish the two cases: $\theta \geq 1/4$ and $\theta < 1/4$. In the former we write 
\begin{align}
\frac{1}{\pi \theta} \int_{\theta + \frac{1}{2n+3}}^{1/2}2c\frac{1}{s-\theta} \frac{C_1}{1+n(\tau'-s)}\; ds &\leq \left[(1/2-\theta )(2n+3) + \frac{1/n + \tau' - \theta - \frac{1}{2n+3}}{1/n + \tau' -1/2} \right]\frac{C_1}{n} \frac{2c}{\pi \theta} \frac{1}{|\tau' - \theta + 1/n|}\\
&\leq \frac{4C_1 c'}{n \pi} 2  \leq \frac{4C_1 c'}{n \pi |\tau|} 2 \label{boundSecondTermtauRightD4plusLemmaTronc}
\end{align}

When $\theta< 1/4$, we keep the $o(n)$ terms inside the logs, use $c/\theta \leq c'$ and write
\begin{align}
\frac{1}{\pi \theta} \int_{\theta + \frac{1}{2n+3}}^{1/2}2c\frac{1}{s-\theta} \frac{C_1}{1+n(\tau'-s)}\; ds &\leq \frac{2c'}{\pi}\frac{C_1}{n} \frac{1}{|\tau' - 1/4 + 1/n|} \left(\log((2n+3)/4) + \log(\frac{\tau' - 1/4 + 1/n}{1/n})\right)\\
&\leq \frac{2c'}{\pi} \frac{C_1}{n}4 \left[\log(3n) + \log(2n|\tau'|)\right]\label{boundSecondTermtauRightD4plusLemmaTroncb}
\end{align}

The total bound for any of the frameworks $\tau<0$ or $0<\tau<\theta - \frac{1}{2n+3}$ discussed above is thus give by the sum
\begin{align}
&(\eqref{DomainD4plusBoundFrameworkTauPositiveOutRealP} \vee (\eqref{DomainD4plusRealLargeTauFullyNegativeA}\vee\eqref{DomainD4plusRealLargeTauFullyNegativeB}) )+ (\frac{1}{2\pi} \frac{C_1}{n}\log(2(1-|\tau|)n) + \eqref{boundSecondTermtauRightD4plusLemmaTronc}\vee \eqref{boundSecondTermtauRightD4plusLemmaTroncb})\\
&\leq \left(\frac{C_1}{n|\tau|} \log(n|\tau|)\frac{1}{2\pi} + \frac{C_1}{n|\tau|} \left(\frac{1}{2\pi} + \frac{10c'}{\pi}\right)\right) \vee \left(\frac{C_1}{n|\tau|}(\frac{1}{2\pi} + \frac{16c'}{\pi}) + \frac{C_1}{n|\tau|} \log(n|\tau|) \frac{8c'}{\pi}\right)\\
&+ \frac{C_1}{n|\tau|}\log(n|\tau|) \left(\frac{1}{2\pi} + \frac{16c'}{\pi}\right) + \frac{C_1}{n|\tau|} \left(\log(2)+ \frac{8c'}{\pi}\left(\log(3)+ \log(2)\right) + \frac{8c'}{\pi}\right)\\
&\leq 2\frac{C_1}{n|\tau|} \log(n|\tau|) + 3.5\frac{C_1}{n|\tau|}.\label{boundD4plusRealLargeTauLeftRightLargeTheta}
\end{align}

To conclude on the real part, we now treat the case $\tau\pm \frac{C_1}{n} \in [\theta + \frac{1}{2n+3}, 1/2]$. The integral can be decomposed as 
\begin{align}
&\frac{1}{\pi \theta} \int_{\theta + \frac{1}{2n+3}}^{1/2} \left(\frac{\theta}{2} + 2c\frac{1}{s-\theta}\right) \; \frac{C_1}{1+n|\tau-s|}\; ds \\
&\leq \frac{1}{\pi \theta} \int_{\theta + \frac{1}{2n+3}}^{\tau - \frac{C_1-1}{n}} \left(\frac{\theta}{2}+ 2c \frac{1}{s-\theta}\right) \frac{C_1}{1+n|\tau-s|}\; ds \label{D4plusTerm1IntegralCentralLemmatronc}\\
& + \frac{1}{\pi \theta} \int_{\tau - \frac{C_1-1}{n}}^{\tau + \frac{C_1-1}{n}} \left(\frac{\theta}{2} + 2c \frac{1}{s-\theta}\right)\; ds\label{D4plusTerm1IntegralCentralLemmatronc22222}\\
&+ \frac{1}{\pi \theta} \int_{\tau + \frac{C_1-1}{n}}^{1/2} \left(\frac{\theta}{2} + \frac{2c}{s-\theta}\right) \frac{C_1}{1+n|s-\tau|}\; ds\label{D4plusTerm1IntegralCentralLemmatronc33333}
\end{align}

For the first term, we write 
\begin{align}
\eqref{D4plusTerm1IntegralCentralLemmatronc}&\leq 	\frac{1}{2\pi }\frac{C_1}{n} \log(\frac{C_1/n}{1/n + \tau - \theta - \frac{1}{2n+3}}) + \frac{2c}{\pi \theta} \left|\log(s-\theta) - \log(1/n + \tau-s)\right|_{\theta + \frac{1}{2n+3}}^{\tau - \frac{C_1-1}{n}} \frac{C_1}{n} \frac{1}{\tau - \theta + 1/n}\\
&\leq \frac{1}{2\pi} \frac{C_1}{n} \log(\frac{2n\tau}{C_1}) + \frac{2c}{\pi \theta}\left|\log(\frac{\tau - \frac{C_1-1}{n} - \theta}{\frac{1}{2n+3}}) - \log(\frac{C_1/n}{1/n + \tau - \theta - \frac{1}{2n+3}})\right| \frac{C_1}{n} \frac{1}{\tau - \theta + \frac{1}{n}}\label{intermediateLemmaTroncD4plus001}
\end{align}
To control the second term, we make the distinction between two cases: $\theta > (\tau + \frac{1}{n})/2$ and $\theta< (\tau + 1/n)/2$. In the first case $\theta>(\tau + 1/n)/2$, we use $\log(x)\leq x$ and write 
\begin{align}
\eqref{intermediateLemmaTroncD4plus001}&\leq \frac{1}{2\pi}\frac{C_1}{n} \log(\frac{2n\tau}{C_1}) + \frac{4c}{(\tau + 1/n)} \left((2n+3) + \frac{n}{C_1}\right) \frac{C_1}{n}\\
&\leq \frac{1}{2\pi} \frac{C_1}{n} \log(\frac{2n\tau}{C_1}) + \frac{4c'}{(\tau + 1/n)} \frac{C_1}{n}  + \frac{4c'}{(2n+3)(\tau + 1/n)}
\end{align}
When $\theta < (\tau + 1/n)/2$, we keep the $o(1/n)$ terms inside the logs and write 
\begin{align}
\eqref{intermediateLemmaTroncD4plus001}&\leq \frac{1}{2\pi}\frac{C_1}{n} \log(\frac{2n\tau}{C_1}) + \frac{C_1}{n}\frac{2c'}{\pi} \frac{2}{(\tau + 1/n)} \left(\log(3\tau n) + \log(\frac{2n\tau}{C_1})\right)
\end{align}
For the last integral, we have 
\begin{align}
&\frac{1}{\pi \theta} \int_{\tau + \frac{C_1-1}{n}}^{1/2} \left(\frac{\theta}{2} + 2c \frac{1}{\theta-s}\right) \frac{C_1}{1+n|s-\tau|}\; ds\\
& \leq \frac{1}{2\pi} \log(\frac{1/n + 1/2 - \tau}{C_1/n}) \frac{C_1}{n} + \frac{2c}{\pi \theta} \left|-\log(s-\theta) + \log(\frac{1}{n} + s-\tau)\right|_{\tau + \frac{C_1-1}{n}}^{1/2} \frac{C_1}{n} \frac{1}{|\tau - 1/n - \theta|}\\
&\leq \frac{1}{2\pi} \log(\frac{n}{C_1}) \frac{C_1}{n} + \frac{2c}{\pi \theta} \left(-\log(\frac{1/2-\theta}{\tau + \frac{C_1-1}{n} - \theta}) + \log(\frac{1/n + 1/2 - \tau}{C_1/n})\right) \frac{C_1}{n}\frac{1}{\tau - 1/n - \theta}\\
&\leq \frac{1}{2\pi} \log(\frac{n}{C_1}) \frac{C_1}{n} + \frac{2c}{\pi \theta} \left(\log(1+ \frac{\tau - 1/n - \theta}{C_1/n}\vee \frac{-\tau + \theta + 1/n}{\tau - \theta - 1/n + C_1/n})\right)\frac{C_1}{n}\frac{1}{\tau -1/n - \theta}\\
&+ \frac{2c}{\pi \theta} \left(\log(1+ \frac{-\theta + \tau - 1/n}{1/n + 1/2 - \tau} \vee \frac{\theta - \tau + 1/n}{1/2-\theta})\right) \frac{C_1}{n} \frac{1}{\tau - \frac{1}{n}  -\theta}\label{lastBoundIntegralRightD4plustmp001}
\end{align}
As before, we consider the two cases $\theta \leq (\tau - \frac{1}{n})(1/2)$ and $\theta \geq \frac{1}{2}(\tau - \frac{1}{n})$. In the first case, we keep the logs and use the constant $c$ to cancel the prefactor $\theta^{-1}$. This gives 
\begin{align}
\eqref{lastBoundIntegralRightD4plustmp001}&\leq \frac{1}{2\pi}\log(\frac{n|\tau|}{C_1})\frac{C_1}{n} + \frac{1}{2\pi}\frac{C_1}{n|\tau|} + \frac{2c'}{\pi} \frac{C_1}{n}\frac{4}{|\tau|}\log(\frac{2\tau n}{C_1})+ \frac{2c'}{\pi} \log(n\tau)  \frac{C_1}{n} \frac{4}{|\tau|}\\
\end{align}
When $\theta > (\tau - \frac{1}{n})$, we use $\log(1+x)\leq x$ to cancel the denominator in $(\tau - \frac{1}{n} - \theta)^{-1}$. This gives 
\begin{align}
\eqref{lastBoundIntegralRightD4plustmp001} &\leq \frac{1}{2\pi} \log(\frac{n}{C_1}) \frac{C_1}{n} + \frac{2c}{\pi \theta} \left(\frac{n}{C_1} + n +  \frac{1}{1/2 - \tau + 1/n} \vee \frac{1}{1/2-\theta}\right) \frac{C_1}{n}\\
&\leq \frac{1}{2\pi} \log(\frac{n}{C_1}) \frac{C_1}{n} + \frac{2c}{\pi (\tau - 1/n)} \left(\frac{n}{C_1} + 5n \right) \frac{C_1}{n}\\
&\leq \frac{1}{2\pi}\log(\frac{n}{C_1}) \frac{C_1}{n} + \frac{4c'}{\pi (2n+3)\tau} +  \frac{4c'}{\pi\tau} 5\frac{C_1}{n}.
\end{align}

For the central integral, if we assume $\tau + \frac{C_1-1}{n} \leq 1/2$, as soon as $\tau - \frac{C_1-1}{n}\geq \theta + \frac{1}{2n+3}$, we have 
\begin{align}
\frac{1}{\pi \theta} \int_{\tau - \frac{C_1-1}{n}}^{\tau + \frac{C_1-1}{n}} \left(\frac{\theta}{2} + 2c\frac{1}{s-\theta}\right)\; ds &\leq \frac{1}{2\pi} \log(1+ \frac{2(C_1-1)}{n} \frac{1}{\tau - \frac{C_1-1}{n}}) + \frac{2c}{\pi \theta} \log(1+ 2\frac{C_1-1}{n} \frac{1}{\tau - \frac{C_1-1}{n} - \theta})\label{lastBoundIntegralRightD4plustmp002}\\
&\stackrel{a}{\leq } \frac{1}{2\pi} \frac{2(C_1-1)}{n}\frac{2}{\tau} + \left\{\frac{2c'}{\pi}\frac{C_1-1}{n}\vee \frac{2c'}{(2n+3)\pi}\log(4C_1)\right\}\frac{4}{\tau}. 
\end{align}

$a$ follows from making the distinction between the cases $\frac{1}{2}(\tau - \frac{C_1-1}{n})\geq \theta$ and $\frac{1}{2}(\tau - \frac{C_1-1}{n})<\theta$. depending on each case, we use $\log(1+x)\leq x$ or keep the logs which gives
\begin{align}
\frac{1}{2}(\tau - \frac{C_1-1}{n})\geq \theta &\Rightarrow \eqref{lastBoundIntegralRightD4plustmp002}\leq \frac{1}{2\pi}\frac{2(C_1-1)}{n} \frac{2}{\tau} + \frac{2c'}{\pi}\frac{2(C_1-1)}{n}\frac{4}{\tau}\\
\frac{1}{2}(\tau - \frac{C_1-1}{n})< \theta &\Rightarrow \eqref{lastBoundIntegralRightD4plustmp002}\leq \frac{1}{2\pi}\frac{2(C_1-1)}{n} \frac{2}{\tau} + \frac{2c}{\pi} \frac{4}{\tau} \log(2C_1). 
\end{align}

Combining the bounds above, we get 
\begin{align}
\eqref{D4plusTerm1IntegralCentralLemmatronc} + \eqref{D4plusTerm1IntegralCentralLemmatronc22222} + \eqref{D4plusTerm1IntegralCentralLemmatronc33333}&\leq \frac{C_1}{n|\tau|}\left[\frac{\log(2)}{2\pi} +\frac{\log(2)}{2\pi} + 8c'\right] + \frac{C_1}{n|\tau|}\log(n|\tau|) \left[\frac{1}{2\pi} + \frac{8c'}{\pi}\right]\\
&\leq 2\frac{C_1}{n|\tau|} + \frac{C_1}{n|\tau|}\log(n|\tau|)\label{boundD4plusRealLargeTauCentralLargeTheta}
\end{align}

The total bound on real part for $D_4^+$ is thus given by taking the maximum of twice this bound (to account for the reflexion) and~\eqref{boundD4plusRealLargeTauLeftRightLargeTheta}, which gives
\begin{align}
2\eqref{boundD4plusRealLargeTauCentralLargeTheta} \vee \eqref{boundD4plusRealLargeTauLeftRightLargeTheta} \leq 4 \frac{C_1}{n|\tau|} + 2\frac{C_1}{n|\tau|}\log(n|\tau|). \label{boundDomainD4plusRealLargeThetaLargeTau}
\end{align}

When $\theta\leq \frac{1}{2n+3}$, the bound $F_{D_4}^+$ turns into
\begin{align}
F_{D_4}^+ &\leq \left(\frac{3}{2} + \frac{1}{8}\right) \frac{\theta}{s - \theta} + \frac{\theta}{2}
\end{align}

From this we consider the following frameworks: either $\tau>1/2$, $\tau<\theta + \frac{1}{2n+3}$, or $\tau\pm \frac{C_1}{n}\in [\theta + \frac{1}{2n+3}, 1/2]$.

\begin{itemize}
\item 
 In the first framework, we write 
\begin{align}
&\frac{1}{\pi \theta} \int_{\theta + \frac{1}{2n+3}}^{1/2} \left(\left(\frac{3}{2}+ \frac{1}{8}\right)\frac{\theta}{s-\theta} + \frac{\theta}{2}\right) \frac{C_1}{1+n(\tau-s)}\; ds\\
&\stackrel{(a)}{\leq } \left(\frac{3}{2} + \frac{1}{8}\right) \frac{1}{\pi} \left|\log(s-\theta) - \log(1/n + \tau - s)\right|_{\theta + \frac{1}{2n+3}}^{1/2}\frac{C_1/n}{1/n + \tau - \theta}\\
&+ \frac{1}{2\pi} \frac{C_1}{n} \log(\frac{1/n + \tau - 1/2}{1/n + \tau - (\theta + \frac{1}{2n+3})})\\
&\stackrel{(b)}{\leq } \left(\frac{3}{2} + \frac{1}{8}\right) \frac{1}{\pi} \left(\log(\frac{1/2 - \theta}{\frac{1}{2n+3}}) - \log(\frac{1/n + \tau - 1/2}{1/n + \tau - \theta - \frac{1}{2n+3}})\right) \frac{2C_1}{n}\\
&\stackrel{(c)}{\leq } \left(\frac{3}{2} + \frac{1}{8}\right) \frac{1}{\pi} \left(\log(\frac{3}{2}n|\tau|) + \log(\frac{1}{|\tau|}) + \log(2n|\tau|)\right) \frac{2C_1}{n}\\
&+ \frac{1}{2\pi} \frac{C_1}{n} \log(2n|\tau|)\\
&\stackrel{(d)}{\leq } \left(\frac{3}{2}+ \frac{1}{8}\right) \frac{4}{\pi} \log(2n|\tau|) \frac{C_1}{n} + \left(\frac{3}{2} + \frac{1}{8}\right) \frac{2}{\pi} \frac{C_1}{n|\tau|} \\
&+ \frac{1}{2\pi} \frac{C_1}{n} \log(2n|\tau|)\\
&\leq 3\frac{C_1}{n|\tau|} + 3\frac{C_1}{n|\tau|}\log(n|\tau|)\label{DomainD4plusSmallThetaReal111}
\end{align}
In the sequence of inequalities above, $(b)$ relies on $\theta \leq \frac{1}{2n+3}$ and $(c)$ uses $\tau \geq 1/2$. 
\item When $\tau<0$ (the case $\tau+ \frac{C_1-1}{n} \geq \theta + \frac{1}{2n+3}$ never arises as we have $\tau>\frac{C_1}{n}$ and $\theta<\frac{1}{2n+3}$), we have 
\begin{align}
&\frac{1}{\pi \theta} \int_{\theta + \frac{1}{2n+3}}^{1/2} \left(\left(\frac{3}{2}+ \frac{1}{8}\right) \frac{\theta}{s-\theta}+ \frac{\theta}{2}\right) \frac{C_1}{1+n(s-\tau)}\; ds\\
&\stackrel{(a)}{\leq } \frac{1}{\pi} \left(\frac{3}{2}+ \frac{1}{8}\right) \left|-\log(s-\theta) + \log(1/n + s - \tau)\right|_{\theta + \frac{1}{2n+3}}^{1/2} \frac{C_1/n}{\tau - 1/n - \theta}\\
&+ \frac{1}{2\pi} \frac{C_1}{n} \log(\frac{1/n + 1/2 - \tau}{1/n + \theta + \frac{1}{2n+3} - \tau})\\
&\stackrel{(b)}{\leq }\frac{1}{\pi}\left(\frac{3}{2} + \frac{1}{8}\right) \left|-\log(\frac{1/2 - \theta}{\frac{1}{2n+3}}) + \log(\frac{1/n + \frac{1}{2} - \tau}{1/n + \theta + \frac{1}{2n+3} - \tau})\right| \frac{C_1/n}{|\theta - \tau + \frac{1}{n}|}\\
&+ \frac{1}{2\pi} \frac{C_1}{n} \log(n)\\
&\stackrel{(c)}{\leq } \frac{1}{\pi}\left(\frac{3}{2} + \frac{1}{8}\right) \left(\log(\frac{1/n + 1/2 - \tau}{1/2 - \theta}) + \log(\frac{1/n + \frac{1}{2n+3} + \theta - \tau}{\frac{1}{2n+3}})\right) \frac{C_1}{n|\tau|}\\
&\stackrel{(d)}{\leq } \frac{1}{\pi}\left(\frac{3}{2}+ \frac{1}{8}\right) \left(\log(1+ \frac{1/n - \tau + \theta}{1/2 - \theta}) + \log(6n|\tau|)\right) \frac{C_1}{n|\tau|}\\
&+ \frac{1}{2\pi} \frac{C_1}{n} \log(n)\\
&\stackrel{(e)}{\leq }\frac{1}{\pi} \left(\frac{3}{2} + \frac{1}{8}\right) \left(\log(5)+ \log(6)+ \log(n|\tau|)\right)\frac{C_1}{n|\tau|}\\
&+ \frac{1}{2\pi} \frac{C_1}{n} \log(n)\\
&\leq 2\frac{C_1}{n|\tau|} + \frac{C_1}{n|\tau|}\log(n|\tau|) \label{DomainD4plusSmallThetaReal222}
\end{align}
\item Finally, we consider the case $\tau \pm \frac{C_1}{n} \in [\theta + \frac{1}{2n+3}, 1/2]$. In this case we write 
\begin{align}
&\frac{1}{\pi \theta} \int_{\theta + \frac{1}{2n+3}}^{\tau - \frac{C_1-1}{n}} \left\{\left(\frac{3}{2} + \frac{1}{8}\right) \frac{\theta}{s-\theta} + \frac{\theta}{2}\right\} \; ds\\
&+ \frac{1}{\pi \theta} \int_{\tau - \frac{C_1-1}{n}}^{\tau + \frac{C_1-1}{n}} \left(\left(\frac{3}{2} + \frac{1}{8}\right) \frac{\theta}{s- \theta} + \frac{\theta}{s}\right)\; ds\\
& + \frac{1}{\pi \theta} \int_{\tau + \frac{C_1-1}{n}}^{1/2} \left(\left(\frac{3}{2}+ \frac{1}{8}\right)\frac{\theta}{s-\theta} + \frac{\theta}{2}\right) \frac{C_1}{1+n(s-\tau)}\; ds\\
&\stackrel{(a)}{\leq } \frac{1}{\pi}\left(\frac{3}{2} + \frac{1}{8}\right) \left|\log(s- \theta) - \log(\frac{1}{n} + \tau-s)\right|_{\theta + \frac{1}{2n+3}}^{\tau - \frac{C_1-1}{n}}\frac{C_1/n}{1/n + \tau - \theta}\\
&+ \frac{1}{2\pi} \frac{C_1}{n} \log(\frac{C_1/n}{1/n + \tau - \theta - \frac{1}{2n+3}})\\
&+ \frac{1}{\pi} \left(\frac{3}{2} + \frac{1}{8}\right) \log(\frac{\tau+ \frac{C_1-1}{n} - \theta}{\tau - \frac{C_1-1}{n} - \theta}) + \frac{1}{\pi} \frac{C_1-1}{n}\\
&+ \frac{1}{\pi} \left(\frac{3}{2} + \frac{1}{8}\right)  \left|-\log(s-\theta) + \log(1/n + s-\tau)\right|_{\tau + \frac{C_1-1}{n}}^{1/2}  \frac{C_1/n}{-1/n + \tau - \theta}\\
&+ \frac{1}{2\pi} \frac{C_1}{n}\log(\frac{1/n + 1/2 - \tau}{C_1/n})\\
&\stackrel{(b)}{\leq }  \frac{1}{\pi} \left(\frac{3}{2} + \frac{1}{8}\right) \left\{\log(\frac{\tau - \frac{C_1-1}{n} - \theta}{\frac{1}{2n+3}}) - \log(\frac{C_1/n}{1/n+ \tau - \theta - \frac{1}{2n+3}})\right\} \frac{C_1/n}{1/n + \tau - \theta}\\
&+ \frac{1}{2\pi} \frac{C_1}{n} \log(\frac{2n|\tau|}{C_1})\\
&+ \frac{1}{\pi} \left(\frac{3}{2} + \frac{1}{8}\right) \log(1+ 2\frac{C_1-1}{n}\frac{1}{\tau - \frac{C_1-1}{n} - \theta}) + \frac{1}{\pi} \frac{C_1-1}{n}\\
&+ \frac{1}{\pi} \left(\frac{3}{2}+ \frac{1}{8}\right) \left(-\log(\frac{1/2-\theta}{\tau + \frac{C_1-1}{n} - \theta}) + \log(\frac{1/n + 1/2 - \tau}{C_1/n})\right) \frac{C_1/n}{-1/n + \tau - \theta}\\
&+ \frac{1}{2\pi} \frac{C_1}{n} \log(\frac{1/n + 1/2 - \tau}{C_1/n})\\
&\stackrel{(c)}{\leq }  \frac{1}{\pi}\left(\frac{3}{2}+ \frac{1}{8}\right) \left(\log(3n|\tau|) + \log(\frac{2|\tau|n}{C_1})\right) \frac{2C_1}{n|\tau|}\\
&+ \frac{1}{2\pi} \frac{C_1}{n} \log(\frac{2n|\tau|}{C_1})\\
&+ \frac{1}{\pi} \left(\frac{3}{2}+ \frac{1}{8}\right) \frac{2(C_1-1)}{n} \frac{2}{\tau} + \frac{1}{\pi} \frac{C_1-1}{n}\\
&+ \frac{1}{\pi} \left(\frac{3}{2} + \frac{1}{8}\right) \frac{2C_1}{n|\tau|} \log(\frac{2n|\tau|}{C_1})\\
&+ \frac{1}{\pi} \left(\frac{3}{2}+ \frac{1}{8}\right) \frac{8C_1}{n}\\
&+ \frac{1}{2\pi} \frac{C_1}{n} \left(\log(\frac{n|\tau|}{C_1}) + \frac{1}{|\tau|}\right)\\
&\leq 10\frac{C_1}{n|\tau|} +3.5 \frac{C_1}{n|\tau|}\log(n|\tau|)\label{DomainD4plusSmallThetaReal333}
\end{align}
\end{itemize}

Combining the bounds~\eqref{DomainD4plusSmallThetaReal111}, ~\eqref{DomainD4plusSmallThetaReal222} and~\eqref{DomainD4plusSmallThetaReal333} as before, we can control the integral on $D_4^+$, whenever $\theta \leq \frac{1}{2n+3}$ as 
\begin{align}
(\eqref{DomainD4plusSmallThetaReal111}+\eqref{DomainD4plusSmallThetaReal222})\vee 2\eqref{DomainD4plusSmallThetaReal333}\leq 20\frac{C_1}{n|\tau|} +7 \frac{C_1}{n|\tau|}\log(n|\tau|)\label{boundDomainD4plusRealSmallThetaLargeTau}
\end{align}

The bound on the real part follows from taking the maximum of~\eqref{boundDomainD4plusRealSmallThetaLargeTau} and~\eqref{boundDomainD4plusRealLargeThetaLargeTau} which is always upper bounded by $20\frac{C_1}{n|\tau|} +7 \frac{C_1}{n|\tau|}\log(n|\tau|)$.

We conclude with the imaginary part. When $|\tau| >1/2$, we make the distinction between $\theta \geq (\tau + 1/n)/2$ and $\theta \leq (\tau + 1/n)/2$. In this last case, we have 
\begin{align}
\frac{1}{\pi\theta} \int_{\theta + \frac{1}{2n+3}}^{1/2} \frac{1}{4}\log(\frac{s}{s-\theta}) \frac{C_1}{1+n(\tau-s)}\; ds &\leq \frac{1}{4\pi } \left|\log(\frac{1/2- \theta}{\frac{1}{2n+3}}) - \log(\frac{1/n + \tau - 1/2}{1/n + \tau - \theta - \frac{1}{2n+3}})\right| \frac{1}{|\tau+ \frac{1}{n} - \theta|} \frac{C_1}{n}\\
&\leq \frac{1}{4\pi} \frac{C_1}{n} \frac{2}{|\tau|} \left(\log(3n|\tau|) + \log(2n|\tau|)\right)\label{ImaginaryPartLargeTauDomainD4plusbound11}
\end{align}
When $\theta \geq (\tau + 1/n)/2$, we keep the log and write 
\begin{align}
\frac{1}{\pi\theta} \int_{\theta + \frac{1}{2n+3}}^{1/2} \frac{1}{4}\log(\frac{s}{s-\theta}) \frac{C_1}{1+n(\tau-s)}\; ds &\leq \frac{2}{\pi |\tau|}\frac{C_1}{n} \log(n|\tau|) \log(\frac{1/n + \tau - 1/2}{1/n + \tau - \theta - \frac{1}{2n+3}})\\
&\leq \frac{2}{4\pi |\tau|}\log^2(n|\tau|) \frac{C_1}{n}\label{ImaginaryPartLargeTauDomainD4plusbound22}
\end{align}

When $\tau < \theta - \frac{1}{2n+3}$, we again make the distinction between $\theta < |\tau - 1/n|/2 $ and $\theta \geq |\tau - 1/n|/2$. In the first case, we have 
\begin{align}
&\frac{1}{\pi \theta} \int_{\theta + \frac{1}{2n+3}}^{1/2} \log(\frac{s}{s-\theta}) \frac{C_1}{1+n(s-\tau)} \; ds\\
&\leq \frac{1}{4\pi} \left[-\log(\frac{1/2 - \theta}{\frac{1}{2n+3}}) + \log(\frac{1/n + 1/2 - \tau}{1/n + \theta - \frac{1}{2n+3} - \tau})\right] \frac{C_1}{|\tau - 1/n - \theta|}\\
&\leq \frac{1}{4\pi}\frac{2C_1}{n|\tau|} \left\{ \log(1+ \frac{\theta - \tau + 1/n}{1/2-\theta}\vee \frac{\tau- \theta - 1/n}{1/n + 1/2 - \tau})\right\} \\
&+ \frac{1}{4\pi} \frac{2C_1}{n|\tau|} \log(1+ \frac{\theta - \tau + 1/n - \frac{2}{2n+3}}{1/(2n+3)}\vee \frac{-1/n -\theta + \frac{2}{2n+3}+\tau}{1/n + \theta - \frac{1}{2n+3} - \tau}) \\
&\leq \frac{C_1}{\pi n|\tau|} \log(3n|\tau|). 
\label{ImaginaryPartLargeTauDomainD4plusbound33}
\end{align}
When $\theta \geq |\tau - 1/n|(1/2)$, we split the integral as follows

\begin{align}
&\frac{1}{\pi \theta} \int_{\theta + \frac{1}{2n+3}}^{1/2} \log(\frac{s}{s-\theta}) \frac{C_1}{1+n|s-\tau|}\; ds\label{originalIntegral000}\\
&\leq \frac{1}{\pi \theta} \int_{\theta + \frac{1}{2n+3}}^{2\theta} \log(\frac{s}{s - \theta}) \frac{C_1}{1+n|s-\tau|}\; ds\\
&+ \frac{1}{\pi \theta} \int_{2\theta}^{1/2} \theta \left(\frac{1}{s-\theta} - \frac{n}{1+(s-\tau)n}\right) \frac{C_1}{n(\theta - \tau)}\; ds\\
&\leq \frac{1}{\pi \theta} \frac{C_1}{n} \log(6n\theta )\log(\frac{1/n + 2\theta - \tau}{1/n + \theta - \tau + \frac{1}{2n+3}})\\
&+ \frac{1}{\pi} \frac{C_1}{n(\theta - \tau)} \left(\log(\frac{1/2 - \theta}{\theta}) - \log(\frac{1/n + 1/2 - \tau}{1/n + 2\theta - \tau + \frac{1}{2n+3}})\right)\\
&\leq \frac{1}{\pi} \frac{C_1}{n\theta} \log(6n\theta) \log(1+ \frac{\theta}{1/n + \theta - \tau + \frac{1}{2n+3}})\\
&+ \frac{1}{\pi} \frac{C_1}{n(\theta - \tau)} \left(\log(1+\frac{\theta - \tau}{\theta}) + \log(1+ \frac{\theta - \tau}{1/2 - \theta})\right)\label{term2TobeDetailedDomainD4plusImag}
\end{align}
we denote each of the two terms above as $H_1$ and $H_2$. We control each of them below. Starting with $H_1$, note that depending on $|\theta - \tau|>\theta/2$ or $<\theta/2$, we either have 
\begin{align}
\frac{1}{\pi} \frac{C_1}{n\theta}\left| \log(6n\theta) \log(1+ \frac{\theta}{1/n + \theta - \tau + \frac{1}{2n+3}})\right|&\leq \frac{1}{\pi} \frac{C_1}{n\theta} \left|\log(6n\theta) + \log(4)\right|\\
&\leq \frac{C_1}{n\theta \pi}\log(\frac{2\theta}{|\tau|})+ \log(n|\tau|) + \log(4) + \log(6) )\\
&\leq \frac{2C_1}{n|\tau| \pi} + (\log(4) + \log(6))\frac{C_1}{n|\tau|\pi} + 2\log(n|\tau|)\frac{C_1}{n|\tau|\pi}. \label{firstBoundDomainD4longseriesofBound001}
\end{align}
The second line follows from the assumption $\theta\geq \tau/2$. When $|\theta - \tau|<\theta/2$, we necessarily have $\theta/2<\tau<3\theta/2$ which implies $\theta>2\tau/3$ and hence
\begin{align}
\frac{1}{\pi} \frac{C_1}{n\theta}\left| \log(6n\theta) \log(1+ \frac{\theta}{1/n + \theta - \tau + \frac{1}{2n+3}})\right|&\leq \frac{1}{\pi} \frac{C_1}{n\theta} \left|\log(6n\theta) + \log(1+ 3\tau n)\right|\\
& \leq \frac{1}{\pi} \frac{C_1}{n\theta} \left|\log(6n\theta) (\log(2) + \log(3\tau n))\right|\\
&\leq \frac{C_1}{n\pi \theta}( \log(\frac{6\theta}{2|\tau|/3})+ \log(2n|\tau|/3)) (\log(3n|\tau|) + \log(2))\\
&\leq \left(\frac{9C_1}{n|\tau|\pi} + \frac{3C_1}{n\pi 2|\tau|} \log(n|\tau|) \right) \left(\log(3n|\tau|) + \log(2)\right).\label{secondBoundDomainD4longseriesofBound001}
\end{align}
Taking the maximum of those two bounds gives 
\begin{align}
H_1 \leq \eqref{firstBoundDomainD4longseriesofBound001}\vee \eqref{secondBoundDomainD4longseriesofBound001} &\leq \left(2+\log(4) + \log(6)\right)\frac{C_1}{n|\tau|} + \frac{2}{\pi} \log(n|\tau|) \frac{C_1}{n|\tau|}\\
&\vee \frac{9}{\pi}\left(\log(2)+\log(3)\right)\frac{C_1}{n|\tau|} + \frac{3}{2\pi} \left(\log(2)+\log(3) + 6\right) \frac{C_1}{n|\tau|}\log(n|\tau|) + \frac{3}{2\pi}\frac{C_1}{n|\tau|}\log^2(n|\tau|)\\
&\leq 5.2\frac{C_1}{n|\tau|} + \frac{C_1}{n|\tau|}\log(n|\tau|) + 0.5\frac{C_1}{n|\tau|}\log^2(n|\tau|). \label{boundH1aaaa}
\end{align}
For~\eqref{term2TobeDetailedDomainD4plusImag}, we first note that  
\begin{align}
&\frac{1}{\pi} \frac{C_1}{n(\theta - \tau)} \left(\log(1+\frac{\theta - \tau}{\theta}) + \log(1+ \frac{\theta - \tau}{1/2 - \theta})\right)\\
& \leq \frac{1}{\pi }\frac{C_1}{n\theta} + \left[\frac{1}{\pi} \frac{C_1}{n (1/2-\theta)}\delta(|1/2-\theta|>|\tau|/2)  +  \frac{C_1}{n(\theta - \tau)} \log(1+ \frac{\theta - \tau}{1/2 - \theta})\delta(|1/2 - \theta|<|\tau|/2)\right]\label{DomainD4plusImagSeriesOfIntermediateSteps1}
\end{align}
where the result of the maximum depends on whether $1/2 - \theta$ is small or not. Then we note that if we simultaneously have $|\theta - \tau|<\tau/2$ and $|1/2 - \theta|<|\tau/2|$, then the whole integral in~\eqref{originalIntegral000} can be reduced by noting that $1/2 - |\tau|/2<\theta<1/2 + |\tau|/2$ as well as $|\tau|/2<\theta<3|\tau|/2$ and hence $1/2 < \theta + |\tau|/2 < |\tau|$which gives
\begin{align}
\eqref{originalIntegral000}&\leq  \frac{1}{\pi \theta} \int_{|\tau|/2 + \frac{1}{2n+3}}^{|\tau|}  \log(\frac{s}{s-\theta}) \frac{C_1}{1+n|s- \tau|}\; ds\\
&\leq \frac{2}{\pi |\tau|}\frac{C_1}{n} \log(\frac{1/n + |\tau - |\tau||}{1/n + |\tau - |\tau|/2|}) \\
&\leq \frac{2}{\pi |\tau|}\frac{C_1}{n} \left[\log(\frac{1/n + 2|\tau|}{1/n + |\tau|/2})\vee \log(\frac{3|\tau|/2}{1/n})\right]\\
&\leq \frac{2}{\pi}\frac{C_1}{n|\tau|} (\log(2)+\log(n|\tau|))\label{boundH2secondcaseDomainD4plus1}
\end{align}
The last line follows from $\theta \geq \tau$.

Following the discussion above, we can thus focus on deriving a bound in~\eqref{DomainD4plusImagSeriesOfIntermediateSteps1} in the case where either $|1/2 - \theta|<|\tau|/2$ or $|\tau - \theta|<|\tau|/2$. Substituting this in~\eqref{DomainD4plusImagSeriesOfIntermediateSteps1}, we get 
\begin{align}
\frac{1}{\pi} \frac{C_1}{n(\theta - \tau)} \left(\log(1+\frac{\theta - \tau}{\theta}) + \log(1+ \frac{\theta - \tau}{1/2 - \theta})\right)& \leq \frac{4}{\pi} \frac{C_1}{n|\tau|} + \frac{C_1}{\theta - \tau} \left[\log(2) + \log(\frac{|\theta - \tau|n}{2})\right]\\
&\leq \frac{4}{\pi}\frac{C_1}{n|\tau|} + \frac{C_1}{n|\theta - \tau|}\left[\log(2) + \frac{2|\theta - \tau|}{|\tau|} + \log(n|\tau|)\right]\\
&\leq \frac{4}{\pi}\frac{C_1}{n|\tau|} + \frac{2C_1}{n|\tau|}\log(2) + \frac{2C_1}{n|\tau|} + \frac{2C_1}{n|\tau|} \log(n|\tau|). \label{boundH2firstcaseDomainD4plus1}
\end{align}
The last line uses $\theta - \tau \geq |\tau|/2$

Combining~\eqref{boundH2firstcaseDomainD4plus1} and~\eqref{boundH2firstcaseDomainD4plus1}, we thus get 
\begin{align}
H_2 &\leq \eqref{boundH2secondcaseDomainD4plus1}\vee \eqref{boundH2firstcaseDomainD4plus1}\\
&\leq \frac{C_1}{n|\tau|}\left(\frac{4}{\pi} + 2\log(2)+2\right) + 2\frac{C_1}{n|\tau|}\log(n|\tau|)\\
&\leq 5\frac{C_1}{n|\tau|} + 2\frac{C_1}{n|\tau|}\log(n|\tau|). \label{boundH2aaaa}
\end{align}

Combining~\eqref{boundH1aaaa} and~\eqref{boundH2aaaa}, we get
\begin{align}
H_1+H_2 \leq 11\frac{C_1}{n|\tau|} + 3\frac{C_1}{n|\tau|}\log(n|\tau|) + \frac{C_1}{n|\tau|}\log^2(n|\tau|). \label{boundDomainD4plusTaulessThanThetaCase2}
\end{align}

Combining~\eqref{boundDomainD4plusTaulessThanThetaCase2} with~\eqref{ImaginaryPartLargeTauDomainD4plusbound33} and taking the maximum, then adding $(\eqref{ImaginaryPartLargeTauDomainD4plusbound11}\vee 
\eqref{ImaginaryPartLargeTauDomainD4plusbound22})$ gives a bound on the case $\tau<\theta$,  
\begin{align}
&(\eqref{boundDomainD4plusTaulessThanThetaCase2}\vee \eqref{ImaginaryPartLargeTauDomainD4plusbound33}) + (\eqref{ImaginaryPartLargeTauDomainD4plusbound11}\vee 
\eqref{ImaginaryPartLargeTauDomainD4plusbound22})(|\tau| \leftarrow 1-|\tau|)  \\
&\leq  11\frac{C_1}{n|\tau|} + 3\frac{C_1}{n|\tau|}\log(n|\tau|) + \frac{C_1}{n|\tau|}\log^2(n|\tau|) \\
&+ \frac{1}{2\pi} \left(\log(3)+\log(2)\right) \frac{C_1}{n|\tau|} + \frac{1}{\pi}\log(n|\tau|)\frac{C_1}{n|\tau|} + \frac{1}{2\pi} \frac{C_1}{n|\tau|}\log^2(n|\tau|)\\
&\leq 12\frac{C_1}{n|\tau|} + 4\frac{C_1}{n|\tau|}\log(n|\tau|) + 2\frac{C_1}{n|\tau|}\log^2(n|\tau|)\label{boundCaseTauLessTheta}
\end{align}

To conclude, we now control the integral when $\theta + \frac{1}{2n+3}\leq \tau \pm \frac{C_1}{n} \leq 1/2$, we have 
\begin{align}
&\frac{1}{\pi \theta} \int_{\theta+\frac{1}{2n+3}}^{\tau - \frac{C_1-1}{n}}\log(\frac{s}{s-\theta}) \frac{C_1}{1+n(\tau-s)}\; ds\label{term1LemmaTroncCentralD4plus}\\
&+ \frac{1}{\pi \theta} \int_{\tau - \frac{C_1-1}{n}}^{\tau + \frac{C_1-1}{n}} \log(\frac{s}{s-\theta})\; ds\label{term2LemmaTroncCentralD4plus}\\
&+ \frac{1}{\pi \theta} \int_{\tau + \frac{C_1-1}{n}}^{1/2} \log(\frac{s}{s-\theta}) \frac{C_1}{1+n(s-\tau)}\; ds\label{term3LemmaTroncCentralD4bbplus}
\end{align}

We control the first term by making the distinction between the two cases $\theta < (\tau + 1/n)/2$ and $(\tau + 1/n)/2<\theta$ as before. We have 
\begin{align}
&\frac{1}{\pi \theta} \int_{\theta+\frac{1}{2n+3}}^{\tau - \frac{C_1-1}{n}}\log(\frac{s}{s-\theta}) \frac{C_1}{1+n(\tau-s)}\; ds \leq \left\{\begin{array}{l}
\frac{1}{\pi} \left\{\log(3n|\tau|) + \log(\frac{n|\tau|}{C_1}) \vee \log(C_1)\right\}\frac{2C_1/n}{|1/n +\tau|}\\
\frac{2}{\pi |\tau|} \frac{C_1}{n} \log(3n|\tau|) \log(\frac{n|\tau|}{C_1}).  
\end{array}\right.\\
&\leq \frac{2}{\pi}\log(3)\frac{C_1}{n|\tau|} + 4\frac{C_1}{n|\tau|}\log(n|\tau|) + \frac{1}{\pi}\frac{C_1}{n|\tau|}\log^2(n|\tau|)\label{finalBoundLeftDomainD4plusImag001}
\end{align}

For~\eqref{term3LemmaTroncCentralD4bbplus}, one uses a similar decomposition. When $\theta \leq (\tau + 1/n)/2$, we have

\begin{align}
\frac{1}{\pi \theta} \int_{\tau + \frac{C_1-1}{n}}^{1/2} \log(\frac{s}{s-\theta})\frac{C_1}{1+n(s-\tau)}\; ds &\leq \frac{1}{\pi} \left|-\log(\frac{1/2 - \theta}{\tau + \frac{C_1-1}{n} - \theta}) + \log(\frac{1/n + 1/2 - \tau}{C_1/n})\right| \frac{2C_1}{n|\tau|}\\
&\leq \frac{1}{\pi} \left(\log(3n|\tau|) + \log(\frac{3n|\tau|}{C_1})\right) \frac{2C_1 }{n|\tau|}\label{boundDomainD4plusImagCentralTau22b}
\end{align}
When $\theta \geq (\tau + 1/n)/2$, we have 
\begin{align}
&\frac{1}{\pi \theta} \int_{\tau + \frac{C_1-1}{n}}^{1/2} \log(\frac{s}{s-\theta})\frac{C_1}{1+n(s-\tau)}\; ds \leq \frac{2}{\pi|\tau|}  \log(1+ \frac{\theta}{1/2-\theta}) \log(\frac{1/n + 1/2-\theta}{C_1/n})\\
&\stackrel{(a)}{\leq} \frac{1}{\pi \theta} \log(1 + \frac{\theta}{1/2- \theta}) \log(\frac{1/n + \frac{1}{2} - \theta}{C_1/n}) \frac{C_1}{n} 
\end{align}
To control the last line, we use 
\begin{align}
&\frac{1}{\pi \theta} \log(1+ \frac{\theta}{\frac{1}{2} - \theta})  \log(\frac{1/n + 1/2 - \theta}{C_1/n}) \frac{C_1}{n} \\
&\leq \left\{\begin{array}{l}
\displaystyle \frac{2}{\pi(1/2 - \theta)} \left|\log(\frac{n(1/2 - \theta)}{C_1})\vee \log(C_1)\right|\frac{C_1}{n}\quad \text{if $(1/2-\theta) \geq \theta$}\\
\displaystyle \frac{2}{\theta} \log(3n\theta) \log(\frac{2n\theta}{C_1}) \leq \frac{2}{\theta} \frac{C_1}{n} \log(3n\theta) \log(\frac{2n|\tau|}{C_1}), \quad \text{when $(1/2-\theta)\leq \theta \leq \tau$}. 
\end{array}\right.\\
&\leq  \left\{\begin{array}{l}
\displaystyle \left|\log(\frac{n|\tau|}{C_1}) + \log(\frac{1/2-\theta}{|\tau|}) + \log(2) + \log(C_1)\right|\frac{C_1}{n\pi(1/2 - \theta)}\quad \text{if $(1/2-\theta) \geq \theta$}\\
\displaystyle  \frac{2C_1}{n}(\log(3) + \log(n)) \log(2n|\tau|/C_1), \quad \text{when $(1/2-\theta)\leq \theta \leq \tau$}. 
\end{array}\right.\\
&\leq  \left\{\begin{array}{l}
\displaystyle (\log(2)+\log(C_1))\frac{2C_1}{n\pi |\tau|} + \frac{2C_1}{n\pi|\tau|} \log(n|\tau|) + \frac{C_1}{n\pi |\tau|},\quad \text{if $(1/2-\theta) \geq \theta$}\\
\displaystyle  \frac{2C_1}{n|\tau|}(\log(3) + \log(n|\tau|)) \log(2n|\tau|/C_1), \quad \text{when $(1/2-\theta)\leq \theta \leq \tau$}. 
\end{array}\right.\\
&\leq 2\frac{C_1}{n|\tau|} + 4\frac{C_1}{n|\tau|}\log(n|\tau|) + 2\frac{C_1}{n|\tau|}\log^2(n|\tau|)\label{boundCase2DomainD4plusCentralCaseIntegral2}
\end{align}

Combining~\eqref{boundDomainD4plusImagCentralTau22b} and~\eqref{boundCase2DomainD4plusCentralCaseIntegral2}, we get 
\begin{align}
\frac{1}{\pi \theta} \int_{\tau + \frac{C_1-1}{n}}^{1/2} \log(\frac{s}{s-\theta})\frac{C_1}{1+n(s-\tau)}\; ds &\leq 2\frac{C_1}{n|\tau|} + 4\frac{C_1}{n|\tau|}\log(n|\tau|) + 2\frac{C_1}{n|\tau|}\log^2(n|\tau|)\label{FinalBoundRightDomainD4plusImag}
\end{align}

Finally for~\eqref{term2LemmaTroncCentralD4plus}, we have 
\begin{align}
\frac{1}{\pi \theta} \int_{\tau - \frac{C_1-1}{n}}^{\tau + \frac{C_1-1}{n}} \log(\frac{s}{s-\theta})\; ds& \leq \left\{\begin{array}{l}
\displaystyle\frac{1}{\pi} \log(1+ 2\frac{C_1-1}{n} \frac{1}{\tau - \frac{C_1-1}{n} - \theta}), \quad \text{when $\theta \leq (\tau - \frac{C_1-1}{n})/2$}\\
\displaystyle \frac{1}{\pi \theta} \log(1+ \frac{\theta}{\tau - \frac{C_1-1}{n} - \theta}) \frac{2(C_1-1)}{n}, \quad \text{when $\theta \geq(\tau - \frac{C_1-1}{n})/2$}\end{array}
\right. \\
&\stackrel{(a)}{\leq} \left\{\begin{array}{l}
\displaystyle \frac{1}{\pi} \frac{2C_1}{n}\frac{2}{|\tau|} \\
\displaystyle  \frac{1}{\pi \theta} \log(4\theta n) \frac{2C_1}{n} \leq \frac{2C_1}{n|\tau| \pi} + \left(\frac{4}{\pi \tau} \log(4n|\tau|) \frac{2C_1}{n}\right)  
\end{array}\right.\\
&\leq 6\frac{C_1}{n|\tau|} + 4\frac{C_1}{n|\tau|}\log(n|\tau|). \label{finalBoundCentralDomainD4plusImag}
\end{align}
The last line holds as soon as $\tau \geq 2\frac{C_1}{n}$.

Summing~\eqref{finalBoundCentralDomainD4plusImag} ~\eqref{FinalBoundRightDomainD4plusImag} and~\eqref{finalBoundLeftDomainD4plusImag001} we get 
\begin{align}
~\eqref{finalBoundCentralDomainD4plusImag}+~\eqref{FinalBoundRightDomainD4plusImag}+~\eqref{finalBoundLeftDomainD4plusImag001} 
&\leq 9\frac{C_1}{n|\tau|} + 12\log(n|\tau|)\frac{C_1}{n|\tau|} + 3\log^2(n|\tau|)\frac{C_1}{n|\tau|}\label{FInalBoundCentralImagD4plus02020}
%
\end{align}

The final bound on the imaginary part is then given by taking the maximum of~\eqref{boundCaseTauLessTheta} and 2\eqref{FInalBoundCentralImagD4plus02020}. I.e.,
\begin{align}
2\eqref{FInalBoundCentralImagD4plus02020}\vee \eqref{boundCaseTauLessTheta} \leq 18 \frac{C_1}{n|\tau|} + 24\frac{C_1}{n|\tau|}\log(n|\tau|) +  6\frac{C_1}{n|\tau|}\log^2(n|\tau|). \label{BoundImagD4plusFinal000101}
\end{align}

Combining~\eqref{BoundImagD4plusFinal000101} with the bound on the real part gives the result of the lemma
\begin{align}
 &38 \frac{C_1}{n|\tau|} + 31\frac{C_1}{n|\tau|}\log(n|\tau|) +  6\frac{C_1}{n|\tau|}\log^2(n|\tau|)
\end{align}

\subsection{\label{proofSmallerDomainsLarge}Proof of lemma~\ref{lemmaTroncSmallerDomains} ($D_5^-,D_4^-, D_2^-, D_1^-$ as well as $D_1^+, D_3^+$, large $|\tau - \alpha|$)}

The bounds on $D_5^-$ and $D_4^-$ can be derived almost directly given that those stripes correspond to values of $\theta$ that are always larger than $1/2 - \frac{1}{2n+3}$. On those stripes, the bounds on the real part of the interior integral read as
\begin{align}
F_{R,D_4^-} \equiv  c(4+ \frac{1}{1+s-\theta} + \frac{1}{\frac{1}{2n+3} - s}) + \frac{\theta}{2} + \frac{\pi}{4}\\
F_{R,D_5^-} \equiv \frac{\pi}{4} + c(4+ \frac{1}{1+s-\theta + \frac{1}{2n+3}} + \frac{1}{-s}) + \frac{\theta}{2}\label{reminderBoundD5minusLemmaTronc}
\end{align}

For $D_4^-$, we can therefore write 
\begin{align}
\frac{1}{\pi \theta} \int_{-\frac{1}{2n+3}}^{-1/2+\theta} F_{R, D_4^-} \; ds &\leq \frac{4}{\pi} \left(\frac{\pi}{4}+ \frac{1}{4} + 4c\right) \frac{1}{2n+3} + \frac{4c}{\pi} \log(1+ \frac{1}{2n+3} \frac{1}{\frac{1}{2n+3}}) + \frac{4 c}{\pi} \log(1+ \frac{2}{2n+3} ) \\
&\leq \frac{4}{\pi} \left(\frac{\pi}{4}+ \frac{1}{4} + 4c\right) \frac{1}{2n+3}  + \frac{8c}{\pi} \log(2) \\
&\leq \frac{4}{\pi} \left(\frac{\pi}{4}+ \frac{1}{4} + 4c\right) \frac{1}{(2n+3)|\tau|}  +  \frac{8c'}{\pi} \log(2) \frac{1}{2n+3 |\tau|} \\
&\leq \frac{1}{n|\tau|}\label{FinalboundD4minusLargeTauReal}
\end{align}

For $D_5^-$, using~\eqref{reminderBoundD5minusLemmaTronc}, we can control the integral of the real part as
\begin{align}
&\int_{-1/2}^{-1+\theta + \frac{1}{2n+3}} \left(\frac{\pi}{4}+ \frac{\theta}{2} + 4c\right) \; ds + \int_{-1/2}^{-1+\theta + \frac{1}{2n+3}} \frac{c}{1+s-\theta + \frac{1}{2n+3}} + \frac{c}{-s}\; ds\\
&\stackrel{(a)}{\leq } \left(\frac{\pi}{4} + \frac{\theta}{2} + 4c\right)\frac{4}{\pi} \left(-\frac{1}{2} + \theta + \frac{1}{2n+3}\right) + \frac{4c}{\pi}\log(\frac{\frac{2}{2n+3}}{1/2-\theta + \frac{1}{2n+3}}) + \frac{4 c}{\pi } \log(\frac{1 - \theta - \frac{1}{2n+3}}{1/2})\\
&\stackrel{(b)}{\leq } \left(\frac{\pi}{4}+ \frac{\theta}{2} + 4c\right) \frac{4}{\pi} \frac{1}{2n+3} + \frac{4c'}{\pi (2n+3)} \log(2) + \frac{4c'}{\pi (2n+3)} \log(\frac{1/2}{1/2 - \frac{1}{2n+3}}) \\
&\stackrel{(c)}{\leq } \left(\frac{\pi}{4} + \frac{\theta}{2} + 4c\right)\frac{4}{\pi} \frac{1}{(2n+3)|\tau|} + \frac{4c'}{(2n+3)|\tau|}\log(2) + \frac{4c'}{\pi (2n+3)|\tau|} \log(1+ \frac{4}{2n+3})\\
&\leq \left(\frac{\pi}{4} + \frac{\theta}{2} + 4c\right)\frac{4}{\pi} \frac{1}{(2n+3)|\tau|} + \frac{8c'}{\pi (2n+3)|\tau|}\log(2)\\
&\leq \frac{1}{n|\tau|}\label{FinalboundD5minusLargeTauReal}
\end{align}
$(a)$, $(b)$ and $(c)$ both follow from $1/2 \geq \theta \geq \frac{1}{2} - \frac{1}{2} - \frac{1}{2n+3}\geq \frac{1}{4}$. 

For $D_1^-$ and $D_2^-$ as $\tau \geq \frac{C_1-1}{n}$, we only need to consider the case where the eigenpolynonial achieves its maximum on the left or on the right of the interval. Recall that we have 
\begin{align}
F_{R,D_2^-}& \leq \frac{(2n+3)\pi \theta}{4}+\frac{\theta}{2}\\
F_{R,D_1^-}&\leq \frac{\pi}{2} + c\left(\frac{1}{-s+\frac{1}{2n+3}} + \frac{1}{\theta-s}\right) + \frac{\theta}{2}
\end{align}
Starting with $D_1^-$, when $\tau>0$, we have 
\begin{align}
&\frac{1}{\pi \theta} \int_{-\frac{1}{2n+3}}^0 \left(\frac{\pi}{4}+ \frac{\theta}{2}\right) \frac{C_1}{1+n(\tau-s)}\; ds + \frac{c}{\pi \theta} \int_{-\frac{1}{2n+3}}^{0} \left(\frac{1}{-s+ \frac{1}{2n+3}} + \frac{1}{\theta-s}\right) \frac{C_1}{1+n(\tau-s)}\; ds\\
&\stackrel{(a)}{\leq } \frac{1}{\pi \theta} \left(\frac{\pi}{4} + \frac{\theta}{2}\right) \frac{C_1}{n} \log(\frac{1/n + \tau}{1/n + \tau + \frac{1}{2n+3}}) \\
&+ \frac{c}{\pi \theta} \frac{C_1}{n} \left|\log(-s  + \frac{1}{2n+3}) - \log(\frac{1}{n} + \tau-s)\right|_{-\frac{1}{2n+3}}^0 \frac{1}{|-\tau - \frac{1}{n} + \frac{1}{2n+3}|}\\
&+ \frac{c}{\pi \theta} \frac{C_1}{n} \left|\log(\theta - s) - \log(1/n + \tau-s)\right|_{-\frac{1}{2n+3}}^0 \frac{1}{|-\tau - \frac{1}{n} + \theta|}\\
&\stackrel{(b)}{\leq }  \frac{1}{\pi \theta} \left(\frac{\pi}{4}+ \frac{\theta}{2}\right) \frac{C_1}{n} \log(1+ \frac{1}{2n+3} \frac{1}{\tau + 1/n})\\
&+ \frac{c}{\pi \theta} \frac{C_1}{n} \left(\log(\frac{\frac{1}{2n+3}}{\frac{2}{2n+3}}) - \log(\frac{1/n + \tau}{1/n +\tau +\frac{1}{2n+3}})\right) \frac{1}{\tau + 1/n - \frac{1}{2n+3}}\\
&+ \frac{c}{\pi \theta} \frac{C_1}{n} \left(\log(\frac{\theta}{\theta + \frac{1}{2n+3}}) - \log(\frac{1/n + \tau}{1/n + \tau + \frac{1}{2n+3}})\right) \frac{1}{|\tau + 1/n - \theta|}\\
&\stackrel{(c)}{\leq } \frac{1}{\pi} \left(\frac{\pi}{4}+ \frac{\theta}{2}\right) \frac{C_1}{n|\tau|}+ \frac{c'}{\pi} \frac{C_1}{n|\tau|} 2\log(2) + 4\frac{C_1}{n} \frac{c'}{\pi (\tau + 1/n)}\label{boundD1minusRealLargeTauTaupositive}
\end{align}
$(a)$ and $(b)$ follow from direct integration. In order to get $(c)$, we consider two frameworks.  Depending on whether $\theta < (\tau + 1/n) \frac{1}{2}$ or $\theta> (\tau +1/n)\frac{1}{2}$, we either keep the log or use it to cancel the prefactor $|\tau+1/n - \theta|^{-1}$. When $\theta \leq  (\tau + 1/n)\frac{1}{2}$, we have  
\begin{align}
\frac{c}{\pi \theta} \frac{C_1}{n} \left(\log(\frac{\theta}{\theta + \frac{1}{2n+3}})  - \log(\frac{1/n + \tau}{1/n + \tau + \frac{1}{2n+3}})\right) \frac{1}{|\tau+1/n - \theta|} \leq \frac{c'}{\pi} \frac{C_1}{n} \frac{2}{|\tau + \frac{1}{n}|} 2\log(2) \leq \frac{4\log(2)c'}{\pi} \frac{C_1}{n|\tau|}.  
\end{align}
When $\theta \geq  (\tau + 1/n)\frac{1}{2}$ and 
\begin{align}
&\frac{c}{\pi \theta} \frac{C_1}{n} \left(\log(\frac{\theta}{\theta + \frac{1}{2n+3}})  - \log(\frac{1/n + \tau}{1/n + \tau = \frac{1}{2n+3}})\right) \frac{1}{|\tau+1/n - \theta|} \\
&\leq \frac{C_1}{n} \frac{c}{\pi} \frac{2}{\tau + 1/n} \left|\log(1+ \frac{\tau + 1/n - \theta}{\theta} \vee \frac{\theta - 1/n  - \tau}{1/n + \tau})\right|\frac{1}{|\tau + 1/n - \theta|}\\
&+ \frac{C_1}{n} \frac{c}{\pi} \frac{2}{\tau + 1/n} \left|\log(1+ \frac{\theta - 1/n - \tau}{1/n +\tau + \frac{1}{2n+3}}\vee \frac{1/n = \tau - \theta}{\theta + \frac{1}{2n+3}})\right|\\
&\leq \frac{C_1}{n} \frac{2c'}{\pi (\tau + 1/n)} + \frac{C_1}{n}\frac{c'}{\pi} \frac{2}{\tau + 1/n} \leq \frac{4C_1}{n |\tau|} \frac{c'}{\pi}
\end{align}
The last line follows from $\theta \geq \frac{1}{2n+3}$.

When $\tau<0$, we write 
\begin{align}
&\frac{1}{\pi \theta} \int_{-\frac{1}{2n+3}}^0 \left(\frac{\pi}{4}+ \frac{\theta}{2}\right) \frac{C_1}{1+n(s-\tau)} \; ds+ \frac{c}{\pi \theta} \int_{-\frac{1}{2n+3}}^0 \left(\frac{1}{-s+\frac{1}{2n+3}} + \frac{1}{\theta-s}\right) \frac{C_1}{1+n (s-\tau)}\; ds\\
&\leq \frac{1}{\pi \theta} \left(\frac{\pi}{4} + \frac{\theta}{2}\right) \frac{C_1}{n} \log(\frac{1/n - \tau}{1/n - \tau - \frac{1}{2n+3}})\\
&+ \frac{1}{\pi \theta} \left| - \log(-s + \frac{1}{2n+3}) + \log(\frac{1}{n} + s-\tau)\right|_{-\frac{1}{2n+3}}^0 \frac{1}{|-\tau + \frac{1}{n} + \frac{1}{2n+3}|}\\
&+ \frac{1}{\pi \theta} \left|-\log(\theta-s) + \log(\frac{1}{n}+s-\tau)\right|_{-\frac{1}{2n+3}}^0 \frac{1}{|\theta - \tau + \frac{1}{n}|}\\
&\leq \frac{1}{\pi \theta}\left(\frac{\pi}{4} + \frac{\theta}{2}\right) \frac{C_1}{n} \log(1+ \frac{1}{2n+3}\frac{1}{-\tau + \frac{1}{n} - \frac{1}{2n+3}})\label{D1minusStart}\\
&+ \frac{c}{\pi \theta} \left(-\log(\frac{\frac{1}{2n+3}}{\frac{2}{2n+3}}) + \log(\frac{1/n - \tau}{1/n - \tau - \frac{1}{2n+3}})\right) \frac{1}{|\tau|} \frac{C_1}{n}\\
&+ \frac{c}{\pi \theta} \left(-\log(\frac{\theta}{\theta + \frac{1}{2n+3}}) + \log(\frac{1/n - \tau}{1/n - \tau - \frac{1}{2n+3}})\right) \frac{1}{|\theta - \tau + 1/n|}\label{stinkyTermLemmaTronc1}
\end{align}
To control~\eqref{stinkyTermLemmaTronc1}, we consider the cases $\theta<(\tau - \frac{1}{n}) \frac{1}{2}$ and $\theta \geq (\tau - \frac{1}{n})\frac{1}{2}$. In the second case, we use the log to cancel the prefactor, and in the first framework, we use  the constant $c$ to cancel the $\theta$ that appears in the denominator. When $\theta>(\tau - \frac{1}{n})\frac{1}{2}$, this gives 
\begin{align}
\eqref{stinkyTermLemmaTronc1}&\leq \frac{c}{\pi \theta} \frac{C_1}{n} \left(\log(1+ \frac{-\theta + \frac{1}{n} - \tau}{\theta} \vee \frac{-\frac{1}{n} + \tau + \theta}{1/n  - \tau})\right) \frac{1}{|\theta - \tau + 1/n|}\\
&+ \frac{c}{\pi \theta} \frac{C_1}{n} \left(\log(1+ \frac{\theta - \frac{1}{n} + \tau}{1/n  - \tau - \frac{1}{2n+3}}\vee \frac{1/n - \tau - \theta}{\theta + \frac{1}{2n+3}})\right)\frac{1}{\theta - \tau + \frac{1}{n}}\\
&\leq \frac{4c'}{\pi}\frac{C_1}{n |\tau - \frac{1}{n}|} \leq \frac{8c'}{\pi} \frac{C_1}{n|\tau|}
\end{align}
When $\theta\leq (\tau - \frac{1}{n})\frac{1}{2}$, we write 
\begin{align}
\eqref{stinkyTermLemmaTronc1}&\leq \frac{c'}{\pi}\frac{C_1}{n} \left|\log(1+ \frac{1}{2n+3} \frac{1}{\theta}) + \log(1+ \frac{1}{2n+3} \frac{1}{1/n - \tau})\right) \frac{2}{|\tau - 1/n|}\\
&\leq \frac{c'}{\pi}\frac{C_1}{n} \log(2) \frac{8C_1}{n|\tau|}
\end{align}

From this, we can write~\eqref{D1minusStart} to~\eqref{stinkyTermLemmaTronc1} as 
\begin{align}
&\frac{1}{\pi \theta} \int_{-\frac{1}{2n+3}}^0 \left(\frac{\pi}{4}+ \frac{\theta}{2}\right) \frac{C_1}{1+n(s-\tau)} \; ds+ \frac{c}{\pi \theta} \int_{-\frac{1}{2n+3}}^0 \left(\frac{1}{-s+\frac{1}{2n+3}} + \frac{1}{\theta-s}\right) \frac{C_1}{1+n (s-\tau)}\; ds\\
&\leq \frac{1}{\pi}\left(\frac{\pi}{4}+ \frac{1}{4}\right) \frac{2C_1}{n|\tau|}\\
&+ \frac{c'}{\pi} \frac{C_1}{n|\tau|} 2\log(2) + \frac{8c'}{\pi}\frac{C_1}{n|\tau|}\log(2)\label{TauNegativeDomainD1minusLargeTau}
\end{align}

Grouping~\eqref{boundD1minusRealLargeTauTaupositive} and~\eqref{TauNegativeDomainD1minusLargeTau}, we get
\begin{align}
\eqref{TauNegativeDomainD1minusLargeTau} + \eqref{boundD1minusRealLargeTauTaupositive} &\leq 2\frac{C_1}{n|\tau|}. \label{FinalboundD1minusLargeTauReal}
\end{align}

For $D_2^-$, depending on whether $\tau>0$ or $\tau<0$, we have 
\begin{align}
\frac{1}{\pi \theta} \int_{-\frac{1}{2n+3}}^0 \left(\frac{(2n+3)\pi \theta}{4} + \frac{\theta}{2}\right) \frac{C_1}{1+n(\tau-s)}\; ds&\leq \left\{\begin{array}{l}
\left(\frac{(2n+3)\pi}{4} + \frac{1}{2}\right) \frac{1}{\pi} \frac{C_1}{n} \log(\frac{1/n + \tau}{1/n + \tau + \frac{1}{2n+3}}), \quad \tau>0\\
\frac{1}{\pi} \left(\frac{(2n+3)\pi}{4} + \frac{1}{2}\right) \frac{C_1}{n} \log(\frac{1/n - \tau}{1/n - \tau - \frac{1}{2n+3}}) ,\quad \tau<0
\end{array}\right.\\
&\leq \left(\frac{\pi}{4}+ \frac{1}{2(2n+3)}\right)\frac{1}{\pi} \frac{C_1}{n|\tau|}.\label{FinalboundD2minusLargeTauReal}
\end{align}

Similar reasonings can be used to derive the bounds on $D_1^+$ and $D_3^+$. We recall the bounds that were derived on the interior integral earlier
\begin{align}
F_{R,D_1^+}&\leq \frac{\pi(2n+3)\theta}{4} + \frac{\theta}{2}\\
F_{R,D_3^+}& \leq \frac{(2n+3)\pi}{4} \left(s + \frac{1}{2n+3}\right) + c\left((2n+3) + \frac{1}{(\theta-s)}\right) + \frac{\theta}{2} \label{boundFD3plusreminderLemmaTronc}
\end{align}
For $D_1^+$ for maximas of the eigenpolynomial located left and right, we respectively have 
\begin{align}
\frac{1}{\pi} \int_{0}^{\frac{1}{2n+3}} \left(\frac{\pi (2n+3)}{4} + \frac{1}{2}\right) \frac{C_1}{1+n(s-\tau)}\; ds &\leq \frac{C_1}{n\pi} \left(\frac{\pi (2n+3)}{4}+ \frac{1}{2}\right) \log(\frac{1/n - \tau + \frac{1}{2n+3}}{1/n - \tau})\\
&\leq \frac{C_1}{n\pi} \left(\frac{\pi (2n+3)}{4} + \frac{1}{2}\right) \frac{1}{2n+3} \frac{1}{|\tau|}
\end{align}
\begin{align}
\frac{1}{\pi} \int_{0}^{\frac{1}{2n+3}} \left(\frac{\pi (2n+3)}{4} + \frac{1}{2}\right)\frac{C_1}{1+n(\tau-s)} \; ds &\leq \frac{C_1}{n\pi} \left(\frac{\pi (2n+3)}{4} + \frac{1}{2}\right) \log(\frac{1/n + \tau - \frac{1}{2n+3}}{1/n + \tau})\\
&\leq \frac{C_1}{n\pi} \left(\frac{(2n+3)\pi}{4} + \frac{1}{2}\right) \frac{1}{2n+3} \frac{1}{|\tau|}. 
\end{align}
For $D_1^+$, we thus get a total bound of 
\begin{align}
\frac{1}{\pi \theta} \int_{0}^{\frac{1}{2n+3}} F_{R,D_1^+}\frac{C_1}{1+n(s-\tau)}\; ds \leq 2\frac{C_1}{n\pi |\tau|} \left(\frac{\pi}{4}+ \frac{1}{2(2n+3)}\right) \label{FinalboundD1plusLargeTauReal}
\end{align}

For $D_3^+$, first note that~\eqref{boundFD3plusreminderLemmaTronc} simplifies into
\begin{align}
F_{D_3^+} & = \frac{(2n+3)\pi}{4} \left(s + \frac{1}{2n+3}\right) + c((2n+3) + \frac{1}{\theta-s}) + \frac{\theta}{2}\\
&\leq \frac{(2n+3)\pi}{4} \left(\frac{2}{2n+3}\right) + c\left((2n+3) + \frac{1}{\theta-s}\right) + \frac{\theta}{2}\\
&\leq \frac{\pi}{2} + c' + \frac{c}{\theta-s} + \frac{\theta}{2} 
\end{align}
Using this expression, we get
\begin{align}
&\frac{1}{\pi \theta} \int_{0}^{\frac{1}{2n+3}} \left(\frac{\pi}{2} + c' + \frac{\theta}{2}\right) \frac{C_1}{1+n(s-\tau)}\; ds + \frac{c}{\pi \theta} \int_{0}^{\frac{1}{2n+3}} \frac{1}{\theta -s} \frac{C_1}{1+n(s-\tau)}\; ds\\
&\leq \frac{1}{\pi \theta} \left(\frac{\pi}{2} + c' + \frac{\theta}{2}\right) \frac{C_1}{n} \log(\frac{1/n + \frac{1}{2n+3} - \tau}{1/n - \tau})\\
&+ \frac{c'}{\pi} \left|-\log(\theta -s) + \log(\frac{1}{n} +s-\tau)\right|_0^{\frac{1}{2n+3}} \frac{1}{|-\tau + \frac{1}{n} + \theta|} \frac{C_1}{n}\\
&\leq \frac{1}{\pi \theta} \left(\frac{\pi}{2} + c' + \frac{\theta}{2}\right) \frac{C_1}{n} \frac{1}{2n+3} \frac{1}{|1/n - \tau|}\\
&+ \frac{c}{\pi \theta} \left( - \log(\frac{\theta - \frac{1}{2n+3}}{\theta}) + \log(\frac{1/n + \frac{1}{2n+3} - \tau}{1/n - \tau})\right) \frac{C_1}{n|\tau|}\\
&\leq \frac{1}{\pi}\frac{C_1}{n|\tau|} \left(\frac{\pi}{2} + c' + \frac{1}{4}\right) + \frac{2C_1 c'}{n\pi |\tau|} \log(2)\label{boundCaseTauOnLeftDOmainD3plus}
\end{align}
which holds whenever $\tau$ is located on the left of $D_3^+$. 
For any polynomial with its maximum located on the right of $D_3^+$, we get 
\begin{align}
&\frac{1}{\pi \theta} \int_{0}^{\frac{1}{2n+3}} \left(\frac{\pi}{2} + c' + \frac{\theta}{2}\right) \frac{C_1}{1+n(\tau - s)}\; ds + \frac{c}{\pi \theta} \int_{0}^{\frac{1}{2n+3}} \frac{1}{\theta-s} \frac{C_1}{1+n(\tau-s)}\; ds\label{tmpStepD3plusa}\\
&\leq  \frac{1}{\pi \theta} \left(\frac{\pi}{2} + c' + \frac{\theta}{2}\right) \frac{C_1}{n} \log(\frac{1/n + \tau - \frac{1}{2n+3}}{ 1/n + \tau})\\
&+ \frac{c}{\pi \theta} \frac{C_1}{n} \left|\log(\theta-s) - \log(1/n + \tau -s)\right|_0^{\frac{1}{2n+3}} \frac{1}{|-\tau - \frac{1}{n} - \theta|}\\
&\leq \frac{1}{\pi \theta} \left(\frac{\pi}{2} + c' + \frac{\theta}{2}\right) \frac{C_1}{n} \frac{1}{2n+3} \frac{1}{|\tau + \frac{1}{n}|}\\
&+ \frac{c}{\pi \theta} \frac{C_1}{n} \left|\log( \frac{\theta - \frac{1}{2n+3}}{\theta}) - \log(\frac{1/n + \tau - \frac{1}{2n+3}}{1/n + \tau})\right| \frac{1}{|\tau + 1/n - \theta|}\label{stinkyTermLemmaTronc2}
\end{align}
To bound~\eqref{stinkyTermLemmaTronc2}, as before we make the distinction between $\theta \leq (\tau + 1/n)\frac{1}{2}$ and $\theta > (\tau + 1/n)\frac{1}{2}$. In the former case, we have 
\begin{align}
\eqref{stinkyTermLemmaTronc2} &\leq \frac{1}{\pi}\left(\frac{\pi}{2} + c' + \frac{\theta}{2}\right) \frac{C_1}{n}\frac{1}{|\tau|} + \frac{c}{\pi \theta} \frac{C_1}{n} \frac{2}{|\tau + 1/n|} 2\log(2)
\end{align}
In the latter case, we write
\begin{align}
\eqref{stinkyTermLemmaTronc2} &\leq \frac{1}{\pi} \left(\frac{\pi}{2} + c' + \frac{\theta}{2}\right)\frac{C_1}{n} \frac{1}{|\tau|} + \frac{c}{\pi \theta} \frac{C_1}{n} \left|\log(1+ \frac{-1/n - \tau + \theta}{1/n + \tau}\vee \frac{-\theta + 1/n + \tau}{\theta})\right| \frac{1}{|\tau + 1/n - \theta|}\\
&+ \frac{c}{\pi \theta} \frac{C_1}{n} \left|\log(1+ \frac{\theta - 1/n - \tau}{1/n + \tau - \frac{1}{2n+3}}\vee \frac{\tau + 1/n - \theta}{\theta - \frac{1}{2n+3}})\right| \frac{1}{|\tau +\frac{1}{n} - \theta|}\\
&\leq \frac{1}{\pi}\left(\frac{\pi}{2} + c' + \frac{\theta}{2}\right) \frac{C_1}{n|\tau|}\\
&+ \frac{4c'}{\pi |\tau + 1/n|} \frac{C_1}{n} 
\end{align}

Generally, we can thus control~\eqref{tmpStepD3plusa} as 
\begin{align}
\eqref{tmpStepD3plusa} \leq   \frac{1}{\pi}\left(\frac{\pi}{2} + c' + \frac{\theta}{2}\right) \frac{C_1}{n|\tau|}+ \frac{4c'}{\pi |\tau + 1/n|} \frac{C_1}{n} \log(2).\label{boundCaseTauOnRightDOmainD3plus}
\end{align}

Combining~\eqref{boundCaseTauOnRightDOmainD3plus} with~\eqref{boundCaseTauOnLeftDOmainD3plus} gives the bound on the real part of $D_3^+$
\begin{align}
~\eqref{boundCaseTauOnRightDOmainD3plus} +~\eqref{boundCaseTauOnLeftDOmainD3plus}\leq 7.5\frac{C_1}{n|\tau|}. \label{FinalboundD3plusLargeTauReal}
\end{align}

We now show how to control the imaginary parts of all those subdomains. Starting with $D_1^+$ and $D_3^+$. On $D_1^+$, we have 
\begin{align}
\frac{1}{\pi \theta} \int_{0}^{\frac{1}{2n+3}} \frac{(n+1)\pi \theta}{2} \frac{C_1}{1+n(\tau-s)}\; ds &\leq \frac{C_1}{n\pi} \frac{(n+1)\pi}{2} \log(\frac{1/n + \tau - \frac{1}{2n+3}}{1/n +\tau})\\
&\leq \frac{C_1}{2n+3} \frac{1}{1/n + \tau - \frac{1}{2n+3}}\\
&\leq \frac{C_1}{2n+3}\frac{2}{|\tau|}  
\end{align}
The last line holds as soon as $\tau \geq \frac{2}{n}$. Similarly, when $\tau<0$, we have 
\begin{align}
\frac{1}{\pi \theta}\int_{0}^{\frac{1}{2n+3}} \frac{(n+1)\pi \theta}{2} \frac{C_1}{1+n(s-\tau)}\; ds &\leq \frac{C_1}{n}(n+1) \log(\frac{1/n  + \frac{1}{2n+3} - \tau}{1/n - \tau})\\
&\leq \frac{2C_1}{(2n+3)} \frac{1}{1/n - \tau}\\
&\leq \frac{2C_1}{(2n+3) |\tau|}
\end{align}

Hence the contribution is simply controled as 
\begin{align}
\frac{1}{\pi \theta} \int_{0}^{\frac{1}{2n+3}} \frac{(n+1)\pi \theta}{2} \frac{C_1}{1+n|\tau-s|}\; ds &\leq4\frac{C_1}{n|\tau|}\label{FinalboundD1plusLargeTauImag}
\end{align}

For $D_3^+$, recall that we have 
\begin{align}
F_{I,D_3}^+& \leq c\left(\frac{1}{\theta-s} + \frac{1}{\frac{2}{n}-s}\right) + \frac{1}{4} \log(\frac{\theta-s}{\frac{2}{n}-s}) + \frac{\pi (n+1)}{n}
\end{align}
For this bound, we thus write 
\begin{align}
&\frac{1}{\pi \theta} \int_{0}^{\frac{1}{2n+3}} \left[c\left(\frac{1}{\theta-s} + \frac{1}{\frac{2}{n}-s}\right) + \frac{1}{4}\log(\frac{\theta-s}{\frac{2}{n} - s}) + \frac{\pi (n+1)}{n}\right] \frac{C_1}{1+n(\tau-s)}\; ds\label{tmpImagD3plusLemmaTronc}\\
&\leq \frac{2}{\theta} \frac{C_1}{n} \log(1+ \frac{1}{2n+3} \frac{1}{1/n + \tau - \frac{1}{2n+3}})\\
&+ \frac{c}{\pi \theta} \frac{C_1}{n} \left|\log(\theta-s) - \log(\frac{1}{n}+\tau-s)\right|_{0}^{\frac{1}{2n+3}}\frac{1}{-\tau - \frac{1}{n} + \theta}\\
&+ \frac{c}{\pi \theta} \frac{C_1}{n} \left|\log(\frac{2}{n}-s) - \log(\frac{1}{n} + \tau-s)\right|_{0}^{\frac{1}{2n+3}} \frac{1}{|-\tau - \frac{1}{n} + \frac{2}{n}|}\\
&+ \frac{1}{4\pi} \frac{C_1}{n} \left|\log(\frac{2}{n}-s) - \log(1/n + \tau-s)\right|_{0}^{\frac{1}{2n+3}} \frac{1}{|-\tau - \frac{1}{n}+ 2/n|}\\
&\leq \frac{C_1}{n}\frac{2}{|\tau|} + \frac{c}{\pi \theta} \frac{C_1}{n} \left[\log(\frac{\theta - \frac{1}{2n+3}}{\theta}) - \log(\frac{1/n + \tau - \frac{1}{2n+3}}{1/n + \tau})\right] \frac{1}{|-\tau - 1/n + \theta|}\label{term1LemmaTroncD3plusRight}\\
&+ \frac{c}{\pi \theta} \frac{C_1}{n} \left[\log(\frac{2/n - \frac{1}{2n+3}}{2/n}) - \log(\frac{1/n + \tau - \frac{1}{2n+3}}{1/n + \tau})\right] \frac{1}{|-\tau  - 1/n + 2/n|}\label{term2LemmaTroncD3plusRight}\\
&+ \frac{1}{4\pi}\frac{C_1}{n} \left[\log(\frac{2/n - \frac{1}{2n+3}}{2/n})  - \log(\frac{1/n + \tau - \frac{1}{2n+3}}{1/n + \tau})\right] \frac{1}{|-\tau - 1/n + 2/n|}\label{term3LemmaTroncD3plusRight}.
\end{align}

We control each of the three terms~\eqref{term1LemmaTroncD3plusRight},~\eqref{term2LemmaTroncD3plusRight} and~\eqref{term3LemmaTroncD3plusRight} below. For the first two terms, depending on whether $\theta \geq \tau/2$ or $<\tau/2$, we consider two different bounds. When $\theta \geq \tau/2$, we write  
\begin{align}
\eqref{term1LemmaTroncD3plusRight}&\leq \frac{C_1}{n}\frac{2}{|\tau|} + \frac{c'}{\pi} \frac{C_1}{n} \left|\log(1+ \frac{\theta - 1/n - \tau}{1/n + \tau - \frac{1}{2n+3}}\vee \frac{-\theta + 1/n + \tau}{\theta - \frac{1}{2n+3}})\right| \frac{1}{|-\tau - 1/n + \theta|}\\
&+ \frac{c'}{\pi}\frac{C_1}{n} \left|\log(1+ \frac{1/n+\tau - \theta}{\theta}\vee \frac{-1/n - \tau + \theta}{1/n + \tau})\right| \frac{1}{|-\tau - 1/n + \theta|}\\
&\leq \frac{c'}{\pi}\frac{2C_1}{n|\tau|}
\end{align}
and when $\theta \leq \tau/2$
\begin{align}
\eqref{term1LemmaTroncD3plusRight}&\leq \frac{c'}{\pi}\frac{C_1}{n}\frac{4}{|\tau|} \log(2)
\end{align}
whenever $\theta < \tau/2$. For the other two terms, we write 
\begin{align}
\eqref{term2LemmaTroncD3plusRight}&\leq \frac{4c}{\pi \theta} \frac{C_1}{n |\tau|} \log(2) \leq \frac{4c'}{\pi}\frac{C_1}{n}\log(2)
\end{align}
as well as
\begin{align}
\eqref{term3LemmaTroncD3plusRight}\leq \frac{1}{4\pi} \frac{C_1}{n} \frac{2}{|\tau|} 2\log(2)
\end{align}

Grouping those three bounds, when $\tau>0$, we can thus write
\begin{align}
\eqref{tmpImagD3plusLemmaTronc} \leq \frac{12c'}{\pi}2\log(2) \frac{C_1}{n|\tau|}\label{boundIMaginaryPartD3plusTauPositive}
\end{align}

When $\tau<0$, the reasoning is similar and we can write 
\begin{align}
&\frac{1}{\pi \theta} \int_{0}^{\frac{1}{2n+3}} \left(c\left(\frac{1}{\theta-s} + \frac{1}{2/n - s}\right) + \frac{1}{4}\log(\frac{\theta-s}{2/n - s}) + \frac{\pi (n+1)}{n}\right) \frac{C_1}{1+n|s-\tau|}\; ds\\
&\leq \frac{2}{\theta} \frac{C_1}{n} \log(\frac{1/n + \frac{1}{2n+3} - \tau}{1/n - \tau})\\
&+ \frac{c}{\pi \theta} \left|-\log(\theta-s) + \log(\frac{1}{n} + s-\tau)\right|_{0}^{\frac{1}{2n+3}} \frac{1}{|1/n - \tau + \theta|}\\
&+ \left(\frac{c}{\pi \theta} + \frac{1}{4\pi}\right) \left|-\log(\frac{2}{n}-s) + \log(1/n + s-\tau)\right|_{0}^{\frac{1}{2n+3}} \frac{1}{1/n - \tau + 2/n}\\
& \stackrel{(a)}{\leq } \frac{2C_1}{n} \frac{2}{|\tau|} + \frac{c'}{\pi} \left|-\log(\frac{\theta - \frac{1}{2n+3}}{\theta}) + \log(\frac{1/n + \frac{1}{2n+3} - \tau}{1/n - \tau})\right| \frac{C_1/n}{1/n - \tau + \theta}\\
&+ \left(\frac{c'}{\pi} + \frac{1}{4\pi}\right) \left(-\log(\frac{2/n - \frac{1}{2n+3}}{2/n}) + \log(\frac{1/n + \frac{1}{2n+3} - \tau}{1/n - \tau})\right) \frac{C_1/n}{1/n - \tau +2/n}\\
&\stackrel{(b)}{\leq } \frac{4C_1}{n|\tau|} + \frac{c'}{\pi} \left(2\log(2)\right)\frac{C_1}{n|\tau|} + \left(\frac{c'}{\pi} + \frac{1}{4\pi}\right) 2\log(2) \frac{C_1}{n|\tau|}\label{boundTauNegativeImaginaryD3plus}
\end{align}
$(a)$ follows from $|\tau| \geq 2/n$ and $\theta \geq \frac{1}{2n+3}$, $(b)$ follows from $\theta\geq \frac{2}{2n+3}$ as well as $\tau<0$. 

Combining~\eqref{boundTauNegativeImaginaryD3plus} and~\eqref{boundIMaginaryPartD3plusTauPositive} we get the bound on the imaginary part of $D_3^+$
\begin{align}
\eqref{boundTauNegativeImaginaryD3plus} + \eqref{boundIMaginaryPartD3plusTauPositive}\leq 5.5 \frac{C_1}{n|\tau|}. \label{FinalboundD3plusLargeTauImag}
\end{align}

For $D_5^-$ and $D_4^-$, as $\theta \geq \frac{1}{2} - \frac{1}{2n+3}$, as for the real part, we consider a constant (i.e loose) bound on the value of the eigenpolynomial $p$. We start by recalling the bounds derived earlier on those domains,
\begin{align}
F_{D_5^-} & = \frac{2\pi}{4} + \frac{1}{4} \log(\frac{1/2}{-s}) + \frac{1}{4} \log(\frac{1/2}{1+s-\theta + 1/n}) + c \left(2 + \frac{1}{1+s-\theta +1/n}\right) + c\left(2 + \frac{1}{-s}\right)
\end{align}
From this we write 
\begin{align}
F_{D_4^-} &\leq \frac{(n+1)\pi}{4n} + c\left(\frac{1}{1+s-\theta} + 2\right) + \log(\frac{1/2}{1+s-\theta}) + c\left(\frac{1}{-s+1/n} +  2\right) + \log(\frac{1/2}{-s+1/n})
\end{align}

For $D_5^-$, we thus write 
\begin{align}
&\frac{1}{\pi \theta} \int_{-1/2}^{-1+\theta + \frac{1}{2n+3}} \frac{2\pi}{4} + \frac{1}{4} \log(\frac{1/2}{-s}) + \frac{1}{4} \log(\frac{1/2}{1+ s - \theta + 1/n}) + c\left(2+ \frac{1}{1+s-\theta + 1/n}\right) + c\left(2+ \frac{1}{-s}\right) \; ds\\
&\leq \frac{4}{\pi |\tau|} \left(4c + \frac{2\pi}{4}\right) \frac{1}{2n+3} + \frac{1}{\pi \theta} c\log(\frac{1- \theta - \frac{1}{2n+3}}{1/2}) + \frac{c}{\pi \theta} \log(\frac{1/n + \frac{1}{2n+3}}{1/2 - \theta + 1/n}) + \frac{1}{4}\frac{1}{\pi \theta} \log(\frac{1/2}{1/n + \frac{1}{2n+3}}) \frac{1}{2n+3}\\
&\leq \frac{4}{\pi |\tau|} \left(4c + \frac{2\pi}{4}\right) \frac{1}{2n+3} + \frac{4c}{\pi} \log(\frac{1/2}{1/2 - \frac{1}{2n+3}}) + \frac{4c}{\pi} \log(\frac{1/n + \frac{2}{2n+3}}{1/n}) + \frac{1}{4\pi \theta} \log(\frac{n}{2}) \frac{1}{2n+3}\\
& \stackrel{(a)}{\leq} \frac{4}{\pi |\tau|} \left(4c+ \frac{2\pi}{4}\right)\frac{1}{2n+3} +  \frac{1}{2n+3}\frac{4c}{|\tau| \pi} \log(2) + \frac{4c'}{\pi 2n+3} \log(2) + \frac{1}{\pi} \log(\frac{n|\tau|}{2}) \frac{1}{(2n+3)|\tau|} + \frac{1}{\pi}\frac{1}{(2n+3)|\tau|}\\
&\leq \frac{4}{\pi} \left(4c+ \frac{\pi}{2} + c\log(2) +c'\log(2) + 1 \right) \frac{1}{2n+3 |\tau|} + \frac{1}{\pi}\log(n|\tau|) \frac{1}{(2n+3)|\tau|}\\
&\leq \frac{2}{n|\tau|} + \frac{0.2}{n|\tau|}\log(n|\tau|). \label{FinalboundD5minusLargeTauImag}
\end{align}
$(a)$ follows from $|\tau|<1$ as well as $\theta \geq 1/4$. 

We can use a similar reasoning for $D_4^-$ and we write 
\begin{align}
 & \frac{1}{\pi \theta} \int_{-\frac{1}{2n+3}}^{-1/2+\theta} \left[\frac{(n+1)\pi}{4n} + c\left(\frac{1}{1+s-\theta} + 2\right) + \log(\frac{1/2}{1+s-\theta}) + c\left(\frac{1}{-s+1/n} + 2\right) + \log(\frac{1/2}{-s+1/n})\right]\; ds\\
&\stackrel{(a)}{\leq} \frac{4}{\pi}\frac{1}{|\tau|} \left(\frac{\pi}{2} + 4c\right) \frac{1}{2n+3} + \frac{(c+1)}{\pi \theta} \log(\frac{1/2}{1-\frac{1}{2n+3} - \theta}) + \frac{c}{\pi \theta} \log(\frac{1/2 - \theta + 1/n}{\frac{1}{2n+3} + 1/n})\\
& + \frac{1}{\pi \theta} \log(\frac{1/2}{\frac{1}{2n+3}}) \frac{1}{2n+3}\\
&\stackrel{(b)}{\leq} \frac{1}{\pi}\frac{2}{|\tau|} \left(\frac{\pi}{2} + 4c\right)\frac{1}{2n+3} + \frac{(c+1)}{\pi}\frac{4}{|\tau|} \log(\frac{1/2}{1/2 - \frac{1}{2n+3}}) + \frac{4c}{\pi} \log(\frac{\frac{1}{2n+3} + 1/n}{1/n}) \\
&+ \frac{4}{ \pi} \log(\frac{3n}{2}) \frac{1}{2n+3}\\
&\stackrel{(c)}{\leq} \frac{1}{\pi} \frac{2}{|\tau|} \left(\frac{\pi}{2} + 4c\right)\frac{1}{2n+3} + \frac{4(c+1)}{\pi}\frac{1}{|\tau|}\frac{4}{(2n+3)} + \frac{4c}{\pi |\tau|} \log(2) + \frac{4}{\pi |\tau|} \log(\frac{3n|\tau|}{2})\frac{1}{(2n+3)|\tau|} + \frac{4}{\pi |\tau|} \frac{1}{(2n+3)}\\
&\leq \frac{1}{(2n+3) |\tau|} \frac{4}{\pi}\left((\frac{\pi}{2}+ 4c) + 4(c+1) + c\log(2) + 1\right) + \frac{4}{\pi |\tau | (2n+3)} \log(\frac{3n|\tau|}{2})\\
&\leq \frac{5}{n|\tau|} + 0.7\frac{1}{n|\tau|}\log(n|\tau|). \label{FinalboundD4minusLargeTauImag}
\end{align}
In the lines above, $(a)$ follows from $\tau \geq 2/n$ and $\theta \geq \frac{1}{2n+3}$. $(b)$ uses $\theta \geq \frac{1}{2} - \frac{1}{2n+3} \geq 1/4 \geq \tau/2$ as well as $1 - \frac{1}{2n+3} - \theta \geq 1/2 - \frac{1}{2n+3}$ and $(c)$ follows from $\theta \geq \frac{1}{4}\geq \tau/2$.

We now deal with the imaginary parts on $D_1^- $ and $D_2^-$.  Recall that we have 
\begin{align}
F_{I, D_1^-} \leq \left\{\begin{array}{l}
\frac{\theta-s}{1/n - s} + c\left(\frac{1}{\theta-s} + \frac{1}{1/n -s}\right) + \frac{\pi (n+1)}{4n}, \quad \theta \geq 1/n\\
\frac{1/n - s}{\theta-s} + c\left(\frac{1}{\theta-s} + \frac{1}{1/n -s}\right) + \frac{\pi (n+1)}{4n}, \quad \theta \leq 1/n
\end{array}\right.
\end{align}
In particular, as $\theta\geq \frac{1}{2n+3}$, note that we always have 
\begin{align}
F_{I,D_1^-} \leq (3\theta+c) \frac{1}{1/n -s} + (3\theta+c)\frac{1}{\theta-s} + \frac{\pi}{2}. \label{totalBoundD1minusLemmaTronc}
\end{align}
as well as 
\begin{align}
F_{I,D_2^-} \leq \frac{\pi (n+1)\theta}{4}
\end{align}
Starting with $D_2^-$, 
\begin{align}
\frac{1}{\pi \theta} \int_{-\frac{1}{2n+3}}^0 \frac{\pi (n+1)\theta}{4} \frac{C_1}{1+n(s-\tau)}\; ds &\leq \frac{1}{4}(n+1)\frac{C_1}{n} \log(\frac{1/n - \tau}{1/n - \tau - \frac{1}{2n+3}})\\
&\leq \frac{2C_1}{4n|\tau|} \frac{2(n+1)}{2n+3} \leq \frac{C_1}{2n|\tau|} \label{boundImaginaryPartD2plusLargeTauTauNegative}
\end{align}
when $\tau<0$ and 
\begin{align}
\frac{1}{\pi \theta} \int_{-\frac{1}{2n+3}}^{0} \frac{\pi (n+1)\theta}{4} \frac{C_1}{1+n(\tau-s)}\; ds&\leq \frac{1}{4} \frac{C_1}{n} \frac{(n+1)}{2n+3}\frac{2}{|\tau|} \label{boundImaginaryPartD2plusLargeTauTauPositive}
\end{align}
when $\tau>0$. Thus giving a total bound on the imaginary part of 
\begin{align}
\eqref{boundImaginaryPartD2plusLargeTauTauNegative} + \eqref{boundImaginaryPartD2plusLargeTauTauPositive} \leq \frac{C_1}{n|\tau|}. \label{FinalboundD2minusLargeTauImag}
\end{align}

For $D_1^-$, using~\eqref{totalBoundD1minusLemmaTronc}, when $\tau<0$, we have 
\begin{align}
&\frac{1}{\pi \theta} \int_{-\frac{1}{2n+3}}^{0} \left\{(3\theta+c) \frac{1}{\theta-s} + (3\theta+c)\left(\frac{1}{1/n - s}\right) + \frac{\pi}{2}\right\} \frac{C_1}{1+n(s-\tau)}\; ds\\
&\stackrel{(a)}{\leq} \frac{1}{\pi} (3+c') \left|-\log(\theta-s) + \log(1/n + s-\tau)\right|_{-\frac{1}{2n+3}}^0\frac{C_1}{|-\tau + 1/n + \theta|}\\
&+ \frac{1}{2\theta} \frac{C_1}{n} \log(\frac{1/n - \tau}{1/n - \frac{1}{2n+3} - \tau})\\
&+ \frac{(3+c')}{\pi} \frac{C_1}{n} \left| - \log(\frac{1}{n} - s) + \log(1/n + s-\tau)\right|_{-\frac{1}{2n+3}}^0 \frac{C_1/n}{2/n - \tau}\\
&\stackrel{(b)}{\leq} \frac{1}{\pi} \left(1+c'\right) \left(\log(\frac{\theta}{\theta + \frac{1}{2n+3}}) + \log(\frac{1/n - \tau}{1/n - \frac{1}{2n+3} - \tau})\right) \frac{C_1/n}{|\tau|}+ \frac{C_1}{n|\tau|}\\
&+ \frac{(3+c')}{\pi} \left(\log(\frac{1/n}{1/n + \frac{1}{2n+3}}) + \log(\frac{1/n-\tau}{1/n - \frac{1}{2n+3} - \tau})\right) \frac{C_1/n}{|\tau|}\\
&\stackrel{(c)}{\leq} \frac{1}{\pi} \left(3+c'\right)\frac{C_1}{n|\tau|} 2\log(2) + \frac{C_1}{n|\tau|} + \frac{C_1}{n|\tau|} \frac{(3+c')}{\pi} 2\log(2). \label{boundTauNegativeImaginaryPartDomainD1plusLargeTau}
\end{align}
In $(b)$, we use $\tau \geq 2/\tau$. When $\tau>0$, we have 
\begin{align}
&\frac{1}{\pi \theta} \int_{-\frac{1}{2n+3}}^{0} \left\{(3\theta +c) \frac{1}{\theta-s} + (3\theta +c)\left(\frac{1}{\frac{1}{n} - s}\right)+ \frac{\pi}{2}\right\} \frac{C_1}{1+ n (\tau-s)}\; ds\\
&\leq \frac{1}{\pi} (3+c') \left|\log(\theta-s) - \log(\frac{1}{n} + \tau-s)\right|_{-\frac{1}{2n+3}}^0 \frac{C_1/n}{|-1/n - \tau+ \theta|}\label{FirstLineDomainD1minusImagLargeTau}\\
&+ \frac{1}{2\theta} \frac{C_1}{n} \log(\frac{\tau + 1/n}{\tau + 1/n + \frac{1}{2n+3}})\\
&+ \frac{(3+c')}{\pi} \frac{C_1}{n} \left|\log(1/n -s)  - \log(1/n + \tau-s)\right|_{-\frac{1}{2n+3}}^0 \frac{C_1/n}{|\tau|}
\end{align}
To control~\eqref{FirstLineDomainD1minusImagLargeTau}, we first consider the case $\theta \leq (\tau + 1/n)/2$
\begin{align}
\frac{1}{\pi}(3+c') \left\{\log(\frac{\theta}{\theta + \frac{1}{2n+3}}) + \log(\frac{1/n + \tau}{1/n + \tau + \frac{1}{2n+3}})\right\} \frac{C_1}{n}\frac{2}{|\tau|}&\leq \frac{1}{\pi} \frac{C_1}{n} \frac{2}{|\tau|} 2\log(2) (c'+3).\label{firstCaseTauPositiveDomainD1plusImagLargeTau00101}
\end{align}
When $\theta \geq \tau/2$, we use 
\begin{align}
&\frac{1}{\pi} (1+c') \left(\log(1+ \frac{\theta - \frac{1}{n} - \tau}{1/n + \tau + \frac{1}{2n+3}}\vee \frac{\tau + \frac{1}{n} - \theta}{\theta + \frac{1}{2n+3}})\right) \frac{C_1/n}{|-\frac{1}{n}-\tau + \theta|}\\
&+ \frac{1}{\pi}(1+c') \log(1+ \frac{\theta - \frac{1}{n} - \tau}{1/n + \tau}\vee \frac{-\theta + 1/n + \tau}{\theta}) \frac{C_1/n}{|-\frac{1}{n}-\tau + \theta|}\\
&\leq \frac{4}{\pi}(3+c') \frac{1}{|\tau|} \frac{C_1}{n} \label{firstCaseTauPositiveDomainD1plusImagLargeTau00102}
\end{align}
For the remaining two lines, we write 
\begin{align}
& \frac{1}{2\theta} \frac{C_1}{n} \log(\frac{\tau + 1/n}{\tau + 1/n - \frac{1}{2n+3}})\\
&+ \frac{(3+c')}{\pi} \frac{C_1}{n} \left|\log(1/n -s)  - \log(1/n + \tau-s)\right|_{-\frac{1}{2n+3}}^0 \frac{C_1/n}{|\tau|}\\
&\leq \frac{1}{2}\frac{C_1}{n|\tau|} + \frac{(3+c')}{\pi} \left(\log(\frac{1/n}{1/n + \frac{1}{2n+3}}) - \log(\frac{1/n + \tau}{1/n + \tau + \frac{1}{2n+3}})\right)\frac{C_1/n}{|\tau|}\\
&\leq \frac{C_1}{2n|\tau|} + \frac{(3+c')}{\pi}\frac{C_1}{n}2\log(2). \label{firstCaseTauPositiveDomainD1plusImagLargeTauRemainingTwoLines}
\end{align}

Combining~\eqref{boundTauNegativeImaginaryPartDomainD1plusLargeTau}, together with~\eqref{firstCaseTauPositiveDomainD1plusImagLargeTau00102},~\eqref{firstCaseTauPositiveDomainD1plusImagLargeTau00101} and~\eqref{firstCaseTauPositiveDomainD1plusImagLargeTauRemainingTwoLines}, we get 
\begin{align}
\eqref{boundTauNegativeImaginaryPartDomainD1plusLargeTau} \vee \left((\eqref{firstCaseTauPositiveDomainD1plusImagLargeTau00102} \vee \eqref{firstCaseTauPositiveDomainD1plusImagLargeTau00101} )+  \eqref{firstCaseTauPositiveDomainD1plusImagLargeTauRemainingTwoLines}\right)&\leq \left(\frac{1}{\pi}(3+c')4\log(2) + 1+ \frac{3+c'}{\pi} 2\log(2)\right)\frac{C_1}{n|\tau|}\\
&\leq 4\frac{C_1}{n|\tau|}. \label{boundD1minusLargeTauImag}
\end{align}

The contribution arising from all the stripes $D_\ell^{+}$, $\ell = $ and $D_\ell^{-}$ can thus be bounded as 

\begin{align}
 \eqref{FinalboundD1plusLargeTauImag} + \eqref{FinalboundD3plusLargeTauImag} + \eqref{FinalboundD5minusLargeTauImag}+ \eqref{FinalboundD4minusLargeTauImag}+\eqref{FinalboundD2minusLargeTauImag}+\eqref{boundD1minusLargeTauImag}& \leq 14.5 \frac{C_1}{n|\tau|} + \frac{7}{n|\tau|} + \frac{1}{n|\tau|}\log(n|\tau|)
\end{align}
for the imaginary parts and 
\begin{align}
\eqref{FinalboundD4minusLargeTauReal} + \eqref{FinalboundD5minusLargeTauReal} + \eqref{FinalboundD1minusLargeTauReal} + \eqref{FinalboundD2minusLargeTauReal} + \eqref{FinalboundD1plusLargeTauReal}+ \eqref{FinalboundD3plusLargeTauReal}\leq 11\frac{C_1}{n|\tau|} + \frac{2}{n|\tau|}
\end{align}
for the real part. Multiplying the bound on the real part by two and adding the bound on the imaginary gives the result of the lemma.

\subsection{\label{proofD2plusLarge}Proof of lemma~\ref{lemmaTroncD2plus} ($D_2^+$, large $|\tau- \alpha|$)}

To conclude, we bound the integral on $D_2^+$. Recall that we have
\begin{align}
F_{R,D_2^+} &\leq \frac{\pi}{2} + \left(\frac{1}{s} + \frac{1}{s- \theta + \frac{2}{2n+3}}\right) c + \frac{\theta}{2}\\
F_{I,D_2^+}&\leq F_{R,D_2^+}  + \pi + \frac{1}{4}\left|\log(\frac{s}{(s-\theta) + 2/n})\right|
\end{align}

From this, we consider three cases: $\tau > \theta +2\frac{C_1}{n}$, $\tau < \theta - 2\frac{2C_1}{n}$ and $|\tau - \theta| < 2\frac{C_1}{n}$. We treat the real and imaginary parts simulataneously

\begin{itemize}
\item In the case $|\theta - \tau|< \frac{2C_1}{n}$, note that we have $\theta \geq \tau - \frac{2C_1}{n}$. As soon as $\tau \geq 4\frac{C_1}{n}$, we can thus write,
\begin{align}
&\frac{1}{\pi \theta} \int_{\theta - \frac{1}{2n+3}}^{\theta + \frac{1}{2n+3}} \left(\frac{3\pi}{2} + \frac{\theta}{2}\right) + \frac{c}{s} + \frac{c}{s-\theta + \frac{2}{2n+3}} + \frac{1}{4}\left|\log(\frac{s}{s-\theta + 2/n})\right|\; ds\\
&\leq \left(\frac{\pi}{2} + \frac{1}{2} + \pi\right) \frac{4}{\pi |\tau|} \frac{2}{2n+3} + \frac{2}{\pi \tau} 2\log(2) \frac{c'}{2n+3}\\
&+ \frac{2}{2n+3} \frac{1}{4\pi \theta} \left(\log(2n\theta) + \log(3)\right) \\
&\leq  \left(\frac{\pi}{2} + \frac{1}{2} + \pi\right) \frac{4}{\pi |\tau|} \frac{2}{2n+3} + \frac{2}{\pi \tau} 2\log(2) \frac{c'}{2n+3} \\
&+ \frac{2}{2n+3} \frac{1}{\pi |\tau|} \left(1+ 2\log(2n|\tau|)\right) + \frac{\log(3)}{\pi |\tau|} \frac{2}{2n+3}\\
&\leq \frac{16}{n|\tau|} + \frac{1.3}{n|\tau|}\log(n|\tau|). 
\end{align}
The total contribution can thus be bounded as 
\begin{align}
\frac{1}{\pi \theta} \int_{\theta - \frac{1}{2n+3}}^{\theta + \frac{1}{2n+3}} \left(\frac{3\pi}{2} + \frac{\theta}{2}\right) + \frac{c}{s} + \frac{c}{s-\theta + \frac{2}{2n+3}} + \frac{1}{4}\left|\log(\frac{s}{s-\theta + 2/n})\right|\; ds \leq \frac{32}{n|\tau|} + \frac{2.6}{n|\tau|}\log(n|\tau|). \label{boundWhenTauisIn}
\end{align}
In the lines above, we use 
\begin{align}
\left|\log(\frac{s}{s-\theta+2/n})\right|& \leq \log(\frac{s}{s-\theta+2/n}\vee \frac{s-\theta+2/n}{s}) \\
&\leq  \log(1+ \frac{\theta-2/n}{s-\theta + 2/n}\vee \frac{-\theta + 2/n}{s})\\
&\leq \log(1+n\theta) \vee \log(3).
\end{align}
\item When $\tau > \theta + \frac{C_1-1}{n}$, we let $T_1$ and $T_2$ respectively denote the contribution of the real and imaginary parts
\begin{align}
T_1& \equiv \frac{1}{\pi \theta}\int_{\theta - \frac{1}{2n+3}}^{\theta + \frac{1}{2n+3}} \left(\frac{3\pi}{2} + \frac{\theta}{2}\right) + c\left(\frac{1}{s} + \frac{1}{s-\theta + \frac{2}{2n+3}}\right) \frac{C_1}{1+n(\tau-s)}\; ds\\
T_2 &\equiv \frac{1}{\pi \theta} \int_{\theta - \frac{1}{2n+3}}^{\theta + \frac{1}{2n+3}} \frac{1}{4} \log(\frac{s}{s-\theta + \frac{2}{n}}\vee \frac{s-\theta + 2/n}{s}) \frac{C_1}{1+n(\tau-s)}\; ds\\
\end{align}
The total contribution on the subdomain can be controled as $2T_1 + T_2$. We have the following two bounds
\begin{itemize}
\item Either $\theta < (\tau + 1/n)/2$. 
We have 
\begin{align}
T_1&\leq \frac{1}{\pi \theta} \left(\frac{3\pi}{2} + \frac{\theta}{2}\right) \frac{C_1}{n} \log(\frac{1/n + \tau - \theta - \frac{1}{2n+3}}{1/n + \tau-\theta})\label{term1D2pluscas1a}\\
&+ \frac{c}{\pi \theta} \left|\log(s) - \log(\frac{1}{n}+ \tau-s)\right|_{\theta - \frac{1}{2n+3}}^{\theta + \frac{1}{2n+3}}\frac{C_1}{n} \frac{1}{1/n + \tau}\label{term1D2pluscas1b}\\
&+ \frac{c}{\pi \theta} \left|\log(s - \theta + \frac{2}{n}) - \log(\frac{1}{n} + \tau-s)\right|_{\theta - \frac{1}{2n+3}}^{\theta + \frac{1}{2n+3}} \frac{C_1}{n} \frac{1}{|\frac{1}{n} + \tau - \theta|}\label{term1D2pluscas1c}
\end{align}
We bound each of those terms below. For the first term, we have 
\begin{align}
\eqref{term1D2pluscas1a}&\leq \frac{1}{\pi \theta} \left(\frac{3\pi}{2} + \frac{\theta}{2}\right) \frac{C_1}{n}  \log(1+ \frac{2}{2n+3} \frac{1}{1/n + \tau - \theta - \frac{1}{2n+3}})\\
&\leq \frac{2}{\pi} \left(\frac{3\pi}{2} + \frac{\theta}{2}\right) \frac{C_1}{n} \frac{2}{|\tau|}\label{D2plusLemmaTroncTerm1caseb}
\end{align}
For the second and third terms, we have 
\begin{align}
\eqref{term1D2pluscas1b} + \eqref{term1D2pluscas1c} &\leq \frac{c'}{\pi} \left(\log(1+ \frac{1}{2n+3} \frac{1}{\theta - \frac{1}{2n+3}}) + \log(n2|\tau|)\right) \frac{C_1}{n|\tau|}\\
&+ \frac{c'}{\pi}\left(\log(2) + \log(2n|\tau|)\right) \frac{C_1}{n} \frac{2}{|\tau|}\\
&\leq \frac{3c'}{\pi}\left(\log(2) + \log(2n|\tau|)\right) \frac{C_1}{n} \frac{2}{|\tau|}\label{D2plusLemmaTroncTerm23caseb}
\end{align}
Combining~\eqref{D2plusLemmaTroncTerm1caseb} and~\eqref{D2plusLemmaTroncTerm23caseb} we get 
\begin{align}
T_1 \leq \left(6.4C_1 + 0.6 C_1 +  0.4C_1\log(n|\tau|)\right)\frac{1}{n|\tau|}
\end{align}
\item When $\theta > (1/n + \tau)/2$, we use the constant $c$ to cancel the $\theta^{-1}$
\begin{align}
T_1& = \frac{1}{\pi \theta} \int_{\theta - \frac{1}{2n+3}}^{\theta + \frac{1}{2n+3}} \left(\frac{3\pi}{2} + \frac{\theta}{2}\right) + c\left(\frac{1}{s} + \frac{1}{s-\theta + \frac{2}{2n+3}}\right) \frac{C_1}{1+n(\tau-s)}\;ds\\
&\leq \frac{1}{\pi }\frac{2}{|\tau|} \frac{C_1}{n} \left(\frac{3\pi}{2} + \frac{\theta}{2}\right) \log(3)
+ \frac{c'}{\pi} \left(2\log(3) \right) \frac{C_1}{n|\tau|} 
+ \frac{3c' C_1}{n|\tau|\pi}\\
&\leq \left(\frac{2}{\pi} \left(\frac{3\pi}{2} + \frac{1}{4}\right) \log(3) + \frac{c'}{\pi}2\log(3) + \frac{3c'}{\pi}\right) \frac{C_1}{n}\\
&\leq  3.5 \frac{C_1}{n|\tau|}
\end{align}
We can thus write $T_1\leq 7\frac{C_1}{n|\tau|} + 0.5\log(n|\tau|)\frac{C_1}{n|\tau|}$. For the last term, we used 
\begin{align}
&+ \frac{2C_1 c'}{\pi|\tau|}\left|\log(\frac{\frac{1}{2n+3}}{\frac{2}{2n+3}}) - \log(\frac{1/n + \tau - \theta - \frac{1}{2n+3}}{1/n + \tau - \theta + \frac{1}{2n+3}})\right| \frac{1}{|1/n + \tau - \theta|}\\
&\leq ((2n+3) + \frac{(2n+3)}{2}) \frac{c'}{2n+3} \frac{2C_1}{n|\tau|\pi}\\
&\leq \frac{3c' C_1}{n|\tau|\pi}
\end{align}
\end{itemize}
Finally $T_2$ can be bounded as 
\begin{align}
T_2 &\leq  \frac{1}{\pi \theta} \frac{1}{4} \log(\frac{\theta+\frac{1}{2n+3}}{\frac{1}{2n+3}}\vee \frac{\frac{3}{2n+3}}{\theta - \frac{1}{2n+3}})\frac{C_1}{n} \log(\frac{1/n+\tau - \theta - \frac{1}{2n+3}}{1/n + \tau - \theta + \frac{1}{2n+3}})\\
&\leq \frac{2}{\pi \theta} \frac{1}{4}\log(6n\theta \vee \log(3)) \frac{C_1}{n}\log(2)\\
&\leq \frac{1}{\pi |\tau|} \frac{C_1}{n}\log(2) + \frac{1}{\pi |\tau|} \log(6n|\tau|)\frac{C_1}{n} \log(2)\\
&\leq 0.7 \frac{C_1}{n|\tau|} + 0.3\log(n|\tau|)\frac{C_1}{n|\tau|}.
\end{align}
when $\theta \geq |\tau|/2$, and 
\begin{align}
T_2 &\leq  \frac{1}{\pi \theta} \frac{1}{4} \log(\frac{\theta+\frac{1}{2n+3}}{\frac{1}{2n+3}}\vee \frac{\frac{3}{2n+3}}{\theta - \frac{1}{2n+3}})\frac{C_1}{n} \log(\frac{1/n+\tau - \theta - \frac{1}{2n+3}}{1/n + \tau - \theta + \frac{1}{2n+3}})\\
&\leq \frac{1}{4\pi}\frac{2C_1}{n|\tau|}\left( \log((2n+3)|\tau|)  + \log(3)\right)\\
&\leq 0.4 \frac{C_1}{n|\tau|} + \frac{1}{2\pi} \log(n|\tau|)\frac{C_1}{n|\tau|}.
\end{align}
when $\theta \leq |\tau|/2$

Combining the bounds on $T_1$ and $T_2$, we get 
\begin{align}
T_1 + T_2 \leq 8\frac{C_1}{n|\tau|}\log(n|\tau|). 
\end{align}

\item When $\tau < \theta - \frac{1}{2n+3}$, using the same decomposition $T_1(\theta)+T_2(\theta)$ as before, we get 
\begin{align}
T_1;&\equiv \frac{1}{\pi \theta} \int_{\theta - \frac{1}{2n+3}}^{\theta + \frac{1}{2n+3}} \left(\frac{3\pi}{2} + \frac{\theta}{2}\right) \frac{C_1}{1+n(s-\tau)}\; ds\\
&\leq \frac{1}{\pi \theta} \left(\frac{3\pi}{2} + \frac{\theta}{2}\right) \log(1+ \frac{2}{2n+3} \frac{1}{1/n + \theta - \frac{1}{2n+3} - \tau})\\
&\leq \left\{\begin{array}{l}
\frac{4C_1}{n|\tau|} \left(\frac{3\pi}{2} + \frac{\theta}{2}\right) \frac{1}{\pi}, \quad \text{when $\tau<0$}\\
\frac{1}{\pi |\tau|} \left(\frac{3\pi}{2} + \frac{\theta}{2}\right) \frac{C_1}{n}\log(4), \quad \text{when $\tau>0$}.
\end{array}\right.\\
&\leq \frac{4}{\pi}\left(\frac{3\pi}{2} + \frac{1}{4}\right)\frac{1}{\pi} \frac{C_1}{n|\tau|}. 
\end{align}
\begin{align}
T_2'&\equiv \frac{c}{\pi \theta} \int_{\theta - \frac{1}{2n+3}}^{\theta + \frac{1}{2n+3}} \left(\frac{1}{s} + \frac{1}{s - \theta + \frac{2}{2n+3}}\right) \frac{C_1}{1+n(s-\tau)} \; ds\\
&\leq \frac{c}{\pi \theta} \left(\log(\frac{\theta - \frac{1}{2n+3}}{\theta + \frac{1}{2n+3}}) + \log(\frac{1/n + \theta + \frac{1}{2n+3} - \tau}{1/n + \theta - \frac{1}{2n+3} - \tau})\right) \frac{C_1}{n}\\
&+ \frac{c}{\pi \theta} \left(\log(\frac{\frac{1}{2n+3}}{\frac{2}{2n+3}}) + \log(\frac{1/n + \theta + \frac{1}{2n+3} - \tau}{1/n + \theta - \frac{1}{2n+3} - \tau})\right) \frac{C_1}{n}\\
&\leq \frac{c}{\pi \theta} \left(\log(1+ \frac{2}{2n+3} \frac{1}{\theta - \frac{1}{2n+3}}) + \log(1+ \frac{2}{2n+3} \frac{1}{1/n + \theta - \frac{1}{2n+3}-\tau})\right) \frac{C_1}{n}\\
&+ \frac{c}{\pi \theta} \frac{C_1}{n} \left(\log(2)+ \log(1+ \frac{2}{2n+3} \frac{1}{1/n + \theta - \frac{1}{2n+3} - \tau})\right) \frac{C_1}{n}\\
&\leq \frac{c'}{\pi} 4\log(4) \frac{C_1}{n}. 
\end{align}
From this, we can write 
\begin{align}
T_1'+T_2' \leq 3\frac{C_1}{n|\tau|}. 
\end{align}
\end{itemize}
Finally for the imaginary part, we consider two cases. Either $\theta < |\tau|$ or $\theta > |\tau|$. In the last case, we have 
\begin{align}
T_3&\equiv  \frac{1}{4\pi \theta} \int_{\theta - \frac{1}{2n+3}}^{\theta + \frac{1}{2n+3}} \log(\frac{s}{s-\theta + 2/n }\vee \frac{s-\theta + 2/n}{s}) \frac{C_1}{1+n(s-\tau)}\; ds\\
&\leq \frac{1}{4\pi \theta}\int_{\theta - \frac{1}{2n+3}}^{\theta + \frac{1}{2n+3}} \log(1+ \frac{\theta - 2/n}{s-\theta + 2/n}\vee \frac{2/n - \theta}{s})\frac{C_1}{1+n(s-\tau)}\; ds\\
&\leq \frac{1}{4\pi \theta} \left\{\log(1+3n\theta) + \log(6)\right\} \frac{C_1}{n} \log(\frac{1/n + \theta + \frac{1}{2n+3} - \tau}{1/n + \theta - \frac{1}{2n+3} - \tau}) \\
&\leq \left(\frac{1}{4\pi |\tau|}  \log(n|\tau|) + \frac{C_1}{4\pi n|\tau|} + \frac{2\log(2)+2\log(3)}{4\pi |\tau|} \right)\frac{C_1}{n}\log(3) 
\end{align}
When $\theta<|\tau|$, we proceed as follows
\begin{align}
&\frac{1}{4\pi \theta} \int_{\theta - \frac{1}{2n+3}}^{\theta + \frac{1}{2n+3}} \log(\frac{s}{s-\theta + \frac{2}{n}}) \frac{C_1}{1+n|\tau-s|}\; ds\\
&\leq \left\{\begin{array}{l}
\displaystyle \frac{1}{4\pi} \int_{\theta - \frac{1}{2n+3}}^{\theta + \frac{1}{2n+3}} \left[\frac{1}{s-\theta + 2/n} + \frac{1}{1/n + \tau-s}\right]\frac{C_1}{n(\tau-\theta)}\; ds, \quad \text{if $\theta>\tau$}\\
\frac{1}{4\pi}\int_{\theta - \frac{1}{2n+3}}^{\theta + \frac{1}{2n+3}} \left[\frac{1}{s-\theta + 2/n}  - \frac{1}{1/n + s-\tau}\right]\frac{C_1}{n(\tau-\theta)}, \quad \text{if $\theta<-\tau$}
\end{array}\right. \\
&\leq \frac{1}{4\pi}\left[\log(\frac{\frac{1}{2n+3} + 2/n}{1/n}) - \log(\frac{1/n + |\theta - \tau| + \frac{1}{2n+3}}{1/n + |\theta - \tau| - \frac{1}{2n+3}})\right] \frac{C_1}{n|\tau-\theta|}\label{tmp01010101}
\end{align}
To control~\eqref{tmp01010101}, we treat separately the case where $|\tau - \theta|>|\tau|/2$ and $|\tau - \theta|<|\tau|/2$. In the first scenario, we simply write 
\begin{align}
\eqref{tmp01010101}\leq \frac{1}{4\pi} \frac{C_1}{n|\tau|} 4\log(3). 
\end{align}
In the second scenario, the bound $|\theta - \tau|<|\tau|/2$ implies $|\tau|/2\leq \theta \leq 3|\tau|/2$ from which we can simply bound the integral as 
\begin{align}
&\frac{1}{4\pi \theta} \int_{\theta - \frac{1}{2n+3}}^{\theta + \frac{1}{2n+3}} \log(\frac{s}{s - \theta + 2/n}) \frac{C_1}{1+n|\tau-s|}\; ds\\
&\leq \frac{8}{\pi }\frac{C_1}{n|\tau|} \log(\frac{3}{2}|\tau|n) \log(\frac{1/n + |\tau - \theta| + \frac{1}{2n+3}}{1/n + |\tau - \theta| - \frac{1}{2n+3}})\\
&\leq \frac{C_1}{n|\tau|} 8 (\log(3/2) + \log(n|\tau|))\log(2).  
\end{align}

Combining all the scenarios gives the bound
\begin{align}
T_3 + T_3' &\leq 3.3 \frac{C_1}{n|\tau|} + 3.3 \frac{C_1}{n|\tau|}\log(n|\tau|). 
\end{align}

Adding the real part gives 
\begin{align}
T_1'+T_2' + T_1+T_2  + T_3+T_3' \leq  6.3\frac{C_1}{n|\tau|} + 12\frac{C_1}{n|\tau|}\log(n|\tau|).\label{boundWhenTauisOut}
\end{align}

Taking the maximum of~\eqref{boundWhenTauisOut} and~\eqref{boundWhenTauisIn} gives the result of the lemma.

}

{\color{black}


}

\bibliography{biblio}
\bibliographystyle{abbrv}


\end{document}